\documentclass[12pt,twoside,onecolumn,openright]{memoir}

\usepackage{amsmath,amssymb}   
\usepackage{amsfonts}          
\usepackage{amsthm}            
\usepackage{hyperref}
\hypersetup{
    colorlinks,
    citecolor=black,
    filecolor=black,
    linkcolor=black,
    urlcolor=black
}

\usepackage{thmtools,thm-restate}
\usepackage{bm}
\usepackage{cleveref}
\usepackage{tikz-cd}
\usepackage{caption}
\usepackage{dirtytalk}
\usepackage{scalerel}

\DeclareFontFamily{U}{mathc}{}
\DeclareFontShape{U}{mathc}{m}{it}%
{<->s*[1.03] mathc10}{}
\DeclareMathAlphabet{\mathcal}{U}{mathc}{m}{it}

\usepackage{bbm}
\newcommand{\Bl}{\mathsf{Bl}}
\newcommand{\lie}{\mathfrak{lie}}
\newcommand{\Sig}{\mathsf{Sig}}

\newcommand{\D}{\mathbb{D}}
\newcommand{\N}{\mathbb{N}}
\newcommand{\Q}{\mathbb{Q}}
\newcommand{\A}{\mathbb{A}}
\newcommand{\R}{\mathbb{R}}
\newcommand{\Z}{\mathbb{Z}}
\newcommand{\C}{\mathbb{C}}
\renewcommand{\H}{\mathbb{H}}
\newcommand{\E}{\mathbb{E}}
\newcommand{\F}{\mathbb{F}}
\renewcommand{\S}{\mathbb{S}}
\renewcommand{\P}{\mathbb{P}}
\newcommand{\RP}{\mathbb{R}\mathsf{P}}
\newcommand{\CP}{\mathbb{C}\mathsf{P}}
\newcommand{\Hs}{\mathbb{H}\mathbbm{s}}

\newcommand{\Hyp}{\mathbb{H}}

\newcommand{\ep}{\varepsilon}

\renewcommand{\tilde}{\widetilde}
\renewcommand{\bar}{\overline}
\renewcommand{\hat}{\widehat}

\newcommand{\mat}[1]{\begin{matrix} #1 \end{matrix}}

\newcommand{\pmat}[1]{\begin{pmatrix} #1 \end{pmatrix}}

\newcommand{\smat}[1]{\left( \begin{smallmatrix} #1 \end{smallmatrix}\right )}

\newcommand{\set}[1]{\left\lbrace #1\right\rbrace}
\renewcommand{\det}{\mathsf{det}}
\newcommand{\tr}{\mathsf{tr}}
\newcommand{\dd}[1]{\frac{\mathsf{d}}{\mathsf{d}\mathsf{#1}}}

\newcommand{\Isom}{\mathsf{Isom}}
\newcommand{\Euc}{\mathsf{Euc}}
\newcommand{\Heis}{\mathsf{Heis}}

\newcommand{\Diag}{\mathsf{Diag}}
\newcommand{\diag}{\mathsf{diag}}

\newcommand{\M}{\mathsf{M}}

\newcommand{\GL}{\mathsf{GL}}
\newcommand{\PGL}{\mathsf{PGL}}

\newcommand{\SL}{\mathsf{SL}}
\newcommand{\PSL}{\mathsf{PSL}}

\renewcommand{\O}{\mathsf{O}}
\newcommand{\PO}{\mathsf{PO}}
\newcommand{\SO}{\mathsf{SO}}
\newcommand{\PSO}{\mathsf{PSO}}

\newcommand{\U}{\mathsf{U}}

\newcommand{\SU}{\mathsf{SU}}
\newcommand{\PSU}{\mathsf{PSU}}

\newcommand{\Sym}{\mathsf{Sym}}
\newcommand{\PSym}{\mathsf{PSym}}
\newcommand{\Herm}{\mathsf{Herm}}
\newcommand{\SkHerm}{\mathsf{SkHerm}}
\newcommand{\St}{\mathsf{St}}
\newcommand{\USt}{\mathsf{USt}}
\newcommand{\Fam}{\mathsf{Fam}}
\newcommand{\PDiag}{\mathsf{PDiag}}

\newcommand{\euc}{\mathfrak{euc}}
\newcommand{\heis}{\mathfrak{heis}}

\newcommand{\gl}{\mathfrak{gl}}

\newcommand{\so}{\mathfrak{so}}

\renewcommand{\u}{\mathfrak{u}}

\newcommand{\su}{\mathfrak{su}}

\newcommand{\inject}{\hookrightarrow}
\newcommand{\surject}{\twoheadrightarrow}

\def\acts{\curvearrowright}

\newcommand{\cat}[1]{\mathsf{#1}} 

\DeclareMathOperator{\Hom}{Hom}
\newcommand{\Aut}{\mathsf{Aut}}
\newcommand{\End}{\mathsf{End}}
\newcommand{\dev}{\mathsf{dev}}
\newcommand{\hol}{\mathsf{hol}}

\newcommand{\Area}{\mathsf{Area}}
\newcommand{\Cl}{\mathfrak{C}}

\newcommand{\inv}{^{-1}}
\newcommand{\id}{\mathsf{id}}
\renewcommand{\fam}[1]{\mathcal{#1}}

\newcommand{\Gr}{\mathsf{Gr}}
\renewcommand{\span}{\mathsf{span}}
\newcommand{\opnorm}[1]{{\left\vert\kern-0.25ex\left\vert\kern-0.25ex\left\vert #1 
    \right\vert\kern-0.25ex\right\vert\kern-0.25ex\right\vert}}

\newcommand{\subg}{\leqslant}

\makeatletter
\newcommand{\tpitchfork}{%
  \vbox{
    \baselineskip\z@skip
    \lineskip-.52ex
    \lineskiplimit\maxdimen
    \m@th
    \ialign{##\crcr\hidewidth\smash{$-$}\hidewidth\crcr$\pitchfork$\crcr}
  }%
}
\makeatother

\newcommand{\timesd}{\!\times_{\!\scriptscriptstyle\Delta}\!}    
\newcommand{\piO}{\pi_{\!\scriptscriptstyle\mathcal{O}}}

\newcommand{\labelarrow}[1]{\stackrel{#1}{\to}}

\usepackage{chngcntr}
\usepackage{apptools}
\AtAppendix{\counterwithin{theorem}{chapter}}
\AtAppendix{\counterwithin{proposition}{chapter}}
\AtAppendix{\counterwithin{definition}{chapter}}

\declaretheoremstyle[
    spaceabove=3pt, 
    spacebelow=3pt, 
    headfont=\sffamily\bfseries, 
    bodyfont = \normalfont\itshape,
    postheadspace=0.5em, 
    headpunct={:}]{definitionstyle} 

 \declaretheoremstyle[
    spaceabove=3pt, 
    spacebelow=0pt, 
    headfont=\sffamily\bfseries, 
    bodyfont =\normalfont\itshape,
    postheadspace=0.5em, 
    headpunct={:}]{theoremstyle} 

\declaretheoremstyle[
    spaceabove=-3pt, 
    spacebelow=3pt, 
    headfont=\sffamily\bfseries, 
    bodyfont = \normalfont,
    postheadspace=0.5em, 
    headpunct={:}]{examplestyle} 

\declaretheoremstyle[
    spaceabove=-6pt, 
    spacebelow=3pt, 
    headfont=\sffamily\itshape, 
    bodyfont = \normalfont,
    postheadspace=0.5em, 
    headpunct={:}]{remarkstyle} 

\declaretheorem[name=Definition,style=definitionstyle,
]{definition}
\declaretheorem[numbered=no,
name=Definition,style=definitionstyle]{definition*}

\declaretheorem[name=Theorem,style=theoremstyle,
]{theorem}
\declaretheorem[sibling=theorem,style=theoremstyle]{lemma}
\declaretheorem[sibling=theorem,style=theoremstyle]{corollary}
\declaretheorem[numbered=no,
name=Theorem,style=theoremstyle]{theorem*}

\declaretheorem[numbered=no,
name=Corollary,style=theoremstyle]{corollary*}

\declaretheorem[name=Proposition,sibling=theorem, style=theoremstyle]{proposition}
\declaretheorem[numbered=no,
name=Proposition,style=theoremstyle]{proposition*}

\declaretheorem[name=Calculation, style=theoremstyle]{calculation}

\declaretheorem[style=examplestyle, name=Example]{example}
\declaretheorem[style=examplestyle,numbered=no,
name=Example]{example*}

\declaretheorem[name=Observation, style=examplestyle]{observation}

\declaretheorem[name=Remark,style=examplestyle,sibling=observation]{remark}
\declaretheorem[name=Remark,numbered=no,
style=examplestyle]{remark*}

\declaretheorem[numbered=no,
name=Exercise, style=examplestyle]{exercise*}

\declaretheorem[numbered=no,style=examplestyle, name=Question]{question*}

\declaretheorem[numbered=no,style=examplestyle, name=Claim]{claim}

\newcommand{\MyName}{Steve J. Trettel}
\newcommand{\ThesisTitle}{Families of Geometries, Real Algebras, and Transitions}
\newcommand{\MyAdvisor}{Darren Long}
\newcommand{\MyCommitteeA}{Daryl Cooper}
\newcommand{\MyCommitteeB}{Jon McCammond}
\newcommand{\GradMonth}{June}

\newcommand{\GradYear}{2019}

\newcommand\YUGE{\fontsize{36}{50}\selectfont}

\DeclareFontFamily{U}{mathc}{}
\DeclareFontShape{U}{mathc}{m}{it}%
{<->s*[1.03] mathc10}{}
\DeclareMathAlphabet{\mathcal}{U}{mathc}{m}{it}

\usepackage{amsfonts}
\usepackage[bbgreekl]{mathbbol}

\DeclareSymbolFontAlphabet{\mathbbm}{bbold}
\DeclareSymbolFontAlphabet{\mathbb}{AMSb}%

\usepackage{Style/LetterSetup}

\makeindex

\begin{document}

\frontmatter


\begin{titlingpage}

\null \vspace{\stretch{2}}

{\noindent\Huge\bfseries\scshape Families of}
\\[1.1\baselineskip]
{\noindent\YUGE\bfseries\scshape  GEOMETRIES,}
\\[0.8\baselineskip]
{\noindent\YUGE\bfseries\scshape  REAL ALGEBRAS,}
\\[1.1\baselineskip]
{\noindent\Huge\bfseries\scshape and Transitions}

\vspace{\stretch{1}}\null

\par
\vspace{0.5\textheight}
{\noindent\sffamily\scshape\bfseries Steve Trettel}\\[\baselineskip]

\end{titlingpage}

   \thispagestyle{empty}
    \null
    \vfill
    {
        \begin{center}
       {\scshape University of California \\ Santa Barbara }\\
        \vspace{0.35in}
        \textbf{\Large \sffamily\scshape \ThesisTitle} \\
        \vspace{0.35in}
        A dissertation submitted in partial satisfaction\\
        of the requirements for the degree \\ 
        {\scshape Doctor of Philosophy} \\ in { Mathematics} \\ by \\ {\scshape \bfseries \MyName}
                \end{center}
        \vspace{0.35in}
        Committee in charge: \\
       {
            \indent \indent {Professor}~{\MyAdvisor}, Chair \\
        }
        \indent \indent {Professor}~\MyCommitteeA\\
        \indent \indent {Professor}~\MyCommitteeB \\
        \begin{center} {\GradMonth}~{\GradYear} \end{center}
    }
    \vfill
    \clearpage

 \thispagestyle{plain}
    \null
    \vspace{2cm}
    {
        \begin{center}
        {\bfseries\sffamily\ThesisTitle}
        
        \vspace{2cm}
        
        \vspace{3cm}
        
       {\scshape Copyright} \GradYear\\
        by\\
        \MyName
        
        \end{center}
        
    }
    \vfill
    \null
    \newpage

 \thispagestyle{plain}
    \null
    \vspace{5cm}
    {
        \begin{center}
       {\Huge\bfseries\scshape Dedication}
        \end{center}
        
        \vspace{1cm}
        
       \begin{center}
To Mr. Dorner\\
Thank you for believing in me \\
and helping create ways to learn in high school.
\end{center}

    }
    \vfill
    \null
    \newpage

 \thispagestyle{plain}
    \null
    \vspace{2cm}
    {
        \begin{center}
       {\Huge\bfseries\scshape Acknowledgements}
        \end{center}
        
        \vspace{1cm}

My time at UC Santa Barbara for graduate school has been an amazing learning experience, in a large part thanks to my advisor, Darren Long.
Thanks Darren, for your continued patience as I attempted to articulate my ideas, and for allowing me to try things on my own for weeks at a time.
While simultaneously being chair of the department and having five graduate students, Darren always made time for me in his crazy schedule; I hope to have that level of organization some day.
I am also incredibly grateful to Daryl Cooper and Jeff Danciger, for the many mathematical conversations we have had these past years.  
Thank you for sharing your knowledge, helping me clean up my ideas, guiding me in my writing, and all the other little things that have helped me grow as a mathematician.

I have benefitted immensely over the years from the kindness and generosity of many in the geometric topology community.
Thanks to Jon McCammond, Jesse Wolfson, Sara Maloni, Jason Manning, and Kelly Delp for entertaining my ideas, sharing your insight, discussing mathematics for hours on end, and providing me opportunities to speak and meet others.
Thanks also to Benson Farb, Dan Margalit, and Ken Millet not only for sharing your mathematical knowledge, but also good advice more generally (including my first experience with the job market).

The past six years at Santa Barbara were productive and memorable in a large part because of my wonderful cohort and the larger graduate community.
To Joe, Kate, Nancy, Kyle \& Jay - thanks for being my first friends in California and working with me on homework/quals until we were delirious.
To Gordon, Michael, Nic, Nadir and Tom - I've learned, created, and been introduced to more mathematics because of conversations we've had than likely any other source.
I look forward to our future as grownups together.
Christian and Joey, I will miss all the times we accidentally got distracted for an hour or three learning cool things that sprung from some simple question.
And thanks once more to Nancy, and also Wade and Sheri; our get-a-job-support-group helped me through my hardest time in graduate school.
Thanks to the other friends I've made during my time in Santa Barbara for working around my crazy schedule; and to Zizzos for all the coffee and beer while I was hiding and typing this.

    }
    \vfill
    \null
    \newpage

 \thispagestyle{plain}
    \null
    \vspace{2cm}
    {
   
        \begin{center}
		{\Huge\bfseries\scshape Abstract}
		
        \vspace{1cm}
        {\large \bfseries \ThesisTitle}\\
        by\\
        {\scshape \MyName}
        \end{center}
        
        \vspace{1cm}
        
       This thesis details the results of four interrelated projects completed during my time as a graduate student at University of California, Santa Barbara.
The first of these presents a new proof of the theorem of Cooper, Danciger and Wienhard classifying the limits under conjugacy of the orthogonal groups in $\GL(n;\R)$.
The second provides a detailed investigation into \emph{Heisenberg geometry}, which is the maximally degenerate such limit in dimension two.

The remaining two projects concern understanding geometric transitions which do not occur naturally as limits under conjugacy in some ambient geometry.
The third project describes a new degeneration of complex hyperbolic space, formed by degenerating the complex numbers as a real algebra, into the algebra $\R\oplus\R$.
Inspired by this example, the final project attempts to build the beginnings of a framework for studying transitions between geometries abstractly.
As a first application of this, we generalize the previous result and describe a collection of new geometric transitions, defined by constructing analogs of familiar geometries (projective geometry, hyperbolic geometry, etc) over real algebras, and then allowing this algebra to vary.
    }
    \vfill
    \null
    \newpage

\listoffigures*
\newpage

\tableofcontents*

\mainmatter

\chapter*{Summary of Results}
\addcontentsline{toc}{chapter}{Summary of Results}
\label{chp:Introduction}

This thesis is a combination of four projects, all connected to the theory of transitional geometry in geometric topology.
Thurston's Geometrization Conjecture 
 placed the study of geometric structures on manifolds at the heart of low dimensional topology.
 The deformation spaces of such structures are intimately related to representation varieties via the Ehresmann-Thurston principle 
 .
In particular, this connection has inspired higher Teichm\"uller theory 
 and the growing area of convex projective structures influenced by Goldman \& Choi \cite{Goldman90,Choi05}, Benoist  \cite{Benoist1,Benoist2,Benoist3,Benoist4}, 
Ballas \& Danciger \cite{BallasBending,Ballas17} and others.
More extreme deformations, which connect different \emph{kinds} of geometric structures are the subject of \emph{transitional geometry}.
A geometric transition is a continuous path $(G_t, X_t)$ of geometries where the isomorphism type is discontinuous in $t$.
  The example that inspires the theory is the continuous family of simply connected model spaces $\mathbb{M}_\kappa$ of constant curvature $\kappa$, which are isomorphic to the hyperbolic space for $\kappa<0$ and the sphere for $\kappa>0$, transitioning through Euclidean space at $\kappa=0$. 

Transitional geometries provide a means to \say{save} geometric structures from collapse - often a collapsing path of geometric structures can be rescaled to converge to a geometric structure of a different type.
Hodgson \cite{Hodgson86} and Porti 
\cite{Porti98} analyze Euclidean limits resulting from hyperbolic conemanifolds collapsing to a point, which plays an important role in the Orbifold Theorem of Cooper, Hodgson, \& Kerckhoff \cite{CooperHK00} and Boileau, Leeb \& Porti \cite{PortiLB05}
generalizing geometrization to certain singular spaces.
Further work of Porti studies the nonuniform collapse of hyperbolic structures to Nil \cite{Porti03} and Sol \cite{PortiHS01}.
Collapsing structures may even have non-Riemannian regenerations, as the transition from hyperbolic to Anti de Sitter space discovered by Danciger \cite{Danciger11,Danciger11Ideal,Danciger13}.

Geometric transitions arise naturally in Riemannian geometry and physics. 
The work of Umehara and Yamada study constant mean curvature tori along the $\Hyp^3\leftrightarrow\S^3$ transition \cite{UmeharaY92}, and Morabito analyzes minimal surfaces along the $\Hyp^2\times\R\leftrightarrow \S^2\times\R$ transition \cite{Morabito11}.
In Lorentzian geometry, transitions give means of realizing the Galilean group as the $c\to\infty$ limit of special relativity \cite{ChoK14}.
Other transitions arise in physics 
, including connections to $\mathsf{AdS}$ geometry and supergravity \cite{Concha17,Rajabov18,Gradechi92}.
There are even applications to classical geometry; Danciger, Maloni, and Schlenker \cite{DancigerMS14} used Half Pipe geometry to classify the polyhedra which inscribe in a quadric.

\section*{Structure of Thesis}
Structurally, this thesis is composed of three parts.
The first part contains the necessary background material, including a brief review of orbifolds, homogeneous geometries, geometric structures, and their deformation / moduli spaces.
The second part contains results pertaining to limits of groups / geometries / geometric structures, which can be modeled within some ambient Lie group / homogeneous space.
The third part contains results pertaining to limits of groups / geometries which are not modeled within some ambient Lie group, but instead as \emph{continuous families}, in a formalism inspired by algebraic geometry.

Of the four projects contained in this thesis, three of them are detailed in Part II.
In \emph{The Space of Orthogonal Groups}, a re-proof of the classification of limits of $\SO(p,q)$ in $\SL(p+q;\R)$ is given, independent of the original argument of Cooper, Danciger and Wienhard in \cite{CooperDW14}.
This classification reveals that the degenerations of the constant curvature geometries in $\RP^n$ form a poset under the relation 'is a degeneration of,' with the most degenerate limit given by the projective action of upper triangular unipotent matrices on an affine patch.
The following chapter, \emph{The Heisenberg Plane}, investigates in detail the two-dimensional case of this geometry.
The classification of compact 2-orbifolds admitting Heisenberg structures is given, and their deformation spaces are computed.
The regeneration of Heisenberg tori as constant-curvature cone tori is investigated, and we classify precisely which Heisenberg tori regenerate.
The final two chapters of Part II concern a new transition of complex hyperbolic space.
Inspired by Danciger's description of the boundaries of $\Hyp^3$, $\mathsf{HP}^3$ and $\mathsf{AdS}^3$ in \cite{Danciger11Ideal}, using the algebras $\C$, $\R[\ep]/\ep^2$, and $\R\oplus\R$, we study the analogs of hyperbolic space over these algebras, in addition to $\Hyp_\C^n$ defining the geometries $\Hyp_{\R_\ep}^n$ and $\Hyp_{\R\oplus\R}^n$.
We show by two different arguments that $\Hyp_\C^n$ transitions to $\Hyp_{\R\oplus\R}^n$ through $\Hyp_{\R_\ep}^n$, and investigate an interesting connection between $\Hyp_{\R\oplus\R}^n$ and real projective space.

Part III concerns the final project, which aims to provide a framework for discussing transitions between homogeneous spaces that does not rely on any ambient homogeneous space / Lie group.
Taking inspiration from the theory of Lie groupoids we introduce the notion of a \emph{family of geometries}, and lay the very basic groundwork of a theory of such families, mimicking to the extent possible the foundational observations in the classical theory of geometries in the sense of Klein.
We then use this framework to uncover a connection between families of real algebras and families of generalized unitary geometries, generalizing the construction of Chapter \ref{chp:HC_To_HRR_Transition} degenerating $\Hyp_\C^n$. 
The following four sections summarize the main results of each of these projects in more detail.

\section*{The Space of Orthogonal Groups}

In their 2014 paper \emph{Limits of Geometries}, Cooper, Danciger, and Wienhard showed that all conjugacy limits of $\SO(p,q)$ in $\SL(p+q;\R)$ may be described as \emph{isometry groups of partial flags of quadratic forms}.
The space of all paths $A_t\in\GL(n;\R)$ with which one may attempt to take conjugacy limits $\lim A_t\SO(p,q)A_t\inv$ is infinite dimensional, and their original argument completes the classification by using the theory of affine symmetric spaces to show that in fact it suffices to check a finite dimensional space of paths, in order recover all limits up to conjugacy.

This project presents an alternative argument producing the same classification but from a different perspective; replacing the difficulty of computing with the space of all paths with the difficulty of computing a closure in the space of closed subgroups.
Every conjugacy limit of $\SO(p,q)$ arises, up to conjugacy, as a limit $D_t\SO(p,q)D_t\inv$ for $D_t$ diagonal.
As diagonal conjugates of $\SO(p,q)$ are isometry groups of diagonal quadratic forms, we are interested in the set $\fam{D}_n$ of subgroups of $\SL(n;\R)$, defined below.

\begin{definition*}
The collection $\fam{O}_n$ of orthogonal groups in $\GL(n;\R)$ is the following 
$\fam{O}_n=\{\O(J)\in\Cl(\GL(n;\R))\mid J=J^T,\det(J)\neq 0\}.$

\noindent
The subcollection $\fam{D}_n\subset\fam{O}_n$ of orthogonal groups preserving a quadratic form diagonal in the standard basis is
$\fam{D}_n=\{\O(D)\in\fam{O}_n\mid D\textrm{ is diagonal}\}.$
\end{definition*}

\begin{figure}
\centering
\includegraphics[width=0.5\textwidth]{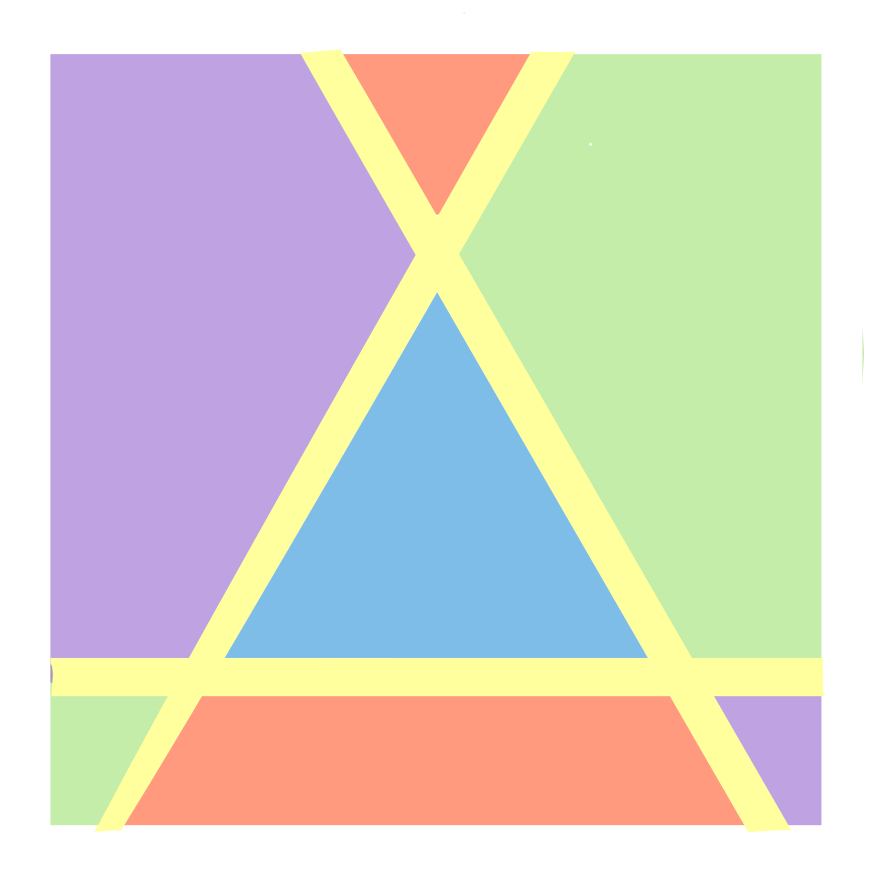}	
\caption{The space $\fam{D}_3$ is isomorphic to the coordinate hyperplane complement in $\RP^2$; the projectivization of the 2-cells of the octahedron.
Thus $\fam{D}_2$ is the disjoint union of four triangles, one containing diagonal conjugates of $\O(3)$ and the others parameterzing diagonal conjugates of $\O(\diag(1,1,-1))$, $\O(\diag(1,-1,1))$ and $\O(\diag(-1,1,1))$.}
\end{figure}

\noindent
The space $\overline{\fam{D}_n}$ contains all degenerations of diagonal orthogonal groups in $\GL(n;\R)$, and hence by the previous observation all degenerations of orthogonal groups up to conjugacy.
$\fam{D}_n$ is homeomorphic to the projectivized coordinate hyperplane arrangement in $\RP^{n-1}$, and the closure $\overline{\fam{D}_n}$ is equipped with a natural map $\pi\colon\overline{\fam{D}_n}\to\RP^{n-1}$, sending each orthogonal group $\O(J)$ to $[J]$, and each degeneration $L=\lim \O(J_t)$ to the limit of the associated forms $[B]=\lim [J_t]$.
A first coarse description of $\overline{\fam{D}_n}$ can be recovered from studying the fibers of $\pi$, which gives an inductive description of $\overline{\fam{D}_n}$, and a method of computing a natural cellulation of $\overline{\fam{D}_n}$ from the cellulation of $\RP^{n-1}$ by coordinate hyperplanes and the cellulation of the $\fam{D}_m$ for $m<n$.

\begin{theorem*}
The fiber of $\overline{\fam{D}_n}\to\RP^{n-1}$ above a point $p\in\RP^{n-1}$ lying in a $k$-dimensional cell of the coordinate hyperplane arrangement is homeomorphic to $\overline{\fam{D}_{n-k}}$.
\end{theorem*}

\noindent
This allows us to deduce the projection $\overline{\fam{D}_3}\to\RP^2$ is $1$ to $1$ away from three points, and the preimage of each of those points is homeomorphic to $\overline{\fam{D}_2}$, which is easily shown to be a circle.
This suggests the topology of $\overline{\fam{D}_3}$ is potentially the blowup of $\RP^2$ at three points, which is confirmed after some work recasting the problem in an algebro-geometric framework.

\begin{theorem*}[The Space of Orthogonal Groups]
$\overline{\fam{D}_n}$ is homeomorphic to the maximal De Concini-Procesi wonderful compactification of the coordinate hyperplane arrangement in $\RP^{n-1}$.	
\end{theorem*}

\noindent
Reference material on this compactification is included in Chapter \ref{chp:Orthogonal_Groups}, and additionally in \cite{DavisJS98}.
This realization as an iterated blowup implies some useful corollaries about the space $\overline{\fam{D}_n}$:

\begin{corollary*}
$\overline{\fam{D}_n}$ is a connected smooth manifold for all $n$.
The top dimensional open simplices of the coordinate hyperplane arrangement in $\RP^{n-1}$ lift homeomorphically to the blowup, with boundary in $\overline{\fam{D}_n}$ isomorphic to the $n-2$ dimensional permutohedron.
\end{corollary*}

\noindent
Of particular interest are the low dimensional cases $\overline{\fam{D}_3}$ and $\overline{\fam{D}_4}$, which parameterize the limits of pseudo-Riemannian subgeometries of $\RP^2$ and $\RP^3$, respectively.

\begin{example*}
The closure $\overline{\fam{D}_3}\subset\Cl(\GL(3;\R))$ is homeomorphic to the blow up of $\RP^2$ at three points; equivalently the connect sum of four copies of $\RP^2$.
\end{example*}

\begin{example*}
The closure $\overline{\fam{D}_4}\subset\Cl(\GL(4;\R))$ is a 3-manifold cellulated by 8 permutohedra.
\end{example*}

\noindent
In \cite{CooperDW14} it is shown that the limits of $\SO(p,q)$ in $\SL(p+q;\R)$ form a poset under the operation \emph{is a limit of}: which in this construction can be read directly off of the cellulation of $\fam{D}_n$.
In particular, a group parameterized by a point in a cell $C_1$ is a conjugacy limit of a group in a cell $C_2$ by diagonal conjugacy if and only if the cell $C_2$ lies in the boundary of $C_1$.
Up to isomorphism there is a unique \emph{most degenerate geometry} in each dimension, represented by the vertices in the cellulation of $\overline{\fam{D}_n}$.
This geometry has automorphisms the unipotent group of upper triangular matrices in $\GL(n;\R)$, and is called \emph{$n$-dimensional Heisenberg geometry} in this thesis as for $n=3$ the isometries are the real Heisenberg group.

\begin{figure}
\centering
\includegraphics[width=\textwidth]{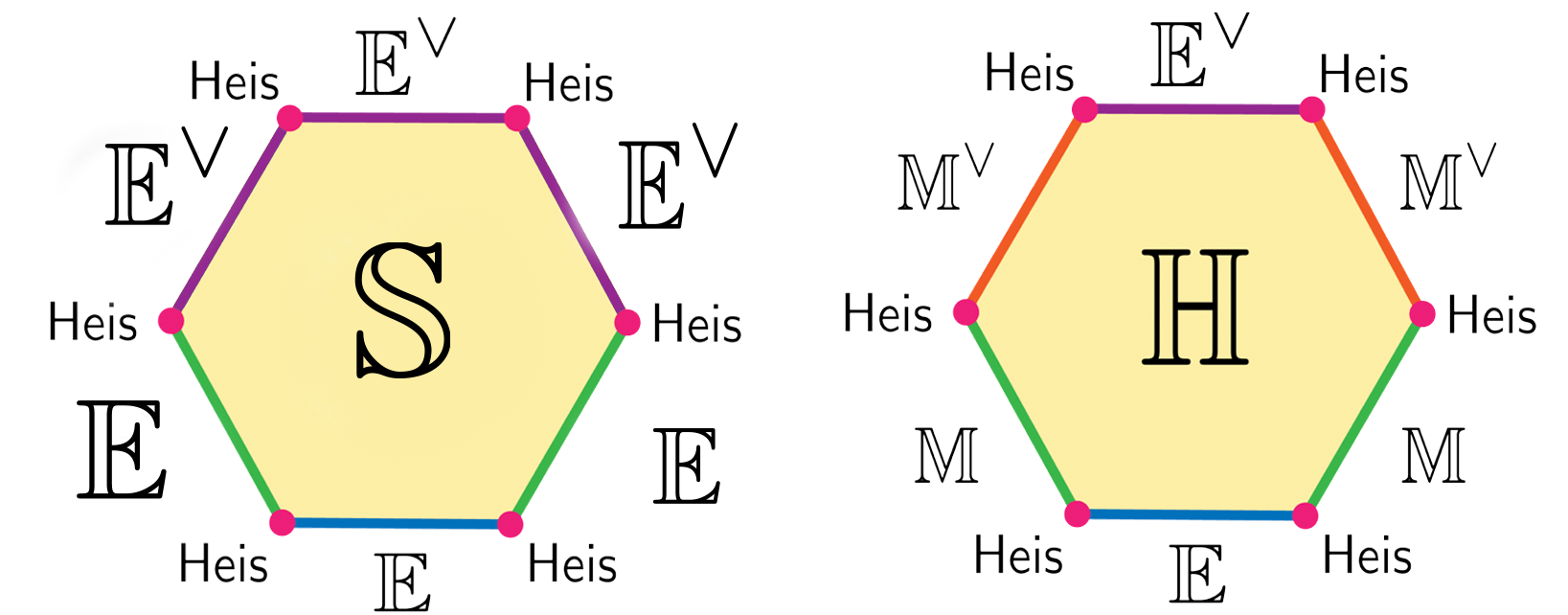}
\caption{The limits of orthogonal groups in $\GL(3;\R)$, parameterized by the closure of simplices containing conjugates of $\O(3)$ and $\O(2,1)$.	}
\end{figure}

\section*{The Heisenberg Plane}

This project provides an in depth study of the degenerate Heisenberg plane
considering the deformation and regeneration of Heisenberg structures on orbifolds. 
In particular, the closed orbifolds admitting Heisenberg structures are classified, and their deformation spaces are computed. Considering the regeneration problem, which Heisenberg tori arise as rescaled limits of collapsing paths of constant curvature cone tori is completely determined in the case of a single cone point.

\begin{definition*}
Heisenberg geometry is the $(G,X)$ geometry $\Hs^2:=(\Heis,\A^2)$ where
$$\Heis=\left\{\pmat{\pm1 & a & c\\0 &\pm1& b\\0&0&1}\;\Bigg | \; a,b,c\in\R\right\} \hspace{0.2cm}\textrm{and}
\hspace{0.2cm}
\A^2=\left\{ [x:y:1]\in\RP^2\mid x,y,\in\R\right\}.$$
\end{definition*}

\noindent
As a subgeometry of the affine plane, every Heisenberg structure on an orbifold $\mathcal{O}$ canonically weakens to an affine structure, which provides strong restrictions on which orbifolds can possibly admit Heisenberg structures.

\begin{theorem*}
Every closed Heisenberg 	orbifold is finitely covered by a Heisenberg torus with holonomy into the identity component of the isometry group $\Heis_0<\Heis$.
\end{theorem*}

\noindent
To classify tori with holonomy into $\Heis_0$ we compute the representation variety $\Hom(\Z^2,\Heis_0)$.
In the interest of computing the deformation space, we are particularly interested in the quotients of $\mathcal{R}$ by homothety and Heisenberg conjugacy.

\begin{proposition*}
$\Hom(\Z^2,\Heis_0)$ is isomorphic to $V(x_1y_2-x_2y_1)\times\R^2$. 
\end{proposition*}

\begin{theorem*}[Heisenberg $\Z^2$ Conjugacy Variety]
The quotient space of representations $\Z^2\to\Heis_0$ with image not into the center, up to homothety and conjugacy, is isomorphic to the following variety.

$$\mathcal{U}^\star=
V\left(\mat{
\|x\|^2+\|y\|^2=1,& \vec{z}\cdot\vec{x}=0\\
x_1y_2-x_2y_1=0,& \vec{z}\cdot\vec{y}=0
}\right)\subset\R^6$$

\noindent
This is a line bundle over $T^2$, twisted over each generator of $\pi_1(T^2)$.
\end{theorem*}

\noindent
This description of the conjugacy variety (after removing the singular collection of representations into the center) allows us to construct all Heisenberg structures on the torus by actually constructing a developing map for each representation, or proving that no developing map exists.

\begin{theorem*}[Teichm\"uller Space of Heisenberg Tori]
The subset $\mathcal{F}\subset\mathcal{U}^\star$ of conjugacy classes which are the holonomies of Heisenberg tori
 is a trivial $\R^\times$ bundle over the cylinder $\mathsf{Cyl}=T^2\smallsetminus S$, for $S$ the circle defined by the intersection of  $T^2=V(x_1y_2-x_2y_1)\cap\S^3$ with the plane $V(y_1,y_2)$.
The projection onto holonomy identifies the Teichm\"uller space of unit area tori with the quotient of $\mathcal{F}$ by the free $\Z_2$ action of conjugacy by $\diag(-1,-1,1)$ and $\mathcal{T}_{\Hs^2}(T^2)\cong \mathcal{F}/\Z^2\cong \R^2\times \S^1$.
\end{theorem*}

\begin{figure}
\centering
\includegraphics[width=0.75\textwidth]{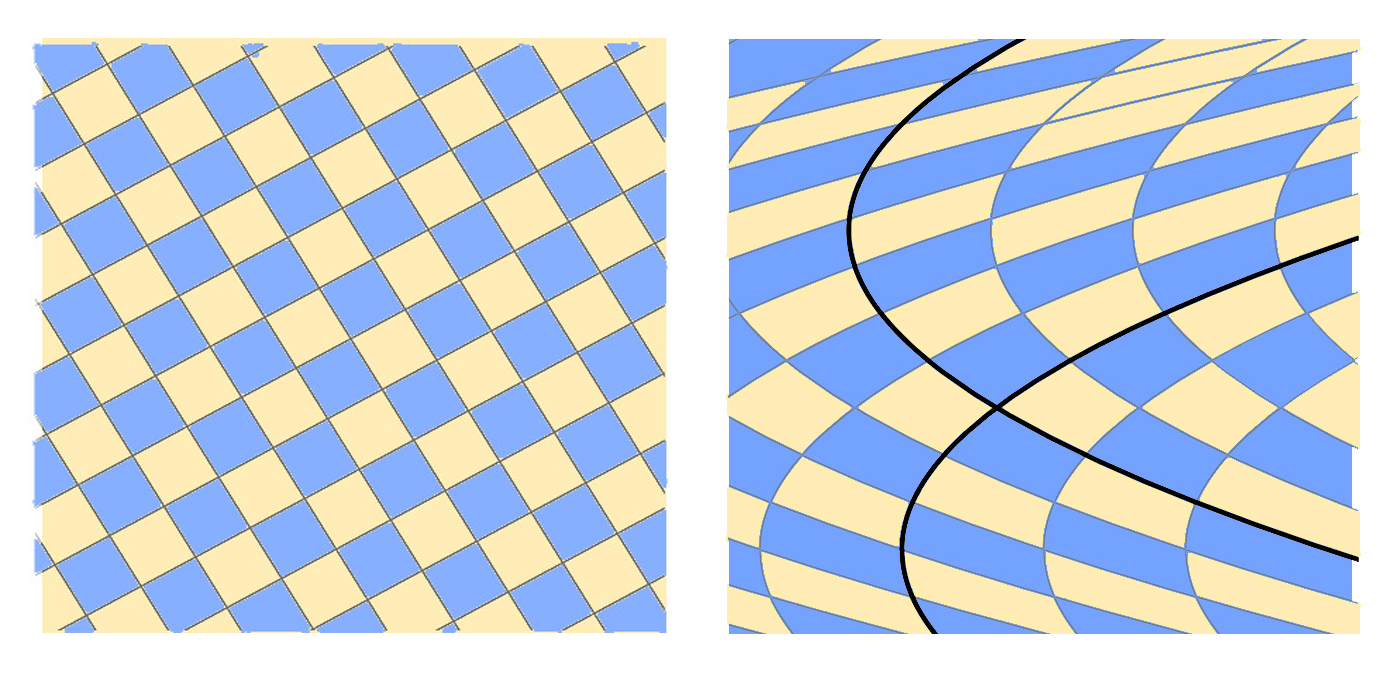}
\caption{The developing map for a Heisenberg translation torus (left) and a shear torus (right).}	
\end{figure}

\noindent
Examining the deformation space $\fam{F}/\Z_2$ reveals that there are essentially two different \emph{kinds} of Heisenberg structures on the torus: \emph{translation tori}, whose holonomy has image strictly contained in the subgroup of $\Heis$ acting by translations on the plane, and \emph{shear tori}, whose holonomy contains a nontrivial shear.

As every Heisenberg orbifold is finitely covered by one of the Heisenberg tori described in the previous theorem, points of the deformation spaces $\mathcal{D}_{\Hs^2}(\fam{O})$ may be parameterized by \emph{extensions} of holonomies $\rho\colon\pi_1(T^2)\to\Heis$ to $\pi_1(\fam{O})>\pi_1(T^2)$.

\begin{theorem*}[Classification of Heisenberg Orbifolds]
All Heisenberg structures on orbifolds are complete, and projection onto the holonomy is an embedding $\mathcal{D}_{\Hs^2}(\mathcal{O})\inject \Hom(\pi_1(\mathcal{O}),\Heis)/\Heis_+$.

The orbifolds admitting Heisenberg structures and their Teichm\"uller spaces are given by the following table:
	
\begin{table}[h]
\centering
\begin{tabular}{ccc}
\toprule
$\mathcal{O}$ &\hspace{1cm}& $\mathcal{T}_{\Hs^2}(\mathcal{O})$ \\
\midrule
$\S^1\times\S^1$ &\hspace{1cm}& $\R^2\times\S^1$ 
\\
$\S^1\widetilde{\times}\S^1$, $\S^1\times I$, $\S^1\widetilde{\times}I$ &\hspace{1cm}& $\R^2\sqcup\R$  
\\
$\S^2(2,2,2,2)$ &\hspace{1cm}& $\R\times\S^1$ 
\\
$\D^2(2,2;\varnothing),\; \D^2(\varnothing; 2,2,2,2),\;\RP^2(2,2)$ &\hspace{1cm}& $\R\sqcup\R$
\\
$\D^2(2;2,2)$ &\hspace{1cm}& $\R\sqcup\R$
\\
\bottomrule
\end{tabular}
\end{table}
\end{theorem*}

\noindent
The second half of this project studies the regeneration of Heisenberg structures, restricting for convenience to Heisenberg tori.
As in many cases considering regenerations of limit geometric structures, conemanifold structures are the important objects to consider.
In particular, we search for collapsing sequences of constant curvature cone tori, which when viewed as projective structures, converge to a Heisenberg torus in the limit.
Restricting to the case of a single cone point, we may represent a constant curvature cone torus as a constant curvature geodesic parallelogram with side identifications.
A collection of arguments in projective geometry then allow us to completely understand the regenerations of Heisenberg tori whose holonomy acts by pure translations.

\begin{theorem*}[Regeneration of Translation Tori]
Let $\mathbb{X}\in\{\S^2,\E^2,\Hyp^2\}$ and $\mathbb{X}_t=D_t.\mathbb{X}$ be a sequence of diagonal conjugates converging to $\Hs^2$.
Given any translation torus $T$ there is a sequence of $\mathbb{X}_t$ cone tori with at most one cone point converging to $T$.	
\end{theorem*}

\noindent
Translation tori form a codimension-1 subset of the deformation space, with the rest being shear tori.
In fact \emph{no shear tori regenerate} as constant curvature cone tori with a single cone point, and the argument showing this nonexistence uses a particularly geometric characterization of shear tori.

\begin{theorem*}
A Heisenberg torus $T$ has a shear in its holonomy if and only if all simple geodesics on $T$ are pairwise disjoint.
\end{theorem*}

Hyperbolic, spherical and Euclidean (cone) tori behave quite differently than this; every pair of generators of $\pi_1$ has intersecting geodesic representatives.
Thus, to show that shear tori do not regenerate, it suffices to see that \emph{any limit} of constant curvature cone tori has intersecting geodesics, and thus is a translation torus.

\begin{theorem*}[Non-regeneration of Shear Tori]
Let $\mathbb{X}\in\{\S^2,\E^2,\H^2\}$ and $\mathbb{X}_t=D_t \mathbb{X}$ a sequence of conjugate geometries converging to the Heisenberg plane.
Let $T_t$ be a sequence of $\mathbb{X}_t$ cone tori with at most one cone point converging to some Heisenberg torus $T$.
Then $T$ has a pair of intersecting geodesics.
\end{theorem*}

\section*{A transition of Complex Hyperbolic Space}

This next project concerns the construction of a new transition of geometries beginning with complex hyperbolic space, and degenerating the geometric structure of $\Hyp_\C^n$ by degenerating the algebraic structure of $\C$.
The results of this project are spread over the final two chapters of Part II, as the work divides neatly into \emph{constructing generalized Hyperbolic geometries} and \emph{proving these geometries form a geometric transition}.

\subsection*{Hyperbolic Geometry Over Algebras}

In the first of these chapters, we generalize the usual definition of complex hyperbolic space, to hyperbolic space defined over a real algebra with involution, and focus on the simplest, two dimensional examples $\C$, $\R_\ep=\R[\ep]/\ep^2$ and $\R\oplus\R$.
Let $q=x_1\overline{x_1}+x_2\overline{x_2}+\cdots+x_n\overline{x_n}-x_{n+1}\overline{x_{n+1}}$, recall that we may define a  model of complex hyperbolic space in $\CP^n$ as follows.

\begin{definition*}
Complex Hyperbolic space is the geometry given by the action of $\U(n,1;\C)$ on the projectivized unit sphere of radius $-1$ for $q$ in $\C^{n+1}$;
$\Hyp_\C^n=(\U(n,1;\C),\mathcal{S}_{\C}(n,1)/\U(\C))$.
\end{definition*}

\noindent
Each of $\R_\ep$, $\R\oplus\R$ can be equipped with the involutions, $a+\ep b\mapsto a-\ep b$ and $(a,b)\mapsto (b,a)$ respectively.
Interpreting the form $q$ with these involutions providing conjugation,
we mimic the construction of $\Hyp_\C^n$ as closely as possible, producing analogous unitary groups and spaces on which they act transitively.

\begin{definition*}
$\R_\ep$ Hyperbolic space is the geometry given by the action of $\U(n,1;\R_\ep)$ on the projectivized unit sphere of radius $-1$ for $q$ in $\R_\ep^{n+1}$;
$\Hyp_{\R_\ep}^n=(\U(n,1;\R_\ep),\mathcal{S}_{\R_\ep}(n,1)/\U(\R_\ep))$.
\end{definition*}

\begin{definition*}
$\R\oplus\R$ Hyperbolic space is the geometry given by the action of $\U(n,1;\R\oplus\R)$ on the projectivized unit sphere of radius $-1$ for $q$ in $\R\oplus\R^{n+1}$;
$\Hyp_{\R\oplus\R}^n=(\U(n,1;\R),\mathcal{S}_{\R\oplus\R}(n,1)/\U(\R\oplus\R))$.
\end{definition*}

\begin{figure}
\centering
\includegraphics[width=0.95\textwidth]{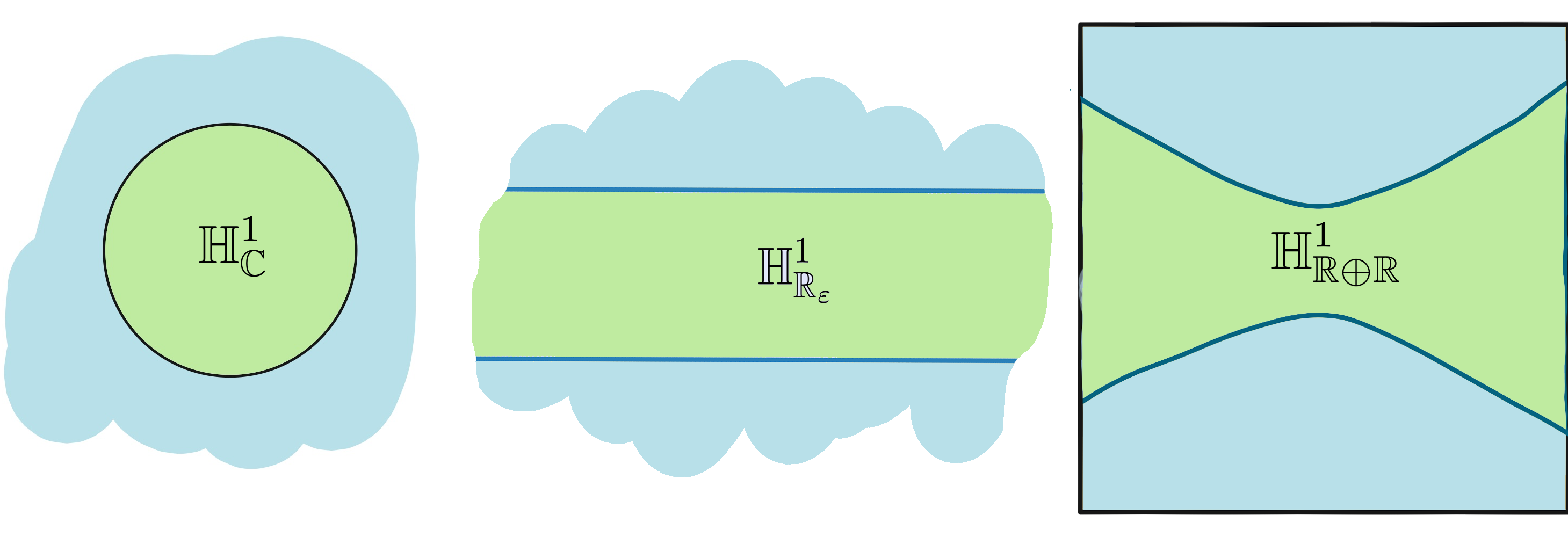}
\caption{The underlying spaces for $\Hyp_{\C}$, $\Hyp_{\R_\ep}$ and $\Hyp_{\R\oplus\R}$ in dimension 1.}	
\end{figure}

\noindent
Over $\R_\ep$, 
the domain of $\Hyp_{\R_\ep}^n$ is a product $\Hyp_\R^n\times\R^n$ in the affine patch $\R_\ep^n$.
The geometry $\Hyp_{\R_\ep}^n$ is not a product of the geometry of $\Hyp_\R^n$ with the geometry of $\R^n$, however 
the embedding $\R\inject\R_\ep$ induces an embedding $\Hyp_\R^n\inject\Hyp_{\R_\ep}^n$, with domain $\mathbb{B}^n\times\{0\}$ in $\H_{\R_\ep}^n=\mathbb{B}^n\times\R^n$.
This, together with some analysis of the automorphism group $\U(n,1;\R_\ep)$ gives the following.

\begin{theorem*}
The group homomorphism $\GL(n+1;\R_\ep)\to\GL(n+1;\R)$ dropping the imaginary part induces an 
 epimorphism of geometries $\Hyp_{\R_\ep}^n\surject \Hyp_\R^n$; thus $\Hyp_{\R_\ep}^n$ fibers over real hyperbolic space.
\end{theorem*}

\noindent
Over $\R\oplus\R$, the analog of hyperbolic space no longer fibers over $\Hyp_\R^n$, but much like $\Hyp_\C^n$ and above, contains $\Hyp_\R^n$ as a codimension $n$ subset, arising from the diagonal embedding $\R\inject\R\oplus\R$.
An investigation of the automorphism group of $\Hyp_{\R\oplus\R}^n$ suggests a possible connection to real projective geometry:

\begin{theorem*}
The group $\U(n,1;\R\oplus\R)$ is abstractly isomorphic to $\GL(n+1;\R)$, and $\SU(n,1;\R\oplus\R)\cong\SL(n+1;\R)$.
\end{theorem*}

\noindent
In fact, we are able to build a model of $\Hyp_{\R\oplus\R}^n$ as a subgeometry of $\RP^n\times\RP^n$, and think of the points of $\R\oplus\R$ hyperbolic space as given by the data of pairs of a point and a disjoint hyperplane in $\RP^n$.

\begin{definition*}
The \emph{point-hyperplane geometry} of $\RP^n$ has as underlying space the collection of all pairs $(H,p)$ of hyperplanes $H\subset \RP^n$ and points $p\in \RP^n$ such that $p\not\in H$.
The automorphisms of this geometry are the full automorphism group of $\RP^n$, acting by $A.(p,H)=(Ap,A^{-T}H)$ if $H$ is the projective covector representing the hyperplane as its kernel.
\end{definition*}

\begin{theorem*}[$\Hyp_{\R\oplus\R}^n$ and $\RP^n$]
Point-Hyperplane projective geometry is locally isomorphic hyperbolic geometry over $\R\oplus\R$.	
\end{theorem*}

\noindent
Just as complex hyperbolic space of dimension 1 is isomorphic to the hyperbolic plane ($\Hyp_\C^1\subset\CP^1$ is the Poincare Disk model), one dimensional $\R_\ep$ hyperbolic space is also an already known geometry: its \emph{half-pipe} 2-space!
Over $\R\oplus\R$, the generalization of hyperbolic space
$\Hyp_{\R\oplus\R}^1$ is the familiar de Sitter space of dimension two, which itself identifies with Anti de Sitter space $\mathsf{AdS}^2$ as a coincidence of low dimensionality.
Thus, in dimension one, the three geometries we have produced in this way coincide exactly with the geometries occurring in the transition studied by Danciger \cite{Danciger11}.
These exceptional isomorphisms fail to continue in any higher dimensions, but in the following chapter we show that there is a way to generalize Danciger's geometric transition, and produce a continuous collection of geometries connecting $\Hyp_\C^n$ to $\Hyp_{\R\oplus\R}^n$ through $\Hyp_{\R_\ep}^n$.

\subsection*{Producing a Transition of $\Hyp_\C$ to $\Hyp_{\R\oplus\R}$}

For each $\delta$, the algebra $\Lambda_\delta=\R[\lambda]/(\lambda^2=\delta)$ is a two dimensional algebra over $\R$, isomorphic to $\C$ for $\delta<0$, to $\R_\ep$ when $\delta=0$ and to $\R\oplus\R$ for $\delta>0$.
Following exactly the methods the previous chapter, it is easy to construct the analogs $\Hyp_{\Lambda_\delta}^n$ of hyperbolic geometry over the algebra $\Lambda_\delta$, and it is clear these are isomorphic to $\Hyp_\C^n$, $\Hyp_{\R_\ep}^n$ and $\Hyp_{\R\oplus\R}^n$ for $\delta<0,=0$, and $>0$ respectively.
The main difficulty is formalizing the \emph{continuity} of this collection, as the geometries involved do not all obviously embed in some ambient projective space.
We show the continuity of this path of geometries in two ways.

First, we try to follow as closely to the standard formalism of conjugacy limits as possible, while acknowledging the lack of an ambient geometry.
We consider a collection of matrix representations of the relevant algebras $\iota_\delta\colon\Lambda_\delta\to\M(2;\R)$, and use these to produce matrix representations of the automorphism groups $\Isom(\Hyp_{\Lambda_\delta}^n)=\SU(n,1;\Lambda_\delta)$ in $\M(2n;\R)$, which we also denote by $\iota_\delta$ for convenience.
As the data of a homogeneous space is captured by its automorphism group together with a stabilizing subgroup, we use the Chabauty topology on the closed subgroups of $\GL(2n;\R)$ as an ambient Lie group to study this transition.

\begin{theorem*}[Continuity of Unitary Groups]
Let $\U(n,1;\Lambda_\delta)$ be the unitary group of signature $(n,1)$ over $\Lambda_\delta$, and $\USt(n,1;\Lambda_\delta)=\smat{\U(n;\Lambda_\delta)&\\& U(\Lambda_\delta)}$ the point stabilizer of a point in $\Hyp_{\Lambda_\delta}^n$.
Then the maps $\R\to\Cl(\GL(2n;\R))$ defined by 
$$\delta\mapsto\iota_\delta(\U(n,1;\Lambda_\delta))
\hspace{1cm}
\delta\mapsto \iota_\delta(\USt(n,1;\Lambda_\delta))
$$
are continuous into the Chabauty space $\Cl(\GL(2n;\R))$.
\end{theorem*}

The main result of this chapter is an alternative approach to showing this collection of geometries varies continuously, inspired by the notion of a \emph{bundle of groups} in \cite{Morabito11} and \emph{families of groups} in algebraic geometry.
With a suitable definition of \emph{continuous family} of algebraic groups or Lie groups, one might hope to study the collection $\SU(n,1;\Lambda_\delta)$
as a 1-parameter family abstractly,
 without making use of the embeddings $\iota_\delta$ into $\GL(2n;\R)$.
We begin with the definition of a one parameter family of algebras.

\begin{definition*}
A one parameter family of algebras $\fam{A}$ is a real vector bundle $\fam{A}\to\R$ together with a section $\fam{1}\to \fam{A}$ selecting point $\fam{1}(\delta)$ for each vector space $\fam{A}_\delta$, and a smooth map $\mu\colon\fam{A}\times_\R\fam{A}\to\fam{A}$ such that for each $\delta\in\R$ the restriction $\mu_\delta\colon \fam{A}_\delta\times\fam{A}_\delta\to\fam{A}_\delta$ is the multiplication of a real algebra structure on $\fam{A}_\delta$ with identity $\fam{1}(\delta)$.
\end{definition*}

\noindent
The algebras $\Lambda_\delta$ form a 1-parameter family, which we denote $\Lambda_\R$ in what follows.
The matrix algebras $\M(n;\Lambda_\delta)$ also form a 1-parameter family.

\begin{figure}
\centering
\includegraphics[width=0.4\textwidth]{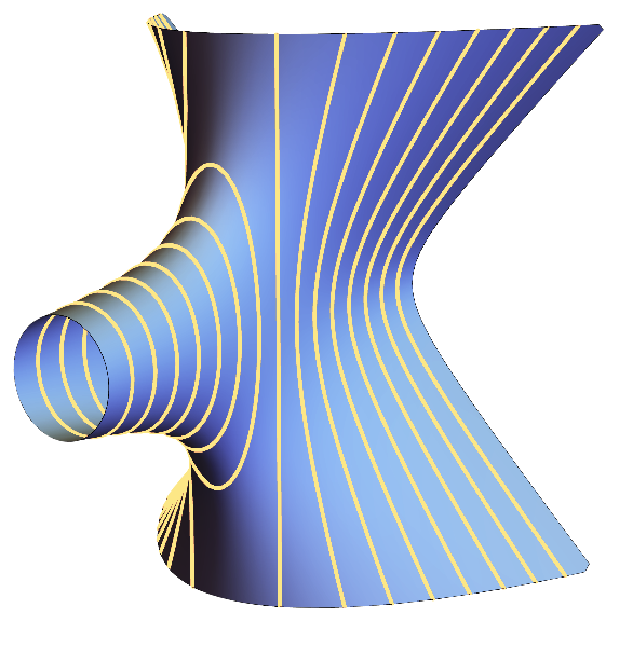}
\caption{The elements of norm $1$ in the algebras $\Lambda_\delta$ together form a 1-parameter family of groups (the vertical slices in the total space above).}	
\end{figure}

\begin{definition*}[1-Parameter Family]
A \emph{one parameter family of Lie groups} is a Lie groupoid $\fam{G}$ with $\mathsf{Ob}(\fam{G})=\R$ and equal source, target maps $s=t\colon G\to\R$.
The fibers $\fam{G}_\delta=s\inv(\delta)=t\inv(\delta)$ each come equipped with the structure of a Lie group, by restricting the composition operation of the groupoid $\fam{G}$.	
\end{definition*}

\noindent
This suggests a definition for continuity of subgroups of $\M(n;\Lambda_\delta)$.

\begin{definition*}
A collection $G_\delta<\GL(n;\Lambda_\delta)$ varies continuously if $\bigcup_\delta G_\delta\times\{\delta\}$ is a 1-parameter family of groups.
\end{definition*}

\noindent
We then set off to develop a particular set of tools to show that certain nice, algebraically defined subgroups of $\GL(n;\Lambda_\delta)$ fit together to form 1-parameter families.
This has many corollaries, such as below.

\begin{corollary*}[Continuity of Unitary Groups]
The collection $\fam{U}(n,1;\Lambda_\R)$ is a 1-parameter family of groups.
The collection $\fam{SU}(n,1;\Lambda_\R)$ is a 1-parameter family of groups.	
The collection of point stabilizers $\fam{USt}(n,1;\Lambda_\R)$ forms a 1-parameter family of groups.
\end{corollary*}

\noindent
Together, this implies that the homogeneous spaces of interest are given by a 1-parameter family of automorphism groups and a 1-parameter family of point stabilizers, which we take as the definition of an abstract, 1-parameter family of geometries.

\begin{theorem*}[Continuity of Hyperbolic Geometries]
The geometries $\Hyp_{\Lambda_\R}^n=(\fam{U}(n,1;\Lambda_\R),\fam{USt}(n,1;\Lambda_\R))$ form a 1-parameter family of geometries.	
\end{theorem*}

\section*{Families of Geometries}

The third part of this thesis deals with extending and fleshing out an abstract theory of continuity for collections of geometries, based on the \emph{1-parameter family of groups} introduced above.
This project is broken up over four chapters: the first introduces the relevant objects \emph{families of spaces} and \emph{families of groups}, the second provides means of constructing examples of these, as well as beginning the theory of \emph{families of geometries}.
The third is disjoint from this theory of continuity, and studies various notions of homogeneous spaces that can be constructed over finite dimensional algebras, generalizing projective spaces, as well as geometries with orthogonal and unitary groups of automorphisms.
The fourth chapter ties these threads together and produces a multitude of examples of families of geometries containing new geometric transitions, directly generalizing the techniques utilized to produce the family $\Hyp_{\Lambda_\R}^n$ as a 1-parameter family of geometries previously.

\subsection*{Families of Spaces, Groups}

Taking inspiration from the fields of complex geometry and algebraic geometry, we define a \emph{family of manifolds} as a bundle like construction.

\begin{definition*}[Family of Smooth Manifolds]
A smooth family of manifolds parameterized by a smooth manifold $\Delta$ is a triple $(\fam{X},\Delta,\pi)$ of smooth manifolds $\fam{X},\Delta$ equipped with a smooth submersion $\pi:\mathcal{X}\to\Delta$.
The space $\fam{X}$ is the \emph{total space} and $\Delta$ is the \emph{base} of the family.
The fibers $\fam{X}_\delta:=\pi\inv\set{\delta}$ are the \emph{members} of the family, and are said to vary smoothly over $\Delta$.
\end{definition*}

\begin{figure}
\centering
\includegraphics[width=0.6\textwidth]{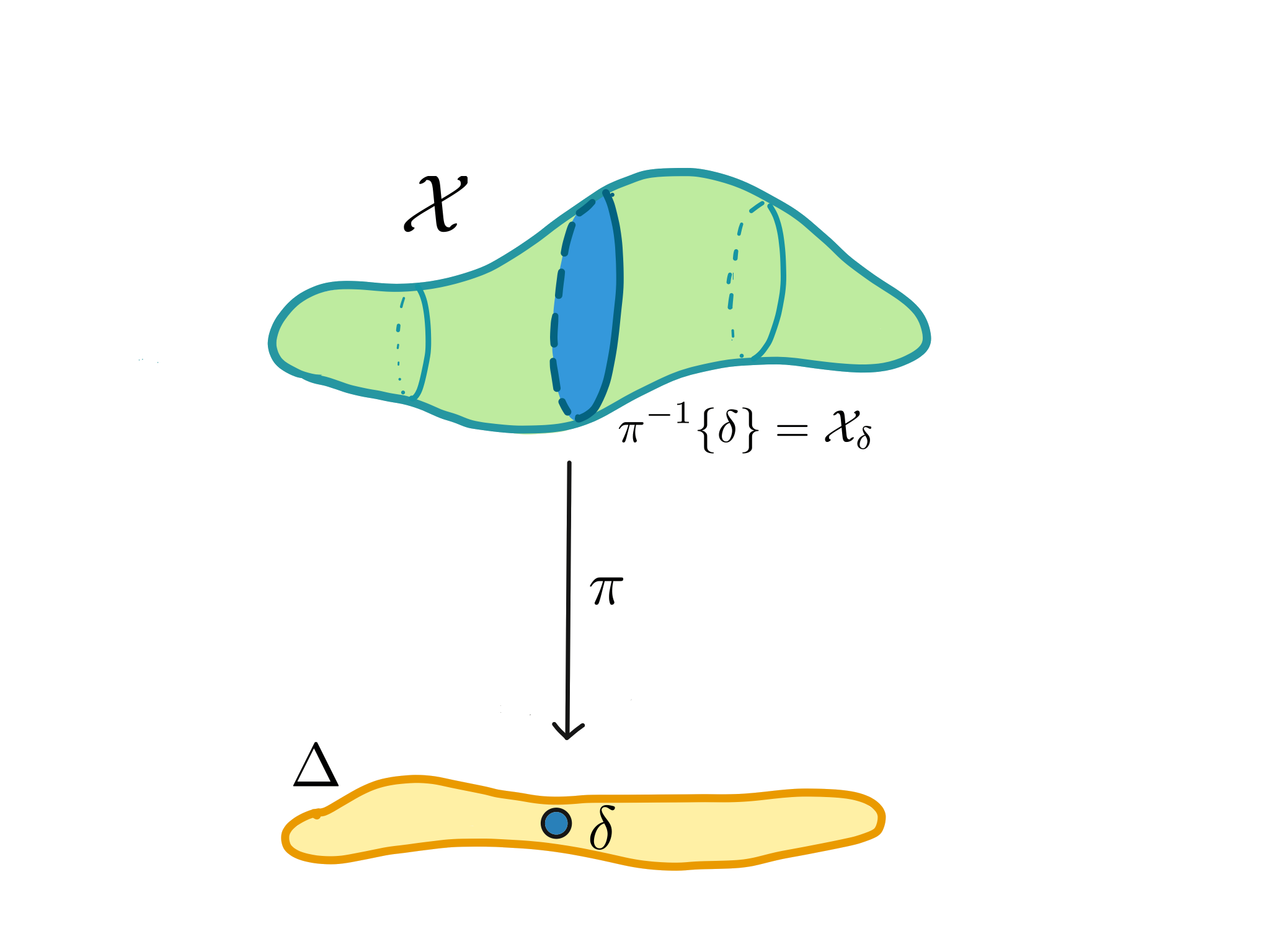}
\caption{A family of spaces is a generalized fiber bundle, consisting of a total space foliated by members varying over a base.}
\end{figure}

\noindent
A family contains a \emph{transition} if there are non-isomorphic members over a single connected component of the base. An object $X$ \emph{has transitions} if it is a member of a transitioning family.  Otherwise $X$ is \emph{rigid}.
Restricting to a fixed base space $\Delta$, we define the category of families.

\begin{definition}[The Category of Families]
The category $\Fam_\Delta$ has as objects all families $\pi_\fam{X}\colon\fam{X}\to\Delta$, with morphisms $\phi\in\Hom(\fam{X}\labelarrow{\pi_\fam{X}}\Delta, \fam{Y}\labelarrow{\pi_\fam{Y}}\Delta)$ given by maps $\phi\in C^\infty(\fam{X},\fam{Y})$ such that $\pi_\fam{X}=\pi_\fam{Y}\phi$.
\end{definition}

\noindent
This category has finite products and a terminal object, so we may speak of \emph{group} objects, and other \emph{algebraic objects } of the category $\Fam_\Delta$.

\begin{definition}[Family of Groups]
A family of groups over $\Delta$ is a group object in $\cat{Fam}_\Delta$.  
That is, a family $\fam{G}\to\Delta$ equipped with a global section $\fam{e}\colon\Delta\to\fam{G}$ and smooth maps $\mu\colon\fam{G}\timesd\fam{G}\to\fam{G}$, $\iota\colon\fam{G}\to\fam{G}$ equipping each fiber $\fam{G}_\delta$ with the structure of a Lie group with identity $\fam{e}(\delta)$ and multiplication, inversion restrictions of $\mu,\iota$ respectively.
\end{definition}

\begin{definition*}[Family of Algebras]
A family of  $\mathbb{F}$-algebras over $\Delta$ 
is an $\F$-algebra object in the category $\Fam_\Delta$.
It is given by the data of a $\mathbb{F}$-vector bundle $\fam{A}\to\Delta$ together with a multiplication $\mu\colon\fam{A}\timesd\fam{A}\to\fam{A}$ giving each fiber the structure of a $\mathbb{F}$-algebra.
\end{definition*}

\begin{definition*}[Family of Lie Algebras]
A family of Lie algebras $\fam{g}\to\Delta$ is a Lie algebra object in $\cat{Fam}_\Delta$.  That is	, it is a family of vector spaces equipped with a bilinear map $[\cdot,\cdot]\colon\fam{g}\timesd\fam{g}\to\fam{g}$ giving each fiber the structure of a Lie algebra.
\end{definition*}

\subsection*{Constructing Families}

This next chapter takes on the task of developing the bare bones of a theory of families, suitable at least to construct basic examples and define the object of interest, a \emph{family of geometries}.
As a geometry in the sense of Klein is given by a transitive group action of a Lie group on a smooth manifold, to define families of such objects we will need a notion of an \emph{action of a family of groups}.

\begin{definition*}[Action of Families]
An action of $\fam{G}$ on $\fam{X}$ in $\Fam_\Delta$ is given by a morphism $\alpha:\fam{G}\timesd\fam{X}\to\fam{X}$ denoted $\alpha(g,x)=g.x$ such that $\alpha(\fam{e},\cdot)=\id_{\fam{X}}$ and $g.(h.(-))=gh.(-)$ as maps $\fam{X}_\delta\to \fam{X}_\delta$, for all $g,h\in\fam{G}_\delta$.
\end{definition*}

\noindent
An action $\fam{G}\acts\fam{X}$ is \emph{proper} if the map $\fam{G}\timesd\fam{X}\to\fam{X}\timesd\fam{X}$ defined by $(g,x)\mapsto (x,g.x)$ is a proper map. 
Proper actions are important, as proper free actions are precisely those with well behaved quotients, as we show shortly.
We show that the action of a family of subgroups by translation is always proper, which underlies some foundational observations in the theory of families of geometries.

\begin{proposition*}[Translation is a Proper Free Action]
Let $\fam{G}\to\Delta$ be a family and $\fam{H}\subg\fam{G}$ a family of closed subgroups.  Then the action of $\fam{H}$ on $\fam{G}$ by translation is proper.
\end{proposition*}

\noindent
Before defining families of geometries, we cover three means of constructing new families from old: namely, pullbacks, exponentials, and quotients.
In addition to being a useful way to produce new families, the pullback construction also allows us to phrase other useful constructions, such as restrictions and subfamilies categorically.

\begin{definition*}[Pullbacks in the Category of Families]
Let $\fam{X}\to\Delta$ be a family, and $f\colon D\to\Delta$	 be a morphism.  
Then the pullback family $f^\star\fam{X}\to D$, if it exists, has total space $\fam{X}\timesd D=\set{(x,d)\mid f(d)=\pi(x)}$ and projection 
$f^\star\fam{X}=\fam{X}\timesd D\stackrel{\pi^\star}{\longrightarrow} D$
defined by $(x,d)\mapsto d$.
\end{definition*}

\begin{theorem*}[Existence of Pullback Families]
Pullbacks always exist along any smooth morphism  $D\labelarrow{f}\Delta$ and any such $f$ induces a functor $\Fam_\Delta\labelarrow{f^\star}\Fam_D$.
\end{theorem*}

\noindent
A potential means of producing a family of groups is to exponentiate a family of Lie algebras.
Abstractly this is fraught with difficulty, as apparent from the existing literature on Lie groupoids and weak Lie algebra bundles. 
We focus on a more narrow scope: when does a family of Lie \emph{subalgebras} exponentiate to a family of Lie \emph{subgroups}?
The theorem below specializes a slightly more general result, but is already sufficient to construct many transitioning families (such as the transitions between the $\SO(p,q)$ in $\SL(p+q;\R)$ mentioned previously).

\begin{theorem*}[Closed Exponentials are Families]
Let $\fam{H}\subset\fam{G}$ be a closed subset such that each fiber $\fam{H}_\delta$ is a connected group, and the Lie algebras $\fam{h}\to\Delta$ form a subfamily of $\fam{g}\to\Delta$.
Then $\fam{H}$ is a subfamily of $\fam{G}\to\Delta$.
\end{theorem*}

\begin{figure}
\centering 
\includegraphics[width=0.8\textwidth]{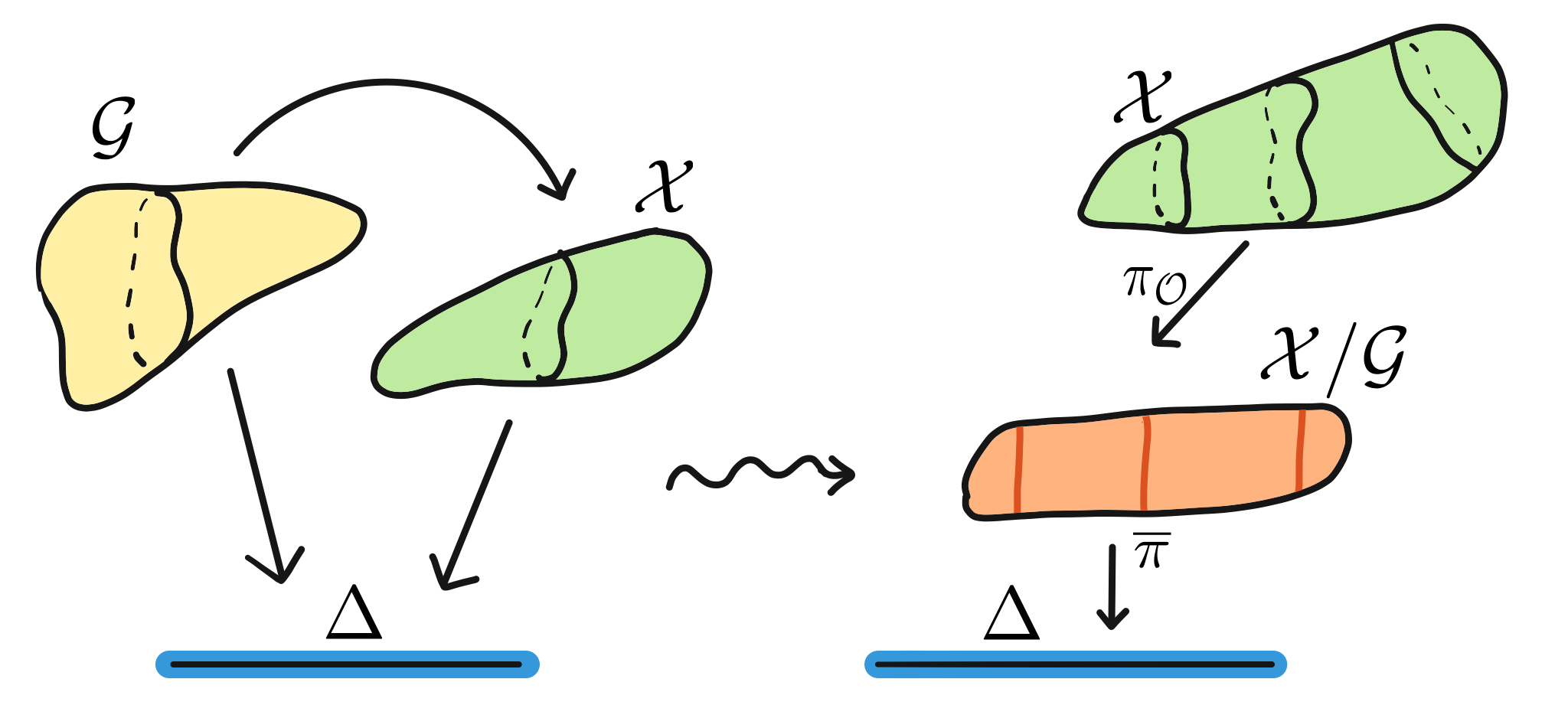}
\caption{The Quotient Family Theorem determines a sufficient condition to take the quotient of a family of spaces by a family of groups in the category of families.}	
\end{figure}

\noindent
Finally, we consider quotients in the category of families.
When does an action of a family of groups on a family of spaces have a quotient in the category of families?
That the action being proper and free suffices is an exceedingly useful fact, termed the \emph{Quotient Family Theorem} throughout this thesis.
The proof of this theorem is rather technical, but is modeled closely on the Quotient Manifold Theorem of smooth topology \cite{Lee}, using familiar techniques.

\begin{theorem*}[Quotient Family Theorem]
Let $\fam{G}\acts\fam{X}$ be a proper free action in $\Fam_\Delta$.  
Then $\fam{X}\labelarrow{\pi}\Delta$ factors as $\fam{X}\labelarrow{\piO}\fam{X}/\fam{G}\labelarrow{\bar{\pi}}\Delta$ with $\fam{X}/\fam{G}\to\Delta$ in $\Fam_\Delta$, as a family of families $\piO\colon\fam{X}\to\fam{X}/\fam{G}$ and $\bar{\pi}\colon\fam{X}/\fam{G}\to\Delta$.
\end{theorem*}

\subsection*{Families of Geometries}

These construction techniques provide enough background to define families of homogeneous spaces, and work out their basic theory.
A homogeneous space for Lie group $G$ is encoded either by a choice of a transitive action of $G$ on a smooth manifold $X$, or equivalently by a choice of a closed subgroup $K$ of $G$ (which is the point stabilizer for the translation action on its cosets $X=G/K$).
Accordingly, there are two natural definitions of a \emph{family of homogeneous spaces}; either a family of groups acting on a family of spaces, or a family of groups together with a subfamily of closed subgroups.

\begin{definition}[Group-Space Geometries]
A \emph{family of Klein geometries over 	$\Delta$} is given by a triple $(\fam{G},(\fam{X},\fam{x}))$ of a family of groups $\fam{G}\to\Delta$ acting fiberwise-transitively on a family of spaces $\fam{X}\to\Delta$ over the same base, equipped with a global section $\fam{x}\colon \Delta\to \fam{X}$ choosing a basepoint in each fiber.  
A morphism of geometries $\Phi\colon (\fam{G},(\fam{X},\fam{x}))\to(\fam{G}',(\fam{X}',\fam{x}'))$ is given by a family homomorphism $\phi_{\fam{Grp}}\colon\fam{G}\to\fam{G}'$ together with an equivariant map $\phi_{\fam{Sp}}\colon\fam{X}\to\fam{X}'$ such that $\phi_{\fam{Sp}}\circ\fam{x}=\fam{x}'$.
The category of such geometries is denoted $\cat{GrpSp}$.
\end{definition}

\begin{figure}
\centering 
\includegraphics[width=0.75\textwidth]{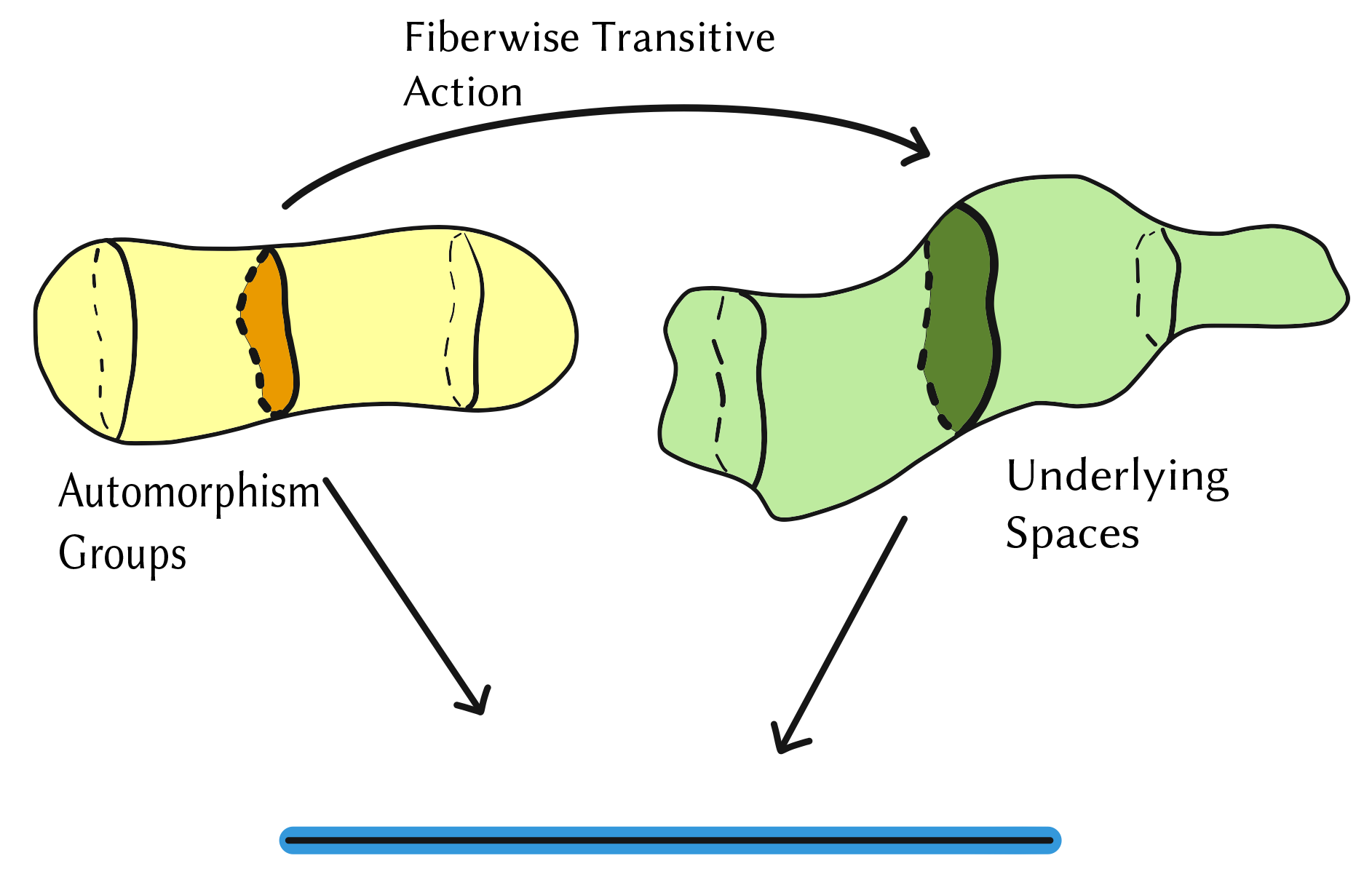}
\caption{A family of geometries is a family of groups acting fiberwise transitively on a family of spaces.}	
\end{figure}

\begin{definition}[Automorphism-Stabilizer Geometries]
A family of Klein geometries over $\Delta$ is given by a pair $(\fam{G,C})$ of a family of groups $\fam{G}\to\Delta$ and a closed subfamily $\fam{C}\subg\fam{G}$.	  
A morphism $\Phi:(\fam{H,K})\to (\fam{G,C})$ is a homomorphism of families $\Phi:\fam{H}\to\fam{G}$ with $\Phi(\fam{K})\subset\fam{C}$.
The category of these geometries is denoted $\cat{AutStb}$.
\end{definition}

\noindent
In practice, we are freely able to pass between these two formalisms when convenient, as there is an equivalence of categories in the background.
Proving this equivalence of categories by explicit construction of functors in each direction is a primary motivation for some of the tools developed earlier in the chapter, including pullbacks and the Quotient Family Theorem.

\begin{theorem*}[Equivalence of Perspectives on Homogeneous Families]
The map $\cat{F}:\cat{AutStb}\to\cat{GrpSp}$ sending a group-stabilizer geometry $(\fam{G,K})$ to the group-space geometry $\fam{(G,(G/K,K))}$ defines a functor.
Likewise, $\Psi\colon \cat{GrpSp}\to\cat{AutStb}$ defined by sending a geometry $(\fam{G,(X,x)})$ to $(\fam{G},\fam{stab}_\fam{G}(\fam{x}))$ defines a functor.	
The functors $\cat{F},\Psi$ define an equivalence of categories $\cat{GrpSp}\cong\cat{AutStb}$.	
\end{theorem*}

\subsection*{Geometries over Algebras}

The original inspiration for defining families of geometries was the 1-parameter families of groups utilized in formalizing the transition of complex Hyperbolic space to $\R\oplus\R$-Hyperbolic space.
As such, for a first application of this theory we will attempt to generalize this as far as possible, in the end proving a collection of theorems saying \emph{continuously varying families of algebras induce continuously varying families of geometries}.
To have such a theorem, we must first extend the definitions of various familiar types of geometries (say, projective geometry, the geometries of $\SO(p,q)$ in $\RP^n$, and the geometries of $\U(p,q)$ in $\CP^n$) to more general algebras.

The correct definition of projective space over an algebra $A$ is subtle, as the existence of zero divisors causes the group of units to fail to act freely on $A^n\smallsetminus\{0\}$.
Letting $Z(A^n)$ be the \emph{generalized zeroes}, the points of $A^n$ such that the $A^\times$ action is not free, provides a reasonable analog of $K\mathsf{P}^n$.

\begin{definition*}[Generalized Projective Space]
The $n-1$ dimensional projective geometry over $A$ has domain $\mathsf{AP}^{n-1}=(A^n\smallsetminus Z(A^n))/\sim$ for 
$\vec{v}\sim\vec{w}$ if there is an $a\in A^\times$ such that $a\vec{v}=\vec{w}$.  The (non-effective) automorphism group is $\GL(n;A)$ or $\SL(n;A)$.
\end{definition*} 

\noindent
This notion of projective space behaves nicely with respect to direct sums of algebras, allowing us for example to easily understand the projective geometries over $\R\oplus\R$.

\begin{proposition*}
Let $A=A_1\oplus A_2$ be a direct sum of commutative algebras.  Then  the projective geometry $A\mathsf{P}^n=(A_1\oplus A_2)\mathsf{P}^n$ decomposes as a direct product of the projective geometries over $A_1\mathsf{P}^n\times A_2\mathsf{P}^n$.
\end{proposition*}

\noindent
To construct generalized orthogonal / unitary groups, we need to consider algebras with involutions $\sigma\colon A\to A$.
Such an involution induces an analog of conjugate-transpose on the matrix algebras $\M(n,A)$, and a matrix $J$ is said to be Hermitian if $\sigma(J)^T=J$.

\begin{definition*}[Generalized Unitary Groups]
A matrix $X$ \emph{preserves} a hermitian $J$ if $\sigma(X)^T JX=J$. 
The generalized unitary group $\U(J,A,\sigma)$ consists of the matrices preserving $J$: $\U(J;A,\sigma)=\set{X\mid \sigma(X)^T JX=J}$.
\end{definition*}

\noindent
This generalized notion of unitary group encompasses both the classical orthogonal and unitary groups (the involution is allowed to be trivial), together with many new examples.
As subgroups of $\GL(n;A)$, these generalized unitary groups define \emph{generalized unitary geometries}, which have models naturally constructed within $\mathsf{AP}^n$.

\begin{definition*}[Generalized Unitary Geometry]
A unitary geometry over $(A,\sigma)$ is given by the pair $(G,C)=(\U(J;A),\mathsf{S}_J)$ for $J\in\Herm(n;A)$ and $\mathcal{S}_J$ the orbit of $[0:\cdots 0:1]\in\mathsf{AP}^n$.
\end{definition*}

\noindent
As with projective geometry, we investigate the isomorphism type of the unitary geometries over decomposable algebras.

\begin{proposition*}
If $A=A_1\oplus A_2$ and $\sigma$ preserves the factors $\sigma_1\oplus\sigma_2: A_1\oplus A_2\to A_1\oplus A_2$, then $\U(J;A,\sigma)\cong \U(J_1;A_1,\sigma_1)\times \U(J_2;A_2,\sigma_2)$ decomposes as a product for $J=J_1e_1+ J_2e_2\in\M(n,A)$.
\end{proposition*}

\noindent
The familiar case of $\R\oplus\R$ is generalized by algebras $\Lambda=A\oplus A$ equipped with the coordinate swap map as an involution.
Unitary geometries over these algebras are also closely tied to projective geometry (in fact, one can build a \emph{generalized point hyperplane projective geometry} for each), as we see below on the level of automorphism groups.

\begin{proposition*}
Let $\Lambda=A\oplus A$ and $\sigma:\Lambda\to\Lambda$ be the coordinate swap map.  
Then $\U(J;\Lambda,\sigma)\cong\GL(n,A)$ for any non-degenerate $\sigma$-hermitian matrix $J$, and the corresponding unitary geometry is point-hyperplane geometry over $A$.
\end{proposition*}

\subsection*{Applications}

Finally, having developed the terminology of families of geometries and potential interesting participants (familiar geometries defined over algebras), we look to provide some first motivating applications of this theory.
In particular, we focus on generalizations of the transition from $\Hyp_\C^n$ to $\Hyp_{\R\oplus\R}^n$, showing that any given family of algebras produces corresponding families of projective, unitary and orthogonal geometries.
We then turn briefly to another application, and study transitions that occur from a group action on a space, when we may interpret the collection of orbits as a smoothly transitioning family of spaces.
This will, among other things, provide a means of transitioning between Hyperbolic and de Sitter geometry, which does not arise within an ambient projective geometry.

\begin{theorem*}[Linear Groups Vary Continuously]
Let $\fam{A}\to\Delta$ be a smooth family of algebras.  Then $\fam{GL}(n,\fam{A})\to\Delta$ is a family of Lie groups.
The groups $\fam{SL}(n;\fam{A})$ are a subfamily of $\fam{GL}(n;\fam{A})\to\Delta$.
\end{theorem*}

\noindent
This has an interesting corollary, in the world of geometries defined over $\C$ and $\R\oplus\R$, providing a new transition between familiar geometries.

\begin{corollary*}
The projective spaces $\Lambda_\delta\mathsf{P}^n$ form a continuous family of geometries, transitioning from $\CP^n$ to $(\R\oplus\R)\mathsf{P}^n\cong\RP^n\times\RP^n$
\end{corollary*}

\noindent
Showing the individual projective spaces $\mathsf{AP}^n$ form a continuous family given any continuous family $\fam{A}\to\Delta$ of algebras is equivalent, by the Quotient Family Theorem, to showing that the associated point stabilizer subgroups form a subfamily of $\fam{GL}(n;\fam{A})$.

\begin{theorem*}[Projective Geometries Vary Continuously]
A smooth family of algebras $\fam{A}\to\Delta$ determines a smooth family of projective geometries $\fam{A}\mathsf{P}^n\to\Delta$ for each $n\in\N$. 	
\label{thm:Proj_Geo_Families}
\end{theorem*}

\noindent
Given a non-degenerate section $\fam{J}\colon\Delta\to\fam{Herm}(n;\fam{A},\sigma)$, one can define for each $\delta$ the unitary group $\U(\fam{J}_\delta;\fam{A}_\delta,\sigma_\delta)\subg\GL(n;\fam{A}_\delta)$.  The union of these is the \emph{generalized unitary family} corresponding to $\fam{J}$ over $\Delta$.

\begin{theorem*}[Unitary Groups Vary Continuously]
Let $(\fam{A},\sigma)\to\Delta$ be a family of algebras and $\fam{J}\colon\Delta\to\fam{Herm}(n;\fam{A},\sigma)$ a smooth non-degenerate section.  
Then $\fam{U(J;A)}$ is a smooth subfamily of $\fam{GL}(n;\fam{A})$.
The special unitary groups $\fam{SU(J;A)}$ are a subfamily of $\fam{U(J;A)}$.
\end{theorem*}

\noindent
Again, showing further that the point stabilizer subgroups of this families action on $\fam{A}\mathsf{P}^n$ form a closed subfamily suffices to prove the corresponding collection of unitary geometries forms a family.

 \begin{theorem*}[Unitary Geometries Vary Continuously]
 Given a smooth family of algebras $\fam{A}\to\Delta$ and a smooth (diagonal) section $\fam{J}\colon\Delta\to\fam{Herm}(n;\fam{A})$, there is a corresponding smooth family of unitary geometries $(\fam{U(J,A)},\fam{UST(J;A)})$.
 \end{theorem*}

\noindent
As a final application, we use the developed techniques to construct a collection of new transitions between familiar subgeometries of real projective space.

The case of most interest concerns $\Hyp^n$, de Sitter space, and the geometry of the lightcone.
There is no transition between these geometries \emph{as subgeometries of $\RP^2$}, as the lightcone loses a dimension under projectivization.  But there is a transition abstractly, as a family.

\begin{theorem*}
There is a transition from $\H^n$ to $\mathsf{d}\S^n$ through the geometry of the canonical line bundle to the conformal $n-1$ sphere.	
\end{theorem*}

\noindent
More generally, if $G$ is any orthogonal or unitary subgroup of $\GL(n;\R)$ or $\GL(n;\C)$  the associated quadratic / hermitian form defines a positive and negative cone, whose projectivizations $X_+$ and $X_-$ are the domains for the two projective geometries $(G,X_+)$, $(G,X_-)$ with automorphism group $G$.  The isomorphism type of the geometries depend on the signature $(p,q)$ of the form: $X_+$ is not isomorphic to $X_-$ unless $p=q$.  

\begin{theorem*}
There is a transition from $(G,X_+)$ to $(G,X_-)$ for any orthogonal or unitary group $G$.	
\end{theorem*}

\part{Geometries}

\chapter{Manifolds \& Orbifolds}
\label{chp:Man_Orb_Cone}

This first chapter provides a brief review of the main objects of study in geometric topology; namely smooth manifolds, and their slightly more subtle cousins the orbifolds.
Due to the introductory nature of this material this review will be succinct, with most proofs left to the references.
Additional reading on this material is highlighted throughout, but some particularly comprehensive sources include \cite{Scott80,Lee,CooperHK00}.

\section{Manifolds}
\label{sec:Manifolds}
\index{Manifold}
\index{Smooth Manifold}
\index{Manifold!Smooth}

\begin{definition}
A \emph{manifold} is a second-countable Hausdorff topological space $M$ which is \emph{locally Euclidean} in the sense that each point $p\in M$ has a neighborhood $U\ni p$ homeomorphic to some subset of $\R^n$.
\end{definition}

\noindent
This data of local homeomorphisms about each point of $M$ is often packaged together into an \emph{atlas of charts}, an open cover $\set{U_\alpha}$ of $M$ together with a collection of maps $\phi_\alpha\colon U_\alpha\to \R^n$ which are homeomorphisms onto their images.
This captures the intuitive idea that a topological manifold $M$ is `glued together out of pieces of $\R^n$' in a precise way: namely we can build $M$ out of the disjoint union of the open sets $\set{\phi_\alpha(U_\alpha)}$ under the quotient topology identifying the subsets $\phi_\alpha(U_\alpha\cap U_\beta)$ and $\phi_\beta(U_\alpha\cap U_\beta)$ by the homeomorphism $\phi_\beta\phi_\alpha\inv$.

\begin{example}[Topological Sphere]
\label{Ex:TopSphere}
Let $\S^2$ be the unit 2-sphere in $\R^3$.
The open sets $U_S=\S^2\smallsetminus (0,0,1)$, $U_N=\S^2\smallsetminus(0,0,-1)$ cover $\S^2$, and together with the charts $\phi_{S/N}\colon U_{S/N}\to\R^2$ given by stereographic projection define an atlas.
\end{example}

\begin{figure}
\centering\includegraphics[width=0.5\textwidth]{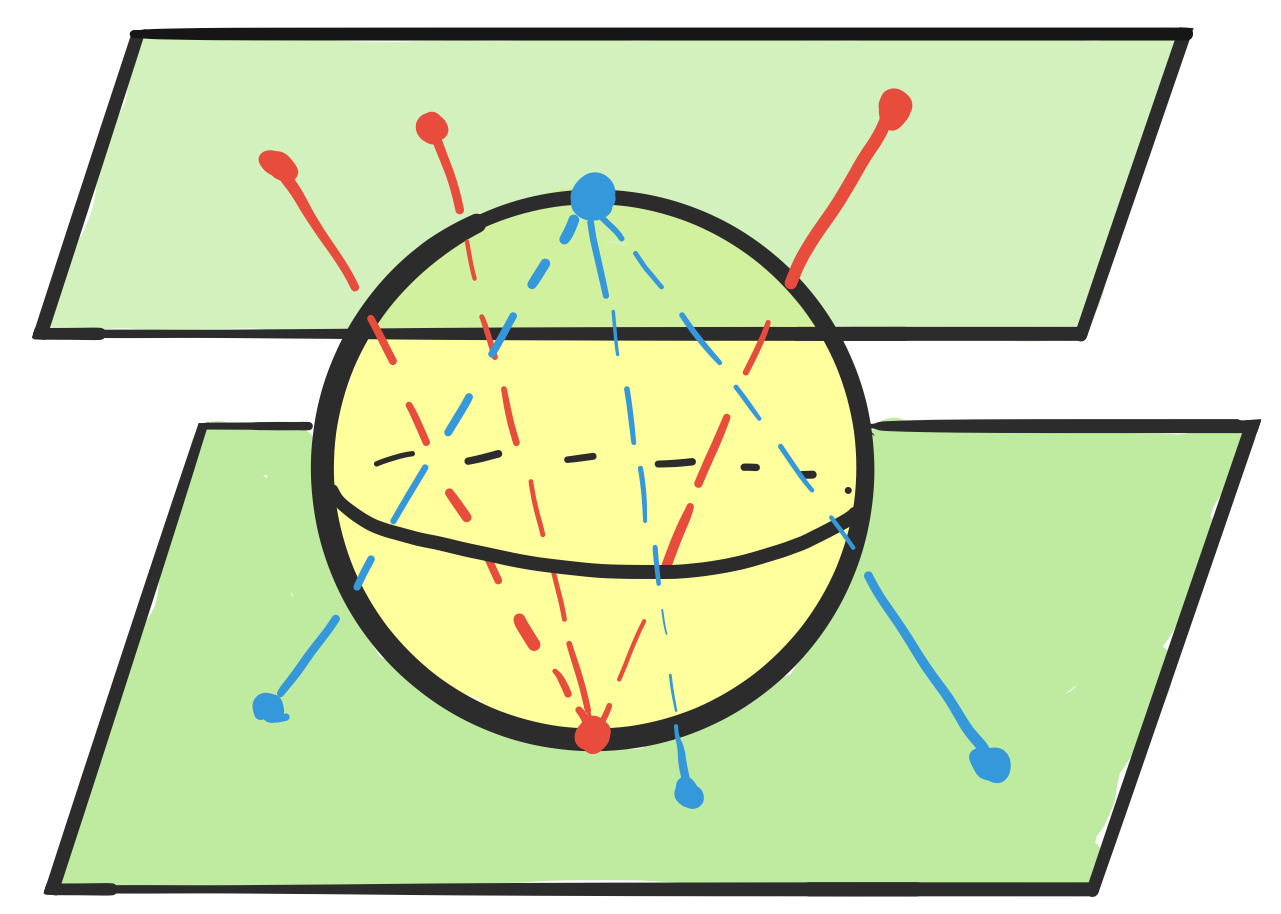}
\caption{The sphere as a manifold with two charts.}	
\end{figure}

\noindent
There is no calculus on a topological manifold, as notions such as differentiability are not invariant under homeomorphism.
To make available the tools of analysis requires more structure, namely a \emph{smooth manifold} with \emph{smoothly compatible charts}.

\begin{definition}
Two charts $(U_\alpha, \phi_\alpha)$ and $(U_\beta,\phi_\beta)$ are \emph{smoothly compatible} if the associated \emph{transition map} $\phi_\beta\phi_\alpha\inv\colon \phi_\alpha(U_\alpha\cap U_\beta)\to\phi_\beta(U_\alpha\cap U_\beta)$ is smooth as a map between subsets of $\R^n$ (with the standard smooth structure).
\end{definition}

\begin{definition}
A smooth manifold $M$ is a topological manifold equipped with a \emph{smooth atlas}; an atlas of \emph{smoothly compatible charts} $\mathcal{A}=\set{(U_\alpha,\phi_\alpha)}$. 
\end{definition}

\begin{example}[Incompatible Charts]
The global charts $(\R,\phi)$ and $(\R,\psi)$ on $\R$ given by $\phi(x)=x^3$ and $\psi(x)=x$ are not smoothly compatible as $\psi\phi\inv(x)=\sqrt[3]{x}$ is not smooth at $0$.
\end{example}

\noindent
To avoid worries about uniqueness, a smooth manifold is often defined to come equipped with a \emph{maximal atlas} of the type described above.
Two atlases $\fam{A}$ and $\fam{A}'$ on $M$ are \emph{compatible} if their union is again an atlas (transition maps corresponding to overlaps of charts in $\fam{A},\fam{A}'$ are also diffeomorphisms).
A quick application of Zorn's lemma shows that every atlas is contained in a unique maximal atlas defining the smooth manifold $M$.

\begin{example}[Smooth Sphere]
\label{Ex:SmoothSphere}
The charts $\phi_S(x,y,z)=(\tfrac{x}{1-z},\tfrac{y}{1-z})$ and $\phi_N(x,y,z)=(\tfrac{x}{1+z},\tfrac{y}{1+z})$ of Example \ref{Ex:TopSphere} are smoothly compatible, and define a smooth structure on $\S^2$.
\end{example}

\noindent
Instead of calculus, one may wish to import more combinatorial notions to the study of manifolds, such as triangulations.
To do so also requires more than a bare topological manifold, with the relevant additional structure no longer being smooth but \emph{piecewise linear}.

\begin{definition}
A \emph{piecewise linear}, or \emph{PL} manifold $M$ is a topological manifold together with a piecewise linear atlas.
That is, an atlas of charts $\fam{A}=\set{(U_\alpha,\phi_\alpha)}$ such that the transition maps $\phi_\beta\phi_\alpha\inv$ are piecewise linear functions between subsets of $\R^n$.
\end{definition}

\noindent
Again, a simple argument shows each piecewise linear atlas is contained in a unique maximal atlas defining the structure.
Throughout this thesis we work in the smooth category unless otherwise specified.
In the case of $(G,X)$ geometries most work will actually take place in the subcategory of \emph{real analytic manifolds}, predictably defined as follows.

\begin{definition}
A \emph{real analytic manifold} is a topological manifold $M$ together with an atlas of charts $\mathcal{A}=\set{(U_\alpha,\phi_\alpha)}$ with transition maps given by restrictions of real analytic functions between subsets of $\R^n$.	
\end{definition}

\begin{observation}
These categories of manifolds coincide in low dimensions.  
More precisely, each topological $n$ manifold for $n\in\set{0,1,2,3}$ admits a unique smooth structure, a unique real analytic structure, and a unique PL structure.
\end{observation}

\noindent{\sffamily \bfseries Warning!}
These categories of	manifolds differ in dimensions 4 and above.
There are topological manifolds which admit no smooth structures, topological manifolds admitting multiple smooth structures, and topological manifolds with admitting no triangulations.

\section{Orbifolds}
\label{sec:Orbifolds}
\index{Orbifold}

The quotient space $M/\Gamma$  of a finite group $\Gamma$  acting on a (smooth) manifold $M$ inherits the structure of a (smooth) manifold when the action is free.
When the $\Gamma$ action has fixed points, the quotient space is no longer necessarily a manifold, but has only mild singularities with neighborhoods homeomorphic to $\R^n/\Gamma_x$ for point-stabilizing subgroups $\Gamma_x<\Gamma$.
Such a space is the prototypical example of an \emph{orbifold}, 
a convenient generalization of manifolds being locally modeled on the quotient spaces of $\R^n$ by finite group actions.
The definition of orbifolds is rather involved, and is given by an \emph{orbifold atlas} on an underlying topological space much as a smooth structure is a smooth atlas on a manifold.
From this we can build a theory of orbifolds that mimics closely the familiar manifold theory; complete with notions of orbifold covering spaces, orbifold fundamental groups and orbifold Euler characteristics.

\subsection{The Category of Orbifolds}

Locally, an orbifold is pieced together out of \emph{orbifold charts}, which are pieces of $\R^n$ quotiented out by the action of finite groups.
These pieces are glued together in a way that is compatible with the group actions, as laid out explicitly below.

\begin{definition}
An \emph{orbifold chart} is a quadruple $(\tilde{U},U,\Gamma, \phi)$ such that $\tilde{U}\subset\R^n$ is open, $\Gamma$ acts on $\widetilde{U}$ by diffeomorphisms and $\phi\colon \widetilde{U}/\Gamma\to U$ is a homeomorphism.
\begin{center}
\begin{tikzcd}
&\hspace{-0.75cm}\Gamma\acts\widetilde{U}\ar[d]\\
U\ar[r,"\phi"]&\widetilde{U}/\Gamma
\end{tikzcd}	
\end{center}

\end{definition}

\noindent
We call $\widetilde{U}$ together with the action of $\Gamma$ the \emph{local model}, $\widetilde{U}$ the \emph{local cover} and $\Gamma$ the \emph{local action}.
To construct an atlas, we need a notion of \emph{compatibility} between different local models.

\begin{definition}
\label{Def:Orbifold_Compatible}
If $V\subset U$, an orbifold chart $(V,\tilde{V},G,\psi)$ is compatible with $(U,\widetilde{U},\Gamma,\phi)$ if the inclusion map $\iota\colon V\inject U$ lifts to an embedding $\tilde{\iota}\colon\widetilde{V}\inject\widetilde{U}$, equivariant with respect to a homomorphism $\rho\colon G\to\Gamma$ making the following diagram commute.
\begin{center}
\begin{tikzcd}
\widetilde{V}\ar[r,"\widetilde{\iota}"]\ar[d]&\widetilde{U}\ar[d]\\
\widetilde{V}/G\ar[r]\ar[dd]&\widetilde{U}/\rho(G)\ar[d]\\
&\widetilde{U}/\Gamma\ar[d]\\
V\ar[r,"\iota"]&U
\end{tikzcd}
\end{center}	
\end{definition}

\begin{definition}
An orbifold atlas of charts on a topological space $X$ is an open covering $\mathcal{U}=\set{U_\alpha}$ of $X$, closed under finite intersection	s, and an orbifold chart $(U_\alpha,\widetilde{U}_\alpha,\Gamma_\alpha,\phi_\alpha)$ for each $U_\alpha\in\mathcal{U}$ such that when $U_\alpha\subset U_\beta$ the associated orbifold charts are compatible in the sense of Definition \ref{Def:Orbifold_Compatible}.
\end{definition}

\begin{definition}
An orbifold $\mathcal{O}$ is a pair $(X_\mathcal{O},\mathcal{A})$ consisting of an underlying paracompact Hausdorff topological space $X_\mathcal{O}$ and a maximal \emph{orbifold atlas} $\mathcal{A}$ of orbifold charts on $X_\mathcal{O}$.
\end{definition}

\noindent
Oftentimes we will abuse notation and write $\mathcal{O}$ for the underlying space as well.
The covering by orbifold charts defines an \emph{isotropy} subgroup at each point $x\in\mathcal{O}$ by taking the point stabilizer of a lift of $x$ in any orbifold chart containing it.
The compatibility condition for charts ensures this is well-defined, though only up to isomorphism.
We denote the isotropy group $\mathsf{Is}(x)$.

\begin{observation}
If $x\in\mathcal{O}$ then $x$ is contained in some orbifold chart $(U,\widetilde{U},\Gamma,\phi)$ with $\Gamma=\mathsf{Is}(x)$ the isotropy group.
\end{observation}

\begin{definition}
A point $x\in\mathcal{O}$ with $\mathsf{Is}(x)=\set{1}$ is called a \emph{smooth point}, as there is a neighborhood of $x$ homeomorphic to an open set in $\R^n$.
If $\mathsf{Is}(x)\neq\set{1}$, then $x$ is called a \emph{singular point} of the orbifold $\mathcal{O}$.
The \emph{singular locus} of $\mathcal{O}$ is the set of singular points $\Sigma(\mathcal{O})=\set{x\in X_\mathcal{O}\mid \mathsf{Is}(x)\neq\set{1}}$.
\end{definition}

\begin{example}
Any smooth manifold is an orbifold, replacing the manifold charts $(U,\phi)$ with the orbifold charts $(U,\phi(U),\set{1},\phi)$.	
\end{example}

\begin{example}
The Euclidean cone with cone angle $\pi$ is an orbifold, defined by the single chart $(\R^2/\Z^2,\R^2,\Z_2,\id)$ where $\Z_2$ acts by a $\pi$ rotation on $\R^2$ about the origin.
The singular locus of this orbifold is a single point.
\end{example}

\begin{figure}
\centering\includegraphics[width=0.5\textwidth]{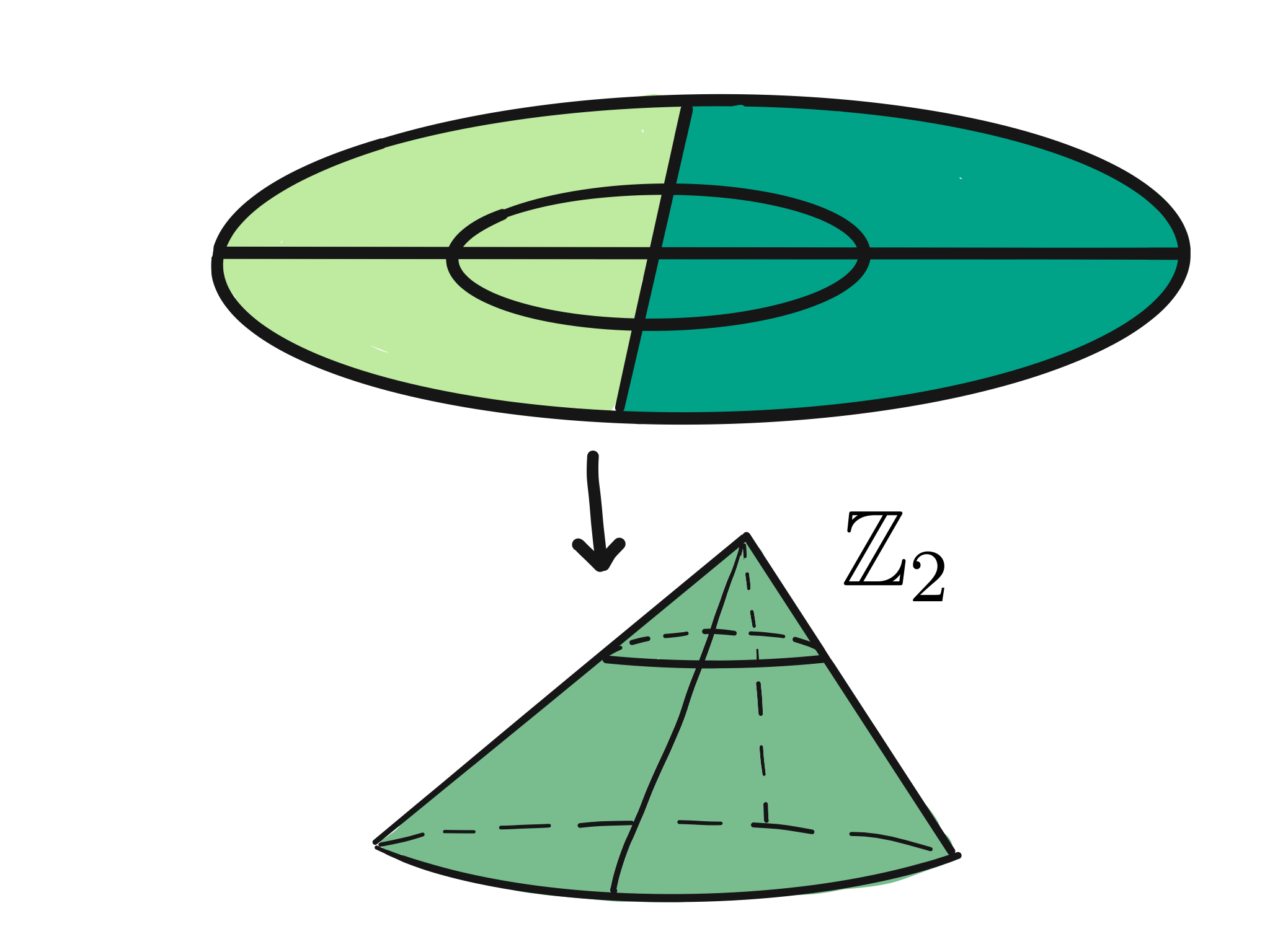}
\caption{The Euclidean right angle cone as an orbifold.}	
\end{figure}

\begin{definition}
An orbifold is \emph{locally orientable} if each local action is by orientation preserving diffeomorphisms on the local models in $\R^n$.
It is \emph{orientable} if in addition the inclusion maps $V\inject U$ are induced by orientation preserving embeddings $\widetilde{V}\inject\widetilde{U}$.
\end{definition}

\begin{example}
\label{Ex:Cone_RP}
The three dimensional analog of the cone, $\mathcal{O}=\R^3/\Z_2$ with the $\Z_2$ action by the antipodal map $x\mapsto -x$ is an example of an orbifold which is not locally orientable.
\end{example}

\begin{example}
The Klein bottle is an orbifold which is locally orientable, but not orientable.	
\end{example}
 
\noindent
Some non locally-orientable orbifolds have underlying space a manifold with boundary, such as the closed upper half plane, viewed as an orbifold quotient of $\R^2$ by reflection across the $x$ axis.
There is an additional notion of \emph{orbifold with boundary} which we do not need in this thesis but we nonetheless mention briefly here.

\begin{definition}
An \emph{orbifold with boundary} is defined similarly to an orbifold, replacing the local models with subsets of the closed upper half space $\R_+^n$ via finite group actions.
The \emph{boundary} of an orbifold $\partial_{\mathsf{orb}}\mathcal{O}$ is the set of points $x\in X_\mathcal{O}$  whose lifts to local models lie in the boundary of upper half space in some chart.
An orbifold is \emph{closed} if it is compact and its orbifold boundary is empty.
\end{definition}

\noindent
We have succeeded in defining the objects in the category of orbifolds.
We now move on to describe the morphisms, or \emph{orbifold maps} between them.

\begin{definition}
An \emph{local orbifold map} between two local models $(U,\widetilde{U},\Gamma,\phi)$ and $(V,\widetilde{V},G,\psi)$ is a pair $(\widetilde{\eta},\gamma)$ for $\gamma\colon \Gamma\to G$   a group homomorphism and $\widetilde{\eta}\colon \widetilde{U}\to\widetilde{V}$ a $\gamma$-equivariant smooth map.
This induces a map $\eta\colon U\to V$.
Conversely, a map $\eta\colon U\to V$ \emph{lifts to a local orbifold map} if there are local models for $U,V$ and a local orbifold map $(\tilde{\eta},\gamma)$ as above.
\end{definition}

\noindent
A local orbifold map is called a \emph{local orbifold isomorphism} when $\gamma$ is faithful and $\widetilde{\eta}$ is an immersion.
This terminology allows a more succinct description of the compatibility condition for orbifold charts: charts based on $V\subset U$ are compatible if the inclusion $V\inject U$ lifts to a local orbifold isomorphism.

\begin{definition}
An orbifold map $f\colon\mathcal{O}\to\mathcal{Q}$ is given by a map between underlying spaces $\overline{f}\colon X_\mathcal{O}\to X_\mathcal{Q}$ such that for each $x\in X_\mathcal{O}$, $\overline{f}(x)\in X_\mathcal{Q}$, there are open neighborhoods $x\in U$, $\overline{f}(x)\in V$ such that $\overline{f}$ lifts to a local orbifold map in local models for $U,V$.
\end{definition}

\noindent
An \emph{orbifold diffeomorphism} is an orbifold map which is bijective between underlying spaces, and whose inverse is also an orbifold map.
A local orbifold map $(\widetilde{\eta},\gamma)$ is a \emph{local immersion} if $\widetilde{\eta}$ is an immersion; an orbifold map $\mathcal{Q}\to\mathcal{O}$ is an immersion if it is locally a local immersion.

\begin{definition}
A \emph{suborbifold} of an orbifold $\mathcal{O}$ is the image of an injective immersion $\mathcal{Q}\inject\mathcal{O}$ from some closed orbifold $\mathcal{Q}$.
\end{definition}

\begin{figure}
\centering\includegraphics[width=0.5\textwidth]{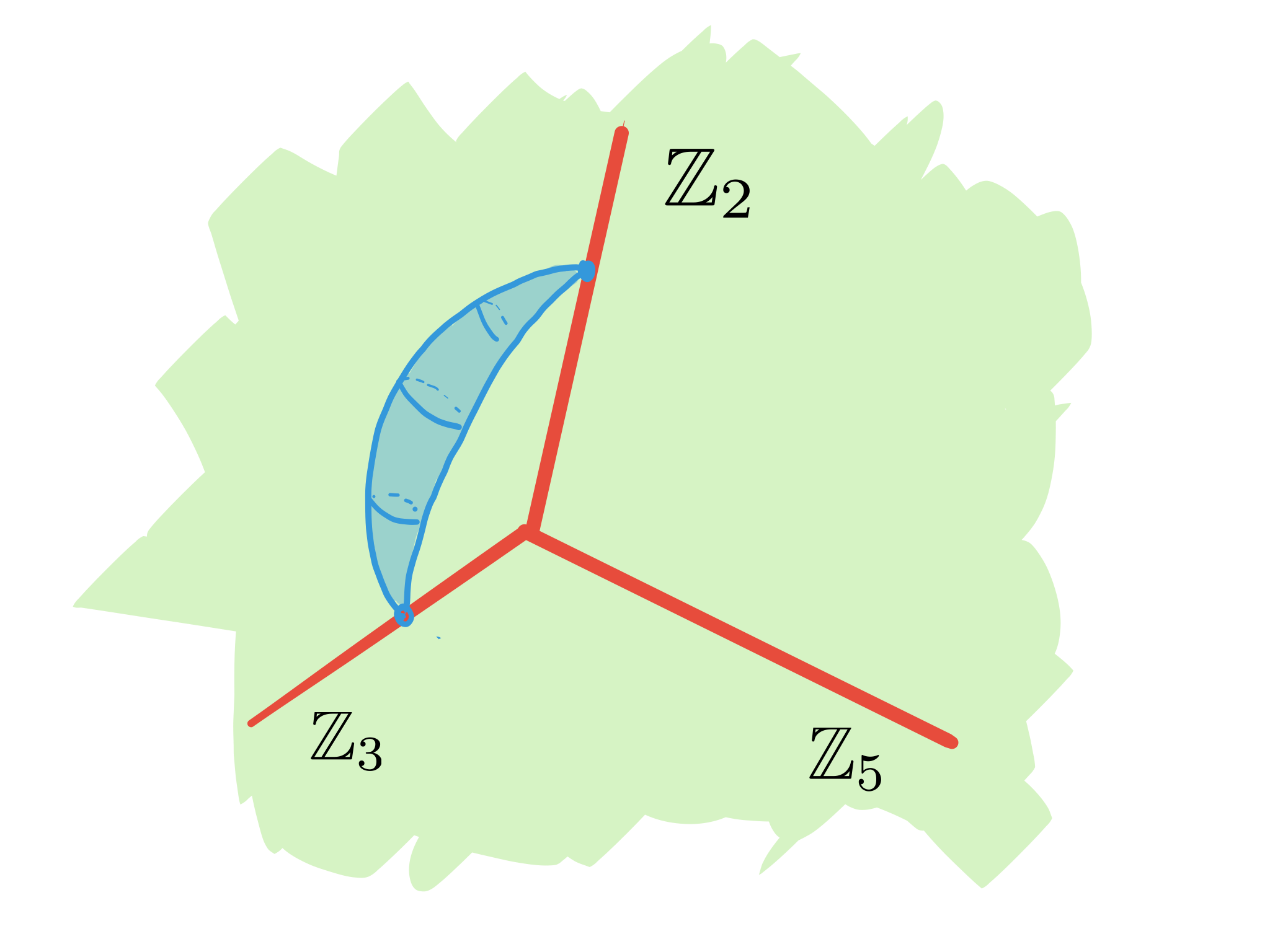}
\caption{An immersion of a 2-orbifold in a 3-orbifold, with singular sets of each labeled by their isotropy groups.}	
\end{figure}

\subsection{Examples Of Orbifolds}

To make the following discussion of the theory of orbifolds more concrete it will be helpful to have a list of examples available.
To start, we list some local models for orbifolds in small dimensions (we will see this list is comprehensive in dimensions 1 and 2 later on).

\begin{example}
The $\Z_2$ action $x\mapsto -x$ on $\R$ has orbifold quotient with underlying space a closed ray $[0,\infty)$.
The point $0$ has $\Z_2$ isotropy group, and is called a \emph{mirror reflector}.	
Similarly the action $(x,y)\mapsto (x,-y)$ on $\R^2$ has quotient the upper half plane with a line $\R\times\set{0}$ of mirror points with isotropy group $\Z_2$.
\end{example}

\begin{figure}
\centering\includegraphics[width=0.5\textwidth]{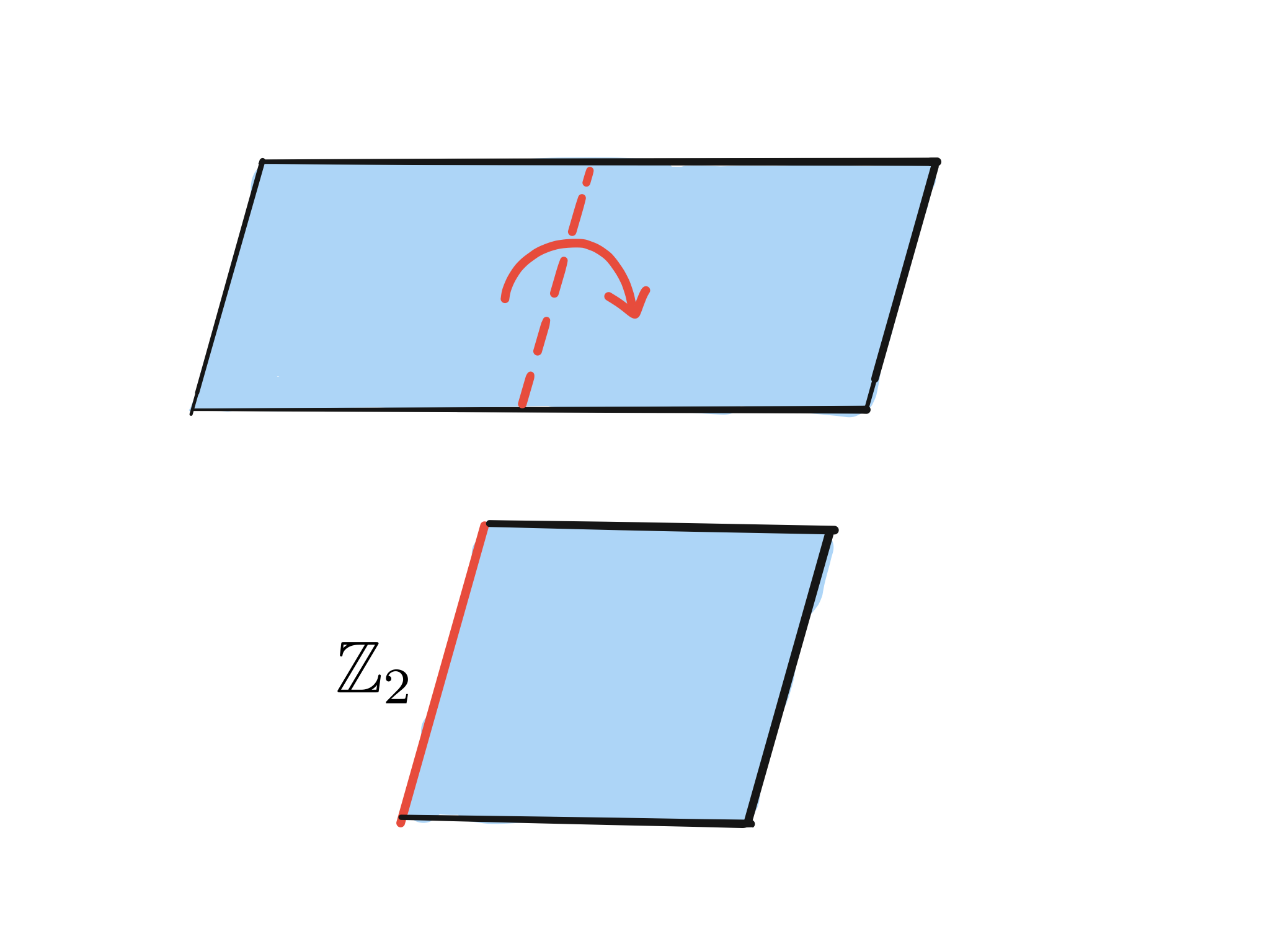}
\caption{Orbifold mirror reflector locus.}	
\end{figure}

\begin{example}
The quotient of $\C$ by the $\Z_2$ action of multiplication by $-1$ is homeomorphic to $\R^2$, but does not inherit a smooth structure in the quotient as $0$ is fixed by the action.
As an orbifold, the quotient $\C/\Z_2$ has an isolated singular point $\set{0}$ with isotropy group $\Z_2$, and is thought of as a cone point of cone angle $\pi$ at $0$.  
Similarly, the quotient $\C/\Gamma$ for $\Gamma$ any finite subgroup of $\U(1)$ is a cone, with singular locus $\set{0}$ and isotropy subgroup $\Gamma\cong\Z_n$.
We already saw an example of this above.
\end{example}

\begin{example}
The action of the dihedral group $D_{2n}$ on $\C$ preserving a regular $n$-gon centered at $0$ has orbifold quotient with underlying space a wedge $X_\mathcal{O}=\set{re^{i\theta}\mid \theta\in [0,\pi/n]}$.
Points on the boundary of this wedge have isotropy group $\Z_2$ and are mirror reflectors, the corner $r=0$ has isotropy group the full dihedral group $D_n$ and is called a \emph{corner reflector}.	
\end{example}

\begin{figure}
\centering\includegraphics[width=0.5\textwidth]{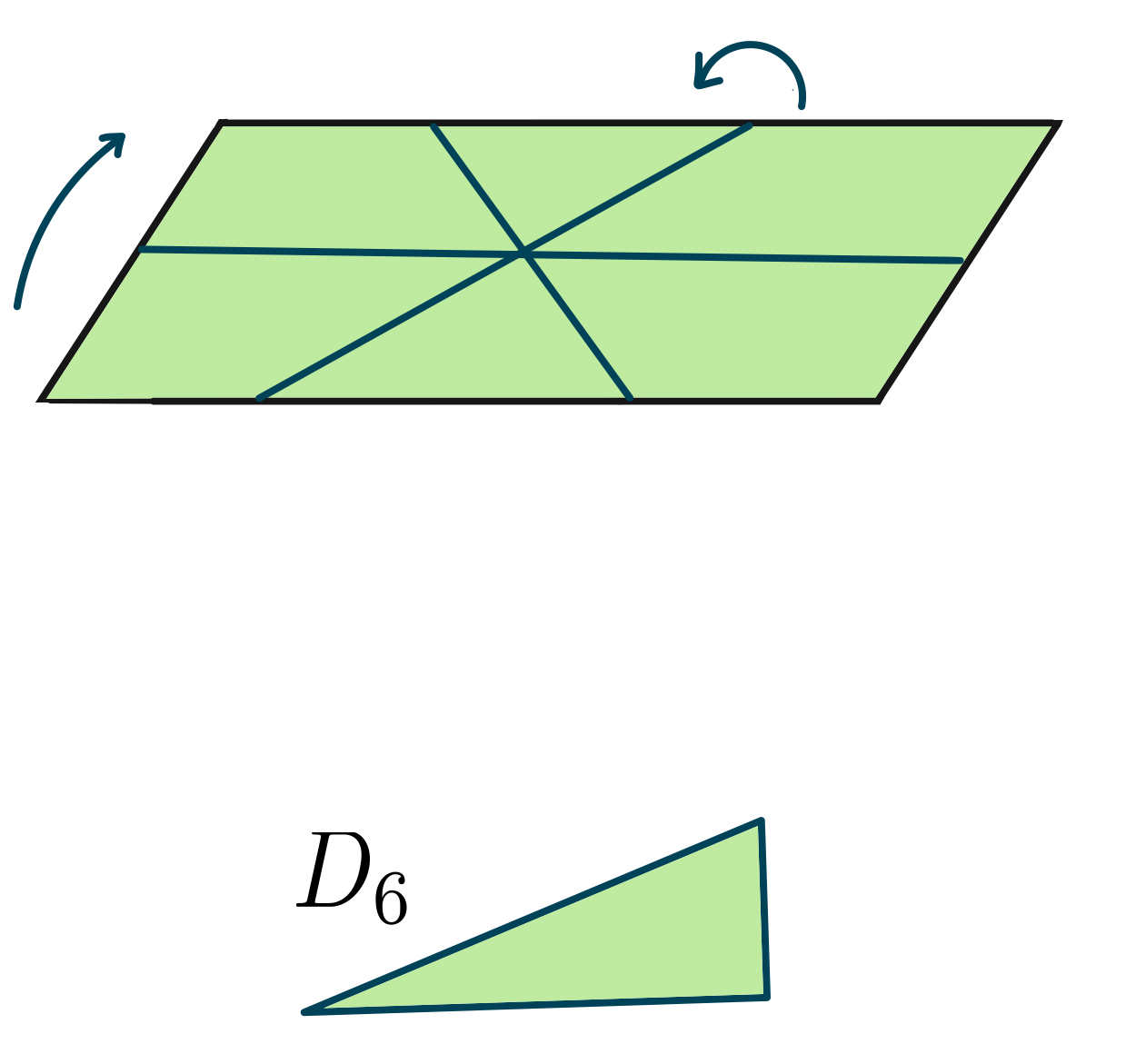}
\caption{Orbifold corner reflector locus.}	
\end{figure}

\begin{example}
The action of $\Z_n$ on $\R^3$ by rotation about the $z$-axis by angle $2\pi/n$ has orbifold quotient with underlying space homeomorphic to $\R^3$, but singular locus $\set{(0,0)}\times\R$ with isotropy group $\Z_n$.
This is the product of a cone $\C/\Z_n$ with $\R$, and the singular locus is called a \emph{cone axis}.
\end{example}

\begin{figure}
\centering\includegraphics[width=0.5\textwidth]{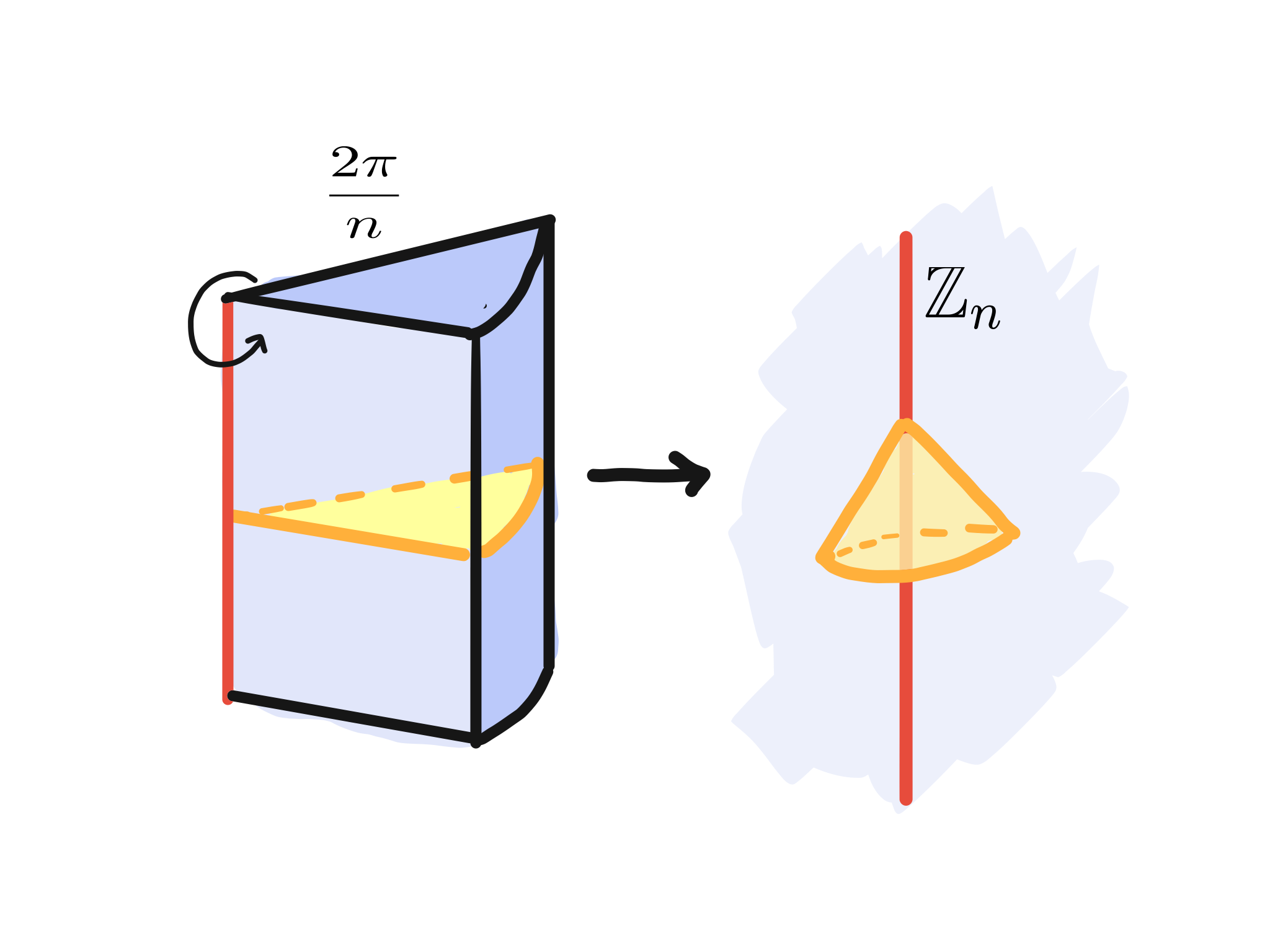}
\caption{A cone axis singularity in a 3-orbifold.}	
\end{figure}

\begin{example}
The symmetries of a dodecahedron form the $(2,3,5)$ triangle group $\Delta(2,3,5)$ and act on $\R^3$ fixing the origin.
The orbifold quotient is again homeomorphic to $\R^3$, but singular set the union of the positive $x,y$ and $z$ axes.
The positive $x$ axis is a cone axis of cone angle $\pi/2$, the $y$ axis has cone angle $\pi/3$ and the $z$-axis angle $\pi/5$.  
The origin has $\Delta(2,3,5)$ as its isotropy subgroup.
We will see shortly that this is the \emph{cone on the $(2,3,5)$ triangle pillowcase orbifold}.
Similarly, any spherical triangle group $\Delta(p,q,r)$ acts on $\S^2$ with quotient a cone on the $(p,q,r)$ triangle pillowcase.
\end{example}

\noindent
In this thesis as in much of geometric topology, we won't actually have to use this definition directly.
The theory of orbifolds was designed to accommodate the quotients of manifolds by finite group actions, and this will be our primary means of creation.

\begin{proposition}
If $M$ is a smooth manifold and $\Gamma$ a finite group of diffeomorphisms acting on $M$, then the orbit space $M/\Gamma$ inherits the natural structure of an orbifold.	
\end{proposition}
\begin{proof}
Let $\pi\colon M\to M/\Gamma$ be the projection onto the orbit space, and $x\in M$. 
If $x$ is not fixed by $\Gamma$ then there is some small neighborhood $U\ni x$ moved off of itself	by all elements of $\Gamma$, and $U$ descends to a chart $(\pi(U),U,\set{1},\pi)$ based at $\pi(x)$.
If $x$ is fixed by $\Gamma$, let $\Gamma_x<\Gamma$ be the stabilizing subgroup, and $U\ni x$ a neighborhood of $x$ preserved by $\Gamma_x$.
Then $(\pi(U),U,\Gamma_x,\pi)$ is a local model at $\pi(x)$.
These local models satisfy the compatibility condition for orbifold charts and so determine an orbifold structure on $M/\Gamma$.
\end{proof}

\begin{observation}
The same result holds for properly discontinuous actions of infinite groups.  The quotient by a free properly discontinuous action is a manifold; removing the freeness assumption results in an orbifold quotient.	
\end{observation}

\begin{example}
\label{Ex:Torus_Branched_Cover}
The torus is a branched cover of the sphere over 4 points \emph{skewering}: take a donut and pierce it all the way through with a chopstick, the quotient under a $\pi$ rotation is a topological sphere.
The quotient sphere inherits an orbifold structure with four cone points with isotropy group $\Z_2$.	
Similarly, the hyperelliptic involution of a genus $g$ surface has orbifold quotient a sphere with $2g+2$ cone points of cone angle $\pi$.
\end{example}

\begin{example}
Let $T^2=\S^1\times\S^1$ and consider the $\Z_2$ action by complex conjugation on the first factor.  The quotient is an annulus with boundary components $\set{1}\times\S^1$ and $\set{-1}\times\S^1$, and inherits an orbifold structure where these are circles of mirror points in the singular locus with isotropy groups $\Z_2$.
\end{example}

\begin{example}
Let $(p,q,r)$ be a triple of natural numbers and $\Delta(p,q,r)$ the corresponding triangle group $\Delta(p,q,r)=\langle \alpha,\beta,\gamma\mid \alpha^p=\beta^q=\gamma^r=\alpha\beta\gamma=1\rangle$.
Then the sphere with three cone points of order $p,q,r$ is an orbifold, arising as a quotient $X/\Delta(p,q,r)$ for $X=\S^2$ when $\tfrac{1}{p}+\tfrac{1}{q}+\tfrac{1}{r}>1$, $X=\E^2$ when this sum is equal to $1$, and $X=\Hyp^2$ otherwise.
\end{example}

\begin{figure}
\centering\includegraphics[width=0.70\textwidth]{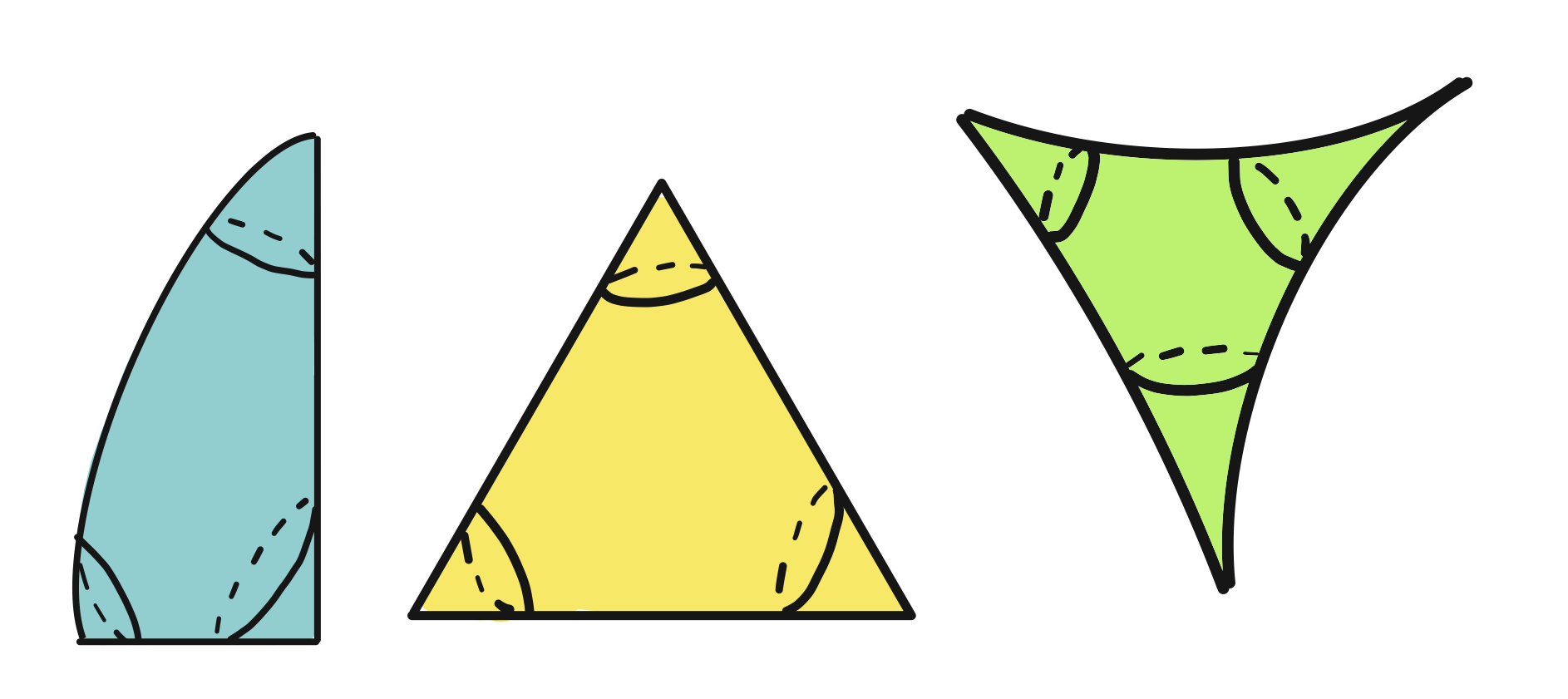}
\caption{Orbifolds and triangle groups.}
\end{figure}

\noindent
To make the discussions surrounding 2-dimensional examples easier, we will denote by $S(n_1,\ldots, n_r)$ the orbifold with underlying space a surface $S$ and $r$ cone points of order $n_1,\ldots n_r$.

\begin{figure}
\centering\includegraphics[width=0.75\textwidth]{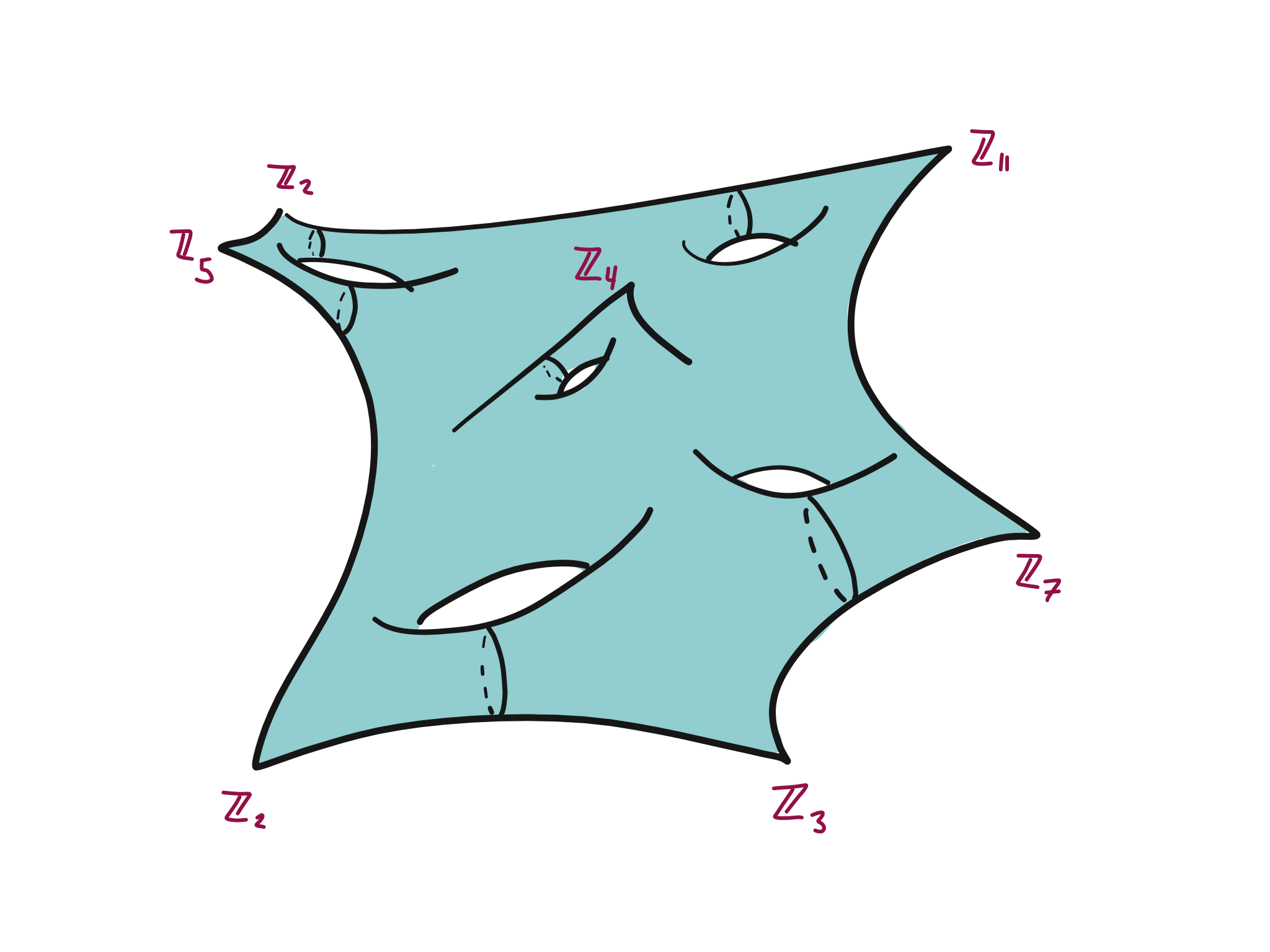}
\caption{Some more generic 2-orbifolds.}	
\end{figure}

\subsection{The Theory of Orbifolds}

Many things carry over from manifold theory to orbifolds, though the definitions become more technical the fundamental theorems remain true.
We give a short review of this theory here, starting with the notion of orbifold covering spaces.

\begin{definition}
An orbifold cover $\pi\colon\widetilde{\mathcal{O}}\to\mathcal{O}$ is an orbifold map such that each point $x\in\mathcal{O}$ has a neighborhood $U$ with local model $(U,\widetilde{U},\Gamma,\phi)$ and the preimage $\overline{\pi}\inv(U)$ is a disjoint union of components $V_i$, each with local models $(V_i,\widetilde{U}, G_i,\psi_i)$ with $G_i<\Gamma$ and $\bar{\pi}\colon V_i\to U$ the natural projection $\widetilde{U}/G_i\to\widetilde{U}/\Gamma$.
\end{definition}

\begin{example}
The branched covering $T^2\to \S^2$ of Example \ref{Ex:Torus_Branched_Cover} is an orbifold covering map of $\S^2(2,2,2,2)$ by the torus.
This cover \emph{unwraps} all the cone points.
\end{example}

\begin{figure}
\centering\includegraphics[width=0.60\textwidth]{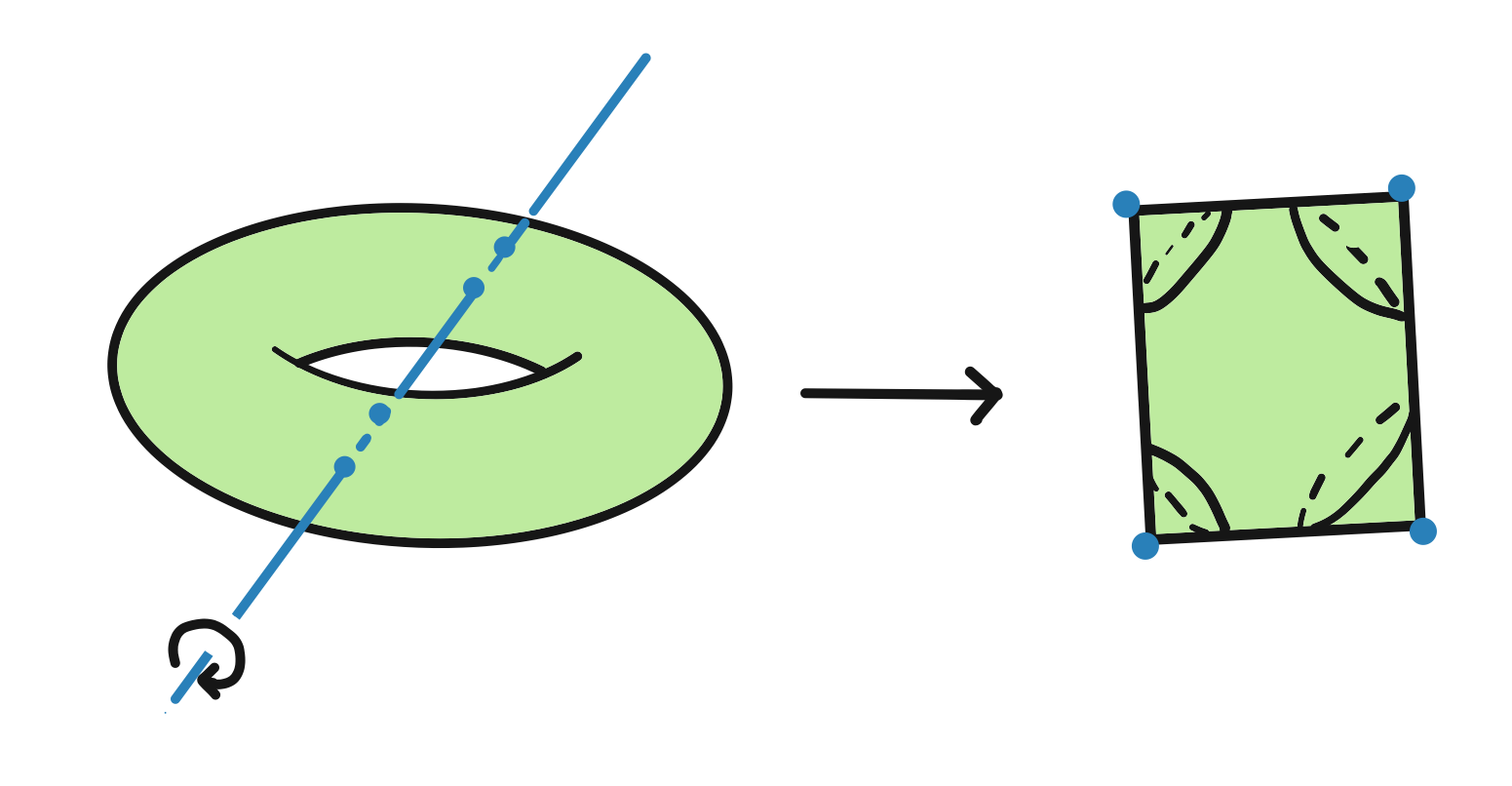}
\caption{Torus Branch cover of the sphere.}	
\end{figure}

\begin{example}
Consider the orbifold $T^2(n,n)$ and
let $\Z_2$ act on $\mathcal{O}$ freely by rotation sending one cone point to the other.
The quotient map $T^2(n,n)\to T^2(n)$ an orbifold covering map.
Note this cover does not \emph{unwrap} the cone point of $\mathcal{O}/\Z_2$ but rather \emph{doubles it}.
\end{example}

\begin{example}
Consider the orbifold $\S^2(n,n)$, and think of the cone points as being at the north and south poles.
Let $\Z_2$ act on $\mathcal{O}$ by a rotation about some axis through the equator, exchanging the cone points.
The quotient map is an orbifold covering $\S^2(n,n)\to\S^2(2,2,n)$.
This cover unwraps two of the cone points and doubles the other.
\end{example}

\begin{example}
If $\mathcal{O}$ is an orbifold with mirror singular locus $\Sigma$, there is a 2-fold cover $\widetilde{\mathcal{O}}$ of $\mathcal{O}$ obtained identifying two copies of $\mathcal{O}$ along the mirror singular locus. 
This is the \emph{local-orientation double cover}.	
\end{example}

\noindent
As for manifolds, we may define an orbifold version of \emph{universal covering space} as a cover which covers all other covers.

\begin{definition}
The \emph{orbifold universal cover} of an orbifold $\mathcal{O}$ is a cover $\pi\colon\widetilde{\mathcal{O}}\to\mathcal{O}$ such that if $p\colon\mathcal{Q}\to\mathcal{O}$ is any other cover, there is a covering map $r\colon \widetilde{\mathcal{O}}\to\mathcal{Q}$ such that $p\circ r=\pi$.
\end{definition}

\begin{observation}
Every orbifold has an orbifold universal cover; for a proof consult \cite{CooperHK00}, Theorem 2.9.
\end{observation}

\begin{definition}
The \emph{orbifold fundamental group} of an orbifold $\pi_1(\mathcal{O})$ is defined as the deck group of the universal covering $\pi_1(\mathcal{O})=\Aut(\widetilde{\mathcal{O}}\to\mathcal{O})$.
\end{definition}

\noindent
The alternative notation $\pi_1^{\mathsf{orb}}(\mathcal{O})$ is used when there is a risk of confusion with the fundamental group of the underlying space $\pi_1(X_\mathcal{O})$.
An orbifold is called \emph{good} if it is covered by a manifold.
In particular, the universal cover of a good orbifold is a manifold, and so good orbifolds are quotients of manifolds by properly discontinuous group actions.
An orbifold $\mathcal{O}$ is called \emph{very good} if it is \emph{finitely covered} by a manifold.

\begin{observation}
If $\mathcal{O}$ is a very good orbifold, $\mathcal{O}=M/\Gamma$ for $M$ a manifold and $|\Gamma|<\infty$	, then $\pi_1(\mathcal{O})$ is an extension of $\pi_1(M)$ by $\Gamma$.
\end{observation}

\noindent
There is a version of Van Kampen's theorem for orbifolds; splitting along a connected suborbifold realizes the orbifold fundamental group as an amalgamated free product of the components.

\begin{proposition}
If $\mathcal{O}$ is an orbifold and suborbifolds $\mathcal{O}_1,\mathcal{O}_2$ such that 
$\mathcal{O}=\mathcal{O}_1\cup\mathcal{O}_2$ and
 $\mathcal{O}_1\cap\mathcal{O}_2$ is connected, $\pi_1(\mathcal{O})=\pi_1(\mathcal{O}_1)\ast_{\pi_1(\mathcal{O}_1\cap\mathcal{O}_2)}\pi_1(\mathcal{O}_2)$.
\end{proposition}

\noindent
In particular, if $\mathcal{O}=\mathcal{O}_1\cup\mathcal{O}_2$ with $\mathcal{O}_2$ a simply connected manifold, then $\pi_1^\mathsf{orb}(\mathcal{O})=\pi_1^\mathsf{orb}(\mathcal{O}_1)$.
In particular, if the underlying space of $\mathcal{O}$ is simply connected and $\Sigma(\mathcal{O})=\set{x}$, then $\pi_1^\mathsf{orb}(\mathcal{O})=\mathsf{Is}(x)$.
This makes particularly simple the computation of orbifold $\pi_1$ in two dimensions.

\begin{observation}
\label{Obs:2Orbifold_Pi1}
Let $\mathcal{O}=S(n_1,\ldots n_r)$.  
Then $\pi_1(\mathcal{O})$ is a quotient of $\pi_1(S\smallsetminus \set{p_1,\ldots, p_r})$ by the relations that the loops $\gamma_i$ around the punctures $p_i$ have order $n$.  
If $\mathcal{O}$ has mirror reflector boundary, its mirror double is a $\Z_2$ cover of the form above.
\end{observation}

\begin{example}
The orbifold fundamental group of $\S^2(2,2,2,2)$ is a $\Z_2$-extension of the fundamental group of the torus, as $\S^2(2,2,2,2)=T^2/\Z^2$ as in Example \ref{Ex:Torus_Branched_Cover}.
Computing via the procedure above gives another presentation, as a quotient of the free group on 3 generators.
Letting $\alpha,\beta,\gamma,\delta$ be the loops about the punctures on a 4-punctured sphere (so $\alpha\beta\gamma\delta=1$), we have
$\pi_1(\S^2(2,2,2,2))=\langle \alpha,\beta,\gamma,\delta\mid \alpha^2=\beta^2=\gamma^2=\delta^2=\alpha\beta\gamma\delta=1\rangle$.
\end{example}

\begin{example}
An orbifold is simply connected if $\pi_1(\mathcal{O})$ is trivial.
Note this is different than having simply connected underlying space, as $\pi_1(\S^2(p,q,r))=\Delta(p,q,r)$ but $\pi_1(X_{\S^2(p,q,r)})=1$ as the underlying space is a sphere.
\end{example}

\noindent
The existence of universal covers and the definition of orbifold $\pi_1$ as the corresponding deck group allows standard results of covering space theory to carry over without change.
In particular, orbifold covers exhibit a \emph{Galois correspondence} with subgroups of their fundamental groups.

\begin{proposition}
Let $\mathcal{O}$ be an orbifold.  
Then there is a $1-1$ correspondence between covers of $\mathcal{O}$ and conjugacy classes of subgroups of $\pi_1(\mathcal{O})$: for each $\Gamma<\pi_1(\mathcal{O})$ there is some covering space $p\colon\mathcal{Q}\to\mathcal{O}$ with $p_\ast \pi_1(Q)$ conjugate to $\Gamma$.
\end{proposition}

\noindent
This allows us to prove that there are examples of orbifolds which do not arise as quotients of manifolds, although they do in each local model.

\begin{proposition}
If $\mathcal{O}$ be a simply connected orbifold.
Then $\mathcal{O}$ admits no nontrivial covers via the Galois correspondence, and so if $\mathcal{O}$ has nonempty singular locus, $\mathcal{O}$ is a \emph{bad orbifold}.
\end{proposition}

\noindent
Such examples already exist in dimension two.

\begin{example}
The \emph{teardrop orbifolds} $\S^2(n)$ are simply connected but have nonempty singular locus, and thus are bad.
\end{example}

\begin{example}
The \emph{spindle orbifolds} $\S^2(m,n)$ have fundamental group $\pi_1(\S^2(m,n))=\langle \gamma\mid \gamma^m=\gamma^n=1\rangle\cong\Z_{\gcd(m,n)}$.
When $n=m$ the corresponding cover is the sphere.
When $\gcd(m,n)=m$ the corresponding cover is a teardrop $\S^2(n/m)$, which is a bad orbifold.
In general the $\Z_{\gcd(m,n)}$ cover is another spindle $\S^2(m',n')$ with conepoints of coprime orders.
This is orbifold simply connected and so a bad orbifold.
Thus spindles are good orbifolds if and only if the cone points are of the same order.
\end{example}

\begin{proposition}
Every 2-dimensional orbifold which is not a teardrop or a bad spindle is good, and in fact very good.	
\end{proposition}

\noindent
The proof of this proposition is not difficult and relies on orbifold covering theory; for reference consult \cite{Scott80}.
The notion of Euler characteristic carries over to the category of orbifolds as well.
For very good orbifolds, we may simply extend the usual Euler characteristic for manifolds to continue to be multiplicative with respect to covers in the category of orbifolds.
An extension to general orbifolds can be created from this together with an extension of the usual relationship with connect sum.

\begin{definition}
The orbifold Euler characteristic is a $\Q$-valued function $\chi$ on the class of orbifolds, defined to extend the usual Euler characteristic of manifolds and satisfy the following: $\chi{\widetilde{\mathcal{O}}}=d\chi(\mathcal{O})$ if there exists a $d$-fold orbifold cover $\widetilde{\mathcal{O}}\to\mathcal{O}$ and $\chi(\mathcal{O})=\chi(\mathcal{O}_1)+\chi(\mathcal{O}_2)-\chi(\mathcal{O}_1\cap\mathcal{O}_2)$ when $\mathcal{O}_1\cup\mathcal{O}_2=\mathcal{O}$.
\end{definition}

\begin{example}
We compute the orbifold Euler characteristic of a surface $S$ with $r$ cone points of order $n_1,\ldots, n_r$ as follows.
A small neighborhood $U_i$ of each cone point is a good orbifold, $n_i$-fold covered by the disk; thus $\chi(U_i)=1/n_i$.
The complement of these disks is a surface of genus $g$ with $r$ punctures, and as a manifold $\chi(S_{g,r})=2-2g-r$.
The intersections of each disk neighborhood with the surface are circles, with manifold Euler characteristic zero.
Thus $\chi(\mathcal{O})=\chi(S_{g,n})+\sum_{i=1}^r\frac{1}{n_i}=2-2g-\sum_{i=1}^r 1-\frac{1}{n_i}$.

\end{example}

The Euler characteristic of an orbifold with mirror and corner reflectors can be doubled to give a locally orientable orbifold such as the above, then again using multiplicativity of covers its Euler characteristic is half that of its double.
This is a powerful tool for understanding the geometrization of orbifolds in dimension two.

To understand orbifolds a bit better it is useful to understand their singular loci.
One reduction theorem that is useful here is that we may without loss of generality consider local models based on $\R^n/\Gamma$ for $\Gamma<\O(n+1)$ instead of $\Gamma<\mathsf{Diffeo}(\R^n)$.

\begin{proposition}
If $\mathcal{O}$ is an orbifold and $x\in\Sigma(\mathcal{O})$ then there is a chart $(U,\mathbb{B}^n,\Gamma,\phi)$ with the action of $\Gamma$ on a ball in $\R^n$ by orthogonal transformations, $\mathsf{Is}(x)<\O(n+1)$.
\end{proposition}
\begin{proof}
Let $x\in\Sigma(\mathcal{O})$ and $(U,\tilde{U},\Gamma,\phi)$ be a local model containing $x$, and $\widetilde{x}\in\widetilde{U}$ a point covering $x$.
Choose a Riemannian metric on $\widetilde{U}$ and average by $\Gamma$ to get a $\Gamma$-invariant Riemannian metric $g$.
As $\widetilde{x}$ is fixed by the action of $\mathsf{Is}(x)$, the derivative of this action gives a representation $\mathsf{Is}(x)\to\GL(T_{\widetilde{x}}\widetilde{U})$; and 
as this action is by isometries this has image in the orthogonal group for $g_x$.
\end{proof}

\begin{corollary}
The local structure of the singular locus of a an $n$-dimensional orbifold is the cone on 	the singular set of an $n-1$-dimensional spherical orbifold (a quotient of $\S^{n-1}$ by isometries).
\end{corollary}
\begin{proof}
Let $x\in\Sigma(\mathcal{O})$ and $(U,\mathbb{B}^n,\Gamma,\phi)$ be a chart with an orthogonal local action as above.
Then $\Gamma$ preserves the concentric radial spheres in $\mathbb{B}^n$, and the quotient $\mathbb{B}^n/\Gamma$ is the cone on the quotient of $\S^{n-1}/\Gamma$.
The singular locus of $U\cong \mathbb{B}^n/\Gamma$ is thus the cone on the singular locus of $\S^{n-1}/\Gamma$.
\end{proof}

\noindent
This leads to a classification of 1- and 2-dimensional orbifolds, which we do not pursue here but state for reference.  Details may be found in \cite{Scott80}.

\begin{theorem}[Classification of 1-orbifolds]
The closed $1$ orbifolds are the circle $\S^1$ and the interval $I=[-1,1]$ with reflector boundary, arising as a quotient $\S^1/\Z_2$ by complex conjugation.
\end{theorem}

\begin{theorem}[Classification of 2-orbifolds]
Every locally orientable $2$-orbifold has underlying space a closed surface, together with a finite number of marked points (cone points) labeled by natural numbers $n_i>1$ (the order of the isotropy subgroups).
Non locally-orientable $2$-orbifolds have underlying space a surface with boundary, which is orbifold mirror singular locus, and in addition to marked points in the interior has finitely many marked points on the boundary (the corner reflectors) labeled by natural numbers $m_i>1$ (relating to the dihedral isotropy groups $D_{2m_i}$).
\end{theorem}

\noindent
In particular, each of these orbifolds has underlying space a topological manifold together with an additional orbifold structure.
This is not true in general, in higher dimensions the underlying space of an orbifold need not be a manifold as we have already seen in Example \ref{Ex:Cone_RP}.
Locally orientable $3$-orbifolds are easily classified, and all have underlying spaces a manifold.

\begin{observation}
The finite subgroups of $\SO(3)$ are infinite cyclic $\Z_n$, dihedral $\Delta(2,2,n)$, or the orientation-preserving symmetry groups of the platonic solids $\Delta(2,3,3),\Delta(2,3,4)$ or $\Delta(2,3,5)$.
The singular locus of $\S^2/\Gamma$ for $\Gamma$ in the list above is either two points with the same isotropy group $\Z_n$ or a triple of points with isotropy groups of orders $(2,2,n)$, $(2,3,3),(2,3,4)$ or $(2,3,5)$.
\end{observation}

\begin{theorem}
Locally orientable $3$-orbifolds $\mathcal{O}$ have underlying space $\mathcal{X}_\mathcal{O}$ a 3-manifold and singular locus $\Sigma(\mathcal{O})$ a 3-regular (possibly disconnected) graph $G\inject X_\mathcal{O}$ equipped with a an admissible labeling of edges: any three edges incident to a vertex are labeled $(2,2,n)$, $(2,3,3), (2,3,4)$ or $(2,3,5)$.
\end{theorem}

\begin{example}
Any link in $\S^3$, together with any labeling, is a 3-orbifold with only cone axis singularities.	
An arbitrarily knotted theta graph embedded in $\S^3$ labeled by an admissible triple gives a 3-orbifold.
\end{example}

\begin{figure}
\centering\includegraphics[width=0.85\textwidth]{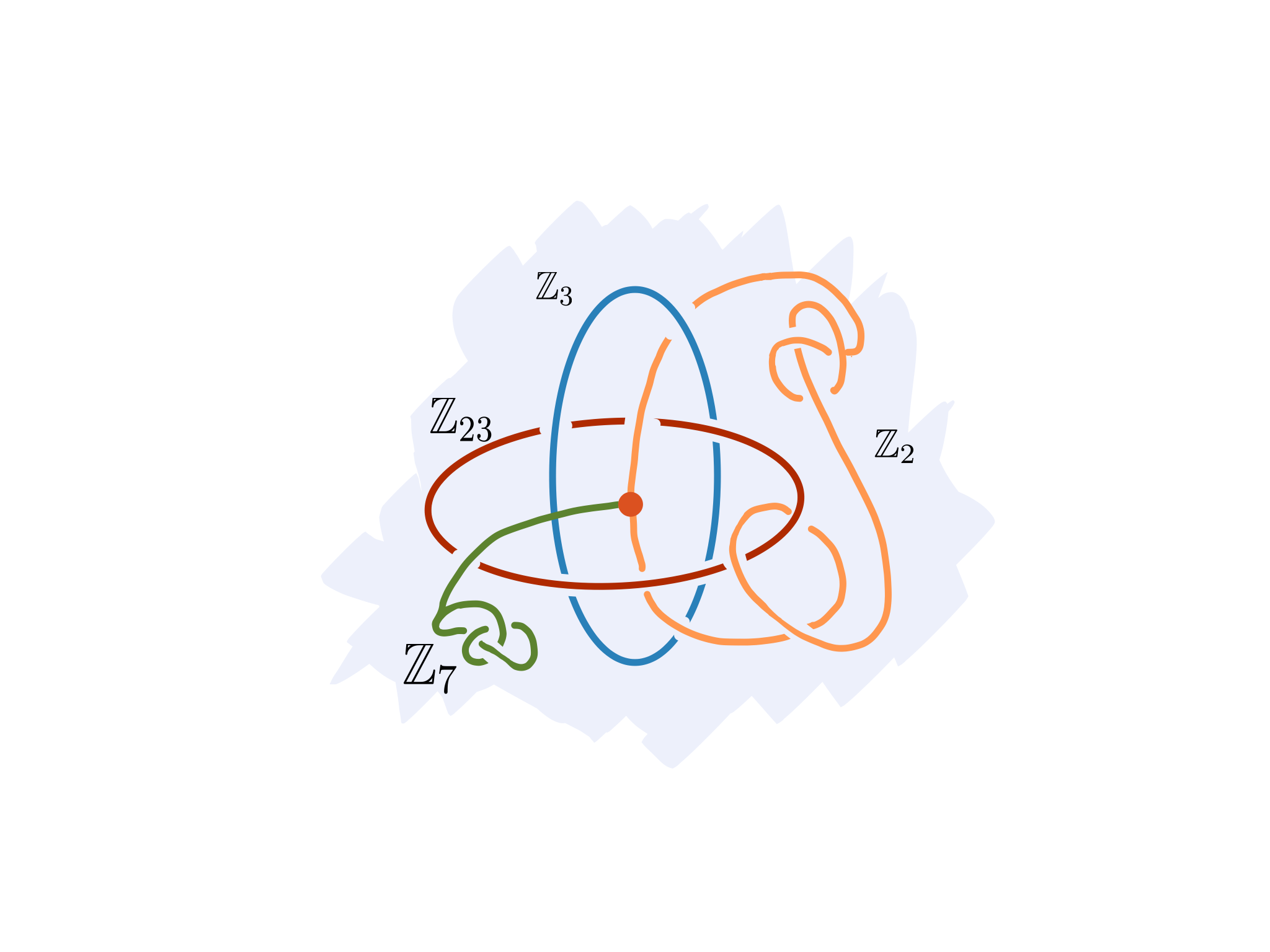}
\caption{An example 3-orbifold with underlying space $\S^3$ and singular locus labeled.}	
\end{figure}


\chapter{Klein Geometries}
\label{chp:Klein_Geo}
\index{Klein Geometry}
\index{Homogeneous Space}

Euclidean space is homogeneous, which means that it looks the same from every point.
More precisely, for any pair of points $p,q\in\E^n$, there is an isometry $\phi_{p,q}\colon\E^n\to\E^n$ such that $\phi_{p,q}(p)=q$.
Given some fixed basepoint $x\in \E^n$, this is implies the orbit of $x$ under $\Isom(\E^n)$ is all of $\E^n$; or that the automorphisms of Euclidean space act transitively.
In thinking about the foundations of geometry, Klein in his Erlagen Program suggested that a fruitful notion of \emph{geometry} more naturally is a direct generalization of this.
Geometries \emph{are} homogeneous spaces: manifolds equipped with a notion of 'rigid transformation' or \emph{automorphism}, which are symmetric enough that the group of automorphisms acts transitively.

Here we give two standard formalizations of homogeneous geometry, and treat the basic theory in detail.
We then discuss some useful notions of equivalence for geometries, and prove some basic results justifying the common practice of switching between different models at will.

\section{Perspectives on Homogeneous Geometry}
\label{sec:Perspectives_Homogeneous}

\subsection{The Group-Space Perspective}
\label{subsec:Grp_Sp_Geometries}
\index{Geometry!Group-Space}

Our first perspective on geometries formally encodes a homogeneous space for a Lie group $G$ by keeping track of the group, smooth manifold and action.

\begin{definition}
A geometry is a triple $(G,(X,x),\alpha)$ of a Lie group $G$ and pointed smooth manifold $(X,x)$	equipped with an analytic and transitive action $\alpha\colon G\times X\to X$.
Encoding geometries this way is called the \emph{Group-Space} perspective in this thesis.
\end{definition}

\noindent
By the transitivity of the $G$ action the particular choice of basepoint is immaterial and serves the technical purpose of selecting a canonical point stabilizer $G_x=\mathsf{stab}_{G}(x)$.
Both the basepoint $x\in X$ and the action map $\alpha$ are omitted from the notation when understood, and a geometry is denoted by the pair $(G,X)$.

\begin{example}
Spherical geometry is given by the linear action of $\SO(n+1)$ on the sphere $\S^n=V(x_1^2+\cdots+ x_{n+1}^2=1)\subset\R^{n+1}$.
Choosing a basepoint, say $p=(0,\cdots,0,1)$ gives the pointed geometry $(\SO(n+1),(\S^n,p))$.
\end{example}

\begin{observation}
\label{Obs:Unique_Extension}
Let $(G,X)$ be a geometry, and $g,h\in G$.
As the $G$ action on $X$ is analytic, 
if for any open $U\subset X$ the restricted actions $g.\colon U\to X$ and $h.\colon U\to X$ agree, then in fact $g=h$.
\end{observation}

\begin{definition}
\label{Def:Geo_Morphism}
A \emph{morphism of geometries} $(G,X)\to(H,Y)$ is a pair $(\Phi, F)$ consisting of a group homomorphism $\Phi\colon G\to H$ with $\Phi(G_x)<H_y$
 together with a $\Phi$-equivariant basepoint-preserving smooth map $F\colon (X,x)\to (Y,y)$.
 A morphism $(H,Y)\to (G,X)$ is an isomorphism if it has an inverse.
\end{definition}

\begin{example}[Klein and Poincare Models]
Let $\Hyp^2_\mathsf{K}$ be the Klein model of hyperbolic space, given by the projectivized linear action of $\SO(2,1)$ on the hyperboloid $\mathcal{H}=V(x^2+y^2-z^2+1)$.
Let $\Hyp^2_\mathsf{P}$ be the Poincare model, given by the action of $\SU(1,1)$ on the unit disk $\D^2=\set{z\mid \|z\|<1}\subset\C$ by linear fractional transformations.
These two geometries are isomorphic, with an explicit isomorphism given by two different projections of the hyperboloid model of hyperbolic space, as shown below.
\end{example}

\begin{figure}
\centering\includegraphics[width=0.85\textwidth]{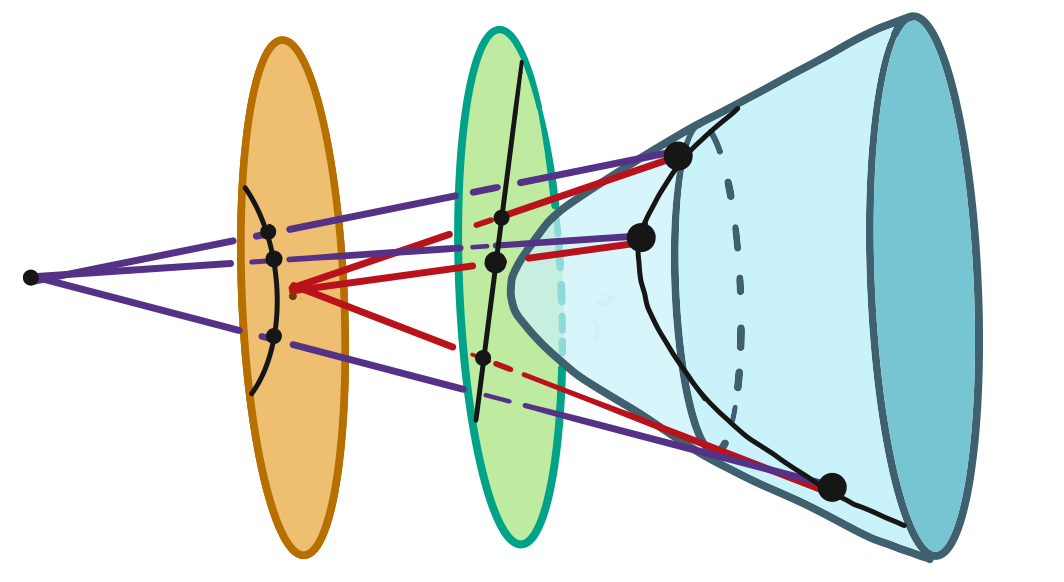}
\caption{The Hyperboloid, Klein Disk, and Poincare Disk models of hyperbolic space.}	
\end{figure}

\noindent
Given a notion of geometry and morphisms between them, we have formed the category of Klein geometries, from the group-space perspective.

\begin{definition}
The category of Klein geometries has as objects the homogeneous spaces $(G,(X,x),\alpha)$ and as morphisms the pairs $(\Phi,F)$ as in Definition \ref{Def:Geo_Morphism}.
\end{definition}

\subsection{The Automorphism-Stabilizer Perspective}
\label{subsec:Aut_Stab_Geometries}
\index{Geometry!Automorphism-Stabilizer}

Alternatively, a homogeneous $G$-space $X$ can be encoded purely algebraically, remembering only a point stabilizer $K$ of the action $G\acts X$ (the space can then be recovered up to diffeomorphism as $G/K$).
This gives an alternate definition of homogeneous space, and together with a corresponding notion of morphism, a different category of homogeneous geometries.

\begin{definition}
A geometry is a pair $(G,K)$ of a Lie group $G$ and a closed subgroup $K$.
Encoding geometries this way is called the \emph{Automorphism-Stabilizer} perspective in this thesis.
\end{definition}

\begin{example}
Spherical geometry is modeled by the automorphism group $\SO(n+1)$ together with the stabilizer of a point under its action on $\S^n$.  Taking the point to be $p=(0,\ldots, 0,1)\in\S^n$ gives $\mathsf{stab}(p)=\smat{\SO(n)&0\\0&1}$.  Abusing notation and calling this $\SO(n)$, we may describe the geometry of the $n$-sphere in the Automorphism-Stabilizer formalism as $(\SO(n+1),\SO(n))$.	
\end{example}

\noindent
Note that we do not require that the closed subgroup be compact, as this is not necessary for stabilizers of homogeneous geometries; the stabilizer of a point in the affine plane is isomorphic to $\GL(2;\R)$ for example.

\begin{definition}
A morphism $\Phi\colon(H,C)\to(G,K)$ of geometries from the Automorphism-Stabilizer perspective is a Lie group homomorphism $\Phi\colon H\to G$ such that $\Phi(C)<K$.
\end{definition}

\begin{example}[Klein and Poincare Models]
From the Automorphism - Stabilizer perspective, the Klein model of hyperbolic space is the pair $(\SO(2,1), S)$ for $S=\smat{\SO(2)&0\\0&1}<\SO(2,1)$.
The Poincare model is given by the pair $(\SU(1,1),\SO(2))$ of subgroups of $\GL(2;\R)$.
\end{example}

\begin{definition}
The category of Klein geometries has as objects the homogeneous spaces $(G,K)$ and as morphisms the Lie group homomorphisms-of-pairs $\Phi:(H,C)\to(G,K)$ as above.
\end{definition}

\subsection{Equivalence}
\label{subsec:GrpSP_AutStab_Equivalence}

Each of these perspectives is useful to have available at times, and it is of little surprise given their definitions that they encode precisely the same information.
In this section we record this fact precisely, by constructing an equivalence of categories between the category of geometries from the Group-Space perspective, (denoted here $\mathsf{GrpSp}$) and the category of geometries constructed from the Automorphism-Stabilizer perspective (denoted $\mathsf{AutStb}$).

\begin{lemma}
The map $\cat{F}:\cat{GrpSp}\to\cat{AutStb}$ sending a group-stabilizer geometry $(G,K)$ to the group-space geometry $(G,(G/K,K))$ defines a functor.
\end{lemma}
\begin{proof}

As $K$ is a closed subgroup of $G$, the $K$ action on $G$ by left translation by is a free and proper action.  
Thus by the quotient manifold theorem of smooth topology \cite{Lee}, the orbit space $G/K$ is a smooth manifold.
The action of $G$ on $G/K$ is just the usual action of $G$ on itself followed by the quotient map, which is transitive and thus defines a geometry of the Group-Space variety.  
The inclusion $K\inject G/K$ provides a cannonical choice of basepoint.
Given a morphism $\Phi:(H,K)\to (G,C)$ we define $\cat{F}(\Phi)=(\Phi, \bar{\Phi})$ where $\bar{\Phi}(gC)=\Phi(g)K$.
This is $\Phi$-equivariant and well-defined as $\Phi(C)\subset K$.  
\end{proof}

\begin{lemma}
The map $\Psi\colon \cat{Grp-Sp}\to\cat{AutStb}$ sending a geometry $(G,(X,x))$ to $(G,\mathsf{stab}_G(x))$ defines a functor.	
\end{lemma}
\begin{proof}
The stabilizer of an analytic action of a Lie group on a smooth manifold is a closed Lie subgroup.
Thus $(G,\mathsf{stab}_G(x))$ is a geometry of the group-stabilizer variety.
Recalling that a morphism $\Phi:(G,(X,x))\to(H,(Y,y))$ consists of a group homomorphism $\Phi_{\cat{Grp}}$ and an equivariant map $\Phi_{\cat{Sp}}$ 
between the spaces, the image $\Psi(\Phi)=\Phi_{\cat{Grp}}$ is simply the group homomorphism, which is well-defined as $\Phi_\cat{Sp}\circ\fam{x}=\fam{y}$ together 
with equivariance implies that $\Phi_\cat{Grp}(\mathsf{stab}_G(x))\subset\mathsf{stab}_H(y)$.   
\end{proof}

\begin{proposition}
The functors $\cat{F},\Psi$ above define an equivalence of categories $\cat{GrpSp}\cong\cat{AutStb}$.	
\end{proposition}
\begin{proof}

The composition $\Psi \cat{F}$ is the identity on $\cat{AutStb}$, and the composition $F\Psi$ takes the geometry 
$(G,(X,x))$ to 
$(G,(G/\mathsf{stab}_G(x)),\mathsf{stab}_G(x))$.  

The collection of maps $\eta|_{(G,X)}\colon \left(G,(X,x)\right) \to \left( G, (G/\mathsf{stab}_G(x),\mathsf{stab}_G(x))\right) $ forms a natural transformation from $\mathsf{id}_{\cat{GrpSp}}$ to $\cat{F}\Psi$.
In more detail, $\eta$ is
given by $\eta=(\mathsf{id}_G,\xi_{(G,X)})$ where $\xi_{(G,X)}$ assigns to a point $p\in X$ the coset $g\mathsf{stab}_G(x)$ of the basepoint stabilizer, for $g$ such that $\mathsf{stab}_G(p)=g\mathsf{stab}_G(x)g\inv$.

 To see this it suffices to check that $\bar{\Phi_\cat{Grp}}\circ\xi_{(G,X)}=\xi_{(H,Y)}\circ\Phi_\cat{Sp}$.  
Let $p\in X$ and $g\in G$ be such that $g.x=p$.  Then $\xi_{(G,X)}(p)=g \mathsf{stab}_G(x)$ and $\bar{\Phi_\cat{Grp}}(g\mathsf{stab}_G(x))=\Phi_\cat{Grp}(g)\mathsf{stab}_H(y))$.
Computing the other way around we find $\Phi_\cat{Sp}(p)=\Phi_\cat{Sp}(g.x)=\Phi_\cat{Grp}(g)\Phi_\cat{Sp}(x)=\Phi_\cat{Grp}(g)y$ and $\xi_{(H,Y)}(\Phi_\cat{Grp}(g)y))=\Phi_\cat{Grp}(g)\mathsf{stab}_H(y)$.

\begin{center}
\begin{tikzcd}
\left(G,(X,x)\right) 
\arrow{d}[swap]{(\Phi_{\cat{Grp}},\Phi_{\cat{Sp}})}
\arrow{rr}{(\mathsf{id}_G,\xi_{(G,X)})}
&& 
\left(G, (G/\mathsf{stab}_G(x),\mathsf{stab}_G(x))\right) 
\arrow{d}{(\Phi_\cat{Grp},\bar{\Phi_\cat{Grp}})} 
\\
\left(H,(Y,y)\right)
\arrow{rr}[swap]{(\mathsf{id}_H,\xi_{(H,Y)})}
&& \left( H, (H/\mathsf{stab}_H(y),\mathsf{stab}_H(y))\right) 
\end{tikzcd}
\end{center}
	
\end{proof}

\noindent
Thus we are justified in moving freely between these perspectives at will when convenient.
In particular, we feel free to define a concept for whichever notion of geometry it is more convenient to do so, and leave it to the reader to transport this definition to the other formalism if desired.

\section{Notions of Equivalence}
\label{sec:Notions_Equivalence}

Oftentimes it is advantageous to be slightly looser with our notion of isomorphism for geometries than what arises from the above definitions.
In particular, there are two common situations where we may want to think of two geometries as being `essentially the same,' even when the groups or spaces differ slightly.
The first case involves a trade-off between two 'good' properties that the automorphism group of a geometry could enjoy.

\subsection{Effective Geometries}
\index{Geometry!Effective}

\begin{definition}
A geometry $(G,X)$ is effective if $g.x=x$ for all $x\in X$ implies that $g=e$.  
That is, the only element of $g$ acting trivially on all of $X$ is the identity.
\end{definition}

Equivalently, a geometry $(G,X)$ is effective if the induced homomorphism $G\to \mathsf{Diffeo}(X)$ given by the action is faithful.
\emph{Effectiveness} is a property of $(G,X)$ geometries, capturing that the action of each element of $G$ on $X$ is distinct.
Oftentimes it is useful to consider \emph{non-effective versions} of a geometry, corresponding to choices of groups $\widetilde{G}$ which surject onto $G$ as automorphisms.
One reason for doing so is that the effective geometry $(G,X)$ has a difficult-to-work with automorphism group, but $G$ is covered by a nice (say, linear) group $\widetilde{G}$.
This allows us to work with matrices, at the cost of dealing with a non-effective action.

\begin{example}
The geometry $(\PGL(3;\R),\RP^2)$ is the effective version of projective geometry in dimension two.
In practice, it is often easier to work with the non-effective versions $(\SL(3;\R),\RP^2)$ or even $(\GL(3;\R),\RP^2)$.
\end{example}

\begin{definition}
Two geometries $(G,X)$ and $(H,X)$ are \emph{effectively equivalent} if the action of $G$ on $X$ and the action of $H$ on $X$ induce homomorphisms $G\to\mathsf{Diffeo}(X)$, $H\to\mathsf{Diffeo}(X)$ with the same image.	
\end{definition}

\noindent
Given any geometry $(G,X)$, it is clear from the above definition that there is a unique effective geometry equivalent to it: if $\Phi\colon G\to\mathsf{Diffeo}(X)$ is the map induced by the action, then $(\Phi(G),X)$ is effective and equivalent to $(G,X)$. 
Denote by $\ker(G,X)$ the subgroup of $G$ acting trivially on $X$.
There is a natural map sending any geometry $(G,X)$ to its corresponding effective version, called \emph{effectivization}, sending $(G,X)$ to $(G/\ker(G,X),X)$.
This is used implicitly to justify passing freely between effective and non-effective versions of the same geometry when convenient in much of the literature.

\begin{observation}
The effectivization map $\cat{Eff}\colon \cat{GrpSp}\to\cat{GrpSp}$ defined by $(G,(X,x))\mapsto (G/\ker(G,X),(X,x))$ is a natural transformation between the identity on $\cat{GrpSp}$ and the effectivization endofunctor.
\end{observation}

\subsection{Local Morphisms}
\index{Geometry!Local Isomorphism}

The second notion of equivalence between geometries that is often useful to consider is \emph{local isomorphism}.
This is most naturally motivated by wanting to pass between a geometry and covers of that geometry when convenient.

\begin{example}
The geometry of the sphere is given by $(\SO(3),\S^2)$.
The action of $\SO(3)$ is equivariant with respect to the antipodal map and so we may use this $\SO(3)$ action to define a geometry on the quotient, $(\SO(3),\RP^2)$.
Locally, this geometry is similar to the geometry of the sphere.
\end{example}

\noindent
We formalize this notion of 'being the same on a small enough subset' via the concept of a \emph{local morphism}.

\begin{definition}
A local map $X\dashrightarrow Y$ is a map from some open set $U\subset X$ into $Y$.
A local homomorphism $\phi\colon G\dashrightarrow H$ is a local map defined on a neighborhood $U\ni e$ such that $\phi(gh)=\phi(g)\phi(h)$ and $\phi(g\inv)=\phi(g)\inv$ whenever all terms are defined.
\end{definition}

\noindent
A local homomorphism is injective if it is injective as a map of sets when restricted to some sufficiently small neighborhood of the identity.
It is locally surjective if the image contains some open set of the identity of the target group, and a local isomorphism if it is both locally injective and surjective.
Local morphisms are conveniently captured by Lie algebra maps, as in the following observation.

\begin{observation}
If $\phi\colon G\dashrightarrow H$ is a local morphism, its derivative $\phi_\ast\colon \mathfrak{g}\to\mathfrak{h}$ is a morphism of Lie algebras.
Conversely, any Lie algebra morphism $\psi\colon\mathfrak{g}\to\mathfrak{h}$ induces a local morphism $\Psi\colon G\dashrightarrow H$ defined on $\exp(\mathfrak{g})\subset G$	
\end{observation}

\noindent
Here we take advantage of the above observation, and the equivalence of categories $\cat{GrpSp}\cong\cat{AutStb}$ to succinctly define the equivalence relation of \emph{local isomorphism} between geometries.

\begin{definition}
A local morphism of geometries $(G,K)\to (H,C)$ is a morphism of Lie groups $\phi\colon \mathfrak{g}\to\mathfrak{h}$ such that $\phi(\mathfrak{k})\subset\mathfrak{c}$.
Two geometries $(G,K)$ and $(H,C)$ are locally isomorphic if there is an isomorphism of Lie algebras $\phi\colon\mathfrak{g}\to\mathfrak{h}$ carrying $\mathfrak{k}$ to $\mathfrak{c}$.	
\end{definition}

\noindent
Unpacking this in the more traditional Group-Space formalism gives the following.

\begin{definition}
A local morphism of geometries $(G,(X,x))\dashrightarrow (H,(Y,y))$ is a local homomorphism $\phi\colon G\dashrightarrow H$ such that the restriction of $\phi$ to $G_x$ is a local morphism $G_x\dashrightarrow H_y$.
This local homomorphism induces a local analytic map $f\colon X\dashrightarrow Y$ defined on a neighborhood of $x$ which is locally $\phi$-equivariant, meaning that $f(g.p)=\phi(g).f(p)$ whenever all terms are defined.
The local morphism is a local isomorphism if $\phi$ is, and additionally $\phi|_{G_x}$ is a local isomorphism $G_x\dashrightarrow H_y$.
\end{definition}

\noindent
Under this notion of equivalence, $(\SO(3),\S^2)$ and its quotient $(\SO(3),\RP^2)$ are locally isomorphic geometries.
We will not freely identify geometries up to local isomorphism as we have done with the previous notions of equivalence, but we will often abuse notation and call a geometry such as $(\SO(3),\RP^2)$ \emph{spherical geometry}, instead of the more correct \emph{a subgeometry of $\RP^2$ locally isomorphic to spherical geometry}.

\section{Properties of Klein Geometries}
\label{sec:Prop_Klein}

This brief section covers some additional miscellaneous terminology that proves useful when discussing homogeneous spaces.

\subsection{Subgeometries and Fibered Geometries}

Spherical geometry of dimension $n$ can be modeled exactly within Euclidean space of one dimension higher: take any round sphere $S\subset\E^{n+1}$, and the subgroup $G<\Isom(\E^{n+1})$ fixing that sphere set-wise is isomorphic to $\SO(n+1)$, acting transitively and thus making $(G,S)$ a model of spherical geometry.
Similarly hyperbolic $n$-space naturally arises as a codimension 1-subset of Minkowski space (a hyperboloid of 2 sheets orthogonal to the time-like axis), with isometries a subset of the automorphisms of $\M^{n+1}$.
In general such constructions are \emph{subgeometries} of the ambient space.

\begin{definition}
A subgeometry $(H,Y)$ of a geometry $(G,X)$ is a	 closed subgroup $H<G$ acting transitively on a subset $Y\subset X$.
Alternatively, a subgeometry of $(G,X)$ is the image of a monomorphism $\iota\colon (H,Y)\to (G,X)$.
\end{definition}

\noindent
Alternatively, we say that $(G,X)$ is a \emph{supergometry of} or a \emph{containing geometry for} the geometry $(H,Y)$.
Oftentimes we are interested in a more narrowly defined collection of \emph{open subgeometries}.

\begin{definition}
An open subgeometry of $(G,X)$ is a geometry $(H,Y)$ with $H<G$ closed and $Y\subset X$ open.	
\end{definition}

\begin{example}
The Klein ball model of hyperbolic space is an open subgeometry of $\RP^n$, but the hyperboloid model is not an open subgeometry of $\M^{n+1}$.	
\end{example}

\noindent
Dually to the notion of a subgeometry is that of a \emph{fibered geometry}, or a geometry $(G,X)$ which \emph{fibers over} a geometry $(H,Y)$.
These are the epimorphisms, as opposed to the monos, in the category of geometries.

\begin{definition}
A geometry $(G,X)$ fibers over a geometry $(H,Y)$ if there is an epimorphism of geometries $\pi\colon (G,X)\to (H,Y)$.  That is, a submersion of spaces $X\to Y$ equivariant with respect to a submersion of Lie groups $G\to H$.
\end{definition}

\noindent
Basic examples of fibered geometries are the products, but more interesting examples occur as degenerations when studying limits of geometries.

\begin{example}
$\Hyp^2\times\R$ fibers over $\Hyp^2$.
Heisenberg geometry, $(\Heis,\R^2)$ is given by the projective action of the real Heisenberg group on the plane, acting by all translations, and shears parallel to a fixed line.
This geometry fibers over the Euclidean line by quotienting the direction of shear.
\end{example}

\subsection{Metric Geometry}

Nowhere in the definition of homogeneous geometry is there a requirement that there exists some invariant metric, only that there is a transitive group of automorphisms.
Oftentimes of course there is such a metric, such as in Euclidean, hyperbolic and spherical geometry.
But there are many cases without as well.
For example, Minkowski, de Sitter, and Anti-de Sitter space are homogeneous spaces admitting a Lorentzian metric, but no invariant Riemannian metric.
Moreover, real projective geometry $(\SL(n+1;\R),\RP^n)$ admits no invariant pseudo-Riemannian metric of any signature.
Clearly nothing can be said about metrics and homogeneous geometry in any generality, but we record a few useful observations below.

\begin{lemma}
If a geometry $(G,X)$ admits a Riemannian metric, it has constant scalar curvature.
\end{lemma}
\begin{proof}
Let $p\in X$ have scalar curvature $k$, and $q\in X$ be any other point.
There is isometry $g\in G$ such that $g.p=q$, and this sends 2-planes through $p$ to 2-planes through $q$ of the same section al curvature.
Thus the scalar curvature at $q$, defined as the integral average of the sectional curvatures through all 2-planes at $q$, is also equal to $k$.
\end{proof}

\begin{observation}
The sectional curvature, or even Ricci curvature of a homogeneous space need not be constant: consider  $\Hyp^2\times\S^2$ for example.
The sectional curvature of a geometry $(G,X)$ is only forced to be constant if $G$ acts transitively on the bundle of 2-planes over $TX$.
\end{observation}

\noindent
If a homogeneous geometry admits a Riemannian geometry it need not be unique (for instance, the homogeneous space $\Hyp^n=(\SO(n,1),\mathbb{B}^n)$ admits an invariant metric of constant curvature $\kappa$ for each $\kappa<0$), and there are few instances in which actually utilizing the invariant metric is required (though we will have some use for it in Chapter \ref{chp:Heisenberg_Plane}).
Nonetheless, there is a very quick check to tell if a given homogeneous geometry admits \emph{some} invariant Riemannian metric; for a proof consult Thurston's book \cite{Thurston80}.

\begin{proposition}
Let $(G,(X,x))$ be a homogeneous geometry with point stabilizer $K=\mathsf{stab}_G(x)$.
Then $X$ admits a $G$-invariant Riemannian metric if and only if the  the image of $K\inject \GL(T_xX)$ given by $k\mapsto dk_x$ has compact closure.
In particular, any geometry with compact point stabilizers admits an invariant Riemannian metric.
\end{proposition}

\section{Examples}
\label{sec:Examples_Klein_Geos}

This section details many of the common examples of $(G,X)$ geometries, especially those relevant to this thesis.
When a particular geometry from the list below is mentioned in the following chapters, it will be assumed to be the particular model specified below, when such a distinction is relevant and unless otherwise specified.

\begin{example}
\emph{Real Projective Space}, $\RP^n$ is the $(G,X)$ geometry usually given by the projective action of $\PSL(n+1;\R)$ on $\RP^n$.
The alternative presentations, with automorphisms group $\SL(n+1;\R)$ or $\GL_+(n+1;\R)$ are not effective as multiplies of the identity act trivially on $\RP^n$.
Allowing for transformations of determinant $-1$ gives the locally isomorphic geometry $(\PGL(n+1;\R),\RP^n)$ or its (potentially) non-effective forms  $(\GL(n+1;\R),\RP^n)$ or $(\SL^\pm(n+1;\R),\RP^n)$.
The universal covering geometry is \emph{positive projective space}.
\end{example}

\begin{example}
Positive projective space, or $\widetilde{\RP^n}$, is given by the action of $\SL(n+1;\R)$ on $\S^n=\set{x\in\R^{n+1}\mid \|x\|=1}$.
\end{example}

\begin{example}
\emph{Spherical Space}, $\S^n$ is the $(G,X)$ geometry given by $(\SO(n+1),\S^n)$ for $\SO(n+1)=\set{A\in\GL(n+1;\R)\mid A^TA=I}$ and $\S^n=\set{x\in\R^{n+1}\mid \|x\|=1}$.
Allowing for orientation reversing isometries gives the locally isomorphic geometry $(\O(n+1),\S^n)$.
\end{example}

\begin{example}
\emph{Elliptic Space}, $\S^n$, is the $(G,X)$ geometry given by $(\PSO(n+1),\RP^n)$.
More commonly we work with the model $(\SO(n+1),\RP^n)$ which is not effective in odd dimensions.
This geometry is locally isomorphic to spherical space as defined above via the $2:1$ covering projection, which is its universal cover.
Following convention, we will often refer to this model as \emph{spherical space} as well.
\end{example}

\begin{example}
\emph{Affine Space}, $\A^n$ is the $(G,X)$ geometry given by the effective action of the \emph{affine group} $\mathsf{Aff}(n)=\GL(n;\R)\rtimes\R^n$ on $\R^n$.  
The usual model is as a subgeometry of projective space, with $\A^n$ the affine patch $\A^n=\set{[x_1:\cdots x_n: 1]}\subset\RP^n$ acted on by $\mathsf{Aff}(n)=\smat{\GL(n;\R)&\R^n\\0&1}$.
This model \emph{is} effective, and has orientation preserving automorphisms given by the index two subgroup $\mathsf{Aff}_+(n)=\GL_+(n;\R)\rtimes\R^n$.
\end{example}

\begin{example}
\emph{Euclidean Space} $\E^n$ is the $(G,X)$ geometry given by the effective action of the Euclidean group $\Euc(n)=\SO(n)\rtimes\R^n$ on $\R^n$.
This is a subgeometry of affine space, and so admits a similar projective model with $\Euc(n)=\smat{\SO(n)&\R^n\\0&1}$ and underlying space the affine patch $\set{[x_1:\cdots x_n: 1]}\subset\RP^n$.
Allowing reflections gives the locally isomorphic geometry $(\O(n)\rtimes\R^n,\R^n$).
\end{example}

\begin{example}
\emph{Similarity Geometry} is a weakening of Euclidean space to allow homotheties as geometric transformations with effective automorphism group $\Sym(n)=\R_+\times\Euc(n)$.  This is also a subgeometry of affine space with projective model $\Sym(n)=\smat{\R_+\SO(n)&\R^n\\0&1}$ acting on the affine patch.
\end{example}

\begin{example}
\emph{Hyperbolic Space} is the $(G,X)$ geometry given by $(\SO(n,1),\Hyp^n)$ for $\Hyp^n$ projectivization of the hyperboloid $V(x_1^2+\cdots x_n^2-x_{n+1}^2=-1)$, which is the unit disk in the affine patch $x_{n+1}\neq 0$.
This model is not effective as $\SO(n,1)$ has two components, switching the two sheets of the hyperboloid which are identified under projectivization; the effective model has automorphisms $\PSO(n,1)$.
Allowing for orientation reversing automorphisms extends the isometry group to $\O(n,1)$ or its effective version $\PSO(n,1)$.
\end{example}

\begin{example}
In dimension 2, there are two additional models of the hyperbolic plane which will be of use, arising as subgeometries of $\CP^1$.
The Poincare disk is a model of $\Hyp^2=(\SU(1,1),\D^2)$ with underlying space the unit disk in $\C$ and automorphism group $\SU(1,1)$ acting effectively on $\D^2$ by linear fractional transformations.
The M\"obius transformation $\smat{i&i\\-1&1}$ maps the disk to the upper half plane, conjugating $\SU(1,1)$ to $\SL(2,\R)$ and giving the \emph{upper half plane model} of $\Hyp^2=(\SL(2,\R),\R^2_+)$.
These models are not effective, as $\smat{-1&\\&-1}$ acts as the identity; the effective versions have automorphism groups $\PSU(1,1)$ and $\PSL(2,\R)$.
\end{example}

\begin{example}
\emph{Minkowski Space} is the Lorentzian analog of Euclidean space, given by the action of the Poincare group $\mathsf{Poin(n)}=\SO(n,1)\rtimes\R^n$ on $\R^n$.  This admits a projective model with $\R^n$ the affine patch $x_{n+1}\neq 0$ and automorphisms $\smat{\SO(n,1)&\R^n\\0&1}$.
\end{example}

\begin{example}
De Sitter space is the complement of the Klein model of hyperbolic space in $\RP^n$.  This is given by the projectivization of a hyperboloid of one sheet, 
$\mathsf{dS}^n=(\SO(n,1),V(x_1^2+\cdots x_n^2-x_{n+1}^2=1))$.
\end{example}

\begin{example}
Anti-de Sitter space is the Lorentzian analog of hyperbolic space, in the sense that we form a signature $(n,1)$ space of negative curvature by embedding it as a sphere of radius $-1$ in the space of signature $(n,2)$.
That is , $\mathsf{AdS}^n=(\SO(n-1,2),$.
\end{example}

\begin{example}
Heisenberg geometry is the $(G,X)$ geometry $\Hs^2:=(\Heis,\A^2)$ where
$\Heis$ is the real Heisenberg group.
\end{example}

\begin{example}
\emph{Complex Projective Space} is the $(G,X)$ geometry $(\SL(n+1;\C),\CP^n)$ with the automorphism group $\SL(2,\C)$ acting projectively.  This is a simply connected geometry with effective version $\PSL(n+1;\C)$.
\end{example}

\begin{example}
\emph{Unitary geometry} is a strengthening of complex projective geometry, acting on the underlying space $\CP^n$ only by unitary transformations $\SU(n+1;\C)\subset\SL(n+1;\C)$.
\end{example}

\begin{example}
\emph{Complex Hyperbolic Space} is the $(G,X)$ geometry with $G=\SU(n,1;\C)$ acting on the complex projectivization $X=\P V$ for $V$ the real algebraic variety $V=V(x_1\overline{x_1}+\cdots+x_n\overline{x_n}-x_{n+1}\overline{x_{n+1}}=-1)$, the analog of the hyperboloid model of hyperbolic space.
This is not effective, and $\SU(n,1;\C)$ $n+1$-fold covers $\PSU(n,1;\C)$ the effective automorphism group.
\end{example}

\chapter{Geometric Structures}
\label{chp:Geo_Strs}
\index{Geometric Structures}

A $(G,X)$ structure on a manifold $M$ locally identifies $M$ with small patches of the geometry $(G,X)$ in a compatible way.
Geometric structures are a direct generalization of \emph{smooth structures}, which themselves are a specialization of the notion of \emph{topological manifold} defined via an atlas of charts.
Below we review this atlas-and-transition approach to defining geometric manifolds, followed by the more modern approach via \emph{developing pairs}.
We then review the \emph{deformation space} and \emph{moduli space} of $(G,X)$ structures on a manifold, which directly generalize the familiar  Teichm\"uller spaces for Riemann surfaces.
Finally, we consider how different geometric structures modeled on different geometries can interact - through \emph{strengthening} and \emph{weakening} as well as \emph{degeneration} and \emph{regeneration}.

\section{Charts and Atlases}
\label{sec:Charts_Atlases}
\index{Manifolds!Charts}
\index{Manifolds!(G,X) Manifolds}

To formalize the notion that a manifold $M$ should `locally look like $\R^n$ we require that each point of $M$ has a neighborhood homeomorphic to some open subset of $\R^n$.
Likewise, each point in a \emph{hyperbolic manifold} $M$ should have a neighborhood isometric to some open subset of $\Hyp^n$.
Writing this down precisely leads directly to the definition of an \emph{atlas of charts} for $M$ together with differing \emph{compatibility conditions} depending on the topological/smooth/geometric structure to be imparted on $M$.

We may (roughly) think of the \emph{topological space} $\R^n$ as a $(G,X)$ geometry with underlying space $\R^n$ and allowable transformations given by all self-homeomorphisms $\mathsf{Homeo}(\R^n)$.
Similarly, the smooth geometry of $\R^n$ has underlying space $\R^n$ and automorphism group the collection of all diffeomorphisms $\mathsf{Diffeo}(\R^n)$.
From this perspective, a \emph{topological manifold} is a topological space equipped with an atlas of charts into $\R^n$, with transition maps in $\mathsf{Homeo}(\R^n)$, and a \emph{smooth manifold} is given by an atlas of $\R^n$-valued charts with transition maps in $\mathsf{Diffeo}(\R^n)$.
This rephrasing of the above definitions suggests an immediate generalization to structures modeled on any homogeneous space $(G,X)$.

\begin{definition}
Let $(G,X)$ be a geometry and $M$ a topological manifold.
A $(G,X)$ structure on $M$ is a maximal atlas of $X$-valued charts on $M$ with transition maps in $G$.
\end{definition}

\begin{observation}
A $(G,X)$ manifold $M$ has an underlying real analytic structure as the action of $G$ on $X$ is analytic by definition.
\end{observation}

\begin{figure}
\centering\includegraphics[width=0.5\textwidth]{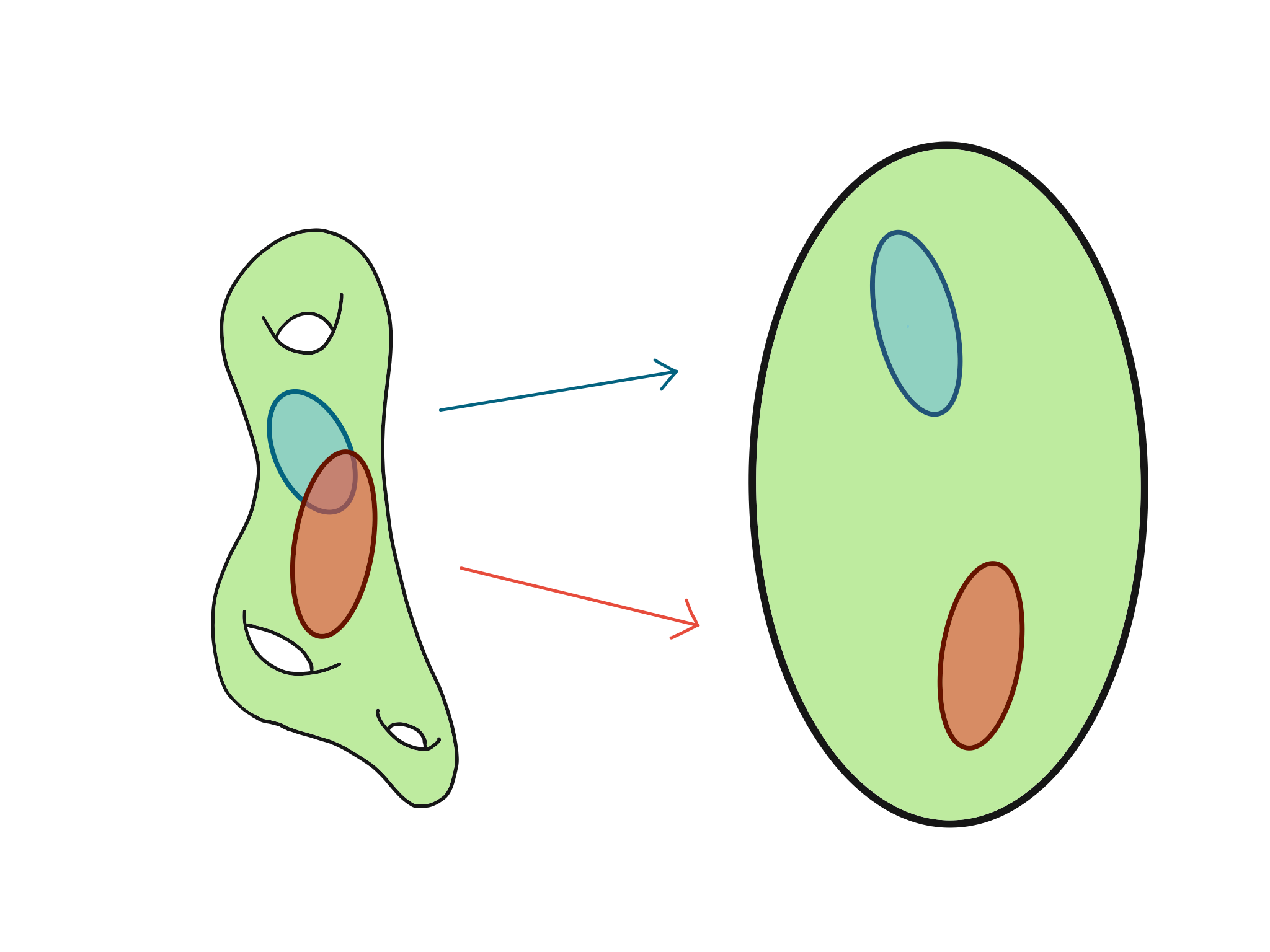}
\caption{An atlas of charts on a hyperbolic surface.}	
\end{figure}

\noindent There is a slight technical annoyance when discussing transition maps for potentially disconnected intersections $U_\alpha\cap U_\beta$ which we address presently.
Given a subset $U\subset X$, a map $f\colon U\to X$ is said to be \emph{locally-$G$} if the restriction $f|_{U_i}$ to each connected component $U_i\subset U$ agrees with the action of some element $g\in G$ restricted to $U_i$.
Such an $f$ is in the \emph{pseudogroup generated by $G$}; see Thurston's book \cite{Thurston80} for example.
Following convention we abuse terminology and say such a map is \emph{in $G$}, as in the definition above.

\begin{example}
The first example of a $(G,X)$ manifold is $X$ itself, with the single chart $\mathsf{id}_X\colon X\to X$.	
\end{example}

\noindent As in the topological case, a $(G,X)$ atlas of charts on $M$ allows us to reconstruct $M$ out of little pieces of $X$, described below.
Let $V=\coprod_\alpha \phi_\alpha(U_\alpha)$ be the disjoint union of images under charts, and define the equivalence relation $\sim$ on $V$ as follows.
If $x\in \phi_\alpha(U_\alpha\cap U_\beta)$ and $y\in \phi_\beta(U_\alpha\cap U_\beta)$, then $x\sim y$ if $\phi_\beta\phi_\alpha\inv(x)=y$.
The details showing this construction appropriately reproduces $M$ can be found in \cite{Goldman10}, Section 5.1.

\subsection{$(G,X)$ Maps}
\index{Geometry!(G,X) Maps}

For a fixed geometry $(G,X)$, constructing the category of $(G,X)$ manifolds requires a notion of $(G,X)$ morphisms between them.

\begin{definition}
Suppose that $M$ and $N$ are two $(G,X)$ manifolds and $f\colon M\to N$ is a map.  
Then $f$ is a $(G,X)$ morphism if for all charts $\phi_\alpha\colon U_\alpha\to X$ on $M$
 and $\psi_\beta\colon V_\beta\to X$ on $N$ the restriction $\psi_\beta f\phi_\alpha\inv$ is in (the pseudogroup generated by) $G$.
\end{definition}

\noindent Note that as $G$ acts on $X$ by diffeomorphisms and the charts $\phi_\alpha,\psi_\beta$ are diffeomorphisms, every $(G,X)$ map is a \emph{local diffeomorphism} by definition.
The set of $(G,X)$ automorphisms $M\to M$ forms a group, which we denote $\Aut_{(G,X)}(M)$, and the automorphism group of a geometry itself $\Aut_{(G,X)}(X)$ is $G$.

To determine if two atlases for $(G,X)$ structures on $M$ actually define the same structure we must determine whether or not both generate the same maximal atlas: that is, whether transition maps between charts from each are in $G$.
The notion of a $(G,X)$ map allows us to phrase this succinctly and 
provides a definition for the \emph{space of $(G,X)$ structures} on a manifold $M$.

\begin{definition}
Let $M_1$, $M_2$ denote two $(G,X)$ structures on a manifold $M$.  Then $M_1$ and $M_2$ are equivalent if the identity map $\id_M\colon M_1\to M_2$ is a $(G,X)$ map.
The set of distinct $(G,X)$ structures on $M$  is denoted $\mathcal{S}_{(G,X)}(M)$.
\end{definition}

\noindent
A $(G,X)$ structure on a manifold $M$ induces a canonical $(G,X)$ structure on its covers and quotients.
More precisely, a chart $(U,\phi)$ on $M$ pulls back to the charts $(\tilde{U_i},\phi\pi)$ for $U_i$ a connected component of $\pi\inv(U)$ when $U\subset M$ is small enough (evenly covered), and conversely small enough charts $(V,\psi)$ on $\tilde{M}$ push forward under $\pi$ to charts $(\pi(V),\psi\pi\inv)$ when $\pi(V)$ is evenly covered.
We record both of these below for future use.

\begin{observation}
\label{Obs:CoveringStructure}
Let $M$ be a $(G,X)$ manifold and $\pi\colon\tilde{M}\to M$ a covering space.  Then $\tilde{M}$ has a canonical $(G,X)$ structure for which the covering projection is a $(G,X)$ map. 
\end{observation}

\noindent 
A particularly simple case of this pullback, that is often useful in practice is the specialization to covers of one sheet, or diffeomorphisms.

\begin{observation}
\label{Obs:Pullback_Structure}
Let $\Sigma$ be a smooth manifold and $M$ a $(G,X)$ manifold.
If $\phi\colon\Sigma\to M$ is a diffeomorphism, there is a unique $(G,X)$ structure on $\Sigma$ making $\phi$ into a $(G,X)$ isomorphism.	
\end{observation}

\begin{observation}
Let $M$ be a $(G,X)$ manifold on which a group $\Gamma$ acts properly and freely by $(G,X)$ maps.  Then the quotient $M/\Gamma$ inherits a $(G,X)$ structure such that the quotient map $\pi\colon M\to M/\Gamma$ is a $(G,X)$ covering.	
\end{observation}

\noindent
This allows us to produce examples of geometric structures from quotients of $X$ by suitable subgroups of $G$.

\begin{example}[Euclidean Torus]
Consider the $\Z^2$ subgroup of $\Isom(\E^2)$ given by translations along the integer lattice in the plane.  
Then $T=\E^2/\Z^2$ is topologically a torus, and inherits a canonical Euclidean structure as the $\Z^2$ action is by $(\Isom(\E^2),\E^2)$-maps.
The atlas of charts is defined as follows: for each point $p\Z^2\in T$ a choice of representative $p\in\E^2$ and a sufficiently small open neighborhood $U\ni p$ provides a chart on $U\Z^2\subset T$ sending each point $q\Z^2$ to the unique representative in $U\subset\E^2$.
\end{example}

\section{Developing Pairs}
\label{sec:Developing_Pairs}
\index{Geometric Structures!Developing Pairs}
\index{Developing Map}
\index{Holonomy}
\index{Geometry!Developing Pairs}

The data of a maximal atlas of $X$ valued charts with transitions in $G$ is unweildly to work with in practice.
The analyticity of the $G$ action on $X$ allows one to globalize the atlas of charts via a \emph{developing map} and the transitions via an associated \emph{holonomy homomorphism}, encoding the entire $(G,X)$ structure as a \emph{developing pair}.
Briefly, back the $(G,X)$ structure on $M$ to the universal cover $\tilde{M}$ and analytically continuing a chosen base chart to a $(G,X)$ map $f\colon \tilde{M}\to X$ called the developing map, and the $\pi_1(M)$ action by covering transformations induces an action on $f(\tilde{M})$ by elements of $G$.

\begin{definition}
A \emph{developing pair} for a $(G,X)$ structure on a manifold $M$ is a pair $(f,\rho)$ of an immersion $f\colon\tilde{M}\to X$, equivariant with respect to the representation $\rho\colon\pi_1(M)\to G$.
\end{definition}

\begin{figure}
\centering\includegraphics[width=0.7\textwidth]{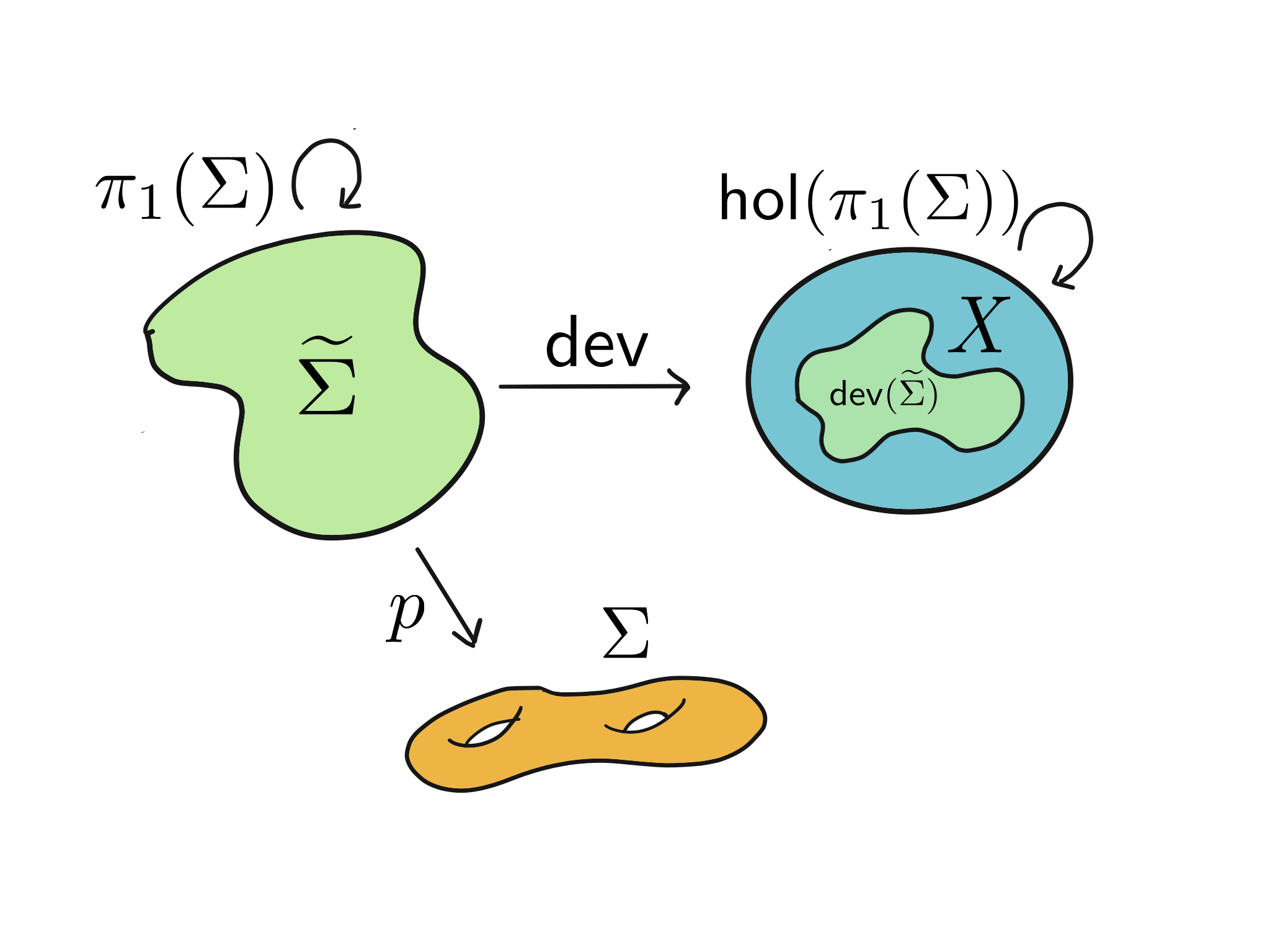}
\caption{A developing pair for a geometric structure.}	
\end{figure}

\noindent
We denote the space of all $(G,X)$ developing pairs for $M$ by $\mathsf{Dev}_{(G,X)}(M)$.
Note that a developing map $\colon\tilde{M}\to X$ uniquely determines the associated holonomy so we may alternatively think of $\mathsf{Dev}_{(G,X)}(M)$ as simply the \emph{space of developing maps}, the subset of immersions in $C^\infty(\tilde{M},X)$ which are equivariant with respect to some homomorphism $\rho\in\Hom(\pi_1(M),G)$.
Topologizing $C^\infty(\tilde{M},X)$ by smooth uniform convergence of all partial derivatives on compact sets provides $\mathsf{Dev}_{(G,X)}(M)$ with the subspace topology.
This agrees with the subspace topology inherited from the full developing pairs in $C^\infty(\tilde{M},X)\times\Hom(\pi_1(M),G)$.

\begin{example}[Euclidean Torus]
Let $T$ be the Euclidean torus represented by the Euclidean metric $ds^2=\tfrac{4}{3}(dx^2-dxdy+dy^2)$ on $\R^2/\Z^2$.
  A developing pair for this structure into the Euclidean plane with metric $ds^2=dx^2+dy^2$ is given by the linear map $f\colon \R^2\to\E^2$, $f(x,y)=(x,\tfrac{x}{2}+y\tfrac{\sqrt{3}}{2})$ and the holonomy $\rho\colon\Z^2\to\Euc(2)$ defined by $\rho(e_1)=\smat{1&0&1\\0&1&1\\0&0&1}$, $\rho(e_2)=\smat{1&0&\tfrac{1}{2}\\0&1&\tfrac{\sqrt{3}}{2}\\0&0&1}$.
\end{example}

\begin{figure}
\centering\includegraphics[width=0.85\textwidth]{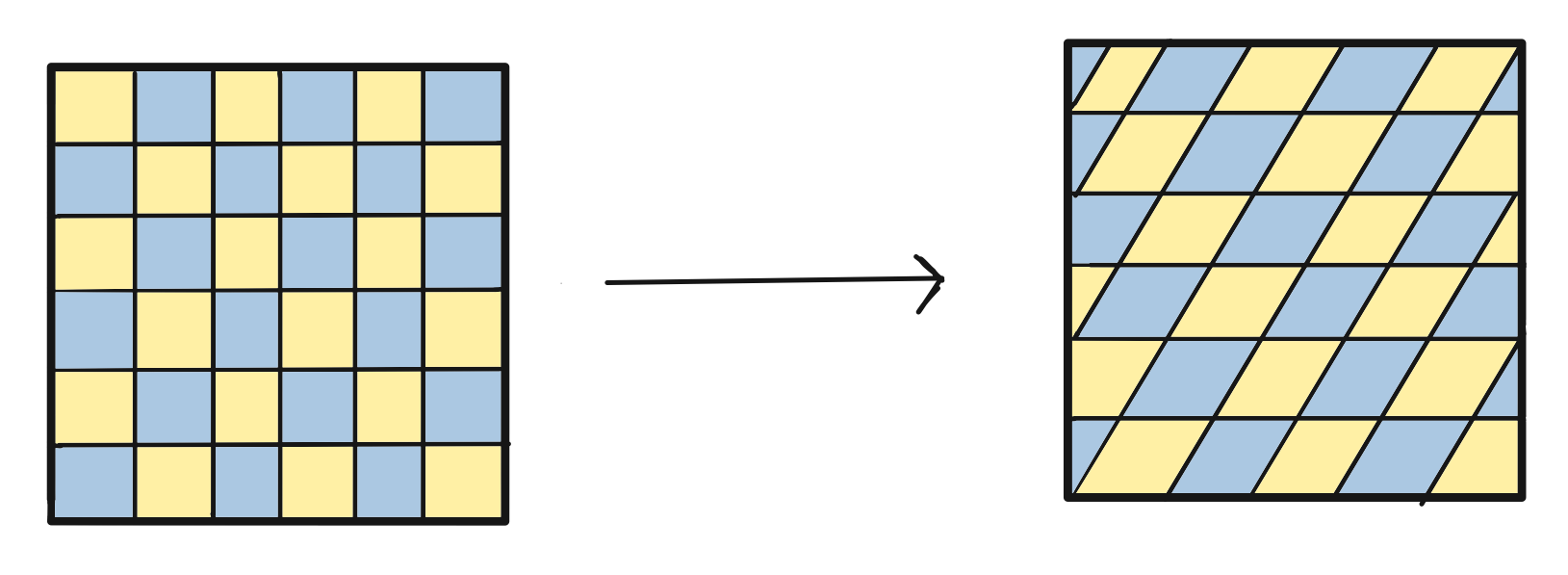}
\caption{The developing map for the hexagonal torus.}	
\end{figure}

\begin{example}[Hopf Torus]
\label{Ex:Hopf_Torus}
The Hopf torus is a similarity structure on $T^2$ with developing map $f\colon \R^2\to \C$ given by $f(x,y)=e^{x+ 2\pi iy}$ and (non-faithful) holonomy $\rho\colon \Z^2\to \Sym(2)$ defined by $\rho(e_1)=e\cdot \mathsf{Id}$ and $\rho(e_2)=\mathsf{Id}$.
\end{example}

\begin{figure}
\centering\includegraphics[width=0.85\textwidth]{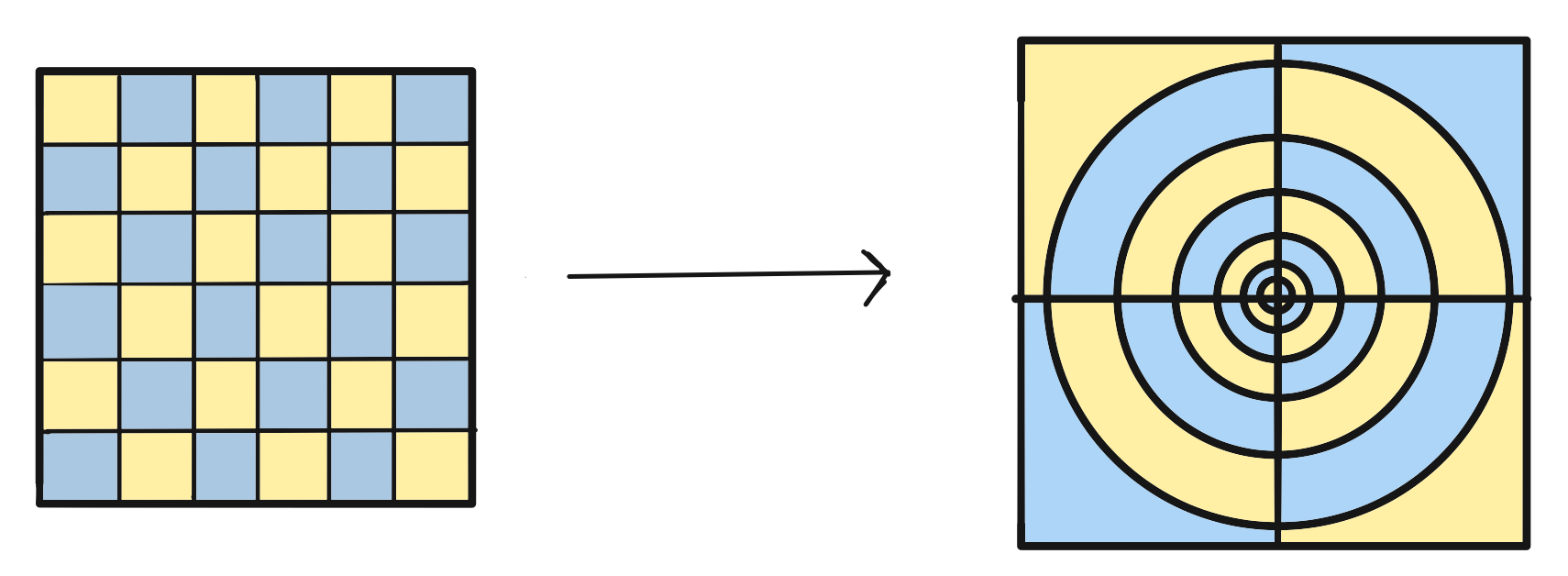}
\caption{A developing map for a similarity torus.}	
\end{figure}

\noindent
Developing pairs provide a useful means of topologizing the space $\mathsf{S}_{(G,X)}(M)$ of all $(G,X)$ structures on $M$.
To do so we need to understand better the construction of developing pairs from atlases to quantify the lack of uniqueness and the choices required in such a construction.
As noted in Observation \ref{Obs:CoveringStructure}, an atlas charts for a $(G,X)$ structure on $M$ pulls back to an atlas on the universal cover $\tilde{M}$.
This structure induces a $(G,X)$ immersion of $\tilde{M}$ into $X$.  
For the details of this construction see \cite{Goldman18}, Proposition 5.2. Here we provide a quick sketch.

\begin{proposition}
Let $M$ be a simply connected $(G,X)$ manifold.
Then there exists a $(G,X)$ map $f\colon M\to X$	, and furthermore $f$ is unique in the following sense: if $f'\colon M\to X$ is any other $(G,X)$ map then there is a $(G,X)$ automorphism $\phi$ of $M$ and a $g\in G$ such that $gf=f'\phi$. 
\end{proposition}
\begin{proof}[Sketch]
Choose a basepoint $x_0\in M$ and a chart $U_0$ containing it.  
We then `analytically continue' this base chart $U_0$ to a $(G,X)$ map defined on all of $M$.
For $x\in M$, we define $f(x)$ by choosing a path $\gamma\colon I\to X$ with $\gamma(0)=x_0$, $\gamma(1)=x$ and sequence of charts $U_1, U_2, \ldots U_n$ covering the image $\gamma(I)$ with $U_i\cap U_{i+1}\neq \varnothing$.
Then the chart $(U_1,\phi_1)$ may be adjusted by the transition map $g_{01}\in G$ such that $g_{01}\phi_1=\phi_0$ on $U_0\cap U_1$ and thus $\phi_0\cup g_{01}\phi_1$ is well-defined on the union $U_0\cup U_1$.
Continuing this way, we adjust the charts $U_i$ by the corresponding transition maps $g_{i-1,i}\in G$ to extend the domain of $\phi_0$ to the union $\cup_{j} U_j$.
Upon reaching $i=n$, the original chart $U_0$ has been extended to the domain $\cup_{i=1}^nU_i$ containing $x$; we define $f(x)$ to be the image of $x$ under this extended chart.

This definition of $f(x)$ requires many choices, but turns out to be independent of all choices other than the original chart $U_0$.
To see this it suffices to prove that the definition of $f(x)$ is invariant under refinement of the covering of $\gamma(I)$ - and thus under choice of cover alltogether as any two covers in a maximal atlas have a common refinement.
We then need to see that the definition of $f(x)$ is independent of the choice of path $\gamma$.
As $M$ is simply connected any two paths from $x_0$ to $x$ are homotopic, and its easy to show that the definition of $f(x)$ is invariant under small homotopies, thus all homotopies of $\gamma$.
Choosing a different initial chart $U_0'$ alters the initial chart, and hence the entire construction, by the transition map $g'\in G$ for $U_0\cap U'_0\ni x_0$.
Thus the developing map $f\colon M\to X$ is uniquely defined only up to post-composition by automorphisms in $G$.
\end{proof}

\noindent
In the context of interest this provides a $(G,X)$ map from the universal cover $\tilde{M}$ of any $(G,X)$ manifold $M$ into $X$, globalizing the atlas of coordinate charts.
This is the main ingredient in the \emph{development theorem} allowing us to study geometric structures strictly from the perspective of developing pairs.

\begin{figure}
\centering\includegraphics[width=0.7\textwidth]{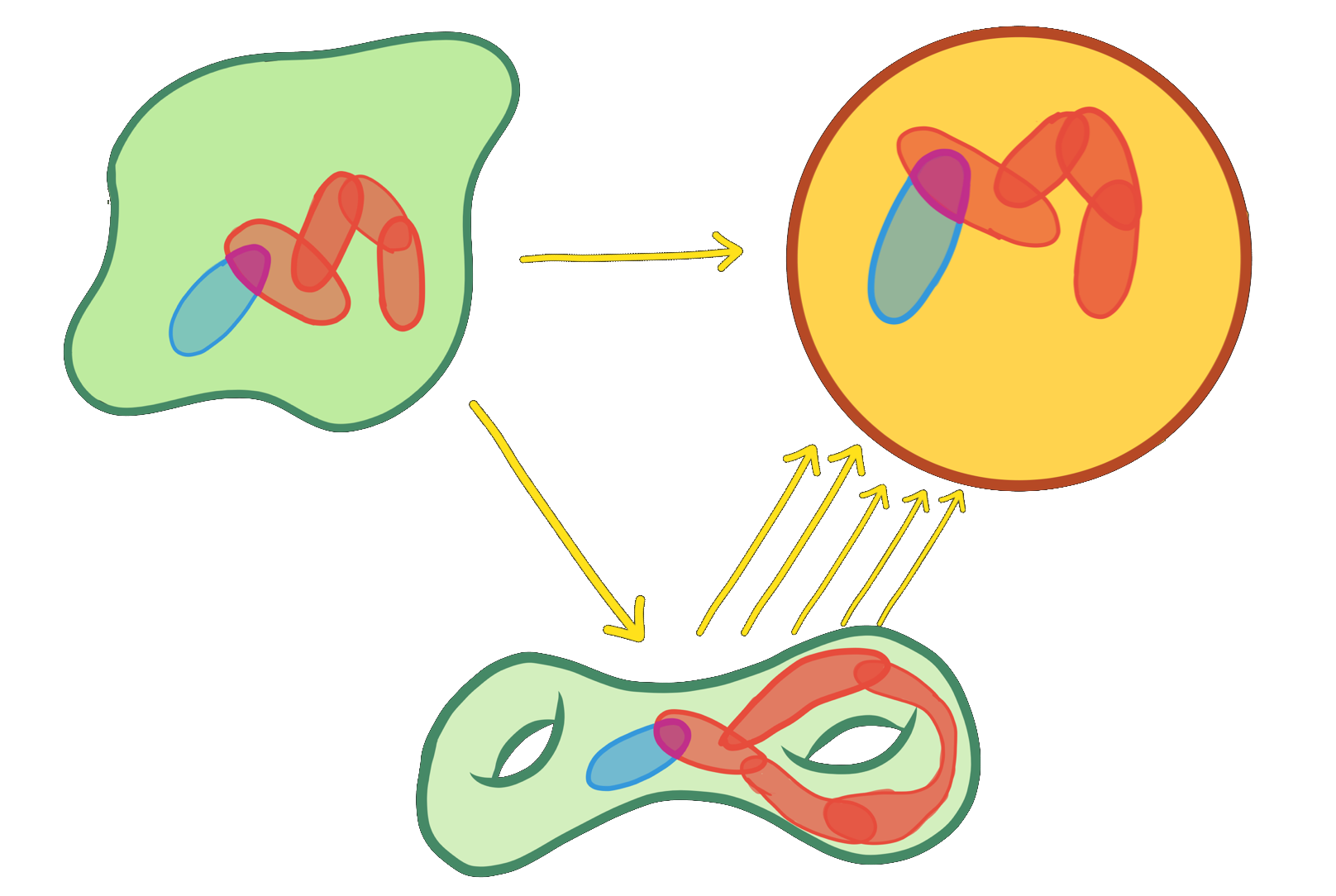}
\caption{Creating the developing map via analytic continuation of a chart.}	
\end{figure}

\begin{theorem}[Development Theorem]
Let $M$ be a $(G,X)$ manifold with universal covering space $\pi\colon\tilde{M}\to M$ and deck group $\pi_1(M)<\Aut(\tilde{M}\to M)$.
Then there exists a \emph{developing pair}
$(f,\rho)$ consisting of a $(G,X)$ map $f\colon \tilde{M}\to X$ and a homomorphism $\rho\colon \pi_1(M)\to G$ such that for each $\gamma\in\pi_1(M)$ and each $m\in\tilde{M}$, $\rho(\gamma).f(m)=f(\gamma.m)$.
Furthermore if $(f',\rho')$ is another such pair, then there is some $g\in G$  such that for all $\gamma\in\pi_1(M)$,
$f'=g\circ f$ and $\rho'(\gamma)=\mathsf{Inn}(g)\circ\rho(\gamma)$.

\begin{center}
\begin{tikzcd}
\tilde{M}\ar[d,"\gamma"]\ar[r,"f"]&X\ar[r,"g"]\ar[d,"\rho(\gamma)"]&X\ar[d,"\rho'(\gamma)"]\\
\tilde{M}\ar[r,"f"]&X\ar[r,"g"]&X
\end{tikzcd}
\end{center}
\end{theorem}

\noindent
Thus, the $G$-orbits of developing pairs uniquely determine $(G,X)$ structures and we may use this description to provide a natural topology to the space $\mathcal{S}_{(G,X)}(M)$.

\begin{corollary}
The space $\mathcal{S}_{(G,X)}(M)$ of $(G,X)$ structures on a manifold $M$ is a topologized as the quotient of the space of developing pairs $\mathcal{S}_{(G,X)}(M)=\mathsf{Dev}_{(G,X)}(M)/G$ by the $G$ action $g.(f,\rho)=(g\circ f, \mathsf{Inn}(g)\circ \rho)$.
\end{corollary}

\noindent
This perspective has some immediate consequences, such as the following.

\begin{observation}
If $M$ is a closed manifold with finite fundamental group, then $M$ admits no $(G,X)$ structures when the underlying space $X$ is noncompact.	
\end{observation}
\begin{proof}
This follows as the universal cover $\tilde{M}$ is compact by the finiteness of $\pi_1(M)$ and thus any continuous image $f(\tilde{M})\subset X$ is compact.
But were $f$ the developing map of a $(G,X)$ structure it is a local diffeomorphism so $f(\tilde{M})$ is open, and thus equal to $X$ by connectedness.
\end{proof}

\begin{observation}
If $X$ is compact and simply connected then every $(G,X)$ manifold is $(G,X)$ isomorphic to a quotient of $X$ by a finite subgroup of $G$.	
\end{observation}
\begin{proof}
A developing map $f\colon\tilde{M}\to X$ of a $(G,X)$ structure on $M$ is a local diffeomorphism into the closed manifold $X$, which is then necessarily a covering map.
As $X$ is simply connected this must be a diffeomorphism, so the holonomy is faithful.
Then
 $M=\tilde{M}/\pi_1(M)\cong f(\tilde{M})/\rho(\pi_1(M))=X/\rho(\pi_1(M))$, realizing $M$ as a quotient of $X$.
The compactness of $X$ implies that $\rho(\pi_1(M))$, and hence $\pi_1(M)$, is finite.
\end{proof}

\section{Completeness}
\label{sec:Completeness}
\index{Geometric Structures!Completeness}

Geometric structures which arise as quotients of the underlying space $X$ have particularly nice algebraic and geometric properties.
In this section we define \emph{completeness}, show that complete structures are determined by their holonomy, and relate this notion of completeness to the familiar metric notion in cases where $(G,X)$ admits an invariant Riemannian metric.

\begin{definition}
A $(G,X)$ structure on $M$ is \emph{complete} if the developing map $f\colon \tilde{M}\to X$ is a covering map.
\end{definition}

\noindent
We begin by noting the two most important properties of complete structures.
When the underlying space $X$ of the geometry is simply connected, the developing map of a complete structure provides a diffeomorphism $\tilde{M}\to X$, which we often use to identify the two spaces.
The action of $\pi_1(M)$ by deck transformations is conjugate by the developing diffeomorphism to the holonomy action on $X$.

\begin{proposition}[Complete Structures are Quotients]
A complete $(G,X)$ structure on a manifold $M$ is $(G,X)$ isomorphic to a quotient $X/\Gamma$ for $\Gamma$ a discrete subgroup of $G$ acting freely and properly discontinuously on $X$, when $X$ is simply connected.
\end{proposition}
\begin{proof}
If $(f,\rho)$ is a developing pair for a complete $(G,X)$ structure on $M$, then $f\colon\tilde{M}\to X$ is a covering map by definition, and as $X$ is simply connected this is a $1$-sheeted cover, so $f$ is a diffeomorphism.
The holonomy homomorphism is conjugate to the action of the deck group $\pi_1(M)$ on $\tilde{M}$ by the developing diffeomorphism $\rho(\gamma).x=f(\gamma.f\inv(x))$; thus $\rho$ is faithful and acts freely and properly discontinuously on $X$, with discrete image $\Gamma<G$.
Pulling back via $f$ equips $\tilde{M}$ with a $(G,X)$ structure for which $f$ is a $(G,X)$ isomorphism intertwining the covering action with the holonomy action.
Thus $f$ descends to a $(G,X)$ isomorphism on the respective quotients $M=\tilde{M}/\pi_1(M)$ and $X/\Gamma$.

\begin{center}
\begin{tikzcd}
\widetilde{M}\ar[d]\ar[r,"f"]& X\ar[d]\\
\widetilde{M}/\pi_1(M)\ar[r,"\bar{f}"]& X/\Gamma
\end{tikzcd}	
\end{center}

\end{proof}

\noindent 
Every $(G,X)$ geometry is locally isomorphic to its universal cover $(\tilde{G},\tilde{X})$, 
so in the following we assume that the underlying space $X$ is simply connected when convenient.
When $X$ is contractible, complete $(G,X)$ manifolds have universal cover diffeomorphic to $X$ and thus are classifying spaces for their fundamental groups.
In fact, as noted by Thurston in \cite{Thurston80}, the holonomy of a complete structure is enough to reproduce the structure itself.

\begin{proposition}[Holonomy Determines Complete Structures]
Let $(G,X)$ be a geometry with contractible underlying space $X$, and $M$ a complete $(G,X)$ manifold with holonomy $\rho$.
Then any other $(G,X)$ manifold with holonomy $\rho$, is $(G,X)$ isomorphic to $M$.
\end{proposition}

\noindent 
We now relate this notion of completeness to the more familiar metric notion from Riemannian geometry via the Hopf-Rinow theorem.

\begin{theorem}[Hopf-Rinow]
Let $(M,g)$ be a connected Riemannian manifold.  Then the following statements are equivalent:
\begin{itemize}
\item Closed and bounded subsets of $M$ are compact.
\item $M$ is complete as a metric space.
\item $M$ is geodesically complete. That is, for each $p\in M$ the exponential map $\exp_p\colon T_p M\to M$ is defined on the entire tangent space.	
\end{itemize}
	
\end{theorem}

\noindent 
Thus the \emph{geodesic completeness} of a Riemannian manifold is equivalent to its \emph{metric completeness}.
As a consequence, we can show that our definition of completeness as $(G,X)$ structures is equivalent to the usual metric notion when $X$ admits a $G$-invariant Riemannian metric.

\begin{proposition} 
Let $(G,X)$ have $G$-invariant Riemannian metric $ds_X^2$, and $M$ be a compact $(G,X)$ manifold.
Then the developing map $f\colon\tilde{M}\to X$ is a covering map.
\end{proposition}
\begin{proof}
The riemannian metric $ds^2_X$ pulls back under the developing map to a metric $f^\ast ds^2_X$ on $\tilde{M}$, 
which is invariant under the deck group $\pi_1(M)$ and so descends to a metric $ds^2_M$ on the quotient $M=\tilde{M}/\pi_1(M)$.
Since $M$ is compact, it is complete as a metric space, and so the metric $f^\ast ds^2_X$ on $\tilde{M}$ is complete as well.
By Hopf-Rinow, $\tilde{M}$ is geodesically complete.
Finally the developing map $f\colon\tilde{M}\to X$ is a local isometry from a complete Riemannian manifold into a Riemannian manifold is a covering map \cite{Sosh63}.
\end{proof}

\noindent
This has some strong implications for $(G,X)$ structures, such as the following.

\begin{corollary}
Every hyperbolic structure on a closed surface is complete, and all hyperbolic surfaces are isomorphic to quotients of $\H^2$ by discrete subgroups of $\PSL(2;\R)$.
\end{corollary}

\noindent
We conclude this section with examples of complete and incomplete structures for reference.

\begin{example}[Hyperbolic Cylinders]
\label{Ex:Hyp_Cylinders}
The representations $\rho_i\colon\Z\to\SL(2;\R)$ given by $\rho_1(1)=\smat{1&1\\0&1}$, $\rho_2(1)=\smat{x&x\\x&x}$ are the holonomies of hyperbolic structures on the cylinder.
The first is the holonomy of a complete structure, with developing map onto the entire upper half plane.
The second represents an incomplete structure, with fundamental domains accumulating on to a vertical geodesic in the model.
\begin{figure}
\centering
\includegraphics[width=0.6\textwidth]{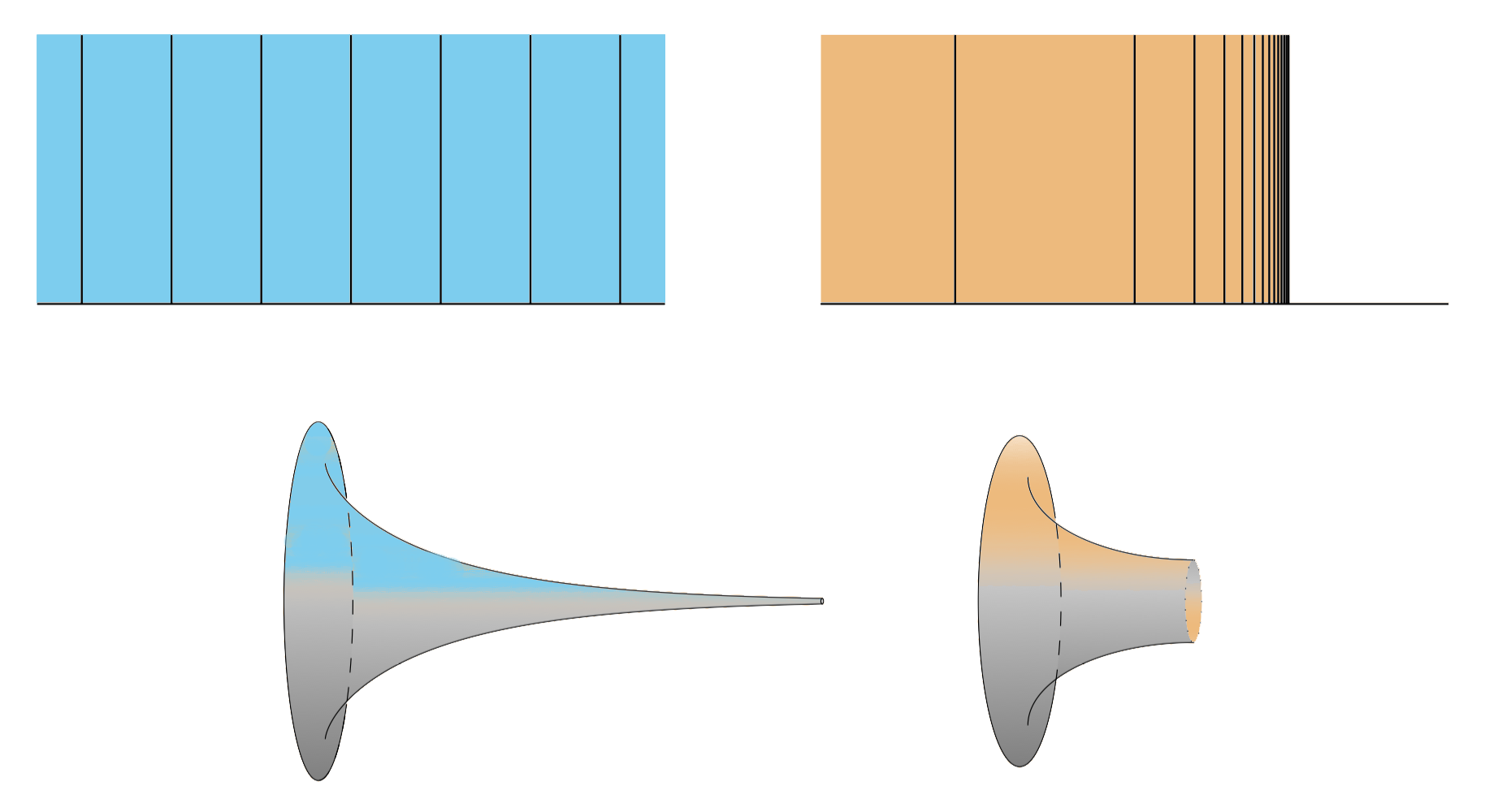}	
\caption{The developing maps of complete (left) and incomplete (right) hyperbolic structures on a cylinder.}
\end{figure}

\end{example}

In the example above, the holonomy of the incomplete structure fails to act properly discontinuously on $\Hyp^2$, but is still a faithful representation $\Z\to\Isom(\Hyp^2)$.  
This is not always the case however; the Hopf torus of Example \ref{Ex:Hopf_Torus} is incomplete as the complex exponential $\exp\colon\C\to \C$ is not a covering map and the holonomy $\Z^2\to\Sym(2)$ is not faithful.
The completeness of a structure depends heavily on the $(G,X)$ geometry under consideration, as further analysis of the Hopf torus reveals.

\begin{example}
\label{Ex:Hopf_Torus_Two}
The Hopf torus of Example \ref{Ex:Hopf_Torus} as an incomplete similarity structure, as $\exp\colon\C\to\C$ is not a covering map.
Restricting the codomain $\C^\times$, the exponential is a covering, and as the holonomy acts by complex multiplication on the plane $\rho(e_1)=1$, $\rho(e_2)=e$; we may consider the Hopf torus as a complete $(\C^\times,\C^\times)$ structure on $T^2$.
\end{example}


\chapter{Moduli \& Degeneration}
\label{chap:Moduli_Degen}
\index{Moduli}

The moduli space of $(G,X)$ structures on a manifold $M$ is a space $\fam{M}_{(G,X)}(M)$ whose points represent inequivalent $(G,X)$ structures on $M$.
Unfortunately these spaces are typically quite complicated and often non-Hausdorff.
Thus we replace this goal with an easier one; parameterizing \emph{marked $(G,X)$ structures} on $M$ by the \emph{deformation space} $\mathcal{D}_{(G,X)}(M)$, whose further quotient by forgetting the marking solves the moduli problem.

Given a topological space parametrizing $(G,X)$ structures on $M$, it is natural to consider the possible \emph{degenerations}, when a sequence of structures leaves every compact set in $\mathcal{D}_{(G,X)}(M)$.
While these sequences fail to converge as $(G,X)$ structures, they may converge as $(H,Y)$ structures for some containing geometry $(H,Y)$.
In such cases, we say that this degenerating path of $(G,X)$ structures limits to an $(H,Y)$ structure, and we will have reason to often consider such limits throughout this thesis.

Sometimes, a uniform construction provides endpoints for all degenerating paths in a deformation or moduli space, resulting in a \emph{compactification} with the boundary points parameterizing limiting structures.
We additionally discuss some techniques from algebraic geometry which will be useful in constructing compactifications in Part II.

\section{Deformation Space}
\label{sec:Def_Space}
\index{Geometric Structures! Deformation Space}
\index{Deformation Space}

\emph{Symmetries correspond to singularities} is a good one-phrase introduction to moduli theory.

\begin{example}
\label{Ex:Conformal_Tori}
The moduli space of conformal structures on the torus is the \emph{modular curve}, the quotient of $\Hyp^2$ by the isometric action of $\SL(2,\Z)$.
This is topologically a disk, equipped with an orbifold structure with two cone points of orders $2,3$ representing the square and hexagonal tori respectively.
\end{example}

\begin{figure}
\centering\includegraphics[width=0.95\textwidth]{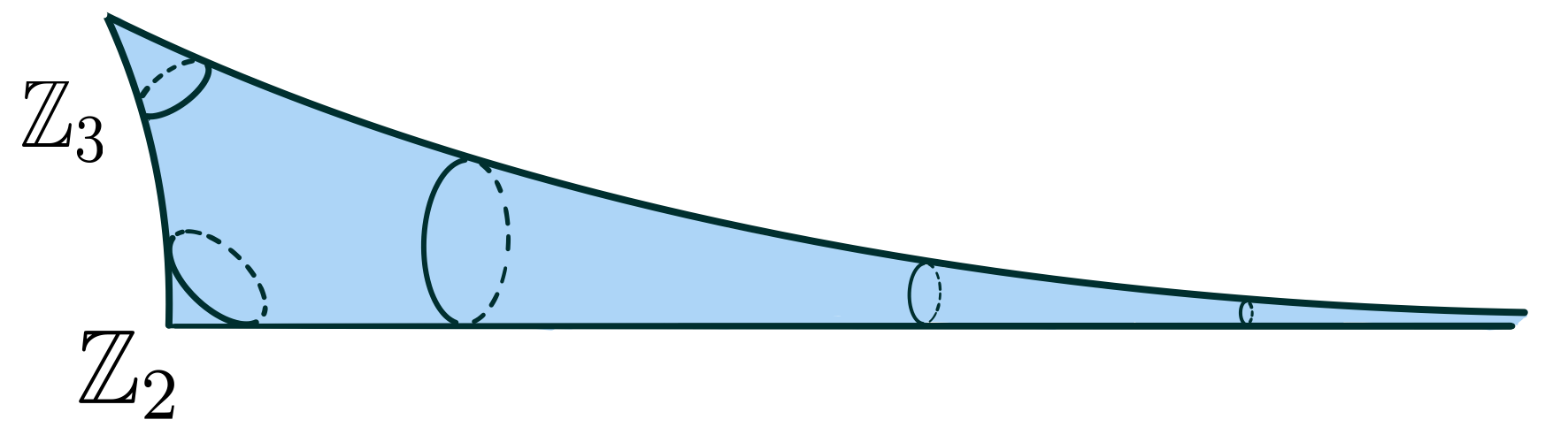}
\caption{The moduli space of conformal tori.}	
\end{figure}

The \emph{deformation space of structures} encodes geometric structures together with some kind of \emph{marking} to break the exceptional symmetries enjoyed by particular structures, and thus preclude the singularities caused by them.
We begin by reviewing the motivating and likely familiar case of Teichm\"uller theory, of which deformation space is a direct generalization.

\noindent{\bfseries\sffamily Teichm\"uller Theory:\;}
Let $\Sigma_g$ denote the closed surface of genus $g$.
A \emph{genus $g$ Riemann surface} is a complex algebraic curve $M$ homeomorphic to $\Sigma_g$.
A \emph{marked Riemann surface} is a pair $(\phi,M)$ of a Riemann surface $M$ together with a fixed homeomorphism $\phi\colon \Sigma_g\to M$.
The Teichm\"uller space $\mathcal{T}_g$ is defined as the space of marked genus $g$ Riemann surfaces up to equivalence, where $(\phi,M)\sim (f',M')$ when there is a biholomorphism $\psi\colon M\to M'$ such that $\psi\phi$ and $\phi'$ are isotopic.
The Teichm\"uller space is a smooth manifold, diffeomorphic to a ball of dimension $6g-6$ when $g>1$ and $\mathcal{T}_1\cong\H^2$.
The moduli space of biholomorphism classes of complex structures on $\Sigma_g$ is $\mathcal{M}_g$ is the quotient of $\mathcal{T}_g$ sending pairs $(\phi,M)$ to the underlying Riemann surface $M$.
Distinct markings $(\phi, M)$ and $(\phi',M)$ give nontrivial self-homeomorphisms $\phi\inv \phi'\colon\Sigma_g\to\Sigma_g$ and so quotient forgetting markings corresponds to the action of the mapping class group $\mathsf{Mod}_g$ on Teichm\"uller space, $\mathcal{M}_g=\mathcal{T}_g/\mathsf{Mod}_g$.
As Riemann surfaces are classifying spaces for their fundamental groups the mapping class group identifies with outer automorphisms of the fundamental group, so $\mathcal{M}_g=\mathcal{T}_g/\mathsf{Out}(\pi_1(\Sigma_g))$.

We develop a very similar story in the more general context of $(G,X)$ structures, defining \emph{deformation space} as equivalence classes of marked $(G,X)$ structures and  realize \emph{moduli space} as the quotient after forgetting the markings.

\begin{definition}
Let $\Sigma$ be a smooth manifold.
A marked $(G,X)$ structure on $\Sigma$ is a pair $(\phi,M)$ of a $(G,X)$ manifold $M$ and a diffeomorphism $\phi\colon\Sigma\to M$.
Two marked $(G,X)$ structures $(\phi,M)$ and $(\phi',M')$ on $\Sigma$ are equivalent if there is a $(G,X)$ map $\psi\colon M\to M'$ where the following triangle commutes up to isotopy.
\begin{center}
\begin{tikzcd}
M\ar[rr,"\psi"]&& M'\\
&\Sigma\ar[ul,"\phi"]\ar[ur,swap,"\phi'"]&
\end{tikzcd}
\end{center}
\end{definition}

\begin{figure}
\centering\includegraphics[width=0.75\textwidth]{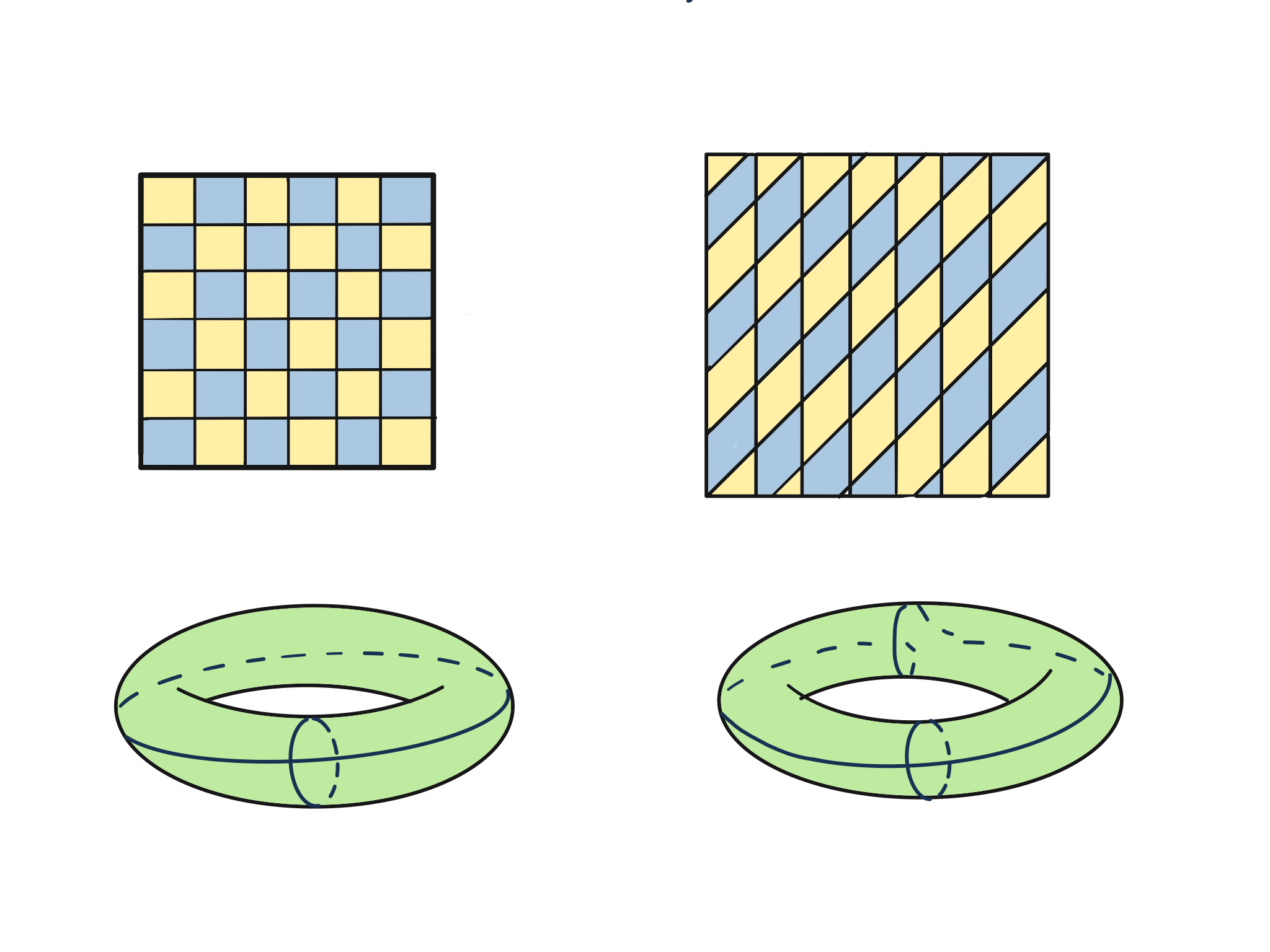}
\caption{Different markings on the same conformal (rectangular) torus.}	
\end{figure}

\noindent
Let $\mathsf{Diffeo}(M)$ denote the group of self-diffeomorphisms of $M$ equipped with the compact-open topology, and $\mathsf{Diffeo}_0(M)$ the connected component of the identity.
Then $\mathsf{Diffeo}(M)$ acts on the space $\mathcal{S}_{(G,X)}(M)$ of $(G,X)$ structures by composition with the marking, $\alpha.(\Sigma\to M)=\Sigma\stackrel{\alpha}{\to}\Sigma\to M$, and two marked structures are isotopic if they differ by the action of some element in $\mathsf{Diffeo}_0(M)$.

\begin{definition}
The \emph{deformation space of $(G,X)$ structures on $M$} is the quotient of the space of marked structures by diffeomorphisms isotopic to the identity.
\end{definition}

\noindent 
Taking a different perspective on marked structures, we may realize deformation space as a quotient of the familiar space $\mathcal{S}_{(G,X)}$ of developing pairs up to $G$-conjugacy.

\begin{proposition}
Pullback of $(G,X)$ structures defines a bijection between the space of marked structures $(\phi,M)$ on $\Sigma$ and the space $\mathcal{S}_{(G,X)}(\Sigma)$ of developing pairs for $(G,X)$ structures on $\Sigma$, up to $G$-conjugacy.
\end{proposition}
\begin{proof}
If $\Sigma$ is a smooth manifold, $M$ a $(G,X)$ manifold and $\phi\colon \Sigma\to M$ a diffeomorphism, then recalling Observation \ref{Obs:Pullback_Structure} there is a unique $(G,X)$ structure on $\Sigma$ for which $\phi$ is a $(G,X)$ isomorphism.
We denote this structure $\Sigma_{(\phi,M)}$ to limit confusion.
This associates to each marked structure a unique $(G,X)$ structure on $\Sigma$ itself.
Conversely, if $[f,\rho]_{(G,X)}$ is a developing pair for a geometric structure on $\Sigma$, we may think of the identity map $\id_\Sigma\colon\Sigma\to\Sigma$ as a diffeomorphism from the smooth manifold $\Sigma$ to the $(G,X)$ manifold $\Sigma_{(\phi,M)}$.
Clearly the geometric structure associated to the marked structure $(\id_\Sigma,\Sigma_{(\phi,M)})$ is $\Sigma_{(\phi,M)}$ itself.
Composing the other way, if $(\phi,M)$ is a marked structure, the pullback $\Sigma_{(\phi,M)}$ gets associated to the marked structure $(\id_\Sigma,\Sigma_{(\phi,M)})$ which is equivalent as a marked structure to $(\phi,M)$ as the relevant triangle commutes on the nose.

\begin{center}
\begin{tikzcd}
M&& \Sigma_{(\phi,M)}\ar[ll,swap,"\phi"]\\
&\Sigma\ar[ul,"\phi"]\ar[ur,swap,"\id_\Sigma"]&
\end{tikzcd}
\end{center}

\end{proof}

\noindent
Under this identification with $\mathcal{S}_{(G,X)}(M)$, the action of $\mathsf{Diffeo}_0(M)$ can be described as follows.
Let $\tilde{M}\to M$ be a fixed universal cover.
Then any $\alpha\in\mathsf{Diffeo}_0(M)$ lifts to a $\pi_1(M)$-equivariant map $\tilde{\alpha}\colon\tilde{M}\to\tilde{M}$ which is isotopic to $\id_{\tilde{M}}$ through a sequence of $\pi_1(M)$-equivariant automorphisms.
Choosing basepoints $m\in M$, $\tilde{m}\in\tilde{M}$ this lift can be chosen uniquely, which provides an embedding $\mathsf{Diffeo}_0(M)\to\mathsf{Diffeo}(\tilde{M})$.
The lift of $\alpha\in\mathsf{Diffeo}_0(M)$ is denoted $\tilde{\alpha}$ and the action of $\mathsf{Diffeo}_0(M)$ on the set of developing pairs is by precomposing the developing map with the lifted diffeomorphism.

\begin{observation}
In terms of developing pairs, the deformation space $\mathcal{D}_{(G,X)}(M)$ of $(G,X)$ structures on $M$ is the quotient space $\mathcal{D}_{(G,X)}(M)=\mathcal{S}_{(G,X)}(M)/\mathcal{Diffeo}_0(M)$ by the action $\alpha.[f,\rho]_G=[f\circ\tilde{\alpha},\rho]_G$
\end{observation}

\noindent 
The quotient of $\mathcal{S}_{(G,X)}(M)$ by this precomposition of the developing map factor by $\mathsf{Diffeo}_0(M)$ yields deformation space.

\begin{example}
The following path of Euclidean tori	, realized as a continuous map $[0,1]\to\mathcal{D}_{\E^2}(T^2)$ smoothly transitions from the square torus to the hexagonal torus.
$$\hol_t:\pi_1(T)=\langle A,B\rangle\to\Isom(\E^2) \hspace{1cm}
\mat{ 
A\mapsto \smat{1&0&1\\0&1&0\\0&0&1} \\
B\mapsto \smat{1&0&\cos\left(\frac{2\pi}{3}t\right)\\0&1&\sin\left(\frac{2\pi}{3}t\right)\\0&0&1}
}
$$

$$\dev_t:\tilde{T}=\R^2\to \R^2 \hspace{1cm} \pmat{x\\y}\mapsto \pmat{x+y\cos\left(\frac{2\pi}{3}t\right)\\ y\sin\left(\frac{2\pi}{3}t\right)}$$
\end{example}

\subsection{Representation Varieties}
\index{Representation Varieties}

Here we quickly review the basic theory of representation varieties.
For a more detailed account, consult \emph{Geometric Structures and Varieties of Representations} by Goldman \cite{Goldman88}.
Given a finitely presented group $\Gamma=\langle s_1,\ldots s_m\mid r_1\ldots r_m\rangle$, evaluation on the generators naturally embeds the space $\Hom(\Gamma,G)$ of representations into $G^m$.
In particular, when $G<\GL(p;\R)$ is a matrix Lie group, the image is a real algebraic set in $\R^{m p^2}$ cut out by the $n p^2$ polynomials arising from the relations $r_1\ldots r_n$ written out in $p\times p$ matrices.
The variety structure inherited from this construction is independent of choice of generating set, and thus is intrinsic to the \emph{representation variety} $\Hom(\Gamma,G)$.
We give $\Hom(\Gamma,G)$ the classical topology as a subset of $\R^{mp^2}$.
The group $G$ acts on this representation variety by conjugacy, and $\Hom(\Gamma,G)/G$ inherits the quotient topology from this.

\begin{example}
The character variety of representations of the free group $\mathbb{F}_2$ on two generators into $\SL(2;\R)$ is the real two dimensional variety $V(x^2+y^2+z^2-xyz)$.
Each component of this variety is an open disk, and one of them identifies with the Teichm\"uller space of complete finite volume hyperbolic structures on the punctured torus.	
\end{example}

In contrast to the example above, the resulting space $\Hom(\Gamma,G)/G$ may be rather ill-behaved, and a selection of `bad behavior' which occurs in practice is listed below.
\begin{SingleSpace}
\begin{itemize}
\item The variety $\Hom(\Gamma,G)$ may not be smooth, and $\Hom(\Gamma,G)/G$ may inherit the singularities of an algebraic variety.
\item The quotient $\Hom(\Gamma,G)\to\Hom(\Gamma,G)/G$ may be nontrivially branched so $\Hom(\Gamma,G)/G$ has orbifold singularities.
\item The action of $G$ on the $\Hom(\Gamma,G)$ may not be proper, so the quotient $\Hom(\Gamma,G)/G$ is not Hausdorff.
\end{itemize}
\end{SingleSpace}

\subsection{Moduli Space}
\index{Moduli Space}

The \emph{moduli space} of $(G,X)$ structures is the further quotient forgetting marking, which is realized by the action of \emph{all} diffeomorphisms of $M$ on $\mathcal{S}_{(G,X)}(M)$, or equivalently by the action of $\mathsf{Diffeo}(M)/\mathsf{Diffeo}_0(M)$ on deformation space.

\begin{definition}
The moduli space $\mathcal{M}_{(G,X)}(M)$ of $(G,X)$ structures on $M$ is the set of all $(G,X)$ structures on $M$ up to $(G,X)$ equivalence.
This naturally identifies with the quotient of deformation space by the diffeotopy group $\mathcal{M}_{(G,X)}(M)=\mathcal{D}_{(G,X)}(M)/\pi_0(\mathsf{Diffeo}(M))$.
\end{definition}

The fact that the action of diffeomorphisms isotopic to the identity have no effect on the holonomy makes this an attractive coordinate on deformation space.
The projection onto holonomy from the space of developing pairs $\mathsf{Dev}_{(G,X)}(M)\subset C^\infty(\tilde{M},X)\times\Hom(\pi_1(M),G)$ induces a projection $\hol\colon \mathcal{S}_{(G,X)}(M)\to\Hom(\pi_1(M),G)/G$ onto representations modulo $G$ conjugacy.
This directly descends to the quotient by isotopy giving a well-defined projection $\mathcal{D}_{(G,X)}(M)\to\Hom(\pi_1(M),G)/G$ associating to each marked structure its conjugacy class of holonomies.

The fact that small deformations in holonomy correspond to small deformations in geometric structure was first noticed by Thurston, and with the work of many others 
is captured by the following theorem.

\begin{theorem}
Let $(G,X)$ be a geometry and $M$ a compact $(G,X)$ manifold with holonomy representative $\rho\colon\pi_1(M)\to G$.  
Then for all $\rho'$ sufficiently near to $\rho$ in the representation variety $\Hom(\pi_1(M), G)$, there exists a nearby $(G,X)$ structure with holonomy $\rho'$.
Furthermore if $M'$ is a $(G,X)$ manifold near $M$ in deformation space which has the same holonomy $\rho$, then $M'$ is isomorphic to $M$ by a $(G,X)$ isomorphism isotopic to the identity.
\end{theorem}

\begin{corollary}
Let $M$ be a closed manifold.  Then the set of representations which are holonomies of some $(G,X)$ structure on $M$ is open in the classical topology on $\Hom(\pi_1(M),G)$.	
\end{corollary}

\noindent 
Thus given that a representation $\rho\colon\pi_1(M)\to G$ is the holonomy of \emph{some} geometric structure, then nearby holonomies \emph{actually correspond} to nearby $(G,X)$ structures.
From this, one may hope that the holonomy actually locally determines everything, and $\hol$ is a local homeomorphism from deformation space.
This is called the Ehresmann-Thurston Principle, which holds in many cases, but is not true in complete generality (as Goldman notes in \cite{Goldman18}, Section 7.4, it was noticed by Kapovich \cite{Kap90} and Baues \cite{Baues10Crys} that this fails in specific cases, where local isotropy groups acting on $\Hom(\pi_1(M),G)$ may not fix marked structures on the corresponding fibers).

\noindent {\sffamily \bfseries Ehresmann-Thurston Principle:} The projection onto holonomy from deformation space $\hol\colon\mathcal{D}_{(G,X)}(M)\to\Hom(\pi_1(M),G)/G$ is a local homeomorphism, with respect to the described topology on $\mathcal{D}_{(G,X)}(M)$ and the quotient topology on $\Hom(\pi_1(M),G)/G$ induced from the classical topology on the real algebraic set $\Hom(\pi_1(M),G)$.


\subsection{Example Deformation \& Moduli Spaces}

We conclude this section with some example deformation spaces of geometric structures.

\begin{example}
Deformation space of Riemannian metrics on $\S^1$ is diffeomorphic to $(0,\infty)$, parameterized by circumference.  The moduli space is $\R_+$ as well.
\end{example}

\begin{example}
The deformation space of conformal tori is the Hyperbolic plane, thought of as rotation-classes of unit co-area lattices $\mathsf{D}_{\E^2}(T^2)=\Hyp^2=\SL(2;\R)/\SO(2)$.
The moduli space is the modular curve $\mathcal{M}_{\E^2}(T^2)=\SL(2,\Z)\textbackslash\SL(2;\R)/\SO(2)$.
\end{example}

\begin{example}
Deformation space of unit area Euclidean $n$-tori is the homogeneous space $\mathcal{D}_{\E^n}(T^n)=\SL(n;\R)/\SO(n)$, and the moduli space is the double quotient by the orientation preserving mapping class group $\mathsf{Mod}(T^n)=\SL(n;\Z)$.
\end{example}

\begin{example}
The deformation space of hyperbolic structures on a genus $g$ surface is homeomorphic to an open ball $\mathcal{D}_{\Hyp^2}(\Sigma_g)\cong\R^{6g-6}$.  The moduli space is the quotient by the action of the mapping class group.	
\end{example}

\begin{example}
The deformation space of hyperbolic structures on a compact manifold of dimension $\geq 3$ is empty or a single point, by Mostow Rigidity.	
\end{example}

\begin{example}
The deformation space of complete affine structures on the torus is diffeomorphic to $\R^2$ \cite{Baues14}.
\end{example}

\begin{example}
The moduli space of complete affine structures on the torus natural identifies with the quotient of $\R^2$ by the linear action of $\SL(2;\Z)$.
This space is non-Hausdorff, and in fact admits no noncontsant continuous maps into any Hausdorff space.
The deformation space is much better behaved, and is diffeomorphic to the plane \cite{BauesGoldman05}.
\end{example}

\section{Degenerations and Regenerations}
\label{sec:Degen_and_Regen}
\index{Geometric Structures! Regeneration}
\index{Regeneration}

Example \ref{Ex:Hopf_Torus_Two} shows, in the context of similarity vs. $\C^\times$ structures,  that a particular developing pair may be fruitfully be viewed as providing a geometric structure into distinct geometries, and its properties depend on the chosen geometry.
This is an example of a more general phenomenon which occurs whenever a geometry $(H,Y)$ arises as a subgeometry of $(G,X)$.
Any $(H,Y)$ structure on $M$ is determined by a developing pair $(f\colon\tilde{M}\to Y, \rho\colon\pi_1(M)\to H)$ which under the inclusions $Y\subset X$, $H<G$ determines a $(G,X)$ structure.

\begin{definition}
Let $\mathbb{Y}=(H,Y)$ and $\mathbb{X}=(G,X)$ be geometries, and $\iota=(\iota_G,\iota_X)\colon (H,Y)\inject (G,X)$ be a fixed monomorphism.
Then $\iota$ induces a map $\iota_\star\colon\mathcal{D}_{(H,Y)}(M)\to\mathcal{D}_{(G,X)}(M)$ defined by $\iota_\ast[f,\rho]_{\mathbb{Y}}=[\iota_X f,\iota_G \rho]_{\mathbb{X}}$ called \emph{weakening}, allowing all $\mathbb{Y}$ structures to be canonically viewed as $\mathbb{X}$ structures.
\end{definition}

\noindent 
Note that if $Y\neq X$ then complete $(H,Y)$ structures are never complete as $(G,X)$ structures.
While the structure $\iota_\ast[f,\rho]$ is determined by the same developing pair as the original; the notion of equivalence has changed and developing pairs must be considered up to the action of $G$ and not just $H$.

\begin{example}
The deformation space of Euclidean tori is homeomorphic to $\R^3$, parameterized by rotation classes of marked planar lattices.
All planar lattices are conjugate by affine transformations so the image of $\mathcal{D}_{\E^2}(T^2)$ under weakening in $\mathcal{D}_{\A^2}(T^2)$ is a point.
\end{example}

\begin{remark}
We often say strengthening for the reverse process...which isn't a well-defined map on deformation space but is only defined for particular developing pairs.	
\end{remark}

\noindent
Weakening into a more flexible ambient geometry is often useful when considering collapse of geometric structures.
A sequence of geometric structures \emph{degenerates} if the developing maps fail to converge to an immersion even after adjusting by diffeomorphisms of $M$ and coordinate changes in $G$.  Of particular interest are \emph{collapsing degenerations}, defined below.

\begin{definition}
A sequence $\{[f_n,\rho_n]\}\subset\mathcal{D}_{(G,X)}(M)$ \emph{collapses} if, after possibly adjusting by diffeomorphisms of $M$ and coordinate changes in $G$, the developing maps converge to a submersion $f_\infty$ into a lower-dimensional submanifold, which is preserved by the action of the algebraic limit $\rho_\infty$ of the holonomy homomorphisms.
\end{definition}

\begin{example}
A trivial example is given by the collapse of Euclidean manifolds under volume rescaling.  Given a Euclidean structure $(f,\rho)$ on a manifold $M^n$ and any $r\in\R_+$, the developing pair $(rf, r\rho)$ describes the rescaled manifold with volume $r^n$ times that of the original.  As $r\to 0$ these structures collapse to a constant map and the trivial holonomy.
\end{example}

More interesting examples include the collapse of hyperbolic structures onto a codimension-1 hyperbolic space as studied by Danciger \cite{Danciger11,Danciger11Ideal,Danciger13} and the collapse of hyperbolic and spherical structures in \cite{Porti10,Porti98}.

\noindent
Collapsing geometric structures can often be `saved' by allowing more flexible coordinate changes.
If a geometry $(H,Y)$ can be realized as an open subgeometry of $(G,X)$ then a sequence $(f_n,\rho_n)$ of collapsing $(H,Y)$ structures may actually converge \emph{as $(G,X)$ structures}, meaning there are $g_n\in G$ such that the developing pairs $g_n.(f_n,\rho_n)$ converge to a $(G,X)$ developing pair $(f_\infty, \rho_\infty)$.

\begin{example}
The sphere $\S^2(\alpha,\beta,\gamma)$ with three cone points of cone angles $\alpha,\beta,\gamma$ has a hyperbolic structure if $\alpha+\beta+\gamma<2\pi$ and a spherical structure when their sum is greater than $2\pi$.  The area of these structures collapse to $0$ (in metrics of constant curvature $\pm 1$) as $\alpha+\beta+\gamma\to 2\pi$, but this collapse may be averted by conjugation in $\RP^2$, limiting to a Euclidean structure.  The picture below shows this for the case $\alpha=\beta=\gamma=t$ for $t\in [0,2\pi/3)\cup (2\pi/3,2\pi]$.
\end{example}

\begin{figure}
\centering\includegraphics[width=0.85\textwidth]{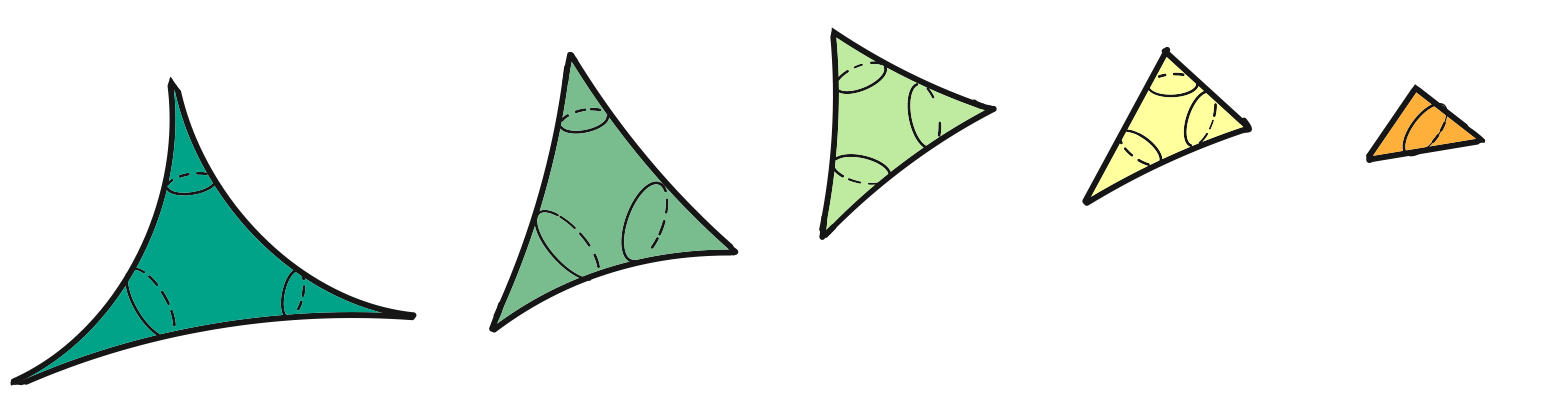}
\caption{Collapsing triangle orbifolds.}	
\end{figure}

\begin{example}
\label{Ex:Degen_Euc_Tori}
Let $\gamma\colon[0,\infty)\to\mathcal{T}_{\E^2}(T^2)$ be a collapsing path of unit area Euclidean structures on the torus (necessarily collapsing onto a circle). 
Weakening to affine geometry this is the constant path of affine translation tori, and so converges in $\mathcal{D}_{\A^2}(T^2)$ to the unique affine translation torus.
\end{example}

\noindent
When $f_\infty$ has image in an open subset $Z\subset X$ and $\rho_\infty$ maps into the subgroup $L<G$ of $Z$-preserving transformations, this $(G,X)$ strengthens to an $(L,Z)$ structure.
It is tempting to say that \emph{within} $(G,X)$ these $(H,Y)$ structures converge to an $(L,Z)$ structure.
Formalizing this notion motivates the field of \emph{transitional geometry}, discussed in \ref{chp:Limits_of_Geos}, and we will revisit Example \ref{Ex:Degen_Euc_Tori} again in Chapter \ref{chp:Heisenberg_Plane}, showing collapsing Euclidean tori rescale to a limit in the \emph{Heisenberg Plane}.

\begin{example}
Let $f\colon(0,1]\to\mathcal{D}_{\Hyp^2}(\S^1\times\R)$ be the path of hyperbolic cylinders with $f(x)$ the cylinder with geodesic neck of circumference $x$.
Viewed in the Klein model as a subgeometry of projective space, this sequence can be rescaled to have limiting projective structure the quotient of an affine patch by translation, which we may then view as a Euclidean cylinder.
\end{example}

\begin{figure}
\centering\includegraphics[width=0.5\textwidth]{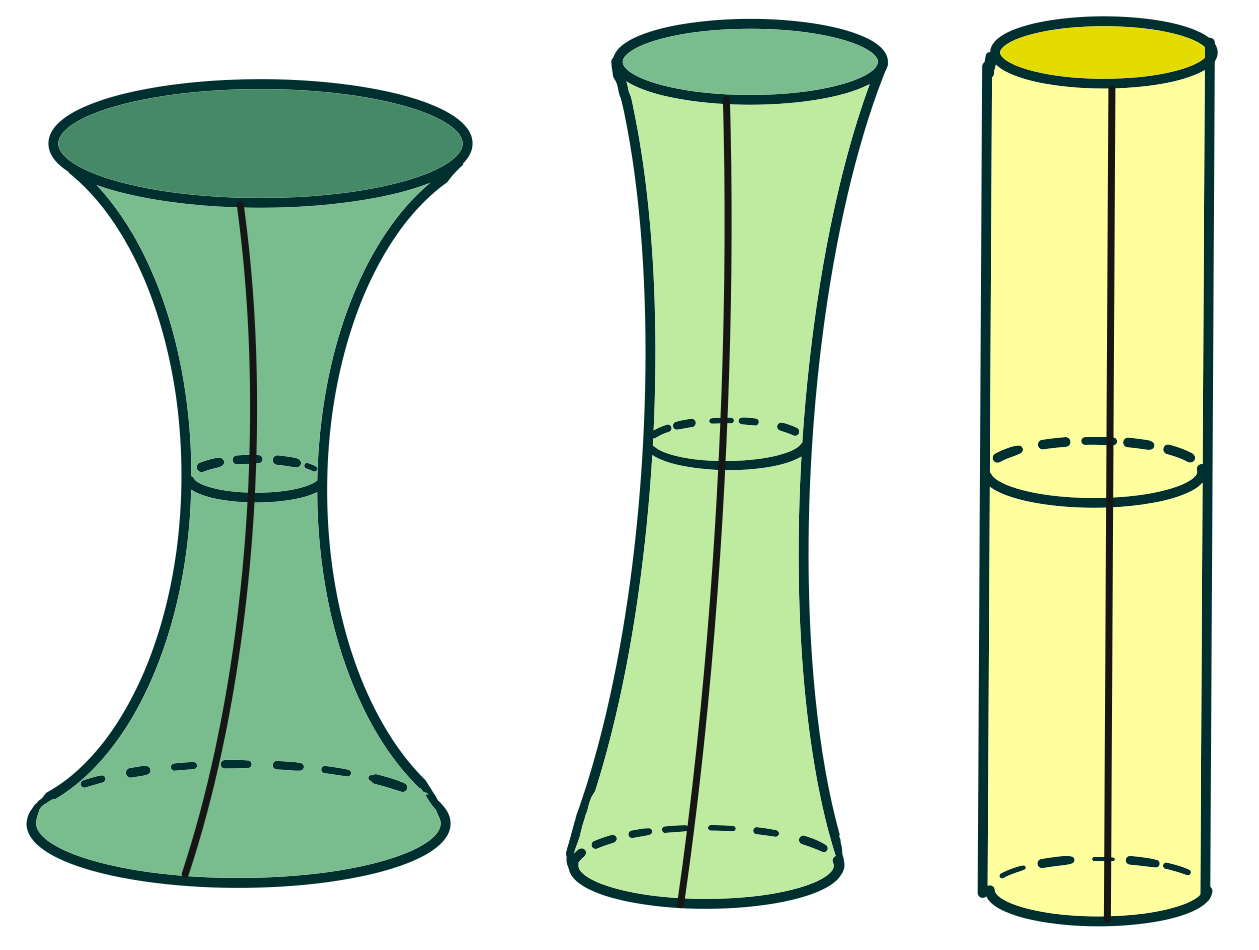}
\caption{Hyperbolic cylinders converging to a Euclidean cylinder.}	
\end{figure}

\begin{definition}
Let $(H,Y)$ and $(L,Z)$ be open subgeometries of $(G,X)$, and $[f_n,\rho_n]_{\mathbb{Y}}$ a collapsing sequence of $(H,Y)$ structures on a manifold $M$.
This sequence \emph{degenerates to an $(L,Z)$ structure in $(G,X)$} if there are representatives of the weakened structures $[f_n,\rho_n]_{\mathbb{X}}$ converging to a limiting $(G,X)$ developing pair $(f_\infty,\rho_\infty)$ with $f(X)\subset Z$ and $\rho_\infty(\pi_1(M))<L$.
\end{definition}

\begin{definition}
Let $(H,Y)$ and $(L,Z)$ be open subgeometries of $(G,X)$ and $[f,\rho]_\mathbb{Z}$ a $(L,Z)$ structure on a manifold $M$.
Then $[f,\rho]$ \emph{regenerates into $(H,Y)$} if there is a collapsing path of $(H,Y)$ structures on $M$ degenerating to $M$ in $(G,X)$.
\end{definition}

\section{Compactification}
\label{sec:Compactification}
\index{Moduli!Compactification}
\index{Blow Ups}

\begin{definition}
A \emph{compactification} of a space $X$ is a compact space $C$ together with an embedding $\iota\colon X\inject C$ with$\iota(X)$ open and dense in $C$.
\end{definition}

\begin{example}
The \emph{Thurston Compactification} of Teichm\"uller space adds to $\mathcal{T}(\Sigma_g)\cong\mathbb{B}^{6g-6}$ a sphere at infinity $\mathbb{S}^{6g-7}$ of points parameterizing degenerations of hyperbolic metrics, as singular measured foliations.	
\end{example}

\noindent
A compactification is connected if $X$ is, but disconnected spaces can also have connected compactifications (one compactification of the disjoint union of two open hemispheres is two closed disks, another is the sphere).
We call such connected compactifications \emph{simultaneous compactifications}, as they will be important in our discussion of the moduli of orthogonal groups in Chapter \ref{chp:Orthogonal_Groups}.

\begin{definition}
A \emph{simultaenous compactification} of a collection of spaces $\{X_i\}$ is a compact connected space $C$ together with an embedding $\iota\colon \sqcup_i X_i\inject C$ as an open dense subset.
\end{definition}

We will be thinking about compactifications of spaces of geometries, and thus mainly about compactification in the context of compactifying some parameter space of Lie groups.

\begin{observation}
Let $G$ be a locally compact topological group and $X\subset \Cl(G)$ a collection of closed subgroups.
The closure $\overline{X}\subset\Cl(G)$ is a compact space, called the \emph{Chabauty compactification} of $X$.	
\end{observation}

\begin{definition}
Let $(G,X)$ be a geometry.
The natural map $\mathsf{st}\colon X\to\Cl(G)$ sending each $x\in X$ to its stabilizer $\mathsf{stab}_G(x)$ under the $G$ action is a continuous injection, and the closure of the image $\overline{\mathsf{st}(X)}\subset\Cl(G)$ is the \emph{Chabauty compactification} of the homogeneous $G$-space $X$.
\end{definition}

\noindent
Different compactifications of a space are suited to different purposes, and we will informally call a certain compactification \emph{good} when it respects particular additional structure inherent to the problem.

\begin{example}
The sphere, viewed as the Riemann Sphere $\widehat{\C}$ is a good compactification of the plane from the context of complex projective geometry.
The real projective plane is a good compactification of the plane in real projective geometry, as here we require a full circle of directions to go to infinity, instead of just one.	
\end{example}

\begin{example}
The Thurston compactification of Teichm\"uller space is a \emph{good compactification} of $\mathcal{D}_{\H^2}(\Sigma_g)$ 
as the closed ball $\mathbb{B}^{6g-6}$
in the sense that the action of the mapping class group extends continuously.	
\end{example}

\noindent
Our particular use of compactifications is in Chapter \ref{chp:Orthogonal_Groups}, where we seek to understand the possible degenerations of orthogonal groups as subgroups of $\GL(n;\R)$.
The particular context is described in detail there, but to compute such compactifications we will make use of elementary tools from Real algebraic geometry, including the theory of blow-ups, which we recount below.

\subsection{Blow Ups over $\R$}

The material in this section is all certainly standard, but is included in relative detail as there seems to be few good sources for topologists to learn to use blowup constructions in their work.
In particular, I could not find a suitable source, and developed the following perspective in collaboration with Nadir Hajouji.
Intuitively, \emph{blowing up a space $X$ about a subspace $Y$} produces a space which \emph{remembers} infintesimal information about paths in $X$ limiting onto points of $Y$.
This replaces $X$ with a new space, $\mathsf{Bl}_Y(X)$ formed from  $X\smallsetminus Y$ and the space of directions approaching $Y$ in $X$.

\noindent
Our approach differs from the usual algebro-geometric introduction, and defines blowups as the topological closure of a graph rather than a vanishing set of polynomials.
As a first introduction to this approach, we reconsider the blowup of $\R^n$ at a point.

\begin{definition}
The blow up of $\R^n$ at $0$ is the closure of the graph of $\iota\colon\R^n\smallsetminus\{0\}\to\RP^{n-1}$ defined by $\iota(x,y)=[x:y]$.
\end{definition}

The map $\phi$ associates to each $\vec{x}\in\R^n$ the point in $\RP^n\smallsetminus\{0\}$ represented by $\span(\vec{x})$, and so the 
graph $\Gamma(\iota)$ of $\iota$ contains all pairs $(\vec{x},[\vec{x}])$.
Note $\iota$ is constant on all lines through the origin, and so cannot have a well defined limit at $0$ as $\iota$ is not the constant map.
Instead, the closure of $\Gamma(\iota)$ contains the entire $\RP^{n-1}$ factor above $0\in\R^n$, corresponding to each direction from which one can approach $0$ in $\R^n$.
Defining $\Bl_0(\R^n)=\overline{\Gamma(\iota)}$ as a graph closure provides a natural map $\Bl_0(\R^n)\to\R^n$ projecting onto the original domain, which is naturally $1-1$ away from $0$, but collapses the entire $\RP^{n-1}$ there to a point.

\begin{figure}
\centering\includegraphics[width=0.7\textwidth]{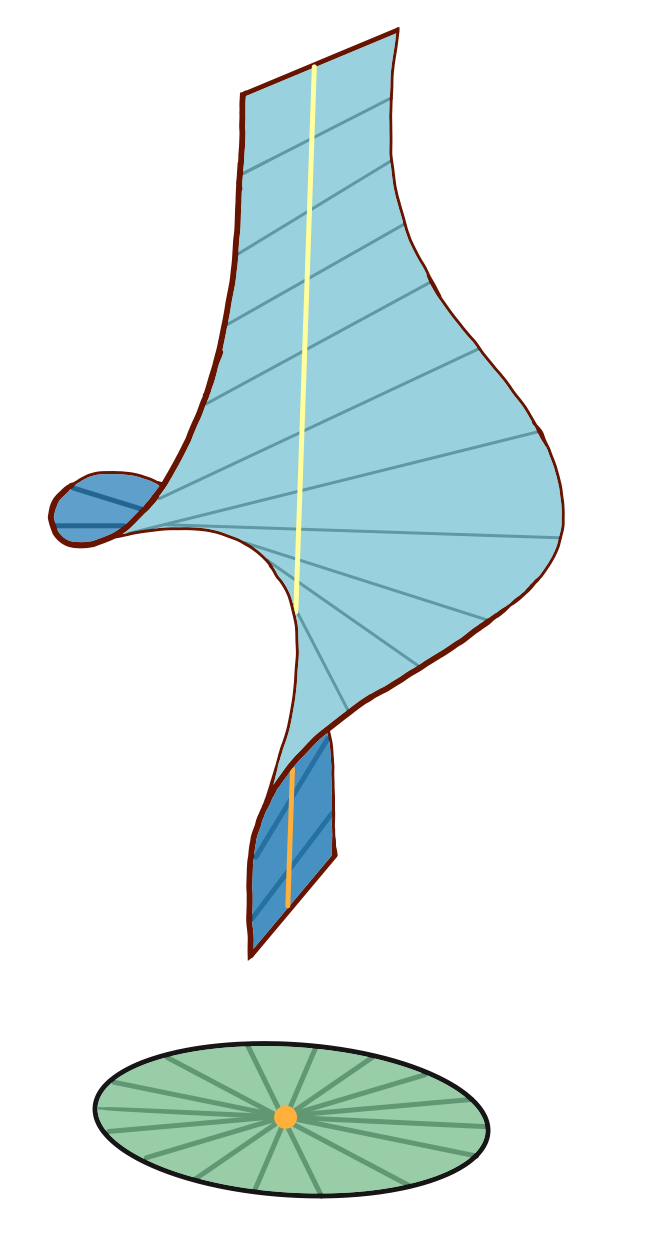}
\caption{The blow up in dimension 2.}
\end{figure}

\begin{observation}
The blow up $\mathsf{Bl}_{0}(\R^n)$ is an algebraic subvariety of $\R^n\times\RP^{n-1}$.
$$\mathsf{Bl}_{0}(\R^n)=\{((x_1,\ldots, x_n),[y_1,\ldots y_n])\mid x_iy_j-x_jy_i=0\}$$
\end{observation}

\noindent
This special case directly generalizes to define the blowup of $Y$ in the product manifold $X=Y\times\R^k$.
Projecting onto the $\R^k$ factor collapses $Y$ to a point, and the blowup of $X$ along $Y$ is simply the product of the blowup of $\R^k$ above with $Y$.

\begin{definition}
\label{Def:Blowup_Product}
Let $Y$ be a smooth manifold, then the blowup of $Y\subset Y\times \R^k$ is the closure of the graph of $\iota\colon Y\times\R^k\smallsetminus Y\times\{0\}\to \RP^{k-1}$ defined by $\iota((y_1,\ldots y_n),(x_1,\ldots, x_k))=[x_1:\cdots x_k]$.
\end{definition}

\begin{observation}
This is just $Y$ times the blowup of $\R^k$ at $0$.	
\end{observation}

\noindent 
Here similarly $\iota$ associates to a point $p\in X$ the point $[v]\in\RP^{k-1}$ giving the direction of the line segment connecting $p$ to $Y$ in a fixed slice $\R^k$.
Geometrically, this is the projective tangent vector $[v]\in \P T_y \R^k$ of the shortest geodesic connecting $x$ to the closest point $y\in Y$, for the product metric of Euclidean $\R^k$ with any Riemannian metric on $Y$.

\noindent 
This in turn, is a special case of the blow up $\Bl_Y(\fam{E})$ of a vector bundle over $Y$. 

\begin{definition}
\label{Def:Blowup_VecBundle}
Let $\mathcal{E}\to Y$ be a $k$-dimensional real vector bundle over $Y$, and $\mathcal{P}\to Y$ the associated fiber bundle of projective spaces, with projection $\pi\colon\fam{E}\to\fam{P}$ over $Y$.
Then the blowup of $\fam{E}$ along $Y$ (identified with the zero section) is the closure of the graph of $\pi$ restricted to the submanifold $\fam{E}\smallsetminus Y$.
\end{definition}

\begin{observation}
This results in a fiber bundle $\Bl_Y(\fam{E})\to Y$ which effectively replaces each fiber $\R^k$ in $\fam{E}$ with $\Bl_0(\R^k)$.	
\end{observation}

\noindent 
This immediately allows a (relatively) coordinate-free description of the blow up about a submanifold $Y$ of a manifold $X$ via the tubular neighborhood theorem.

\begin{theorem}[Tubular Neighborhoods]
Let $X$ be a smooth manifold, and $Y\subset X$ a smooth submanifold with normal bundle $N_Y(X)\to Y$.
Then there is an open neighborhood $U\subset X$ of $Y$, a convex open neighborhood $V$ of the zero section $\iota_0\colon Y\to N_Y(X)$ of the normal bundle, and a diffeomorphism $\phi\colon V\to U$ such that the following commutes:

\begin{center}
\begin{tikzcd}
Y\ar[d,hook]\ar[dr,"\iota_0"]\\
X&N_Y(X)\ar[l,"\phi"]
\end{tikzcd}	
\end{center}

Rescaling, we may take $V=N_Y(X)$ without loss of generality.
Such a neighborhood $U$ is called a \emph{tubular neighborhood} of $Y$ in $X$.
\end{theorem}

\begin{definition}
Let $Y\subset X$ be an embedded submanifold of a smooth manifold $X$, and $U\subset X$ a tubular neighborhood of $Y$ identified with the normal bundle $N_Y(X)\to Y$ via the homeomoprhism $\phi\colon N_Y(X)\to U$.
Then the blowup $\Bl_Y(X)$ is defined as follows.  Form the blowup $\Bl_Y(N_Y(X))$ as in Definition \ref{Def:Blowup_VecBundle}, and note that the projection onto the domain $p\colon \Bl_Y(N_Y(X))\to N_Y(X)$ is a homeomorphism away from $Y$.  Thus $\phi \circ p\colon \Bl_Y(N_Y(X))\to U$ is a homeomorphism away from $Y$, and we define
$$\Bl_Y(X)=\Bl_Y(N_Y(X))\sqcup (X\smallsetminus Y)/\sim$$
where $x\in \Bl_Y(N_Y(X))\smallsetminus Y$ is related to $\phi(p(x))\in X\smallsetminus Y$. 
\end{definition}

\noindent 
We will have no direct need for this general construction here, as working locally in any coordinate chart every submanifold $Y\subset X$ looks like $\R^k\subset \R^n$ and we may construct a local model of the blowup directly using Definition \ref{Def:Blowup_Product}.
In fact, in our applications in Chapter \ref{chp:Orthogonal_Groups}, we do not set out with the intent of constructing a blowup but rather the closure of some embedding, and only after realize in coordinate charts that the result is actually a sequence of blowups.

\part{Geometric Transitions}

\thispagestyle{plain}

\noindent

\vspace{0.5cm}
{\sffamily\bfseries\LARGE Limits of Geometries}
Reviews the standard definitions and examples of \emph{geometric transitions} in low dimensional topology.
We review the construction of the topology on the space of subgeometries of a Klein geometry $(G,X)$ through the Chabauty topology on its automorphism group $G$, and methods of computing in this space; particularly in the special case of \emph{conjugacy limits}.
We then review the classic example of a transition: the degeneration of both hyperbolic and spherical space to Euclidean in the limit as curvature approaches zero.
We provide a detailed exposition of formalizing this transition as a collection of \emph{subgeometries of projective space} as this is a model for more general conjugacy limits in $\GL(n;\R)$ such as those studied by Cooper, Danciger and Wienhard, which we review next.

\vspace{0.5cm}
{\sffamily\bfseries\LARGE Orthogonal Groups in $\GL(n;\R)$}
This chapter presents a new approach to the classification of conjugacy limits of the quadratic form geometries in $\RP^n$, recovering the results of Cooper, Danciger and Wienhard in \cite{CooperDW14}, while also providing a description of the \emph{Chabauty closure} of the set of orthogonal groups in $\GL(n;\R)$.
Most notably, the techniques utilized in this alternative approach do not require actually computing conjugacy limits along paths, and so may be applicable even in cases where it is no longer true that all limits occur via conjugation by one parameter subgroups.

\vspace{0.5cm}
{\sffamily\bfseries\LARGE The Heisenberg Plane}
The classification of limits of the quadratic form geometries $(\O(p,q),X(p,q))$ shows that each dimension has a unique \emph{most degenerate} geometry, to which all quadratic form geometries can degenerate to through conjugacy.
This chapter presents a detailed case study of this geometry in dimension two, which is given by the projective action of the Heisenberg group on the affine plane.
In particular, the closed orbifolds admitting Heisenberg structures are classified, and their deformation spaces are computed. 
Considering the regeneration problem, which Heisenberg tori arise as rescaled limits of collapsing paths of constant curvature cone tori is completely determined in the case of a single cone point.

\vspace{0.5cm}
{\sffamily\bfseries\LARGE $\Hyp_\C$ and $\Hyp_{\R\oplus\R}$}
Generalizing the construction of complex hyperbolic space, this chapter investigates the other analogs of hyperbolic geometry which can be created through substituting $\R$ with other two dimensional real algebras.
Up to isomorphism there are three such geometries, the familiar $\Hyp_\C^n$, together with $(\R\oplus\R)$ hyperbolic space and hyperbolic space over $\R[\ep]/(\ep^2)$.
A surprising connection between $\R\oplus\R$ hyperbolic space and the geometry of $\RP^n$ is unearthed as well.

\vspace{0.5cm}
{\sffamily\bfseries\LARGE The Transition $\Hyp({\R[\sqrt{\delta}]})^n$}
The algebras $\C$, $\R[\ep]/(\ep^2)$ and $\R\oplus\R$ represent three algebraic structures on $\R^2$ which can be deformed into one another.
In this chapter we show this continuity actually implies the existence of a new transition of geometries connection $\Hyp_\C$ to $\Hyp_{\R\oplus\R}$ through $\Hyp_{\R_\ep}$.
Together with the relationship between $\Hyp_{\R\oplus\R}^n$ and $\RP^n$, this provides a means of relating real projective and complex hyperbolic deformations of hyperbolic manifolds.

\clearpage

\chapter{Limits of Geometries}
\label{chp:Limits_of_Geos}
\index{Limits of Geometries}
\index{Geometric Limits}
\index{Geometry!Limits}

\section{The Space of Closed Subgroups}
\label{sec:Closed_Subgroups}
\index{Limits of Geometries!Chabauty}
\index{Chabauty Topology}
\index{Space of Closed Subgroups}

Given a topological space $X$, the \emph{hyperspace of closed subsets} is denoted $\Cl(X)$.
When $X$ is a compact metric space, $\Cl(X)$ inherits a topology from the \emph{Hausdorff metric}.

\begin{definition}
\label{def:Hausdorff_Metric}
Let $(X,d)$ be a compact metric space and $\Cl(X)$ the hyperspace of closed subsets.
The metric $d$ induces a \emph{Hausdorff distance} on $\Cl(X)$, given by

$$d_H(A,B)=\max 
\big\{ 
\sup_{a\in A}\inf_{b\in B}d(a,b),\sup_{b\in B}\inf_{a\in A}d(a,b)
 \big\}
 $$
 $$
 =\inf\{\ep\geq 0\mid A\subset N_{\ep}(B) \textrm{ and } B\subset N_\ep(A)\}
 $$
 for $N_\ep(Y)$ the set of points lying at most distance $\ep$ in $(X,d)$ from some point of $Y$.
\end{definition}

\begin{figure}
\centering\includegraphics[width=0.5\textwidth]{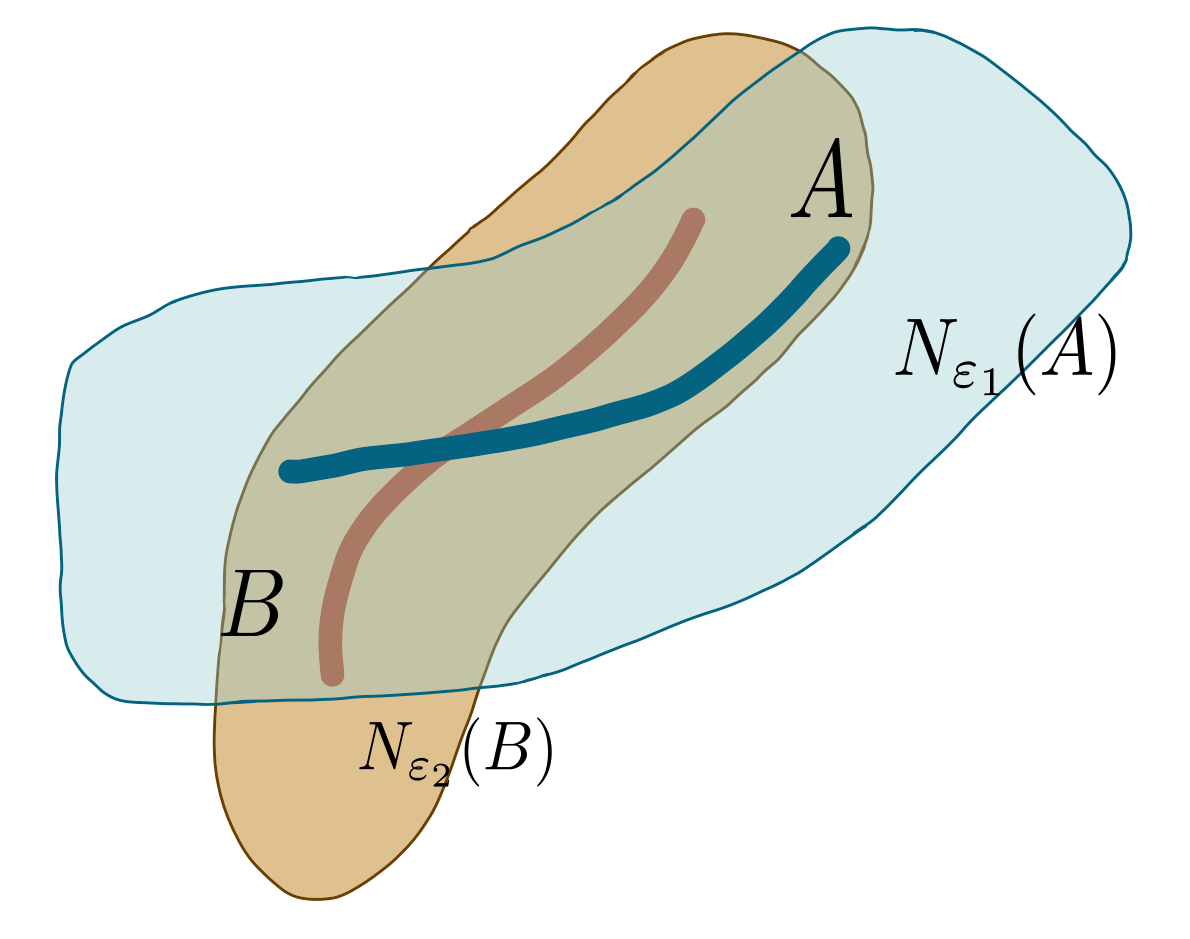}
\caption{Determining Hausdorff distance.}	
\end{figure}

\noindent
The \emph{Hausdorff topology} induced by this metric makes $\Cl(X)$ into a compact space.
More surprisingly perhaps, this topology is independent of the original metric on $X$, and so all metrizable compact spaces $X$ have a natural topology on $\Cl(X)$.
When $X$ is noncompact the formula given in Definition \ref{def:Hausdorff_Metric} fails to define a metric, as distances between sets can be infinite and disjoint closed sets can fail to be separated by any $\ep$ neighborhoods.

\begin{example}
Any two nonparallel lines in the plane are not contained in any $\ep$ neighborhood of each other and so have infinite Hausdorff distance.
\end{example}

\noindent
One method of extending the Hausdorff topology to noncompact spaces restricts the Hausdorff metric on the one point compactification.
This is justified by the lemma below, whose proof appears in Section 2 of \cite{BaikC13}.

\begin{lemma}
Let X be a second-countable, locally compact metrizable space.
Then the one point compactification $\overline{X}=X\cup\{\infty\}$ is metrizable.
\end{lemma}

\begin{proposition}
Let $M$ be any manifold.  
The Hausdorff topology on $\Cl(\overline{M})$ restricts to $\Cl(M)$, and extends the Hausdorff topology on $\Cl(K)\subset\Cl(M)$ for every compact $K\subset M$.
\end{proposition}
\begin{proof}
$M$ is second countable locally compact and metrizable, so the one point compactification $\overline{M}$ is metrizable, with metric 
$d_{\overline{M}}$.
Topologize $\Cl(\overline{M})$ with respect to the Hausdorff metric 
induced by $d_{\overline{M}}$.
The natural inclusion $M\inject \overline{M}$ induces an inclusion $\Cl(M)\inject\Cl(\overline{M})$ sending compact sets to themselves and noncompact closed sets $F\subset M$ to $F\cup\{\infty\}$.
We use this inclusion to pull back the topology on $\Cl(\overline{M})$ to a topology $\mathcal{T}_M$ on $\Cl(M)$.

For any compact $K\subset M$, choosing a metric on $K$ topologizes   $\Cl(K)$ via the Hausdorff topology.
Note that subset $U\subset\Cl(K)$ is open if and only if it is open in $\Cl(M)$ as everything is occurring in a compact set away from $\infty$.
That is, the natural inclusion map $\Cl(K)\inject\Cl(M)$ is continuous, and in fact a continuous bijection onto its image from the compact space $\Cl(K)$ into the Hausdorff space $\Cl(M)$.
Thus the inclusion is a homeomorphism, and $\mathcal{T}_M$ extends the Hausdorff topology on $K$.
\end{proof}

\begin{definition}
The Chabauty topology on $\Cl(M)$ is the restriction of the Hausdorff topology on $\Cl(\overline{M})$.	
\end{definition}

\noindent
This topology was introduced by Chabauty in 1950 \cite{Chabauty50} and independently by Fell in 1962 
.
Over the years it has went by a number of names, including the Chabauty Topology, Fell Topology, and geometric topology (due to Thurston).
For additional reference material, consult \cite{Fisher09,Fisher09-II}.
Some properties of the hyperspace $\Cl(X)$ topologized by the Chabauty topology are that it is compact and metrizable \cite{Biringer18}, independent of any further assumptions on the topology of $X$.
The Chabauty topology is a so-called \emph{hit-and-miss} topology on the hyperspace of closed sets, due to a particularly convenient description in terms of subbasic open sets.

\begin{definition}
The Chabauty topology on $\Cl(X)$ is generated by the subbasis $\mathcal{O}_{K,U}$ of open sets indexed by pairs of a compact $K$ and open $U$ in $X$.
$$\fam{O}_{K,U}=\{Z\in\Cl(M)\mid Z\cap U\neq \varnothing, Z\cap K=\varnothing\}$$
\end{definition}

\begin{figure}
\centering\includegraphics[width=0.5\textwidth]{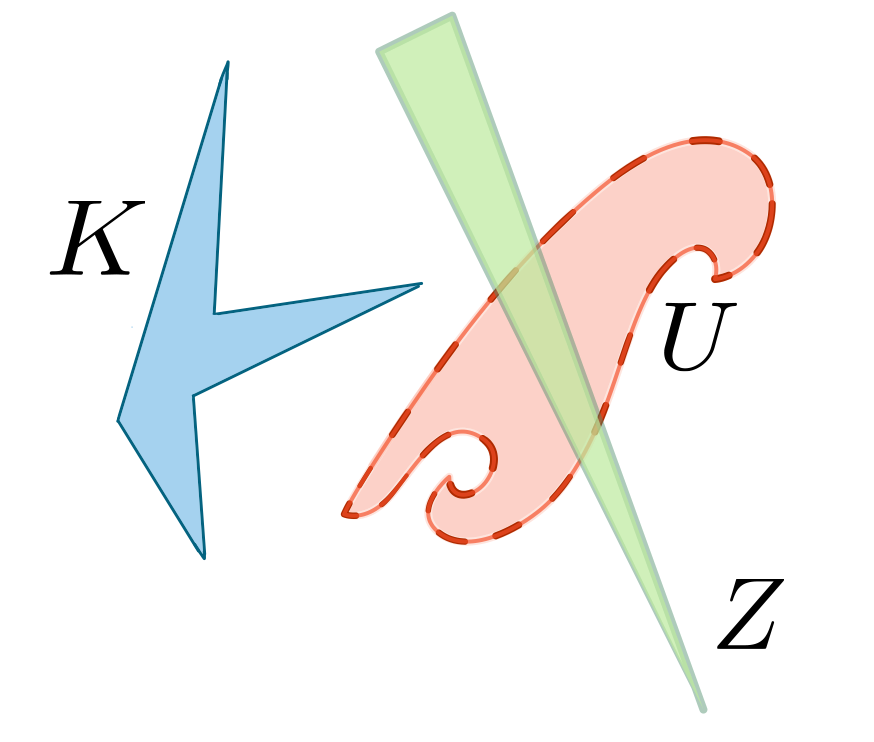}
\caption{Elements of the subbasic open set $\fam{O}_{K,U}$.}	
\end{figure}

\noindent
As $\mathfrak{C}(M)$ is metrizable, it is a sequential space and the Chabauty topology may be completely described by the convergence of sequences instead of specifying the open sets.

\begin{definition}
\label{def:Chab_Seqs}
The Chabauty topology on $\Cl(X)$ is the topology of \emph{subsequential convergence}: a sequence $\{Z_n\}\subset \Cl(X)$ converges to $Z_\infty$ 
if every subsequence $z_{n_k}\in Z_{n_k}$ of points converging in $X$ has limit $z_\infty\in Z_\infty$, 
and $Z_\infty$ is minimal with respect to this: every $z\in Z_\infty$ is the limit of some convergent subsequence $z_{n_k}\in Z_{n_k}$.
\end{definition}

\noindent
Continuity with respect to the Chabauty topology captures the notion closed subsets evolving into nearby closed sets.

\begin{example}
Let $f\colon[-1,1]\to\Cl(\R^3)$ be the function sending $t$ to the closed subvariety $f(t)=V(x^2+y^2-z^2-t)$.
Then $f$ is Chabauty continuous, and the hyperboloid of 2 sheets can transition to the hyperboloid of one sheet through a cone in $\R^3$.
\end{example}

\begin{figure}
\centering\includegraphics[width=0.85\textwidth]{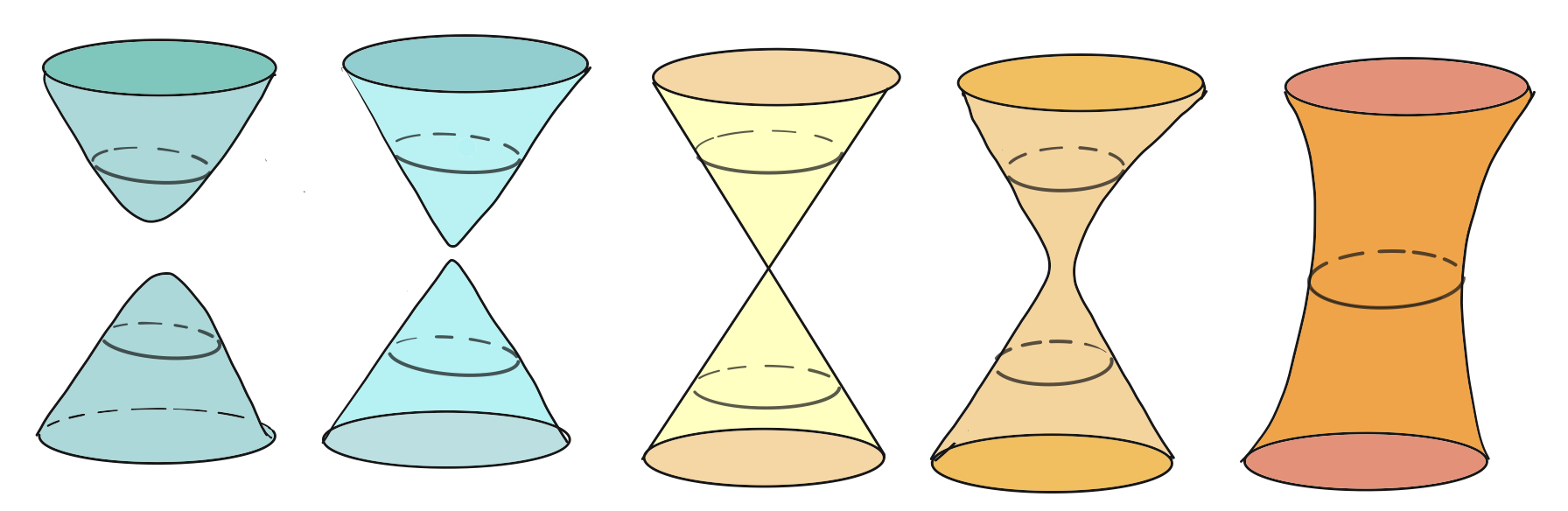}
\caption{The continuous path $V(x^2+y^2-z^2-t)$ of subvarieties of $\R^3$.}
\end{figure}

\noindent
Much wilder behavior is also possible, making the Chabauty space challenging to work with.
As an extreme case, the limit of a sequence of points can become a cube of arbitrary dimension.
The 1-dimensional case is given below.

\begin{example}
Let $f\colon (0,1]\to\R$ be the topologists' sine curve $f(t)=\sin(1/t)$ and	 consider associated map $\hat{f}\colon(0,t]\to\Cl(\R)$ given by $t\mapsto\{f(t)\}$.
Then $\widehat{f}$ extends continuously $0$ with $\widehat{f}(0)=[-1,1]$ the entire closed interval.
Thus in the Chabauty space of the line, a sequence of points can converge to a closed interval.
\end{example}

\noindent
When $G$ additionally has the structure of a topological group, our main interest is in the subset of $\Cl(G)$ of \emph{closed subgroups}.
This is closed in the full hyperspace of closed subsets, so limit points, closures, and compactification can be taken with respect to either space.

\begin{lemma}
The space of closed subgroups is closed in the space of closed subsets, for a second countable locally compact topological group $G$.
\end{lemma}
\begin{proof}
Let $G_n$ be a sequence of closed subgroups of $G$, converging in $\Cl(G)$ to a limiting point $G_\infty$.
Let $g,h\in G_\infty$.  We now show that $gh$ and $g\inv\in G_\infty$, so that $G_\infty<G$ is a subgroup.
Let $g_\ell,h_\ell\in G_\ell$ be sequences converging to $g,h$ in $G$, and consider their product $g_\ell h_\ell\in G_\ell$.
This sequence converges as both factors do; and as $G_n\to G$, the limit $gh\in G_\infty$.
Similarly, for each $\ell$ the sequence $g_\ell\inv$ lies in $G_\ell$ and converges to $g\inv$ in $G$; thus $g\inv\in G_\infty$ so $G_\infty$ is a group.
\end{proof}

\noindent
We repurpose the notation $\Cl(G)$ to mean the hyperspace of \emph{closed subgroups} when $G$ is a topological group.
While much more manageable than the entire hyperspace of closed subsets, the topology on $\Cl(G)$ is still difficult to work with in general.

\begin{example}[The space $\Cl(\R)$]
A closed subgroup of $\R$ is either $\R$ itself or discrete and so  either trivial or isomorphic to $\Z$.
Thus the Chabauty space is 
homeomorphic to the closed interval $[0,\infty]$, via the map $f\colon[0,\infty]\to\Cl(\R)$ with $f(0)=\R$, $f(\infty)=\{0\}$ and $f(\alpha)=\alpha \Z$.
\end{example}

\begin{figure}
\centering\includegraphics[width=0.5\textwidth]{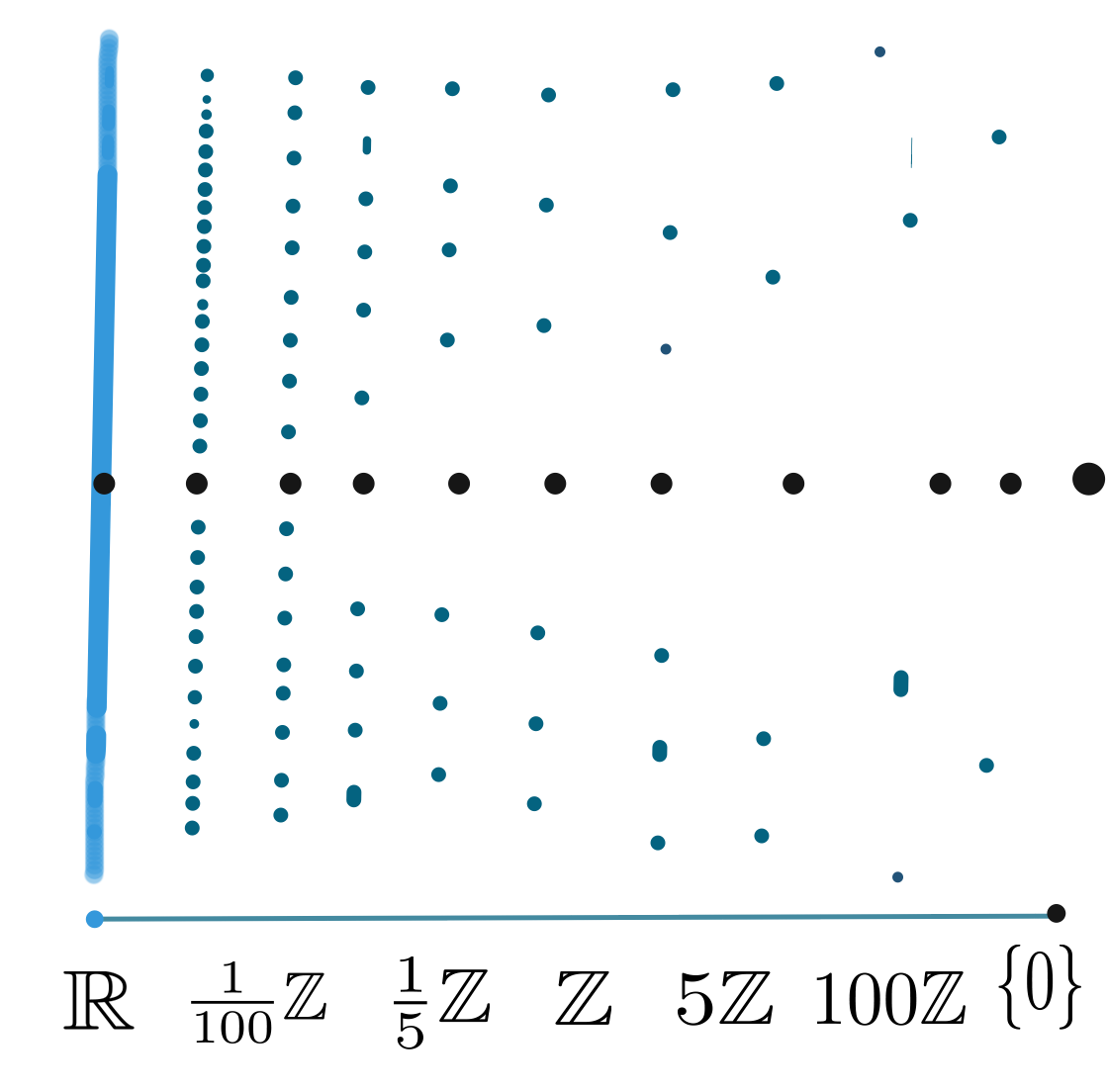}
\caption{Points in the Chabauty space $\Cl(\R)$.}	
\end{figure}

\begin{example}[The space $\Cl(\C)$]
\label{ex:Cl_C}
A closed subgroup of the plane is either $\{0\},\R,\R^2$, or $\Z,\Z^2,\Z\times\R$.
By the work of Hubbard and Pourezza \cite{PourezzaH79}, $\Cl(\C)$ is homeomorphic to the 4-sphere, realized as the suspension of $\S^3$ with suspension points $\{0\}$ and $\R^2$.
The lattices form an open dense subset, and their complement is a non-flatly embedded $2$ sphere of degenerations, which is the suspension of a trefoil knot in t $\S^3$.
The Chabauty spaces of $\R^n$ have been studied by Kloeckner \cite{Kloeckner09}, though are no longer manifolds for $n>2$.
\end{example}

\begin{figure}
\centering\includegraphics[width=0.5\textwidth]{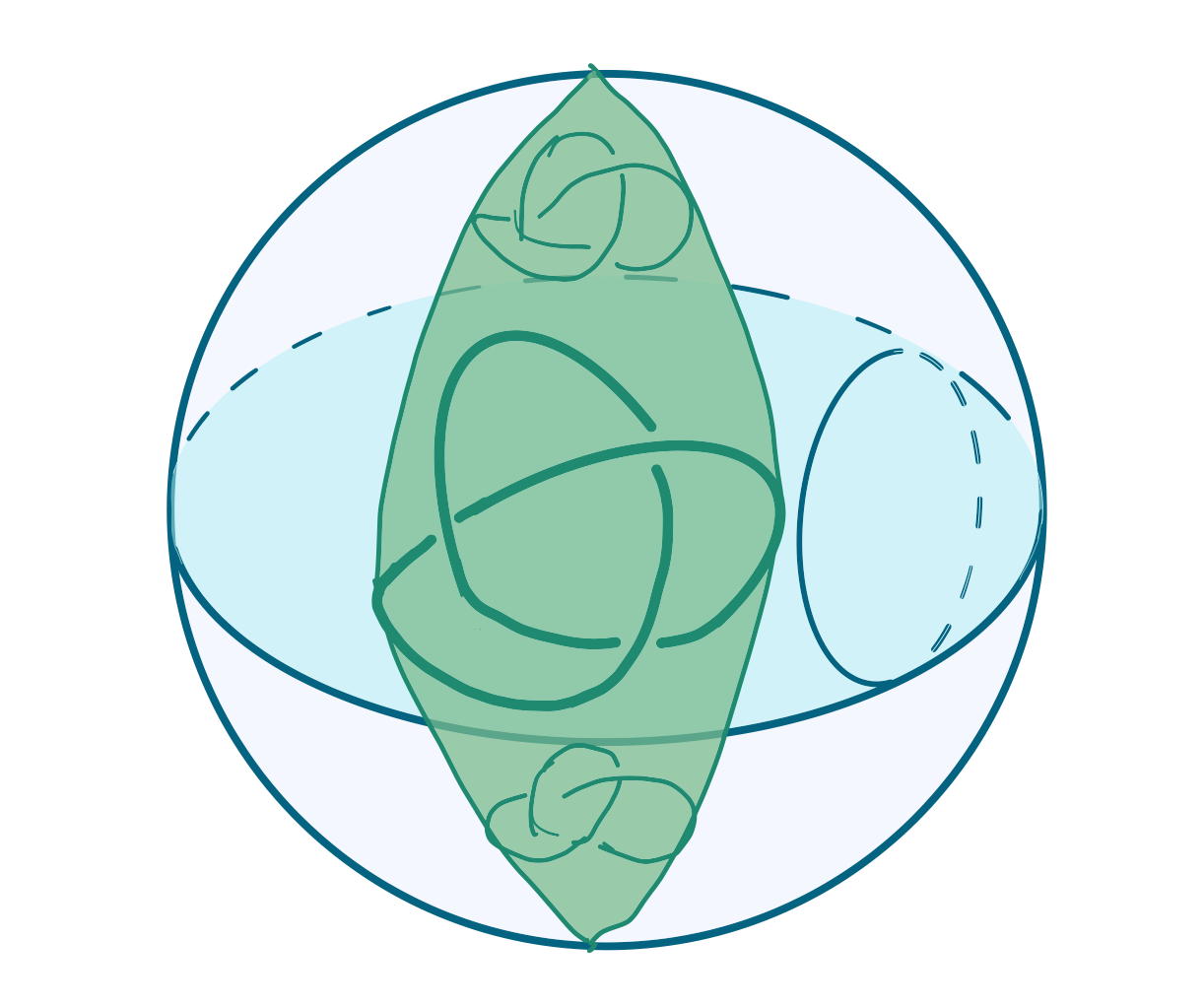}
\caption{The Chabauty space $\Cl(\C)$.
The suspension of the trefoil knot, in green, represents the subgroups of $\C$ which are not lattices.}	
\end{figure}

\noindent
Limit points of a collection $\mathcal{S}\subset\Cl(G)$ represent ways that the elements of $\fam{S}$ can \emph{degenerate} inside of $G$.
A common use for this is understanding the limiting behavior of subgroups of a Lie group $G$ under conjugacy, as studied by Haettel \cite{Haettel10,Haettel11,Haettel12}, as well as Leitner \cite{Leitner14,Leitner16,Leitner16-Cusp}.
Focusing on the Cartan subgroup of $\SL(n;\R)$ this work has been applied by Ballas, Cooper and Leitner to the study of cusps on real projective manifolds \cite{BallasCL17}.
Our use for the Chabauty space $\Cl(G)$ is as a means of topologizing the space of subgeometries of a fixed geometry $(G,X)$.
This allows us to talk about continuous variation of subgeometries, as well as take limits.

\section{The Space of Subgeometries}
\label{sec:Space_of_Subgeos}
\index{Geometries!Space of Subgeometries}
\index{Limits of Geometries! Space of Subgeometries}

Fixing a geometry $(G,(X,x))$, recall that a \emph{subgeometry} is a pair $(H,(Y,x))$ of a closed subgroup $H$ acting transitively on a submanifold $Y\subset X$. 
The set of subgeometries of $(G,(X,x))$ is denoted $\mathfrak{S}_{(G,X)}$.
An \emph{open subgeometry} of $(G,(X,x))$ is a pair $(H,(Y,x))$ of a closed subgroup $H<G$ acting transitively on an \emph{open submanifold} $Y\subset X$, with the set of open subgeometries of $(G,X)$ denoted $\mathfrak{S}^O_{(G,X)}\subset\mathfrak{S}_{(G,X)}$.
Limits of open subgeometries of the group-space variety were first formalized by Cooper, Danciger and Wienhard in \cite{CooperDW14}.
Utilizing the equivalence of categories between the Group-Space and Automorphism-Stabilizer perspectives, we find it more convenient to topologize the space of subgeometries of $(G,(X,x))\cong(G,\mathsf{stab}_G(x))$ using only the topology of $\Cl(G)$.

\begin{definition}
\label{def:Subgeo_Space}
The space of subgeometries of $(G,K)$ is given by $\mathfrak{S}_{(G,K)}=\{(H,C)\mid H\in\Cl(G),\; C=H\cap K\}$, topologized as a subset of $\Cl(G)\times\Cl(K)$
\end{definition}

\begin{definition}
The space of open subgeometries of $(G,X)$ is given by $\mathfrak{S}^O_{(G,K)}=\{(H,C)\mid H\in\Cl(G),\; C=H\cap K,\;\dim{G}-\dim K=\dim H-\dim(H\cap K)\}$, topologized as a subset of $\mathfrak{S}_{(G,K)}\subset\Cl(G)\times\Cl(K)$.
\end{definition}

\noindent
Having a topology on the set of subgeometries allows us to make precise the notion of a \emph{limit of geometries}: a sequence $(H_n,Y_n)$ of subgeometries of $(G,X)$ is \emph{convergent} if it converges in $\mathfrak{S}_{(G,X)}$.
The particular limits of interest here are \emph{conjugacy limits}, as studied by Cooper Danciger and Wienhard in \emph{Limits of Geometries}.
The definition in \emph{Limits of Geometries} differs from this in wording but is equivalent in practice, as we show below.

\begin{definition}[Conjugacy Limit]
\label{Def:Conj_Lim}
A sequence $(H_n,Y_n)$ converges in as subgeometries if it converges in $\mathfrak{S}_{(G,X)}$.
A subgeometry $(L,Z)$ is a \emph{conjugacy limit} of $(H,Y)$ in $(G,X)$ if there is a sequence $\{g_n\}\subset G$ such that $g_n.(H,Y)=(g_n Hg_n\inv, g_n Y)$ converges in $\mathfrak{S}_{(G,X)}$.
\end{definition}

\begin{definition}[Conjugacy Limit: Cooper Danciger \& Wienhard]
\label{Def:CDW_Limit}
A sequence of subgeometries $(H_n,Y_n)<(G,X)$ converges to the subgeometry $(L,Z)<(G,X)$
if $H_n$ converges geometrically to $L$ and there exists $z\in Z\subset X$ such that for all $n$ sufficiently
large $z\in Y_n$.
We say that a subgeometry $(L,Z)$ is a \emph{conjugacy limit} (or just \emph{limit}) of $(H,Y)$ in $(G,X)$ if there is a sequence $g_n\in G$ such that the conjugate subgeometries $(g_nHg_n\inv, g_nY)$ converge to $(L,Z)$.	
\end{definition}

\begin{proposition}
Let $(L,Z)$ be a conjugacy limit of $(H,Y)$ in $(G,X)$ by the original definition of Cooper Danciger and Wienhard.
Then there is a choice of basepoints such that $(L,(Z,z))$ is a conjugacy limit of $(H,(Y,z))$ in $(G,(X,z))$ in the sense of Definition \ref{Def:Conj_Lim}. 
\end{proposition}
\begin{proof}
Let $g_n$ be such that $H_n=g_n H g_n\inv$ converges to $L$ in $\mathfrak{C}(G)$, and $z\in Z$ such that $z\in g_n Y$ for all sufficiently large (and thus, without loss of generality, all) $n$.
Let $C=\mathsf{stab}_H(z)$, $C_n=\mathsf{stab}_{H_n}(z)$, and $K=\mathsf{stab}_G(z)$.
Then $(H_n,C_n)$ is a subgeometry of $(G,K)$ for all $n$, and as $n\to \infty$ the stabilizing subgroup $	C_n=g_nCg_n\inv$ converges (as a sequence of subgroups of a convergent sequence of groups) to the limiting stabilizer of $z$ under the action of $L$.
Thus $(H,C)=(H,\mathsf{stab}_H(z))$ converges under $g_n$ conjugacy to $(L, \mathsf{stab}_L(z))$.
The $L$ orbit of $z$ is $Z\subset X$ (as $(L,Z)$ is a geometry, $L$ acts transitively on $Z$).
\end{proof}

\noindent
The set of all \emph{conjugacy limits} of $(H,Y)$ in $(G,X)$ is the collection of all limit points of sequences $g_n.(H,C)$ in $\mathfrak{S}_{(G,X)}$.
Geometrically, this collection of points represents the boundary of the set of conjugates of $(H,Y)$ in $(G,X)$, providing us a topological object (\emph{the Chabauty compactification}) parameterizing all conjugates of $(H,Y)$ together with all limits.

\begin{definition}
Let $(H,Y)$ be a subgeometry of $(G,X)$.  Then $G.(H,Y)\subset \mathfrak{S}_{(G,X)}$ is the set of all conjugate geometries
$G.(H,Y)=\{g.(H,Y)\mid g\in G\}$ and its Chabauty compactification $\overline{G.(H,Y)}$ is its closure in the Chabauty space $\mathfrak{S}_{(G,X)}$.	
\end{definition}

\noindent
There are many natural questions one can ask about the limits of subgeometries of $(G,X)$ which can be phrased geometrically from this perspective.

\begin{SingleSpace}
\begin{itemize}
\item What are all the possible conjugacy limits of $(H,Y)$ = calculate the Chabauty closure $\overline{G.(H,Y)}$.
\item Which geometries are conjugacy limits of $(H,Y)$ = what are the isomorphism types of points in $\partial(G.(H,Y))=\overline{G.(H,Y)}\smallsetminus G.(H,Y)$?
\item Do $(H,Y)$ and $(H',Y')$ share a common conjugacy limit = do the Chabauty closures $\overline{G.(H,Y)}$ and $\overline{G.(H',Y')}$ intersect?
\end{itemize}
\end{SingleSpace}

\noindent
Restricting to algebraic groups (which, for example, covers the classical subgeometries of projective space) Cooper, Danciger and Wienhard additionally observed that there was a natural poset structure on the set of limit groups, and thus on limits of subgeometries.

\begin{theorem}[Cooper, Danciger, Wienhard]
Let $G$ be an algebraic Lie group. The relation of being a connected geometric limit
induces a partial order on the connected, algebraic, sub-groups of $G$. Moreover the length of
every chain is at most $\dim G$.
\end{theorem}

\noindent
Geometrically, this means the partition of the Chabauty closure $\overline{G.(H,Y)}$ into conjugacy classes can be equipped with a partial ordering, stratifying the space of limits into "more degenerate" and "less degenerate" geometries.
We see in Chapter \ref{chp:Orthogonal_Groups} that this stratification actually arises naturally when studying orthogonal groups; division into conjugacy classes gives a cellulation of $\overline{G.(H,Y)}$ and the partial ordering is by inclusion of lower dimensional cells in the boundary of higher dimensional ones.

Recalling the notions of equivalence in Chapter \ref{chp:Klein_Geo}, there are many models of Klein geometries that at times we want to consider equivalent, it's natural that we have a weaker notion of limit available as well.
In particular, if we are only concerned with geometries up to \emph{local isomorphism} then we should only be concerned with the local isomorphism class of limit as well.
Two locally isomorphic geometries may have non-isomorphic automorphism groups in two ways: they may differ in the number of connected components and the components of one may be covers of the components of the other.
However, two locally isomorphic \emph{subgeometries} of $(G,X)$ have automorphism groups differing only in number of connected components, and so the isomorphism type connected component of the identity is an easy local-isomorphism invariant.

\begin{definition}
The \emph{connected geometric limit} of a sequence of geometries $(H_n, Y_n)$ with limit $(L,Z)$ in $\mathfrak{S}_{(G,X)}$ is the geometry $(L_0,Z)$ for $L_0$ the connected component of $\mathsf{id}\in L$.	
\end{definition}

\noindent
W have laid all the necessary ground to speak precisely about geometric limits without any examples, as the space $\mathfrak{S}_{(G,X)}\subset\Cl(G)\times\Cl(G)$ is difficult to work with directly.
Before providing our first example, we will discuss a useful computational simplification which will often allow us to exchange taking limits in $\mathfrak{C}(G)$ with taking limits in an appropriate Grassmannian.

\subsection{Computing with the Grassmannian}
\index{Limits of Geometries!Lie Algebra Limits}

\noindent
Given a vector space $V$, the Grassmannian variety $\Gr(n,V)$ is the set of all vector subspaces of $V$ of dimension $n$.
Choosing an inner product on $V$, sending each subspace to its intersection with the unit sphere identifies $\Gr(n,V)$ with the set of great $n-1$ spheres in $\S^{\dim V-1}$.
Thus the natural topology on $\Gr(n,V)$ inherited from the Chabauty space $\Cl(V)$ is equivalent to the Hausdorff topology on the set of great spheres in $\S^{\dim V-1}$.
We may realize the Grassmannians as homogeneous spaces via the automorphism-stabilizer perspective.
The group $\GL(V)$ acts transitively on the space of $n$ dimensional vector subspaces of $V$, and so $\Gr(n,V)=\GL(V)/S$ for $S$ the stabilizer of a fixed subspace.
Choosing a basis/inner product to identify $V$ with $(\R^m,\langle,\rangle)$ we note that $\O(m)$ also acts transitively on the space of $n$-dimensional subspaces, so $\Gr(n,V)\cong \O(m)/S'$ for $S'$ the stabilizer of a subspace under this action.
Taking this fixed subspace to be the span of the first $n$ coordinate vectors, we see that $S'=\O(n)\times \O(m-n)$ and realize the Grassmannian as the homogeneous space $\Gr(n,m)=\O(m)/\O(n)\times\O(m-n)$.

Our use of Grassmannians will be in thinking about the space of Lie subalgebras of a Lie algebra $\mathfrak{g}$.
As in the case of groups, we will abuse notation and use $\Cl(\mathfrak{g})$ to denote this space.

\begin{definition}
The space $\Cl(\mathfrak{g})$ is the space of Lie subalgebras of the Lie algebra $\mathfrak{g}$, topologized with respect to the Chabauty topology on the closed subsets of $\mathfrak{g}$.
\end{definition}

\begin{proposition}
$\Cl(\mathfrak{g})$ is a disjoint union of closed subsets of grassmannians.
\end{proposition}
\begin{proof}
Each Grassmannian $\Gr(n,m)$ is compact by its description as $\O(m)/\O(n)\times\O(n-m)$ above, and so any convergent path of fixed dimensional subspaces of a vector space converges to a vector subspace of the same dimension.
Also, the description of $\Gr(n,m)$ in terms of great $n-1$ spheres in $\S^{m-1}$ with the Hausdorff metric shows that subspaces of distinct dimension cannot be arbitrarily close.

Thus, the space of vector subspaces of $\R^m$ is a disjoint union of Grassmannians $\sqcup_{n=1}^m\Gr(n,m)$, and forgetting the Lie bracket embeds the space of Lie subalgebras of $\mathfrak{g}$ into the space of vector subspaces of $\mathfrak{g}$, that is $\Cl(\mathfrak{g})\subset\coprod_{n=1}^{\dim \mathfrak{g}}\Gr(n,\mathfrak{g})$.
Lie subalgebras of $\mathfrak{g}$ are closed under the Lie bracket, which is a set of polynomial conditions in each dimension.
Thus the set of $n$-dimensional Lie subalgebras of $\mathfrak{g}$ is an algebraic subvariety of $\Gr(n,\mathfrak{g})$, and so closed in the classical topology.
\end{proof}

\noindent
The Chabauty space $\Cl(\mathfrak{g})$ is actually quite reasonable to work with: if a sequence $\mathfrak{h}_n$ of Lie algebras has a limit in $\Cl(\mathfrak{g})$ then we may actually forget the bracket and consider convergence within a fixed Grassmannian - all convergent paths must have eventually constant dimension, and as the subset of Lie algebras is closed we may continue ignoring the bracket as if the underlying spaces converge so do the inherited Lie algebra structures.

\begin{corollary}
Any limit in $\Cl(\mathfrak{g})$ can be taken in the appropriate Grassmannian $\Gr(k,\mathfrak{g})$.	
\end{corollary}

\noindent
We will make use of this to study conjugacy limits of subgroups of an algebraic group $G$, via analyzing conjugacy limits of Lie algebras.

\begin{definition}
Let $G$ be a Lie group, and $H\subset G$ a Lie subgroup with Lie algebras $\mathfrak{g},\mathfrak{h}$ respectively.
If $g_n\in G$ is a sequence, the \emph{Lie Algebra limit} of $g_n h g_n\inv$ is its limit in $\Cl(\mathfrak{g})$.
We say that the Lie algebra limit of $g_n H g_n\inv$ is the group generated by the exponentiation of this $\langle \lim g_n \mathfrak{h} g_n\inv\rangle$.
\end{definition}

\noindent
When the Lie algebra limit of $g_n Hg_n\inv$ agrees with the Chabauty limit in $\Cl(G)$, this provides a powerful means of computing conjugacy limits.
Of course, this often fails, as the Lie algebra limit is connected by definition, whereas there are many examples of Chabauty limits being disconnected.
By the work of Cooper Danciger and Wienhard, the connected geometric limit of conjugates $g_n H g_n\inv$ is exactly the Lie algebra limit when $G,H$ are algebraic.

\begin{theorem}[Cooper, Danciger, Wienhard]
\label{thm:CDW_AlgGrps}
Let $G$ be an algebraic group (defined over $\C$ or $\R$). Suppose that $H$ is an algebraic subgroup and $L$ a conjugacy limit of $H$. Then $L$ is algebraic and $\dim L = \dim H$.
\end{theorem}

\begin{corollary}
The Lie algebra limit is locally isomorphic to the conjugacy limit when $H,G$ are algebraic.
\end{corollary}
\begin{proof}
Let $H<G$ be algebraic groups with Lie algebras $\mathfrak{h}$, $\mathfrak{g}$ respectively.
Let $g_n\in G$ be a sequence such that $\lim g_n Hg_n\inv=L$ in $\mathfrak{C}(G)$.
By compactness of $\Cl(\mathfrak{g})$, the path $g_n\mathfrak{h}g_n\inv$ converges to some Lie algebra $\mathfrak{a}<\mathfrak{g}$, and the Lie algebra limit $\langle \exp a\rangle$ is a subgroup of $L$ by the definition of the Chabauty topology on $\Cl(G)$.
But, by Theorem \ref{thm:CDW_AlgGrps} above, this subgroup is of the same dimension as $L$ and so is the entire connected component of the identity.
Thus the Lie algebra limit is the connected geometric limit, as claimed.
\end{proof}

\begin{corollary}
If $G$ is an algebraic group and $H<G$ an algebraic subgroup, any conjugacy limit $L=\lim A_t HA_t\inv$ is locally isomorphic to the Lie algebra limit $\mathfrak{l}=\lim A_t\mathfrak{h} A_t\inv$ taken with respect to the standard topology on $\Gr(\dim\mathfrak{h},\mathfrak{g})$.	
\end{corollary}

\noindent
A word of warning; it is crucially important that the limit is of \emph{algebraic groups} and \emph{by conjugacy} as the Lie algebra limit can be of strictly smaller dimension than the Chabauty limit in general.
Below are two examples of sequences of 1-dimensional Lie subgroup converging to a 2-dimensional group.

\begin{example}
Recall the discussion in Example \ref{ex:C_to_C} of the Chabauty space $\Cl(\C)$.
The sequence of subgroups $\frac{1}{n}\Z\times\R$ converges to $\R^2$ as $n\to\infty$.
\end{example}

\noindent
As the next example shows, this behavior can occur even when all the groups involved are all connected.
In fact, this example informs the theory greatly enough that we will name it the \emph{Barber Pole Example} for future reference.

\begin{example}[Barber Pole Example]
Consider the sequence of subgroups $H_n=\{(t/n, e^{it})\mid t\in\R\}$ of the cylinder $G=\R\times\S^1$.  The geometric limit of $H_n$ is the entire cylinder, but the Lie algebra limit is a circle, $\{(0,e^{it})\mid t\in\R \}$.
\begin{figure}[h!]
\centering
\includegraphics[width=0.85\textwidth]{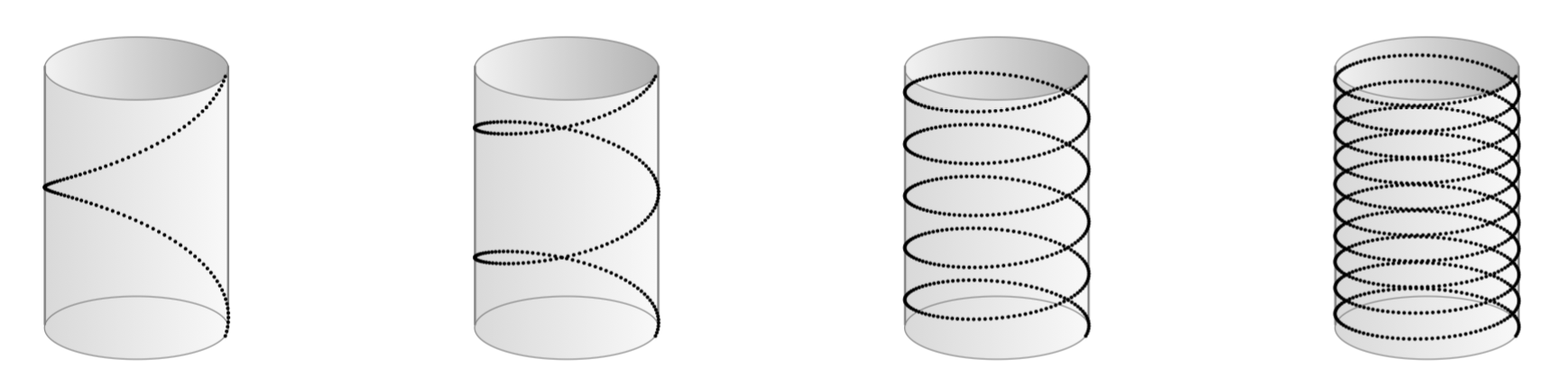}
\caption{A sequence of subgroups isomorphic to $\R$ converging geometrically to $\S^1\times\R$.  The Lie algebras converge to a horizontal line in the tangent space, and so the Lie algebra limit is a single horizontal circle.
}
\end{figure}
\end{example}

\section{The $\Hyp^2\to\E^2\leftarrow\S^2$ Transition}
\label{sec:Hyp_Sph_Transition}
\index{Limits of Geometries! Hyperbolic to Spherical}
\index{Geometric Limits! Hyperbolic to Spherical}

As a first example, we formailze the familiar transition of $\Hyp^2$ to $\S^2$ through Euclidean space in this framework.
The standard projective models of $\Hyp^2$, $\S^2$ are 
$\Hyp^2=(\SO(2,1),\P V(z^2-x^2-y^2-1))$ and $(\SO(3),\P V(z^2+x^2+y^2-1))$ 
are naturally subgeometries of $\RP^2$, so we will work in the Chabauty space $\mathfrak{S}_{\RP^2}$ of subgeometries.
The point $[0:0:1]$ lies in the domain of each geometry, and in the point stabilizers $\mathsf{stab}_{\SO(3)}[0:0:1]=\mathsf{stab}_{\SO(2,1)}[0:0:1]$ are equal, both to the block diagonal group $\smat{\SO(2)&\\&1}$.
Denoting this copy of $\SO(2)$ in $\GL(3;\R)$ by $S$ for the rest of this argument, we record these geometries in the automorphism stabilizer formalism as $\S^2=(\SO(3),S)$ and $\Hyp^2=(\SO(2,1),S)$.

For each $t\in (0,1)$, let $C_t=\diag(1,1,\sqrt{t})$, and use $C_t$ to define conjugate models of both $\S^2$ and $\Hyp^2$.
Recalling that the isomorphism type of a geometry is invariant under conjugacy, this gives a model of $\S^2$ and of $\Hyp^2$ for each $t\in(0,1)$.

\begin{definition}
For each $t\in (0,1)$, the $C_t$-conjugate of spherical geometry is $\gamma(t)=C_t.\S^2=(C_t\SO(3)C_t\inv),C_t SC_t\inv)$ and the $C_t$ conjugate of hyperbolic space is $\eta(t)=C_t.\Hyp^2=(C_t\SO(2,1)C_t\inv,C_t S C_t\inv)$.	
\end{definition}

\begin{figure}
\centering\includegraphics[width=0.5\textwidth]{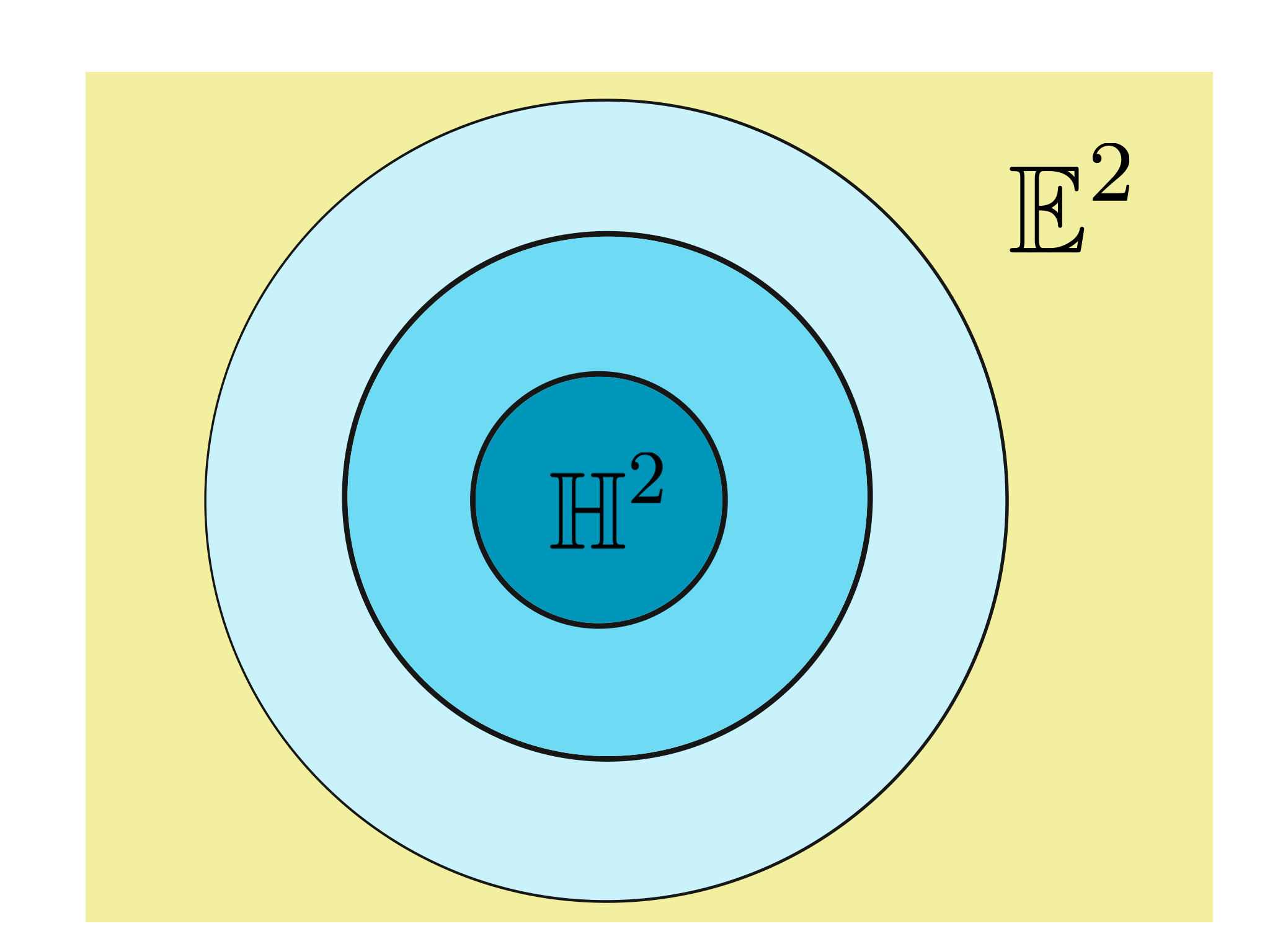}
\caption{Domains for the models $C_t.\Hyp^2$ in an affine patch of $\RP^2$.}
\end{figure}

\begin{observation}
The action of $\GL(3;\R)$ on itself by conjugation induces a continuous action on $\Cl(\GL(3;\R))$.
Thus, the paths $\gamma(t)=C_t.(\SO(2,1),S)$ and $\eta(t)=C_t.(\SO(3),S)$ are continuous functions $(0,1)\to \mathfrak{S}_{\RP^2}$. 
\end{observation}

\noindent
These two intervals of subgeometries of $\RP^2$, one of distorting models of $\Hyp^2$ and the other models of $\mathbb{S}^2$ have a common limit in the space of subgeometries, which is a model of the Euclidean plane.
We compute this limit in detail below.

\begin{proposition}
The limit $\lim_{t\to 0^+}\gamma(t)=(\Isom(\E^2),\E^2)$ is the standard model of the Euclidean plane as a subgeometry of $\RP^2$ with domain the affine patch $z=1$.	
\end{proposition}
\begin{proof}
Recall that $C_t.\S^2=C_t.(\SO(3), S)$ for $S=\smat{\SO(2)&0\\0&1}$ the stabilizer of $[0:0:1]$ under $\SO(3)$.
Because $C_t=\diag(I_2,\sqrt{t})$ is block diagonal with scalar matrices of the same size as the blocks of $S$, it is easy to see that $C_tSC_t\inv=S$ is constant under conjugacy.
Thus the limit of $\gamma(t)=C_t.\S^2$ depends only on the limit of the automorphism group $\lim_{t\to 0^+}C_t\SO(3)C_t\inv$ under conjugacy.
As $\SO(3)$ is an algebraic subgroup of the algebraic group $\GL(3;\R)$, the identity component of the geometric limit is exactly the Lie algebra limit.
As we only care about geometries up to local isomorphism, it suffices to compute $\lim_{t\to 0^+}C_t\so(3)C_t\inv$.

The Lie algebra $\so(3)$ is a 3-dimensional subspace of $\gl(3;\R)\cong\R^9$ given by
$\so(3)=\left\{
\smat{0&x&y\\-x&0&z\\-y&-z&0}
\right\},$
where $x,y,z$ range over $\R$.
The conjugate Lie algebra $C_t\so(3)C_t\inv$ is then the following element of $\Gr(3,9)$.
$$\so(Q_t)=C_t\so(3)C_t\inv=
\R\pmat{0&1&0\\-1&0&0\\0&0&0}
\oplus
\R\pmat{0&0&1\\0&0&0\\-t&0&0}
\oplus
\R\pmat{0&0&0\\0&0&1\\0&-t&0}
$$

\noindent
As $t\to 0$ this path of points converges in $\Gr(3,9)$ to the lie algebra spanned by $\smat{0&1&0\\-1&0&0\\0&0&0}$, $\smat{0&0&1\\0&0&0\\0&0&0}$, and $\smat{0&0&0\\0&0&1\\0&0&0}$, which is the Lie algebra of the Euclidean group $\euc(2)=\smat{0&x&y\\-x&0&z\\0&0&0}$, exponentiating to $\Euc(2)=\smat{\SO(2)&\R^2\\0&1}$.
Together with the limiting point stabilizer $\smat{\SO(2)&0\\0&1}$ this is the automorphism-stabilizer description of the familiar projective model of Euclidean space, acting on the affine patch $z=1$ in $\RP^2$.
\end{proof}

\begin{proposition}
The limit $\lim_{t\to 0^+}\eta(t)=(\Isom(\E^2),\E^2)$ is the \emph{same} standard model of the Euclidean plane as a subgeometry of $\RP^2$ with domain the affine patch $z=1$.	
\end{proposition}
\begin{proof}
The point stabilizers are again constantly equal to $S=\smat{\SO(2)&0\\0&1}$ so the computation reduces to the limit $\lim_{t\to 0^-}C_t\SO(2,1)C_t\inv$ which may likewise be computed via the Lie algebra.
In this case, the conjugate Lie algebras are

$$C_t\so(2,1)C_t\inv=
\R\pmat{0&1&0\\-1&0&0\\0&0&0}
\oplus
\R\pmat{0&0&1\\0&0&0\\t&0&0}
\oplus
\R\pmat{0&0&0\\0&0&1\\0&t&0},
$$
which differ from the spherical case only in the lack of minus signs attached to the $t$'s in the second two basis vectors.
As $t\to 0$ the limit is identical to the above, $\euc(2)=\smat{0&x&y\\-x&0&z\\0&0&0}$, which exponentiates to the usual representation of the Euclidean group.
\end{proof}

\noindent
The two paths $\gamma$ and $\eta$ have a common limit as $t\to 0$, and we may use this to define a single continuous path of geometries.

\begin{corollary}
The map $f\colon[-1,1]\to\mathfrak{S}_{\RP^2}$ below is continuous providing a transition from $f(1)=(\SO(3),\RP^2)$ to $f(-1)=(\SO(2,1),\Hyp^2))$.

$$f(t)=
\begin{cases}
\gamma(t) & t>0\\
(\Euc(2),\{[x:y:1]\}) & t=0\\
\eta(-t) & t<0
\end{cases}
$$
\end{corollary}

\noindent
The behavior of the domains of these geometries throughout the transition may seem mysterious at first, as on one side $C_t.\Hyp^2$ is a sequence of disks in $\RP^2$ converging on an affine patch, but on the other $C_t.\S^2$ is independent of $t$ and equal to the entire projective space.
The transition of domains is easier to visualize directly in the double cover before projectivization, identifying $\S^2$ with the unit sphere in $\R^3$ and $\Hyp^2$ with the unit hyperboloid of two sheets.
Then the models $C_t.\S^2$ and $C_t.\Hyp^2$ are their images under the linear action of $C_t$.
As $t\to 0$, the sphere $C_t.\S^2$ flattens out like a pancake, converging to the union of two affine planes $z=\pm 1$, which are the simultaneous limit of the two sheets of the hyperboloids $C_t.\Hyp^2$ as the flatten out.

\begin{figure}
\centering\includegraphics[width=0.65\textwidth]{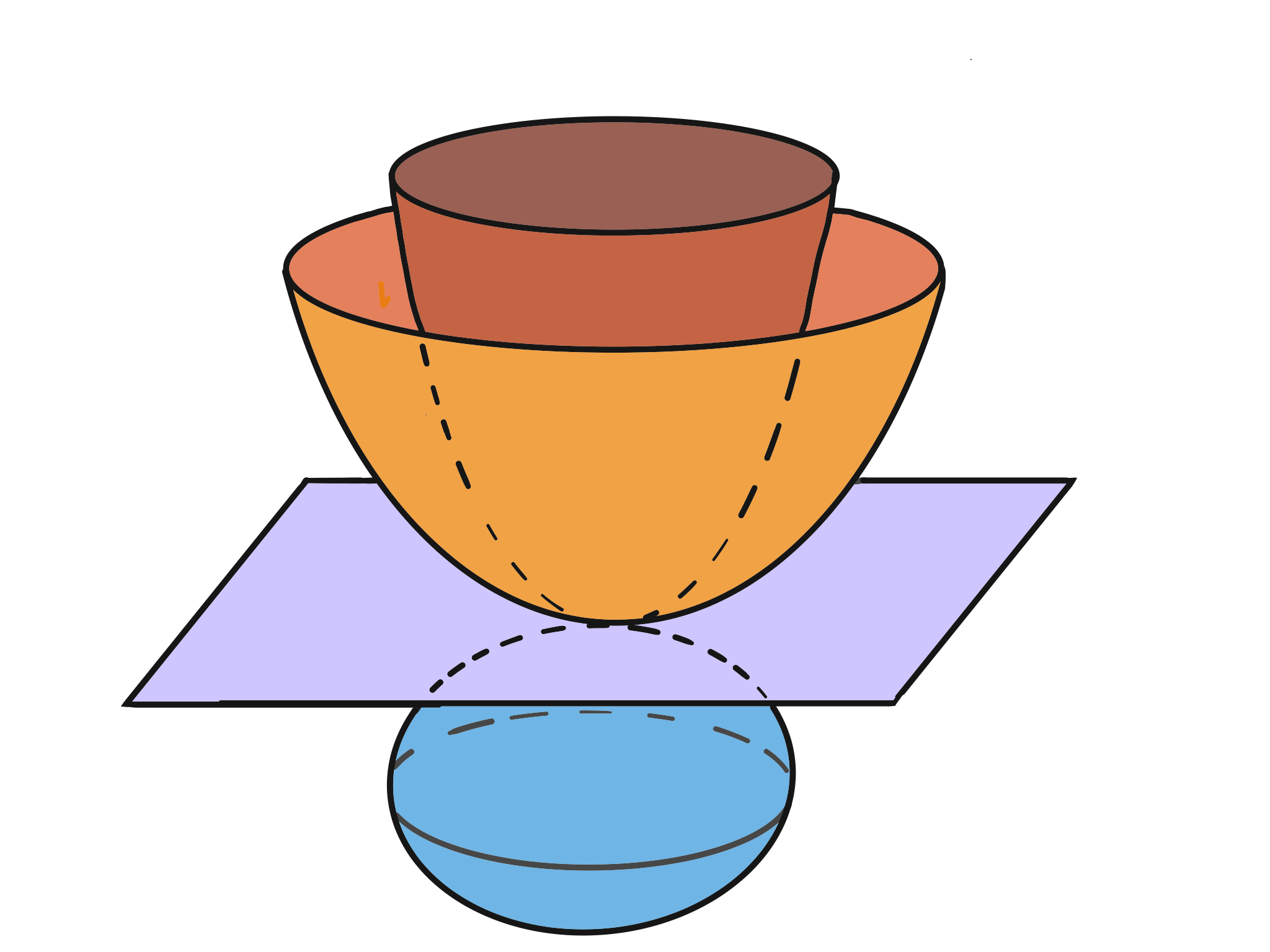}
\caption{The surfaces $C_t.\S^2$ and $C_t.\Hyp^2$ in $\R^3$.}
\end{figure}

\section{Limits of Orthogonal Subgeometries}
\label{sec:Lims_Orthog_Subgeos}
\label{Limits of Geometries!Cooper-Long-Thistlethwaite}

Beyond the classically - understood degeneration of $\Hyp^n$ to $\E^n$, the next well studied conjugacy limit of hyperbolic space was discovered by Jeff Danciger during his PhD work at Stanford \cite{Danciger11,Danciger11Ideal,Danciger13}.
Whereas a Euclidean limit is reached by uniformly stretching the ball model of $\Hyp^n$ in the affine patch $\R^n\subset\RP^n$ in all directions, Danciger considered conjugacy limits which stretch $\Hyp^n$ only in one direction, fixing a codimension-1  copy of $\Hyp^{n-1}$ in $\Hyp^n$.
The limiting geometry under this sequence of conjugacies has domain a cylinder $\mathbb{B}^{n-1}\times\R$, and is variously called Half Pipe, or co-Minkowski geometry in the literature
\footnote{The name Half-Pipe comes from the hyperboloid model of the limiting geometry in dimension two \cite{Danciger11}.  
The term co-Minkowski arises as the automorphism group is the contragredient representation of the automorphisms of Minkowski spacetime \cite{FillastreS16}.}

\begin{figure}
\centering\includegraphics[width=0.85\textwidth]{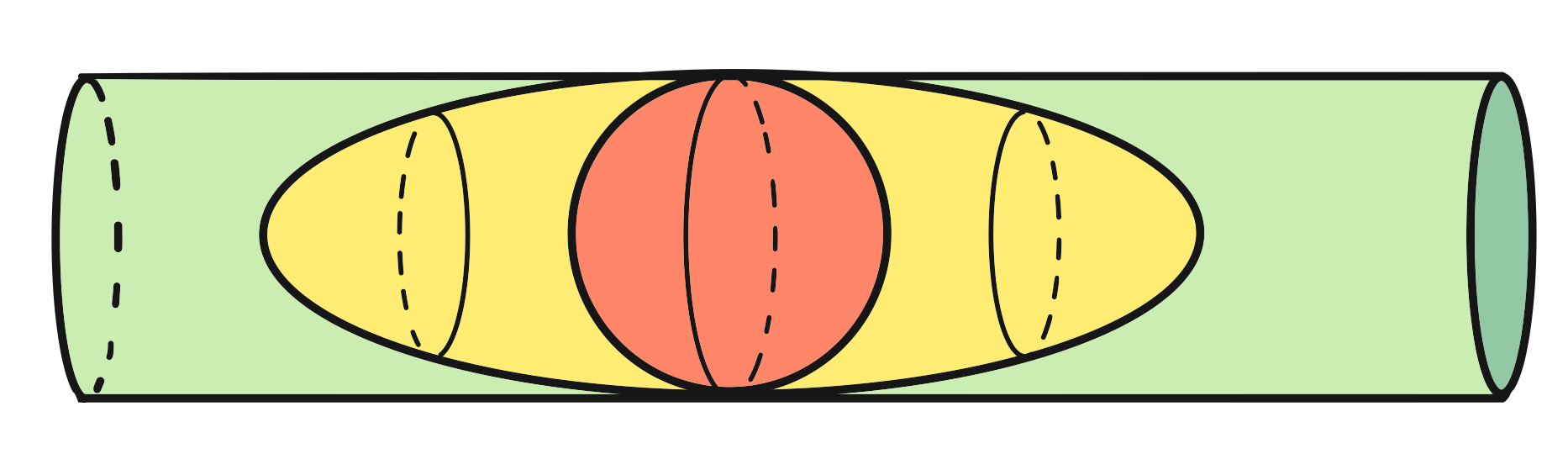}
\caption{The degeneration of $\Hyp^3$ to Half-Pipe, or co-Minkowski geometry via conjugacy limit in $\RP^3$.}	
\end{figure}

\noindent
This conjugacy limit appears as part of a new \emph{geometric transition}, connecting $\Hyp^n$ to its Lorentzian analog, Anti-de Sitter space $\mathsf{AdS}^n$, much as $\E^n$ interpolates between $\Hyp^n$ and $\S^n$.
Danciger has used this transition to study the collapse of singular hyperbolic, as well as Anti-de Sitter structures, as well as to answer questions in classical geometry \cite{DancigerMS14}.

From this stems multiple possible generalizations: what about stretching some other number of directions in $\Hyp^n$ to produce a limit?
What about stretching in multiple different directions \emph{and} at multiple rates?
What about other geometries, such as Anti-de Sitter and its pseudo-Riemannian relatives, besides hyperbolic space?
All of these potential generalizations were taken on, and completed by the aforementioned joint work of Cooper, Danciger and Wienhard, \emph{Limits of Geometries} \cite{CooperDW14}.
Below we review the main results of this work as a prelude to Chapter \ref{chp:Orthogonal_Groups}.

Hyperbolic and spherical geometry, along with their Lorentzian analogs de Sitter and Anti-de Sitter space, are special cases of \emph{orthogonal geometries}, or \emph{geometries of quadratic forms}.

\begin{definition}
Let $\beta$ be a nondegenerate quadratic form on $\R^n$, and $\Isom(\beta)<\GL(n;\R)$ the group of linear transformations preserving $\beta$ in the sense that $\beta(x,y)=\beta(Ax,Ay)$ when $A\in\Isom(\beta)$.
Let $\mathsf{X}(\beta)\subset\RP^{n-1}$ be the projectivized negative cone for $\beta$; $\mathsf{X}(\beta)=\{[x]\in\RP^{n-1}\mid \beta(x)<0\}$.
Then $(\mathsf{P}\Isom(\beta),\mathsf{X}(\beta))$ is a Group-Space subgeometry of projective space.	
\end{definition}

\begin{remark}
When $\beta$ is of signature $(p,q)$, meaning $\beta$ is similar to $-I_p\oplus I_q$, the group $\mathsf{P}\Isom(\beta)$ is conjugate to $\PO(p,q)$ and $(\mathsf{P}\Isom(\beta),\mathsf{X}(\beta))$ is a projective model for a semi-Riemannian geometry of constant curvature of dimension $p+q-1$	and signature $(p-1,q)$.
In the cases $(p,q)=(n,0),(1,n-1),(n-1,1),(2,n-2)$ we obtain spherical, hyperbolic, de Sitter and Anti-de Sitter space respectively.
When the particular choice of $\beta$ is irrelevant to present discussion, we will use the notation $\mathsf{X}(p,q)$ to denote the semi-Riemannian geometry arising from a signature $(p,q)$ form.
\end{remark}

In \emph{Limits of Geometries}, Cooper, Danciger and Wienhard manage to classify \emph{all} conjugacy limits of the geometries of quadratic forms, as subgeometries of $\RP^n$.
In general it is quite difficult to compute the totality of conjugacy limits of $H$ in $G$, as one has no control over which possible paths $C_t\in C^\infty(\R_+,G)$ give distinct limits $C_t HC_t\inv$.
This difficulty is averted for the study of orthogonal groups in $\GL(n;\R)$ via the following result of \cite{CooperDW14} regarding limits of symmetric subgroups of semisimple Lie groups.

\begin{theorem}[Theorem 1.1 in \emph{Limits of Geometries}]
Let $H$ be a symmetric subgroup of a semisimple Lie group $G$ with finite center.
Then any limit of $H$ in $G$ is the limit under conjugacy of a one parameter subgroup.
More precisely, let $L'$ be a conjugacy limit of $H$.
Then there is an $X\in\mathfrak{g}$ such that $L'$ is conjugate to $L=\lim_{t\to\infty} \exp(t X)H\exp(-t X)$.
\end{theorem}

\noindent
Thus, the space one must search for conjugacy limits can be reduced from the infinite dimensional space $C^\infty(\R_,G)$ to the one parameter subgroups, which is parameterized by the unit sphere in $\mathfrak{g}$ via $[X]_+\mapsto \{\exp t X\}_{t\in\R}$.
This already reduces the problem for conjguacy limits of  $\O(p,q)<\GL(n;R)$ to understanding the map $\S^{n^2-1}\to \Cl(\GL(n;\R))$ given by $[X]\mapsto \lim_{t\to\infty} \exp(t X)\O(p,q)\exp(-tX)$, but further reduction is still possible.
Indeed, via various matrix factorization theorems, we have the following.

\begin{observation}
Every conjugacy limit of $\O(p,q)$ in $\GL(p+q;\R)$ is conjugate to a conjugacy limit $\lim_{t\to\infty} D_t \O(p,q)D_t\inv$ for $D_t$ diagonal matrices.
Furthermore, the path $D_t$ can be taken to be a one parameter subgroup.
\end{observation}

\noindent
This further reduces the search space, and to classify all conjugacy limits one must only understand the map $\S^{n-1}\to \Cl(\GL(n;\R))$ taking a point $\vec{v}\in\S^{n-1}$ to the conjugacy limit $\lim e^{t\vec{v}}\O(p,q) e^{-t\vec{v}}$ for $e^{\vec{w}}$ the diagonal matrix with entries $e^{w_i}$.
As all limits under consideration are conjugacy limits of algebraic subgroups of an algebraic group, it is admissable to compute using the Lie algebra limit.

\begin{corollary}
All connected limits of the orthogonal group $\O(p,q)$ in $\GL(p+q;\R)$ are conjugate to the exponential of $\lim_{t\to\infty}e^{t\vec{v}}\so(p,q)e^{-t\vec{v}}$ for some $\vec{v}\in\S^{n-1}$.	
\end{corollary}

\noindent
In the resulting analysis, Cooper, Danciger and Wienhard describe these limits as the geometries of \emph{partial flags of quadratic forms}.
Their definition, description, and the resulting classification are below.

\begin{definition}
A \emph{partial flag} $\fam{F}=\{V_0,V_1,\ldots, V_k,V_{k+1}\}$ of $\R^n$ is a descending chain of vector subspaces $\R^n=V_0\supset V_1\cdots V_k\supset V_{k+1}=\{0\}$.
A \emph{partial flag of quadratic forms}
$\beta=(\beta_0,\beta_1,\ldots, \beta_k)$ on $\fam{F}$ is a collection of nondegenerate quadratic forms $\beta_i$, defined on each quotient $V_i/V_{i+1}$ of the partial flag, respectively.
The group $\Isom(\beta,\fam{F})$ contains all linear transformations of $\R^n$ which preserve $\fam{F}$ and induce isometries of $\beta_i$ on each of the respective quotients.
\end{definition}

\begin{definition}
The $(G,X)$ geometry associated to a partial flag of quadratic forms $(\beta,\fam{F})$ has domain $\mathsf{X}(\beta)\subset\RP^{n-1}$ defined by $\mathsf{X}(\beta)=\{[x]\in\RP^{n-1}\mid \beta_0(x)<0\}$, and automorphism group $\mathsf{P}\Isom(\beta,\fam{F})$.
\end{definition}

\begin{observation}
For any partial flag of quadratic forms $(\beta,\fam{F})$, the group $\Isom(\beta,\fam{F})$ is conjugate to the group of matrices of the form, below, where $\star$ denotes an arbitrary block.
$$\pmat{ \O(p_0,q_0) &0&0 &0\\
\star &\O(p_1,q_1) &0 &0 \\
\star &\star &\ddots &0\\
\star&\star &\star &\O(p_k,q_k)}
$$
\end{observation}

\begin{theorem}[Theorem 1.2 in \emph{Limits of Geometries}]
The limits of the constant curvature semi-Riemannian geometries	$(\PO(p,q),\mathsf{X}(p,q))$ in $\RP^{p+q-1}$ are all of the form $(\mathsf{P}\Isom(\beta,\fam{F}),\mathsf{X}(\beta,\fam{F}))$ for $(\beta,\fam{F})$ a partial flag of quadratic forms on $\R^{p+q}$.
Further, $\mathsf{X}(\beta)$ is a limit of $\mathsf{X}(p,q)$ if and only if $p_0\neq 0$ and the signatures $((p_0,q_0),(p_1,q_1),\ldots (p_k,q_k))$ of $\beta$ partition the signature $(p,q)$ in the sense that 
$$p_0+p_1+\ldots +p_k=p
\hspace{1cm}
q_0+q_1+\ldots q_k=q
$$
after exchanging $(p_i,q_i)$ with $(q_i,p_i)$ for some collection of indicies $i\in\{1,\ldots, k\}$ (the first signature $(p_0,q_0)$ cannot be reversed as it determines the domain $\mathsf{X}(\beta,\fam{F})$.
\label{thm:Main_Limit_Thm_CDW}
\end{theorem}

\begin{figure}
\label{fig:Limits_of_H3}
\centering
\begin{tikzcd}
& \mathsf{X}(1,3)\arrow[dl]\arrow[d]\arrow[dr] &\\
\mathsf{X}((1,0)(3))\arrow[d]\arrow[dr] & \mathsf{X}((1,2)(1))\arrow[d]\arrow[dr] & \mathsf{X}((1,1)(2))\arrow[d]\arrow[dll]\\
\mathsf{X}((1,0)(1)(2))\arrow[dr] & \mathsf{X}((1,0)(2)(1))\arrow[d] &\mathsf{X}((1,1)(1)(1))\arrow[dl]\\
&\mathsf{X}(1,0)(1)(1)(1) &	
\end{tikzcd}
\label{fig:Hyp3_Limits}
\caption{The limits of $\Hyp^3=\mathsf{X}(1,3)$ as a subgeometry of $\RP^3$.}	
\end{figure}

\noindent
In Figure \ref{fig:Limits_of_H3}, the limits of $\Hyp^3=\mathsf{X}(1,3)$ appear to form a poset, which is intuitively plausible: if $L$ is a limit of $H$ and $K$ is a limit of $L$, then $K$ should be achievalbe as a limit of $H$ as well.
That this is in fact the case is another theorem of \cite{CooperDW14}, reproduced below.

\begin{theorem}[Theorem 3.3 in \emph{Limits of Geometries}]
Let $G$ be an algebraic Lie group.
Then the relation of being a connected conjugacy limit induces a partial order on the set $\mathsf{Grp}_0(G)$ of all connected algebraic subgroups of $G$.
Moreover the length of every chain is at most $\dim G$.
\end{theorem}

\noindent
With the classification of limits of the semi-Riemannian geometries $\mathsf{X}(p,q)$ above, we notice the following.

\begin{corollary}
Each semi-Riemannian geometry $\mathsf{X}(p,q)$ has the geometry $\mathsf{X}((1,0)(1)\cdots (1))$ as a common, 'most degenerate' limit.
\end{corollary}

\noindent
The autmorphisms of this geometry are the unipotent group of upper triangular $n\times n$ matrices acting projectively on the affine patch $x_{n}=1$ in $\RP^{n-1}$.
When $n=3$, this geometry is given by the action of the Heisenberg group on the affine plane, and is studied extensively in Chapter \ref{chp:Heisenberg_Plane}.

\begin{definition}
Heisenberg geometry of dimension $n$ is given by the projective action of the upper triangular unipotent group of matrices in $\M(n+1;\R)$ on the affine patch $\A^n=\{x_{n+1}=1\}$.
\end{definition}

In Chapter \ref{chp:Heisenberg_Plane}, we study the two dimensional version of this geometry in detail.

\chapter{Orthogonal Groups in $\GL(n;\R)$}
\label{chp:Orthogonal_Groups}
\index{Limits of Geometries!Orthogonal Groups}
\index{Orthogonal Geometries}

The main difficulty in computing all degenerations the geometries of quadratic forms $(\O(p,q),X(p,q))$ as subgeometries of projective space is the computation of conjugacy limits of their automorphism groups $\O(p,q)$.
Topologically, these degenerations are the limit points of the space $\fam{O}_n$ of orthogonal groups in $\GL(n;\R)$.

\begin{definition}
A group $G<\GL(n;\R)$ is an \emph{orthogonal group} if there is a nondegenerate quadratic form $q$ on $\R^n$ such that $g^\ast q=q$ for all $g\in G$.
The set of all orthogonal groups in $\GL(n;\R)$ is $\fam{O}_n\subset\GL(n;\R)$.
\end{definition}

\noindent
In section \ref{sec:Hyp_Sph_Transition} of the previous chapter,  we explicitly showed that $\Euc(2)$ is a conjugacy limit of both $\SO(3)$ and $\SO(2,1)$ in $\GL(3;\R)$, and in section 
\ref{sec:Lims_Orthog_Subgeos}, reviewed the classification of all conjugacy limits up to isomorphism by Cooper, Danciger and Wienhard.
Their result can be reprhased geometrically as below.

\begin{theorem}[Limits of Geometries]
Every point in the closure $\overline{\fam{O}_n}\subset\Cl(\GL(n;\R))$ is the isometry group of some partial flag of quadratic forms.
\end{theorem}

\noindent
Here we refine this result and study the full Chabauty compactification $\overline{\fam{O}_n}$, through an argument independent of the methods of \cite{CooperDW14}.
The motivation for exhibiting this is twofold: this recovers more information about the space of degenerations than simply listing the isomorphism type of boundary points, and second, the ideas here likely have further applications and this provides a well-studied testing ground to exhibit them.

\begin{definition}
$\fam{D}_n\subset\Cl(\GL(n;\R))$ is the subset of $\fam{O}_n$ containing the orthogonal groups $\O(J)$ for $J=\diag(\lambda_1,\ldots,\lambda_n)$ a diagonal quadratic form.
\end{definition}

\noindent
We show in \ref{sec:Space_O} that the full closure $\overline{\fam{O}_n}$ can be recovered from knowledge of $\fam{\overline{D}_n}$, which can be described combinatorially.

\begin{theorem}
\label{thm:Orthog_Closure_Main}
$\overline{\fam{D}_n}$ is homeomorphic to the maximal de Concini Procesi blowup of the coordinate hyperplane arrangement in $\RP^{n-1}$, equipped with a natural cellulation by $2^{n-1}$ permutohedra.
Any two groups in the same facet of this cellulation are conjugate, and the codimension of the cell gives the length of the partial flag of quadratic forms associated to the limit group.
\end{theorem}

\noindent
The main advantage of this argument is that it does not rely on a priori finding a 'nice' collection of paths and proving that every conjugacy is achieved (up to isomorphism) along one of these.
Thus these techniques can be employed even in cases where not all limits are achieved along 1-parameter subgroups, or no other suitable collection of paths is known.

\section{The Space of Orthogonal Groups}
\label{sec:Space_O}
\index{Orthogonal Groups!Space of}

\noindent
A group $G<\GL(n;\R)$ is an \emph{orthogonal group} if it is the isometries of some nondegenerate quadratic form on $\R^n$.
Choosing a basis for $\R^n$ identifies these quadaratic forms with nondegenerate symmetric matrices $\Sym^\times(n;\R)=\{A\in\GL(n;\R)\mid A^T=A\}$, as $A$ determines the map $x\mapsto x^TAx$.
We use this here to identify $\fam{O}_n$ with projective classes of nondegenerate symmetric matrices, and give insight into the topology of $\fam{O}_n\subset\Cl(\GL(n;\R))$.

\begin{observation}
The map $\phi\colon\Sym^\times(n;\R)\to \fam{O}_n$ sending a symmetric matrix $J\mapsto \O(J)$ to its orthogonal group is surjective, by definition.
\end{observation}

\begin{lemma}
The	map $\phi\colon J\mapsto \O(J)$ above is continuous into the Chabauty space.
\end{lemma}
\begin{proof}
Let $J\in\Sym^\times(n;\R)$, we show that $\phi$ is continuous at $J$.
As a nondegenerate real symmetric matrix, $J$ has nonzero eigenvalues, and there is a sufficiently small euclidean ball $B\subset \Sym^\times(n;\R)$ such that $J\in B$ and all eigenvalues of $J'\in B$ are of the same sign as those of $J$.
Then in fact all matrices in $B$ are similar to $J$; there is an open neighborhood $U$ of the identity in $\GL(n;\R)$ such that $B=U.J=\{A^TJA\mid A\in U\}$.
As $\O(M^TJM)=M\inv \O(J)M$ and the conjugation action of $\GL(n;\R)$ on $\Cl(\GL(n;\R))$ is continuous, the map $U\to \Cl(\GL(n;\R))$ given by $M\mapsto M\inv\O(J)M$ is continuous.
This descends through the orbit map $\pi\colon U\to B$ to a continuous map $B\to \Cl(\GL(n;\R))$, which is $\phi|_B$ by definition.
Thus $\phi$ is continuous at $J$.
\end{proof}

\noindent
The map $\phi$ is not injective, as $\O(J)=\O(\lambda J)$ for $\lambda\neq 0$, for instance.
However, this is the only obstruction; if $\O(K)=\O(J)$ then $K=\lambda J$ for some $\lambda\in\R^\times$.	

\begin{corollary}
The continuous map $\phi\colon \Sym^\times\to \Cl(\GL(n;\R))$ factors through projectivization to a continuous bijection $\iota\colon\PSym^\times\to\fam{O}_n$, and we implicitly identify $\PSym^\times$ and $\fam{O}_n$ via this map.
\end{corollary}

\begin{example}
The subspace of $2\times 2$ symmetric matrices is three dimensional, and $\det\inv\{0\}\subset\Sym(2,\R)$ is the quadratic cone $x^2+y^2=z^2$ in the coordinates $\smat{z-x&y\\y&z+x}$.
Thus $\PSym^\times(2;\R)$ is the complement of a separating circle in $\RP^2$.
\end{example}

\noindent
In general, 
$\fam{O}_n$ is disconnected, and is a disjoint union of $\lceil (n+1)/2 \rceil$ components, one for each unordered partition $\{p,q\}$ such that $p+q=n$.
Each component $\fam{O}_{p,q}$ corresponds to orthogonal groups of signature $(p,q)$, and is homeomorphic to the coset space of $\SO(p,q)$ in $\SL(n;\R)$.

\begin{example}
$\fam{O}_2=\SL(2;\R)/\SO(2)\sqcup \SL(2;\R)/\SO(1,1)$ is the union of a disk and a M\"obius band.	We can see this directly from the fact that $\fam{O}_2\cong\PSym^\times(2;\R)\cong\RP^2\smallsetminus V(x^2+y^2=z^2)$
\end{example}

\noindent
At this point it may appear that the natural move is to restrict individually to each component $\mathcal{O}_{p,q}$ and study their Chabauty compactifications separately.
However, from our computation in Section \ref{sec:Hyp_Sph_Transition} of $\Euc(2)$ as a common conjugacy limit of both $\SO(3)$ and $\SO(2,1)$ in $\SL(3;\R)$, we see that the closures are not necessarily disjoint.
In fact, only slightly modifying the argument of Section \ref{sec:Hyp_Sph_Transition}, we can produce a transition between $\O(p,q)$ and $\O(p',q')$ for any $p+q=p'+q'$.
Thus there is a  compelling reason to study the entire collection $\fam{O}_n$ and its closure together.

\begin{observation}
The closure $\overline{\fam{O}}$ is connected.
Even stronger, the boundaries $\partial \fam{O}_{p,q}$ and $\partial\fam{O}_{p',q'}$ of any two components have nontrivial intersection.	
\end{observation}

\noindent
Instead of restricting to each signature component individually, it turns out that a rather efficient route to recovering the result of Theorem \ref{thm:Main_Limit_Thm_CDW} is to consider the subcollection of \emph{diagonal orthogonal groups}.
An orthogonal group $\O(J)$ is said to be \emph{diagonal} if it is the isometries of a diagonal quadratic form $J=\diag(\lambda_1,\ldots, \lambda_n)$.
The collection of diagonal orthogonal groups is denoted $\mathcal{D}_n$.

\begin{definition}
$\mathcal{D}_n\subset\mathcal{O}_n$ is the subcollection of isometry groups of nondegenerate diagonal quadratic forms.
$\mathcal{D}_n=\{\O(J)\mid J=\diag(\lambda_1,\ldots, \lambda_n),\lambda_i\in\R^\times\}$.	
\end{definition}

\noindent
The diagonal orthogonal groups are a useful subset of $\fam{O}_n$, as every symmetric matrix over $\R$ can be orthogonally diagonalized.
In fact, to classify the possible conjugacy limits of orthogonal groups in $\GL(n;\R)$ it suffices to understand the closure of $\fam{D}_n$:
if $G\in\overline{\fam{O}_n}$, then there is some $Q\in\O(n)$ such that $QDQ\inv\in\overline{\fam{D}_n}$.	
Rephrased geometrically, the action of $\O(n)$ on $\GL(n;\R)$ by conjugation induces a continuous $\O(n)$ action on $\Cl(\GL(n;\R))$, and the above observation is equivalent to the proposition below.

\begin{proposition}
$\overline{\fam{O}}=\O(n).\overline{\fam{D}}$ in $\Cl(\GL(n;\R))$
\end{proposition}
\begin{proof}
Let $H\in\partial\fam{O}_n=\overline{\fam{O}}_n\smallsetminus\fam{O}_n$.
Then $H=\lim H_k$ for some sequence $H_k\subset \fam{O}_n$,	but each $H_k\in\fam{O}_n$ is conjugate to some $D_k\in\fam{D}_n$ by some element $Q_k\in\O(n)$; that is $H_k=Q_k D_kQ_k\inv$.
As $\O(n)$ is compact, the sequence $Q_k$ subconverges $Q_k\to Q\in\O(n)$, and so 
$H=\lim H_k=\lim Q_k D_kQ_k\inv=Q(\lim D_k)Q\inv$.
Thus $D_k\to D$ converges and to a limiting group, conjugate to $H$ by $Q$.
Said another way, the arbitrary limit point $H$ lies in the same $\O(n)$ orbit as some group $D\in\overline{\fam{D}_n}$, completing the proof.
\end{proof}

\begin{observation}
The space $\fam{D}_n\cong\PDiag^\times(n;\R)$ is the projectivization of the space $\R^n$ of diagonal matrices, less those with determinant zero, corresponding to the union of the coordinate hyperplanes.
That is, $\fam{D}\cong\RP^{n-1}\smallsetminus \fam{A}$ is the projectivized complement of the coordinate hyperplane arrangement $\fam{A}$.
Any two orthogonal groups in the same connected component of $\fam{D}_n$ are conjugate, and in fact the connected components are conjugacy classes by \emph{diagonal conjugacy}.
\end{observation}

\begin{example}
For $n=2$, $\fam{O}_2$ is $\RP^2$ less a circle, and $\fam{D}_2$ is a twice punctured projective line (in the double cover $\fam{O}_2$ is a sphere minus the north and south arctic circles, and $\fam{D}_2$ is a great circle of longitude).
The action of $\O(2)$ by conjugation fixes a single point and is free on the complement of this point (in the double cover, this action is by rotation along the polar axis of $\S^2$)
\end{example}

\begin{figure}
\centering\includegraphics[width=0.6\textwidth]{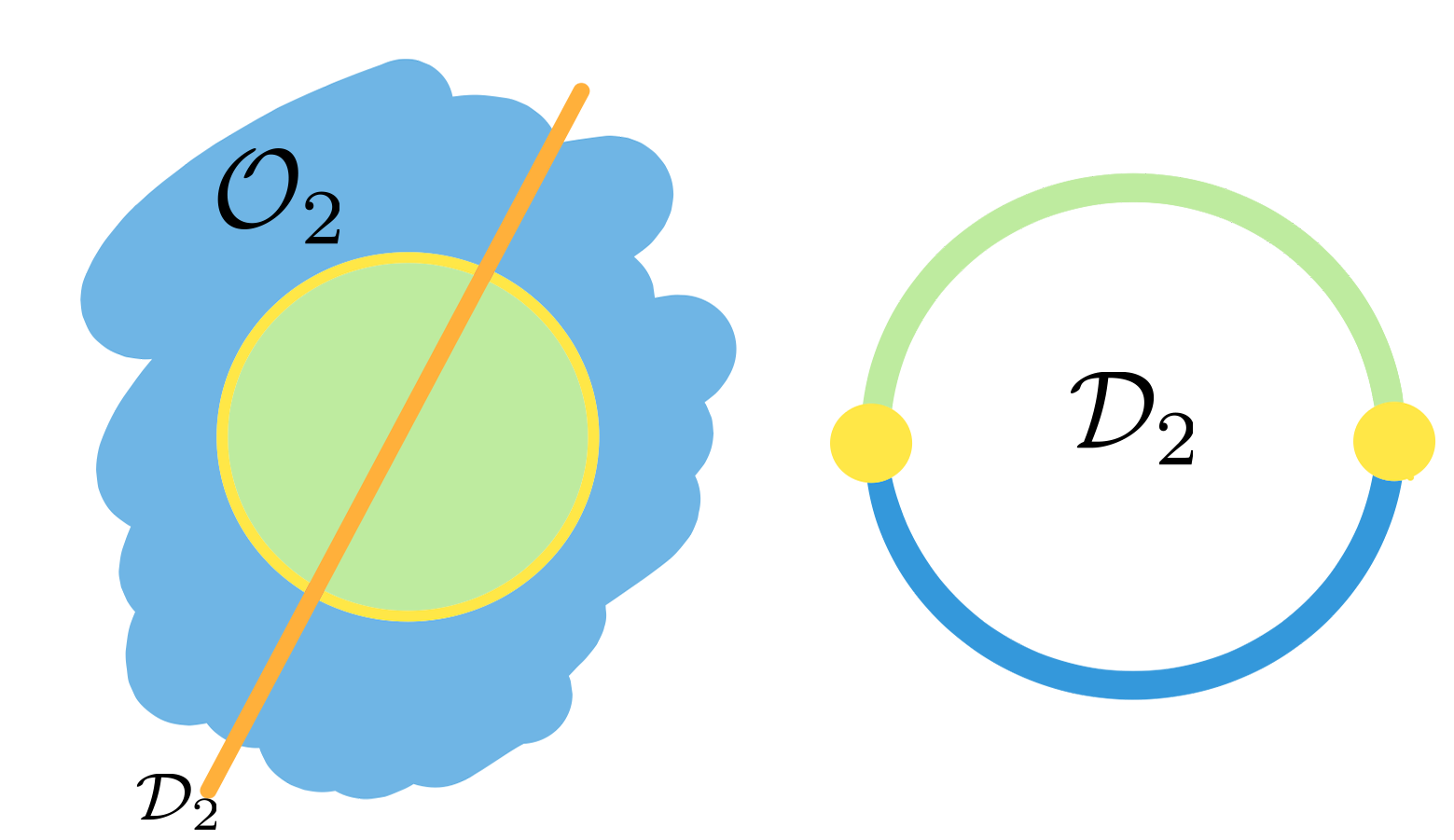}
\caption{The space $\fam{O}_2$ and the slice $\fam{D}_2$.}
\end{figure}

\begin{example}
For $n=3$, $\fam{O}_3$ is an open $5$-manifold and $\fam{D}_3$ is the complement of the coordinate hyperplanes in $\RP^2$.
The action of $\O(3)$ on $\fam{O}_3$ fixes the point representing $\O(3)$, and generic orbits $\O(3).\O(J)$ pass through $\fam{D}_3$ three times, corresponding to the three permutations of the diagonal entries of $J=\diag(x,y,z)$.
\end{example}

\begin{figure}
\centering\includegraphics[width=0.65\textwidth]{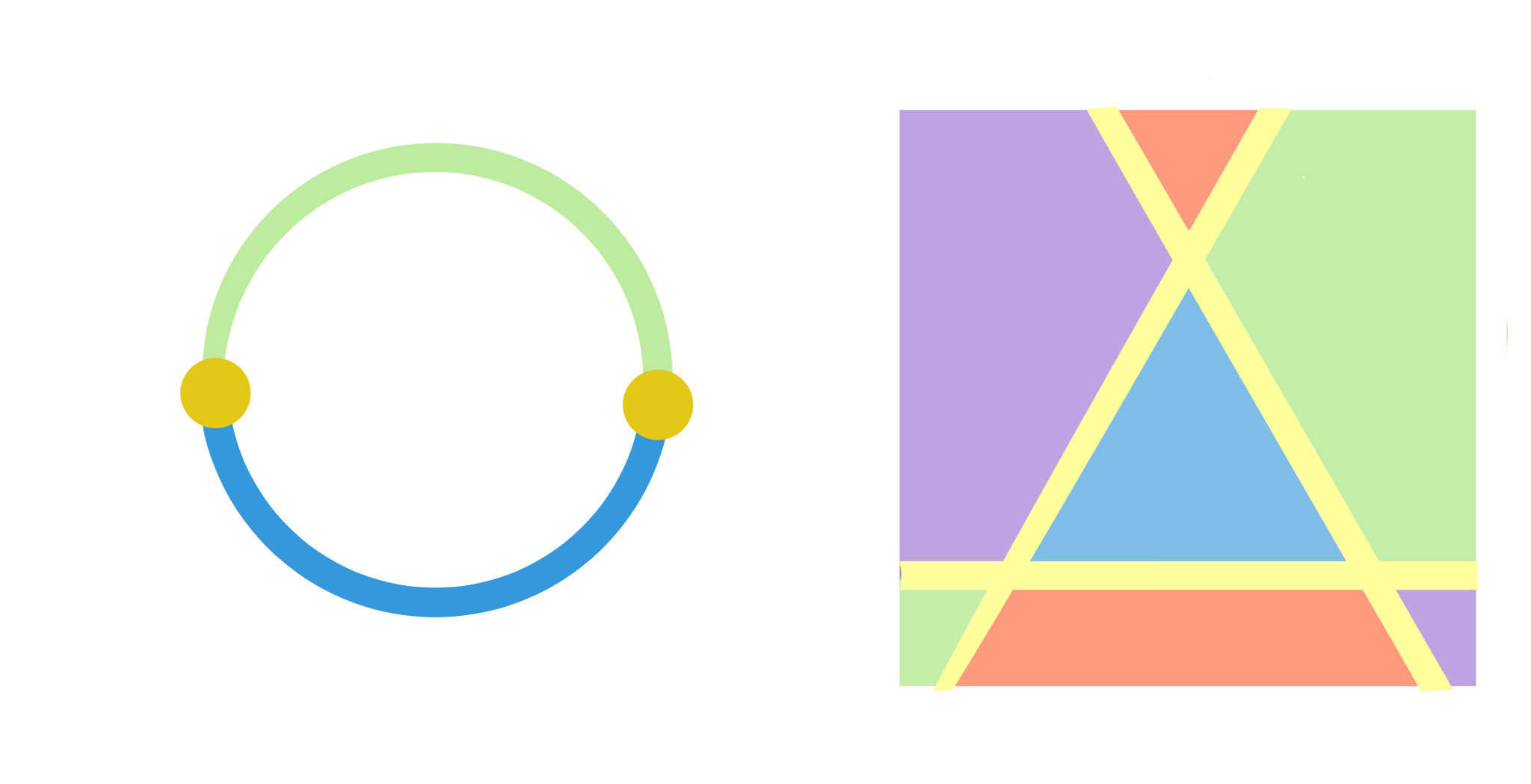}
\caption{The slices $\fam{D}_2\cong\PDiag^\times(2;\R)$ and  $\fam{D}_3\cong\PDiag^\times(3;\R)$.	}
\end{figure}

\section{Simplifying the Problem}
\label{sec:Simplifying_Orthogonal_Problem}

The remainder of this chapter is aimed at computing the closure $\overline{\fam{D}_n}$, proving Theorem \ref{thm:Orthog_Closure_Main}.
To do so, we proceed by a sequence of simplifications, aimed at reducing the complexity of the codomain of the embedding $\iota\colon\PSym^\times(n;\R)\to\Cl(\GL(n;\R))$.
We begin by replacing the hyperspace $\Cl(\GL(n;\R))$ with the space of closed Lie subalgebras of $\gl(n;\R)$.
We then carefully consider the image of $\fam{D}_n$ in $\Gr(\smat{n\\2},n^2)$ and show it lies in an $\smat{n\\2}$-dimensional torus.
Studying this embedding allows us to compute the closure $\overline{\fam{D}_n}$ using algebro-geometric techniques.

\subsection{From $\Cl(\GL)$ to $\Cl(\gl)$}

\noindent
For any Lie group $G$, the closed subgroups of $G$ are precisely the Lie subgroups (by Lie's theorem), and so there is a natural map $\mathfrak{lie}\colon\Cl(G)\to\Cl(\mathfrak{g})$ sending each closed subgroup $H$ to its tangent space $\mathfrak{lie}(H)=\mathfrak{h}$ at the identity.
Because the space $\Cl(\mathfrak{g})$ is much easier to work with than $\Cl(G)$ (recall Section \ref{sec:Space_of_Subgeos}, it is a union of closed subsets of Grassmannians), one may hope to attempt an understanding of the closure of $X\subset\Cl(G)$ by computing not $\overline{X}$, but $\overline{\mathfrak{lie}(X)}$.
Unfortunately there are some severe problems with this: the map $\mathfrak{lie}$ is obviously not injective (the groups $\O(3)$ and $\SO(3)$ have the same Lie algebras in $\GL(3;\R)$ for example), but even worse $\mathfrak{lie}$ is not even \emph{continuous} with respect to the topologies on $\Cl(G), \Cl(\mathfrak{g})$ (Recall the Barber Pole Example \ref{ex:Spiraling_Cyl}).
Thus, in general $\mathfrak{lie}(\overline{X})\neq \overline{\mathfrak{lie}(X)}$, but as we show below, in the special case $X=\fam{O}_n$ or $X=\fam{D}_n$, this holds.

\begin{lemma}
Restricted to $\overline{\fam{D}_n}$, the map $\mathfrak{lie}$ is continuous.
\end{lemma}
\begin{proof}
Recall that $\fam{D}_n$ is a disjoint union of connected components, each a conjugacy class of orthogonal groups up to diagonal conjugacy.
Note that as $\Cl(\GL(n;\R))$, $\Cl(\gl(n;\R))$ are metrizable, it suffices to check continuity using sequences.
Let $G\in\fam{D}_n$ and $G_k$ a sequence converging to $G$ in $\fam{D}_n$.
Passing to a subsequence if necessary, we may assume that each $D_k$ lies in the same component of $\fam{D}_n$ as $G$.
Then each $G_k$ is in the same conjugacy orbit as $G$, so $G_k=A_kGA_k\inv$ for some sequence $A_k\in\Diag(n;\R)$, converging to the identity $I$ as $k\to \infty$.
As conjugate Lie groups have conjugate Lie algebras, we have
$$\lim\lie(G_k)=\lim\lie(A_kGA_k\inv)=\lim A_k\lie(G)A_k\inv=\lie(G)$$
using $A_k\to I$.
Thus, $\lie(\lim G_k)=\lim \lie (G_k)$ for all convergent sequences $G_k\in \fam{D}_n$, so $\lie$ is continuous on $\fam{D}_n$.

It only remains to show $\lie$ is continuous at the points of $\partial\fam{D}_n=\overline{\fam{D}_n}\smallsetminus\fam{D}_n$.
Let $H\in\partial\fam{D}_n$ and let $H_k$ be a sequence of groups converging to $H$.
Again passing to a subsequence if necessary, we may assume that all of the $H_k$ lie in a single component of $\subset\fam{D}_n$, and thus that all $H_k$ are in the same conjugacy class.
By compactness, the sequence $\mathfrak{h}_k=\lie(H_k)$ subconverges in $\Cl(\mathfrak{gl}(n;\R))$ to some limiting Lie algebra $\mathfrak{h}$, and it suffices to show that $\mathfrak{h}=\lie(H)$.

First, we note $\mathfrak{h}\subset\lie(H)$ follows from the general fact that the exponential of a Lie algebra limit is a subgroup of the geometric limit, which we review here.
Let $X\in\mathfrak{h}=\lim \mathfrak{h}_k$.
Then $X=\lim X_{k_j}$ for some $X_{k_j}\in \mathfrak{h}_{k_j}$, and as the exponential map $\exp\colon\gl(n;\R)\to\GL(n;\R)$ is continuous, $\exp(X_{k_j})\to \exp(X)$.
But $\exp(X_{k_j})\in H_{k_j}$, so we have exhibited $\exp(X)$ as the limit of a convergent sequence of elements of $H_{k_j}$, as $k_j\to\infty$.
Thus, by the sequential definition of the Chabauty topology (Definition \ref{def:Chab_Seqs}), $\exp(X)\in\lim H_k=H$.
Equivalently, $X\in\lie(H)$ as required.

To show in fact $\mathfrak{h}=\lie(H)$, we show the reverse inclusion by dimension count.
As each $H_k$ are conjugate, the Lie algebras $\mathfrak{h}_k$ are all of the same dimension $\smat{n\\2}$, and thus $\mathfrak{h}=\lim \mathfrak{h}_k$ is of dimension $\smat{n\\2}$.
But as all of the $H_k$ are conjugate, and in fact conjugate to $\O(p,q)<\GL(n;\R)$ for some fixed $p+q=n$, we note that $H=\lim H_k$ is a conjugacy limit of algebraic subgroups of the algebraic group $\GL(n;\R)$.
Thus by Theorem \ref{thm:CDW_AlgGrps} (Proposition 3.11 of \cite{CooperDW14}), $\dim H=\dim H_k$, and so $\lie(H)$ is of the same dimension as its subalgebra $\mathfrak{h}$.
So, $\mathfrak{h}=\lie(H)$ as claimed and $\lie$ is continuous at $H$.
\end{proof}

\begin{theorem}
$\overline{\fam{D}_n}\cong \mathfrak{lie}(\overline{\fam{D}_n})=\overline{\mathfrak{lie}(\fam{D}_n)}.$
\end{theorem}
\begin{proof}
Note that $\lie$ is injective on $\fam{D}_n$, and in fact on its closure: if $G,H$ are both limits of orthogonal groups with the same lie algebra in $\GL(n;\R)$, they must have the same connected component of the identity.
To see $\overline{\fam{D}_n}$ is homeomorphic to $\mathfrak{lie}(\overline{\fam{D}_n})$, note that by continuity proven above, $\mathfrak{lie}$ is a continuous bijection onto its image from the compact space $\overline{\fam{D}_n}$ into the Hausdorff space $\Cl(\mathfrak{g})$.
By continuity of $\mathfrak{lie}$ when restricted to $\overline{\mathcal{D}_n}$, we have that $\mathfrak{lie}(\overline{\fam{D}_n})\subset \overline{\mathfrak{lie}(\fam{D}_n)}$.
But by the compactness of $\overline{\mathcal{D}_n}$ the image $\mathfrak{lie}(\overline{\fam{D}_n})$ is compact and thus closed, and obviously contains $\mathfrak{lie}(\fam{D}_n)$ so $\overline{\mathfrak{lie}(\fam{D}_n)}\subset\mathfrak{lie}(\overline{\fam{D}_n})$ proving equality.
\end{proof}

\subsection{From $\Cl(\gl)$ to $(\RP^1)^M$}

The following simplification is just an extended observation about where the image of $\PDiag^\times$ under $\lie\circ\iota$ in $\Cl(\gl(n;\R))$ lies.
Recall the space of Lie subalgebras of a Lie algebra $\mathfrak{g}$ under the Chabauty topology is homeomorphic to a disjoint union of subsets of grassmannians over $\gl(n;\R)$.
In fact, $\overline{\lie(\fam{D}_n)}$ lies in a single Grassmannian by connectedness, and as $\dim \O(p,q)=\smat{n\\2}$, this gives
$\overline{\lie(\fam{D}_n)}\subset\Gr\left(\smat{n\\2},n^2\right)$.
But as the set of $\smat{n\\2}$-dimensional Lie subalgebras of $\gl(n;\R)$ is closed subset of $\Gr(\smat{n\\2},n^2)$, the closure of $\lie(\fam{D}_n)$ in $\Cl(\gl(n;\R))$ and in $\Gr(\smat{n\\2},n^2)$ agree.

\begin{corollary}
$\overline{\fam{D}_n}$ is homeomorphic to $\overline{\lie(\fam{D}_n)}\subset\Gr(\smat{n\\2},n^2)$.	
\end{corollary}

\noindent
We can do even better however; the image of $\lie\circ\iota$ lies not just in this Grassmannian, but in a particularly nice closed subset, homeomorphic to a high dimensional torus.
To see this, we first recall the particular form of the Lie algebra of $\so(J)$ for $J$ a diagonal matrix.

\begin{remark}
The Lie algebra $\so(J)$ for $J=\diag(\lambda_1,\ldots,\lambda_n)$ is 
$$\so(J)=\span \left\{ \lambda_j e_{ij}-\lambda_i e_{ji}\right\}_{i<j}$$
for $e_{ij}$ the standard basis for $\M(n;\R)$.	
\end{remark}

\begin{example}
$$\so\smat{x&&\\&y&\\&&z}=\span\left\{ 
\pmat{0& y &0\\-x&0&0\\0&0&0},
\pmat{0&0&z\\0&0&0\\-x&0&0}
\pmat{0&0&0\\0&0&z\\0&-y&0}
\right\}$$
\end{example}

\noindent
Note that the basis chosen above for $\so(J)$ consists of pairwise orthogonal vectors, for all nonzero choices of $\lambda_1,\ldots, \lambda_n$.
Moreover, each basis vector $\lambda_j e_{ij}-\lambda_i e_{ji}$ lies in the 2-plane $\span\{e_{ij}, e_{ji}\}$ which is orthogonal to the span of the remaining basis vectors.
This already provides useful information, as in taking the closure of $\lie(\fam{D}_n)$ we are interested in looking at limits of the vector subspaces $\so(J)$ as some of the eigenvalues of $J$ limit to $0,\infty$.
Describing a path of linear subspaces as the span of a path of vectors is in general problematic, as if in the limit the chosen basis vectors become linearly dependent, there are many continuous ways to regain linear dependence, but this does not always translate to continuity of their span.
Knowing that our chosen basis always consists of orthogonal vectors ensures us this cannot happen.

\begin{observation}
For each $1\leq i<j\leq n$, let $E_{ij}\colon \PDiag^\times\to \Gr(1,n^2)$ be the map $[\lambda_1,\ldots:\lambda_n]\mapsto \span\{\lambda_j e_{ij}-\lambda_i e_{ji}\}$.
Then we may express the Lie algebra $\so(J)$ for $J=\diag(\lambda_1\ldots, \lambda_n)$ as
$$\so(J)=\bigoplus_{i<j}E_{ij}([J])$$
\end{observation}

\noindent
We now use this to show that $\lie(\fam{D}_n)$ lies in a $\smat{n\\2}$-dimensional torus inside of $\Gr(\smat{n\\2},n^2)$.
For convenience in what follows, we will index vectors of length $\smat{n\\2}$ by $\vec{x}=(x_{ij})_{i<j}$ with two indices $i,j$ subject to the constraint $1\leq i<j\leq n$.

\begin{proposition}
The map $\Phi\colon \PDiag^\times\to \Gr(\smat{n\\2},n^2)$ defined by $J=\diag(\lambda_1,\ldots, \lambda_n)\mapsto \so	(J)$ factors through an inclusion 
$\eta\colon(\RP^1)^{\smat{n\\2}}\inject \Gr(\smat{n\\2},n^2).$
\label{prop:Factor_Eta}
\end{proposition}
\begin{proof}
For each $i<j$, define the map $\eta_{ij}\colon\RP^1\to \Gr(1,n^2)$ by $\eta_{ij}([x:y])=\span\{ ye_{ij}-x e_{ji}\}$.
The produce of these maps defines a map 
$\eta=\prod_{i<j}\eta_{ij}\colon\left(\RP^1\right)^{\smat{n\\2}}\to\prod_{i<j}\Gr(1,n^2)$
Noting that for all $\vec{x}\in(\RP^1)^{\smat{n\\2}}$ the image $\eta(\vec{x})$ consists of $\smat{n\\2}$ pariwise orthogonal vectors, we may take their direct sum to get a well-defined vector space of dimension $\smat{n\\2}$, providing a map
$$\eta\colon \left(\RP^1\right)^{\smat{n\\2}}\to\Gr\left(\smat{n\\2},n^2\right)\hspace{1cm}
\left([x_{ij},y_{ij}]\right)_{i<j}\mapsto \bigoplus_{i<j}\eta_{ij}\left([x_{ij}:y_{ij}]\right)$$
This map is a continuous bijection, and thus a homeomorphism onto its image as $(\RP^1)^{\smat{n\\2}}$ is compact and $\Gr(\smat{n\\2},n^2)$ is Hausdorff.
Now, looking at the map $\Phi\colon \PDiag^\times(n;\R)\to\Gr(\smat{n\\2},n^2)$ given in the proposition statement, we see that $\Phi=\eta\circ\Psi$ for $\Psi$ the map $\Psi\colon\PDiag^\times\to (\RP^1)^{\smat{n\\2}}$ with components $\Psi=(\psi_{ij})_{i<j}$ given by $
\psi_{ij}(\diag(\lambda_1,\ldots,\lambda_n))=[\lambda_i:\lambda_j]$.

\begin{figure*}
\centering
\begin{tikzcd}
\PDiag^\times(n;\R)\arrow[rr,"\Phi"]\arrow[dr,"\Psi"]&&\Gr\left(\smat{n\\2},n^2\right)\\	
&\left(\RP^1\right)^{\smat{n\\2}}\arrow[ur,"\eta"]&
\end{tikzcd}
\end{figure*}

\end{proof}

\begin{corollary}
As $\eta((\RP^1)^{\smat{n\\2}})$ is closed in $\Gr(\smat{n\\2},n^2)$,  the space of interest $\overline{\fam{D}_n}\cong\overline{\lie(\fam{D}_n)}$ may be computed either as $\overline{\Phi(\PDiag^\times)}\subset\Gr(\smat{n\\2},n^2)$ or $\overline{\Psi(\PDiag^\times)}\subset(\RP^1)^{\smat{n\\2}}$.
\end{corollary}

\noindent
After this collection of simplifications, we have replaced the original question of calculating $\overline{\fam{D}_n}$ in $\Cl(\GL(n;\R))$ with something significantly easier:

\begin{theorem}
Let $\PDiag^\times$ be the coordinate hyperplane complement in $\RP^{n-1}$, thought of as the projective space of nondegenerate diagonal matrices, and let $\Psi\colon\PDiag^\times\to(\RP^1)^{\smat{n\\2}}$ be the map
$\Psi([\lambda_1:\ldots:\lambda_n])=\left([\lambda_i:\lambda_j]\right)_{1\leq i<j\leq n}$
Then $\overline{\Psi(\PDiag^\times)}$ is homeomorphic to $\overline{\fam{D}_n}$.
\end{theorem}

\section{Computing the Closure $\overline{\mathcal{D}}$}
\label{sec:Computing_D_Closure}

We've succeeded in describing the Chabauty compactification $\overline{\fam{D}_n}$ not as the closure of $\fam{D}_n$ in the poorly behaved space $\Cl(\GL(n;\R))$ but instead as the closure of a particular embedding in the $\smat{n\\2}$ torus!
This already provides a wealth of information, as calculating limit points of $\fam{D}_n$ explicitly along any path is now a trivial exercise.

\begin{example}
\label{ex:Computing_Limits}
Consider the path $[xt:yt:t^2:1]\in\PDiag^\times(4;\R)$, which leaves every compact set of $\PDiag^\times$ as $t\to\infty$.
Its image under $\Psi$ is the path
$$\Psi([xt:yt:t^2:1])=\left(
[xt:yt],[xt:t^2],[xt:1],[yt:t^2],[yt:1],[t^2:1]
\right)
$$
Which as $t\to\infty$ has limit,
$$\lim_{t\to\infty}\Psi([xt:yt:t^2:1])=\left(
[x:y],[1:0],[0:1],[1:0],[0:1],[0:1]
\right).$$
And the information encoded by this limit point is easily converted to the actual limiting Lie algebra in $\gl(4;\R)$ via $\eta\colon (\RP^1)^6\to\Gr(6;16)$.
$$\lim\Phi([xt:yt:t^2:1])=$$
$$\span\left\{
\smat{0&y&0&0\\-x&0&0&0\\0&0&0&0\\0&0&0&0},
\smat{0&0&0&0\\0&0&0&0\\1&0&0&0\\0&0&0&0},
\smat{0&0&0&1\\0&0&0&0\\0&0&0&0\\0&0&0&0},
\smat{0&0&0&0\\0&0&0&0\\0&1&0&0\\0&0&0&0}
\smat{0&0&0&0\\0&0&0&1\\0&0&0&0\\0&0&0&0},
\smat{0&0&0&0\\0&0&0&0\\0&0&0&1\\0&0&0&0}
\right\}
$$
\end{example}

Computing this example gives some intuition for the information captured by limit points: the limit preserves information about the \emph{pairwise relative rate of divergence} of the eigenvalues of a path of matrices in $\PDiag^\times$.
In the computation above, this information encodes that $\lambda_1\sim\lambda_2$ and their ratio is $[x:y]$, and that $\lambda_1>\lambda_3$,$\lambda_1<\lambda_4$, $\lambda_2>\lambda_3$, $\lambda_2<\lambda_4$, $\lambda_3<\lambda_4$.
This information can be summarized by $\lambda_4>\lambda_1\sim\lambda_2>\lambda_3$ together with the extra information that $[\lambda_1:\lambda_2]=[x:y]$.

\begin{proposition}
Limit points of the image $\Psi(\PDiag^\times)$ correspond to an ordered partition of $\{\lambda_1,\ldots,\lambda_n\}$ together with the additional data of a point $x\in\RP^{k-1}$ with all nonzero entries, for each collection in the partition of size $k$.
\label{prop:Limit_Partitions}
\end{proposition}
\begin{proof}
We construct this partition and the accompanying projective points inductively.
Let $p\in\overline{\mathsf{im}\Psi}$, that is $p=\lim \Psi(\alpha(t))$ for $\alpha(t)\in\RP^{n-1}\smallsetminus \fam{A}$.
In the trivial case, $\lim\alpha(t)$ also lies in the hyperplane complement, in which case there is a single set in the partition $J_0=\{1,\ldots,n\}$ and $\lim[\alpha(t)]$ is the corresponding projective point.
Otherwise, $\lim[\alpha(t)]\in\fam{A}$ and some coordinates of $\alpha$ limit to $0$ as $t\to\infty$.
Chose a representative $\alpha(t)$ of $[\alpha(t)]$ chosen so all coordinates remain bounded but do not all converge to $0$ (say, a norm 1 representative).
Then let $I_1\subset\{1,\ldots, n\}$ be the set of indices such that $\alpha_i(t)\to 0$, and $J_1=J_0\smallsetminus I_1$.
Let $\alpha_{J_1}$ denote the projection of $\alpha$ onto the coordiantes in $J_1$; and note $\lim [\alpha_{J_1}]=\ell_1$ is a point of $\RP^{|J_1|-1}$ with nonzero coordinates by construction.
The remaining coordinates $\alpha_{I_1}$ all converge to zero, but we may begin the process again with the projective point $[\alpha_{I_1}]$: after suitably rescaling either all coordinates converge to a nonzero value in the limit; or there is a further division of rates.
In the first case, $J_2=I_1$ and our partition is $\{1,\ldots, n\}=J_1\cup J_2$ with corresponding projective points $\lim[\alpha_{J_1}]$ and $\lim[\alpha_{J_2}]$.
In the second case, we divide $I_2=I_3\cup J_2$ into the indices converging to zero / not zero respectively, and repeat.
This terminates in a partition $\{1,\ldots n\}=J_1\cup\cdots \cup J_k$ and a collection of projective points $L_i=\lim [\alpha_{J_i}]$.

We now show that this data is equivalent to, that is, \emph{uniquely determines} and \emph{is uniquely determined by} the limiting point $p=\lim \Phi(\alpha(t))$.
For each $1\leq i<j\leq n$ the limit $p$ has a coordinate $p_{ij}=\lim [\alpha_i:\alpha_j]$ by definition, encoding the pairwise limiting behavior.
These values are determined by the partition \& projective points as follows:
if $i\in J_\ell$ and $j\in J_m$, then $p_{ij}=\lim [\alpha_i:\alpha_j]$ is $[0:1]$ if $j<i$ and $[1:0]$ if $i<j$ by the definition of the partition $\{J_\ell\}$.
This determines all the coordinates of the limit point $p$ except for those $p_{ij}$ with $i,j$ in the same partition.
But, if $i,j\in J_m$ then $p_{ij}$ is directly determined by the associated projective point $L_m\in\RP^{|J_m|-1}$, by simply selecting the elements corresponding to the $i^{th}$ and $j^{th}$ coordinates.

Conversely, let $p=(p_{ij})_{1\leq i<j\leq n}$ be the limit of $\Psi(\alpha(t))$, and we see that we may reconstruct the data $(\{J_i\},\{L_i\})$ from $p$ directly, without reference to the path $\alpha$.
The set $J_1$ contains the coordinates of $\alpha(t)$ not limiting to $0$, which is recovered from $p$ by noting $i\in J_1$ if and only if $[p_i:p_j]=[1:x]$ for all $j\in \{1,\ldots, n\}$.
Continuing inductively, $J_2=\{i\mid p_{ij}=[1:x]\;\mid\; j\not\in J_1\}$, and so on.
The points $L_k\in\RP^{|J_k|-1}$ are produced easily from the set $\{p_{ij}\mid i,j\in J_k\}$ as follows:
choose some index $\ell\in J_k$, and choose the representatives $p_{i\ell}=[x_i:1]$ for each $i\neq \ell$ in $J_k$.
Then $L_k$ has as coordinates $x_i$ for each $i^{th}$ coordinate, and $1$ for the $\ell^th$.
This is well defined and independent of the choice of $\ell$, and recovers the limit point of the original construction.
\end{proof}

\noindent
We will have much use for this description in what follows.
Below we show that $\overline{\fam{D}_n}$ can be described as a compactification of the hyperplane complement $\RP^{n-1}\smallsetminus\fam{H}$ achieved via a sequence of blowups.
To begin this analysis, we aim to re-express the closure of the image under $\Psi$ as the closure of the \emph{graph of $\Psi$}, as this is a common framework in algebraic geometry.

\begin{lemma}
Let $\Gamma_\Psi\subset\RP^{n-1}\times(\RP^1)^{\smat{n\\2}}$ be the graph of $\Psi$, and $\overline{\Gamma_\Psi}$ its topological closure as a subspace.
Then the projection $\pi\colon\RP^{n-1}\times(\RP^1)^{\smat{n\\2}}(\RP^1)^{\smat{n\\2}}$ restricts to a homeomorphism
 $\overline{\Gamma_\Psi}\to \overline{\Psi(\PDiag^\times)}$ of the graph closure onto the image closure.
 \label{lem:Psi_Injective}
\end{lemma}
\begin{proof}
It suffices to show the restriction of $\pi$ is injective, as this implies it is a continuous bijection onto its image, and thus a homeomorphism by compactness of the graph closure $\overline{\Gamma_\Psi}$.
Let $(x,p)$ and $(y,p)$ be points of $\overline{\Gamma_\Psi}$, and choose representative paths $\alpha, \beta$ such that $\alpha(t)\to x$, $\beta(t)\to y$ and $\lim \Psi(\alpha(t))=\lim\Psi(\beta(t))=p$ as $t\to\infty$.
By Proposition \ref{prop:Limit_Partitions}, the point $p$ encodes a partition $J_1\cup\cdots\cup J_k=\{1,\ldots,n\}$ and corresponding values $L_m\in\RP^{|J_m|-1}$, which describe the limiting behavior of any path $\gamma$ with $\lim\Psi(\gamma(t))=p$.
The actual limiting value of the path in $\RP^{n-1}$ is completely determined by the first partition $J_1$ and associated value $L_1$: in the limit of the $j^{th}$ coordinate is $0$ if $j\not\in J_1$, and the full limit point is simply $L_1$ with these $0$'s sprinkled in.
Thus, as $\alpha(t), \beta(t)$ both have limit $p$, they have the same $J_1, L_1$ and thus $\lim\alpha(t)=x=y=\lim(\beta(t))$ so $(x,p)=(y,p)$ as desired.
\end{proof}

Thus, we may think of $\overline{\fam{D}}=\overline{\Gamma_\Psi}$ as coming equipped with a projection down onto $\PDiag\cong \RP^{n-1}$.
We analyze the closure in terms of this projection below.

\subsection{The Structure of $\overline{\fam{D}_n}$}

To understand $\overline{\fam{D}_n}$, we decompose it into smaller pieces, determined by the structure of the coordinate hyperplane arrangement $\fam{A}\subset\RP^{n-1}$.

\begin{definition}
The coordinate hyperplane arrangement $\fam{A}$ in $\RP^{n-1}$ consists of the projectivized coorinate hyperplanes themselves $A_i=\{[\vec{\lambda}]\mid \lambda_i=0\}$ together with all intersections.
We denote these via multi-index notation: for $I\subset \{1,\ldots n\}$ let $A_I=\cap_{i\in I} A_i$.
\end{definition}

\begin{observation}
This provides $\RP^{n-1}$ with the cell structure of the projectivized cross polytope of dimension $n-1$.
We denote the set of all open cells (of all dimensions) by $\fam{S}_n$ and $\fam{S}_n^k$ the subset containing cells of dimension $k$, and note that $\fam{S}_n^k$ consists of  $2^k\smat{n\\k+1}$ regular open $k$-simplicies.
\end{observation}

\begin{example}
For $n=2$, $\fam{S}_1$ is $\RP^1$ divided into two open intervals by the points $[0:1]$ and $[1,0]$.
\end{example}
\begin{example}
For $n=3$, $\fam{S}_3=\fam{S}_3^0\cup\fam{S}_3^1\cup\fam{S}_3^2$ consists of four triangles, six edges and 3 vertices.
\end{example}

\begin{figure}
\centering\includegraphics[width=\textwidth]{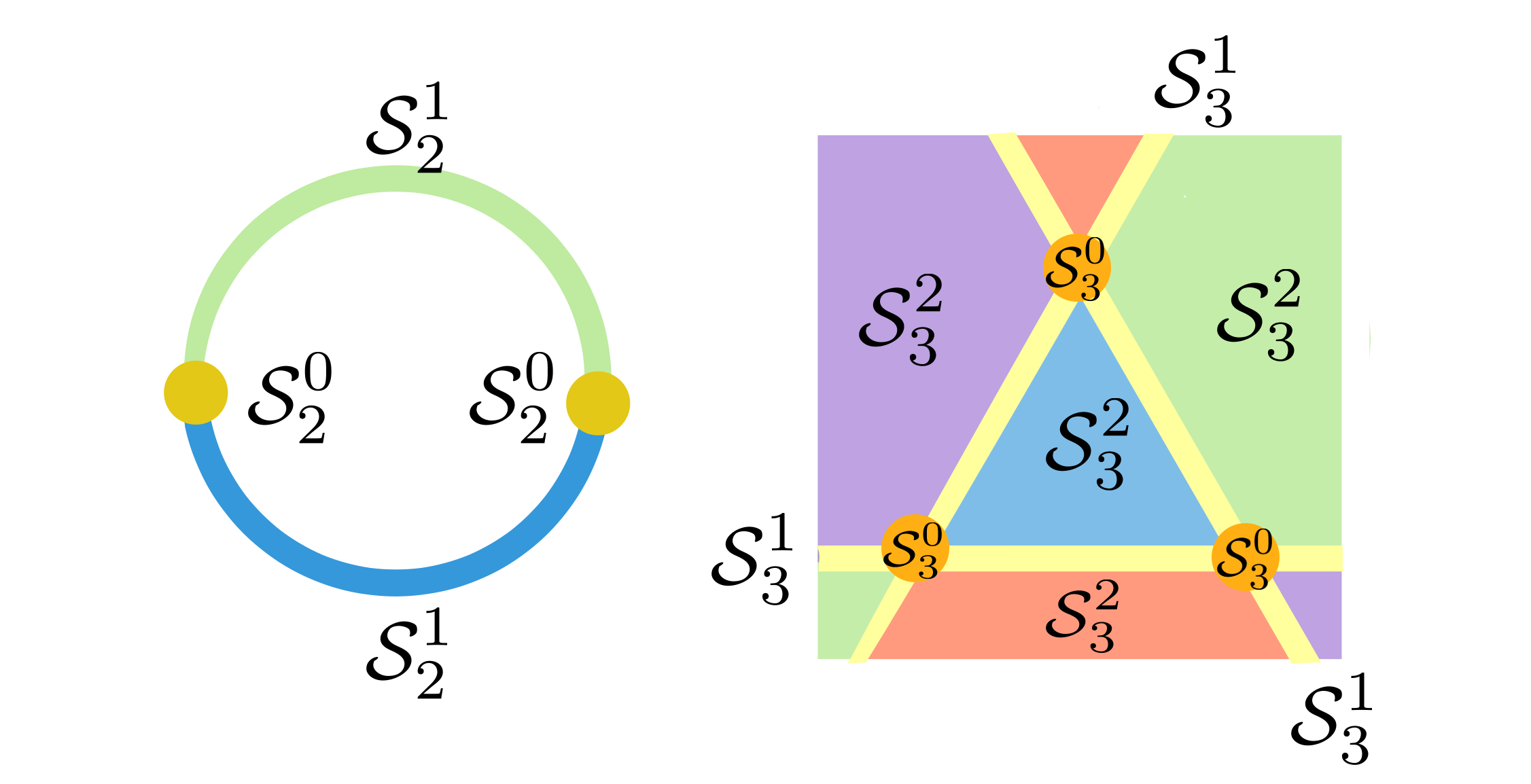}
\caption{Cellulation of $\RP^1$ and $\RP^2$.}	
\end{figure}

\begin{example}
When $n=4$, the arrangement $\fam{A}$ contains the four coordinate hyperplanes $x=0,y=0,z=0,w=0$ as well as their six intersections of dimension two, and their additional four intersections of dimension $1$, the coordinate axes.
Broken into cells, there are four points in $\fam{S}^0_4$, twelve edges in $\fam{S}^1_4$, sixteen triangles in $\fam{S}^2_4$, and eight tetrahedral cells in $\fam{S}^3_4$.
\end{example}

\begin{figure}
\label{fig:16Cell}
\centering
\includegraphics[width=0.5\textwidth]{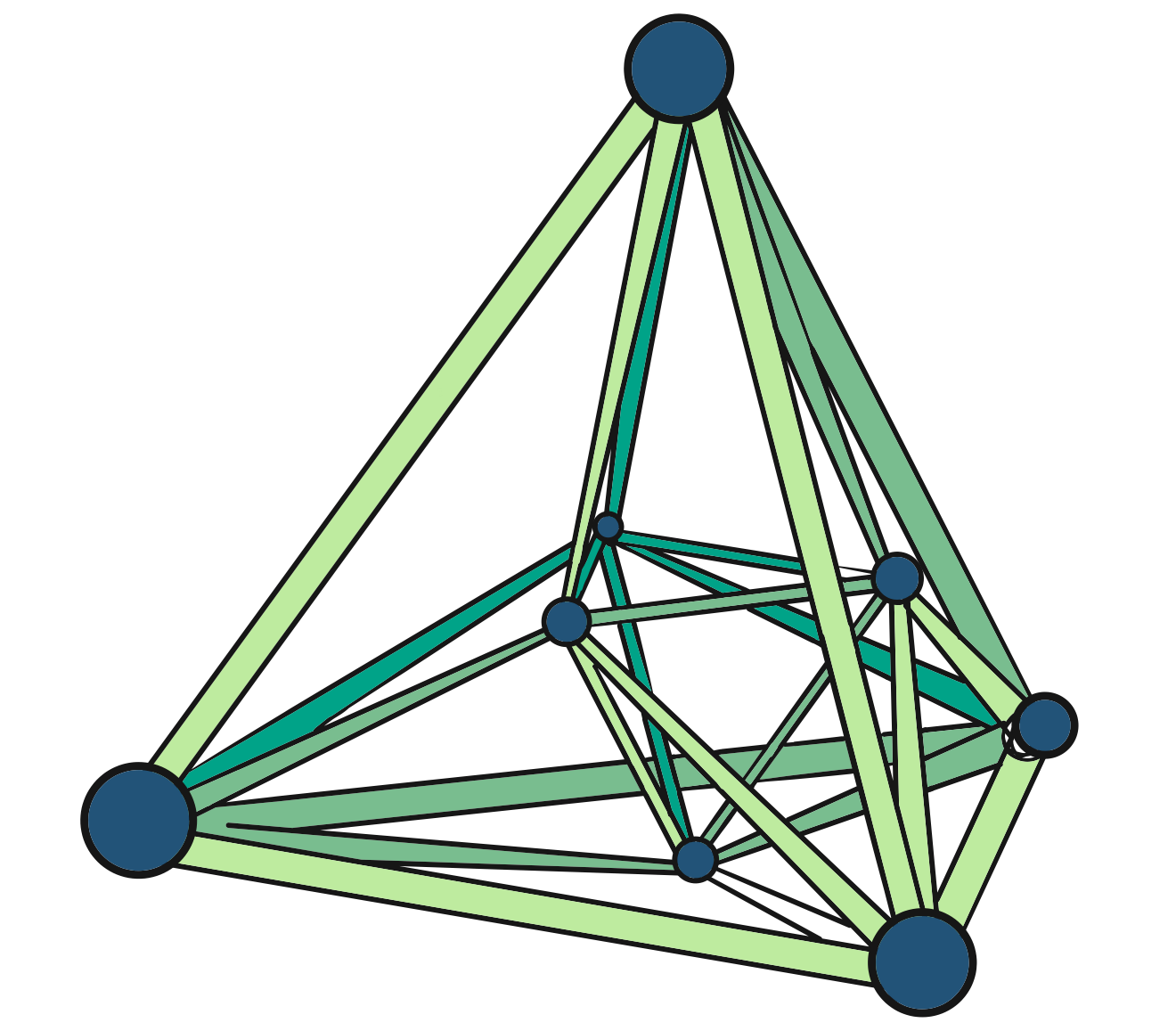}
\caption{The Cellulation of $\RP^3$, in the double cover; isomorphic to the 16-cell.}	
\end{figure}

\noindent
This cellulation of $\PDiag=\RP^{n-1}$ is useful for understanding the global structure of $\overline{\fam{D}_n}$ as any two points lying in the same face have isomorphic fibers under the projection map $\pi\colon\overline{\fam{D}_n}=\overline{\Gamma_\Psi}\to\RP^{n-1}$.
The first case, classifying fibers over the points in the top dimensional cells follows immediately from the fact that $\Psi$ is a well defined function on $\RP^{n-1}\smallsetminus \fam{A}$.

\begin{observation}
The fibers of $\overline{\fam{D}}$ over a point 	in $\fam{S}_n^{n-1}=\RP^{n-1}\smallsetminus\fam{A}$ are singletons.
\end{observation}

\noindent
The interesting points (unsurprisingly) are the points of the closure projecting to points in the hyperplanes $\fam{A}$.
These correspond to actual degenerations of orthogonal groups, as $[J]$ approaches a degenerate quadratic form lying in the union of the hyperplanes.
Before working more generally, we give the two smallest-dimensional examples for motivation.

\begin{example}
When $n=2$, the domain is $\RP^1\smallsetminus \fam{A}$ for $\fam{A}=\{[0:1],[1:0]\}$ and the map $\Psi\colon \RP^1\smallsetminus\fam{A}\to\RP^1$ is $[x:y]\mapsto[x:y]$.
This formula extends continuously to the two points missing, and so the graph closure  $\overline{\Gamma_\Psi}=\overline{\fam{D}_2}$ is all of $\RP^1$.
\end{example}

\begin{example}
When $n=3$, the domain is $\RP^2\smallsetminus\fam{A}$ for $\fam{A}=\{[x:y:z]\mid x=0 \vee y=0\vee z=0\}$.
The map $\phi$ embeds this in $(\RP^1)^3$ via
$$[x:y:z]\mapsto ([x:y],[y:z],[x:z])$$
When only one of $x,y,z$ is zero, $\phi$ is still well-defined and so extends continuously to the complement of the three points $\{[0:0:1],[0:1:0],[1:0;0]\}$ on $\RP^2$.
At these three points $\psi$ is undefined, but taking for example $p=[0:0:1]$ there are curves $p_t=[x_t:y_t:1]$ limiting to $p$ such that $\phi(p_t)\to ([u:v],[0:1],[0:1])$ for all $[u:v]\in\RP^1$ (take for example $x_t=ut$, $y_t=v_t$).
Thus the closure of the graph near $[0:0:1]$ is given by the blow up at this point, and $\overline{\fam{D}_3}$ is the blowup of $\RP^2$ at three points.
\end{example}

If $[p]\in\fam{S}^{n-1}_n$ is a point of a top-dimensional cell, then as previously noted the fiber above $[p]$ is a singleton as this is in the domain of the function $\psi$.
If $[p]\in\fam{S}^{n-2}_n$ then one coordinate of $[p]$ is zero.
The definition of $\Psi$ here extends without issue to $[p]$ as even with a single coordinate zero; each pair of coordinates represents a well-defined point of $\RP^1$, as visible in the $n=3$ example above.

The first interesting cases arise when more than one coordinate of $[p]$ is zero. 
For $[p]\in\fam{S}^{n-3}_n$, two coordinates are zero, say $p=(0,0,x_3,\ldots, x_n)$.
Then each $\psi_{ij}$ is well defined as one of $x_i,x_j\neq 0$ with the exception of $\psi_{12}$.
Thus, to understand the closure, it suffices to understand the limiting values of $\psi_{12}$ as we approach $p$.
For any $(x_1,x_2)\in\R^2\smallsetminus 0$ the path $(tx_1,tx_2)$ limits to $0$ as $t\to 0$, and so the path $[p_t]=[tx_1:tx_2:x_3:\cdots:x_n]$ limits to $[p]$.
But $\psi_{12}([p_t])=[x_1:x_2]$ is constant so $([p],[x_1:x_2])$ is in the graph closure of $\psi_{12}$.
Thus, the fiber above $[p]$ in $\overline{\fam{D}}$ is a copy of $\RP^1$.
This continues more generally, and the fiber above a point with multiple zeroes is determined by the graph closure of a restricted number of the functions $\psi_{ij}$.
This leads to an inductive description of the full space $\overline{\fam{D}_n}$.

\noindent
As a first step towards this, we observe that while $\Psi$ is not well-defined at any point in the arrangement $\fam{A}$, for each $[p]\in\fam{A}$ we may divide $\Psi$ into two parts $\Psi=\Psi_A\times\Psi_B$ where $\Psi_A$ contains all components ill-defined at $[p]$ and $\Psi_B$ contains all components which extend continuously over $[p]$.

\begin{lemma}
If $p\in\fam{S}^k_n$ then $\Psi\colon\RP^{n-1}\smallsetminus \fam{A}\to(\RP^1)^{\smat{n\\2}}$ factors as $\Psi=\Psi_A\times\Psi_B$ for $\Psi_A=(\psi_{ij})_{ij\in A}$ and $\Psi_B=(\psi_{ij})_{ij\in B}$ and $\Psi_B$ has a continuous extension to $[p]$.
\label{lem:Factoring_Psi}
\end{lemma}
\begin{proof}
If $[p]\in\fam{A}$ is on a $k$-dimensional component, meaning $k+1$ entries of $p$ are nonzero, and so $n-(k+1)$ entries are zero.
Without loss of generality we consider $[p]=[0\cdots :0:x_{n-k}:\cdots :x_n]$; all other possibilities are simply permutations of this.
The definition of $\psi_{ij}[p]=[p_i:p_j]$ extends continuously to $[p]$ so long as both $p_i$ and $p_j$ are not simultaneously zero.
Defining $A=\{(i,j)\mid i<j<n-k\}$ and $B$ to be the remaining indices, this means that $\Psi_B=\prod_{ij\in B}\psi_{ij}$ extends continuously to $[p]$ and all 
 $\smat{n-k-1\\2}$ functions $\psi_{ij}$ with $1\leq i<j\leq n-k-1$, are undefined at $p$.	
\end{proof}

\noindent
This simplifies the problem of computing the closure, at each point selecting out a subcollection $\Psi_A$ to study in more detail.
Understanding which values actually occur as limiting values of $\Psi_A$ results in an inductive description of the fibers of $\overline{\fam{D}_n}\to\RP^{n-1}$.

\begin{proposition}
The fiber over a point of $\fam{S}^k_n$ is homeomorphic to $\overline{\fam{D}_{n-k}}$.
\end{proposition}
\begin{proof}
Again, if $[p]\in \fam{S}^k_n$ is on a $k$-dimensional component, after possibly permuting entries we may assume $[p]=[0\cdots :0:x_{n-k}:\cdots :x_n]$.
By Lemma \ref{lem:Factoring_Psi}, we may write $\Psi=\Psi_A\times\Psi_B$ where $\Psi_A$ is undefined at $[p]$ but $\Psi_B$ extends continuously over $[p]$.
Thus, we are concerned only with the $\smat{n-k-1\\2}$ undefined functions of $\Psi_A\colon\PDiag^\times\to(\RP^1)^{\smat{n-k-1//2}}$.
These functions are independent of the final $k+1$ components of $[p]$, by definition, and so $\Psi_A$ factors through the projection $\RP^{n-1}\smallsetminus\fam{A}_n\to\RP^{n-k-2}\smallsetminus\fam{A}_{n-k}$ onto the first $n-k$ components.

$$\widetilde{\Psi}_A\colon \RP^{n-k-1}\smallsetminus \fam{A}\to \left(\RP^1\right)^{\smat{n-k\\2}}$$

$$[x_1:\cdots:x_{n-k}]\mapsto \left([x_i:x_j]\right)_{1\leq i<j\leq n-k}$$

Points in the closure $\overline{\Gamma_{\Psi}}$ above $[p]$ are in $1-1$ correspondence with in the closure of the graph of $\widetilde{\Psi_A}$.
But this is exactly the original problem, now of dimension $n-k$ instead of $n$.
Thus by definition, the closure of this image $\overline{\Gamma_{\widetilde{\Psi}_A}}\cong\fam{D}_{n-k}$.
\end{proof}

\begin{corollary}
The cellulation $\RP^{n-1}=\coprod_{k} \fam{S_k}$ induces a division of $\overline{\fam{D}_n}$ into components $\overline{\fam{D}_n}=\coprod_k \fam{S}_k\times\fam{D}_{n-k}$.
The component $\fam{S}_n^{n-1}\times\fam{D}_1\cong\fam{S}^{n-1}_n$ has dimension $n-1$, and is open and dense in the resulting space; all other components $\fam{S}_k\times\overline{\fam{D}_{n-k-1}}$ have dimension $n-2$.
\end{corollary}

\noindent
This division leads us to a natural cellulation of the closure, defined inductively, which we explore through examples here.
By convention $\overline{\fam{D}_0}=\{\star\}$ is a singleton.

\begin{example}
$\overline{\fam{D}_1}=\fam{D}_1=\{\star\}$ is a single point, representing the Orthogonal group $\O(1)=\{\pm 1\}\subset\R^\times$.	
\end{example}

\begin{example}
As $\RP^1=\fam{S}^0_2\cup \fam{S}^1_2$ is the union of two intervals and two points, the corresponding decomposition of $\overline{\fam{D}_2}$ is 
$\overline{\fam{D}_2}=\fam{S}_2^0\times\overline{\fam{D}_1}\cup \fam{S}^1_2\times\overline{\fam{D}_0}\cong \fam{S}_2^0\cup \fam{S}_2^1=\RP^1$.
Thus, as we know from previous discussion nothing strange happens above codimension-1 faces of the cellulation of $\RP^{n-1}$, and in this first nontrivial case, $\overline{\fam{D}_2}\cong\RP^1$.
\end{example}

\begin{example}
Inductively using the above,
$
\overline{\fam{D}_3}=
\fam{S}_3^0\times\overline{\fam{D}_2}
\cup
\fam{S}_3^1\times\overline{\fam{D}_1}
\cup
\fam{S}_3^0\times\overline{\fam{D}_0}
$, and

$$
\overline{\fam{D}_3}=\fam{S}_3^0\times
\left(\fam{S}_2^0\cup \fam{S}_2^1\right)
\cup
\fam{S}_3^1\times\left(\{\star\}\right)
\cup
\fam{S}_3^2\times\left(\{\star\}\right)
$$
$$=
(\fam{S}_3^0\times\fam{S}_2^0)\cup (\fam{S}_3^0\times\fam{S}_2^1)\cup \fam{S}_3^1\cup\fam{S}_3^2
$$
Altogether, this is a collection of $|\fam{S}_3^0||\fam{S}_2^0|=6$ vertices, $|\fam{S}_3^0||\fam{S}_2^1|+\fam{S}_3^1=6+6=12$ edges, and $|\fam{S}_3^2|=4$ two-cells.
\end{example}

A more detailed analysis here gives the attaching maps for these cells, allowing us to construct $\overline{\fam{D}_n}$ combinatorially.
Working this out in low dimensions shows that the resulting space $\overline{\fam{D}_n}$ is a manifold, and the closed top dimensional cells are permutohedra.
Below we justify this in an alternative way, by realizing our construction as a familiar object from algebraic geometry.

\section{$\overline{\fam{D}_n}$ as a Blowup}

In their 1996 paper \emph{Wonderful Models of Subspace Arrangements}, De Concini and Procesi defined the \emph{wonderful compactification} of a hyperplane arrangement complement \cite{ConciniP96}, inspired by the compactification of Fulton and MacPherson \cite{Fulton94}.
This compactification has many nice algebro-geometric properties, replacing replacing the arrangment with a divisor with normal crossings.
The compactification is a well-behaved geometric-topological object as well; it naturally carries the structure of a smooth manifold into which the original hyperplane complement embeds as an open dense subset.
The remainder of this section is devoted to (1) a brief introduction to the wonderful compactification, followed by (2) a proof of the following identificaiton.

\begin{theorem}
The Chabauty compactification $\overline{\fam{D}_n}$	 is the maximal wonderful compactification of the projectivized coordinate hyperplane arrangement in $\RP^{n-1}$.
Consequently, $\overline{\fam{D}_n}$ is a smooth manifold, cellulated by $2^{n-1}$ permutohedra.
\end{theorem}

\noindent

Our presentation of the wonderful compactification closely follows the treatment in \cite{Feit}.
A hyperplane arrangement in a real or complex vector space $V$ is a finite family $\fam{A}=\{U_1,\ldots, U_n\}$ of linear subspaces.
The combinatorial data associated to such an arrangement is the \emph{intersection lattice} $\mathcal{L}(\fam{A})$, which is the set of all nonempty\footnote{This deviates from the exposition of \cite{Feit} where $\fam{L}$ is the collection of \emph{all} intersections and $\fam{L}_{>0}$ is the collection of \emph{nonempty} intersections.} intersections of subspaces in $\fam{A}$, ordered by inclusion\footnote{This also differs from \cite{Feit}, where $\fam{L}_{\geq 0}$ is ordered by reverse inclusion but the sequence of blowups is indexed by $\mathcal{L}_{>0}^\mathsf{op}$.}.

\begin{example}
The coordinate hyperplane arrangement $\fam{A}_2\subset\R^2$	 is the union of the coordinate axes, and $\mathcal{L}(\fam{A}_2)$ contains the empty intersection $\R^2$, both axes and their intersection $\{0\}$.
The arrangement $\fam{A}_3\subset\R^3$ contains three coordinate hyperplanes; and the intersection lattice $\fam{L}(\fam{A}_3)$
additionally contains the 3 coordinate axes and the origin.
\end{example}

\noindent
A hyperplane arrangement $\fam{A}$ is \emph{central} if all hyperplanes in $\fam{A}$ pass through $\vec{0}$.
A \emph{projective hyperplane arrangement} is the projectivization of a central hyperplane arrangement, and the intersection poset is defined identically as the set of nonempty intersections of projective hyperplanes; which identifies with the intersection poset of the original arrangement after removing $\{0\}$.

We now give two descriptions of the maximal De Concini Procesi wonderful model for an arrangment $\fam{A}$: a definition as the closure of a graph, which we will use to connect with our previous work, and a definition as an iterated sequence of blow ups which is useful for intuition and inductive arguments.
In both cases, we have adapted the definitions of \cite{ConciniP96,Feit} to the case of a projective hyperplane arrangement.

\begin{definition}[Graph Closure Construction]
\label{def:Wonderful_Closure}
Let $\fam{A}$ be an arrangement of linear subspaces of a real vector space $V$.
The map $\Psi$ the map 
$$F \colon\mathbb{P}(V\smallsetminus \fam{A})\to \prod_{X\in\fam{L}(\fam{A})}\mathbb{P}\left(V/X\right)$$
encodes the relative position of each point in the arrangement complement with respect to the intersection of subspaces of $\fam{A}$.
The map $F$ is an open embedding; the closure of its graph is called the (maximal) De Concini-Procesi wonderful model for $\fam{A}$, and is denoted $Y_\fam{A}$.
\end{definition}

\begin{definition}[Blow Up Construction]
\label{def:Wonderful_Blow_Up}
Let $\fam{A}$ be a projective hyperplane arrangement in $\mathbb{P}V$ and let $X_1<X_2<\ldots X_t$ be a linear extension of the partial ordering on $\fam{L}(\fam{A})$.
Then the (maximal) De Concini-Procesi wonderful model for $\fam{A}$ is the result $Y_\fam{A}$ of successively blowing up the subspaces $X_1,\ldots, X_t$; respectively their proper transforms.	
\end{definition}

\begin{theorem}[De Concini Procesi]
The constructions of definitions \ref{def:Wonderful_Closure} and \ref{def:Wonderful_Blow_Up} give isomorphic algebraic varieties.
The resulting arrangement model $Y_\fam{A}$ is a smooth algebraic variety with a natural projection map to the original ambient space $\pi\colon Y_\fam{A}\to\mathbb{P} V$, which is one-to-one on the original arrangement complement $\mathbb{P}(V\smallsetminus\fam{A})$.
\end{theorem}

\noindent
Additionally, the following theorem collects some of the nice algebro-geometric properties of the wonderful arrangement models.

\begin{theorem}[De Concini and Procesi, Theorems in 3.1 and 3.2]

\begin{enumerate}
\item The preimage $\pi\inv(\mathbb{P}\fam{A})$ in $Y_\fam{A}$ is a divsior with normal crossings; its irreducible components are the proper transforms $D_X$ of intersections of $X$ in $\fam{L}$,
$$\pi\inv(\mathbb{P}\fam{A})=\bigcup_{X\in\mathcal{L}}D_X.$$
\item Irreducible components $D_X$ for $X\in\Sigma\subset\fam{L}$ in a subset $\Sigma$ of the intersection poset intersect in $Y_\fam{A}$ if and only if $\Sigma$ is a linearly ordered subset of $\fam{L}$.  If we think of $Y_\fam{A}$ as stratified by the irreducible components of the normal crossing divisor and their intersections, then the poset of strata coincides with the face poset of the order complex of $\fam{L}^\mathsf{op}$.
\end{enumerate}

\end{theorem}

\noindent
First, we look at a familiar case; $\fam{A}_2$ the coordinate hyperplane arrangement in $\RP^2$, which illustrates the equality of these two definitions.

\begin{example}[$Y_{\fam{A}_2}$]
\label{ex:BlowUp_RP2}
The elements of $\fam{L}(\fam{A}_2)$ are the coordinate hyperplanes $A_x,A_y,A_z$ and the coordinate axes $A_xy,A_yz,A_xz$.
The codomain of $\Psi$ in the graph closure construction is the product of the six projective spaces $\P(\R^3/A_I)$ for $I\in\{x,y,z,xy,xz,yz\}$; but noting that the quotient of $\R^3$ by a coordinate hyperplane is 1 dimensional so has trivial projectivization, we may write 
$F\colon\RP^2\smallsetminus\fam{A}_2\to \RP^1\times\RP^1\times\RP^1$,
$$F([x:y:z])=\left([x:y],[y:z],[x:z]\right)$$
But this is exactly the map $\Psi$ defining $\overline{\fam{D}_3}$!

From the blow-up construction, we see that we also get the correct answer, $Y_{\fam{A}_2}$ is the blow up of $\RP^2$ at three points.
Linearlizing the partial order on $\fam{L}$ means to place the projective points before projective lines, and otherwise order arbitrairly.
Blowing up at each of the projective points corresponding to a coordinate axis gives $\RP^2$ blown up at 3 points, and then blowing up along codimension-1 edges does nothing.
\end{example}

\noindent
Below, we consider the first really nontrivial case of each of these constructions, which occurs for coordinate hyperplane arrangement in $\RP^3$.
The projective arrangement here consists of the four coordinate hyperplanes $\fam{A}=\{A_x,A_y,A_z,A_w\}$, and the intersection poset additionally contains all six projectivized coordinate $2-planes$ and four vertices  (projectivized coordinate axes)

\begin{observation}
For $\fam{A}_4$ the projectivized coordinate hyperplane arrangement in $\RP^3$, $\fam{L}_4=\fam{L}(\fam{A}_4)$ is as below.

\begin{figure}
\centering
\begin{tikzcd}
                                                & A_x \arrow[ld, no head] \arrow[d, no head] \arrow[rd, no head] & A_y \arrow[lld, no head] \arrow[rd, no head] \arrow[rrd, no head] & A_z \arrow[lld, no head] \arrow[d, no head] \arrow[rrd, no head] & A_w \arrow[lld, no head] \arrow[d, no head] \arrow[rd, no head] &                                                 \\
A_{xy} \arrow[rd, no head] \arrow[rrd, no head] & A_{xz} \arrow[d, no head] \arrow[rrd, no head]                 & A_{xw} \arrow[d, no head] \arrow[rd, no head]                     & A_{yz} \arrow[rd, no head] \arrow[lld, no head]                  & A_{yw} \arrow[d, no head] \arrow[lld, no head]                  & A_{zw} \arrow[ld, no head] \arrow[lld, no head] \\
                                                & A_{xyz}                                                        & A_{xyw}                                                           & A_{xzw}                                                          & A_{yzw}                                                         &                                                
\end{tikzcd}
\caption{The intersection poset $\fam{L}_4$}
\label{fig:Intersection_Poset}
\end{figure}

\end{observation}

\begin{example}[Graph Closure Construction]
First we construct the codomain of $F$, the space $\P(\R^4)\times\prod_{X\in\fam{L}_4}\P(\R^4/X)$.
Recalling that $A_I$ denotes the projective hyperplane with $x_i=0$ for all $i\in I$, the quotient space $\R^4/A_I$ naturally identifies with the orthogonal complement $A_{\{x,y,z,w\}\smallsetminus I}$, and its projectivization with the corresponding projective space.
As a bit of notation, denote by $\mathbb{P}_I$ the projective space $\P\{(x_i)_{i\in I}\}$; then $\P(\R^4/A_I)=\P_I$, and we the codomain of $F$ is
$$
\left(\P^2_{xyz}\times\P^2_{xyw}\times\P^2_{xzw}\times\P^2_{yzw}\right)
\times\left(\P^1_{xy}\times\P^1_{xz}\times\P^1_{xw}\times\P^1_{yz}\times\P^1_{yw}\times\P^1_{zw}\right)\times
$$
$$
\times
\left(\P^0_x\times\P^0_y\times\P^0_z\times\P^0_w\right)$$

\noindent
The map $F$ itself, defined on the complement $\RP^3\smallsetminus\fam{A}$, is as follows
$$\Psi([x:y:z:w])=
\pmat{
[x:y:z],\;[x:y:w],\;[x:z:w],\;[y:z:w]\\
[x:y],\;[x:z],\;[x:w],\;[y:z],\;[y:w],\;[z:w]\\
[x],\;[y],\;[z],\;[w]
}
$$
Then $\fam{Y}_4=Y_{\fam{A}_4}$ is the graph closure $\overline{\Gamma_{F}}$.
Noting that the projective space $\RP^0=(\R\smallsetminus 0)/\R^\times=\{\star\}$ is a singleton, the four final factors of the codomain are all points and the four last coordinates of $F$ are constant maps: thus we may leave them out for simplicity if desired.
\end{example}

\begin{example}[Blow Up Construction]
The partial order on $\fam{L}_4$ by inclusion can be extended to a linear order by choosing arbitrary orderings on the subspaces of each fixed dimension, and then ordering the resulting blocks by dimension.
For example, the bottom-to-top, left-to-right dictionary ordering on the intersection poset of Figure \ref{fig:Intersection_Poset} gives
$$A_{xyz}<A_{xyw}<A_{xzw}<A_{yzw}<$$
$$<A_{xy}<A_{xz}<A_{xw}<A_{yz}<A_{yw}<A_{zw}<$$
$$<A_x<A_y<A_z<A_w$$
With respect to this order, the iterated blow-up is constructed as follows.
Beginning with $\RP^3$, blow up at the vertex $[A_{xyz}]$, the projectivization of the $w$-axis.
This procedure is local, and does not affect the topology of $\RP^3$ outside of a small neighborhood of $[A_{xyz}]$.
We successively blow up at the points $[A_{xyw}]$, $[A_{xzw}]$ and $[A_{yzw}]$ respectively (note that the order this is done does not affect the end result, which is why we were allowed to choose \emph{any} linearization of the partial order in Definition \ref{def:Wonderful_Blow_Up}).
Following this, we blow up the resulting space along the proper transform of the circle $A_{xy}\subset\RP^3$, and follow this by similar blow ups along the remaining five circles $[A_{ij}]$.
Again, the order in which this is completed is specified by our chosen linear ordering, but the final topology is independent of this choice, as the blow up operation is local and the proper transforms of the circles $[A_{ij}]$ do not intersect.
This point is worth thinking a bit about before moving on - below we illustrate in a figure the point $[A_{xyz}]$ in $\RP^3$ (visualized in the affine patch $w=1$) together with the circles $A_{xy},A_{xz},A_{yz}$ passing through it, followed by a depiction of their proper transforms after blowing up at $[A_{xyz}]$.
\begin{figure}
\centering\includegraphics[width=0.65\textwidth]{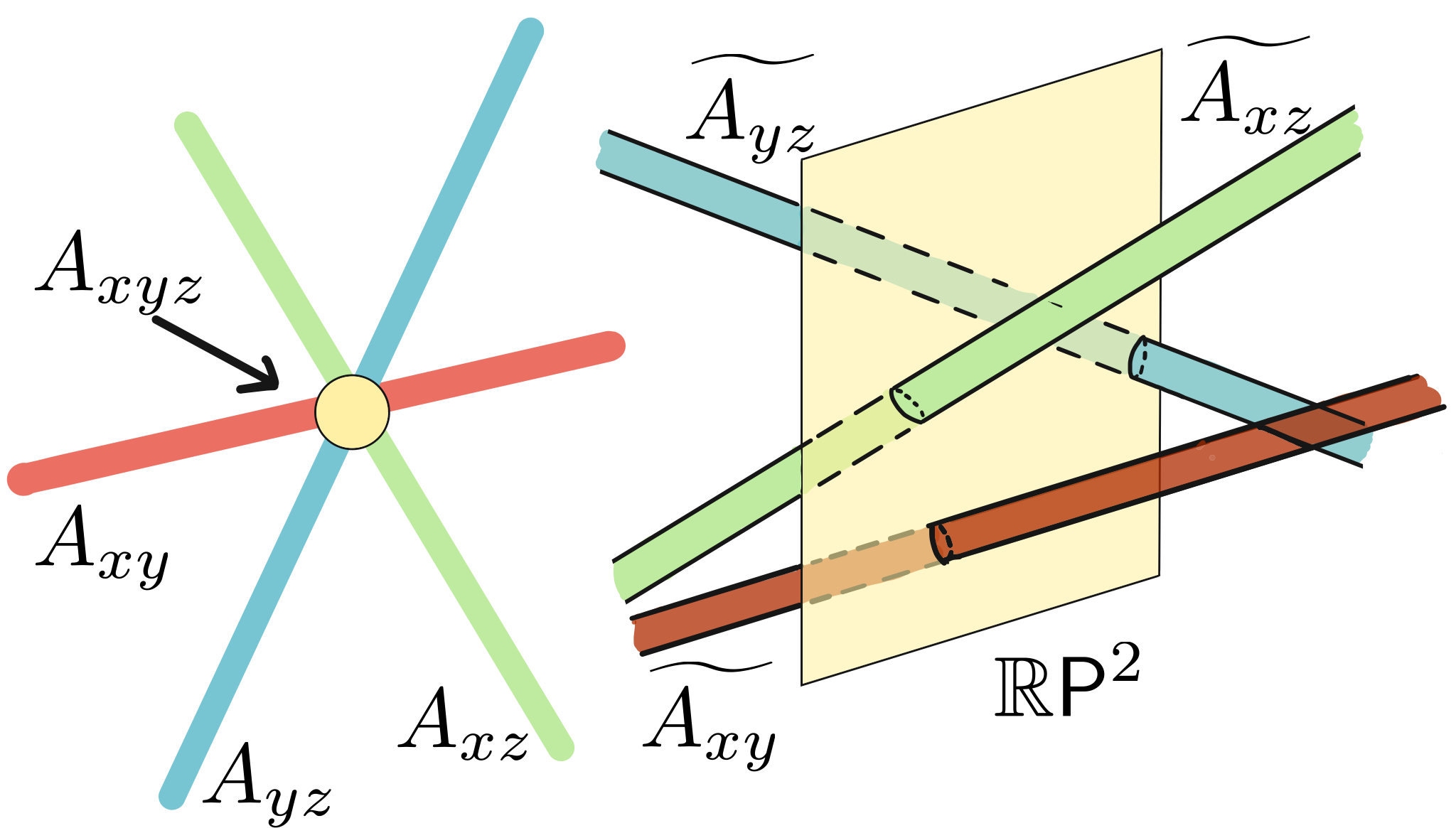}
\caption{
The circles $[A_{ij}]$ and their proper transforms.
Blowing up at $[A_{xyz}]$ introduces a copy of $\RP^2$, and the proper transforms of $[A_{ij}]$ meet this $\RP^2$ at a point encoding the original angle at which they were incident to $[A_{xyz}]$.
}	
\end{figure}

\noindent
Topologically, the blow up of a 3-manifold along a simple closed curve $\gamma$ is homeomorphic to the space resulting from deleting a small regular neighborhood of $\gamma$ and identifying the resulting boundary torus by the map fixing the longitude direction and acting as the antipodal meridianally \ref{sec:Degen_and_Regen}.

Finally, we blow up along the remaining spaces in the intersection lattice: the coordinate hyperplanes themselves.
As codimension one objects, blowing up along these does not change the topology of the space and so we may ignore this step.
\end{example}

\noindent
To connect these constructions to the space $\overline{\fam{D}_n}$, we exploit that both are defined as graph closures into products of projective space.
In fact, the defining map for $\overline{\fam{D}_n}$ is actually a \emph{factor} of the map $\Psi$ in Definition \ref{def:Wonderful_Closure}, recording only projections onto $\RP^1$ factors.
Below, we show that this information is actually enough: if we know only the projection of a point $p\in \fam{Y}_\fam{A}$ onto the 1-dimensional factors, we can recover the point exactly.

\begin{proposition}
The projection 
$$\mathsf{proj}\colon \P V\times \prod_{X\in\fam{L}}\P(V/X)\to \P V\times \prod_{\scaleto{\mat{X\in\fam{L}\\\dim X=1}}{14pt}}\P(V/X)$$
is an injective when restricted to the arrangement model $Y_\fam{A}$.
\end{proposition}
\begin{proof}
This argument is just a finer analysis in the spirit of Lemma \ref{lem:Psi_Injective} again relying on the partition description of Proposition \ref{prop:Limit_Partitions}.
Let $([x_1],U_1)$ and $([x_2],U_2)$ be two points of $Y_\fam{A}$ projecting onto the same point $([y],p)$ of $\RP^{n-1}\times(\RP^1)^{\smat{n\\2}}$.
Comparing first coordinates, clearly $[x_1]=[x_2]=[y]$ and if $[y]\in\RP^{n-1}\smallsetminus \fam{A}$ then additionally $U_1=U_2=\F(y)$ as above the hyperplane the graph closure is simply the graph of $F$.

Thus, we assume $[y]\in\fam{A}$.
To show $U_1=U_2$, it suffices to show that the data $([y], V)$ completely determines the limiting value $\lim [\alpha_S]$ of the projective point with coordinates in $S\subset\{1,\ldots,n\}$ an arbitrary subset, for any path $\alpha$ with $\lim \Psi(\alpha)=([y],V)$.
Let $J_1\cup\cdots J_k=\{1,\ldots n\}$ and $L_m\in\RP^{|J_m|-1}$ be the partition and projective points corresponding to $V\in\overline{\mathsf{im}\Psi}$ as in Proposition \ref{prop:Limit_Partitions}, and let $\ell$ be the minimal value such that $S_\ell=S\cap J_\ell$ is nonempty.
Let $\alpha$ be any path with $\lim\Psi(\alpha(t))=V$.
Then $S_\ell$ contains the indices $i\in S$ for which $\alpha_i(t)$ goes to zero slowest, so $\lim [\alpha_S]$ has zeroes at all other indices.
The values corresponding to indices in $S_\ell$ can be read off of the limit point $L_\ell$ by simply projecting from $\RP^{|J_\ell|-1}$ to $\RP^{|S_\ell|-1}$ (equivalently, they may be reconstructed from the pairs $p_{ij}$ for $i,j\in S_\ell$ as in the proof of Proposition \ref{prop:Limit_Partitions}.
\end{proof}

\noindent
Because $Y_\fam{A}$ is compact and the codomain is Hausdorff, this immediately implies the following important corollary.

\begin{corollary}
The projection above restricts to a homeomorphism on $Y_\fam{A}$.
That is, $Y_\fam{A}$ is the closure of the graph of NEW NOTATION $\mathsf{proj}\circ F$, which records the position of points relative the $n-2$ dimensional coordinate hyperplanes.
\end{corollary}

\noindent
But this map, as mentioned above, is precisely the map $\Psi$ defining $\overline{\fam{D}_n}$ as a graph closure, proving the main theorem.

\begin{theorem}
The Chabauty compactification $\overline{\fam{D}_n}$ is homeomorphic to the maximal De Concini Procesi wonderful compactification of the coordinate hyperplane arrangement in $\RP^{n-1}$.	
\end{theorem}

\section{$\overline{\fam{D}_3}$: An Example}

\label{sec:Orthog_Low_Dims}

The space $\overline{\fam{D}_3}$ was described above in Example \ref{ex:BlowUp_RP2} as the blowup of $\RP^2$ at three points.
Here we look a bit more in detail at this space, describing its cellulation and the limit groups attached to each cell.
Consider first $p=[0:0:1]\in\RP^2$, and the $\RP^1$ fiber $\{([x:y],[0:1],[0:1])\}$ lying above $[p]$.
This $\RP^1$ is divided into two components by the points $[1:0]$ and $[0:1]$ (corresponding to the hyperplanes $y=0$ and $x=0$ intersecting at $p$ in $\RP^1$)
Locally, we can construct this space by cutting out a small neighborhood of $[p]\in\RP^2$ and identifying the boundary via the antipodal map.

\begin{figure}
\centering
\includegraphics[width=0.5\textwidth]{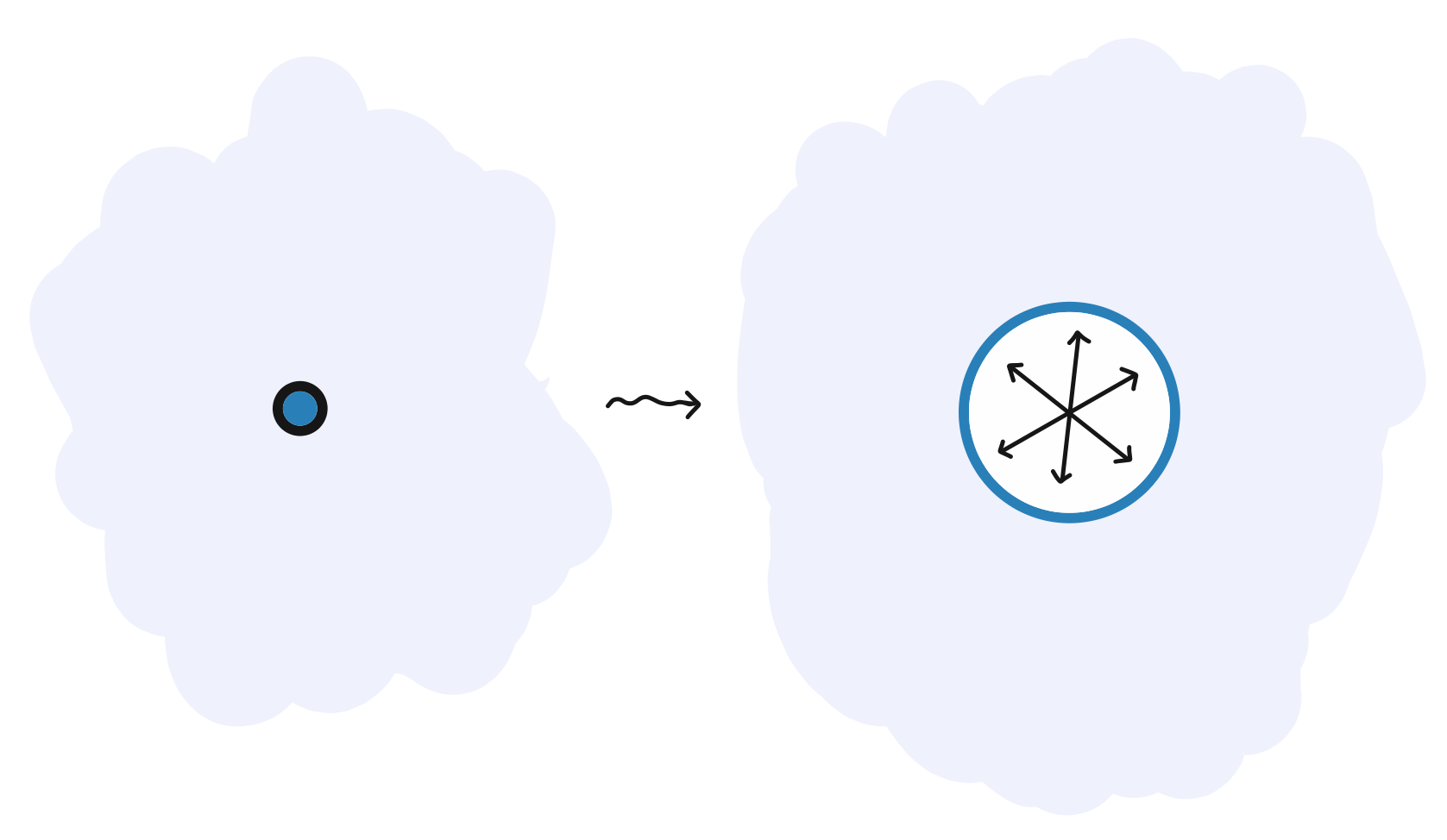}
\caption{Blow up at point construction}
\end{figure}

\begin{observation}
The four triangles from the original cellulation of $\RP^2$ appear as hexagons in the closure, as each vertex of the original tiling has been replaced with a circle subdivided into two edges, and each triangle adjacent to that vertex picks up an edge from this.
\end{observation}

\begin{figure}
\centering
\includegraphics[width=0.75\textwidth]{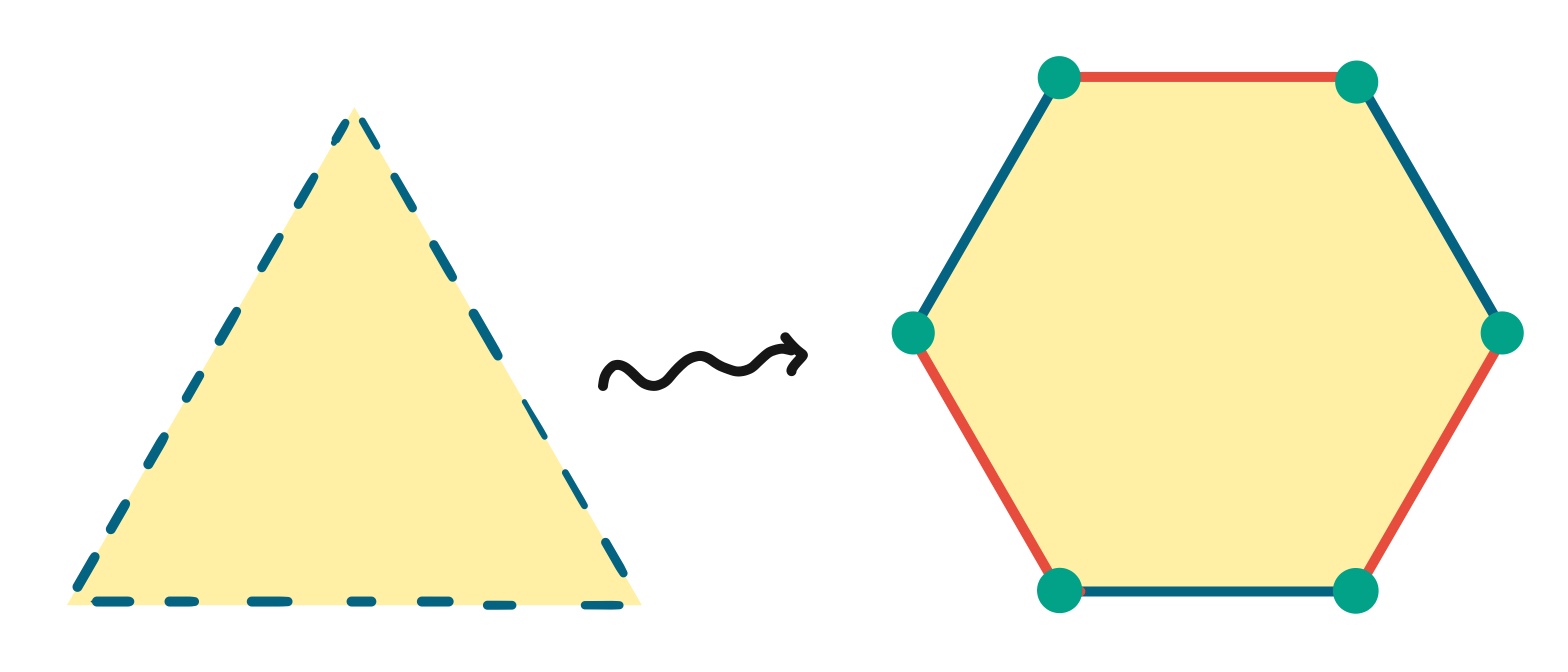}
\caption{The open trianglular cells $\fam{S}_3^2$ of $\RP^2$ have closure $\overline{\fam{S}_3^2}$ a hexagon, with the three new sides composed of half of each $\RP^1$ added in the blowup.}
\end{figure}

\begin{corollary}
The closure $\overline{\fam{D}_3}$ is tiled by four hexagons.	
\end{corollary}

\noindent
These hexagons meet two to an edge, and four to a vertex, as can be seen by considering the blowup of the original triangle tiling of $\RP^2$ at a vertex.
Geometrically we may choose these to be equilateral right angled hexagons in the hyperbolic plane, and give $\overline{\fam{D}_3}$ a natural hyperbolic structure.

\begin{figure}
\centering
\includegraphics[width=0.4\textwidth]{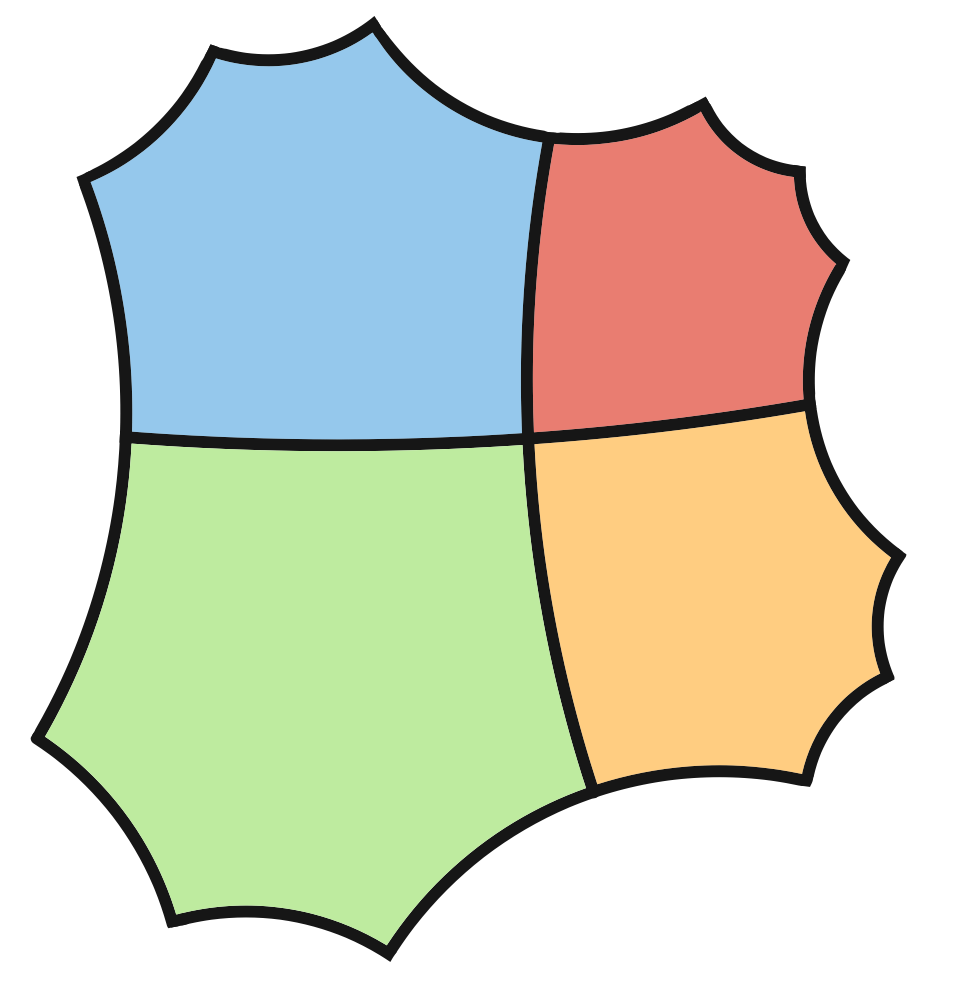}
\caption{Tiling of $\overline{\fam{D}_3}$ by Hexagons.}	
\end{figure}

\noindent
Now we turn to understand the groups parameterized by $\overline{\fam{D}_3}$, which classifies the limits of quadratic form geometries in dimension 2.
The points in the graph closure $\overline{\Gamma_\Psi}$ directly represent Lie subalgebras of $\gl(n;\R)$ under the identification $\eta\colon (\RP^1)^{\smat{n\\2}}\to\Gr(\smat{n\\2},n^2)$ of Proposition \ref{prop:Factor_Eta}.
Three of the four triangles (those containing the points $[1:1:-1]$, $[1:-1:1]$ and $[-1:1:1]$) all contain conjugates of $\SO(2,1)$ and are related by conjugation via a permutation matrix.
Thus it suffices to analyze only one of these, the diagonal conjugates of $[1:1:-1]$.
The remaining triangle containing $[1:1:1]$ parameterizes diagonal conjugates of $\O(3)$.

\begin{proposition}
The conjugacy limits of $\O(3)$ in $\GL(3;\R)$ are the Euclidean group $\Euc(2)$, its contragredient representation $\Euc(2)^{-T}$, and the real Heisenberg group.	
\end{proposition}
\begin{proof}
The boundary of the hexagon containing diagonal conjugates of $\O(3)$ contains six line segments: three of which are lifts of the original three sides of the triangle (containing the images of $[x:y:0],[x:0:z],[0:y:z]$).
The image of these edges under $\Psi$ in $(\RP^1)^3$ are 
$$([x:y],[1:0],[1:0]),\hspace{0.25cm} ([1:0],[x:z],[0:1]),\hspace{0.25cm} ([0:1],[0:1],[y:z])$$
for $x,y,z$ all positive.
Following by $\eta\colon(\RP^1)^3\to\Gr(3;9)$ reconstitutes the corresponding Lie algebras, denoted below with $u_1,u_2,u_3$ ranging over $\R$ exactly as in Example \ref{ex:Computing_Limits}.

$$
\pmat{
0& y u_1 &0\\-xu_1 &0&0\\u_2&u_3&0
},
\hspace{0.5cm}
\pmat{
0&0& zu_1\\
u_2&0&u_3\\
-xu_1&0&0
},
\hspace{0.5cm}
\pmat{
0&0&0\\
u_2&0&zu_1\\
u_3&-yu_1&0
}
$$

\noindent
These Lie algebras are all isomorphic, and in fact conjugate in $\gl(3;\R)$ to the Lie algebra for the contragredient
representation of the Euclidean group $\smat{0&u_1&0\\-u_1&0&0\\u_2&u_3&0}$.
The three remaining edges of the hexagon lie in the blow up above the vertices $[0:0:1],[0:1:0],[1:0:0]$. 

$$
([x:y],[0:1],[0:1]),\hspace{0.25cm}
([0:1],[x:z],[1:0]),\hspace{0.25cm}
([1:0],[1:0],[y:z])
$$

These are the points of $(\RP^1)^3$, which correspond under $\eta$ to Lie algebras, all of which are conjugate to that of the Euclidean group $\smat{0&u_1&u_2\\-u_1&0&u_3\\0&0&0}$.

$$
\pmat{
0& y u_1 &u_2\\-xu_1 &0&u_3\\0&0&0
},
\hspace{0.5cm}
\pmat{
0&u_2& zu_1\\
0&0&0\\
-xu_1&u_3&0
},
\hspace{0.5cm}
\pmat{
0&u_2&u_3\\
0&0&z u_1\\
0&-y u_1& 0
}
$$

Finally we come to the vertices of the hexagon, which are represented by the points of $(\RP^1)^3$ with each coordinate equal to $[0:1]$ or $[1:0]$.
For instance the sequence $p_t=[1:t:t^2]$ limits to $[0:0:1]$ and $\psi(p_t)=([1:t],[1:t^2],[t:t^2])$, which has limit $([1:0],[1:0],[1:0])$.
The Lie algebras corresponding to these points are all conjugate, and represent the Lie algebra of the Heisenberg group.

$$
\pmat{0&u_1&u_2\\0&0&u_3\\0&0&0}
\hspace{0.5cm}
\pmat{0&u_1&0\\0&0&0\\u_2&u_3&0}
\hspace{0.5cm}
\pmat{0&u_1&u_2\\0&0&0\\0&u_3&0}
$$
 
 $$
\pmat{0&0&u_2\\u_1&0&u_3\\0&0&0}
\hspace{0.5cm}
\pmat{0&0&0\\u_1&0&u_3\\u_2&0&0}
\hspace{0.5cm}
\pmat{0&0&0\\u_1&0&0\\u_2&u_3&0}
$$
\end{proof}

\begin{figure}
\centering
\includegraphics[width=\textwidth]{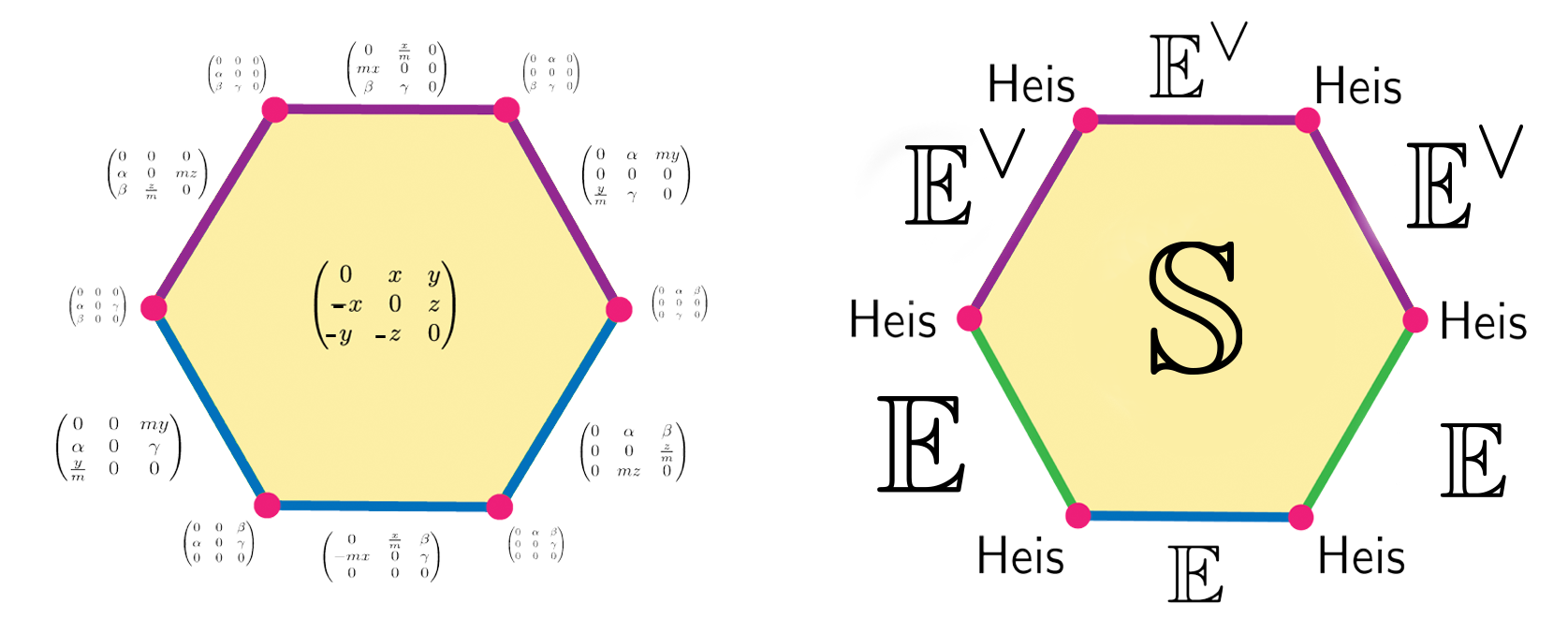}
\caption{Limits of $\SO(3)$ in $\GL(3;\R)$.  Lie algebras on the left, isomorphism types on the right.}	
\end{figure}

\noindent
This analysis gives a combinatorial description of the points lying in the boundary of the hexagon: they are given by partitions of $(x,y,z)$ into subsets which converge towards zero at different rates.
The interior corresponds to none of $x,y,z$ diverging.
The edges correspond to $(x,y,z)$ being partitioned into two sets, one going to $0$ and the other remaining bounded.
The edges from the original triangle correspond to having two remain bounded and the third go to zero.
The edges coming from the blowup construction represent the limits along paths with a single value remaining bounded and the other two going to zero.
The vertices correspond to strict orderings $x>y>z$ of which there are six.
A nearly identical story plays out for the limits of $\O(2,1)$.

\begin{proposition}
The distinct limits of $\O(2,1)$ as a subgroup of $\GL(3;\R)$ are the isometries of Euclidean and Minkowski space, their contragredient representations, and the real Heisenberg group.
\end{proposition}
\begin{proof}
Again the distinct limits correspond to distinct types of cells in the boundary, which correspond to different partial orderings on the coordinates.
By our choice to consider the triangle containing diagonal conjugates of $[1:1:-1]$, our coordinates are 
$x,y$ and $-z$ for $x,y,z>0$.
The three original edges of this triangle appear in the closure of $\Gamma_\Psi$ as the points
$$([x:y],[0:1],[0:1]),([1:0],[x:-z],[0:1]),([0:1],[0:1],[y:-z])$$
These correspond to the following Lie algebras,
which are \emph{not} all isomorphic: the first family are conjugates of the contragredient representation of the Euclidean group, and the second two families contain contragredient representations of of the isometries of $1+1$ dimensional Minkowski space.

$$
\pmat{
0& y u_1 &0\\-xu_1 &0&0\\u_2&u_3&0
},
\hspace{0.5cm}
\pmat{
0&0& zu_1\\
u_2&0&u_3\\
xu_1&0&0
},
\hspace{0.5cm}
\pmat{
0&0&0\\
u_2&0&zu_1\\
u_3&yu_1&0
}
$$

\noindent
The three remaining edges of the hexagon lie in the blow up above the vertices $[0:0:1],[0:1:0],[1:0:0]$. 
These are the following points of $(\RP^1)^3$, which correspond to Lie algebras in two isomorphism classes, depending on the relative divergence rates of $x,y,z$.
The first case gives conjugates of the Euclidean group, and the second two give conjugates of the Minkowski group.
$$
([x:y],[0:1],[0:1]),\hspace{0.25cm}
([0:1],[x:-z],[1:0]),\hspace{0.25cm}
([1:0],[1:0],[y:-z])
$$
Again, the six vertices all represent conjugates of the Lie algebra of the Heisenberg group.
\end{proof}

\begin{figure}
\centering
\includegraphics[width=\textwidth]{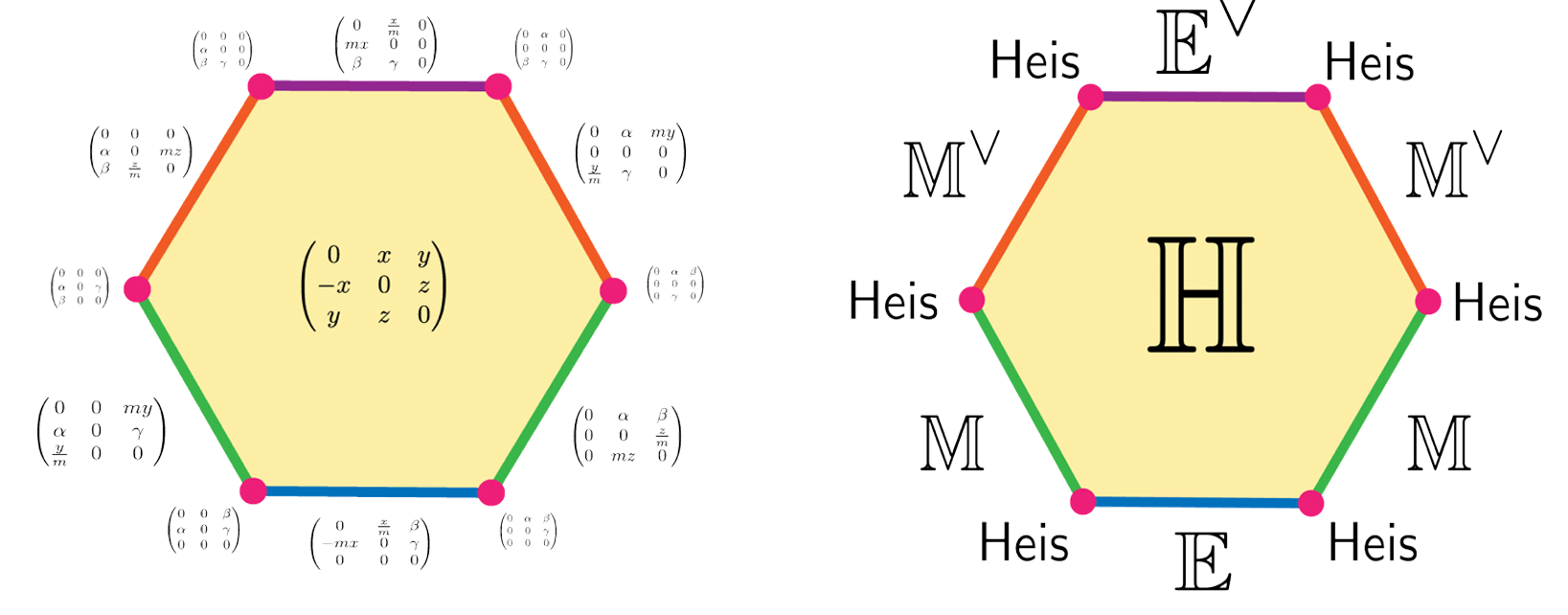}
\caption{Limits of $\SO(2,1)$ in $\GL(3;\R)$.}	
\end{figure}

This sort of analysis continues in higher dimensions.
For $n=4$, the coordinate hyperplane complement in $\RP^3$ is a union of 8 3-simplices, and taking the closure amounts to blowing up along the vertices and then the 1-cells of the cellulation in Figure \ref{fig:16Cell}.
The closure of each of the original open 3 simplices has boundary with the combinatorial structure of a permutohedron.
In general, the closure of a simplex in $\fam{S}_n^{n-1}$ of $\RP^{n-1}\smallsetminus \fam{A}$ has the structure of an omnitruncated simplex, whose boundary is the $n-1$ dimensional permutohedron.

\begin{figure}
\centering
\includegraphics[width=0.95\textwidth]{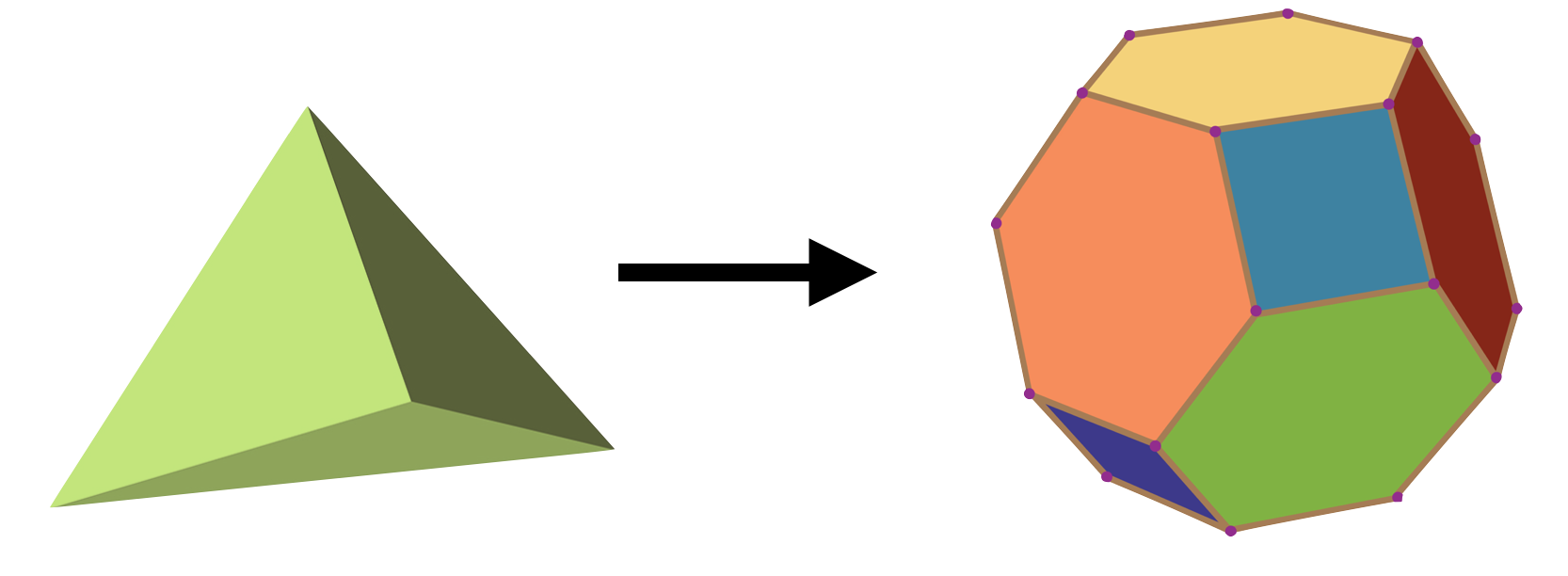}
\caption{The simplex in $\fam{D}_4$ containing all diagonal conjugates of $\O(\diag(1,1,1-1))$ in $\GL(4;\R)$, and its closure in $\overline{\fam{D}_4}$.
Recording the isomorphism type of points in the boundary recovers the limits of $\Hyp^3$ in $\RP^3$.}
\end{figure}

\chapter{The Heisenberg Plane}
\label{chp:Heisenberg_Plane}
\index{Geometries!Heisenberg}
\index{Heisenberg Geometry}

\begin{figure}
\centering
\begin{tikzcd}
& 
\mathbb{S}^2
\ar[dl]\ar[dr]
 && 
 \mathbb{H}^2 
 \ar[dl]\ar[dr]
 && 
 \mathsf{AdS}^2=\mathsf{dS}^2 
 \ar[dlllll]\ar[dl]\ar[dr]
 &\\
\widehat{\mathbb{E}}^2
\ar[drrr]
 && 
\mathbb{E}^2
 \ar[dr]
  && 
  \widehat{\mathbb{M}}^2
  \ar[dl]
   && 
   \mathbb{M}^2
   \ar[dlll]
   \\
&&&\Hs^2 &&&
\end{tikzcd}
\end{figure}

\noindent
The diagram above depicts the limits of orthogonal geometries in $\GL(3;\R)$, as previously calculated in Section \ref{sec:Orthog_Low_Dims}.  Spherical, hyperbolic and (anti)-de Sitter geometry collectively degenerate to the Euclidean \& Minkowski plane, as well as their contragredient duals.
All of these in turn degenerate to $\Hs^2$, the Heisenberg plane.

\begin{definition}
Heisenberg geometry is the $(G,X)$ geometry $\Hs^2:=(\Heis,\A^2)$ where
$$\Heis=\left\{\pmat{\pm1 & a & c\\0 &\pm1& b\\0&0&1}\;\Bigg | \; a,b,c\in\R\right\} \hspace{0.2cm}\textrm{and}
\hspace{0.2cm}
\A^2=\left\{ [x:y:1]\in\RP^2\mid x,y,\in\R\right\}.$$
The identity component $\Heis_0<\Heis$ is the real Heisenberg group, and the index 2 subgroup of orientation-preserving transformations is denoted $\Heis_+$.
\end{definition}

\noindent
Heisenberg geometry is a geometry on the plane given by all translations together with shears parallel to a fixed line.
Viewing this fixed line as `space' and any line intersecting it transversely as `time,' this is the geometry of $1+1$ dimensional Galilean relativity.
This chapter provides a detailed exploration of Heisenberg geometry, to add to the literature describing explicit geometric transitions.
We pay pay particular attention to aspects of interest to geometric topology; namely classifying Heisenberg orbifolds, calculating deformation their spaces and constructing regenerations of Heisenberg structures into familiar geometries.

\section{Heisenberg Geometry}
\label{sec:Heis_Geo}

The Heisenberg plane is not a metric geometry but supports other familiar geometric quantities.
The standard area form $dA=dx\wedge dy$ on $\R^2$ is invariant under the action of $\Heis_+$, furnishing $\Hs^2$ with a well-defined notion of area.
The one form $dy$ is $\Heis_0$ invariant, and induces a $\Heis$-invariant foliation of $\Hs^2$ by horizontal lines together with a transverse measure.  
As a subgeometry of the affine plane, $\Hs^2$ inherits an affine connection and notion of geodesic. A curve $\gamma$ is a geodesic if $\gamma''=0$, tracing out a constant speed straight line in $\Hs^2$.

Heisenberg geometry arises as a limit of the constant curvature spaces $\S^2,\Hyp^2$ and $\E^2$ by `zooming into while unequally stretching' a projective model.
Details can be reconstructed from \cite{CooperDW14}.
Here we briefly explore one degeneration of hyperbolic space to the Heisenberg plane as subgeometries of $\RP^2$.
Acting on $\Hyp^2\in\mathfrak{S}_{\RP^2}$ by the path $A_t=\diag(t^2,t,1)$ results in a path of subgeometries $A_t\Hyp^2$ isomorphic to the hyperbolic plane with underlying space the origin-centered ellipsoid in $\A^2$ with semimajor,semiminor axes of lengths $t^2,t$ parallel to the $x,y$ axes respectively.
The limit of these domains is $\A^2$ and the groups $A_t\O(2,1)A_t\inv$ limit to $\Heis$.
The aforementioned invariant foliation on $\Hs^2$ is a remnant of this stretching, and is parallel to the limiting direction of the major axes of $A_t\Hyp^2$.

\begin{figure}
\includegraphics[width=\textwidth]{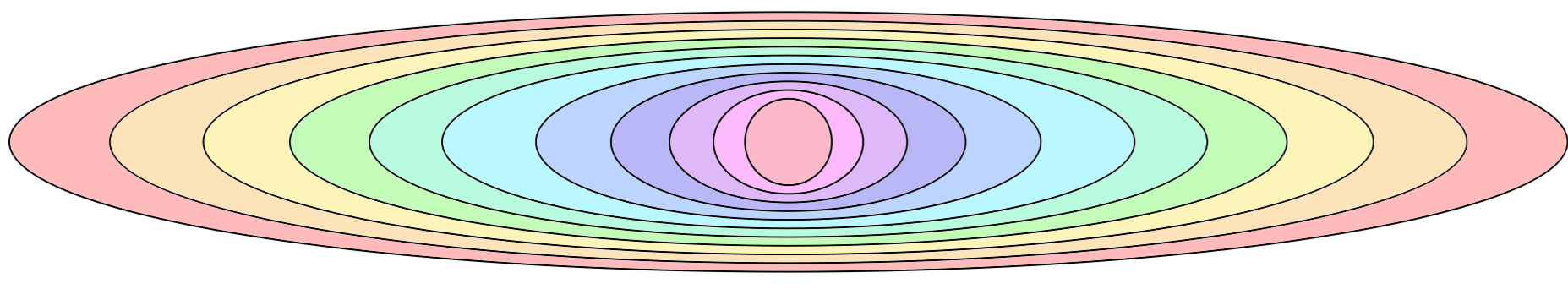}
\caption{The transition of $\Hyp^2$ to $\Hs^2$ as a conjugacy limit via the action of $A_t=\diag(t^2,t,1)$.}
\end{figure}

\noindent 
Unlike the degeneration of $\S^2$ and $\Hyp^2$ to Euclidean space, the uneven stretching required to produce a Heisenberg limit distorts even the point stabilizer subgroups, which become noncompact in the limit.
Conjugation by $A_t$ stretches the circle $S=\smat{\SO(2)&0\\0&1}\subset\mathsf{M}(3;\R)$ into ellipses of increasing eccentricity limiting to the parallel lines $\smat{1&\pm x \\0&1}$ in the upper $2\times 2$ block.
As a consequence, role of the unit tangent bundle in the constant curvature geometries is replaced for the Heisenberg plane by an appropriate space of based lines.
Indeed let $\mathcal{L}=\mathbb{P}\mathsf{T}(\Hs^2)$ be the space of pointed lines in the Heisenberg plane, and $\mathcal{H}\subset\mathcal{L}$ those belonging to the invariant horizontal foliation.  
The action of $\Heis_0$ on the plane extends to a simple transitive action on $\mathcal{L}\smallsetminus\mathcal{H}$, analogous to the action of $\Isom(\mathbb{X})$ on the unit tangent bundle $\mathsf{UT}(\mathbb{X})$ for $\mathbb{X}\in\{\Hyp^2,\E^2,\S^2\}$.
The noncompactness of point stabilizers is sufficient to preclude  an invariant Riemannian metric, but moreover the existence of shears in the automorphism group of $\Heis$ forces any continuous $\Heis$-invariant map $d:\R^2\times\R^2\to\R$ to be constant along the lines $\{x\}\times\R$ in both factors of the domain, so there are no continuous $\Heis$-invariant distance functions at all.

\subsection{Heisenberg Structures on Orbifolds}
\index{Orbifolds!Heisenberg Structures}
\index{Heisenberg Geometry!Orbifolds}

As a subgeometry of the affine plane, every Heisenberg structure on an orbifold $\mathcal{O}$ canonically weakens to an affine structure.
This provides strong restrictions on which orbifolds can possibly admit Heisenberg structures, it follows from a result of Benzecri that closed affine orbifolds have Euler characteristic zero \cite{Benz}; an additional self contained proof appears in \cite{Baues14}.
The deformation space of affine tori has been computed \cite{Baues14}, and weakening Heisenberg structures to affine structures provides a (non-injective) map $\omega\colon\mathcal{D}_{\Hs^2}(T^2)\to\mathcal{D}_{\A^2}(T^2)$.
Each Heisenberg orbifold inherits an area form from $\Hs^2$ and has a well defined finite total area.  The group $\R_+$ of homotheties of the plane acts on $\mathcal{D}_{\Hs^2}(\mathcal{O})$ sending an orbifold $\mathcal{O}$ with total area $\alpha$ to an orbifold $r.\mathcal{O}$ with area $r^2\alpha$, allowing
 the deformation space to be easily recovered from the space of unit area structures.
 
 \begin{observation}
 \label{obs:Area}
 The action of $\R_+$ by homotheties on the plane induces an action on $\mathcal{D}_{\Hs^2}(\mathcal{O})$ defined by $r.[f,\rho]=[rf,r\rho]$.  
 This gives a homeomorphism $\mathcal{D}_{\Hs^2}(\mathcal{O})=\R_+\times\mathcal{T}_{\Hs^2}(\mathcal{O})$ for $\mathcal{T}_{\Hs^2}(\mathcal{O})$ the subspace of unit area structures, analogous to the Techim\"uller space for Euclidean tori.
 \end{observation}

\noindent 
As $dy$ is invariant under the action of $\Heis_0$, any Heisenberg surface with holonomy into $\Heis_0$ inherits a closed nondegenerate 1-form and corresponding foliation.  
This observation leads to a self-contained proof that every Heisenberg orbifold has vanishing Euler characteristic, simple enough that we include it for completeness.

\begin{proposition}
Every closed Heisenberg orbifold is finitely covered by a torus with holonomy in $\Heis_0$.
\end{proposition}
\begin{proof}
Let $\mathcal{O}$ be a Heisenberg orbifold, with developing map $f\colon\widetilde{\mathcal{O}}\to\Hs^2$ and holonomy $\rho\colon\pi_1(\mathcal{O})\to\Heis$.
As $f$ immerses $\widetilde{\mathcal{O}}$ in the plane it has no singular locus; thus $\widetilde{\mathcal{O}}$ a manifold and $\mathcal{O}$ is good.
By the classification of two dimensional orbifolds then, $\mathcal{O}$ is not the spindle or teardrop, and is finitely covered by some surface $\Sigma\to\mathcal{O}$.
The Heisenberg structure on $\mathcal{O}$ pulls back to $\Sigma$ with developing pair $(f,\rho|_{\pi_1(\Sigma)})$.
Passing to an at most 4-sheeted cover, we may assume the holonomy of $\Sigma$ takes values in $\Heis_0$.
Thus $\Sigma$ inherits a nondegenerate 1-form $\omega\in\Omega^1(\Sigma)$ from $dy$ on $\Hs^2$.  Choose a Riemannian metric $g$ on $\Sigma$.  
Then $\omega$ defines a non-vanishing vector field $X_\omega$ by $\omega(\cdot)=g(X_\omega,\cdot)$, and so $\chi(\Sigma)=0$.  As $\Heis_0$ acts by orientation preserving transformations, $\Sigma$ is a torus.
\end{proof}

\noindent 
Thus Heisenberg tori with holonomy in $\Heis_0$ play a fundamental role to the classification of Heisenberg orbifolds, and it is natural to study them first.
By the previous observation, in particular it suffices to study the Teichm\"uller space of unit area structures, whose holonomy are determined up to conjugacy and homotheties of the plane.

\section{The Deformation Space of Tori}
\label{sec:Def_Space_Tori}
\index{Heisenberg Geometry!Deformation Space}

\subsection{The Representation Variety $\Hom(\Z^2,\Heis_0)$}
\label{subsec:Rep_Var_Tori}
\index{Representation Varieties!Heisenberg}
\index{Heisenberg Geometry!Representation Variety}

To classify tori with holonomy into $\Heis_0$ we compute the representation variety $\mathcal{R}=\Hom(\Z^2,\Heis_0)$.
The quotients of $\mathcal{R}$ by homothety and Heisenberg conjugacy are denoted $\mathcal{H}=\mathcal{R}/\R_+$ and $\mathcal{X}=\mathcal{R}/\Heis_0$ respectively.
The holonomies of unit area structures lie in the double quotient $\mathcal{U}=\mathcal{X}/\R_+\cong\mathcal{H}/\Heis_0$.
Representations into the center of $\Heis_0$ act by collinear translations on $\Hs^2$, and a simple argument of section 3.3 precludes these from being the holonomy of any Heisenberg structure.
Thus, we are primarily concerned with the subset $\mathcal{R}^\star\subset\mathcal{R}$ of representations not into the center, and its quotients $\mathcal{X}^\star\subset\mathcal{X},\mathcal{H}^\star\subset\mathcal{H}$ and $\mathcal{U}^\star\subset\mathcal{U}$.
Explicitly dealing with these representation spaces is easiest using coordinates from the Lie algebra, introduced below.

\begin{proposition}
The map $\log\colon \Heis_0\to\heis$ induces an isomorphism of varieties
$\Hom(\Z^2,\Heis_0)\cong\Hom(\R^2,\heis)$.
\end{proposition}
\begin{proof}
Inclusion in $\mathsf{M}(3;\R)$ equips $\Heis_0$ and $\heis$ with the structure of algebraic varieties.
As $\heis$ is nilpotent, the exponential $\exp\colon\heis\to\Heis_0$ is algebraic, and in fact isomorphism of varieties with inverse $\log\colon\Heis_0\to\heis$.
Recall that evaluation on the generators $e_1,e_2\in\Z^2\subset\R^2$ identifies the collections of representations with subvarieties of $\Heis_0\times\Heis_0$, $\heis\times\heis$ respectively.
Applying the exponential/logarithm coordinatewise provides the required algebraic isomorphism $\Hom(\Z^2,\Heis_0)\cong\Hom(\R^2,\heis)$.

\begin{center}
\begin{tikzcd}
\mathsf{Hom}(\mathbb{Z}^2,\mathsf{Heis}_0) \arrow[r, "\mathsf{ev}"] \arrow[d,xshift=0.7ex,"\log"]
& \mathsf{Heis}_0\times\mathsf{Heis}_0 \arrow[d, xshift=0.7ex, "\log\times\log"] \\
\mathsf{Hom}(\mathbb{R}^2,\mathfrak{heis}) \arrow[r,  "\mathsf{ev}"] \arrow[u, xshift=-0.7ex, "\exp"]
& \mathfrak{heis}\times\mathfrak{heis} \arrow[u,xshift=-0.7ex,"\exp\times\exp"]
\end{tikzcd}	
\end{center}

\end{proof}

\noindent 
We continue to denote the induced isomorphisms $\mathcal{R}\cong\Hom(\R^2,\heis)$ by $\exp$ and $\log$,
and call the vector $(\vec{x},\vec{y},\vec{z})\in\R^6$ the \emph{Lie algebra coordinates} for the representation $\rho\in\mathcal{R}$ 
when $\mathsf{ev}(\log\rho)=\left(\smat{x_1 & z_1\\&y_1},\smat{x_2 & z_2\\&y_2}\right)$.

\begin{proposition}
$\mathcal{R}$ is isomorphic to $V(x_1y_2-x_2y_1)\times\R^2$. 
\end{proposition}
\begin{proof}
Evaluation on the generators identifies the representation variety $\Hom(\R^2,\heis)$ with the kernel of the Lie bracket $[\cdot,\cdot]\colon\heis^2\to\heis$.	
Indeed $\left[\smat{x_1&z_1\\&y_1},\smat{x_2&z_2\\&y_2}\right]=\smat{0 & x_1y_2-x_2y_1\\&0}$,
so $\ker[\cdot,\cdot]$ is cut out precisely by $x_1y_2=x_2y_1$ in $\heis^2$ and $(\vec{x},\vec{y},\vec{z})\in\R^6$ is the Lie algebra coordinates of a representation $\rho\in\mathcal{R}$ if and only if $(\vec{x},\vec{y})\in V(x_1y_2-x_2y_1)$ and $(z_1,z_2)\in\R^2$.
\end{proof}

\begin{proposition}
The homothety quotient $\mathcal{H}^\star$ of representations not into the center of $\Heis$ is homeomorphic to $\R^2\times T^2$.	
\end{proposition}
\begin{proof}
Denote by $\R^2_{(\vec{x},\vec{y})}$ the $\R^2=\{(z_1,z_2)\}$ fiber above $(\vec{x},\vec{y})$.
The hypersurface $V=V(x_1y_2-x_2y_1)$ has one singularity at $0$, above which $\R^2_{(0,0)}$ consists of the representations into the center.
Homotheties of $\Hs^2$ induce the $\R_+$ action $t.(\vec{x},\vec{y},\vec{z})=(t\vec{x},t\vec{y},t\vec{z})$ on $\mathcal{R}$; thus $V\subset\R^4$ is a cone and $\mathcal{H}^\star$ identifies with the product of $\R^2$ with the intersection $V\cap\S^3$.
The change of coordinates on $\R^4$ given by $(x_1,x_2,y_1,y_2)=(u_1+v_1,v_2+u_2,v_2-u_2,u_1-v_1)$ provides an isomorphism $V\cong V(u_1^2+u_2^2-v_1^2-v_2^2)$ 
identifying $V\cap \S^3$ with the Clifford torus $\|\vec{u}\|=\|\vec{v}\|=1/\sqrt{2}$, so $V^\star=V\smallsetminus\vec{0}\cong \R_+\times T^2$.	
\end{proof}

\begin{corollary}
The section of $\mathcal{R}^\star\to\mathcal{H}^\star$ sending each homothety class $[\rho]_{\R_+}=[(\vec{x},\vec{y},\vec{z})]_{\R_+}$ to the representative with $(\vec{x},\vec{y})\in T^2\subset \S^3$ gives an identification of $\mathcal{H}^\ast$ with the algebraic variety
$\mathcal{H}^\star=V(x_2y_1-x_1y_2,\|x\|^2+\|y\|^2-1)\subset\R^6$.
\end{corollary}

\begin{proposition}
The conjugacy quotient $\mathcal{X}^\star$ is a line bundle over $V^\star\cong\R_+\times T^2$ twisted above each generator of $\pi_1(V^\star)$.
\end{proposition}
\begin{proof}

A computation reveals the conjugation action of $\Heis_0$ on $\mathcal{R}$ in Lie algebra coordinates is expressed
$\smat{1& g&k\\&1&h\\&&1}.(\vec{x},\vec{y},\vec{z})=(\vec{x},\vec{y},\vec{z}+g\vec{y}-h\vec{x})$.
Thus $\Heis_0$ acts trivially on the first factor of $\mathcal{R}=V\times\R^2$ and the orbit of a point $\vec{z}\in\R^2_{(\vec{x},\vec{y})}$ 
is the coset of $\mathsf{span}\{\vec{x},\vec{y}\}\subset\R^2_{(\vec{x},\vec{y})}$ containing it.
In the subset $\mathcal{R}^\star$ at least one of $\vec{x},\vec{y}$ is nonzero, and the condition that $(\vec{x},\vec{y})\in V(x_1y_2-x_2y_1)=V(\det\smat{x_1&y_1\\x_2&y_2})$ implies $\vec{x}$ and $\vec{y}$ are linearly dependent.
It follows that the $\Heis_0$ orbits on $\mathcal{R}^\star$ are lines, foliating each $\R^2_{(\vec{x},\vec{y})}$ over $V^\star$ and the leaf space is a line bundle over $V^\star$.

Equipping each $\R^2_{(\vec{x},\vec{y})}$ with the standard inner Euclidean inner product, 
a canonical choice of representatives for cosets of $\ell_{(\vec{x},\vec{y})}=\mathsf{span}\{\vec{x},\vec{y}\}$ is given by the orthogonal line $\ell_{(\vec{x},\vec{y})}^\perp\subset\R^2_{(\vec{x},\vec{y})}$.
This defines a section $\mathcal{X}^\star\to\mathcal{R}^\star$ sending a conjugacy class $[\rho]_{\Heis_0}=[(\vec{x},\vec{y},\vec{z})]_{\Heis_0}$ to its representation with $\vec{z}$-coordinate on $\ell_{(\vec{x},\vec{y})}^\perp$, and identifies
 $\mathcal{X}^\star=\{(\vec{x},\vec{y},\vec{z})\mid (\vec{x},\vec{y})\in V^\star, \; \vec{z}\in\ell_{(\vec{x},\vec{y})}^\perp\}$ with a subbundle of $V^\star\times\R^2\to V^\star$.

Line bundles over $V^\star\cong\R_+\times T^2$ are in bijection with $H^1(T^2,\Z_2)\cong\Z_2^2$, determined up to isomorphism by whether pulling back along generators of $\pi_1(T)^2$ gives cylinders or M\"obius bands.
A convenient choice of generators in the $(\vec{u},\vec{v})$ coordinates introduced above are $\alpha(\theta)=(\vec{e_1},\vec{p_\theta})$ and $\beta(\theta)=(\vec{p_{\theta}},\vec{e_1})$ for $e_1=\smat{1\\0}$ and $\vec{p_\theta}=\smat{\cos\theta\\sin\theta}$.
An explicit computation using the description of $\mathcal{X}^\star$ above shows the bundle restricts to a M\"obius band above each of $\alpha,\beta$, so $\mathcal{X}^\star$ is the  line bundle over $\R_+\times T^2$ represented by $(1,1)\in H^1(T^2,\Z_2)$.
\end{proof}

\noindent 
The choice of explicit sections has identified $\mathcal{H}^\star$ and $\mathcal{X}^\star$ with subsets of $\mathcal{R}$.
The space of interest $\mathcal{U}^\star$ identifies with their intersection, $\mathcal{X}^\star\cap\mathcal{H}^\star$, which is the restriction of $\mathcal{X}^\star\to V^\star$ to the base $T^2\subset\S^3$.

\begin{corollary}
\label{cor:U}
The quotient $\mathcal{U}^\star$ by homothety and conjugacy is the doubly twisted line bundle over $T^2$, realized as the subvariety of $\mathcal{U}^\star\subset\R^6$ consisting of triples of vectors $(\vec{x},\vec{y},\vec{z})$ such that $\vec{x}$ and $\vec{y}$ are collinear, and $\vec{z}$ is orthogonal to their span.
$$\mathcal{U}^\star=
V\left(\mat{
\|x\|^2+\|y\|^2=1,& \vec{z}\cdot\vec{x}=0\\
x_1y_2-x_2y_1=0,& \vec{z}\cdot\vec{y}=0
}\right)\subset\R^6$$
\end{corollary}

\noindent
The developing pair of a Heisenberg torus is only well defined up to orientation preserving transformations, so potential holonomies lie in the space $\mathcal{R}/\Heis_+$, a twofold quotient of $\mathcal{U}^\star$ computed here.
We will deal with this $\Z_2=\Heis_+/\Heis_0$ ambiguity after determining which points of $\mathcal{U}^\star$ are in fact holonomies.


\subsection{The space $\mathcal{D}_{\Hs^2}(T^2)$.}

As a warm-up to computing the deformation space of Heisenberg tori, we review the analogous problem for Euclidean and affine structures.
Euclidean tori are complete metric spaces, and so are determined by their holonomy, which is necessarily discrete and faithful (for instance, by Thurston's book \cite{Thurston80}, Proposition 3.4.10). 
Discrete subgroups $\Z^2<\Isom(\E^2)$ act by translations,
thus the deformation space of Euclidean tori identifies with the $\Isom(\E^2)$-conjugacy classes of marked planar lattices, $\mathcal{D}_{\E^2}(T^2)\cong \GL(2;\R)/\O(2)$.
The unit area structures parameterized by the familiar Teichm\"uller space $\Hyp^2=\SL(2;\R)/\SO(2)$.

The affine plane admits no invariant metric, which complicates the story significantly.
Complete affine structures have universal cover affinely diffeomorphic to $\A^2$, but in contrast to the Euclidean case incomplete structures abound.
The work of Baues \cite{Baues14} provides a remarkably comprehensive description of the classification of affine tori, in particular containing the following classification theorem.

\begin{theorem}[\cite{Baues14}, Theorem 5.1]
The universal cover of an affine torus is affinely diffeomorphic to one of the following spaces: the affine plane $\A^2$, the half plane $\mathcal{H}=\{(x,y)\mid y>0\}$, the quarter plane $\mathcal{Q}=\{(x,y)\in\A^2\mid x,y>0\}$ or the universal cover of the punctured plane $\mathcal{P}=\widetilde{\A^2\smallsetminus 0}$.
Furthermore the developing maps of affine structures are covering projections onto their images.	
\end{theorem}

\noindent 
As $\Hs^2$ admits no invariant metric, we must be prepared for complications similar to the affine case.
Such difficulties do not materialize however, as canonically weakening Heisenberg structures to affine ones, we may use the classification above to show all Heisenberg tori are complete.

\begin{corollary}
All Heisenberg structures on the torus are complete.
\label{cor:Complete}
\end{corollary}
\begin{proof}
Let $(f,\rho)$ be the developing pair for a Heisenberg torus $T$, considered as an affine structure.
If $T$ is not complete, there is an affine transformation $A$ with $A.f(\widetilde{T})\in\{\mathcal{H},\mathcal{Q},\A^2\smallsetminus 0\}$ and holonomy $A\rho A\inv$ preserving this developing image.
But by the classification of affine tori, holonomies of these tori contain elements of $\det\neq 1$, whereas $\Heis$ is unipotent so $\det A\rho(\Z^2)A\inv=\{1\}$.
Thus $T$ is in fact complete, with developing map a diffeomorphism $f\colon \widetilde{T}\to \A^2$.
\end{proof}

\subsection{Constructing Developing Maps}
\index{Heisenberg Geometry!Developing Maps}
\index{Deformation Space!Heisenberg Tori}

Here we pursue a self-contained computation the deformation space $\mathcal{D}_{\Hs^2}(T^2)$, using the understanding of representations $\Z^2\to \Heis_0$ up to conjugacy developed in section 3.1.
Specifically, for $\rho\in\Hom(\Z^2,\Heis)$ we either construct a corresponding developing map $f$ giving a Heisenberg structure $(f,\rho)$ on $T^2$ (and prove its uniqueness), or we show no developing map for $\rho$ can exist.

A developing map for $\rho\colon\Z^2\to\Heis$ is a $\rho$-equivariant immersion $f\colon\R^2\to\Hs^2$.
A natural $\rho$-equivariant self map of the plane can be constructed directly from $\rho$, relying on the fact that
each representation of $\Z^2$ extends uniquely to a representation $\hat{\rho}\colon\R^2\to\Heis_0$ via $\widehat{\rho}(x,y)=\rho(e_1)^x\rho(e_2)^y$.
The orbit map $f_\rho\colon \R^2\to \Hs^2$ defined by $(x,y)\mapsto\widehat{\rho}(x,y).\vec{0}$ for this extended representation is $\rho$-equivariant, and thus a developing map for a Heisenberg structure when it is an immersion.
As the following two propositions show, this construction actually produces developing maps for all complete Heisenberg tori (and thus by Corollary \ref{cor:Complete} for all Heisenberg tori, although with the aim of producing a self-contained proof we do not presume that here).

\begin{proposition}
Let $\mathcal{F}\subset\mathcal{U}$ be the subset of representations $\rho$ with extensions $\widehat{\rho}$ acting freely on $\Hs^2$.
Then each $\rho\in\mathcal{F}$ determines a unique Heisenberg structure on $T^2$, which is complete, and all complete structures with holonomy in $\Heis_0$ arise this way.
\end{proposition}
\begin{proof}
If $\widehat{\rho}$ acts freely, the orbit map $f_\rho\colon\R^2\to\Hs^2$ is injective, and a computation reveals $(df_\rho)_0\colon T_0\R^2\to T_0\Hs^2$ is injective.
Furthermore $(df_\rho)_x=\widehat{\rho}(x).(df_\rho)_0$ so $f_\rho$ is an immersion of $\R^2$ and $(f_\rho,\rho)$ is a developing pair for a Heisenberg torus.
Similarly, the other orbit maps $\vec{u}\mapsto\widehat{\rho}(\vec{u}).q$ are immersions (thus open maps) for any $q\in\Hs^2$, and distinct $\widehat{\rho}(\R^2)$ orbits partition $\Hs^2$ into a disjoint union of open sets.
By connectedness then $f_\rho$ is onto, hence a diffeomorphism so the corresponding Heisenberg structure is complete.

Alternatively, let $\rho\colon\Z^2\to\Heis_0$ be the holonomy of a complete torus, but assume $\widehat{\rho}\colon\R^2\to\Heis_0$ fails to act freely.
Then some element, and hence some 1-parameter subgroup $L<\R^2$, fixes a point under the action induced by $\widehat{\rho}$.  
This line $L$ intersects $\Z^2$ only in $\vec{0}$ (as $\rho$ acts freely by completeness); and so is dense in the quotient $\R^2/\Z^2$.  
Thus there are sequences $\vec{v}_n\in\Z^2$ with $\rho(v_n)$ coming arbitrarily close to stabilizing a point, and $\widehat{\rho}$ does not act properly discontinuously, contradicting completeness.

Finally, let $(f,\rho)$ be a complete structure and $(\phi,\rho)$ another structure with the same holonomy.  
Then $f\inv \phi:\widetilde{T}\to\widetilde{T}$ is $\pi_1(T)$-equivariant and descends to a diffeomorphism $\psi:T\to T$.  
But $\psi_\ast$ is the identity on fundamental groups and as the torus is a $K(\pi,1)$, $\psi$ is isotopic to the identity.
Thus $(f,\rho)$ and $(\phi,\rho)$ are developing pairs for the same Heisenberg structure.
\end{proof}

\noindent 
Constructing developing maps from the extensions $\widehat{\rho}$ provides endows these tori with the structure of a commutative group via the identification $\widehat{\rho}(\R^2)/\rho(\Z^2)\cong f_\rho(\R^2)/\rho(\Z^2)$.
The existence of this group structure can more generally be deduced from the similar observation of Baues and Goldman concerning affine structures \cite{BauesGoldman05}.

\begin{corollary}
Complete Heisenberg tori are the group objects in the category of Heisenberg manifolds, analogous to elliptic curves in the category of Riemann surfaces.	
\end{corollary}

\begin{proposition}
The subset $\mathcal{F}\subset\mathcal{U}$ of conjugacy classes with freely acting extensions $\widehat{\rho}\colon\R^2\to\Heis_0$ is a trivial $\R^\times$ bundle over the cylinder $\mathsf{Cyl}=T^2\smallsetminus S$, for $S$ the circle defined by the intersection of  $T^2=V(x_1y_2-x_2y_1)\cap\S^3$ with the plane $V(y_1,y_2)$.
\end{proposition}
\begin{proof}
A representation $\widehat{\rho}\in\mathcal{U}$ is faithful if and only if the logarithm of its generators $\smat{x_1&z_1\\&y_1}$ and $\smat{x_2 &z_2\\&y_2}$ are linearly independent in $\heis$.
In Lie algebra coordinates, linearly dependent elements of $\heis^2$ form the variety $\mathsf{Rk}_1\subset\mathsf{M}_{3\times 2}(\R)$ of rank one matrices $(\vec{x},\vec{y},\vec{z})=\smat{x_1&y_1&z_1\\x_2&y_2&z_2}$.
There are no faithful $\R^2$ representations into the 1-dimensional center of $\Heis$, so it suffices to consider the representations in $\mathcal{U}^\star$.
The intersection $\mathcal{U}^\star\cap\mathsf{Rk}_1$ is a torus, coming from the $\S^1$ factor and a great circle in $\S^2\times\S^1$ also described as the zero section of the bundle $\mathcal{U}^\star\to T^2$.
which is easily seen from the coordinate description.
The rank one variety is cut out by the $2\times 2$ minors $\mathsf{Rk}_1=V(x_1y_2-x_2y_1,x_1z_2-x_2z_1,y_1z_2-y_2z_1)$ and thus consists of triples of simultaneously collinear vectors $\vec{x}\parallel\vec{y}\parallel\vec{z}\in\R^2$.
Recalling \ref{cor:U}, points $(\vec{x},\vec{y},\vec{z})$ of $\mathcal{U}^\star$ satisfy $\vec{x}\parallel\vec{y}$ and $\vec{z}$ perpendicular to their span.
Thus any $(\vec{x},\vec{y},\vec{z})\in\mathcal{U}^\star\cap\mathsf{Rk}_1$ necessarily has $\vec{z}=0$, so the intersection $\mathcal{U}^\star\cap\mathsf{Rk}_1$ is the torus
$(\vec{x},\vec{y},0)\subset\mathcal{X}^\star$.
 The conjugacy classes of faithful representations constitute the complement of this zero section of $\mathcal{U}^\star\to T^2$.

A non-identity element of $\Heis_0$ stabilizes a point of $\Hs^2$ if and only if it acts trivially on the leaf space of the invariant foliation and has nontrivial shear.  In Lie algebra coordinates this forms the set $\mathcal{S}=\left\{\smat{x&z\\&0}\mid x\neq 0\right\}\subset\heis$. 
The extension $\widehat{\rho}$ acts freely if and only if in Lie algebra coordinates, each generator misses $\mathcal{S}$.
All faithful representations $(\vec{x},\vec{y},\vec{z})$ with $y_1,y_2\neq 0$ act freely, and all with $\vec{y}=0$ fail to.
If $\vec{y}=(0,y_2)$ then $\rho\in\mathcal{R}$ implies $x_1=0$ so $\rho$ acts freely, and similarly for $\vec{y}=(y_1,0)$.
Thus faithful representations fail to act freely if and only if $\vec{y}=0$, and the space of freely acting representations is $\mathcal{F}=\mathcal{U}^\star\smallsetminus V(z_1,z_2)\cup V(y_1,y_2)$.

The intersection $S=T^2\cap V(y_1,y_2)$ is a $(1,1)$ curve with respect to the $(\vec{u},\vec{v})$ coordinates, and $\mathcal{U}^\star\smallsetminus V(y_1,y_2)$ is an $\R$-bundle over $\mathsf{Cyl}=T^2\smallsetminus S$.
This bundle is trivial as the generator of $\pi_1(\mathsf{Cyl})$ is parallel to $V(y_1,y_2)$ and the restriction the doubly twisted bundle $\mathcal{X}$ to a $(1,1)$ curve in the base is a cylinder.
The subvariety $V(z_1,z_2)$ is the zero section of this bundle, thus its complement is the trivial $\R^\times$ bundle over $\mathsf{Cyl}$.
\end{proof}

\noindent 
This classification gives a simple, self contained argument that no incomplete structures exist.
An incomplete structure must have holonomy in $\mathcal{U}\smallsetminus\mathcal{F}$, but geometric reasons preclude these from being the holonomy of Heisenberg tori.
This completes the classification of tori with $\Heis_0$ holonomy, and a quick observation implies there can be no others.

\begin{proposition}
Representations $\rho\in\mathcal{U}\smallsetminus \mathcal{F}$ are not the holonomy of any Heisenberg torus.
Consequently all Heisenberg tori are complete, with holonomy into $\Heis_0$.
\end{proposition}
\begin{proof}
There are three classes of elements in $\mathcal{U}\smallsetminus \mathcal{F}$: representations into the center, representations $(\vec{x},\vec{y},\vec{z})$ with $\vec{z}=0$	 and representations with $\vec{y}=0$.
These classes are all topologically conjugate, and preserve a fibration of the plane $\Hs^2\surject \R$.
Representations into the center act by translations parallel to the $x$ axis, preserving the invariant foliation of $\Hs^2$, and similarly for those with $\vec{y}=0$.
Representations with $\vec{z}=0$ are not faithful, and factor through a representation $\R\to\Heis$ with orbits foliating the plane by parabolas.

To see these cannot be the holonomy of tori, let
$\rho\in\mathcal{U}\smallsetminus\mathcal{F}$ preserve the fibration $\pi\colon\Hs^2\surject\R$, and assume $(f,\rho)$ is a developing pair for some Heisenberg torus.
Let $\Omega=f(\widetilde{T})$ be the developing image, and note $\pi(\Omega)\subset \R$ is open as $f$ is a local diffeomorphism and $\pi$ is a bundle projection.
Let $Q\subset\widetilde{T}$ be a compact fundamental domain for the action of $\Z^2$ by covering transformations, and note that $\pi(f(Q))=\pi(f(\Omega))$ as $\rho$ is fiber preserving.
But $\pi(f(Q))$ is compact, and thus not open in $\R$, a contradiction.

It follows from this that all Heisenberg tori are complete, and have holonomy in $\Heis_0$.
Indeed $T$ be any Heisenberg torus with developing pair $(f,\rho)$ and $\widetilde{T}\to T$ the cover corresponding to the subgroup $\rho(\Z^2)\cap\Heis_0$.  
Then $\widetilde{T}$ is complete so $T$ is also, and $\rho(\Z^2)$ acts freely and properly discontinuously on $\Hs^2$.
As $T^2$ is orientable the holonomy takes values in $\Heis_+$, but every element of $\Heis_+\smallsetminus \Heis_0$ fixes a point in $\Hs^2$ so in fact $\rho$ is $\Heis_0$ valued and $T=\widetilde{T}$.
\end{proof}

\noindent 
Thus a representation $\rho\colon\Z^2\to\Heis$ is either the holonomy of a unique complete structure on $T^2$, or is not the holonomy of any geometric structure at all.
After dealing with the slight annoyance of $\Heis_0$ vs. $\Heis_+$ conjugacy, this directly provides a description of the 
the Teichm\"uller space $\mathcal{T}_{\Hs^2}(T^2)$ of unit area structures and the corresponding deformation space $\mathcal{D}_{\Hs^2}(T^2)=\R_+\times\mathcal{T}_{\Hs^2}(T^2)$.

\begin{theorem}
The projection onto holonomy identifies the Teichm\"uller space of unit area tori with the quotient of $\mathcal{F}$ by the free $\Z_2$ action of conjugacy by $\diag(-1,-1,1)$ and $\mathcal{T}_{\Hs^2}(T^2)\cong \mathcal{F}/\Z^2\cong \R^2\times \S^1$.
\end{theorem}
\begin{proof}
The map $\hol\colon\mathsf{Dev}_{\Hs^2}(T^2)\to\mathcal{R}$ projecting a developing pair onto its holonomy is a local homeomorphism by the Ehresmann-Thurston principle, which induces a continuous map $\overline{\hol}\colon \mathcal{D}_{\Hs^2}(T^2)\to\mathcal{R}/\Heis_+$.
The work above shows the map $\dev\colon\mathcal{F}\to\mathcal{D}_{\Hs^2}(T^2)$ defined by $\rho\mapsto [f_\rho,\rho]$ is a continuous surjection onto Teichm\"uller space $\mathcal{T}_{\Hs^2}(T^2)$.
As $\mathcal{F}\subset\mathcal{U}$ was defined only up to $\Heis_0$ conjugacy, $\dev$ factors through the quotient by $(\Heis_+/\Heis_0)\cong\Z_2$ conjugacy to a continuous bijection $\overline{\dev}\colon\mathcal{F}/\Z_2\to\mathcal{T}_{\Hs^2}(T^2)$.  
The composition $\overline{\hol}\circ\overline{\dev}$ is the identity on $\mathcal{F}/\Z_2$, so $\overline{\dev}$ is a homeomorphism.

Thus, $\mathcal{T}_{\Hs^2}(T^2)\cong\mathcal{F}/\Z_2$.
The quotient $\Heis_+/\Heis_0\cong\Z_2$, generated by $\diag(-1,-1,1)$, acts by conjugation in Lie algebra coordinates as $\diag(-1,-1,1).(\vec{x},\vec{y},\vec{z})=(\vec{x},-\vec{y},-\vec{z})$.
This action is free on $\mathcal{F}$ and the quotient $\mathcal{T}_{\Hs^2}(T^2)$ is the trivial $\R_+$ bundle over $\mathsf{Cyl}$, which is homeomorphic to the open solid torus  $\R^2\times\S^1$, and $\mathcal{D}_{\Hs^2}(T^2)\cong \R^3\times\S^1$.
\end{proof}

\begin{figure}[h!]
\includegraphics[width=\textwidth]{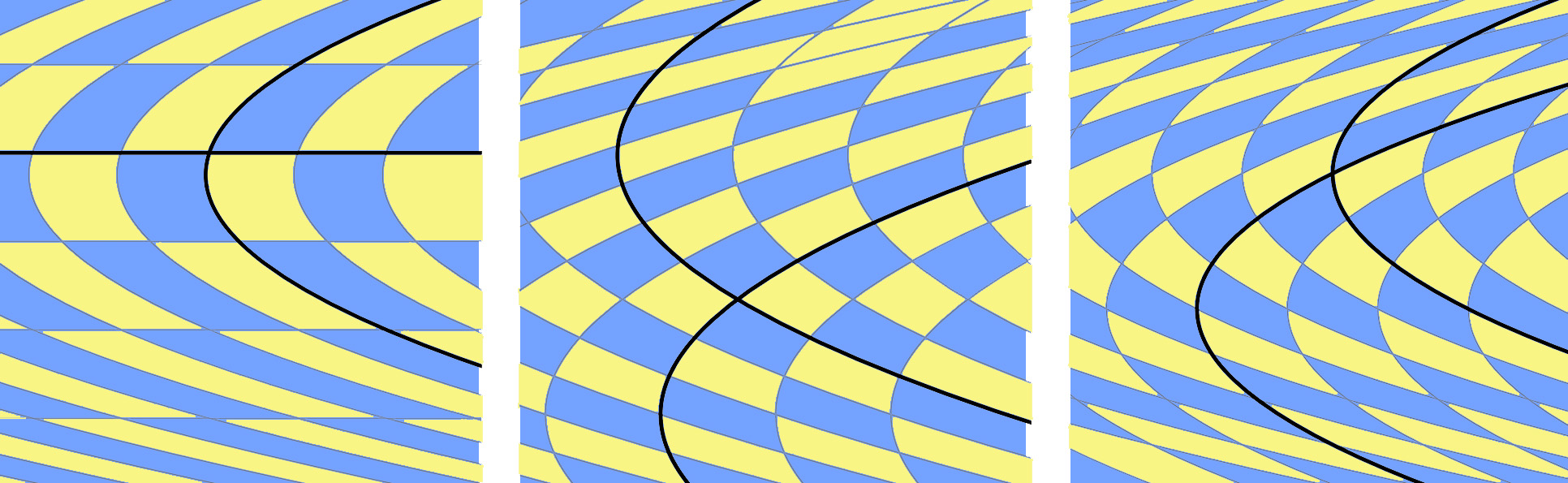}	
\caption{Some examples of developing maps for Heisenberg shear tori.}
\end{figure}

\noindent 
The identification $\mathcal{T}_{\Hs^2}(T^2)=\mathcal{F}/\Z_2$ identifies two distinct classes of Heisenberg tori; those containing a shear in their holonomy and those with holonomy into the subgroup of translations of the plane.
We will refer to these as \emph{shear tori} and \emph{translation tori} respectively.

\begin{corollary}
The space of unit-area translation tori is homeomorphic to $\R\times \S^1$, corresponding to the points of $\mathcal{F}\cap V(x_1,x_2)$.
\end{corollary}

\noindent 
It is notable that the set of developing pairs for Heisenberg translation tori is the same as the set of developing pairs for Euclidean tori, but the corresponding deformation spaces are not homeomorphic, with $\mathcal{T}_{\E^2}(T^2)$ a disk and $\mathcal{T}_{\Hs^2}(T^2)$ a cylinder.
This is due to the different notion of equivalence coming from $\Heis_+$ and $\Isom_+(\E^2)$ conjugacy; the former acting by shears and the latter by rotations.
The familiar fact that Euclidean torus has a representative holonomy containing horizontal translations is a consequence of this, as is the fact that each Heisenberg translation torus has a representative holonomy translating along (Euclidean) orthogonal lines.

\begin{figure}[h]
\includegraphics[width=\textwidth]{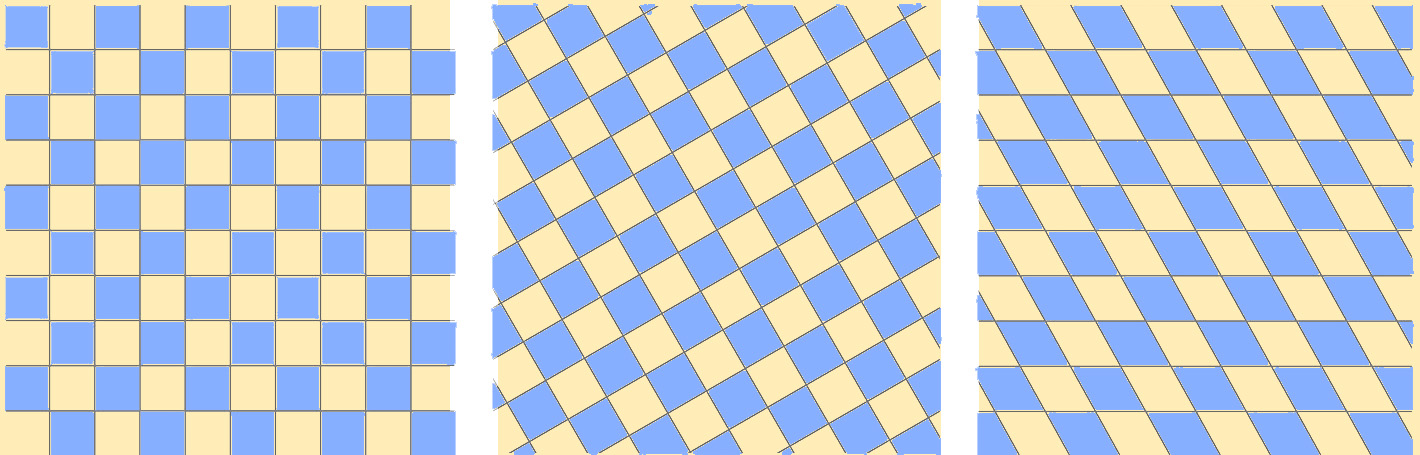}
\caption{Developing maps for translation tori.  The left two are equivalent as Euclidean structures, whereas the right two are as Heisenberg structures.
All three represent the same (unique) affine translation torus.}	
\end{figure}

\noindent 
Every Heisenberg structure canonically weakens to an affine structure, defining the map $\omega\colon\mathcal{D}_{\Hs^2}(T^2)\to\mathcal{D}_{\A^2}(T^2)$ with image in the complete structures.

\begin{corollary}
The space $\omega(\mathcal{D}_{\Hs^2}(T^2))$ of Heisenberg structures up to affine equivalence is one dimensional, homeomorphic to $\R$.
\end{corollary}
\begin{proof}
By Goldman and Baues \cite{BauesGoldman05}, the space of complete affine strutures on $T^2$ is diffeomorphic to the plane, and by completeness we identify this with its projection onto holonomy.
This realizes $\omega(\mathcal{D}_{\Hs^2}(T^2))$ as the quotient of $\mathcal{F}$ by affine conjugacy, on which the subgroups of rotations and linearly independent scalings act freely.
Thus the $\S^1$ factor and $\R_+^2$ directions of independent scalings collapse in the quotient, and $\omega(\mathcal{D}_{\Hs^2}(T^2))\cong\R$.

 \end{proof}

\section{Other Heisenberg Orbifolds}
\label{sec:Heis_Orbifolds}
\index{Heisenberg Geometry!Orbifolds}
\index{Orbifolds!Heisenberg Structures}

We may use this description of the deformation space of tori to understand all Heisenberg orbifolds. 
An orbifold covering $\pi\colon\mathcal{Q}\to\mathcal{O}$ induces a map $\pi^\ast\colon\mathcal{D}_{\Hs^2}(\mathcal{O})\to\mathcal{D}_{\Hs^2}(\mathcal{Q})$ by pullback of geometric structures,  easily expressed on developing pairs as $\pi^\ast([f,\rho])=[f,\rho|_{\pi_1(\mathcal{Q})}]$ for $\pi_1(\mathcal{Q})<\pi_1(\mathcal{O})$ the subgroup corresponding to the cover.

\begin{proposition}
All Heisenberg structures on orbifolds are complete, and projection onto the holonomy is an embedding $\mathcal{D}_{\Hs^2}(\mathcal{O})\inject \Hom(\pi_1(\mathcal{O}),\Heis)/\Heis_+$.
Under this identification, a finite sheeted covering $\mathcal{Q}\to\mathcal{O}$ describes the deformation space $\mathcal{D}_{\Hs^2}(\mathcal{O})$ 
as the preimage of $\mathcal{D}_{\Hs^2}(\mathcal{Q})$ under the restriction $\pi^\ast\colon \rho\mapsto \rho|_{\pi_1(\mathcal{Q})}$.
\label{Prop:Orbifold_Def}
\end{proposition}
\begin{proof}
Let $\mathcal{O}$ be a Heisenberg orbifold with developing pair $[f,\rho]$, and choose a finite covering $\pi\colon T\to \mathcal{O}$.  
Then by the completeness of $\pi^\ast[f,\rho]\in\mathcal{D}_{\Hs^2}(T)$, the developing map $f$ is a diffeomorphism and $\rho|_{\pi_1(T^2)}$ (hence $\rho$, as $\pi_1(T^2)$ is finite index in $\pi_1(\mathcal{O})$) acts properly discontinuously.
As $\pi_1(T^2)<\pi_1(\mathcal{O})$ is an essential subgroup for all orbifolds covered by the torus, the faithfulness of $\rho|_{\pi_1(T^2)}$ implies faithfulness of $\rho$.
Thus the structure $[f,\rho]$ on $\mathcal{O}$ is complete.
Let $[\phi,\rho]$ be another Heisenberg structure on $\mathcal{O}$ with the same holonomy, then $\phi f\inv:\widetilde{\mathcal{O}}\to\widetilde{\mathcal{O}}$ is $\pi_1(\mathcal{O})$ equivariant and descends to a Heisenberg map $\mathcal{O}\to\mathcal{O}$, inducing the identity on fundamental groups.  
Thus these structures represent the same point in deformation space so projection onto holonomy is an embedding.
\end{proof}

\noindent 
This further restricts the possible topologies of Heisenberg orbifolds.  In particular, any torsion in the fundamental group is represented faithfully by the holonomy so orbifolds may only have corner reflectors and cone points of order two.

\begin{corollary}
If $\mathcal{O}$ is a Heisenberg orbifold, necessarily $\mathcal{O}$ is $T^2$, the Klein bottle $\S^1\widetilde{\times}\S^1$, and the pillowcase $\S^2(2,2,2,2)$ or one of their quotients: the cylinder $\S^1\times I$, the Mobius band $\S^1\widetilde{\times} I$, the square $\D^2(\varnothing;2,2,2,2)$, $\D^2(2,2;\varnothing)$, $\D^2(2; 2,2)$ and $\RP^2(2,2)$, .
\end{corollary}

\noindent 
In the remainder of this section, we show that all the above admit Heisenberg structures and compute their deformation spaces.
As with tori, the deformation spaces of the remaining orbifolds can be recovered from the Teichm\"uller spaces of unit area structures by homothety, $\mathcal{D}_{\Hs^2}(\mathcal{O})\cong\R_+\times\mathcal{T}_{\Hs^2}(\mathcal{O})$.

\begin{theorem}
The orbifolds admitting Heisenberg structures and their Teichm\"uller spaces are given by the following table:
	
\begin{table}[h]
\centering
\begin{tabular}{ccc}
\toprule
$\mathcal{O}$ &\hspace{1cm}& $\mathcal{T}_{\Hs^2}(\mathcal{O})$ \\
\midrule
$\S^1\times\S^1$ &\hspace{1cm}& $\R^2\times\S^1$ 
\\
$\S^1\widetilde{\times}\S^1$, $\S^1\times I$, $\S^1\widetilde{\times}I$ &\hspace{1cm}& $\R^2\sqcup\R$  
\\
$\S^2(2,2,2,2)$ &\hspace{1cm}& $\R\times\S^1$ 
\\
$\D^2(2,2;\varnothing),\; \D^2(\varnothing; 2,2,2,2),\;\RP^2(2,2)$ &\hspace{1cm}& $\R\sqcup\R$
\\
$\D^2(2;2,2)$ &\hspace{1cm}& $\R\sqcup\R$
\\
\bottomrule
\end{tabular}
\end{table}
\end{theorem}

\noindent 
Recall that a \emph{translation torus} has holonomy acting purely by translations.
The Teichm\"uller space of translation tori is homeomorphic to $\R_+\times\S^1$, parameterized by rectangular lattices with ratio of generator lengths in $\R_+$ and angle of first vector $\theta\in\S^1$ with the horizontal.
A translation torus is called \emph{axis aligned} if the holonomy contains a translation along the invariant foliation (up to $\Heis_0$ conjugacy such a structure can actually be assumed to have holonomy generated by translations along the coordinate axes).
Within the Teichm\"uller space $\mathcal{T}_{\Hs^2}(T^2)$, the subset of axis-aligned translation tori is homeomorphic to $\R_+\sqcup\R_+$ corresponding to the points of $\mathcal{F}\cap V(x_1,\; x_2,\; y_1y_2)$.

The following figure shows all Heisenberg orbifolds, with arrows representing the finite covers used in the calculation of their deformation spaces.

\begin{figure}
\centering
\begin{tikzcd}
&& \mathsf{T}^2\ar[dr]\ar[d]\ar[dl]\arrow[ddrr, bend left] &&\\
&\mathsf{K}\ar[dl] &\mathsf{Cyl}\ar[lld]\ar[ld]\ar[d, crossing over] &\mathbb{S}^2(2,2,2,2)\ar[dr]\ar[d]\ar[dl]\ar[dll]&\\
\mathsf{Mob} &\mathbb{D}^2(\varnothing,2,2,2,2) 
&\mathbb{D}^2(2,2;\varnothing) &\mathbb{D}^2(2;2,2)&\mathbb{R}\mathsf{P}^2(2,2)
\end{tikzcd}	
\caption{All Heisenberg orbifolds are finitely covered by a Heisenberg torus, and furthermore all with cone points or corner reflectors are covered by the pillowcase $\S^2(2,2,2,2)$.}
\end{figure}
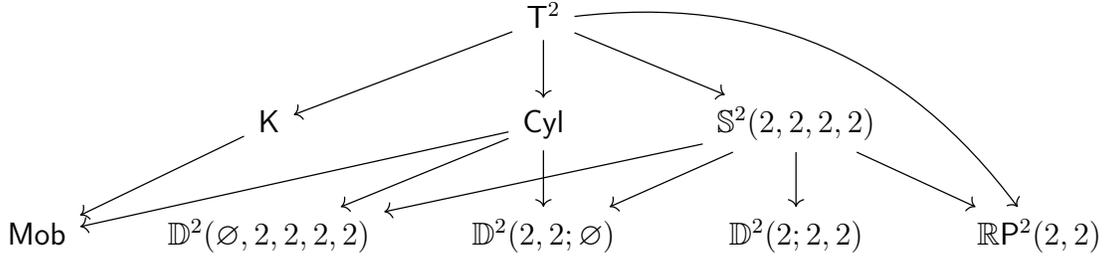

\begin{proposition}
Every Heisenberg structure on the pillowcase $P=\S^2(2,2,2,2)$ is uniquely covered by a translation torus, and so $\mathcal{T}_{\Hs^2}(P)\cong \R\times\S^1$.
\end{proposition}
\begin{proof}
The twofold branched cover $T\to \S^2(2,2,2,2)=P$ exhibits $\pi_1(P)$ as a $\Z_2=\langle r\rangle$ extension of $\pi_1(T)=\langle a,b\rangle$ with $rar=a\inv$, $rbr=b\inv$.
Thus $\mathcal{D}_{\Hs^2}(P)$ is parameterized by pairs $[\rho, R]$ for $R$ conjugating images under $\rho$ to their inverses.  
Any orientation-preserving element of order two in $\Heis$ is a $\pi$-rotation about some point $p\in\Hs^2$.  Rotations only conjugate translations to their inverses so $\rho$ is the holonomy of a translation torus.  
Given any translation torus, the $\pi$-rotation about any point in the plane provides an extension of $\rho$, and any two are conjugate by conjugacies fixing $\rho$.  Thus restriction provides a bijection from $\mathcal{D}_{\Hs^2}(\S^2(2,2,2,2))$ onto translation tori.
\end{proof}

\begin{proposition}
All Heisenberg Cylinders are quotients of an axis-aligned translation torus, or a shear torus with one generator of the holonomy a horizontal translation.  Thus $\mathcal{T}_{\Hs^2}(\mathsf{Cyl})\cong\R\sqcup \R^2$.
\end{proposition}
\begin{proof}
The doubling mirror double of a cylinder is a torus, and the corresponding orbifold cover $T\to \mathsf{Cyl}$ exhibits $\pi_1(\mathsf{Cyl})$ as a $\Z_2=\langle f\rangle$ extension of $\pi_1(T)$ with $faf=a$, $fbf=b\inv$.  
Thus $\mathcal{D}_{\Hs^2}(\mathsf{Cyl})$ is parameterized by conjugacy classes of pairs $[\rho, F]$ with $\rho\in\mathcal{D}(T)$ and $F$ satisfying the relations above with respect to $\rho(a)$, $\rho(b)$.  
For each $\rho$ with $\rho(a)$ a horizontal translation, there is a one-parameter family of solutions $F$ to the system, all conjugate via conjugacies fixing $\rho$ to a reflection across the horizontal, $\diag\{1,-1,1\}$.
 Thus there is a unique quotient corresponding to each $\rho\in\mathcal{D}_{\Hs^2}(T)$ with $\rho(a)$ a horizontal translation.
If $\rho(a)$ is not a horizontal translation, the system of equations above only has solutions when $\rho\in\mathcal{D}(T)$ is an axis aligned translation torus with $\rho(a)$ vertical, $\rho(b)$ horizontal and $F=\diag\{-1,1,1\}$. 
Thus the Teichm\"uller space consists of the union of the space of axis-aligned tori with all tori having $\rho(a)$ a horizontal translation.	The space of tori with $\rho(a)$ horizontal identifies with a slice $\R_+\times\R$ of $\mathcal{T}_{\Hs^2}(T^2)=\R_+\times\R\times\S^1$ with fixed $\theta=0\in\S^1$, intersecting the space $\R_+\sqcup\R_+$ of axis-aligned translation tori in one copy of $\R_+$. 
\end{proof}

\begin{proposition}
All Heisenberg Klein bottles are quotients of an axis-aligned translation torus, or a shear torus with one generator of the holonomy a horizontal translation.  Thus $\mathcal{T}_{\Hs^2}(\mathsf{K})\cong\R\sqcup \R^2$.
\label{Prop:KleinBottle}
\end{proposition}
\begin{proof}
The Klein bottle $K$ has orientation double cover $T\to K$ corresponding to $\pi_1(K)=\langle x,b\mid xbx\inv=b\inv\rangle$ with $\pi_1(T)=\langle x^2,b\rangle$ so $\mathcal{D}(K)$ is parameterized by pairs $[\rho,X]$ for $\rho\in\mathcal{D}_{\Hs^2}(T)$ and $X^2=\rho(a)$ satisfying $X\rho(b)X\inv\rho(b)=I$.  
As orientation reversing elements of $\Heis$ square to translations, $\rho(a)\in\mathsf{Tr}$, and we distinguish two cases depending on the component $X$ lies in.

If $X\in\diag\{-1,1,1\}\Heis_0$ reflects across the vertical and conjugates $\rho(b)\in\Heis_0$ to its inverse, $\rho(b)$ cannot have any vertical translation component, and so preserves the horizontal foliation.  
As $\rho\in\mathcal{D}_{\Hs^2}(K)$, combining with $\rho(a)\in\mathsf{Tr}$ shows $\rho$ is the holonomy of an axis-aligned translation torus,
and there is a unique solution for $X$ up to conjugacy $\tilde{\rho}(X)=\smat{-1 & 0 & 0\\ 0 &1& r/2\\0&0&1}$.  
If $X\in\diag\{1,-1,1\}\Heis_0$ reflects across the horizontal, the only solutions to $X^2=\rho(a)$ are horizontal translations, and $\rho(b)$ must not have horizontal translational component.
There is a one-parameter family of solutions $X$ to the system, all conjugate via conjugacies fixing $\rho$ to a glide reflection across the horizontal, $\smat{-1 &0&-\lambda/2\\0&1&0\\0&0&1}$.
	
\end{proof}

\begin{corollary}
The space of M\"obius bands identifies with the space of Klein bottles or Cylinders, $\mathcal{T}_{\Hs^2}(\mathsf{M})\cong\R\sqcup \R^2$.
\end{corollary}
\begin{proof}
A Heisenberg M\"obius band has mirror double a Klein bottle and orientation double cover an annulus,	 so points of $\mathcal{D}_{\Hs^2}(M)$ correspond to triples $[\rho, F, X]$ for $[\rho, X]\in\mathcal{D}(K)$, $[\rho, F]\in\mathcal{D}(\mathsf{Cyl})$ satisfying $FX=XF$.  
Every $\rho\in\mathcal{D}_{\Hs^2}(T)$ that extends to a representation of $\pi_1(\mathsf{Cyl})$ does so uniquely, and also uniquely extends to a representation of $\pi_1(K)$ and so there is a unique M\"obius band covered by the torus with holonomy $\rho$.
\end{proof}

\begin{proposition}
Each Heisenberg structure on $\mathcal{O}\in\{D^2(2,2;\varnothing), \D^2(\varnothing,2,2,2,2),\RP^2(2,2)\}$ is the quotient of a unique axis-aligned translation torus.
Thus $\mathcal{T}_{\Hs^2}(\mathcal{O})\cong \R_+\sqcup\R_+$.
\end{proposition}
\begin{proof}
These three orbifolds are twofold covered by $\S^2(2,2,2,2)$, and thus fourfold covered by translation tori.
The orbifolds $\D^2(2,2;\varnothing)$ and $\D^2(\varnothing;2,2,2,2)$ are also covered by the annulus, and the only translation annuli are axis aligned.  
Each such axis aligned torus has a unique $\D^2(2,2;\varnothing)$ and $\D^2(\varnothing;2,2,2,2)$ quotient.  
The orbifold $\RP^2(2,2)$ arises as a fourfold quotient of the torus by glide reflections $x,y$ such that $\pi_1(T^2)=\langle x^2,y^2\rangle$.  
As seen in the Proposition \ref{Prop:KleinBottle}, each glide reflection squaring to a generator of $\pi_1(T^2)$ is along an axis of $\R^2$, so in this case the torus cover must be an axis-aligned translation torus.  
Each such admits a unique $\RP^2(2,2)$ quotient.

\end{proof}

\begin{proposition}
The orbifold $\D^2(2;2,2)$ has Teichm\"uller space homeomorphic to $\R\sqcup\R$.
\end{proposition}
\begin{proof}
This orbifold is the quotient of the pillowcase by a reflection passing through two opposing cone points, and thus is fourfold covered by a translation torus.
Algebraically this is an extension of $\pi_1(P)=\langle a,b,r\rangle$ by $\langle f\rangle=\Z_2$ satisfying $faf=b$, $fbf=a$, $frf=r\inv$.
Up to $\Heis_+$ conjugacy we may choose representations for homothety classes of translation tori translating along $v_\theta=\smat{\cos \theta\\\sin\theta}$ and $\lambda v_\theta^\perp=\smat{-\lambda \sin\theta\\\lambda\cos\theta}$ uniquely defined for $\theta\in[0,\pi)$, $\lambda>0$.
The only reflections $F$ representing $f$ are parallel to the $x$ or $y$ axes; so the covering torus $T$ cannot be axis 
aligned for this to pass through the cone points of the pillow quotient.
For $F\in\diag(-1,1,1)\Heis_0$ computing with the relations shows there is a solution if and only if $\theta\in(0,\pi)$ and $\lambda=\tan\theta$.
Similarly, for $F\in\diag(1,-1,1)\Heis_0$, a solution exists for $\theta\in(\pi/2,\pi)$ and $\lambda=-\tan\theta$.
These solutions are unique up to conjugacy and so $\mathsf{T}_{\Hs^2}(\D^2(2;2,2))\cong\R\sqcup\R$.
\end{proof}

\section{Degenerations and Cone Tori}
\label{sec:Heis_Degens_and_Cone_Tori}
\index{Heisenberg Geometry!Cone Tori}
\index{Conemanifolds!Tori}
\index{Hyperbolic Manifolds!Cone Tori}

Unless otherwise specified, $\mathbb{X}$ denotes any one of the constant curvature geometries $\S^2,\E^2$ or $\Hyp^2$ realized as a subgeometry of $\RP^2$ (see Section 2.4) throughout.
Conjugate models will be denoted $C.\mathbb{X}$ for $C\in\GL(3;\R)$.
Recall a collapsing path $[f_t,\rho_t]$ of $\mathbb{X}$ structures degenerates to a Heisenberg structure if 
there is a path $C_t\in\GL(3;\R)$ with $C_t.[f_t,\rho_t]=[C_t f_t,C_t\rho_tC_t\inv]$ converging in the space of developing pairs to $[f_\infty,\rho_\infty]$ with $f_\infty$ an immersion into the affine patch $\Hs^2=\{[x:y:1]\}$ and $\rho_\infty$ with image in $\Heis$.
We may view these rescaled $\mathbb{X}$ structures as geometric structures modeled on the conjugate subgeometry $C_t.\mathbb{X}$, which converge to a Heisenberg structure as $C_t.\mathbb{X}$ itself converges to $\Hs^2$.
The following proposition, a consequence of \cite{CooperDW14} (or a straightforward calculation of conjugacy limits of Lie algebras) describes which conjugacies of $\mathbb{X}\in\{\S^2,\E^2,\Hyp^2\}$ limit to the Heisenberg plane.

\begin{proposition}
Let $\mathbb{X}\in\{\S^2,\Hyp^2\}$ and $C_t\colon [0,\infty)\to \GL(3;\R)$ a path of diagonalizable matrices with eigenvalues $|\lambda_t|>|\mu_t|$.  
Then $C_t.\mathbb{X}$ limits to Heisenberg geometry in 
$\RP^2$ if and only if $|\lambda_t|,|\mu_t|$ and $|\lambda_t/\mu_t|$ all diverge to $\infty$.
For $\mathbb{X}=\E^2$, the divergence $|\lambda_t/\mu_t|\to\infty$ alone is necessary and sufficient.
\end{proposition}

Up to $\O(3)$ conjugacy we may always arrange things so that $C_t.\mathbb{X}\cong D_t.\mathbb{X}$ for $D_t$ a path of diagonal matrices $D_t=\diag(\lambda_t,\mu_t,1)$ with $\lambda_t>\mu_t>1$, and we focus on these \emph{diagonal conjugacy limits}.
In this section, we classify which Heisenberg tori arise as rescaled limits of collapsing constant-curvature geometric structures.
As all constant-curvature tori are Euclidean, we consider the natural generalization of \emph{conemanifold structures} on the torus, which exist in both positive and negative curvature.

\subsection{Constant Curvature Cone Tori}

\begin{definition}
An $\mathbb{X}$ cone-surface is a surface $\Sigma$ with a complete path metric that is the metric completion of an $\mathbb{X}$-structure on the complement of a discrete set.	
\end{definition}

\noindent
An $\mathbb{X}$ cone torus $T$ with cone points $C=\{p_1,\ldots p_n\}$ gives an incomplete $\mathbb{X}$-structure on $T_\star^2=T^2\smallsetminus C$ encoded by a class of developing pairs \cite{CooperHK00}.
The space of all such $\mathbb{X}$ cone tori can be identified with the subset $\mathcal{C}_\mathbb{X}(T^2)\subset\mathcal{D}_\mathbb{X}(T^2_\star)$ with metric completions $T^2$, given the subspace topology under this inclusion.

\begin{definition}
A path $T_t$ of $\mathbb{X}$ cone tori converges projectively if the associated incomplete structures $(f_t,\rho_t)\in\mathcal{D}_{\mathbb{X}}(T^2_\star)$ converge in $\mathcal{D}_{\RP^2}(T_\star^2)$ to a projective structure $(f_\infty,\rho_\infty)$, which can be completed to a projective torus $T$. 
A Heisenberg torus $T$ \emph{regenerates} to $\mathbb{X}$ structures if there is a sequence of $\mathbb{X}$ cone tori converging to $T$ in $\RP^2$.
\end{definition}


Cone tori with a single cone point admit a convenient combinatorial description via marked parallelograms, which provides us substantial control.
A marked $\mathbb{X}$-parallelogram is a quadrilateral $Q\subset \mathbb{X}$ with opposing geodesic sides of equal length, equipped with a a cyclic ordering of the vertices.  
 Such a marked parallelogram is determined by a vertex $v$,the geodesic lengths of the sides adjacent to $v$ and the angle of incidence at $v$.  
The moduli space $\mathcal{P}(\mathbb{X})$ of marked parallelograms nonpositive curvature is $\R_+^2\times(0,\pi)$, and $\left(0,\tfrac{\pi}{2\kappa}\right)^2\times (0,\pi)$ in spherical space of radius $\kappa$.
Just as deformation space of Euclidean tori can be identified with isometry classes of marked parallelograms $\mathcal{P}(\mathbb{E}^2)$ (thought of as $\R_+$ cross the upper half plane), so can the deformation spaces of $\mathbb{H}^2$ and $\mathbb{S}^2$ cone structures.

\begin{proposition}
The map $\mathsf{Glue}\colon\mathcal{P}(\mathbb{X})\to\mathcal{C}_{\mathbb{X}}(T_\star)$ induced by isometrically identifying opposing sides of $Q\in\mathcal{P}(\mathbb{X})$ is a homeomorphism.
\end{proposition}
\begin{proof}
There is a unique orientation preserving isometry sending any oriented line segment in $\mathbb{X}$ to any other of the same length.
Thus marked quadrilateral $Q\subset\mathbb{X}$ determines unique side pairings $A,B\in\Isom_+(\mathbb{X})$ identifying opposing sides.
The quotient is a topologically a torus and inherits an $\mathbb{X}$ structure on the complement of $[v]$.
If $Q'$ is isometric to $Q$ then there is a $g\in\Isom(\mathbb{X})$ with $g.Q=Q'$ so the induced structures are isomorphic and $\mathsf{Glue}$ is well defined.

We may also define a map $\mathsf{Cut}\colon\mathcal{C}_{\mathbb{X}}(T_\star)\mapsto\mathcal{P}(\mathbb{X})$ as follows.
An marked $\mathbb{X}$ cone torus $T$ has generators $a,b\in\pi_1(T)$ based at the cone point, which
may be pulled tight relative $p$ to length minimizing representatives $\alpha,\beta$ as $T$ is a compact path metric space.
These are locally length minimizing, and so $\mathbb{X}$-geodesics away from $p$.
As $a\simeq\alpha,b\simeq \beta$ generate $\pi_1(T)$, $\alpha$ and $\beta$ have algebraic intersection number $1$.  
As each is globally minimizing in its pointed homotopy class, the complement $T\smallsetminus\{\alpha\cup\beta\}$ contains no bigons.  
From this it follows that $\alpha\cap\beta=\{p\}$, and so cutting along $\alpha,\beta$ gives an $\mathbb{X}$ parallelogram $Q$.
These maps are easily seen to be inverses and thus define homeomorphisms $\mathcal{P}(\mathbb{X})\cong\mathcal{C}_{\mathbb{X}}(T_\star)$. 
\end{proof}

\begin{figure}[h!]
\centering
\includegraphics[width=0.8\textwidth]{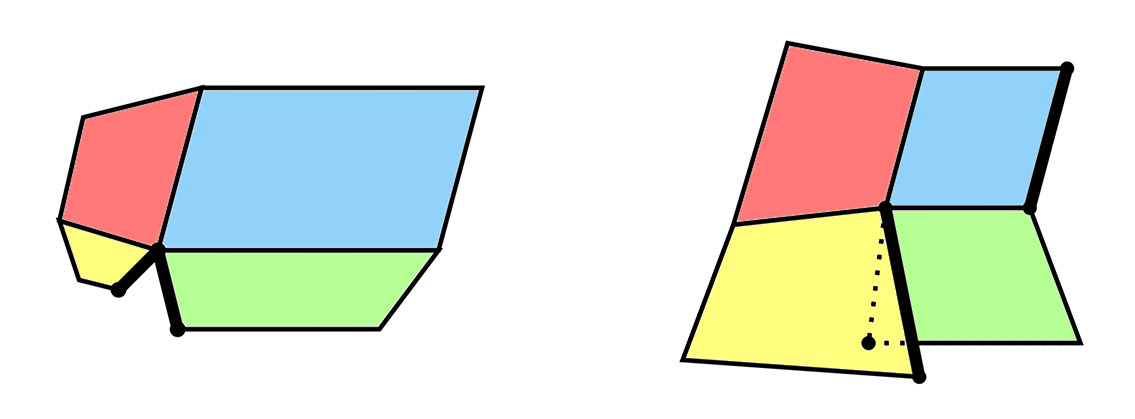}	
\caption{Small portions of the developing map for a hyperbolic and spherical cone torus}
\end{figure}
 
 \noindent
 To study regenerations from this combinatorial perspective, we characterize when a collapsing path in $\mathcal{C}_\mathbb{X}(T_\star)$ converges in $\mathcal{D}_{\RP^2}(T_\star)$ in terms of marked parallelograms.
 
 \begin{proposition}
 Let $\mathbb{X}_t=D_t\mathbb{X}$ be a sequence of conjugate geometries converging to $\Hs^2$ in $\mathcal{S}_{\RP^2}$ and $T_t$ an $\mathbb{X}_t$ cone torus for each $t$ with marked parallelogram $Q_t$.
 Then $T_t$ converges to a Heisenberg torus if and only if there is a choice of embeddings of $Q_t$ into $\mathbb{X}_t\subset\RP^2$ with $Q_t\to Q$ in the Hausdorff space of closed subsets of $\RP^2$ with induced side pairing $A_t,B_t$ converging to $A,B$ in $\PGL(3;\R)$ such that $[A,B]=I$.
 \label{Prop:Quad_Convergence}
 \end{proposition}
\begin{proof}

Let $(f_t,\rho_t)$ be a convergent sequence of developing pairs for the incomplete structures on $T_\star=T^2\smallsetminus \{\ast\}$ for $\mathbb{X}_t$ cone tori $T_t$.
Choose a generating set $a,b\in\pi_1(T_\star)$ and a basepoint $q\in\widetilde{T_\star}$.
The universal cover $\widetilde{T_\star}$ is tiled by ideal quadrilaterals formed from the lifts of $a,b$.
For each $t$ these can be straightened to geodesics in the  $\mathbb{X}_t$ structure, let $\widetilde{Q_t}\subset\widetilde{T_\star}$ be the geodesic quadrilateral containing $q\in\widetilde{T_\star}$.
Then $f_t(\widetilde{Q}_t)=Q_t\subset\mathbb{X}_t$ is a parallelogram for each $t$, with sides paired by $A_t=\rho_t(a)$, $B_t=\rho_t(b)$.
The convergence of developing pairs then implies $A_t, B_t$ are convergent in $\PGL(3;\R)$  to $A, B$ and $Q_t$ converges to $Q_\infty$, a fundamental domain for the Heisenberg structure $T$ with sides paired by the commuting transformations $A,B$.

Conversely let $Q_t$ be a sequence of $\mathbb{X}_t$ parallelograms convergent in the Hausdorff space $\mathfrak{C}_{\RP^2}$ to an affine parallelogram $Q$.
The triples $(Q_t, A_t, B_t)$ of the quadrilateral with side pairings define $\mathbb{X}_t$ cone tori, and hence $\RP^2$ punctured tori for all $t$.  
As $t\to \infty$ these converge to a punctured torus $T_\infty$ with holonomy in $\Heis$, and so $T_\infty\in\mathcal{D}_{\Hs^2}(T_\star)$.  
As $[A,B]=I$ the limiting holonomy factors through $\Z\oplus \Z$ and so the limiting torus can be completed to a torus $T_\infty$.
That the limits $A,B\in\Heis$ follows from the definition of $\mathbb{X}_t$ converging to $\Hs^2$, so this limiting projective structure canonically strengthens to a Heisenberg structure.
\end{proof}

\section{Regeneration of Tori}
\label{sec:Heis_Regen}
\index{Heisenberg Geometry!Regeneration}
\index{Geometries!Heisenberg!Regeneration}

\subsection{Translation Tori}
\index{Heisenberg Geometry!Regeneration!Translation Tori}

This combinatorial description of cone tori with at most one cone point provides enough control to completely understand the regeneration of translation tori.


\begin{theorem}
Let $\mathbb{X}\in\{\S^2,\E^2,\Hyp^2\}$ and $\mathbb{X}_t=D_t.\mathbb{X}$ be a sequence of diagonal conjugates converging to $\Hs^2$.
Given any translation torus $T$ there is a sequence of $\mathbb{X}_t$ cone tori with at most one cone point converging to $T$.	
\end{theorem}

\begin{proof}[Proof (Euclidean Case):]

Heisenberg tori arise as limits of collapsing families of \emph{smooth} Euclidean tori (there are no Euclidean cone tori with a single cone point, per Gauss-Bonnet).
Let $T$ be a Heisenberg translation torus and $\E_t=D_t.\E^2$ be a sequence of diagonal conjugates of $\E^2$ converging to the Heisenberg plane.  
Choose a fundamental domain $Q$ for $T\subset \Hs^2$, together with side pairings $A,B$ by translations for $T$.
The underlying space for the models $\E^2$, $\E_t$ and $\Hs^2$ in $\RP^2$ are all the entire affine patch $\A^2=\{[x:y:1]\}$; and group $\mathsf{Tr}$ of translations acting on this affine patch is contained in each conjugate $D_t\Isom(\E^2)D_t\inv$ as well as $\Heis$.
Thus $(Q,A,B)$ encodes an $\E_t$-structure $[f,\rho]_{\E_t}$ on $T^2$ for each $t\in\R_+$.
Canonically weakening to projective structures, this is the constant sequence $[f,\rho]_{\RP^2}$ thus clearly convergent.
As $\rho(\Z^2)\subset\mathsf{Tr}< \Heis$, the limit canonically strengthens to the original Heisenberg structure $[f,\rho]_{\Hs^2}$.
\end{proof}

\noindent
Viewed as Euclidean structures in the fixed model $\E^2$, the developing pairs $[D_t\inv f, D_t\inv \rho D_t]$ encode a collapsing collection of tori with one of the generators of the holonomy shrinking much faster than the other.
That is, even after rescaling to unit area structures this path fails to converge in Teichm\"uller space and limits to a point in the Thurston boundary.
The foliation represented by this point can actually be seen in the limiting Heisenberg structure as the invariant foliation pulled back from $dy$ on $\Hs^2$.

The approach for producing translation tori as limits of hyperbolic and spherical cone tori is similar in spirit, but more involved in the details.
Again we take a fundamental domain with side pairings $(Q,A,B)$ for the proposed limit, and view $Q$ as a geometric parallelogram in each of the model geometries $\mathbb{X}_t$. 
Side pairings $A_t,B_t\in\Isom(\mathbb{X}_t)$ are uniquely determined by each $\mathbb{X}_t$ structure on $Q$, and converge to $A,B$ in the limit.

\begin{figure}
\centering
\includegraphics[width=0.75\textwidth]{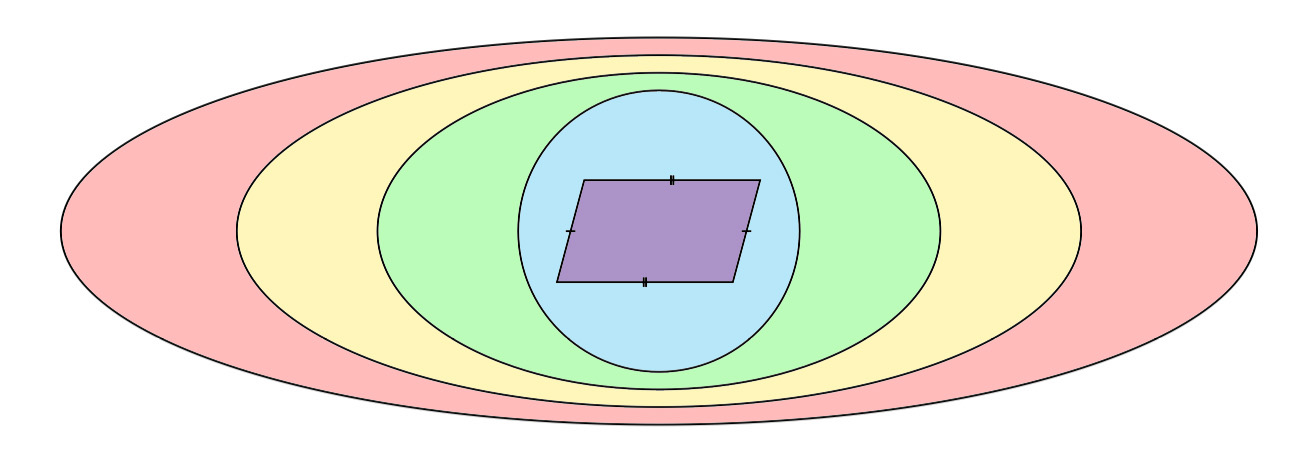}
\caption{A fixed Quadrilateral and various conjugate models of $\Hyp^2$ containing it.}
\end{figure}

\begin{proof}[Proof: Hyperbolic and Spherical Cases]

If $\mathbb{X}\in\{\S^2,\Hyp^2\}$, let $Q$ be an origin-centered fundamental domain for $T$ with side pairings $A,B\in\mathsf{Tr}$.  The existence of a convergent sequence of $\mathbb{X}_t$ cone tori $T_t\to T$ follows from the following facts.

\begin{itemize}
\item[] \textbf{Claim 1:} For large $t$, the quadrilateral $Q$ defines an $\mathbb{X}_t$ parallelogram.
\item[] \textbf{Claim 2:} The side pairing $A_t$ preserves the entire projective line through the $\mathbb{X}_t$ midpoints of paired sides.
\item[] \textbf{Claim 3:} If $Q$ is an $\mathbb{X}_t$ parallelogram for all $t$ and $A_t\in\Isom(\mathbb{X}_t)$ pairs opposing sides, $A_t$ converges as a sequence of projective transformations.
\item[] \textbf{Claim 4:} The $\mathbb{X}_t$ midpoints of the edges of $Q$ converge to the Euclidean midpoints as $t\to\infty$.
\end{itemize}

\noindent
Given that $Q$ defines an $\mathbb{X}_t$ parallelogram, there are unique side pairing transformations $A_t,B_t\in\Isom(\mathbb{X}_t)$ determining an $\mathbb{X}_t$ cone torus.  
By the third claim, these sequences of transformations converge in $\PGL(3,\R)$, and as $\mathbb{X}_t\to\Hs^2$ in fact $A_\infty, B_\infty\in\Heis_0$. 
Recalling the discussion in Section 3, $\Heis_0$ acts simply transitively on the subspace $\mathcal{L}\smallsetminus\mathcal{H}$ of pointed lines, so the limiting transformations are completely determined by their action on  a pair $(p,\ell)$ of a point $p$ on a non-horizontal line $\ell$.  

\begin{figure}
\centering
\includegraphics[width=0.5\textwidth]{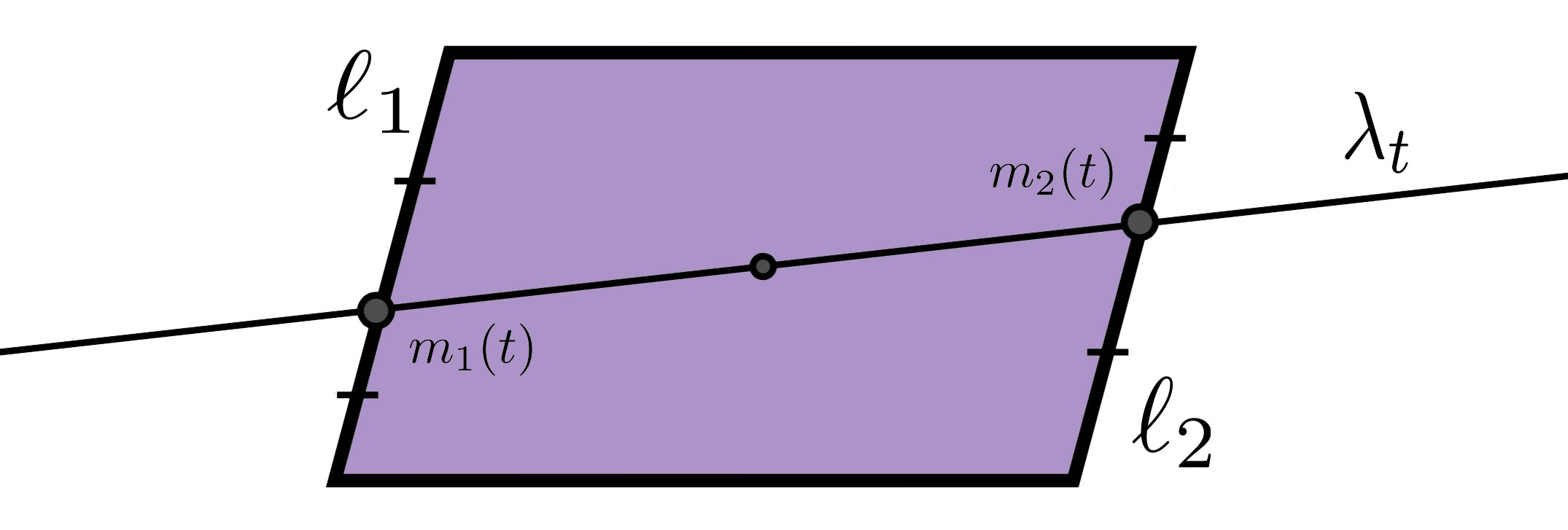}
\end{figure}

\noindent
Let $\ell_1,\ell_2$ be a pair of opposing sides of $Q$, with Euclidean midpoints $m_1,m_2$.  
For each $t$, let $m_1(t)$ and $m_2(t)$ be the $\mathbb{X}_t$ corresponding midpoints, and $\lambda_t$ the projective line connecting them.
The second claim implies $A_t$ preserves $\lambda_t$ and so the fourth fact above implies that $A_\infty$ preserves $\lambda=\overline{m_1m_2}$.
Thus $A_\infty$ sends the pair $(m_1,\ell_1)$ to $(m_2,\ell_2)$, as well as the pair $(m_1,\lambda)$ to $(m_2,\lambda)$.  
At least one of the lines $\ell_1,\lambda$ is non-horizontal, and so this completely determines the behavior of $A_\infty$.  
As this agrees precisely with the action of the original transformation $A$, we have $A_\infty=A$ and similarly for $B$.  
Thus the sequence of cone tori corresponding to the triples $(Q, A_t,B_t)$ converge to the original Heisenberg torus $T$ as $t\to\infty$.
\end{proof}

\noindent
Thus the proof reduces to an argument for the four claims above.  
Throughout its often helpful to switch between the perspectives of a fixed fundamental domain $Q$ in expanding model geometries $\mathbb{X}_t$ and the equivalent picture of shrinking domains $Q_t$ in the fixed model $\mathbb{X}$.

\begin{claim}[1]
Let $Q$ be a affine parallelogram centered at $\vec{0}\in\A^2$ and $\mathbb{X}_t\to\Hs^2$ a sequence of diagonal conjugates of $\mathbb{X}\in\{\S^2,\Hyp^2\}$.  Then for all $t>>0$, $Q$ defines an $\mathbb{X}_t$ parallelogram.	
\end{claim}
\begin{proof}
The $\pi$-rotation about $\vec{0}\in\mathbb{A}^2$ represented by $R=\diag(-1,-1,1)$ is in $\O(3)\cap\O(2,1)$ and is invariant under diagonal conjugacy.  
Thus for each $t$, $R\in\Isom(\mathbb{X}_t)$.  
As $Q$ is an affine parallelogram with centroid $\vec{0}$, $RQ=Q$ so there is an $\mathbb{X}_t$ isometry exchanging opposing sides of $Q$.  
Thus if $Q\subset \mathbb{X}_t$ it defines an $\mathbb{X}_t$ parallelgoram.  
For $\mathbb{X}=\S^2$ this is always satisfied, and for $\mathbb{X}=\Hyp^2$, the domains $\mathbb{X}_t$ limit to the affine patch and so eventually contain any compact subset.
\end{proof}

\begin{claim}[2]
Let $A\in\Isom(\mathbb{X})$ pair opposing sides of the $\mathbb{X}$ parallelogram $Q$.  Then $A$ preserves the projective line through the midpoints of the paired sides.	
\end{claim}
\begin{center}
\includegraphics[width=0.9\textwidth]{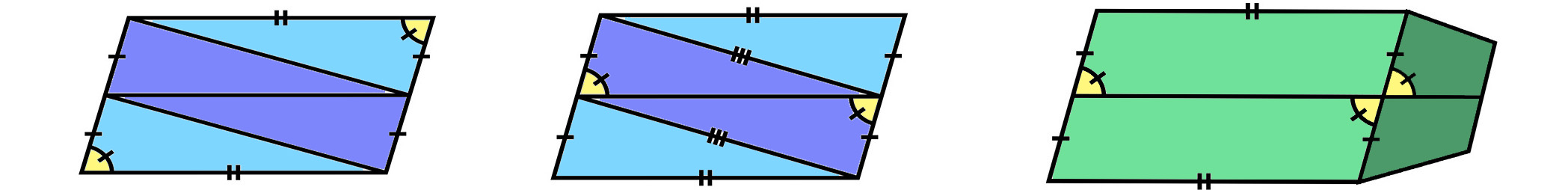}
\end{center}
\vspace{-0.5cm}  
\begin{proof}
We argue in classical axiomatic geometry without assuming the parallel postulate as this applies equally to $\S^2,\Hyp^2$.
Opposite angles of a constant-curvature parallelogram are congruent.
Connect the opposing sides of $Q$ paired by $A_t$ with a line segment $\lambda$ through their midpoints. 
This divides $Q$ into two quadrilaterals, subdivided by their diagonals into four triangles.  
The outer two of these triangles are congruent by side-angle-side, and so the diagonals are congruent.  
Thus the inner two triangles are congruent by side-side-side, meaning the opposite angles made by the edges with the line connecting their midpoints are equal.
Consider $Q$ and its translate $A.Q$.
These share an edge, which is meets the segments $\lambda$ and $A_t\lambda$ at its midpoint $m$.
As $A$ is an isometry, it follows that opposite angles at $m$ are congruent.
Thus $\lambda$ and $A.\lambda$ are segments of a single projective line, so 
$A$ preserves the line extending $\lambda$ as claimed.
\end{proof}

\begin{claim}[3]
The side pairings $A_t, B_t\in\Isom(\mathbb{X})$ converge in $\PGL(3,\R)$.
\end{claim}
\begin{proof}
A projective transformation of $\RP^2$ is completely determined by its values on a projective basis (a collection of four points in general position). 
The vertices $(v_i)$ of $Q$ form a convenient projective basis with images $(A_tv_i)$ completely specifying the transformations $A_t$.
These transformations converge in $\PGL(3;\R)$ if and only if $(A_tv_i)$ limits to a projective basis, which, as the images $A_tv_i$ remain in a bounded neighborhood of $Q$
\footnote{The conjugating path $C_t$ is \emph{expansive}, with eigenvalues $\lambda_t>\mu_t$ each monotonic in $t$.  Then for $\mathbb{X}=\Hyp^2$, its easy to see $A_tQ\subset AQ$, and for $\mathbb{X}=\S^2$, that $A_tQ<A_0 Q$ for all $t>0$.}
is equivalent to no triangle $\Delta\subset Q$ formed by 3 vertices of $Q$ collapsing in the limit.
That is, it suffices to show 
$\Area_{\E^2}(A_t\Delta)/\Area_{\E^2}(\Delta)\not\to 0$.

Diagonal transformations act linearly on the affine patch and do not change ratios of areas, thus we may transform this to the fixed model $\mathbb{X}$ with a collapsing sequence of triangles $\Delta_t$ being moved by transformations $C_t=D_tA_tD_t\inv$.
For large $t$, both $\Delta_t$ and $C_t\Delta_t$ are extremely close to the origin $\vec{0}\in\A^2$ and we may estimate their area ratio analytically.
By claim 2, $C_t$ preserves the geodesic through the midpoints of paired sides, thus is either a hyperbolic in $\Isom(\Hyp^2)$ or rotation in $\Isom(\S^2)$ with axis represented by an ideal point relative the affine patch.
In each of these cases we may bound the distortion of Euclidean area under such an isometry $C$ with translation length $\tau$ within the Euclidean ball $B_{\E^2}(0,\ep)$ of radius $\ep$ as follows:
 
$$\frac{1}{(c(\tau) + \ep s(\tau))^3}\leq \frac{\Area_{\E^2}(X.S)}{\Area_{\E^2}(S)}\leq\frac{1}{(c(\tau)-\ep s(\tau))^3}.$$

\noindent
Where $(c,s)=(\cosh,\sinh)$ for $\mathbb{X}=\Hyp^2$ and $(\cos,\sin)$ for $\mathbb{X}=\S^2$.  
As $t\to\infty$,  $\Delta_t$ collapses to $\vec{0}$ and so the translation length $\tau_t$ of $C_t$ goes to $0$.  
Choosing a sequence $\ep_t\to 0$ such that $\Delta_t\subset B_{\E^2}(0,\ep_t)$ the above bounds squeeze the limiting area of $C_t \Delta_t$ to $\Delta_t$ by $1$, so the area of $A_t\Delta$ does not collapse in the limit.
\end{proof}

\begin{claim}[4]
Let $\ell\subset \A^2$ be a line segment and $\mathbb{X}_t\to\Hs^2$ as above.  Then the $\mathbb{X}_t$ midpoint of $\ell$ converges to the Euclidean midpoint.
\end{claim}
\begin{proof}

Let $\ell=\overline{pq}$ and $m\in\ell$ be the Euclidean midpoint.
Viewing $\ell$ in $\mathbb{X}_t$, it has $\mathbb{X}_t$ midpoint $y_t$, and to show $y_t\to m$ it suffices to see $d_{\mathbb{X}_t}(p,m)/d_{\mathbb{X}_t}(m,q)\to 1$. 
Ratios of collinear line segment lengths are invariant under linear transformations, so we may choose to view this situation in the fixed model $\mathbb{X}$ for ease of calculation, with a shrinking line segment $\ell_t=\overline{p_tq_t}$ with Euclidean midpoint $m_t$ and $\mathbb{X}$ midpoint $x_t$.

For $\mathbb{X}=\Hyp^2$ a straightforward computation shows the length of any segment $\ell\subset B_{\E^2}(0,\ep)$ is bounded by a multiple of its Euclidean length 
$\mathsf{Length}_{\E^2}(\ell)\leq \mathsf{Length}_{\mathbb{X}}(\ell)\leq K_\ep \mathsf{Length}_{\E^2}(\ell)$ where $K_\ep$ may be chosen\footnote{For hyperbolic space we may choose $K_\ep=1/\sqrt{1-4\ep^2}$ and for the sphere $K_\ep=1/(1+\ep^2)$ with $\ep$ measured in the Euclidean metric on the affine patch} so that $K_\ep>1, \lim_{\ep\to 0} K_\ep=1$.
Similarly pulling back the spherical metric to the affine patch there is such a $K_\ep>1$ with $\mathsf{Length}_{\E^2}(\ell)/K_\ep \leq \mathsf{Length}_{\mathbb{X}}(\ell)\leq \mathsf{Length}_{\E^2}(\ell)$.
We may use this to bound the difference between the $\mathbb{X}$ and Euclidean midpoints of the shrinking segments $\ell_t$.
	
$$\frac{1}{K_\ep}=\frac{d_{\E^2}(p_t,m_t)}{K_\ep d(m_t,q_t)}\leq \frac{d_{\mathbb{X}}(p_t,m_t)}{d_{\mathbb{X}}(m_t,q_t)}=\frac{d_{\mathbb{X}_t}(p,m)}{d_{\mathbb{X}_t}(m,q)}\leq \frac{K_\ep d_{\E^2}(p_t,m_t)}{d_{\E^2}(m_t,q_t)}=K_\ep.$$
As $\mathbb{X}_t\to\Hs^2$, $\ell_t$ collapses to $\vec{0}$ and we may take smaller and smaller $\ep$ so this ratio converges to 1.
\end{proof}

\subsection{Shear Tori}
\index{Heisenberg Geometry!Regeneration!Shear Tori}

Every translation Heisenberg torus arises as a limit of Euclidean, Hyperbolic and Spherical cone tori with at most one cone point.
Translation structures are rather special Heisenberg tori, compromising a codimension-one subset of deformation space.
Here we investigate the generic case, Heisenberg tori with nontrivial shears in their holonomy, and show none regenerate as cone structures with a single cone point.
Shears of the plane fix a single line, and alter the slope of all lines not parallel to this.
All shears in $\Heis$ are parallel, so the holonomy of any shear torus leaves invariant precisely one slope on $\Hs^2$.
This has strong consequences for the distribution of geodesics on Heisenberg orbifolds.

\begin{proposition}
A Heisenberg orbifold $\mathcal{O}$ has a shear in its holonomy if and only if all simple geodesics on $\mathcal{O}$ are pairwise disjoint.
\end{proposition}
\begin{proof}

Let $\mathcal{O}$ be a shear orbifold and $\gamma$ a simple geodesic on $\mathcal{O}$.  
As $\mathcal{O}$ is covered by a complete torus we identify $\widetilde{\mathcal{O}}$ with $\Hs^2$, and the preimage of $\gamma$ under the covering with a $\pi_1(\mathcal{O})$-invariant collection $\{\tilde{\gamma}\}$ of lines in $\Hs^2$.  
As $\gamma$ is simple these are pairwise disjoint and so parallel in $\A^2$.  
Because $\mathcal{O}$ has a shear structure, some $\alpha\in\pi_1(\mathcal{O})$ acts on $\Hs^2$ by a nontrivial shear, which alters the slope of all non-horizontal lines.  Thus, $\{\tilde{\gamma}\}$ is a subset of the horizontal foliation.  
But this holds for any simple geodesic on $\mathcal{O}$ so any two must each lift to a subset of the horizontal foliation, which are then disjoint or (by $\pi_1(\mathcal{O})$ invariance) equal. 
If the two geodesics lift to disjoint collections then their projections are also disjoint, meaning any two distinct simple geodesics on $T$ cannot intersect.

Conversely assume $\mathcal{O}$ is an orbifold covered by a translation torus $T$ given by the developing pair $(f,\rho)$, for $\rho\colon \Z^2\to \mathsf{Tr}$.  
Then $\rho(e_1)$ and $\rho(e_2)$ are linearly independent translations, each preserving each component of a family of parallel lines descending to closed intersecting geodesics on $T$ and further descend to intersecting geodesics on $\mathcal{O}$.
\end{proof}

\noindent
Hyperbolic, spherical and Euclidean (cone) tori behave quite differently than this.  
Recall that any generators $\langle a,b \rangle=\pi_1(T)$ have geodesic representatives through the cone point and cutting along these gives a constant-curvature parallelogram with side pairings.  
Claim 2 of the previous section shows these side parings must preserve the full projective lines through the midpoints of the paired edges, so these descend to intersecting closed geodesics on $T$.  
The following argument shows this property remains true in the limit.

\begin{theorem}
Let $\mathbb{X}\in\{\S^2,\E^2,\Hyp^2\}$ and $\mathbb{X}_t=D_t \mathbb{X}$ a sequence of conjugate geometries converging to the Heisenberg plane.
Let $T_t$ be a sequence of $\mathbb{X}_t$ cone tori with at most one cone point converging to some Heisenberg torus $T$.
Then $T$ is a translation torus.
\end{theorem}

\begin{proof}

By Proposition \ref{Prop:Quad_Convergence} we may represent these structures by a sequence of $\mathbb{X}_t$ parallelograms $(Q_t,A_t,B_t)$ converging to the triple $(Q_\infty,A_\infty, B_\infty)$ describing the Heisenberg torus $T$.

Claim 2 of the previous section implies that for each $t$, the side pairing $A_t$ preserves the projective line $\alpha_t$ connecting the $\mathbb{X}_t$ midpoints of the paired sides. 
As $t\to\infty$ this sequence of lines in $\RP^2$ subconverges to a projective line $\alpha_\infty$.  
Since $A_t(\alpha_t)=\alpha_t$ for all $t$, it follows that $A_\infty(\alpha_\infty)=\alpha_\infty$, so this line is preserved by the limiting action.
By Claim 3, $\alpha_\infty$ passes through the Euclidean midpoints of opposing sides of $Q_\infty$.
Thus $\alpha_\infty$ and $\beta_\infty$ descend to closed geodesics on $T$. 

As $\alpha_t, \beta_t$ intersect $\partial Q_t$ in the $\mathbb{X}_t$ midpoints of opposing sides, they divide $Q_t$ into four congruent quadrilaterals. Thus the lines $\alpha_t, \beta_t$ intersect at the center of mass of $Q_t$.  
It follows that in the limit the lines $\alpha_\infty,\beta_\infty$ intersect at the center of $Q_\infty$ and  the closed geodesics on $T$ given by the projections of $\alpha_\infty,\beta_\infty$ intersect.
As $T$ has intersecting geodesics, $T$ cannot have any shears in its holonomy, and thus is a translation torus.
\end{proof}

\chapter{$\mathbb{H}_\C$ and $\mathbb{H}_{\R\oplus\R}$}
\label{chap:HC_and_HRR}

Complex hyperbolic space is a generalization of the usual (real) hyperbolic space, replacing $\R$ with the field $\C$.
In this chapter, we take the standard model of $\Hyp_\C^n$, a subset of $\CP^n$ with automorphisms $\SU(n,1;\C)$ and attempt to further generalize, producing a collection of analogs of hyperbolic space not defined over $\R$ or $\C$, but over a general real algebra $\Lambda$ with involution.
These geometries all contain a copy of $\Hyp_\R^n$ as their real points, arising from the embedding $\R\inject\Lambda$.
Much as complex hyperbolic geometry provides an interesting arena to study the deformation theory of real hyperbolic manifold groups (for example, see \cite{Parker10,Parker16-2,Parker16,Schwartz01-2, Schwartz01}),
the geometries $\Hyp_\Lambda^n$ provide a collection of new such potential deformation theories.

The three simplest geometries arising from this construction (after real hyperbolic space $\Hyp_\R^n$ itself) correspond to the three isomorphism classes of 2-dimensional algebras, namely $\Hyp_\C^n$, $\Hyp_{\R_\ep}^n$, and $\Hyp_{\R\oplus\R}^n$.
We construct each of these in detail below, and focus especially on understanding the new geometries corresponding to $\R_\ep$ and $\R\oplus\R$ as a search did not find discussion of these in the literature.

\section{\bfseries Algebras and Hyperbolic Geometry}

We briefly review the construction of real hyperbolic space.
Minkowski space $\R^{n,1}$ is the vector space $\R^{n+1}$ together with a quadratic form of signature $(n,1)$, for specificity $q(x,y)=\sum_{i=1}^n x_i y_i-x_{n+1}y_{n+1}$.
This quadratic form induces an indefinite norm on $\R^{n,1}$, by $x\mapsto q(x,x)$ whose negative level sets are hyperboloids of two sheets and positive level sets are hyperboloids of one sheet\footnote{When $n=1$ both the positive and negative level sets are hyperbolas of one components in the plane.}, separated by the \emph{lightcone} $\sum_{i=1}^n x_i^2=x_{n+1}^2$.

\begin{figure}
\centering\includegraphics[width=0.5\textwidth]{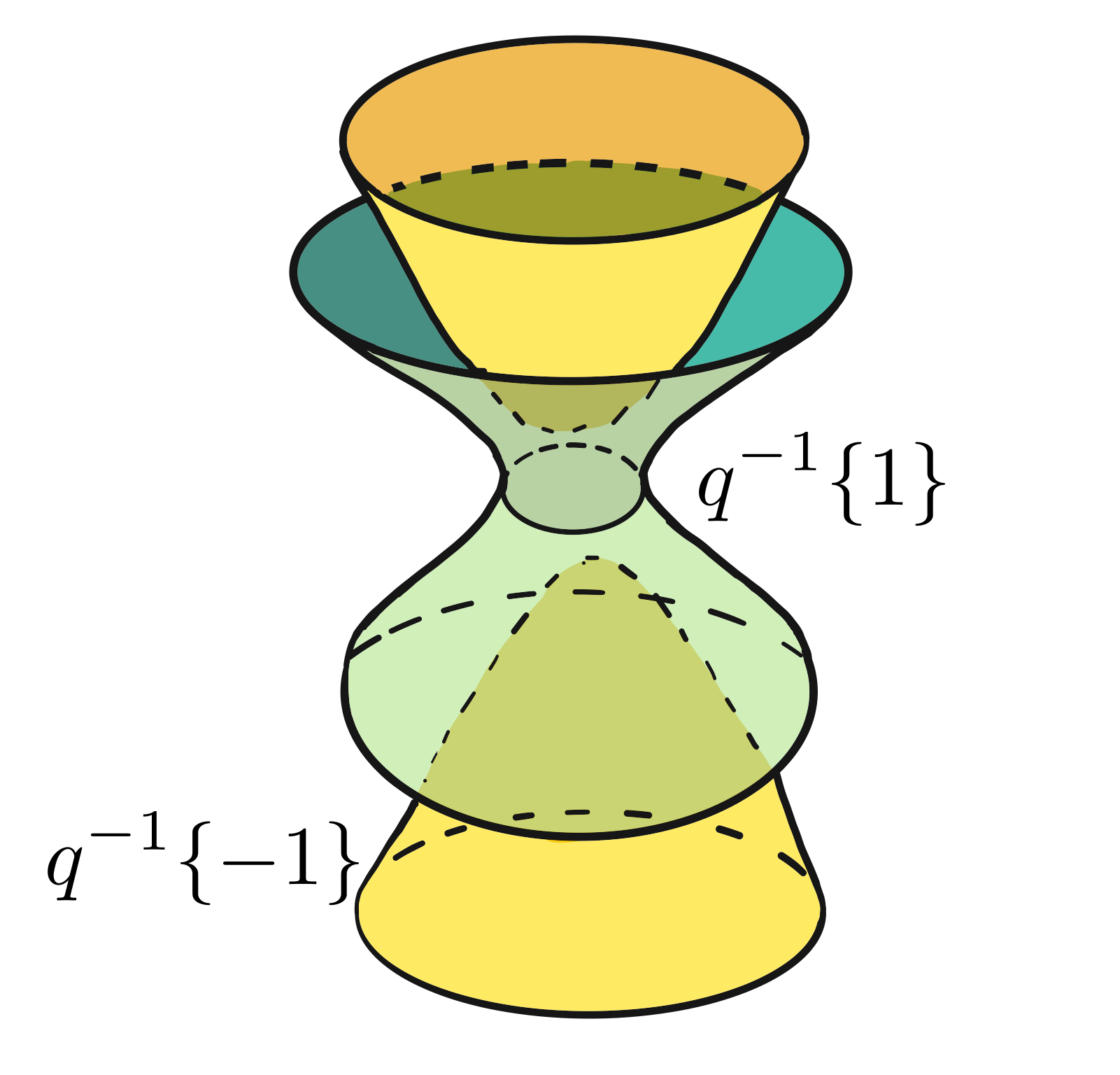}
\caption{The level sets of $q$ in $\R^{2,1}$.}	
\end{figure}

\noindent 
The linear transformations $A\in\GL(n+1;\R)$ which preserve the quadratic form $q$ in the sense that $q(x,y)=q(Ax,Ay)$ form the \emph{indefinite orthogonal group} $\O(n,1;\R)=\{A\in\GL(n+1;\R)\mid A^TQA=Q\}$ for $Q=\diag(I_n,-1)$ the matrix such that $q(x,y)=x^TQy$.
This group has 4 components, with index two orientation preserving subgroup $\SO(n,1;\R)$ and identity component $\SO_0(n,1;\R)$.
The action of $\O(n,1;\R)$ preserves the level sets of $q$ by definition, and in fact restricts to a transitive action on each
\footnote{The action on the lightcone is transitive on the complement of $\vec{0}$.}.
Hyperbolic space can be realized from the action of $\SO(n,1;\R)$ on the negative level sets of $q$ in a variety of models.

\begin{figure}
\centering\includegraphics[width=0.75\textwidth]{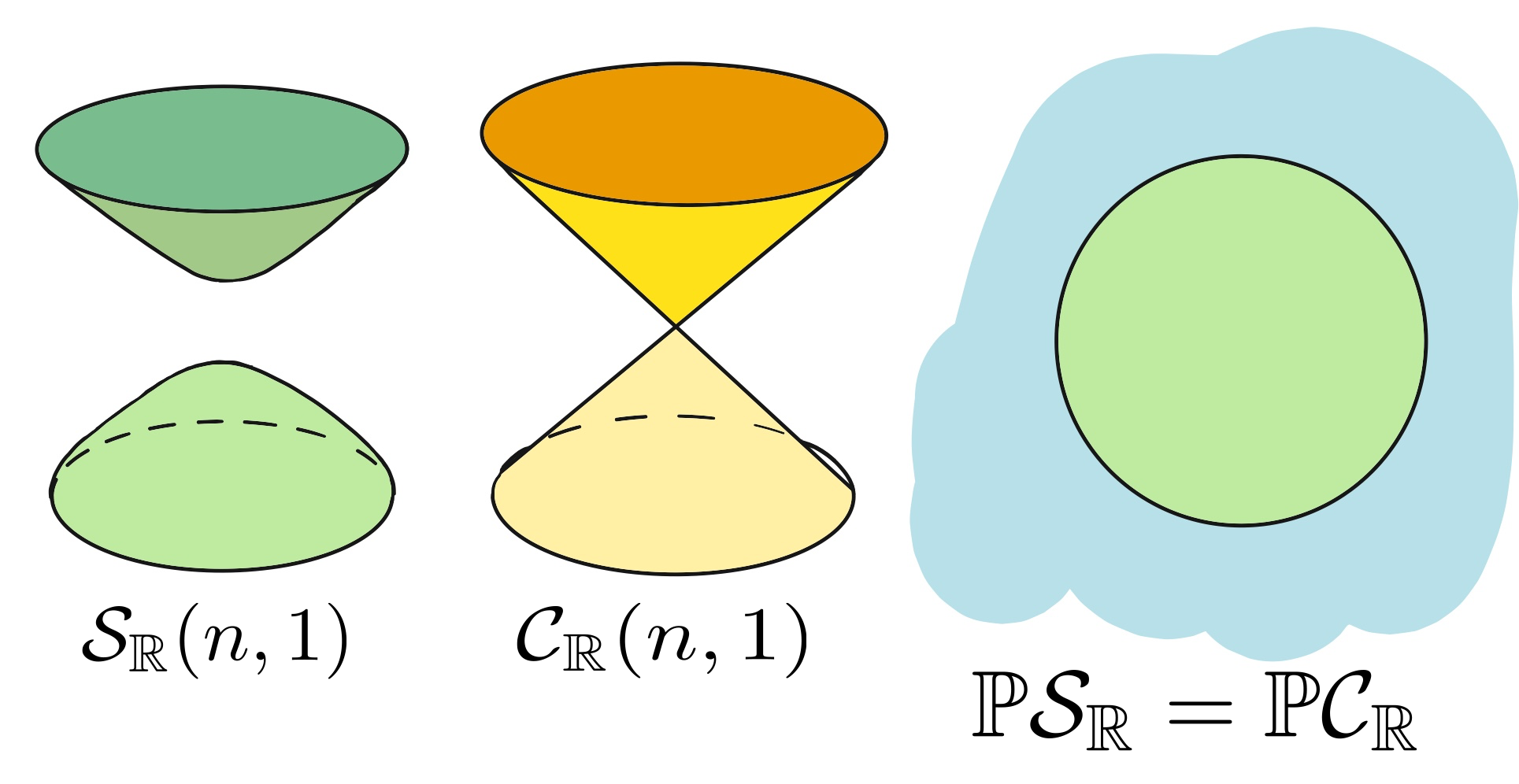}
\caption{The negative cone, the sphere of radius $-1$ in $\R^{2,1}$, and their projectivization in $\RP^2$.}	
\end{figure}

\subsubsection{The Hyperboloid Model}
The sphere of radius negative one $\mathcal{S}_{\R}(n,1)=\{x\in\R^{n,1}\mid q(x,x)=-1\}$ is a hyperboloid of two sheets, and the selection of of a single sheet (say the upper with $x_{n+1}>0$ for specificity) determines a model of hyperbolic $n$ space as a subgeometry (but not an open subgeometry) of $(\GL(n+1;\R),\R^{n+1}\smallsetminus 0)$.
This model, $(SO_0(n,1;\R),\mathcal{S}_{\R}(n,1)\cap \{x_{n+1}>0\})$ is effective, but often less convenient to work with than the \emph{projective models}, which arise as open subgeometries of $\RP^n$.

\subsubsection{Projective Hyperboloid Model}
Instead of selecting a single sheet of the sphere of radius $-1$ we may instead consider its projectivization, an open $n$ ball in $\RP^n$.
This defines the geometry $(\O(n,1;\R),\mathbb{P}\mathcal{S}_{\R}(n,1))$.
This is not an effective presentation (the transformations exchanging sheets of the hyperboloid act trivially) but is naturally effectivized via projectivization, dividing out by the elements $\U(\R)=\{\pm 1\}$ of unit norm\footnote{This nonstandard notation for $\Z_2$ is used in the coming generalizations, where $\U(\Lambda)$ will denote the elements of norm $1$ in $\Lambda$.}
 in $\R$ to give 
 $(\PO(n,1;\R),\mathcal{S}_{\R}(n,1)/\U(\R))$.
Restricting to orientation preserving isometries gives $(\PSO(n,1;\R),\mathcal{S}_{\R}(n,1)/\U(\R))$.

\subsubsection{Projective Cone Model}
Equivalently, as all negative level sets of $q$ are taken to one another by homotheties of $\R^{n+1}$, we may construct this geometry as the projectivization of the entire negative cone $\mathcal{C}_\R(n,1)=\{x\in\R^{n,1}\mid q(x,x)<0\}$ of $q$ giving $(\PO(n,1;\R),\mathcal{C}_\R(n,1)/\R^\times)$ or $(\PSO(n,1;\R),\mathcal{C}_\R(n,1)/\R^\times)$.

\subsection{Real Hyperbolic Space}

All of these constructions give the Klein model of hyperbolic space, and we mention them in detail here only because these three methods of defining $\Hyp_\R^n$ do not agree in various generalizations.
To remove ambiguity moving forwards, we select the projective hyperboloid model as \emph{the default model} of $\Hyp_\R^n$ unless otherwise specified.

\begin{definition}[$\Hyp_\R^n$: Group - Space]
Real hyperbolic space is the geometry given by the action of $\SO(n,1;\R)$ on the projectivized unit sphere of radius $-1$ for $q$ on $\R^{n,1}$; $\Hyp_\R^n=(\SO(n,1;\R),\mathcal{S}_{\R}(n,1)/\U(\R))$.	
\end{definition}

\noindent
We may alternatively encode this geometry in the autmorphism-stabilizer formalism by choosing some $p\in\fam{S}_\R(n,1)$ and computing its projective stabilizer.
A natural choice for the given form $q$ is the basis vector $e_{n+1}=(0,\ldots, 0,1)$, which is the $-1$ eigenvector of $Q$.
An easy computation shows that the stabilizer of $[e_{n+1}]$ in $\mathcal{S}_{\R}(n,1)/\{\pm 1\}$ is $\mathsf{stab}_{\SO(n,1;\R)}[e_{n+1}]=\smat{\SO(n)&\\&1}$.
When there is no worry of ambiguity, we will denote this group by $\SO(n;\R)$ for simplicity.

\begin{definition}[$\Hyp_\R^n$: Automorphism - Stabilizer]
Real hyperbolic space is the geometry given by the pair $\H_\R^n=(\SO(n,1;\R),\SO(n;\R))$.	
\end{definition}

\subsection{The Algebras $\R[\sqrt{-1}],\R[\sqrt{0}]$, and $\R[\sqrt{1}]$}

Up to isomorphism there are three 2-dimensional algebras over $\R$; any such $\Lambda$, viewed as a real vector space can be expressed $\Lambda=\span_\R\{1,u\}$ for $u^2\in \R$ and the isomorphism type of $\Lambda$ depends only on if $u^2<0$, equals 0 or $u^2>0$.
Thus, we focus on adjoining an abstract square root of $-1,0$ and $1$, forming the algebras $\C$, $\R_\ep$ and $\R\oplus\R$.

\begin{definition}
The algebra $\Lambda$ defined by adjoining an abstract square root of $\delta\in\{-1,0,1\}$ to $\R$ is defined by $\Lambda=\R\oplus\lambda \R$  with multiplication $(a+\lambda b)(c+\lambda d)=ac+\delta bd+\lambda(ac+bd)$.
\end{definition}

\noindent
When $\delta=-1$ this is a model of the complex numbers, and we denote $\lambda$ by its traditional name $i$.
When $\delta=0$, this is the so-called \emph{dual numbers} $\R[\ep]/(\ep^2)$, and following convention write elements as $a+\ep b$.
When $\delta=1$, this is isomorphic to $\R\oplus\R$, as can be seen via the decomposition $\R[\sqrt{1}]=\R e_+\oplus \R e_-$ for $e_{\pm}$ the principal idempotents $e_\pm=\tfrac{1}{2}(1\pm \lambda)$.
Each of these algebras admits an analog of complex conjugation defined by $a+\lambda b\mapsto a-\lambda b$, which induces a (not necessarily positive) multiplicative norm $\R[\sqrt{\delta}]\to\R$ given by $z\mapsto z\overline{z}$.
In the coordinates $a+\lambda b$, this norm is expressed $\|a+\lambda b\|=a^2-\delta b^2$.

\begin{figure}
\centering\includegraphics[width=0.75\textwidth]{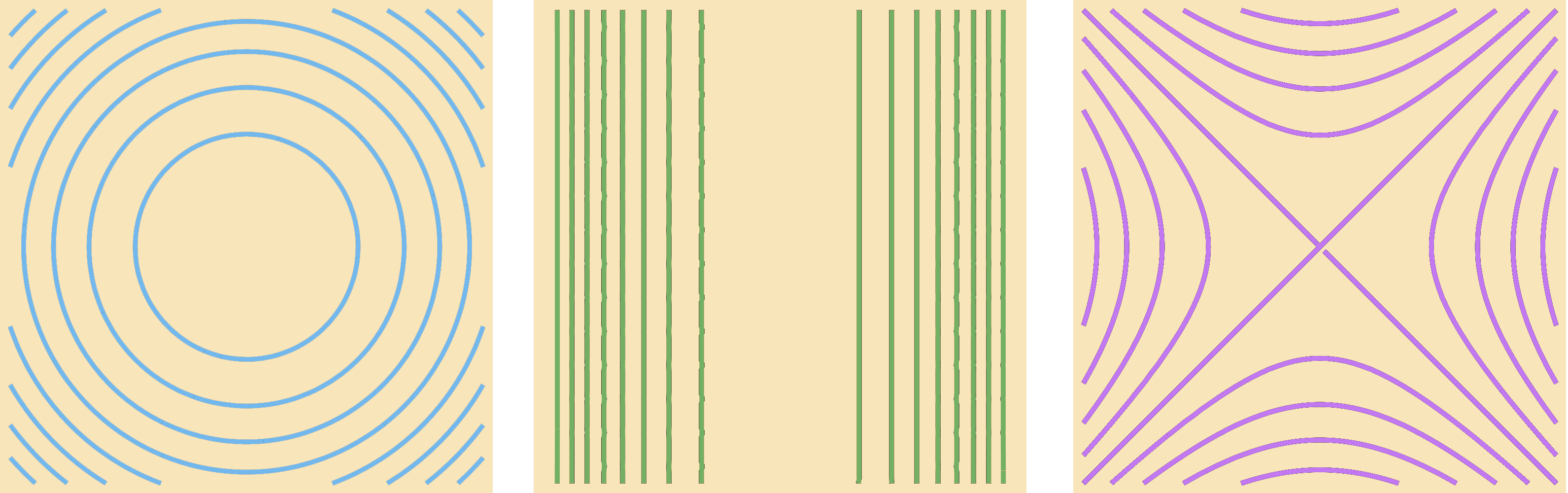}
\caption{The level sets of the norm $z\mapsto z\overline{z}$ on $\C,\R_\ep$ and $\R\oplus\R$ respectively.}
\end{figure}

\noindent
The elements of zero norm are precisely the zero divisors of $\R[\sqrt{\delta}]$, which for $\C$ consists of just $\{0\}$, for $\R_\ep$ the entire line $\ep\R=\{0+\ep x\mid x\in\R\}$ and the lines $\R e_+\cup \R e_-$ for $\R\oplus\R$.
As the norm is multiplicative, the elements of norm $1$ form a group $\U(\R[\sqrt{\delta}])$ under multiplication.
For $\C$, this is the unit circle group $\U(\C)=\S^1$ of complex numbers.
For $\R_\ep$, this is $\R\rtimes \Z_2=\{\pm 1+\ep\R\}$, and for $\R\oplus \R$ it is a pair of hyperboloids asymptoting to $\R e_+\cup \R e_-$.

\begin{figure}
\centering\includegraphics[width=0.75\textwidth]{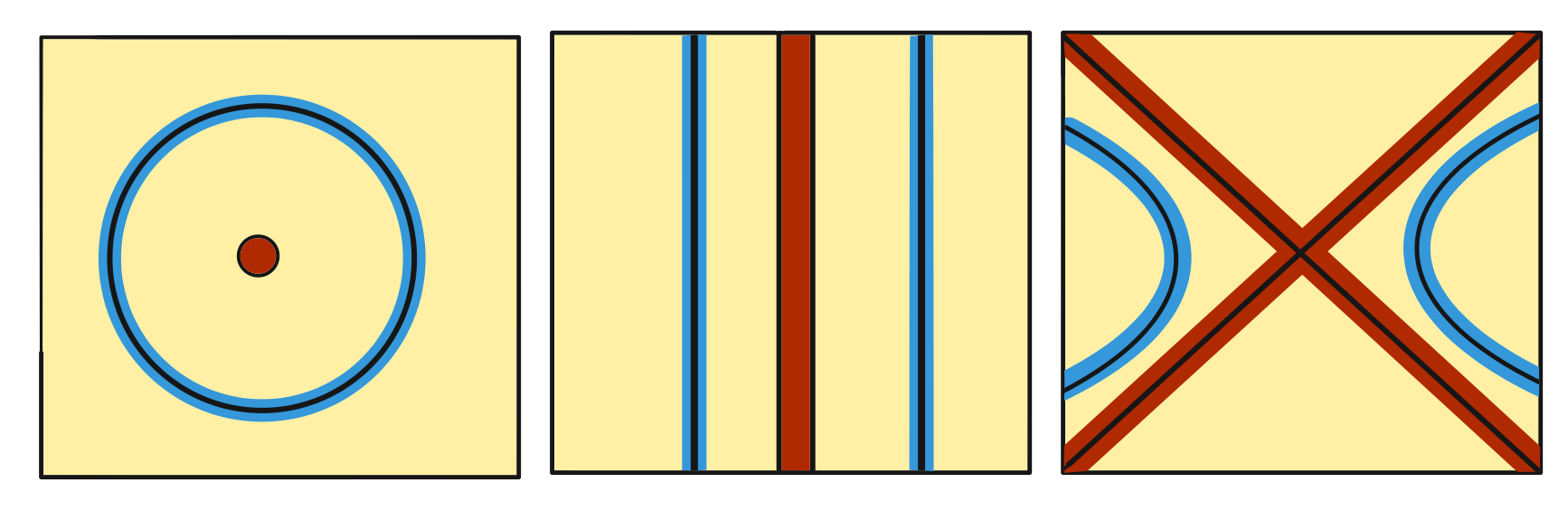}
\caption{The zero divisors (thick) and the group $\U(\Lambda)$ (thick) of $\C, \R_\ep$ and $\R\oplus\R$ respectively.}	
\end{figure}

\section{Complex Hyperbolic Space}
\label{sec:Cplx_Hyp_Space}
\index{Complex Hyperbolic Space}
\index{Geometries!Complex Hyperbolic}

The construction of complex hyperbolic space follows that of $\Hyp_\R^n$ as closely as possible, with $\C$ replacing $\R$.
The construction of $\Hyp_\C^n$ below is more detailed in elementary concepts than necessary, and lacking in many geometric details.
Our goal is to use this as a motivating example for the construction of the geometries $\Hyp_{\R_\ep}^n$ and $\Hyp_{\R\oplus\R}^n$.
For more information on the geometry of complex hyperbolic space, good references include \cite{Parker_Notes, Goldman_CplxHyp} and \cite{Epstein_CplxHyp}.

\noindent
Over the complex numbers, all nondegenerate quadratic forms are equivalent, and the correct generalization of the signature $(n,1)$ form $\sum_{i=1}^n x_iy_i-x_{n+1}y_{n+1}$ is the Hermitian form $q(w,z)=\sum_{i=1}^n w_i \overline{z_i}-w_{n+1}\overline{z_{n+1}}$.
This Hermitian form has matrix representation $Q=\diag(I_n,-1)$, evaluated $q(w,z)=w^TQ\overline{z}$, and the linear maps preserving it form the associated \emph{unitary group}.

\begin{definition}[The Unitary group $\U(n,1;\C)$]
Then the unitary group $\U(n,1;\C)$ is the group of linear transformations of $\C^{n+1}$ preserving $q$: that is $A\in\U(n,1;\C)$ if for all $w,z\in\C^{n+1}$, $q(w,z)=q(Aw,Az)$.
In terms of the matrix $Q=\diag(I_n,-1)$, this is 
$\U(n,1;\C)=\{A\in\M(n+1;\C)\mid A^\dagger Q A=Q\}$, where $A^\dagger=\bar{A}^T$ is the conjugate transpose of $A$.
The special unitary group $\SU(n,1;\C)$ is the subgroup with determinant $1$.
\end{definition}

\subsection{Group - Space Description}

\noindent 
By definition the action of $\U(n,1;\C)$ preserves the level sets of $q$ on $\C^{n+1}$, and similarly to the real hyperbolic case, acts transitively on each\footnote{Again, the action on the zero level set is transitive on the complement of $\vec{0}$.}.
However, the complex analogs of the Hyperboloid Model is not isomorphic to the Projective Hyperboloid or Projective Cone models.
The unit sphere $\mathcal{S}_{\C}(n,1)=\{z\in\C^{n+1}\mid q(z,z)=-1\}$ supports an action of the elements of $\C$ with unit norm $\U(\C)=\{z\in \C\mid z\overline{z}=1\}$ which is a 1-dimensional Lie group, thus the hyperboloid and projective geometries differ in dimension.

The correct analog of hyperbolic space over $\C$ is given by the projective models, and the quotient of $\fam{S}_\C(n,1)$ by this $\U(\C)$ action gives a model of \emph{Complex Hyperbolic Space}, $\Hyp_\C^n=(\U(n,1;\C),\mathcal{S}_{\C}(n,1)/\U(\C))$.
This geometry is not effective, as the scalar matrices $wI$ for $w\in \U(\C)$ act trivially on the projectivization.
A locally effective version can be made by restricting to the special unitary group $\Hyp_\C^n=(\SU(n,1;\C),\mathcal{S}_{\C}(n,1)/\U(\C))$, with automorphism group $n+1$-fold covering the effective version $(\PSU(n,1;\C),\mathcal{S}_{\C}(n,1)/\U(\C))$.
As in the real case, the two projective models (projective hyperboloid and projective cone) remain isomorphic over $\C$.
We may take the domain of $\Hyp_\C^n$ to be the projectivization of the entire negative cone of $q$, $\mathcal{N}_q=\{z\in\C^{n+1}\mid q(z,z)<0\}$, under the quotient by the action of $\C^\times$ instead of just the units $\U(\C)$.
All of these various projective models, effective and non-effective, define models of complex hyperbolic space.
For convenience, we select a single model to work with, unless otherwise specified.

\begin{definition}[$\Hyp_\C^n$: Group - Space]
Complex Hyperbolic space is the geometry given by the action of $\U(n,1;\C)$ on the projectivized unit sphere of radius $-1$ for $q$ in $\C^{n+1}$;
$\Hyp_\C^n=(\U(n,1;\C),\mathcal{S}_{\C}(n,1)/\U(\C))$.
\end{definition}

\subsection{Automorphism-Stabilizer Description}

For the purposes of constructing a transition between the three different analogs of hyperbolic geometry introduced in this chapter, it is most convenient to have available a description of each from the automorphism - stabilizer perspective.
The coordinate basis vectors $e_i\in\C^{n+1}$ are eigenvectors of $Q$, with $e_{n+1}\in\mathcal{S}_{\C}(n,1)$.
Thus the stabilizer of $[e_{n+1}]$ in $\mathcal{S}_{\C}(n,1)/\U(\C)$ gives a natural representation $\H_\C^n=(\U(n,1;\C),\mathsf{stab}_{\U(n,1;\C)}[e_{n+1}])$.

\begin{calculation}
The stabilizer of $[e_{n+1}]$ under the action of $\U(n,1;\C)$ on $\Hyp_\C^n$ is $\smat{\U(n;\C)&\\&\U(\C)}$.
This \emph{unitary stabilizer group} is denoted $\USt(n,1;\C)$.
\end{calculation}
\begin{proof}
Let $A\in \U(n,1;\C)$ be such that $A.[e_{n+1}]=[e_{n+1}]$, that is $A e_{n+1}=u e_{n+1}$ for $u\in\U(\C)$.
As $A\in\U(n,1;\C)$ its columns are orthogonal with respect to the signature $(n,1)$	 Hermitian form $q$, and so in particular the final entry of the first $n$ columns is necessarily $0$.
Thus $A=\smat{B&0\\0&u}$ for some $U\in\M(n;\C)$.
As $A$ is block diagonal, $A^\dagger Q A=Q$ decomposes as $B^\dagger I_n B=I_n$ and $\bar{u}u=1$.
This second condition is just a restatement that $u\in \U(\C)$, and the first condition shows $B\in \U(n;\C)$.
\end{proof}

\begin{definition}[$\mathbb{H}_\C$: Automorphism-Stabilizer]
$\Hyp_\C^n=\left(\U(n,1;\C),\USt(n,1;\C)\right)$.
\end{definition}

\subsection{Properties of $\Hyp_\C^n$}

Complex hyperbolic space is constructed in as close an analogy as possible to real hyperbolic space, and so it is not surprising that the resulting spaces share many similarities.

\begin{calculation}
The domain of $\Hyp_\C^n$ is the open ball $\mathbb{B}^{2n}$ in $\CP^n$.	
\end{calculation}
\begin{proof}
Projectivization identifies $\Hyp_\C^n=\mathcal{S}_{\C}(n,1)/\U(\C)=\mathcal{N}/\C^\times$ with a subset of the complex projective space $\CP^n$.
Clearly for a point $\vec{z}\in\CP^n$ to lie in the negative cone of $q$ the final coordinate must be nonzero, and thus $\Hyp_\C^n$ actually lies in the affine patch $z_{n+1}\neq 0$.
Choosing affine coordinates $z_{n+1}=1$, the form $q$ defines
$\Hyp_\C^n=\{(z_1,\ldots, z_n, 1)\mid \sum_{j=1}^n z_j\bar{z_j}-1<0\}$, which writing $z_j=x_j+i y_j$ gives
$$\Hyp_\C^n=\{(x_1+iy_1,\ldots, x_n+i y_n)\mid x_1^2+y_1^2+\cdots +x_n^2+y_n^2<1\}$$
Which is the interior of the open unit ball in the affine patch $\C^n\subset\CP^n$ as claimed.
\end{proof}

\noindent 
As previously mentioned, complex hyperbolic space contains a copy of real hyperbolic space of half the dimension, arising from the inclusion $\R\subset \C$.

\begin{observation}
The inclusion $\R\subset \C$ realizes $\Hyp_\R^n$ as a half dimensional slice of $\Hyp_\C^n$, with domain the real points $\Hyp_\R^n=\Hyp_\C^n\cap\R^n\subset \C^n$ and automorphism group the real points $\O(n,1;\R)$ of $\U(n,1;\C)$.
\end{observation}

\subsection{Low Dimensional Examples}

The space $\Hyp_\C^n$ has dimension $2n$ and so quickly becomes impossible to visualize directly.
Here we focus on the low dimensional examples of $\Hyp_\C^1$ and $\Hyp_\C^2$.
The construction of complex hyperbolic 1-space begins with the Hermitian form $q(z,w)=z_1\bar{w_1}-z_2\bar{w_2}$ on $\C^2$.
The induced norm $z\mapsto q(z,z)=\|z_1\|^2-\|z_2\|^2$ divides $\C^2$ into positive and negative cones, separated by the lightcone $\{z\in\C^2\mid \|z_1\|^2=\|z_2\|^2\}$ which is the cone on the square torus in $\S^3\subset\C^2$.
Projecting first by real homotheties of $\C^2$, the positive and negative unit spheres of $q$ are homeomorphic, each identified with one of the open solid tori in the standard decomposition of $\S^3$.

\begin{figure}
\centering\includegraphics[width=0.65\textwidth]{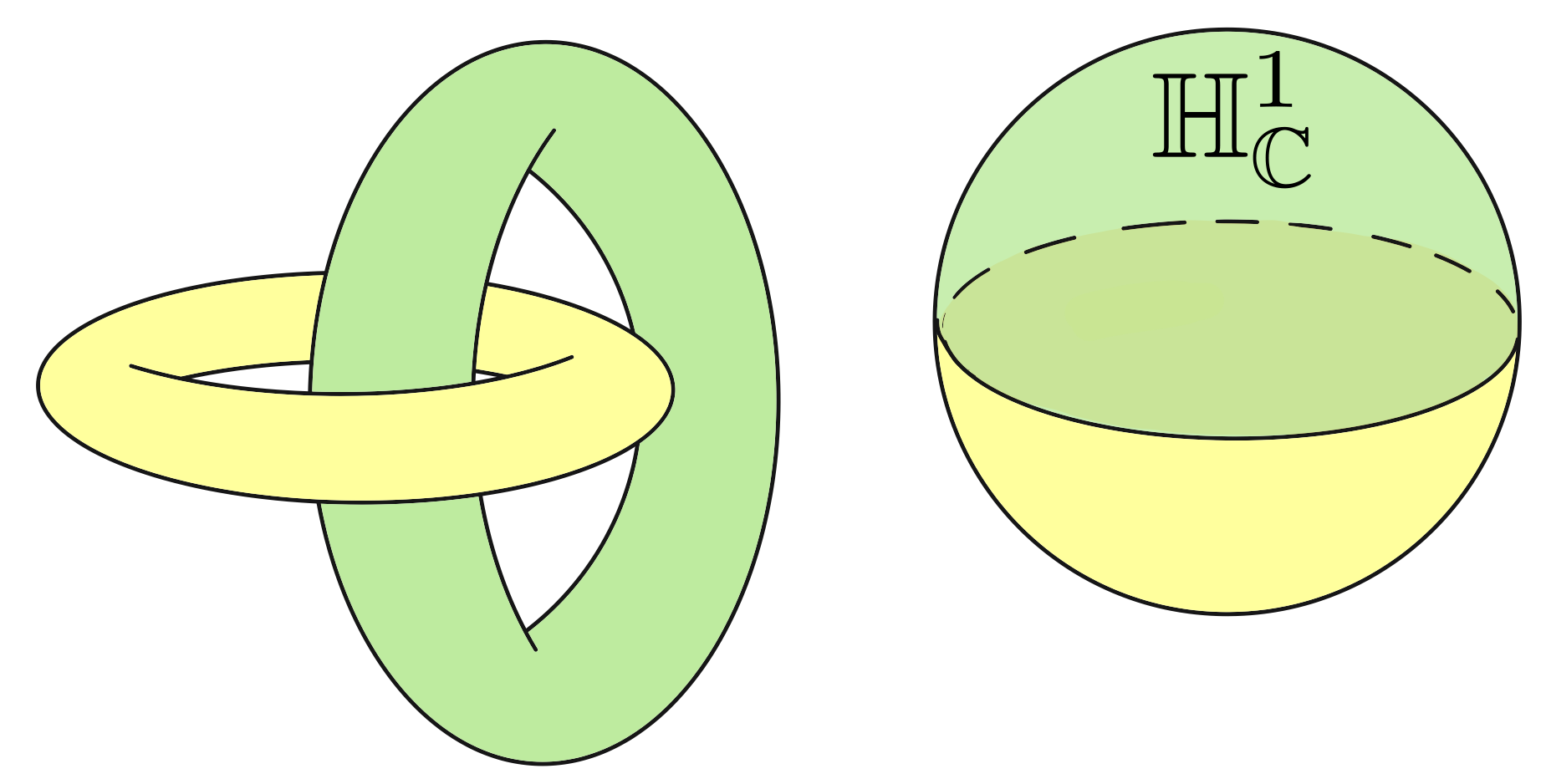}
\caption{The positive, negative and lightcones of $q$ on $\C^2$, intersected with the three sphere (left) form the standard decomposition along two linked solid tori.
The images of these in $\CP^1$ (right).
}	
\end{figure}

\noindent
The action of $\U(\C)$ on $\C^2$ restricts to an action on $\S^3$ tracing out the circles of the Hopf fibration.
In the quotient $\S^3\to \S^3/\U(\C)=\CP^1=\S^2$, each of the positive and negative cones of $q$ project to hemispheres, with the lightcone projecting to the equator.
Each hemisphere gives a model of $\Hyp_\C^1$ when equipped with the action of $\U(1,1;\C)$; though this action is not even locally effective as all diagonal matrices $uI$ act trivially on the projectivization.
A locally effective model takes instead the action of $\SU(1,1;\R)$ on the unit disk, which is conjugate in $\GL(2;\R)$ to $\SL(2;\R)$.

\begin{observation}
Complex Hyperbolic 1-space is isomorphic to real hyperbolic 2-space, and the standard construction of the projective model in $\CP^1$ produces the Poincare disk model of $\Hyp_\R^2$.	
\end{observation}

\noindent
Each geodesic in $\Hyp_\C^1$ is a half-dimensional subgeometry isomorphic to real hyperbolic 1-space.
The particular model of $\Hyp_\R^1\subset\Hyp_\C^1$ given by the embedding $\R\subset\C$ is the projectivization real plane $\{(x,y)\mid x,y\in\R\}\subset\C^2$ intersect the negative cone, giving the diameter of $\Hyp_\C^1$ preserved by $\SO(1,1;\R)$.

\begin{figure}
\centering\includegraphics[width=0.85\textwidth]{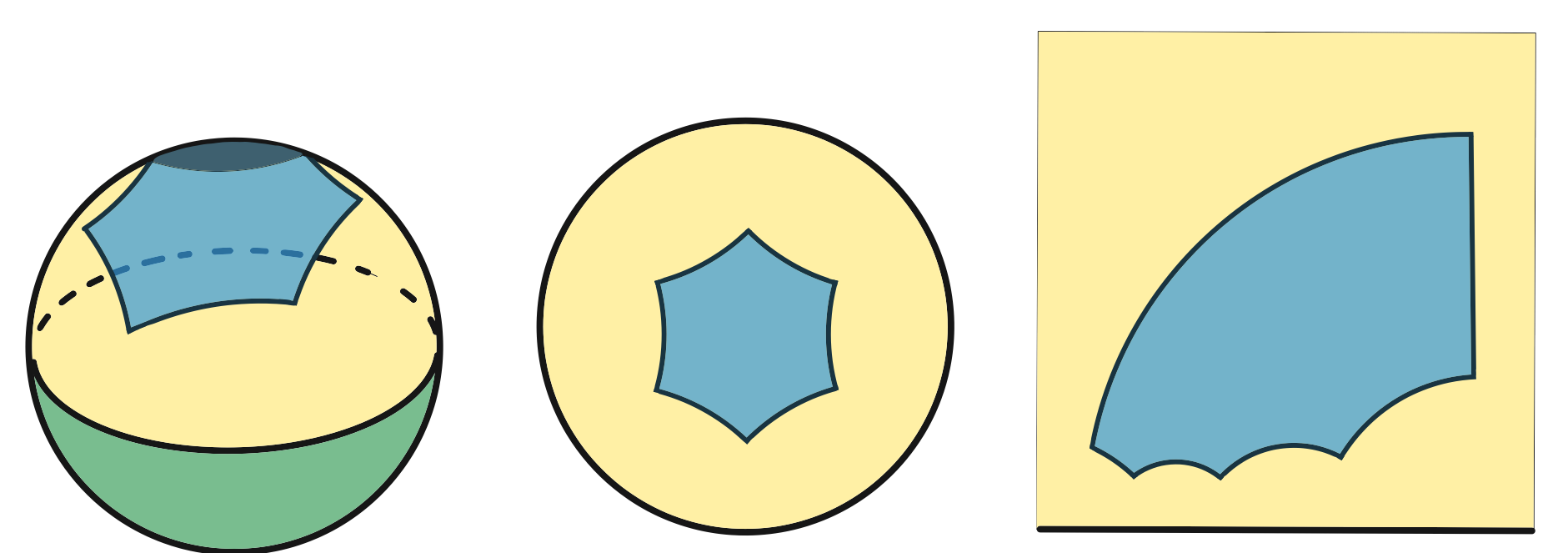}
\caption{Complex Hyperbolic Space $\Hyp_\C^1$ in $\CP^1$, the Poincare Disk model of $\Hyp_\R^2$, and the equivalent upper-half plane model under a M\"obius transformation.}
\end{figure}

\noindent
Complex Hyperbolic 2 space is a genuinely new homogeneous space, constructed from the projectivization of the negative cone of the norm $\|z_1\|^2+\|z_2\|^2-\|z_3\|^2$ on $\C^3$.
In the affine patch $z_3=1$ this appears as the interior of the unit ball $\mathbb{B}^4\subset\C^2$, and the copy of real hyperbolic space given by the inclusion $\R\subset \C$ is the intersection of the totally real plane $\{(x,y)\in\C^2\mid x,y\in\R\}$ with the unit ball.
This totally geodesic subspace naturally identifies with the Klein model of the hyperbolic plane, as geodesics in $\Hyp_\C^2$ between two points of $\Hyp_\R^2$ are the line segments connecting them.
This is not the only copy of $\Hyp_\R^2$ inside of $\Hyp_\C^2$ however: looking at the intersection of $\mathbb{B}^4$ with the complex plane $\{(z,0)\mid z\in \C\}$ in $\C^2$ gives a model of complex hyperbolic $1$-space, which as we saw above is isomorphic to the Poincare disk.
Thus in the metric on $\Hyp_\C^2$ the geodesics in these hyperbolic planes appear to be circular arcs orthogonal to the boundary sphere.
These two types of hyperbolic planes in $\Hyp_\C^n$ are not isometric, but have different curvatures: with complex slices having curvature $-1$ and real slices constant curvature $-1/4$.
These are the extrema of the sectional curvature for $\Hyp_\C^n$, which takes all values in $[-1,-1/4]$.

\begin{figure}
\centering\includegraphics[width=0.65\textwidth]{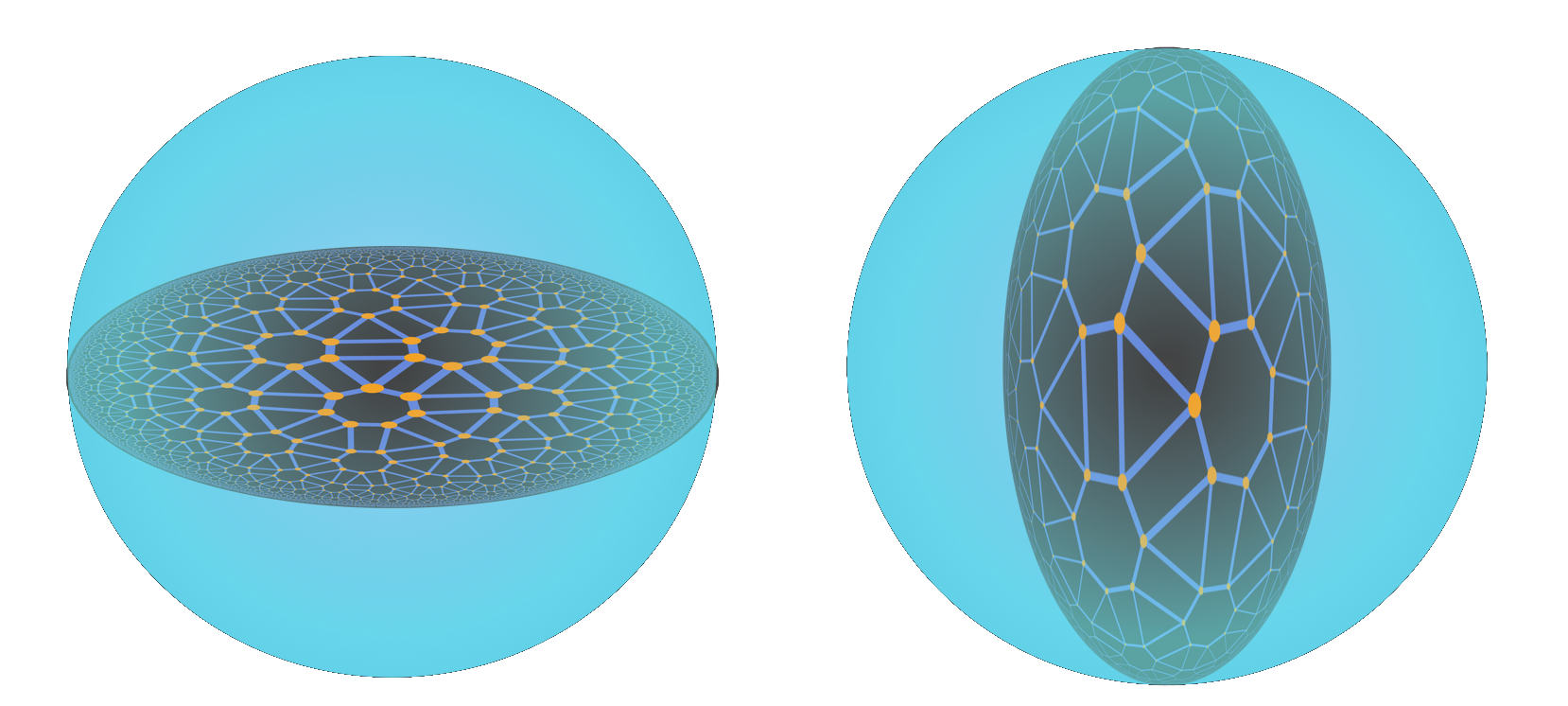}
\caption{Slices of $\Hyp_\C^2$ by totally real planes, and by complex planes give embedded copies of $\Hyp_\R^2$}	
\end{figure}

\section{Hyperbolic Geometry over $\R[\ep]/(\ep^2)$}
\label{sec:Hyp_Lambda_0}
\index{Geometries!$\R_\ep$ Hyperbolic Space}
\index{$\R_\ep$ Hyperbolic Space}

Just as complex hyperbolic space replaces $\R$ with $\C$, here we replace $\C$ with another 2-dimensional real algebra, namely that of the so called \emph{dual numbers} $\R_\ep=\R[\ep]/(\ep^2)$.

\begin{definition}[The Algebra $\R_\ep$]
The algebra $\R_\ep=\R[\sqrt{0}]$ is a two dimensional algebra over $\R$.
Each $z\in\R_\ep$ can be written uniquely as $a+\ep b$ for $\ep^2=0$.
The analog of complex conjugation on $\R_\ep$ negates the epsilon part, $a+ \ep b\mapsto a-\ep b$.
\end{definition}

\noindent
The ring of matrices $\M(n;\R_\ep)$ inherits a notion of adjoint from conjugation on $\R_\ep$, denoted $A\mapsto A^\dagger$ and defined by taking the transpose and component-wise conjugate of all entries.
The involution of $\R_\ep$ given by conjugation also provides a notion of \emph{Hermitian form} and in particular, the form $q(z,w)=\sum_{i=1}^n z_i\overline{w_i}-z_{n+1}\overline{w_{n+1}}$ defined identically to the complex case.
The matrix representation of $q$ is again $Q=\diag(I_n, -1)$ evaluated $q(z,w)=z^TQ \bar{w}$.
The $\R_\ep$ linear transformations preserving $q$ form the analog of the indefinite \emph{unitary group} over $\R_\ep$.

\begin{definition}[The Unitary group $\U(n,1;\R_\ep)$]
Then the unitary group $\U(n,1;\R_\ep)$ is the group of linear transformations of $\R_{\ep}^{n+1}$ preserving $q$: that is $A\in\U(n,1;\R_\ep)$ if for all $w,z\in\R_\ep^{n+1}$, $q(w,z)=q(Aw,Az)$.
In terms of $Q$, this is 
$\U(n,1;\R_\ep)=\{A\in\M(n+1;\R_\ep)\mid A^\dagger Q A=Q\}$.
The special unitary group $\SU(n,1;\R_\ep)$ is the subgroup with determinant $1$.
\end{definition}

\subsection{Group - Space Description}

\noindent 
By definition the action of $\U(n,1;\R_\ep)$ preserves the level sets of $q$ on $\R_\ep^{n+1}$, and similarly to the real hyperbolic case, acts transitively on each\footnote{Again, the action on the zero level set is transitive only on the complement of $\vec{0}$.}.
Like over $\C$, the units $\U(\R_\ep)$ are 1-dimensional so the hyperboloid and projective hyperboloid geometries corresponding to $\U(n,1;\R_\ep)$ are not isomorphic.
The unit sphere $\mathcal{S}_{\R_\ep}(n,1)=\{z\in\R_\ep^{n+1}\mid q(z,z)=-1\}$ supports an action of the elements of $\R_\ep$ with unit norm $\U(\R_\ep)=\{z\in \R_\ep\mid z\overline{z}=1\}$, and the quotient under this action gives a projective model of \emph{$\R_\ep$ Hyperbolic Space}, $\Hyp_{\R_\ep}^n=(\U(n,1;\R_\ep),\mathcal{S}_{\R_\ep}(n,1)/\U(\R_\ep))$.
This geometry is not effective, as the scalar matrices $wI$ for $w\in \U(\R_\ep)$ act trivially on the projectivization.
A locally effective version can be made by restricting to the special unitary group $\Hyp_\C^n=(\SU(n,1;\R_\ep),\mathcal{S}_{\R_\ep}(n,1)/\U(\R_\ep))$.
We may take the domain of $\Hyp_{\R_\ep}^n$ to be the projectivization of the entire negative cone of $q$, $\mathcal{N}_q=\{z\in\R_\ep^{n+1}\mid q(z,z)<0\}$, under the quotient by the action of $\R_\ep^\times$ instead of just the units $\U(\R_\ep)$.
All of these various presentations, effective and non-effective, define models of $\R_\ep$ hyperbolic space.
For convenience, we select a single model to work with, unless otherwise specified.

\begin{definition}[$\Hyp_{\R_\ep}^n$: Group - Space]
$\R_\ep$ Hyperbolic space is the geometry given by the action of $\U(n,1;\R_\ep)$ on the projectivized unit sphere of radius $-1$ for $q$ in $\R_{\ep}^{n+1}$;
$\Hyp_{\R_\ep}^n=(\U(n,1;\R_\ep),\mathcal{S}_{\R_\ep}(n,1)/\U(\R\ep))$.
\end{definition}

\subsection{Automorphism - Stabilizer Description}

To describe $\Hyp_{\R_\ep}^n$ in the automorphism-stabilizer formalism, we must again choose some point in the geometry's domain and compute the corresponding stabilizer.
Because the hermitian form $q$ is identically defined over $\R_\ep^{n+1}$, the element $e_{n+1}=(0,\ldots, 0,1)$ in the standard basis of $\R_{\ep}^{n+1}$ as an $\R_\ep$ module lies in $\fam{S}_{\R_\ep}(n,1)$ and provides a natural choice.

\begin{calculation}
The stabilizer of $[e_{n+1}]$ under the action of $\U(n,1;\R_\ep)$ on $\Hyp_{\R_\ep}^n$ is the unitary stabilizer group  $\USt(n,1;\R_\ep)=\smat{\U(n;\R_\ep)&\\&\U(\R_\ep)}$
\end{calculation}
\begin{proof}
Let $A\in \U(n,1;\R_\ep)$ be such that $A.[e_{n+1}]=[e_{n+1}]$, that is $A e_{n+1}=u e_{n+1}$ for $u\in\U(\R_\ep)$.
As $A\in\U(n,1;\R_\ep)$ its columns are orthogonal with respect to $q$, and so in particular the final entry of the first $n$ columns is necessarily $0$ (as $q((v_1,\ldots v_{n+1}),e_{n+1})=v_{n+1}$).
Thus $A=\smat{B&0\\0&u}$ for some $U\in\M(n;\R_\ep)$.
As $A$ is block diagonal, $A^\dagger Q A=Q$ decomposes as $B^\dagger I_n B=I_n$ and $\bar{u}u=1$.
This second condition is just a restatement that $u\in \U(\R_\ep)$, and the first condition shows $B\in \U(n;\R_\ep)$.
\end{proof}

\begin{definition}[$\mathbb{H}_{\R_\ep}$: Automorphism-Stabilizer]
$\Hyp_{\R_\ep}^n=\left(\U(n,1;\R_\ep),\USt(n,1;\R_\ep)\right)$.

\end{definition}

\subsection{Properties of $\Hyp_{\R_\ep}^n$}

\begin{calculation}
The domain of $\Hyp_{\R_\ep}^n$ is a product $\Hyp_\R^n\times\R^n$ in the affine patch $\R_\ep^n$.	
\end{calculation}
\begin{proof}
For a point $z\in\R_\ep^{n+1}$ to lie on the unit sphere of radius $-1$, necessarily the final coordinate $z_{n+1}$ is nonzero, as restricted to $e_{n+1}^\perp$ the norm induced by $q$ is positive semidefinite.
Thus, up to scaling by some unit $u\in \U(\R_\ep)$ we may assume $z_{n+1}=1$ and construct a model of $\Hyp_{\R_\ep}^n$ within the 'affine patch' $\R_\ep^n$.
\footnote{
It is possible, though we do not go through the trouble here, of defining projective space over $\R_\ep$, where this corresponds precisely with an actual affine coordinate chart $z_{n+1}=1$.}
A point $(z_1,\ldots, z_n, 1)$ lies in $\fam{N}_q$ if $\sum_{i=1}^n \|z_i\|^2-1<0$ with $\|\cdot\|$ the $\R_\ep$ norm $\|x+\ep y\|^2=x^2$.
Thus the points $(x_1+\ep y_1,\ldots x_n+\ep y_n)$ lie in $\H_{\R_\ep}^n$ if and only if $\sum_{i=1}^n x_i^2<1$, or $\vec{x}=(x_1,\ldots, x_n)\in\mathbb{B}^n$.
The \emph{'$\ep$' coordinates} $\vec{y}=(y_1,\ldots,y_n)$ are free to take on arbitrary values.
\end{proof}

\begin{observation}
The embedding $\R\inject\R_\ep$ induces an embedding $\Hyp_\R^n\inject\Hyp_{\R_\ep}^n$, with domain $\mathbb{B}^n\times\{0\}$ in $\H_{\R_\ep}^n=\mathbb{B}^n\times\R^n$.
\end{observation}

\noindent
Further analysis shows that the geometry actually fibers over $\Hyp_\R^n$.

\begin{lemma}
The group $\U(n,1,\R_\ep)$ is an extension of $\O(n,1;\R)$ by the additive group $\R^{n(n+1)/2}$.
\label{lem:limitUnitary}
\end{lemma}
\begin{proof}
Let $X+\ep Y\in \U(n,1,\R_\ep)$ for $X,Y\in\M(n+1,\R)$.  
Then
$(X+\ep Y)^\ast I_{n,1} (X+\ep Y)=(X^T-\ep Y^T)Q(X+\ep Y)=Q$, 
and expanding using that $\ep^2=0$;
$$X^T QX+\ep(X^TQY-Y^TQX)=Q.$$
Equating real and $\ep$-parts gives $X\in\O(n,1,\R)$ and $X^TQY=Y^TQX$, so $X^TQY$ is symmetric.
The map $\pi:\U(n,1;\R_\ep)\to\O(n,1,\R)$ given by $X+\ep Y\mapsto X$ is actually a surjective homomorphism:
$\pi\left((X+\ep Y)(Z+\ep W)\right)=\pi\left(XZ+\ep(XW+YZ)\right)=XZ=\pi(X+\ep Y)\pi(Z+\ep W)$
It remains to investigate
$\ker\pi=\{I+\ep Y\in\U(n,1,\R_\ep)\}$.
The condition that $X^TQY$ is symmetric reduces to the condition that $QY$ is symmetric,(using that $Q=Q\inv)$ we have map from symmetric matrices to $\ker \pi$ given by 
$S\mapsto I+\ep QY$.
Thinking of the symmetric matrices as an additive group, this is an injective homomorphism as
$Y+Z\mapsto I+\ep(Y+Z)=(I+\ep Y)(I+\ep Z)$.
Thus, we have a short exact sequence
$$0\to \R^{(n+1)(n+2)/2}\to\U(n,1,\R_\ep)\to\O(n,1;\R)\to 1.$$	
\end{proof}

\begin{corollary}
The group homomorphism $\GL(n+1;\R_\ep)\to\GL(n+1;\R)$ dropping the $\ep$-part induces an 
 epimorphism of geometries $(\U(n,1;\R_\ep),\mathsf{USt}(n,1;\R_\ep))\surject (\SO(n,1;\R),\SO(n))$ fibering over real hyperbolic space $\Hyp^n_\R=(\SO(n,1;\R),\SO(n;\R))$.
\end{corollary}

\subsection{Low Dimensional Examples}

The construction of complex hyperbolic one space begins with the Hermitian form $q(z,w)=z_1\bar{w_1}-z_2\bar{w_2}$ on $\R_\ep^2$.
The induced norm $z\mapsto q(z,z)=\|z_1\|^2-\|z_2\|^2$ in coordinates $z=x+\ep y$ is $q(z,z)=x_1^2-x_2^2$, which divides $\R_{\ep}^2$ into positive and negative cones.
The unit sphere of radius $-1$ is cut out by the hyperbola $x_1^2-x_2^2=-1$.

\begin{figure}
\centering\includegraphics[width=0.4\textwidth]{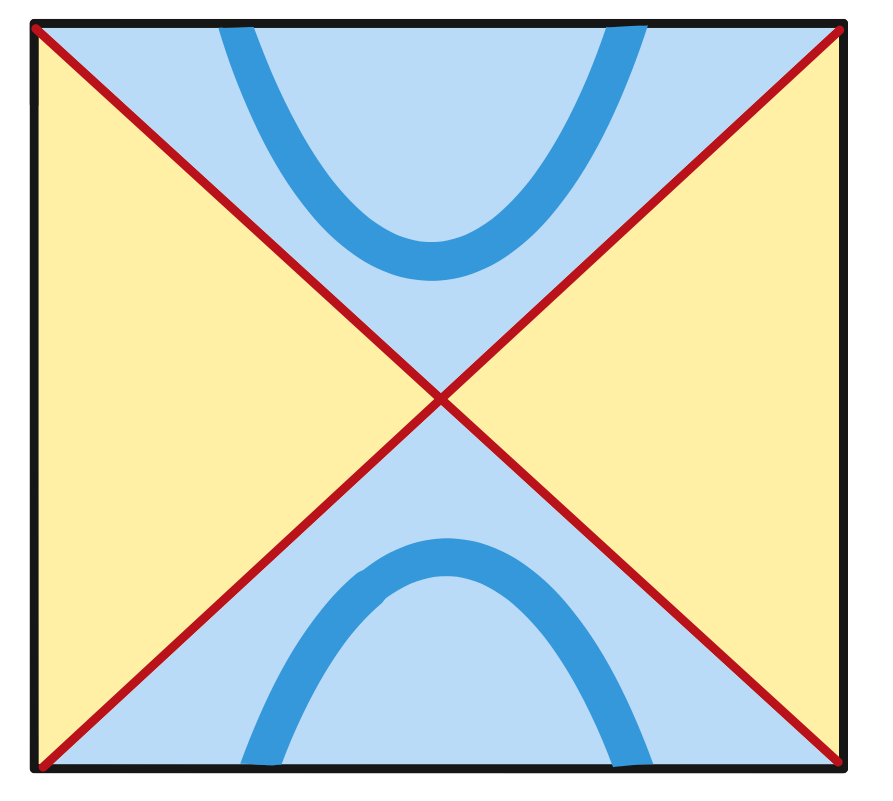}
\caption{The division into positive and negative cones of $q$, projected onto the $(x_1,x_2)$ plane (negative cone top/bottom, positive cone left/right), along with the sphere of radius $-1$.
The value of $q$ is independent of the remaining coordinates $(y_1,y_2)$.}	
\end{figure}

\noindent
The action of $\U(\R_\ep)$ on $\mathcal{S}_{\R_\ep}(n,1)$ takes the point $(x_1+\ep y_1, x_2+\ep y_2)\in\mathcal{S}_{\R_\ep}(n,1)$ to $(\pm 1+\ep u).(x_1+\ep y_1, x_2+\ep y_2)=(\pm x_1+\ep(ux_1\pm y_1,\pm x_2+\ep(ux_2\pm y_2)$.
The quotient by this action identifies each branch of the hyperbola in the $(x_1,x_2)$ plane with each other, and collapses a foliation of lines in the $\vec{y}$ direction to a point.
The result is a hyperbola times $\R$, which projectivizes in the affine patch $z_2=1$ to a strip.
The group $\SU(1,1;\R_\ep)$ is an extension of $\SO(1,1)=\Isom_+(\Hyp_\R^1)$ by $\R^2$, acting by shears perpendicular to $\Hyp_\R^1$ and translations along the $\R$ factor of $\Hyp_{\R_\ep}^1=\mathbb{B}^1\times\R$.

\begin{figure}
\centering\includegraphics[width=0.5\textwidth]{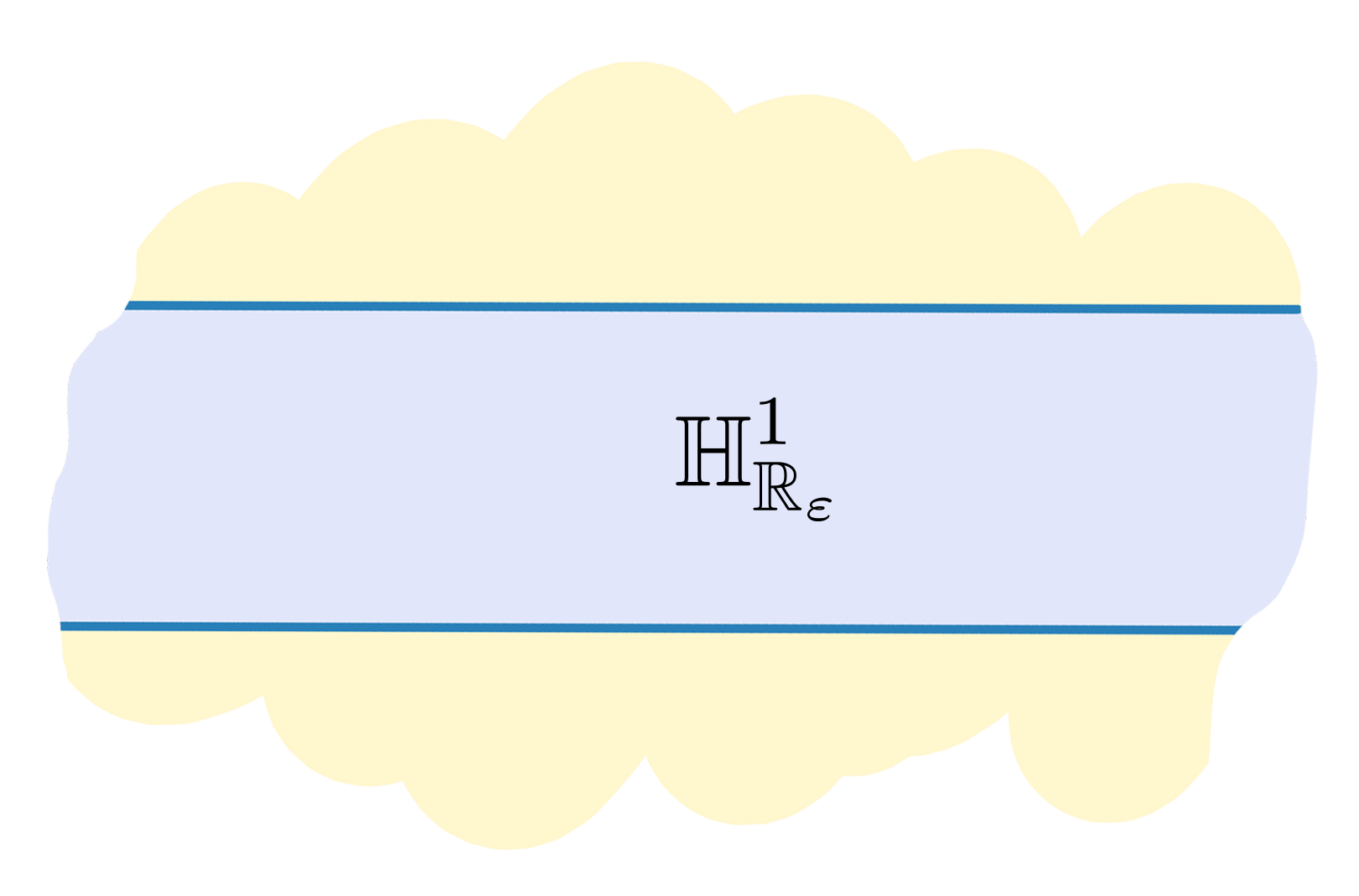}
\caption{The domain of $\Hyp_{\R_\ep}^1$}	
\end{figure}

\begin{observation}
$\Hyp_{\R_\ep}^1$ is equal to Half-Pipe geometry in dimension 2.	
\end{observation}

\noindent
In dimension two, $\Hyp_{\R_\ep}^2$ no longer coincides with Half-Pipe geometry, but can be thought of along similar lines.
$\mathsf{HP}^n$ fibers over $\Hyp_\R^{n-1}$ and has as isometries $\Isom(\Hyp_\R^{n-1})$ together with transformations not preserving the embedded copy of $\Hyp_\R^{n-1}$ but instead encoding \emph{infinitesimal ways} that $\Hyp_\R^{n-1}$ can be \emph{pushed off of itself} inside of $\Hyp_\R^n$.
Similarly, $\Hyp_{\R_\ep}^n$ has a subgroup of isometries preserving the embedded copy of $\Hyp_\R^n$, and the remaining transformations encode \emph{infinitesimal ways to push $\Hyp_\R^n$ off of itself inside of $\Hyp_\C^n$}.
We will justify this observation in the following chapter, when we construct a transition of geometries with $\Hyp_\C^n$ degenerating to $\Hyp_{\R_\ep}^n$ in the limit.

\section{$\R\oplus\R$ Hyperbolic Space}
\label{sec:RR_Hyp_Space}
\index{Geometries! R+R Hyperbolic}
\index{R+R Hyperbolic Space}

In the third iteration of this procedure, we replace $\R$ with the algebra $\R[\sqrt{1}]=\R\oplus\R$.

\begin{definition}[The Algebra $\R\oplus\R$]
The algebra $\R\oplus\R=\R[\sqrt{1}]$ is a two dimensional algebra over $\R$.
Each $z\in\R\oplus\R$ can be written uniquely as $a+\lambda b$ for $\lambda^2=1$.
\end{definition}

\noindent
The ring of matrices $\M(n;\R\oplus\R)$ inherits a notion of adjoint from conjugation on $\R\oplus\R$, denoted $A\mapsto A^\dagger$ and defined by taking the transpose and component-wise conjugate of all entries.
The involution of $\R\oplus\R$ given by conjugation also provides a notion of \emph{Hermitian form} and in particular, the form $q(z,w)=\sum_{i=1}^n z_i\overline{w_i}-z_{n+1}\overline{w_{n+1}}$ defined identically to the complex case.
The matrix representation of $q$ is again $Q=\diag(I_n, -1)$ evaluated $q(z,w)=z^TQ \bar{w}$.
The $\R\oplus\R$ linear transformations preserving $q$ form the analog of the indefinite \emph{unitary group} over $\R\oplus\R$.

\begin{definition}[The Unitary group $\U(n,1;\R\oplus\R)$]
Then the unitary group $\U(n,1;\R\oplus\R)$ is the group of linear transformations of $(\R\oplus\R)^{n+1}$ preserving $q$: that is $A\in\U(n,1;\R\oplus\R)$ if for all $w,z\in\R\oplus\R^{n+1}$, $q(w,z)=q(Aw,Az)$.
In terms of $Q$, this is 
$\U(n,1;\R\oplus\R)=\{A\in\M(n+1;\R\oplus\R)\mid A^\dagger Q A=Q\}$.
The special unitary group $\SU(n,1;\R\oplus\R)$ is the subgroup with determinant $1$.
\end{definition}

\subsection{Group - Space Description}

\noindent 
By definition the action of $\U(n,1;\R\oplus\R)$ preserves the level sets of $q$ on $(\R\oplus\R)^{n+1}$, and similarly to the real hyperbolic case, acts transitively on each nonzero level set.
As expected, over $\R\oplus\R$ the hyperboloid geometry differs from the projective ones, as $\dim U(\R\oplus\R)=1$.
But in contrast to the previous two cases, the two projective geometries are no longer isomorphic!

To construct the projective hyperboloid model,
the unit sphere $\mathcal{S}_{\R\oplus\R}(n,1)=\{z\in(\R\oplus\R)^{n+1}\mid q(z,z)=-1\}$ supports an action of the elements of $\R\oplus\R$ with unit norm $\U(\R\oplus\R)=\{z\in \R\oplus\R\mid z\overline{z}=1\}$, and the quotient under this action gives a model of \emph{$\R\oplus\R$ Hyperbolic Space}, $\Hyp_{\R\oplus\R}^n=(\U(n,1;\R\oplus\R),\mathcal{S}_{\R\oplus\R}(n,1)/\U(\R\oplus\R))$.
This geometry is not effective, as the scalar matrices $wI$ for $w\in \U(\R\oplus\R)$ act trivially on the projectivization.
A locally effective version can be made by restricting to the special unitary group $\Hyp_\C^n=(\SU(n,1;\R\oplus\R),\mathcal{S}_{\R\oplus\R}(n,1)/\U(\R\oplus\R))$.

\begin{figure}
\centering\includegraphics[width=0.65\textwidth]{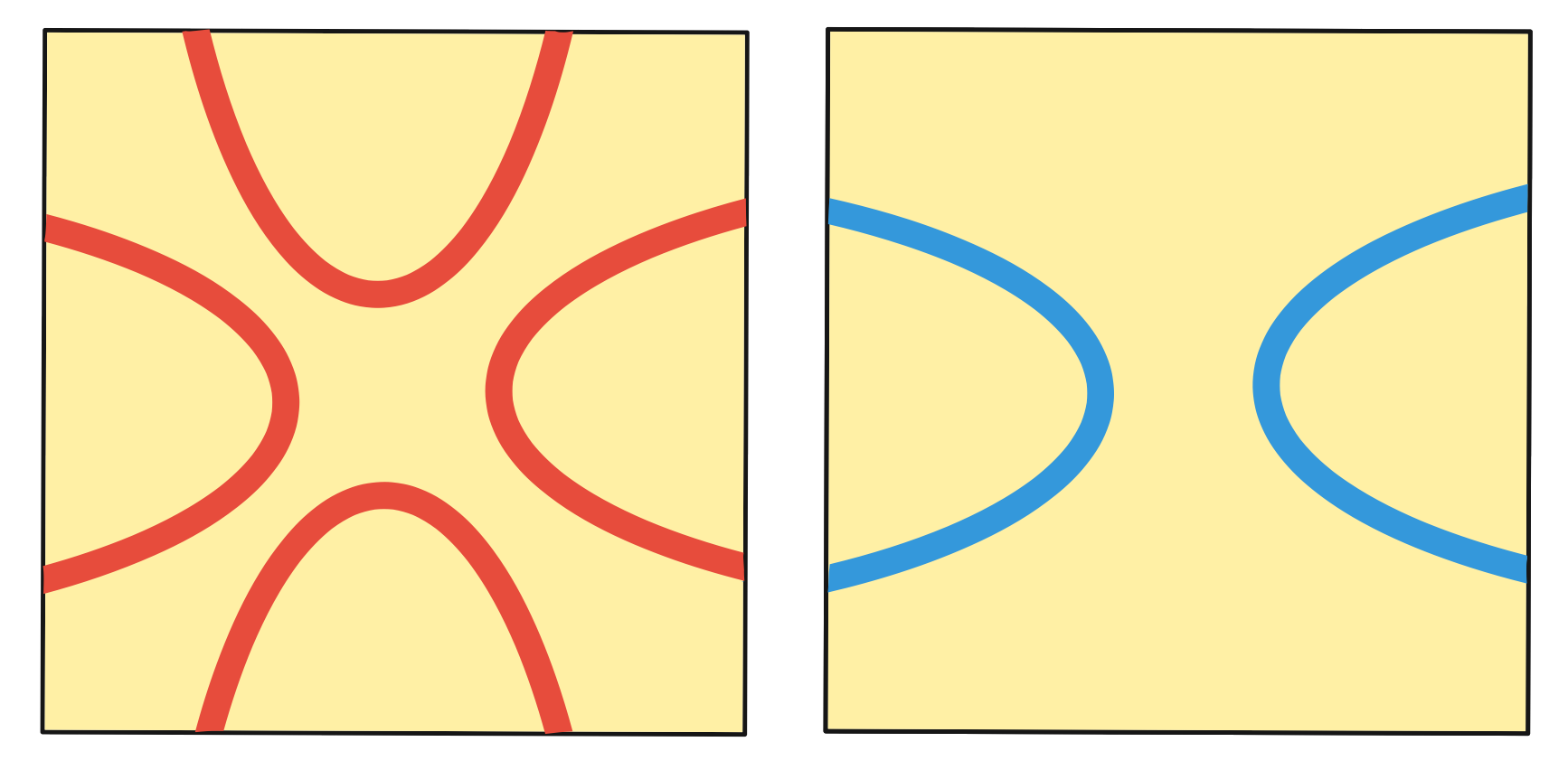}
\caption{$\U(\R\oplus\R)$ and $(\R\oplus\R)^\times)$}	
\end{figure}

\noindent 
This is actually distinct from taking the domain of $\Hyp_{\R\oplus\R}^n$ to be the projectivization of the entire negative cone of $q$, $\mathcal{N}_q=\{z\in(\R\oplus\R)^{n+1}\mid q(z,z)<0\}$, under the quotient by the action of $(\R\oplus\R)^\times$.
The group of units $(\R\oplus\R)^\times$ is the complement of the idempotent axes in $\R\oplus\R$, and thus has four components, two components of elements with positive norm and two with elements of negative norm.
The quotient of the negative cone by the index two subgroup of invertible elements with \emph{positive norm} is indeed isomorphic to the construction above; however the full projectivization is a twofold quotient.
Thus while distinct, both choices produce locally isomorphic geometries and we may freely consider either model when convenient.
Fixing a definition, we continue to utilize the projective hyperboloid model.

\begin{definition}[$\Hyp_{\R\oplus\R}^n$: Group - Space]
\label{def:RRHyp}
Complex Hyperbolic space is the geometry given by the action of $\U(n,1;\R\oplus\R)$ on the projectivized unit sphere of radius $-1$ for $q$ in $(\R\oplus\R)^{n+1}$;
$\Hyp_{\R\oplus\R}^n=(\U(n,1;\R\oplus\R),\mathcal{S}_{\R\oplus\R}(n,1)/\U(\R\oplus\R))$.
\end{definition}

\subsection{Automorphism - Stabilizer Description}

To describe $\Hyp_{\R\oplus\R}^n$ in the automorphism-stabilizer formalism, we must again choose some point in the geometry's domain and compute the corresponding stabilizer.
The standard basis vector $e_{n+1}$ evaluates to $-1$ under the norm induced by $q$, and so $[e_{n+1}]$ is a natural choice of point.

\begin{calculation}
The stabilizer of $[e_{n+1}]$ under the action of $\U(n,1;\R\oplus\R)$ on $\Hyp_{\R\oplus\R}^n$ is the unitary stabilizer group $\USt(n,1;\R\oplus\R)=\smat{\U(n;\R\oplus\R)&\\&\U(\R\oplus\R)}$.
\end{calculation}
\begin{proof}
Let $A\in \U(n,1;\R\oplus\R)$ be such that $A.[e_{n+1}]=[e_{n+1}]$, that is $A e_{n+1}=u e_{n+1}$ for $u\in\U(\R\oplus\R)$.
As $A\in\U(n,1;\R\oplus\R)$ its columns are orthogonal with respect to $q$, and so in particular the final entry of the first $n$ columns is necessarily $0$.
Thus $A=\smat{B&0\\0&u}$ for some $U\in\M(n;\R\oplus\R)$.
As $A$ is block diagonal, $A^\dagger Q A=Q$ decomposes as $B^\dagger I_n B=I_n$ and $\bar{u}u=1$.
This second condition is just a restatement that $u\in \U(\R\oplus\R)$, and the first condition shows $B\in \U(n;\R\oplus\R)$.
\end{proof}

\begin{definition}[$\mathbb{H}_{\R\oplus\R}$: Automorphism-Stabilizer]
$\Hyp_{\R\oplus\R}^n=\left(\U(n,1;\R\oplus\R),\USt(n,1;\R\oplus\R)\right)$.
\end{definition}

\subsection{Properties of $\Hyp_{\R\oplus\R}^n$}

As this homogeneous space does not appear to be treated in the literature, we discuss some basic properties.
The unitary subgroups of $\GL(n;\R\oplus\R)$ share formal similarities with the orthogonal subgroups of $\GL(n;\C)$ in that signature is ill-defined and all unitary groups over $\R\oplus\R$ are isomorphic.

\begin{observation}
The notion of signature is not well-defined on similarity classes as the simple computation below shows.
$$\pmat{1 &\\&\lambda}^\dagger\pmat{1 &\\&1}\pmat{1 &\\&\lambda}=\pmat{1&\\&-\lambda^2}=\pmat{1&\\&-1}$$
\end{observation}

\begin{corollary}
All unitary groups over $\R\oplus\R$ are conjugate to one another, and in particular $\diag(I_{n},\lambda)$ conjugates $\U(n,1;\R\oplus\R)$ to $\U(n+1;\R\oplus\R)$.
\end{corollary}

\begin{corollary}
The geometry $\Hyp_{\R\oplus\R}^n$ is conjugate to the standard unitary geometry\footnote{This may make you think that the \emph{correct}, or interesting geometries over $\R\oplus\R$ do not come from the unitary, but rather the orthogonal groups.
We study these as well in Chapter \ref{chp:Geos_over_Algs} and show only already-familiar geometries result.  For example, the geometry corresponding to $\O(n,1;\R\oplus\R)$ is $\Hyp_\R^n\times\Hyp_\R^n$.}
$(\SU(n+1;\R\oplus\R),\USt(n+1;\R\oplus\R))$ by $C=\diag(I_n,\lambda)$.	
\end{corollary}

\noindent 
To avoid the proliferation of negative signs in what follows, we will analyze this conjugate model instead.
As a first observation, 
the level sets of $\sum_i z_i\bar{z_i}$ are cut out in $\R^{2(n+1)}$ as $\sum_i x_i^2-y_i^2$ under the identification $z_i=x_i+\lambda y_i$ so the associated representation of $\SU(n+1;\R\oplus\R)$ has image in $\SO(n+1,n+1)\subg\SL(2n+2;\R)$.
The general linear group itself $\GL(n+1;\R\oplus\R)$ is isomorphic to the direct product $\GL(n+1;\R)\times\GL(n+1;\R)$ via the projections onto $\GL(n+1;\R)$ given by multiplication by the principal idempotents $A\mapsto (Ae_+,Ae_-)$.  
We may use this decomposition to understand $\U(n+1;\R\oplus\R)$.

\begin{proposition}
The group $\U(n+1;\R\oplus\R)$ is abstractly isomorphic to $\GL(n+1;\R)$, and $\SU(n+1;\R\oplus\R)\cong\SL(n+1;\R)$.
\label{prop:U(R+R)_iso_type}
\end{proposition}
\begin{proof}
Let $A\in\U(n+1;\R\oplus\R)$ and write $A=Xe_++Ye_-$ for $X,Y\in\GL(n+1;\R)$.  
Recalling that conjugation on $\R\oplus\R$ transposes the principal idempotents, 
we have $A^\dagger A=(X^Te_-+Y^Te_+)(Xe_++Ye_-)=Y^TXe_++X^TYe_-$ and expanding $e_\pm$ and equating real and $\lambda$-parts of $A^\dagger A=I$ shows $X^TY=I$.
The injection $X\mapsto Xe_++X^{-T}e_-$ from $\GL(n+1;\R)$ to $\U(n+1;\R\oplus\R)$ is easily checked to be a homomorphism, and is surjective by the above computation.
By the orthogonality of the principal idempotents, $\det(Xe_++Ye_-)=\det(X)e_++\det(Y)e_-$, the matrices of real determinant necessarily satisfy $\det(X)=\det(Y)$.  
Applying this to the elements of $\SU(n+1;\R\oplus\R)$ shows $\det(X)=\det(X^{-T})=\frac{1}{\det(X)}$, thus $\det(X)=1$.
\end{proof}

\noindent
It's useful to quickly revisit the point stabilizer with respect to this description.  A matrix $A=Xe_++X^{-T}e_-$ projectively stabilizes the vector $u=ve_++we_-$ if $Au=\alpha u$ for $\alpha=\beta e_++\gamma e_-$ a unit in $\R\oplus\R$.  Writing this out, $Xv=\beta v$ and $X^{_T}w=\gamma w$ so $v$ is an eigenvector of $X$ and $w$ an eigenvector of $X^{-T}$.
The basis vector $e_{n+1}\in(\R\oplus\R)^{n+1}$, is expressed in the $e_+,e_-$ basis as $(0,\ldots, 0)e_++(0,\ldots,0,1)e_-$ provides a convenient choice for computing the stabilizer.

\begin{observation}
Unitary geometry of dimension $2n$ over $\R\oplus\R$	 is given by $(\GL(n+1;\R),\mathsf{Stab})$ for $\mathsf{Stab}=\{X\in\GL(n;\R)\mid e_{n+1}\textrm{ is an eigenvector of}\; X, X^{-T}\}$.
\end{observation}

\subsection{Low Dimensional Examples}

Hyperbolic space of dimension $1$ over $\R\oplus\R$ is cut out as (a quotient of) the sphere of radius $-1$ with respect to the norm $q(z,z)=\|z_1\|^2-\|z_2\|^2=x_1^2+y_2^2-x_2^2-y_1^2$ for $z_i=x_i+\lambda y_i$.

\begin{observation}
This surface $\mathcal{S}_{\R\oplus\R}(n,1)$ is actually homeomorphic to an open solid torus, as can be seen through the identification with $\SL^-(2;\R)$, the $2x2$ matrices of determinant -1.

$$\det\pmat{ x_1+x_2& y_1+y_2\\y_1-y_2 & x_1-x_2}=x_1^2-x_2^2-y_1^2+y_2^2=-1$$
\end{observation}

\noindent
The action of $\U(\R\oplus\R)$ foliates this copy of $\SL(2;\R)$ with cosets of $\SO(1,1)=\U(\R\oplus\R).\smat{1&0\\0&1}$.
Thus the resulting space $\Hyp_{\R\oplus\R}^1$ is the familiar de Sitter space of dimension two $\mathsf{dS}^2=(\SO(2,1),\SO(1,1))=(\SL(2;\R),\SO(1,1))$, which itself identifies with Anti de Sitter space $\mathsf{AdS}^2$ as a coincidence of low dimensionality.

\begin{figure}
\centering\includegraphics[width=0.65\textwidth]{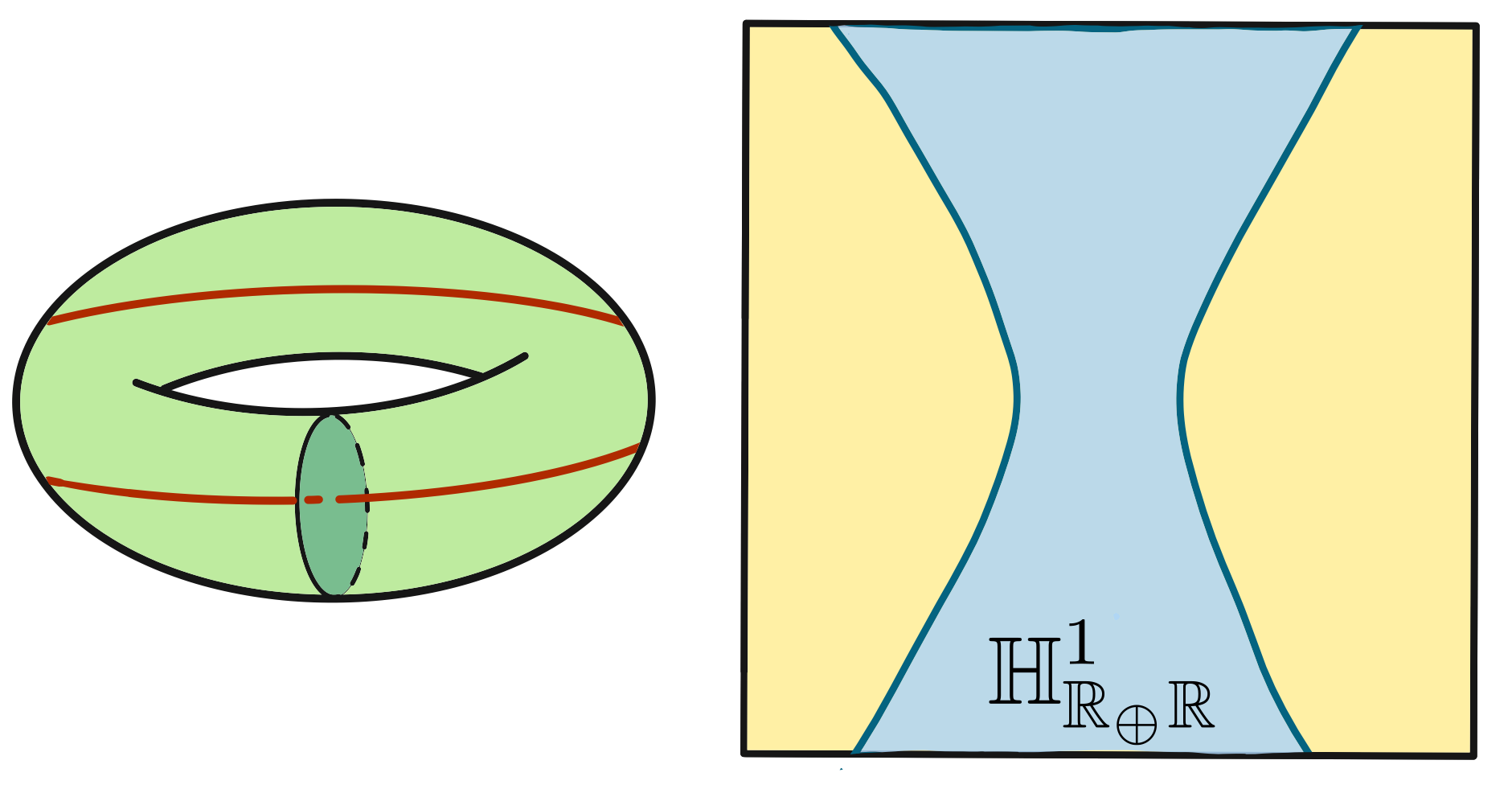}
\caption{The solid torus foliated by cosets of $\SO(1,1)$ and the familiar model of $\mathsf{dS}^2=\mathsf{AdS}^2$ as a subgeometry of $\RP^2$.}	
\end{figure}

\noindent
Again in higher dimensions this connection breaks down, and $\Hyp_{\R\oplus\R}^n$ is not isomorphic to either de Sitter or Anti-de Sitter space of the appropriate dimension.
Instead, $\Hyp_{\R\oplus\R}^n$ identifies in general with another geometry constructed from $\RP^n$ and its dual.
This geometry, \emph{point-hyperplane projective space} is explored on in the next section, and provides the means of using transitional geometry to build a connection between complex hyperbolic and real projective geometry.

\section{Point-Hyperplane Projective Space}
\label{sec:SDRP_Geo}
\index{Geometries!Point-Hyperplane Projective Space}
\index{Projective Geometry! Point-Hyperplane Projective Space}

In this final section, we construct a geometry of a different flavor, built directly from the projective geometry of a real vector space $V$.
The dual space $V^\vee$ is the vector space of linear functionals $V^\vee=\Hom(V,\R)$.
Evaluation provides a natural pairing on $V^\vee\times V\to\R$ by $(\phi,v)\mapsto \phi(v)$.
The action of $\GL(V)$ on $V$ by left multiplication gives a left action on $V^\vee$ respecting the pairing; that is for all $X$ in $\GL(V)$ and all $(\phi,v)$ we have $(X.\phi)(Xv)=\phi(v)$ by precomposition with the inverse.

Expressed in a basis for $V$ and the corresponding dual basis for $V^\vee$, the action of $X\in\GL(V)$ on $V^\vee$ is represented by left multiplication by the inverse transpose $X^{-T}$.
This gives an action of $\GL(V)$ on $V^\vee\times V$ by
$X.(\phi,v)=(X^{-T}\phi,Xv)$.
This action leaves the level sets $\mathcal{L}_c:=\set{(\phi,v)\in V^\vee\times V\;\mid\; \phi(v)=c}$ of the pairing invariant, and in fact acts transitively on them.

\begin{calculation}
Given two vectors $\phi,v$ such that $\phi^Tv=1$	 there is a matrix $X$ such that the first column of $X$ is $v$ and the first row of $X\inv$ is $\phi$.
\end{calculation}
\begin{proof}
Let $Q$ be any invertible matrix with $v$ as the first column.  Then notice that the first row of $Q\inv$ has inner product 1 with $v$ and all other rows are orthogonal to $v$, as $QQ\inv=I$.  The rows of $Q\inv$ (thought of as column vectors) which we will denote $\{r_i\}$ form a basis for $V$, and so we may express $\phi$ in this basis
$\phi=\sum_i \alpha_i r_i$ for $\alpha_i\in\R$.
But as $\phi^Tv=1$, we have
$1=\phi^Tv=\left(\sum_i \alpha_i r_i\right)^Tv=\sum_i \alpha _ir_i^Tv=\alpha_1$.
Thus in coordinates, $\phi=r_1+\alpha_2r_2+\cdots+\alpha_n r_n$.  We now let $A$ be the identity matrix with the first row replaced with the expression of $\phi$ in basis $\{r_i\}$.
Then $AQ\inv$ has as its first row $\phi$, and  $(AQ\inv)\inv=QA\inv$ still has $v$ as its first column.

$$A=\pmat{
1&\alpha_2 &\alpha_3&\cdots &\alpha_n\\
0&1&0&\cdots &0\\
\vdots &\vdots &\vdots &\vdots &\vdots\\
0&\cdots &\cdots &0&1
}$$

\end{proof}

\begin{corollary}
The action of $\GL(V)$ on $V^\vee\times V$ is transitive on the 1-level set of the pairing $(\phi,v)\to \phi(v)$.	
\end{corollary}
\begin{proof}
Choose a basis for $V$ and take the corresponding dual basis for $V^\vee$.
The points of $\mathcal{L}_1$ are all of the pairs of column vectors $(\phi,v)	$ with $\phi^Tv=1$.  In particular, the vectors $b_1$ and $b_1^\vee$ make this list, both represented as $(1,0,\ldots,0)^T$.  The orbit of the point $(b_1^\vee,b_1)$ is the collection
$\set{(X^{-T}b_1^\vee,Xb_1)\;\mid\; X\in\GL(V)}$.
But $Xb_1$ is simply the first column of $X$ and $X^{-T}b_1^\vee$ is the first column of $X^{-T}$, which is the first row of $X\inv$.  The previous proposition tells us then that if $(\phi,v)$ is any point of $\mathcal{L}_1$ there is some $X$ such that $(X^{-T}b_1^\vee,Xb_1)=(\phi,v)$ and so we are done.
\end{proof}

\subsection{Group Space Description}

By the calculation above, the action of $\GL(V)$ on any nonzero level set of the pairing is transitive, and defines a geometry.
   
\begin{definition}
The \emph{point-hyperplane} geometry of $V$ is given by the Group - Space pair $(\GL(V),\mathcal{L}_1)$ described above.
\end{definition}

\noindent
As in the construction of hyperbolic space, we may view this geometry either as a fixed level set together with the action of $\GL(V)$, or build a model projectively.
The action of $\GL(V)$ on the coordinate-wise projectivization $\pi:V^\vee\times V\to \P(V^\vee)\times \P(V)$ factors through the quotient $\GL(V)\to\PGL(V)$ and so we have a well-defined action 
$\PGL(V)\acts \P(V^\vee)\times \P(V)$,
$[X].\left([\phi],[v]\right):=\left([X^{-T}\phi],[Xv]\right)$.
After projectivization however, the notion of level set for any particular value fails to remain well-defined.

\begin{lemma}
\label{lem:PH_Proj_Domain}
If $r\neq 0$, $\pi\mathcal{L}_1=\pi\mathcal{L}_r$ for $\pi$ the the projectivization $V^\vee\times V\to \P V^\vee\times \P V$.
\end{lemma}
\begin{proof}
Let $(\phi,v)\in\mathcal{L}_1$ and $\pi(\phi,v)=([\phi],[v])$ its image in $\P V^\vee\times \P V$.  Given any $r\in\R^\ast$ we may choose the representative $(r\phi,v)$ of $([\phi],[v])$ and note that this is a point of $\mathcal{L}_r$.  The map $\mu_r:\mathcal{L}_1\to\mathcal{L}_r$ given by $(\phi,v)\to (r\phi,v)$ is clearly a homeomorphism, but following with $\pi$ leaves the projection unchanged: thus $\pi(\mathcal{L}_r)=\pi\circ\mu_r(\mathcal{L}_1)=\pi(\mathcal{L}_1)$.
\end{proof}

\noindent
Thus, $\P V^\vee\times \P V$ decomposes naturally into two subsets: the projectivization of the zero level set, and the nonzero ones.

\begin{corollary}
$\P V^\vee\times \P V=\pi(\mathcal{L}_0)\sqcup \pi(\mathcal{L}_1)$	
\end{corollary}
\begin{proof}
The evaluation pairing sends each point of $V^\vee\times V$ to a real number and so we may write 
$V^\vee\times V=\bigcup_{r\in\R}\mathcal{L}_r$
Applying $\pi$ to both sides gives
$\pi(V^\vee\times V)=\P V^\vee\times \P V=\pi\left(\bigcup\mathcal{L}_r\right)=\bigcup_{r\in\R}\pi(\mathcal{L}_r)$,
but the proposition above tells us that for all $r\in\R^\ast$, $\pi(\mathcal{L}_r)$ coincide, and so this union is really just 
$\P V^\vee\times \P V=\pi(\mathcal{L}_0)\cup\pi(\mathcal{L}_1)$.
It remains only to see that this union is disjoint.  If $([\phi],[v])\in\pi(\mathcal{L}_0)$ then there is some representative for which $\phi(v)=0$  But this clearly holds for all such representatives as $r\phi(sv)=(rs)\phi(v)=0$ and so if $\psi(w)=1$ then $([\psi],[w])\not\in\pi(\mathcal{L}_0)$.
\end{proof}

\noindent 
This provides a second group-space description of the same geometry:

\begin{definition}
The point-hyperplane geometry of $V$ is given by $(\GL(V), \P V^\vee\times \P V\smallsetminus \pi(\mathcal{L}_0))$, as this complement of the zero locus of the pairing is homeomorphic to $\mathcal{L}_1$ as above.
\end{definition}

\noindent
It is this second description, as a subset of $\P V^\vee\times \P V$, from which the name \emph{point-hyperplane geometry} is derived.
Projective classes of linear functionals are determined by their kernels, which are hyperplanes in $\P V$ under projectivization.
Thus, a point in $\P V^\vee\times\P V$ can be thought of as a pair of a projective point and hyperplane.
The points which evaluate to $0$ under the pairing are exactly the pairs $(\phi,v)$ such that $v\in\ker \phi$, that is $[v]$ lies on the hyperplane determined by $[\phi]$.
This gives a geometric description of the geometry, completely in terms of the intrinsic geometry of $\P V$.
We state this for $V=\R^{n+1}$ below.

\begin{definition}
The \emph{point-hyperplane geometry} of $\RP^n$ has as underlying space the collection of all pairs $(H,p)$ of hyperplanes $H\subset \RP^n$ and points $p\in \RP^n$ such that $p\not\in H$.
The automorphisms of this geometry are the full automorphism group of $\RP^n$.	
\end{definition}

\subsection{Equivalence with $\Hyp_{\R\oplus\R}^n$}

Point-hyperplane projective space seems to be a geometry of a very different flavor than the hyperbolic-analogs that we have been discussing in the rest of this chapter.
The reason for introducing it is, of course, that there is a close relationship - unitary geometry over $\R\oplus\R$ is locally isomorphic to point-hyperplane projective space!
Thus we can learn a lot about $\Hyp^n_{\R\oplus\R}$ from this easier to study model in $\RP^n\times\RP^n$.
Hints of this isomorphism are already out there: unitary groups over $\R\oplus\R$ are isomorphic to the general linear groups over $\R$, and the unit spheres for Hermitian forms over $\R\oplus\R$ are cut out by equations isomorphic to the pairing $\R^n\times(\R^n)^\vee\to\R$ after a linear change of coordinates.
Below, we use the conjugate model $(\U(n;\R\oplus\R),\USt(n;\R\oplus\R))$ for $\Hyp_{\R\oplus\R}^n$ to avoid conjugacy and negative signs everywhere.

\begin{calculation}
The change of coordinates $f\colon \R^n\times\R^n\to\R^n\times\R^n$ 
by $(\phi_i,v_i)=(x_i+y_i,x_i-y_i)$ identifies the unit sphere $\mathcal{S}_q(-1)$ of radius $-1$ for the Hermitian form $q$ on $(\R\oplus\R)^n$ with the level set $\mathcal{L}_1$ of the pairing on $(\R^n)^\vee\times(\R^n)$.
\label{calc:HypRR_Domain_Calc}
\end{calculation}
\begin{proof}
In the coordinates $(\phi,v)$ the $1$ level set of the dual pairing on $\R^n\times\R^n$ is $\phi(v)=\sum_{i=1}^n \phi_i v_i=1$.
In the coordinates $\vec{z}=\vec{x}+\lambda \vec{y}$ on $(\R\oplus\R)^n$, the sphere of radius $-1$ is $\sum_{i=1}^{n-1} x_i^2-y_i^2-(x_{n}^2-y_{n}^2)$.
We define the change of coordinates $f\colon \R^n\times\R^n\to\R^n\times\R^n$ 
by $(\phi_i,v_i)=(x_i+y_i,x_i-y_i)$, taking $\mathcal{L}_1$ to $\mathcal{S}_{\R\oplus\R}(n,1)$.
\end{proof}

\noindent 
This change of coordinates can actually be interpreted wholly internally to the geometry of $\Hyp_{\R\oplus\R}^n$ as taking a point $\vec{x}+\lambda\vec{y}$ and expressing it in terms not of $\{1,\lambda\}$ but the basis of orthogonal idempotents $\{e_+,e_-\}$.

\begin{proposition}
Point-Hyperplane projective geometry is locally isomorphic to the unitary geometry over $\R\oplus\R$.	
\end{proposition}
\begin{proof}
Recall from Proposition \ref{prop:U(R+R)_iso_type} that the unitary group $\U(n;\R\oplus\R)$ can be described in the basis of idempotents $\{e_+,e_-\}$ as $\U(n,\R\oplus\R)=\{Xe_++X^{-T}e_-\mid X\in\GL(n;\R)\}$.
Thus, the action of $\U(n;\R\oplus\R)$ on $(\R\oplus\R)^n$ is an action of $\GL(n;\R)$ on $\R^n\times\R^n$.
It's easy to see in coordinates that this action is precisely the same as the twisted diagonal action of $\GL(n;\R)$ on $\R^n\times(\R^n)^\vee$ defining point-hyperplane projective space.
Indeed, let $p=ve_++we_-\in(\R\oplus\R)^n$, and $A=Xe_++X^{-T}e_-\in\U(n;\R\oplus\R)$.
Then $A.p=(Xe_++X^{-T}e_-).(ve_++we_-)=Xve_++X^{-T}we_-$, which is identical to the formula defining the action at the beginning of Section \ref{sec:SDRP_Geo}.
Recalling Calculation \ref{calc:HypRR_Domain_Calc}, not only do both geometries share the same linear action of $\GL(n;\R)$, but the domains (before projectivization) are diffeomorphic.

Thus it remains only to consider the effect of projectivization in both cases.
The norm $x\mapsto x\bar{x}$ on $\R\oplus\R$ is surjective onto $\R$, and the units compromise the four connected components of the coordinate axis complement.
Full projectivization, that is quotienting the unit sphere 
$\mathcal{S}_{\R\oplus\R}(n,1)$ in $(\R\oplus\R)^n$ by the action of $(\R\oplus\R)^\times$ identifies the result with a subset of $\RP^{n-1}\times\RP^{n-1}$ as the action of a unit $u=u_1e_++u_2e_-$ on a point $ve_++we_-$ acts component-wise, so $ve_++we_-$ projectivizes to $[v]e_++[w]e_-$ as $u_1,u_2$ vary independently over the nonzero reals.
This is precisely the domain of point-hyperplane projective space, as Lemma \ref{lem:PH_Proj_Domain} implies here too that the projective image of any nonzero level set of $\sum_i x_i\overline{x_i}$ is the complement of the zero level set in $\RP^P{n-1}\times\RP^{n-1}$.

This is not \emph{precisely} the geometry $\Hyp_{\R\oplus\R}^n$ in Definition \ref{def:RRHyp}, but rather the two-fold quotient of it given by the projective cone model.
This is because, as noted previously, we chose in the definition of $\Hyp_{\R\oplus\R}$ to quotient only by the action of $\U(\R\oplus\R)$, which are the elements of norm $1$.
This is equivalent to quotienting by the action of elements in $(\R\oplus\R)^\times$ of positive norm, instead of the index-2 supergroup of all units.
\end{proof}

\chapter{The Transition $\mathbb{H}_{\R[\sqrt{\delta}]}^n$}
\label{chp:HC_To_HRR_Transition}
\index{Limits of Geometries!$\mathbb{H}_\C$}
\index{Geometric Limits $\mathbb{H}_\C$}
\index{Complex Hyperbolic Geometry!Transition}

The geometries $\Hyp_\C^n$, $\Hyp_{\R_\ep}^n$ and $\Hyp_{\R\oplus\R}^n$ defined in Chapter \ref{chap:HC_and_HRR} are deeply related to one another because of a strong relationship between their underlying algebras of definition.
The algebra $\R_\ep$ is a common degeneration of the algebraic structures of $\C$ and $\R\oplus\R$, and this chapter exploits this relationship to show this carries over to the geometries.

\begin{theorem}
The geometry $\Hyp_{\R_\ep}^n$ is a common degeneration of  $\Hyp_\C^n$ and $\Hyp_{\R\oplus\R}^n$.
\label{thm:Hyp_Rep_Cts_Degen}
\end{theorem}

First, we make explicit the connection between the algebras bellow.

\begin{definition}
For each $\delta$, the algebra $\Lambda_\delta=\R[\lambda]/(\lambda^2=\delta)$ is a two dimensional algebra over $\R$, isomorphic  to $\C$ when $\delta<0$, to $\R_\ep$ for $\delta=0$ and to $\R\oplus\R$ for $\delta>0$.
\end{definition}

\begin{observation}
The algebraic structure on $\R^2=\R\oplus\lambda_\R$ induced by identification with $\Lambda_\delta$ varies continuously with $\delta$.
\end{observation}
\begin{proof}
Each $\Lambda_\delta$ is a quadratic extension of $\R$, and thus has underlying vector space $\R^2$.
The multiplication of each $\Lambda_\delta$, defined by $\lambda^2=\delta$, is given in these coordinates as follows.
For each $\delta\in\R$ we have$(a,b)\times_\delta (c,d)=(ac+\delta bd, ad+bc)$.
This defines the collection of algebra multiplications as a 1-parameter family of maps $\times_\delta\colon\R^2\times\R^2\to\R^2$, which fit together as $\delta$ varies to a map $\times\colon \R^2\times\R^2\times\R\to\R^2$ defined by $((a,b),(c,d),\delta)\mapsto (a,b)\times_\delta(c,d)$.
\end{proof}

\noindent
This family of algebras was already used in the work of Danciger \cite{Danciger11Ideal} to describe the special case of the transition from $\Hyp^3$ to $\mathsf{AdS}^3$, using the identification $\Isom(\Hyp^3)=\SL(2;\C)$ and $\Isom(\mathsf{AdS}^3)=\SL(2,\R)\times\SL(2;\R)=\SL(2;\R\oplus\R)$.

\section{Notions of Continuity}

Following exactly the methods of Chapter \ref{chap:HC_and_HRR}, it is easy to construct the analogs $\Hyp_{\Lambda_\delta}^n$ of hyperbolic geometry over the algebra $\Lambda_\delta$.

\begin{definition}[$\Hyp_\Lambda^n$: Group - Space]
$\Lambda$ Hyperbolic space is the geometry given by the action of $\U(n,1;\Lambda)$ on the projectivized unit sphere of radius $-1$ for $q$ in $\Lambda^{n+1}$;
$\Hyp_\Lambda^n=(\U(n,1;\Lambda),\mathcal{S}_{\Lambda}(n,1)/\U(\Lambda))$.
\end{definition}

\begin{definition}[$\Hyp_\Lambda^n$: Automorphism - Stabilizer]
Let $\USt(n,1;\Lambda)=\smat{\U(n;\Lambda)&\\&\U(\Lambda)}$.  Then
$$\Hyp_\Lambda^n=\left(\U(n,1;\Lambda),\USt(n,1;\Lambda)\right).$$	
\end{definition}

The first step in proving Theorem \ref{thm:Hyp_Rep_Cts_Degen} is to define what we mean by a \emph{degeneration}, or more generally a \emph{continuous path} of homogeneous spaces in this context.
In the work of Danciger, and further work on transitional geometry by Cooper, Danciger and Wienhard among others, continuity is formalized by embedding all geometries under consideration into the \emph{space of subgeometries} of some large, fixed ambient geometry.
This approach has sufficed thus far in this thesis as well, as all geometries considered have naturally arisen as subgeometries of real projective space.
The problem here is that our geometries $\Hyp_{\Lambda_\delta}^n$ as defined above and studied in Chapter \ref{chap:HC_and_HRR} have each been constructed indpendently, and not as subgeometries of some ambient space\footnote{We could have stopped to define projective space over algebras here, and realized that our models of $\Hyp_{\Lambda_\delta}^n$ actually are all subgeometries of the corresponding projective space $\Lambda_\delta\mathsf{P}^n$.
However this would do nothing to solve the present problem, as these spaces are not constant in $\delta$ and in fact undergo their own geometric transition as $\delta$ passes through $0$.
To utilize the standard notion of continuity given in Chapter \ref{chp:Limits_of_Geos}, we need each $\Hyp_{\Lambda_\delta}^n$ to simultaneously embed in the same ambient space.
}.
As an alternative to attempting to construct some ambient geometry in which all of the $\Hyp_{\Lambda_\delta}^n$ simultaneously embed, it is more useful to take this as an opportunity to consider generalizations of the framework reviewed in Chapter \ref{chp:Limits_of_Geos} to acomodate this, and future situations where there is no canonical ambient geometry.
This chapter provides two potential such generalizations, and shows that in each case, the family $\Hyp_{\Lambda_\delta}^n$ of geometries provides a transition from $\Hyp_{\C}^n$ to $\Hyp_{\R\oplus\R}^n$ through $\Hyp_{\R_\ep}^n$.

\subsection{Representations into an Ambient Lie Group}

The first is a very mild alteration of the usual framework - the main utility of considering a collection of subgeometries of some ambient geometry is so that we may use the Chabauty space of the ambient automorphism and stabilizer subgroups to define continuity.
Recall in that in the Automorphism-Stabilizer formalism, we say a path $(H_t, C_t)$ of subgeometries of $(G,K)$ is continuous if the assignment $t\mapsto (H_t, C_t)$ is continuous into $\Cl(G)\times\Cl(K)$.
Weakening this, we drop the requirement that for all $t$, the stabilizing subgroups $C_t$ are subgroups of some fixed $K<G$, and consider continuity only with respect to a fixed Lie group $G$.

\begin{definition}
Let $G$ be a fixed Lie group, and for each $t$ let  $(H_t, C_t)$ be a Klein geometry in the Automorphism-Stabilizer formalism, with $H_t<G$.
Then $(H_t,C_t)$ is a continuous path of geometries if the map $t\mapsto (H_t,C_t)$ is continuous in the Chabauty space $\Cl(G)\times\Cl(G)$.
\end{definition}

This allows us to speak of continuity of a path of homogeneous geometries, given only embeddings of their automorphism groups into some fixed Lie group $G$.
This is a much easier demand to satisfy, in many instances it is possible to construct simultaneous embeddings into some large enough $\GL(n;\R)$ via linear representations.
Using this formalism, we prove the following in Section \ref{sec:HC_To_HRR_Conjugacy}.

\begin{theorem}
\label{thm:ConjLim_HRR}
For each $\delta$, let $\iota_\delta\colon\GL(n;\Lambda_\delta)\to\GL(2n;\R)$ be the representation arising from thinking of the $\Lambda_\delta$ module $\Lambda_\delta^n$ as a real vector space.
Then the assignments
$$\delta\mapsto\iota_\delta(\U(n,1;\Lambda_\delta)
\hspace{1cm}\delta\mapsto \iota_\delta(\USt(n,1;\Lambda_\delta))$$
are continuous as functions $\R\mapsto \Cl(\GL(2n;\R))$.
\end{theorem}

\subsection{1-Parameter Families of Lie Groups}

The second approach is a more radical departure from the existing literature in transitional geometry, and does away with the fixed ambient Lie group $G$.
Indeed, the spirit of the previous definition was that \emph{a continuous path of geometries is a continuous path of automorphism groups together with a continuous path of stabilizer subgroups}, and the ambient group $G$ exists only for convenience, to provide a space in which to formalize this continuity.
The notion of a fiber bundle of groups is too restrictive for the study of transitional geometry, as many interesting transitions involve automorphism groups changing homeomorphism, or even homotopy type along the way.
The correct notion of a parameterized family of groups is formalized through the theory of Lie groupoids, and has already been used Riemannian geometry to understand certain transitions \cite{BettiolPS14} .

\begin{definition}
A \emph{groupoid} is a category where all morphisms are isomorphisms.
That is, a groupoid $G$ consists of a set $\mathsf{Ob}(G)$ of objects, and a set $\mathsf{Mor}(G)$ of morphisms such that each $f\in \mathsf{Mor}(G)$ has an inverse $f\inv\in\mathsf{Mor}(G)$.
\end{definition}

\noindent
A groupoid with one object $\{\star\}$ is a group, with the elements of the group being the morphisms in $\Hom(\star,\star)$.

\begin{definition}
A Lie groupoid is a groupoid $G$ where the set of objects and the set of morphisms both have the structure of smooth manifolds, and the maps $s,t\colon\mathsf{Mor}(G)\to\mathsf{Ob}(G)$ sending a morphism $f$ to its source $s(t)$ and target $t(g)$ are submersions with respect to the given smooth structures.	
\end{definition}

\begin{figure}
\centering\includegraphics[width=0.5\textwidth]{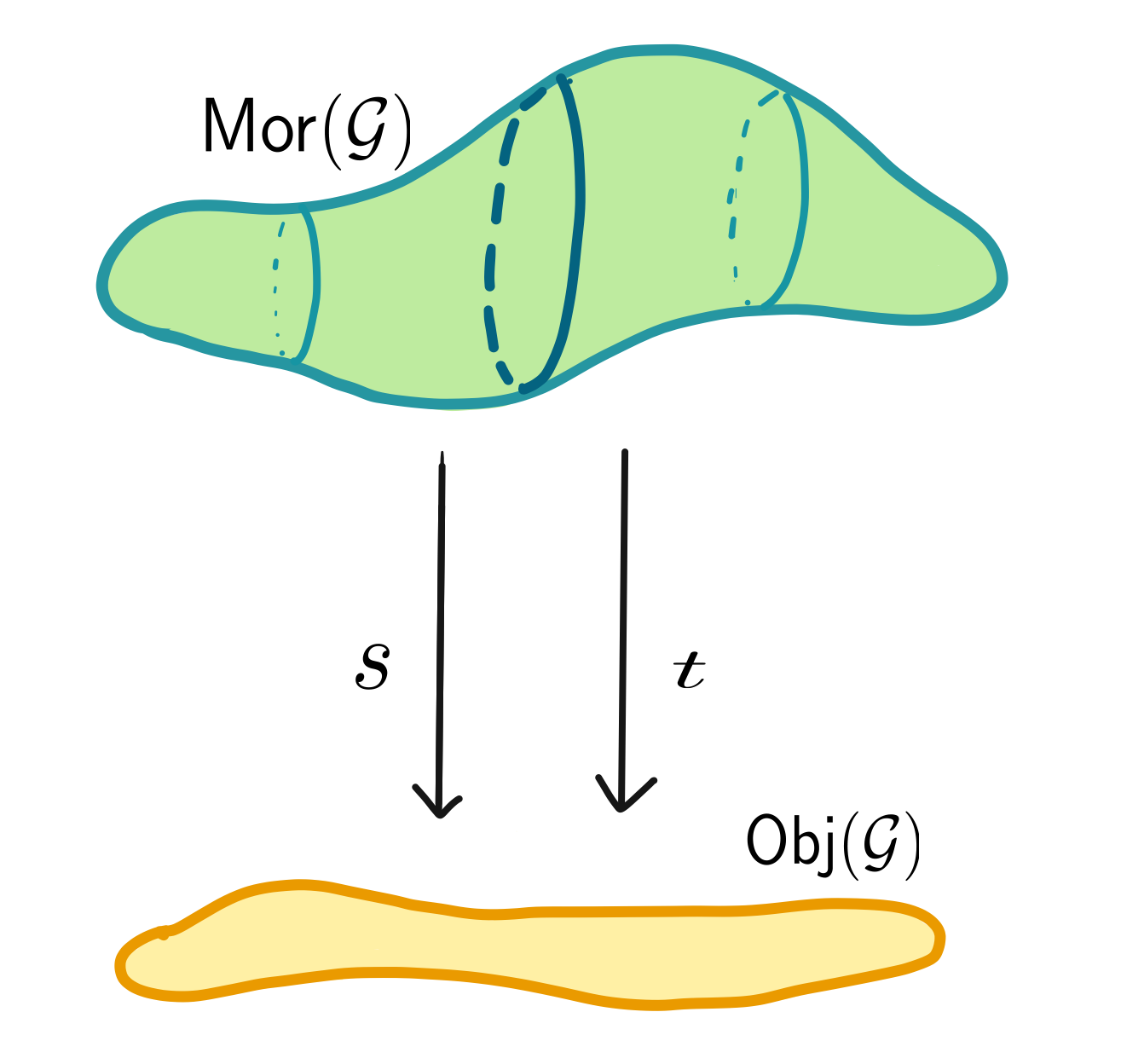}
\caption{A Lie Groupoid, schematically.}	
\end{figure}

\noindent
Similarly, a Lie groupoid with one object is a Lie group $G=\Hom(\star, \star)$, where the source and target maps are both the constant map $G\to\star$.
When the space of objects has a more complex topology, a Lie groupoid is no longer a group, and two morphisms $g,h\in\fam{G}$ can only be composed if the target of one is the source of the other, $t(g)=s(h)$.
Thus, the fibers of the source and target maps, which are smooth submanifolds of $\fam{G}$ by requirement that $s,t$ be submersions, are actually groups only in the case that $s=t$.

\begin{definition}[1 Parameter Family of Groups]
A \emph{one parameter family of Lie groups} is a Lie groupoid $\fam{G}$ with $\mathsf{Ob}(\fam{G})=\R$ and equal source, target maps $s=t\colon G\to\R$.
The fibers $\fam{G}_\delta=s\inv(\delta)=t\inv(\delta)$ each come equipped with the structure of a Lie group, by restricting the composition operation of the groupoid $\fam{G}$.	
\end{definition}

\begin{definition}
A collection $G_\delta<\GL(n;\Lambda_\delta)$ varies continuously if $\bigcup_\delta G_\delta\times\{\delta\}$ is a 1-parameter family of groups.
\end{definition}

\noindent
This provides an ambient space to work in (the bundle of matrix algebras $\M(n;\Lambda_\delta)$) without requiring there be any fixed group or algebra containing each member of the family individually.
Using this formalism, we also show that the geometries $\Hyp_{\Lambda_\delta}^n$ vary continuously.

\begin{theorem}
\label{thm:ParamFam_HRR}
The collection $\fam{U}(n,1;\Lambda_\R)=\bigcup_\delta\in\R \U(n,1;\Lambda_\delta)\times\{\delta\}$, and $\fam{USt}(n,1;\Lambda_\delta)=\bigcup_{\delta\in\R}\USt(n,1;\Lambda_\delta)\times\{\delta\}$ form 1-parameter families of Lie groups, 
when equipped with the subspace topology coming from $\bigcup_{\delta\in\R}\M(n+1;\Lambda_\delta)\times\{\delta\}$.
\end{theorem}

\section{The Transition as a Conjugacy Limit}
\label{sec:HC_To_HRR_Conjugacy}
\index{Conjugacy Limit!$\mathbb{H}_\C$ and $\mathbb{H}_{\R\oplus\R}$}

Underlying the algebra $\Lambda_\delta$ is the real vector space $\R\oplus\lambda\R$ where we only remember how to multiply elements of $\Lambda_\delta$ by real scalars.  
Stemming from this if we forget how to multiply by $\lambda$ then $\Lambda_\delta$ modules $\Lambda_\delta^n$ give rise to $2n$-dimensional real vector spaces, $\Lambda_\delta^n=(\R\oplus\lambda\R)^n$.
As the action of $\End(n,\Lambda_\delta)$ on $\Lambda_\delta^n$ is $\Lambda_\delta$ linear, it is clearly $\R$-linear and gives a representation $\End(n,\Lambda_\delta)\to\M(2n,\R)$.

\begin{observation}
The $\R$-linear action of $\Lambda_\delta$ on $\Lambda_\delta$ viewed as the real module $\R^2$	 is $\iota_\delta\colon\Lambda_\delta\to \M(2;\R)$ given by 
 $a+\lambda b\mapsto \smat{a& \delta b\\b&a}$.
\end{observation}

\begin{observation}
\label{obs:iota_delta_blockwise}
Viewing $\Lambda_\delta=(\R\oplus\lambda\R)^n$ as the real vector space $\R^{2n}$ the $\R$-linear action of $\M(n;\Lambda_\delta)$ on $\Lambda_\delta^n$ is given by the homomorphism  $\M(n;\Lambda_\delta)\to\M(2n;\R)$ acting component-wise by $\iota_\delta$:
$(A)_{ij}\mapsto \iota_\delta((A)_{ij})$.
\end{observation}

\begin{example}
$$\pmat{
a_1+\lambda a_2 & b_1+\lambda b_2 
\\
 c_1+\lambda c_2 & d_1+\lambda d_2
 }
 \mapsto 
 \pmat{
 \pmat{a_1 &\delta a_2\\a_2 &a_1}  
 & 
 \pmat{b_1 &\delta b_2\\b_2 &b_1}
 \\ 
 \pmat{c_1 &\delta c_2\\c_2& c_1} 
 &
 \pmat{d_1 &\delta d_2\\d_2 &d_1}
 }
 \mapsto
 \pmat{
 a_1 &\delta a_2 & b_1 &\delta b_2
 \\ 
 a_2 &a_1 &b_2 &b_1
 \\
 c_1 & \delta c_2 & d_1 &\delta d_2
 \\
 c_2 & c_1 & d_2 & d_1
 }
 $$	
\end{example}

\begin{remark}
We denote this map by $\iota_\delta\colon \M(n;\Lambda_\delta)\to\M(2n;\R)$ as well.	
For each $\delta$, the matrix algebra $\M(n;\Lambda_\delta)$ embeds into $\M(2n;\R)$, so $\GL(2n;\R)$ can be used as a universal containing group for all of linear groups over $\Lambda_\delta$.

\end{remark}

\noindent
The remainder of this section is devoted to the proof of Theorem \ref{thm:ConjLim_HRR}, using a collection of standard techniques.
First, we note just as the isomorphism type of $\Lambda_\delta$ depends only on the sign of $\delta$; the conjugacy class of $\iota_\delta(\GL(n;\Lambda_\delta))$ does as well.

\begin{proposition}
The images $\iota_\delta(\M(n;\Lambda_\delta))$ are conjugate in $\M(2n;\R)$ iff $\mathsf{sgn}(\delta)=\mathsf{sgn}(\mu)$.
\end{proposition}
\begin{proof}
We consider the case $n=1$ of the algebra itself; as this suffices by Observation \ref{obs:iota_delta_blockwise}.
When $\mathsf{sgn}(\delta)\neq\mathsf{sgn}(\mu)$ then $\Lambda_\delta$ is not even isomorphic to $\Lambda_\mu$, and so clearly their respective images in $\M(2;\R)$ are not conjugate.
Thus assume $\mathsf{sgn}(\delta)=\mathsf{sgn}(\mu)$, and consider $1,\lambda$ as elements of each.
The image of $1$ is the identity $I_2\in\M(2;\R)$ under each of $\iota_\delta,\iota_\mu$ but the image of $\lambda$ differs,
$$\iota_\delta(\lambda)=\pmat{0&\delta \\1&0},\hspace{0.5cm}\textrm{and}\hspace{0.5cm}
\iota_\mu(\lambda)=\pmat{0&\mu \\1&0}.
$$
As $\delta,\mu$ are of the same sign, $\mu/\delta$ is positive.
The matrix $C=\smat{1&0\\0&\sqrt{\tfrac{\mu}{\delta}}}$ conjugates $\iota_\delta(\lambda)$ to $\iota_\mu(\lambda)$, and thus by linearity $C\iota_\delta(\Lambda_\delta)C\inv=\iota_\mu(\Lambda_\mu)$.
In higher dimensions, the correct conjugating matrix is simply block diagonal with copies of $C$, or 
$$C=\diag\left(1,\sqrt{\tfrac{\mu}{\delta}},1,\sqrt{\tfrac{\mu}{\delta}},\ldots, 1, \sqrt{\tfrac{\mu}{\delta}}\right)$$. 
\end{proof}

\begin{remark}
We fix the notation $C_\delta=\diag(1,\sqrt{|\delta|})$ and note that for $\delta<0$, $C_\delta$ conjugates the standard embedding of $\C\subset\M(2;\R)$ to $\iota_\delta(\Lambda_\delta)$, and when $\delta>0$ the same $C_\delta$ conjugates the standard embedding of $\R\oplus\R\subset\M(2;\R)$ to $\iota_\delta(\Lambda_\delta)$.
\end{remark}

\begin{corollary}
The Lie groups $\iota_\delta(\GL(n;\Lambda_\delta))$ and $\iota_\mu(\GL(n;\Lambda_\mu))$ are conjugate in $\GL(2n;\R)$ if and only if $\mathsf{sgn}(\delta)=\mathsf{sgn}(\mu)$.
\end{corollary}

\noindent
It will be useful to describe the map $\iota_\delta$ on a basis for $\M(n;\Lambda_\delta)$ to aid in future Lie algebra computations.

\begin{definition}
For each $i,j\in\{1,\ldots, n\}$ let $E_{ij}\in\M(n;\R)$ be the matrix with all zeroes except a $1$ in the $ij^{th}$ position.
Then the collection $\fam{E}=\{E_{ij},\lambda E_{ij}\}_{1\leq i,j\leq n}$ forms a basis for $\M(n;\Lambda_\delta)$.

Define $R_{jk}=E_{2j-1,2k-1}+E_{2j,2k}\in\M(2n,\R)$ to be built out of $2\times 2$ blocks, all zero except for the identity block in the $jk^{th}$ position.
Define $\mathcal{I}_{jk}E_{2j,2k-1}+\delta E_{2j-1,2k}\in\M(2n,\R)$ 
similarly, except with the $jk^{th}$ block given by $\smat{0&\delta\\1&0}$.
For example, consider $R_{23}$ and $I^\delta_{23}$ in $\M(6;\R)$:

$$R_{23}=
\pmat{
0&0&0&0&0&0\\
0&0&0&0&0&0\\
0&0&0&0&1&0\\
0&0&0&0&0&1\\
0&0&0&0&0&0\\
0&0&0&0&0&0
}
\hspace{1cm}
I_{23}^\delta=
\pmat{
0&0&0&0&0&0\\
0&0&0&0&0&0\\
0&0&0&0&0&\delta\\
0&0&0&0&1&0\\
0&0&0&0&0&0\\
0&0&0&0&0&0
}
$$

\end{definition}

\noindent
An easy calculation reveals that these are precisely the images of the basis $\{E_{jk},\lambda E_{jk}\}$ under the representation $\iota_\delta$.

\begin{calculation}
\label{calc:Image_of_M_Lambda_Basis}
$\iota_\delta(E_{jk})=R_{jk}$ and $\iota_\delta(\lambda E_{jk})	=I^\delta_{jk}$.
\end{calculation}

\noindent
Most importantly for our future use, the maps $\R\to\R^{(2n)^2}$ which sends $\delta\mapsto I^\delta_{jk}$ are continuous in $\delta$ and never pass through the zero matrix.
Thus for any fixed collection of $E_{jk}$ and $\lambda E_{jk}$, their images under $\iota_\delta$ span a continuously varying linear subspace of $\M(2n;\R)$ as $\delta$ varies.

\subsection{The Image of $\U(n,1;\Lambda_\delta)$}

The first step in analyzing the continuity of the path $\Hyp_{\Lambda_\delta}^n$ is to study the embeddings of the groups $\U(n,1;\Lambda_\delta)$ themselves.
We begin with the following surprising fact.

\begin{calculation}
\label{calc:Const_Lie_Alg}
For all $\delta$ the Lie algebra $\mathfrak{u}(n,1;\Lambda_\delta)$ is constant as a subset of $\M(n;\R)\oplus\lambda \M(n;\R)$.	
\end{calculation}
\begin{proof}
The elements of $\mathfrak{u}(n,1;\Lambda_\delta)$ are derivatives of paths $A_t\colon I\to \U(n,1;\Lambda_\delta)$ through the identity.
Let $X\in \mathfrak{u}(n,1;\Lambda)$ be the derivative of some path $A_t$ with 
$X=\tfrac{d}{dt}\mid_{t=0}A_t$.
Then as $A_t\in \U(n,1;\Lambda_\delta)$, for all $t$ we have $A_t^\dagger Q A_t=Q$.
Taking the derivative of both sides gives $(A'_t)^\dagger QA_t+A_tQA'	=0$, and evaluating at $t=0$ gives $X^\dagger Q+QX=0$.

Now $Q=\diag(I_n,-1)$ is a real matrix, and so all multiplication occurring in the expression $X^\dagger Q+QX$ is purely between one real number and one element of $\Lambda_\delta$.
Thus, at no point does the fact that $\lambda^2=\delta$ arise in the computation, and the Lie algebra $\mathfrak{u}(n,1;\Lambda_\delta)$ is independent of $\delta$, as a subset of $\M(n,\Lambda_\delta)=\M(n;\R)\oplus\lambda\M(n;\R)$.
\end{proof}

\noindent
Not only are the Lie algebras constant \emph{for different $\delta$ of the same sign} but rather $\mathfrak{u}(n,1;\Lambda_\delta)$ is constant \emph{for all $\delta$ in $\R$}.
This may appear counterintuitive as 
the Lie groups $\U(n,1;\Lambda_\delta)$ are clearly not constant; but this 
results from the exponential map,$\exp_\delta\colon\M(n;\Lambda_\delta)\to\M(n,\Lambda_\delta)$, not the Lie algebra, varying with $\delta$.

\begin{definition}
The exponential map $\exp_\delta	\colon \M(n;\Lambda_\delta)\to\M(n;\Lambda_\delta)$ is defined by 
$$\exp_\delta(X)=I+X+\frac{1}{2!}X^2+\cdots +\frac{1}{n!}X^n+\cdots$$
but with matrix multiplication using the multiplicative structure of $\Lambda_\delta$.
\end{definition}

\begin{example}
The 1-dimensional vector subspace $\lambda \R\subset\R\oplus\lambda\R$ is invariant as $\delta\in\R$ varies, but its image under the 
exponential map $\exp_\delta$ is a different subgroup for each $\delta$: in particular, 
$\exp_{-1}(\lambda t)=\cos(t)+\lambda\sin(t)$ and $\exp_{1}(\lambda t)=\cosh(t)+\lambda\sinh(t)$.	
\end{example}

\begin{figure}
\centering
\includegraphics[width=0.6\textwidth]{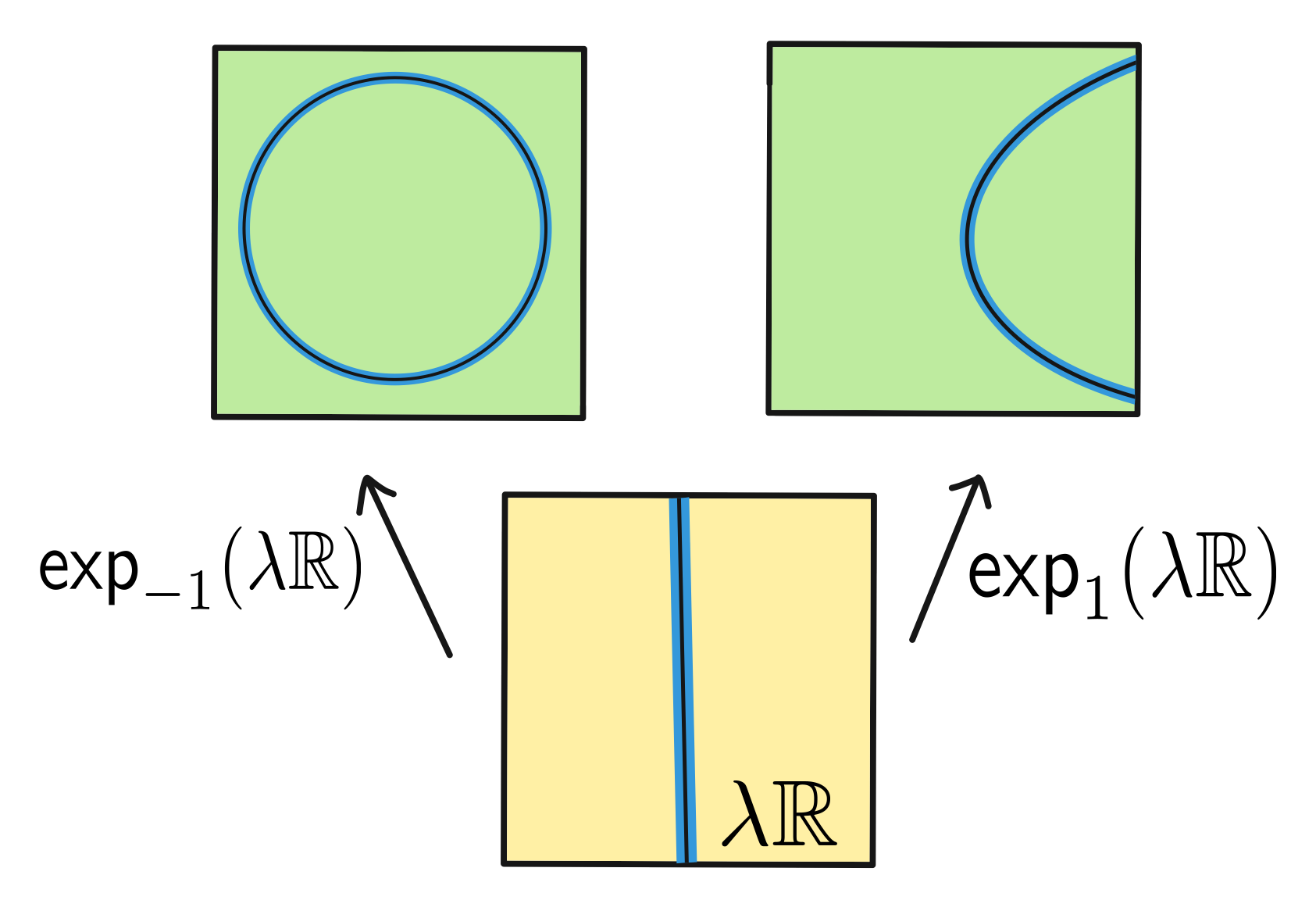}
\caption{The image of the same Lie algebra, $\lambda\R\subset\R\oplus\lambda\R$ under the exponential maps $\exp_{-1}$ and $\exp_{1}$.}	
\end{figure}

\noindent
To relate the matrix exponential of $\M(n;\Lambda_\delta)$ to the standard matrix exponential on $\GL(2n;\R)$ we exploit the fact that the representation $\iota_\delta$ is a homomorphism of algebras.

\begin{calculation}
$\iota_\delta\circ\exp_\delta=\exp\circ\iota_\delta$.	
\label{calc:Exp_Commuting}
\end{calculation}
\begin{proof}
This is a simple computation, showing that for all $N$ the partial sums of each side truncated at the $N^{th}$ degree are equal.
Let $X\in\M(n;\Lambda_\delta)$ be arbitrary.  
On the left hand side, we have
$$(\iota_\delta \exp_\delta(X))_N=\iota_\delta\left(
I+X+\frac{1}{2}X^2+\cdots +\frac{1}{N!}X^N\right)$$
Which, as $\iota_\delta$ is an algebra homomorphism, distributes through to give
$$I+\iota_\delta(X)+\frac{1}{2!}\iota_\delta(X)^2+\cdots+\frac{1}{N!}\iota_\delta(X)^N$$
which is precisely the $N^{th}$ truncation of the right hand side.
Thus, as the two are equal for every partial sum they are equal in the limit, and $\iota_\delta(\exp_\delta(X))=\exp(\iota_\delta(X))$.
\end{proof}

\begin{proposition}
The groups $\iota_\delta(\SU(n,1;\Lambda_\delta))$ and $\iota_\mu(\SU(n,1;\Lambda_\mu))$ are conjugate if and only if $\mathsf{sgn}(\delta)=\mathsf{sgn}(\mu)$.	
\end{proposition}
\begin{proof}
Let $\delta$ and $\mu$ be of the same sign, and let $\mathfrak{g}=\mathfrak{u}(n,1;\Lambda_\delta)=\mathfrak{u}(n,1;\Lambda_\mu)<\M(n,\R\oplus\lambda\R)$.
The connected component of the identity in $\U(n,1;\Lambda_x)$ is the group generated by the exponential image of $\mathfrak{u}(n,1;\Lambda_x)$, and so we have
$$\U(n,1;\Lambda_\delta)_0=\langle \exp_\delta(\mathfrak{g})\rangle
\hspace{1cm}
\U(n,1;\Lambda_\mu)=\langle \exp_\mu(\mathfrak{g})\rangle$$
Thus, the groups $\iota_x(\U(n,1;\Lambda_x)_0)$ are generated by the image of $\exp_x(\mathfrak{g})$ under $\iota_x$:
$$\iota_\delta(\U(n,1;\Lambda_\delta)_0)=\langle \iota_\delta\circ\exp_\delta(\mathfrak{g})\rangle
\hspace{1cm}
\iota_\mu(\U(n,1;\Lambda_\mu)_0)=\langle \iota_\delta\circ\exp_\mu(\mathfrak{g})\rangle
$$
Using Calculation \ref{calc:Exp_Commuting}, we may re-express these as
$$
\iota_\delta(\U(n,1;\Lambda_\delta)_0)=\langle \exp \iota_\delta(\mathfrak{g})\rangle
\hspace{1cm}
\iota_\mu(\U(n,1;\Lambda_\mu)_0)=\langle \exp \iota_\mu(\mathfrak{g})\rangle
$$
But as $\delta$ and $\mu$ are of the same sign, the embeddings $\iota_\delta$ and $\iota_\mu$ are conjugate, so in particular $\iota_\mu(\mathfrak{g})=C\iota_\delta(\mathfrak{g})C\inv$.
This conjugacy pulls out of the exponential map and the 'group generated by' to give
$$\iota_\mu\left(\U(n,1;\Lambda_\mu)_0\right)=\langle \exp( C\iota_\delta(\mathfrak{g})C\inv)\rangle=C\langle \exp\iota_\delta(\mathfrak{g})\rangle C\inv
=C\iota_\delta\left(\U(n,1;\Lambda_\delta)_0\right)C\inv$$
\end{proof}

\noindent
This allows us to study the path $\iota_\delta(\U(n,1;\Lambda_\delta)$ as a conjugacy limit inside of $\GL(2n;\R)$.
That the same holds for the stabilizers is an easy consequence of the following observation.

\begin{observation}
The stabilizer subgroup $\mathsf{USt}(n,1;\Lambda_\delta)$ is block diagonal, with unitary blocks $\U(n;\Lambda_\delta)$ and $\U(1;\Lambda_\delta)$.
By an analogous argument to Calculation \ref{calc:Const_Lie_Alg}, the Lie algebras of each of these are constant as vector subspaces of $\M(n,\R)\oplus\lambda\M(n;\R)$	 and $\R\oplus\lambda\R$ respectively,
 and so $\mathfrak{ust}(n,1;\Lambda_\delta)$ is constant in $\M(n;\Lambda_\delta)$ as a vector subspace, even as $\delta$ varies in $\R$.
\end{observation}

\begin{corollary}
The groups $\iota_\delta(\mathsf{USt}(n,1;\Lambda_\delta))$ and 	$\iota_\mu(\mathsf{USt}(n,1;\Lambda_\mu))$ are conjugate in $\GL(2n;\R)$ if and only if $\mathsf{sgn}(\delta)=\mathsf{sgn}(\mu)$.
\end{corollary}

\subsection{Computing the Conjugacy Limit}

Recall the definition of continuity of \ref{sec:HC_To_HRR_Conjugacy} for Automorphism-Stabilizer geometries whose automorphism groups all embed in a fixed group $G$; phrased here to deal with the specific situation at hand.

\begin{definition}
\label{def:HC_HRR_conj_lim}
If $G_\delta<\GL(n;\Lambda_\delta)$ is a collection of groups, one for each $\delta\in\R$, we say that this collection is continuous if the map $
\delta\mapsto \iota_\delta(G_\delta)$ is continuous as a function $\R\to\Cl(\GL(2n;\R))$.
Further, if $(G_\delta, C_\delta)$ is a geometry of the Automorphism-Stabilizer variety with $G_\delta<\GL(n;\Lambda_\delta)$, we say $(G_\delta, C_\delta)$ is a continuous family of geometries if the map $\delta\mapsto (\iota_\delta(G_\delta),\iota_\delta(C_\delta))$ is continuous as a function $\R\to\Cl(\GL(2n;\R))\times\Cl(\GL(2n;\R))$.
\end{definition}

\noindent
The discussion of the previous section determines the continuity of the assignment $\delta\mapsto (\iota_\delta(\SU(n,1;\Lambda_\delta)),\iota_\delta(\mathsf{USt}(n,1;\Lambda_\delta))$ everywhere except for $\delta=0$.
To see this, note for $\delta\in\R_+$, the assignment $\delta\mapsto C_\delta=\diag(1,\sqrt{\delta}, \ldots, 1, \sqrt{\delta})$ provides a continuous map $\R_+\to \M(2n;\R)$.
Then by the previous discussion

$$\iota_\delta(\SU(n,1;\Lambda_\delta))=C_{|\delta|}\iota_{-1}(\SU(n,1;\C))C_{|\delta|}\inv$$
$$\iota_\delta(\mathsf{USt}(n,1;\Lambda_\delta))=C_{|\delta|}\iota_{-1}(\mathsf{USt}(n,1;\C))C_{|\delta|}\inv,$$

\noindent
 where we identify $\Lambda_{-1}=\C$ and $\iota_{-1}$ is the map sending each entry $a+i b$ to the $2\times 2$ sub-matrix $\smat{a&-b\\b&a}$.
 Similarly, when $\delta>0$ we have
 
 $$\iota_\delta(\SU(n,1;\Lambda_\delta))=C_{\delta}\iota_{1}(\SU(n,1;\R\oplus\R))C_{\delta}\inv$$
$$\iota_\delta(\mathsf{USt}(n,1;\Lambda_\delta))=C_{\delta}\iota_{1}(\mathsf{USt}(n,1;\R\oplus\R))C_{\delta}\inv,$$

\noindent 
where $\Lambda_1=\R\oplus\R$ and $\iota_1$ is the map sending each
entry  $a+\lambda b$ to $\smat{a&b\\b&a}$.
As conjugating a subgroup by a continuous path of matrices results in a continuous path of subgroups, we have:

\begin{corollary}
\label{cor:HC_HRR_Conj_Continuous}
The following maps are continuous into $\Cl(\GL(2n;\R))\times\Cl(\GL(2n;\R))$.
$$f_-\colon\delta\mapsto \left(\iota_\delta(\SU(n,1;\Lambda_\delta)),\iota_\delta(\mathsf{USt}(n,1;\Lambda_\delta))\right)\hspace{1cm} \delta\in\R_-$$
$$f_+\colon \delta\mapsto \left(\iota_\delta(\SU(n,1;\Lambda_\delta)),\iota_\delta(\mathsf{USt}(n,1;\Lambda_\delta))\right)\hspace{1cm} \delta\in\R_+$$
\end{corollary}

\noindent This leaves only checking continuity at the transition point, where the associated geometry switches from $\Hyp_\C^n$ to $\Hyp_{\R\oplus\R}^n$ through $\Hyp_{\R_\ep}^n$.

\begin{observation}
In light of the already completed work above, the continuity of the family of geometries $\Hyp_{\Lambda_\delta}^n$ amounts to checking that $\lim_{\delta\to 0^-}f_-(\delta)$ and $\lim_{\delta\to 0^+}f_+(\delta)$ have the same limit in $\Cl(\GL(2n;\R))\times\Cl(\GL(2n;\R))$.
\end{observation}

\noindent
To compute these two limits we once again leverage the work of the previous section, which shows individually each of these can be expressed as a conjugacy limit in $\GL(2n;\R)$.
In particular, each of these is a pair of  \emph{conjugacy limits of algebraic groups}.

\begin{lemma}
Let $G<\GL(n;\Lambda_\delta)$ be an algebraic group.
Then $\iota_\delta(G)$ is an algebraic subgroup of $\GL(2n;\R)$.	
\end{lemma}
\begin{proof}
If $G$ is an algebraic subgroup of $\GL(n;\Lambda_\delta)$ then $G$ is cut out by a collection of polynomials $G=V(p_1,\ldots, p_k)$ for $p_i\in \Lambda_\delta(x_{11},\ldots, x_{nn})$.
The substitution $x_{ij}=y_{ij}+\lambda z_{ij}$ converts each $p_m$ into a pair of real polynomials $p_m^\R$ and $p_m^\lambda$ determined by equating real and $\lambda$ parts.
Thus, $G=V(p_1^\R, p_1^\lambda, \ldots, p_k^\R, p_k^\lambda)$ is a real subvariety of $M(n;\R\oplus\lambda\R)=\R^{2n^2}$.

\noindent
As a representative example, consider $p=x_{11}^2+x_{21}^2\in\Lambda_\delta[x_{11},x_{12},x_{21},x_{22}]$	 which is one of the three defining polynomials for $\SO(2;\Lambda_\delta)$.
Substituting and multiplying out using $\lambda^2=\delta$ gives $y_{11}^2+y_{21}^2+\delta(z_{11}^2+z_{21}^2)+2\lambda(y_{11}z_{11}+y_{21}z_{21})=1$, and equating real and imaginary parts gives $p^\R=y_{11}^2+y_{21}^2+\delta(z_{11}^2+z_{21}^2)-1$, $p^\lambda=2(y_{11}z_{11}+y_{21}z_{21})$.

\noindent 
The map $\iota_\delta\colon\M(n;\R\oplus\lambda\R)\to\M(2n;\R)$ is algebraic, and the image of a subvariety in $\M(n;\R\oplus\lambda\R)$ is a subvariety of $\M(2n;\R)$.
It is easy to write down the explicit equations, as each number $y+\lambda z$ is represented by a matrix $\iota_\delta(y+\lambda z)=\smat{u&v\\y&z}$ where $u=z$ and $\delta y=v$.

\end{proof}

\noindent
The group $\U(n,1;\Lambda_\delta)$ is cut out by polynomials in $\Lambda_\delta[x_{11},\ldots, x_{nn}]$, as is easily seen by expanding out the relation $A^\dagger Q A-Q=0$ in coordinates $A=(x_{ij})$; and similarly $\mathsf{USt}(n,1;\Lambda_\delta)$ is algebraic in $\M(n+1;\Lambda_\delta)$.

\begin{corollary}
The groups $\iota_\delta(\SU(n,1;\Lambda_\delta))$ and $\iota_\delta(\mathsf{USt}(n,1;\Lambda_\delta))$ are algebraic subgroups of $\GL(2n;\R)$.	
\end{corollary}

\noindent
Thus, as conjugacy limits of algebraic subgroups of an algebraic group, Proposition 3.11 in \cite{CooperDW14} implies that the dimension of the conjugacy limits is the same as the dimension of the groups limiting to them, and thus that up to local isomorphism we may compute the conjugacy limits via the Lie algebra limit at $\delta=0$.
And furthermore, if the Lie algebra limits from both sides agree, the entire path of groups is continuous by Corollary \ref{cor:HC_HRR_Conj_Continuous}.

\begin{corollary}
The map $\delta\mapsto \iota_\delta(\SU(n,1;\Lambda_\delta)_0)$ is continuous if and only if the map $\delta\mapsto \iota_\delta(\su(n,1;\Lambda_\delta))$ is continuous.	
\end{corollary}

\noindent
The continuity of this map follows easily from our previous work.

\begin{lemma}
The maps $\iota_\delta$ induce a continuous map $\R\to \Gr(2 n^2,(2n)^2)$ defined by $\delta\mapsto \iota_\delta(M(n,\Lambda_\delta))$.
\end{lemma}
\begin{proof}
On the basis $\{E_{jk},\lambda E_{jk}\}$ for $\M(n;\Lambda_\delta)$ the map $\iota_\delta$ is expressed $\iota_\delta(E_{jk})=R_{jk}$ and $\iota_\delta(\lambda E_{jk})=I^\delta_{jk}$ by Calculation \ref{calc:Image_of_M_Lambda_Basis}.
Thus,
$$\iota_\delta(\M(n;\Lambda_\delta))=\span_{\R}\left(R_{jk}, I^\delta_{jk}\right)_{1\leq j,k\leq n}=\bigoplus_{1\leq j,k\leq n} \span_\R (R_{jk})\oplus \span_\R (I^\delta_{jk})$$
\noindent 
where the second equality comes from the observation that for all $\delta$, the basis vectors $R_{jk}$ and $I^\delta_{jk}$ are nonzero and orthogonal.
The vectors $R_{jk}$ are independent of $\delta$, and $I^\delta_{jk}$ is a continuous nonzero function of $\delta$ for all $j,k$.
Thus their span is a continuously varying subspace of $\M(2n;\R)$ of dimension $2n^2$.
\end{proof}

\noindent
Recalling that the Lie algebras $\mathfrak{u}(n,1)=\mathfrak{u}(n,1;\Lambda_\delta)$ are constant as vector subspaces of $\M(n;\R+\lambda\R)$, the above argument immediately implies the continuity of their images under $\iota_\delta$.

\begin{corollary}
The restriction of $\iota_\delta$ to the subset $\mathfrak{u}(n,1)\subset \M(n,\R\oplus\lambda\R)$ induces a continuous map $\R\to\Gr(dim,(2n)^2)$ defined by $\delta\mapsto \iota_\delta(\mathfrak{u}(n,1))$.	
\end{corollary}

\noindent
The same holds for the Lie algebras $\mathfrak{ust}(n,1;\Lambda_\delta)$, as they are likewise constant a a vector subspace of $\M(n;\R+\lambda\R)$.
The space of Lie subalgebras of $\M(n;\Lambda_\delta)=\gl(n;\Lambda_\delta)$ is a union of closed subsets of Grassmannians, and so a continuous path in some Grassmannian, all of which's points are Lie subalgebras, is automatically a continuous path in the space of Lie subalgebras.

\begin{corollary}
The map $\R\to\Cl(\gl(2n;\R))$ given by $\delta \mapsto \iota_\delta(\mathfrak{u}(n,1;\Lambda_\delta))$ is continuous.
Thus the groups $\iota_\delta(\U(n,1;\Lambda_\delta))$ limit to $\iota_0(\U(n,1;\Lambda_0))$ as $\delta\to 0$, and by definition \ref{def:HC_HRR_conj_lim}, the groups $\U(n,1;\Lambda_\delta)$ vary continuously as $\delta$ varies in $\R$.	
 \end{corollary}

\noindent
Together with the analogous corollary for the stabilizer subgroups, we have successfully constructed a transition of geometries.

\begin{theorem}
The geometries $\Hyp_{\Lambda_\delta}^n$ vary continuously with $\delta$, forming a transition from complex hyperbolic space $\Hyp_\C^n$ to point-hyperplane projective space $\Hyp_{\R\oplus\R}^n$.
\end{theorem}

\section{The Transition as a 1-Parameter Family}
\label{sec:HC_HRR_Family}
\index{Family!1 Parameter}
\index{Lie Groupoid!1-Parameter Family}

We turn next to the second notion of continuity given by Definition CITE, and prove the groups $\U(n,1;\Lambda_\delta)$ naturally fit together to form a $1-$parameter family as $\delta$ varies in $\R$.
In doing so, we need to consider not only 1-parameter families of Lie groups, but also 1-parameter families of algebras, defined presently.

\begin{definition}
A one parameter family of algebras $\fam{A}$ is a real vector bundle $\fam{A}\to\R$ together with a section $\fam{1}\to \fam{A}$ selecting point $\fam{1}(\delta)$ for each vector space $\fam{A}_\delta$, and a smooth map $\mu\colon\fam{A}\times_\R\fam{A}\to\fam{A}$ such that for each $\delta\in\R$ the restriction $\mu_\delta\colon \fam{A}_\delta\times\fam{A}_\delta\to\fam{A}_\delta$ is the multiplication of a real algebra structure on $\fam{A}_\delta$ with identity $\fam{1}(\delta)$.
\end{definition}

\begin{observation}
\label{def:Lambda_R}
The algebras $\Lambda_\delta$ form a 1-parameter family: the vector bundle $\R^3\to\R$ with coordinates $\R^3=\{(x,y,\delta)\}$ and bundle projection $(x,y,\delta)\to\delta$.
The section $\delta\mapsto (1,0,\delta)$ together with the multiplication map $\mu$ defined on the fiber product $\R^3\times_\R\R^3$ by $\mu((x,y,\delta),(z,w,\delta))=(xz+\delta yw, xw+yz,\delta)$ makes each $\R^2_\delta$ isomorphic to $\Lambda_\delta$ under the change of coordinates $(x,y)\mapsto x+\lambda y$.
This family will be denoted $\Lambda_\R$.
\end{observation}

\begin{figure}
\centering\includegraphics[width=0.5\textwidth]{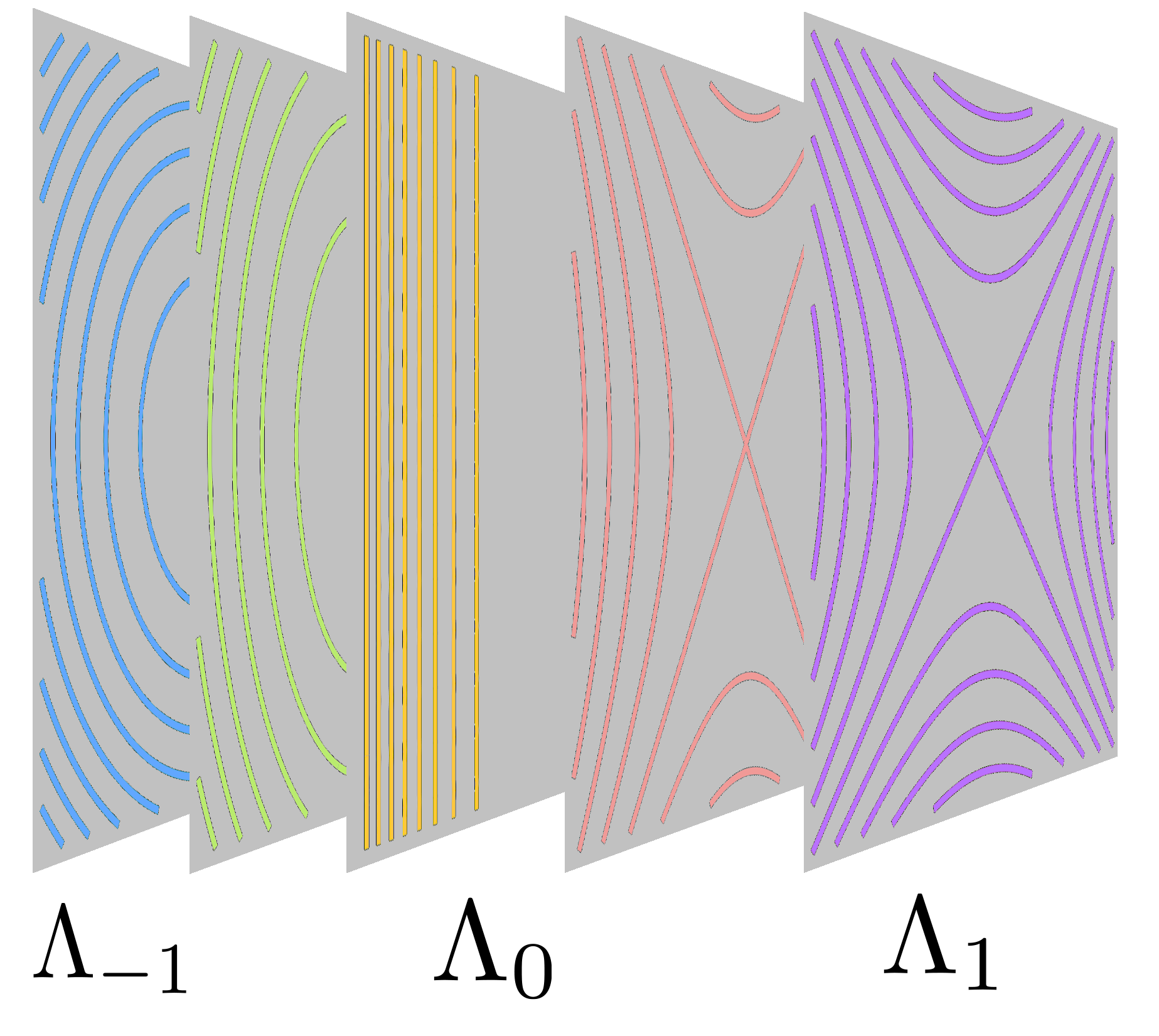}
\caption{The one parameter family of algebras $\Lambda_\R\to\R$, with each showing the level sets of its associated norm.}	
\end{figure}

\noindent
From one family of algebras springs many more: for instance, it is immediate to see that the matrix algebras over a 1-parameter family of algebras also form 1-parameter families.

\begin{corollary}
The matrix algebras $\M(n;\Lambda_\delta)$ form a 1-parameter family as $\delta$ varies in $\R$.
\end{corollary}
\begin{proof}
Let $\Lambda_\R$ be the total space of the $1$-parameter family of algebras above, with the underlying structure of a 2-dimensional vector bundle over $\R$.
Then $\Lambda_\R^{n^2}$ is naturally a $2n^2$ dimensional vector bundle over $\R$, and is equipped with a multiplication $m\colon \Lambda_\R^{n^2}\times_\R\Lambda_\R^{n^2}\to\Lambda_\R^{n^2}$ given by the usual for matrix multiplication:
$$m(((a_{ij}),\delta),((b_{ij}),\delta))=\left(\left(\sum_{k=1}^n \mu((a_{ik},\delta),(b_{kj},\delta))\right),\delta\right)$$
Which is smooth as the component operations of vector bundle addition, and multiplication given by $\mu$ are.
The identity section for this multiplication is $\delta\mapsto (I_n,\delta)$ for $I_n$ the real $n\times n$ identity matrix.
We will denote this family $\M(n;\Lambda_\R)\to\R$ from here on.
\end{proof}

\noindent
This family of algebras $\M(n;\Lambda_\R)$ provides a natural setting to consider continuity for the automorphism groups $\U(n,1;\Lambda_\delta)$ intrinsically.
Recalling definition CITE, 
\emph{a collection $G_\delta<\GL(n;\Lambda_\delta)$ varies continuously if $\bigcup_\delta G_\delta\times\{\delta\}$ is a 1-parameter family of groups}.

In this rest of this section, we develop some basic tools for analyzing subsets of $\M(n;\Lambda_\R)$ and determining when they form 1-parameter families of groups.
We will then apply this to the particular families relevant to the transition $\Hyp_{\Lambda_\delta}^n$; namely $\bigcup_\delta \SU(n,1;\Lambda_\delta)$ and $\bigcup_\delta\mathsf{USt}(n,1;\Lambda_\delta)$.

\begin{proposition}
Let $G_\delta<\GL(n;\Lambda_\delta)$ be a Lie subgroup for each $\delta\in\R$.  Then $\fam{G}=\bigcup_\delta G_\delta$ is a 1-parameter family of groups if and only if $\fam{G}$ is a smooth submanifold of $\M(n;\Lambda_\R)$ which is transverse to the fibers $\M(n;\Lambda_\delta)$ of $\M(n;\Lambda_\R)\to\R$.	
\end{proposition}
\begin{proof}
Let $\fam{G}$ be as described in the proposition.
The multiplication and inversion for each $G_\delta$ are direct restrictions of the multiplication and inversion on $\M(n;\Lambda_\delta)$; each of which given by polynomials in the multiplication of $\Lambda_\delta$ away from the noninvertible locus.
The multiplication of $\M(n;\Lambda_\R)$ is given by a smooth map $\mu\colon \M(n;\Lambda_\R)\times_\R\M(n;\Lambda_\R)\to\M(n;\Lambda_\R)$, induced by the smoothly varying multiplication on $\Lambda_\R$; and similarly inversion is smooth restricted to the subcollection of invertible elements.
As the multiplication and inversion of each $G_\delta$ come from the restriction of multiplication/inversion on $\M(n;\Lambda_\delta)$, the operations of composition and inversion on $\fam{G}=\cup_\delta G_\delta$ are smooth, as restrictions of the corresponding operations on $\M(n;\Lambda_\R)$.
As each element in $\cup_\delta G_\delta$ is invertible by assumption, the collection $\fam{G}$ form the set of morphisms of a groupoid, with objects given by the base space $\R$.
The product of two elements $x,y\in\fam{G}$ is only defined if they lie in the same fiber of the projection map $G_\delta=\pi_|\fam{G}\inv(\delta)$; thus the source and target map of the groupoid $\fam{G}$ are each given by the restriction of the projection $\pi\colon\M(n;\Lambda_\R)\to\R$ to $\fam{G}$.

As the space of objects and morphisms are both smooth manifolds, with smooth composition and inversion, $\fam{G}\to\R$ is a Lie groupoid if this restricted projection remains a submersion.
This follows easily from the assumption that for each $p\in \fam{G}$, the tangent space $T_p\fam{G}$ is transverse to $T_p\M(n;\Lambda_\delta)$, as then $T_p\M(n;\Lambda_\R)=T_p\fam{G}+T_p\M(n;\Lambda_\delta)$ and the projection $d\pi_p$ on all of $T_p\M(n;\Lambda_\R)$ is surjective, but $d\pi_p\M(n;\Lambda_\delta)=0$ as $\M(n;\Lambda_\delta)=\pi\inv(\delta)$.
Thus $(d\pi|_\fam{G})_p\colon T_p\fam{G}\to T_{\pi(p)}\R$ must be surjective and so $\pi$ is a submersion on $\fam{G}$.
\end{proof}

\noindent
This allows us to produce our first example of a 1-parameter family of groups, from the unit spheres with respect to the norm $x\mapsto x\bar{x}$ on the algebras $\Lambda_\delta$,
and furthermore this family is topologically nontrivial as the unit spheres change from circles (when $\delta<0$) to a pair of hyperbolas (when $\delta>0$).

\begin{example}
The elements of norm one, $\bigcup_\delta\U(\Lambda_\delta)=\U(\Lambda_\R)\subset\Lambda_\R$ form a 1-parameter family of groups.	
\label{ex:Units_1_Param_Fam}
\end{example}
\begin{proof}

In the coordinates $(x,y,\delta)$ on the family of algebras $\Lambda_\R$, conjugation $x+\lambda y\mapsto x-\lambda y$ is given by the map $(x,y,\delta)\mapsto (x,-y,\delta)$.
Thus the equation $z\bar{z}=1$ defining $\U(\Lambda_\delta)$ for each $\delta$ cuts out $\U(\Lambda_\R)$ as a subvariety of $\R^3$:
$$\U(\Lambda_\R)=\{(x,y,\delta)\in\R^3\mid x^2-\delta y^2=1\}$$
Thus $\U(\Lambda_\R)$ is a smooth submanifold of $\Lambda_\R$, which we take as the morphisms of a groupoid with objects given by $\R$.
To see $\U(\Lambda_\R)$ is transverse to the vertical foliation $\{\Lambda_\delta\}$ of $\Lambda_\R$, we note that for each point $p\in\U(\Lambda_\R)$ the tangent plane $T_p\U(\Lambda_\R)$ is not vertical, or equivalently the gradient $\nabla (x^2-\delta y^2-1)$ at $p$ is not parallel to the $\delta$ axis.
Calculating, $\nabla(x^2-\delta y^2-1)=(2x, -2y,\delta)$ is parallel to $(0,0,1)$ if and only if $x=y=0$, which occurs for no points of $\U(\Lambda_\R)$.
\end{proof}

\begin{figure}
\centering
\includegraphics[width=0.65\textwidth]{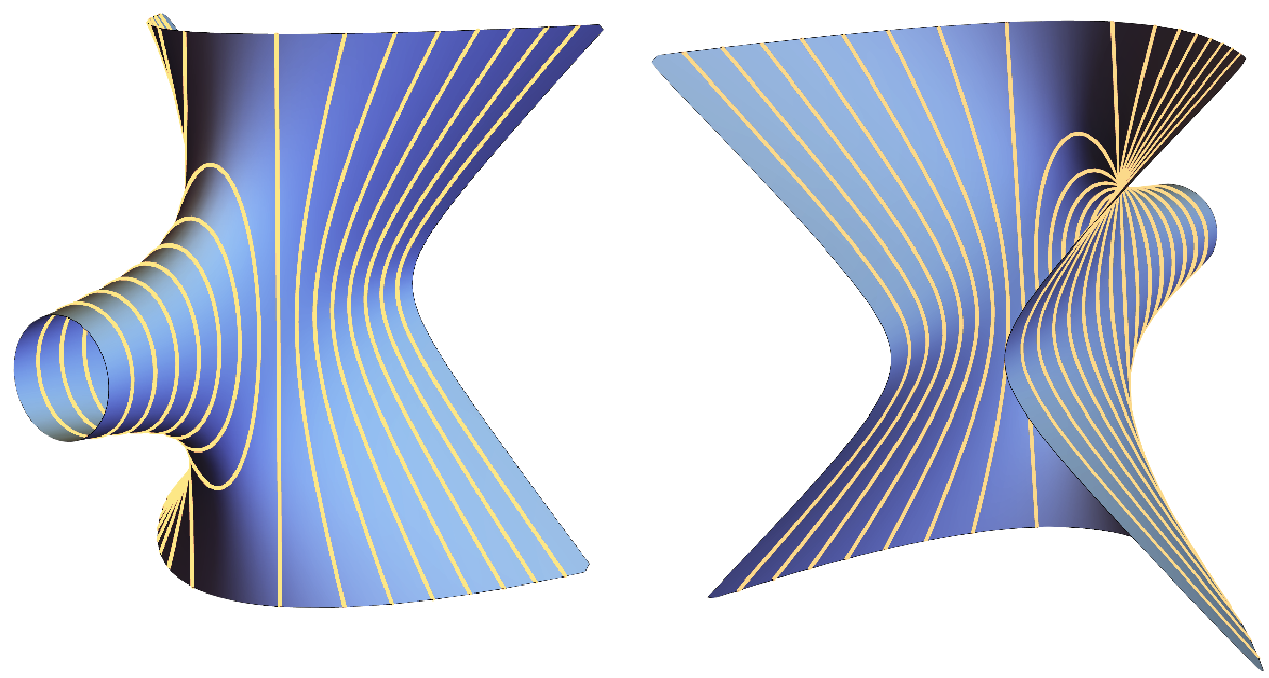}
\caption{The units $\U(\Lambda_\R)$ as a 1-parameter family.
The vertical slices exhibit the transitioning groups, from $\U(\C)\cong\S^1$ to $\U(\R\oplus\R)\cong\R\rtimes\Z_2$.}	
\end{figure}

\noindent
To proceed further, we draw an analogy to smooth topology to produce new 1-parameter families.
Just as smooth manifolds arise as point preimages of smooth submersions, 1-parameter families arise as point preimages of \emph{1-parameter families of submersions.}

\begin{definition}
Let $X$ be a smooth manifold.
A map $\Phi\colon\M(n;\Lambda_\R)\to X$ is a 1-parameter family of submersions if the restriction $\Phi_\delta\colon \M(n;\Lambda_\delta)\to X$ is a submersion for each $\delta\in\R$.
\end{definition}

\begin{theorem}
Let $\Phi\colon	\fam{M}(n;\Lambda_\R)\to X$  be slicewise submersion.
Then for any $x\in X$ the preimage $\Phi\inv(x)\subset\fam{M}(n;\Lambda_\R)$ is a smooth manifold on which the projection $\pi\colon\fam{M}(n;\Lambda_\R)\to\R$ restricts to a smooth submersion.
\label{thm:Slicewise_Submersion}
\end{theorem}
\begin{proof}
As $\Phi|_\delta\colon\M(n;\Lambda_\delta)\to X$ is a submersion for all $\delta$, the total map $\Phi$ itself is also a submersion, and hence $\Phi\inv(x)$ is a smooth manifold for each $x\in X$, by the preimage theorem from smooth topology.
Thus it only remains to show the restriction of $\pi$ to $\Phi\inv(x)$ is a smooth submersion.
Create from $\Phi$ the smooth map $\widetilde{\Phi}\colon \M\to X\times \R$ given by $\widetilde{\Phi}((A,\delta))=(\Phi(A),\delta)$.
Observe that $\widetilde{\Phi}$ is still a submersion, as follows.
The tangent space to any point $(x,\delta)\in X\times\R$ factors as a product $T_(x,\delta) X\times\R=T_x X\times T_\delta\R$.
Let $(A,\delta)\in \M(n;\Lambda_\R)$ with $\widetilde{\Phi}(A,\delta)=(x,\delta)$.
The condition that $\Phi$ is a slicewise submersion is exactly that the derivative of $\widetilde{\Phi}$, restricted to $\M(n;\Lambda_\delta)$ is onto the $T_xX$ factor, and the derivative of $\widetilde{\Phi}$ along the path $(A,t)$ is onto the $T_\delta\R$ factor by construction.

To show that the restriction of $\pi\colon\M(n;\Lambda_\R)\to\R$ to $\Phi\inv(x)$ is a submersion, we will use the following equivalent description of submersions: \emph{a smooth map $f\colon M\to X$ is a submersion if and only if through each point $m\in M$ there is a local section $\sigma\colon U\to M$ of $f$ with $f(m)\in U$ and $m\in\sigma(U)$.}
Choose a point $(A,\delta)\in\Phi\inv(x)$, and consider its image $\widetilde{\Phi}(A,\delta)=(x,\delta)\in X\times\R$.
As $\widetilde{\Phi}$ is a submersion, we may use the characterization above to produce a smooth local section $\sigma\colon U\to \M(n;\Lambda_\R)$ with $(A,\delta)$ in the image.
Possibly after shrinking, we may assume $U=V\times (\delta-\ep, \delta+\ep)$ for $V$ a neighborhood of $x\in X$.
Now consider the map $c\colon \R\to X\times\R$ given by $c(t)=(x,t)$ for all $t\in\R$, and the composition $\sigma\circ c$ defined on $(\delta-\ep,\delta+\ep)$.
This is a smooth map as it is a composition of smooth maps, and is a section of the projection map $\pi\colon\M(n;\Lambda_\R)\to\R$ by construction.
But finally, notice that for all $t\in(\delta-\ep,\delta+\ep)$, the point $\sigma\circ c(t)$ lies in $\Phi\inv(x)$, 
as $\sigma$ is a section of $\widetilde{\Phi}$ so $\widetilde{\Phi}\circ\sigma(c(t))=c(t)=(x,t)$ so $\Phi(\sigma(c(t))=x$.
Thus, the restricted projection admits smooth sections through every point $(A,\delta)\in\Phi\inv(x)$, and so it is a submersion by the alternative characterization above.
\end{proof}

\begin{corollary}
If $\fam{G}=\bigcup_\delta G_\delta$ is a collection of Lie subgroups of $\GL(n;\Lambda_\delta)$, then $\fam{G}$ is a 1-Parameter family of groups if $\fam{G}=\Phi\inv(x)$ for some smooth manifold $X$, some 1-parameter family of submersions $\Phi\colon\M(n;\Lambda_\R)\to X$ and some $x\in X$.
\end{corollary}

\begin{figure}
\centering\includegraphics[width=0.85\textwidth]{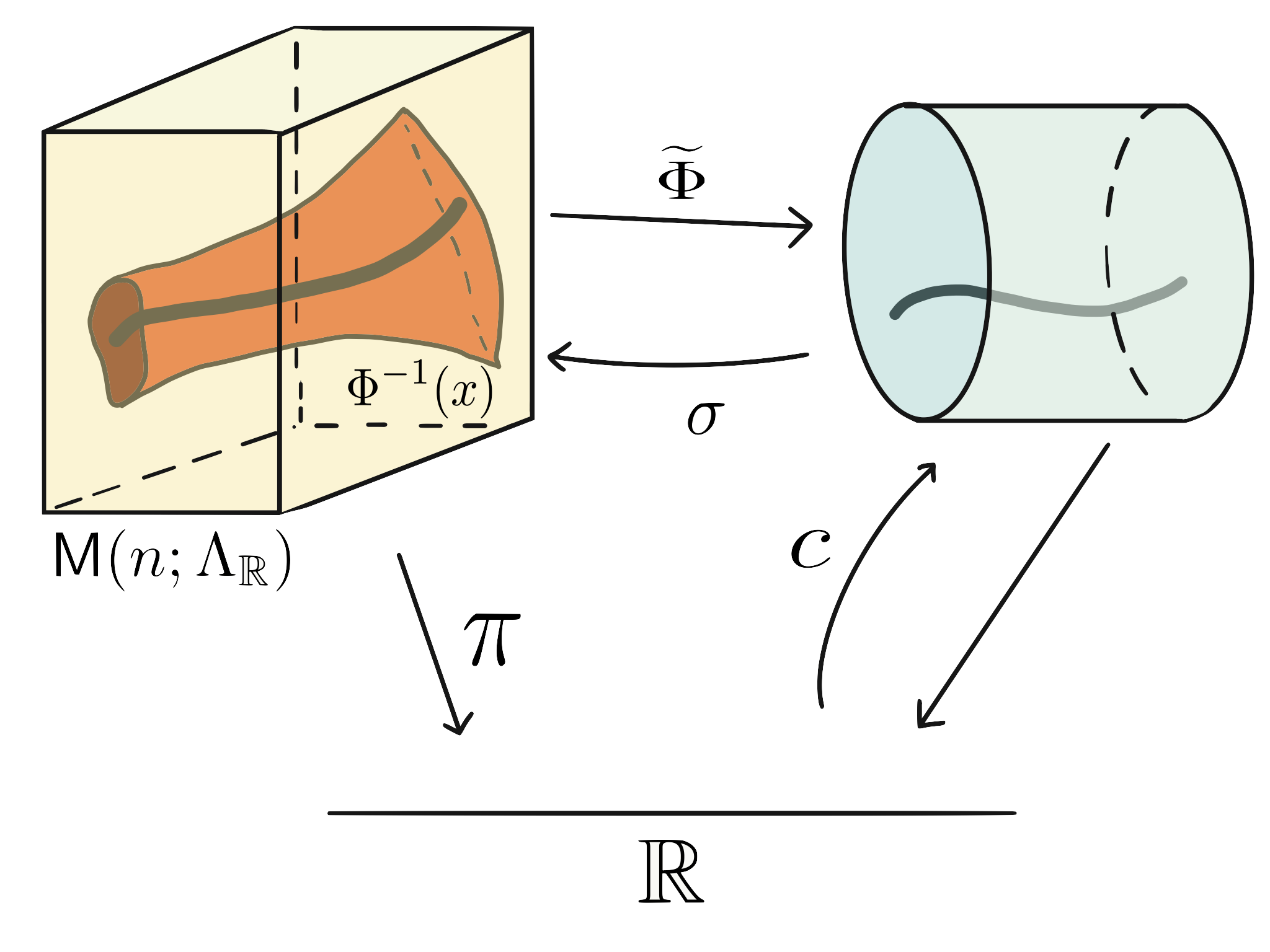}
\caption{Figure illustrating the proposition above and its proof.}
\end{figure}

\subsection{The 1-Parameter Family $\Hyp_{\Lambda}^n$}

From the Automorphism-Stabilizer perspective, we are interested in the families of unitary groups and their point stabilizers.

\begin{definition}
The collection $\fam{U}(n,1;\Lambda_\R)=\bigcup_{\delta\in\R} \U(n,1\Lambda_\delta)$ is the union of automorphism groups for the geometries $\Hyp_{\Lambda_\delta}^n$ in $\M(n;\Lambda_\R)$.
Restricting to $\det=1$ gives the collection of special unitary groups, $\fam{SU}(n,1;\Lambda_\R)$.
\end{definition}

\begin{definition}
The collection $\fam{USt}(n,1;\Lambda_\R)=\bigcup_{\delta\in\R}\mathsf{USt}(n,1;\Lambda_\delta)$ is the union of stabilizers for the geometries $\Hyp_{\Lambda_\delta}^n$ in $\M(n;\Lambda_\R)$.
\end{definition}

\begin{definition}
The collection of hyperbolic geometries is given by the pair $\Hyp_{\Lambda_\R}^n=(\fam{U}(n,1;\Lambda_\R),\fam{USt}(n,1;\Lambda_\R))$.
\end{definition}

\noindent
Using the observations and techniques developed above regarding slicewise submersions, it is quick work to show that each of these collections of groups forms a 1-parameter family, and thus $\Hyp_{\Lambda_\R}^n$ is a 1-parameter family of geometries.
The inspiration for this technique derives from the usual definition of $\U(n,1)=\{A\mid A^\dagger JA=J\}$ over $\R$, as the preimage of $J$ under the map $A\mapsto A^\dagger JA$.
Recall an element of $X\in\M(n+1;\Lambda_\delta)$ is \emph{Hermitian} if $X^\dagger=X$.

\begin{definition}
For each $\delta\in\R$ let $\Herm(n;\Lambda_\delta)\subset\M(n,\Lambda_\delta)$ be the collection of Hermitian matrices,
$\Herm(n;\Lambda_\delta)=\{X\in\M(n;\Lambda_\delta)\mid X^\dagger=X\}$.	
\end{definition}

\noindent
Note that the definition of Hermitian does not involve the multiplication of $\Lambda_\delta$, and so identifying each $\M(n;\Lambda_\delta)$ with $\M(n;\R)\oplus\lambda \M(n;\R)$ as vector spaces, the collections $\Herm(n;\Lambda_\delta)$ are constant in $\delta$.

\begin{remark}
Because $\Herm(n;\Lambda_\delta)$ is constant in $\delta$, we write $\Herm(n)$ when $\delta$ is irrelevant to present discussion, or allowed to vary.
\end{remark}

\noindent
Directly mimicking the standard construction over $\C$, we aim to exhibit each unitary group, and indeed the collection as a whole, as the point preimage of a submersion.

\begin{observation}
The collection $\fam{U}(n,1;\Lambda_\R)$ is the preimage of $J=\diag(I_n,-1)$ under the map $\Phi\colon \M(n+1;\Lambda_\R)\to \Herm(n+1)$ defined by $(A,\delta)\mapsto (A^\dagger JA, \delta)$
\end{observation}

\noindent
This map $\Phi$ is smooth as it is defined using the addition and multiplication on $\M(n+1;\Lambda_\R)$.
But moreover it is a 1-parameter family of submersions, as restricting to each slice $\M(n+1;\Lambda_\delta)$ gives the polynomial $\Phi_\delta\colon A\mapsto A^\dagger JA$ cutting out $\U(n,1;\Lambda_\delta)=V(\Phi_\delta(A)-J)$. 

\begin{proposition}
The restriction of $\Phi$ to $\Phi_\delta\colon\M(n+1;\Lambda_\delta)\to\Herm(n,1)$ is a submersion on $\U(n,1;\Lambda_\delta)$ for all $\delta\in\R$.	
\label{prop:Restricted_Projection_Lambda}
\end{proposition}
\begin{proof}
Let $B\in\U(n,1;\Lambda_\delta)$, then for any $X\in\M(n,\Lambda_\delta)$ we may construct the path $B_t=B+tX$ which remains in $\GL(n,\Lambda_\delta)$ for small $t$.  
Computing the derivative we see $\dd{t}|_{t=0}\Phi_\delta(B_t)=X^\dagger JB+B^\dagger JX$, and so $\Phi_J$ is a submersion if $X\mapsto X^\dagger JB+B^\dagger JX$ surjects onto $T_{\Phi_\delta(B)}\Herm(n)=\Herm(n)$.  
This map is $\R$-linear and so we proceed by dimension count, noting $\dim\mathrm{image}\;\Phi_\delta=\dim\M(n,\Lambda_\delta)-\dim\ker\Phi_\delta$.  
The kernel of $\Phi_\delta$ is given by $\ker\Phi_\delta=\set{X\mid X^\dagger JB=-B^\dagger JX}$, which can be expressed $\ker\Phi_\delta=(B^\dagger J)\inv\SkHerm(n)$ for $\SkHerm(n)$ the skew-Hermitian matrices over $\Lambda_\delta$, $\SkHerm(n)=\{A\in\M(n;\Lambda_\delta\mid A^\dagger=-A\}$. 
Thus $\dim\ker\Phi_\delta$ is the dimension of the space of skew-Hermitian matrices, so $\dim\mathrm{image}\;\Phi_\delta=\dim\Herm(n)$ and $(D\Phi_\delta)_B$ is surjective, making $\Phi_\delta$ is a submersion. 
\end{proof}

\noindent
Thus, by Theorem \ref{thm:Slicewise_Submersion} concerning 1-parameter families of submersions, the preimage of any point of $\Herm(n,1)$ is automatically a smooth submanifold of $\M(n+1;\Lambda_\R)$ on which $\pi\colon \M(n+1;\Lambda_\R)\to\R$ restricts to a submersion.

\begin{corollary}
The collection $\fam{U}(n,1;\Lambda_\R)$ is a 1-parameter family of groups.	
\end{corollary}
\begin{proof}
Take $\fam{U}(n,1;\Lambda_\R)$ to be the set of morphisms, and $\R$ to be the set of objects.
The morphism set additionally has the structure of a smooth manifold, by Proposition \ref{prop:Restricted_Projection_Lambda}.
The group operations of multiplication and inversion are smooth on all of $\fam{GL}(n+1;\Lambda_\R)$, and hence restrict to smooth operations on $\fam{U}(n,1;\Lambda_\R)$, giving $\fam{U}(n,1;\Lambda_\R)$ the structure of a groupoid.
The multiplication of two elements $A,B$ is defined if and only if $A$ and $B$ lie in the same slice $\U(n,1;\Lambda_\delta)$; thus the source and target maps of this groupoid are equal, and given by the restriction of $\pi\colon \M(n+1;\Lambda_\R)\to\R$.
But this restriction is a submersion on $\fam{U}(n,1;\Lambda_\R)$  by Proposition \ref{prop:Restricted_Projection_Lambda} above,
making $\fam{U}(n,1;\Lambda_\R)$ into a Lie groupoid, and a 1-parameter family of groups.
\end{proof}

\noindent
Given now that $\fam{U}(n,1;\Lambda_\R)$ is a 1-parameter family, a similar style argument can be applied to show that $\fam{SU}(n,1;\Lambda_\R)$ is a 1-parameter family as well.
While we have focused thus far in this chapter on the full unitary group (as, without the further $\det=1$ restriction, the arguments of section \ref{sec:HC_To_HRR_Conjugacy} were slightly simpler), in practice it is often better to work with $\SU(n,1;\Lambda_\delta)$ as the action on $\Hyp_{\Lambda_\delta}^n$ is locally effective.

\begin{observation}
As each $\Lambda_\delta$ is commutative, the usual formula for the determinant induces a map $\det_\delta\colon\M(n;\Lambda_\delta)\to\Lambda_\delta$.
The union of these maps provides a map $\det\colon\M(n;\Lambda_\R)\to\M(n;\Lambda_\R)$, which is smooth as it is polynomial in the additional and multiplication of the 1-parameter family $\M(n;\Lambda_\delta)$.	
\end{observation}

\begin{lemma}
For each $\delta\in\R$, the map $\det_\delta$ is a submersion $\U(n,1;\Lambda_\delta)\to\U(\Lambda_\delta)$.	
\label{lem:det_submersion}
\end{lemma}
\begin{proof}
The defining condition of $\U(n,1;\Lambda_\delta)$ implies $\det|_{\U(n,1;\Lambda_\delta)}$ takes values in $\U(\Lambda_\delta)$ as $\det_\delta(A^\dagger J A)=-\det_\delta(A^\dagger)\det_delta(A)=-1$, so $\det_\delta(A^\dagger)=\det_\delta(A)=1$.
Noting that $\det_\delta(A^\dagger)=\overline{\det_\delta A}$ finishes the claim.
Thus, $\det_\delta$ defines the short exact sequence
$1\to\SU(n,1;\Lambda_\delta)\to\U(n,1;\Lambda_\delta)\to\U(\Lambda_\delta)\to 1$.  
This is right-split by the section $\alpha\mapsto \diag(\alpha, 1,\ldots, 1)$ so $\U(n,1;\Lambda_\delta)$ is topologically a product $\U(\Lambda_\delta)\times\SU(n,1;\Lambda_\delta)$.
Under these coordinates the determinant is a projection, thus a smooth submersion.
\end{proof}

\noindent 
In particular this shows $\SU(n,1;\Lambda_\delta)$ is a smooth submanifold of $\U(n,1;\Lambda_\delta)$ (though this was already clear by the closed subgroup theorem).
The codomain of each $\det_\delta$ differs, and so it is not appropriate to ask $\det$ to be a 1-parameter family of submersions as before.
However, recalling Theorem \ref{thm:Slicewise_Submersion}, the first step was to promote a 1-parameter family of submersions $\Phi$ to a submersion between 1-parameter families $\widetilde{\Phi}$.
In this case, $\det$ is already such a map.
To show this, we note the following.

\begin{observation}
Let $\sigma:\R\to\fam{X}$ be a smooth section of a submersion $\pi\colon\fam{X}\to\R$.  
Then for each $x=\sigma(\delta)$ the tangent space $T_x\fam{X}$ decomposes as a direct sum $T_x\fam{X}=T_x\sigma(\R)\oplus T_x\pi\inv\set{\delta}$ into `vertical` and `horizontal' factors.
\end{observation}

\begin{proposition}
The determinant restricts to a submersion $\U(n,1;\Lambda_\R)\to\U(\Lambda_\R)$.
\end{proposition}
\begin{proof}
Let $X\in\fam{U}(n,1;\Lambda_\R)$ with $\pi(X)=\delta$, we will show that $\det$ is a submersion at $X$.
The projection $\pi\colon\fam{U}(n,1;\Lambda_\R)\to\R$ is a submersion, so choose a section $\sigma\colon V\to\fam{U}(n,1;\Lambda_\R)$ through $X$
(recall a map is a smooth submersion if and only if it admits smooth sections through each point of the domain).
Then $\det\sigma\colon V\to \fam{U}(\Lambda_\delta)$ is a section through $\alpha=\det(X)=\det_\delta(X)$, and so by the observation above $\sigma$ and $\det\circ\sigma$ provide the direct sum decompositions
$T_X\fam{U}(n,1;\Lambda_\R)=T_X\sigma(V)\oplus T_X\U(n,1;\Lambda_\delta)$ and $T_\alpha\fam{U}(\Lambda_\R)=T_\alpha\det\sigma(V)\oplus T_\alpha \U(\Lambda_\delta)$.
Restricting $\det$ to $\sigma(V)$ gives a homeomorphism $\sigma(V)\to\det\sigma(V)$ so $d\det_X|_{T_X\sigma(V)}$ is an isomorphism onto $T_\alpha\det\sigma(V)$.  
By Lemma \ref{lem:det_submersion},
the restriction $\det_\delta\colon\U(n,1;\Lambda_\delta)\to\U(\Lambda_\delta)$ is a submersion, 
thus $d\det_X|_{T_X\U(n,1;\Lambda_\delta)}$ maps onto $T_\alpha\U(\Lambda_\delta)$ so all together $d\det_X\colon T_X\fam{U}(n,1;\Lambda_\R)\to T_\alpha\fam{U}(\Lambda_\R)$ is surjective and $\det$ is a submersion.
\end{proof}

\noindent
Thus, we may use the remainder of Theorem \ref{thm:Slicewise_Submersion} to conclude that $\SU(n,1;\Lambda_\delta)$ is also a 1-parameter family.

\begin{corollary}
The collection $\fam{SU}(n,1;\Lambda_\R)$ is a 1-parameter family of groups.	
\end{corollary}
\begin{proof}
Similarly to before, the collection $\fam{SU}(n,1;\Lambda_\R)$ is the morphisms of a groupoid with objects $\R$ and source, target the restricted projection $\fam{SU}(n,1;\Lambda_\R)$.
The group operations are automatically smooth as restrictions of the operations on $\fam{GL}(n+1;\Lambda_\R)$, and the projection $\pi$ is a submersion by the arguments of Theorem \ref{thm:Slicewise_Submersion}, making $\fam{SU}(n,1;\Lambda_\R)$ into a Lie groupoid and thus a 1-parameter family of groups.	
\end{proof}

\noindent
This leaves only the collection of stabilizers $\fam{USt}(n,1;\Lambda_\delta)$, which is quick work given all that is done above.

\begin{observation}
Switching $J=\diag(I_{n-1},-1)$ to $J=I_{n}$ in the arguments above gives immediately that $\fam{U}(n;\Lambda_\R)$ and $\fam{SU}(n;\Lambda_\R)$ are 1-parameter families of groups.
Specializing to $n=1$ (or recalling Example \ref{ex:Units_1_Param_Fam}) gives $\fam{U}(\Lambda_\R)$ is a 1-parameter family as well.
\end{observation}

\begin{observation}
Let $\fam{G}\subset	\M(p;\Lambda_\R)$ and $\fam{H}\subset \M(q;\Lambda_\R)$ be 1-parameter families of groups.
Then their block-diagonal product $\fam{G}\times\fam{H}=\smat{\fam{G}&\\&\fam{H}}\subset\M(p+q;\Lambda_\R)$ is a 1-parameter family.
\end{observation}
\begin{proof}
Let $\pi_\fam{G}$ and $\pi_\fam{H}$ be the corresponding restricted projection maps.
Then $\fam{G}\times\fam{H}$ is the smooth manifold of morphisms for a Lie groupoid with source, target maps given by the submersion $\pi_\fam{G}\times\pi_\fam{H}\colon\fam{G}\times\fam{H}\to\R$, and thus has the structure of a 1-parameter family of groups.
\end{proof}

\begin{corollary}
The collection of point stabilizers $\fam{USt}(n,1;\Lambda_\R)$ forms a 1-parameter family of groups.	
\end{corollary}

\noindent
Putting this all together proves the main theorem from the context of 1-parameter families.

\begin{theorem}
The geometries $\Hyp_{\Lambda_\R}^n=(\fam{U}(n,1;\Lambda_\R),\fam{USt}(n,1;\Lambda_\R))$ form a 1-parameter family of geometries.	
\end{theorem}

\noindent
The definition of a 1-parameter family of groups suggests a natural notion of a 1-parameter family of spaces (namely, a smooth manifold $\fam{X}$ equipped with a submersion $\pi\colon\fam{X}\to\R$) and so it is natural to consider whether there is a group-space version of this 1-parameter family of geometries $\Hyp_{\Lambda_\R}^n$.
Fixing a $\delta$, we may construct a domain for the geometry $\Hyp_{\Lambda_\delta}^n$ in two ways: abstractly as the coset space $\U(n,1;\Lambda_\delta)/\USt(n,1;\Lambda_\delta)$, or as the quotient of the sphere of radius -1 by the elements of unit norm,$\Hyp_{\Lambda_n}=(\U(n,1;\Lambda),\mathcal{S}_\Lambda(n,1)/\U(\Lambda))$.
Letting $\fam{S}_{\Lambda_\R}(n,1)=\cup_{\delta\in\R}\fam{S}_{\Lambda_\delta}(n,1)\subset\M(n;\Lambda_\R)$,
each of these give natural candidates for a one-parameter family of domains,

$$\Hyp_{\Lambda_\R}^n=\fam{U}(n,1;\Lambda_\R)/\fam{USt}(n,1;\Lambda_\R)
\hspace{1cm}
\Hyp_{\Lambda_\R}^n=\cup_{\delta\in\R}\mathcal{S}_{\Lambda_\R}(n,1)/\fam{U}(\Lambda_\R)$$

\noindent
The inherent difficulty here is that in each case, the family of domains is presented as a family of spaces, \emph{quotiented by the action of a transitioning 1-parameter family of groups}.
It is a subtle issue to determine when the action of a 1-parameter family of groups on a 1-parameter family of spaces admits a quotient in the category of 1-parameter families.
The necessary work to formalize this, and take quotients of 1-parameters of spaces by sufficiently nice actions of 1-parameter families of groups, is one of the motivations for developing the theory of \emph{families of geometries}, undertaken in Part III.

\begin{figure}
\centering
\includegraphics[width=0.5\textwidth]{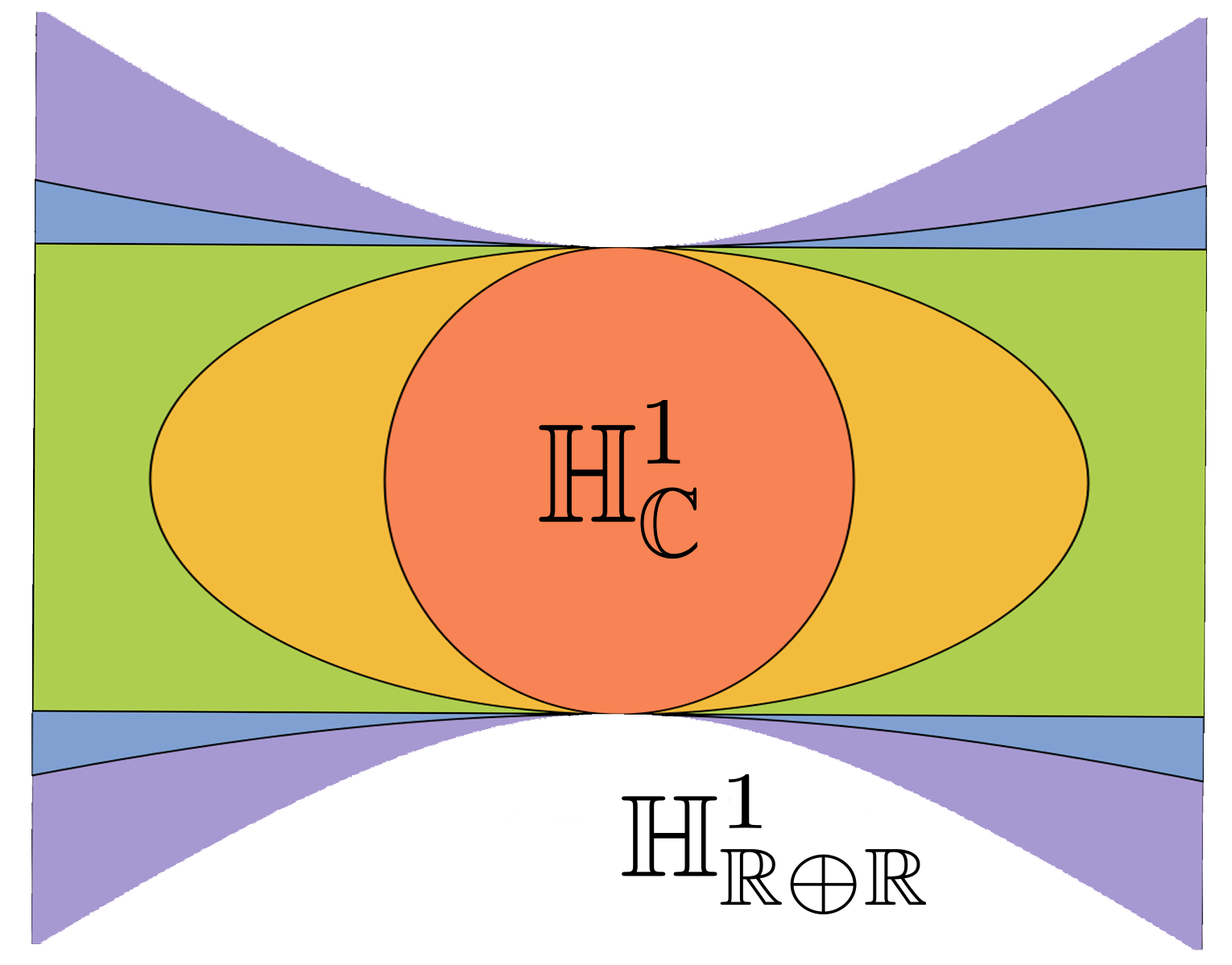}
\caption{The domains for $\Hyp_{\Lambda_\delta}^n$ as $\delta$ varies from $-1$ to $1$.}	
\end{figure}

\part{Families of Geometries}

\thispagestyle{plain}

\textcolor{white}{.}\vspace{2cm}

\vspace{0.5cm}
{\sffamily\bfseries\LARGE Families of Spaces, Groups, Algebras}
This chapter generalizes the abstract notion of continuity introduced in Chapter 9 studying the transition of $\Hyp_{\Lambda}^n$.
Taking inspiration from algebraic geometry and the deformation theory of complex manifolds, we introduce a notion of \emph{families of smooth manifolds} which is appropriate for understanding transitional behavior in geometric topology.
We utilize the resulting \emph{category of families} to define continuous families of more structured objects: such as families of Lie groups, rings, modules, and algebras.

\vspace{0.5cm}
{\sffamily\bfseries\LARGE Constructing Families of Geometries}
Having defined the algebraic and geometric families needed to describe continuously varying geometries in the abstract, this chapter provides a toolset aimed at constructing new families from old.
We study actions of families of groups on families of spaces, their orbits and their stabilizing subgroups on the road to defining families of geometries.
We also consider pullbacks and quotient families; providing conditions under which such operations can be preformed within the category of families.

\vspace{0.5cm}
{\sffamily\bfseries\LARGE Geometries over Algebras}
An immediate use for this new formalism is to extend the results of Chapters 9 and 10 describing the transitioning family of geometries $\Hyp_{\Lambda_\delta}^n$.
Here we consider various classes of geometries (projective geometries, geometries associated to unitary groups, and geometries associated to orthogonal groups) defined over arbitrary real algebras.
We briefly consider some generalities relating geometric properties to algebraic ones, generalizing the connection between $\Hyp_{\R\oplus\R}$ and Point-Hyperplane projective space.

\vspace{0.5cm}
{\sffamily\bfseries\LARGE Applications}
Finally, we provide a sample of applications of this general theory.
As noted above, we focus on generalizing the connection between smoothly varying algebraic and geometric structures, and show that any family of algebras induces families of projective / unitary geometries.
We also give an example application of this theory to the familiar study of subgeometries of $\RP^n$: providing a transition between various subgeometries of projective space which can occur \emph{abstractly}, but not as \emph{embedded subgeometries}.

\clearpage

\chapter{Families of Spaces, Groups, Algebras}
\label{chp:Families_Sp_Grp_Alg}
\index{Family}

This chapter introduces the theory of \emph{families}; smoothly varying collections of spaces, groups, or other gadgets parameterized by a smooth manifold.
A family of spaces, like a fiber bundle, should encode the continuity of its members intrinsically rather than by reference to some other ambient space.
Once we have settled on a good definition for a \emph{family of manifolds parameterized by a manifold}, the rest of the chapter follows easily.
Families of groups, algebras, modules and other objects of interest are all defined by endowing a family of spaces with extra structure.

\section{Families of Spaces}
\label{sec:Fam_Spaces}
\index{Family!Spaces}

A family of manifolds parameterized by the manifold $\Delta$ should be some object $\fam{X}$, decomposed into smooth manifolds $\fam{X}=\bigcup_{\delta\in \Delta}X_\delta$ in a coherent way with respect to the topology of $\Delta$.
To motivate the correct definition, we first look to nearby fields for inspiration.
Most prominently among these is algebraic geometry, which has produced a multitude of definitions and techniques for analyzing continuously varying collections of algebraic objects.

\begin{definition}[Algebro-Geometric Family]
A family is a flat morphism $f\colon X\to Y$ between schemes of finite time.
The members of the family are the fibers of $f$.
\end{definition}

\noindent
The definition of a \emph{flat morphism} in algebraic geometry 
is essentially scheme-theoretic ($f:X\to Y$ is flat if it induces  flat morphisms of rings $f_\star\colon \fam{O}_{Y,f(p)}\to\fam{O}_{X,p}$ on the level of stalks\footnote{
Even Mumford says: "The concept of flatness is a riddle that comes out of algebra, but which technically is the answer to many prayers".})
but itself generalizes more concrete situations, such as
\emph{complex analytic families} of complex manifolds, for reference see \cite{Kodaira}.

\begin{definition}[Complex Analytic Family]
A complex analytic family of compact complex manifolds is given by a domain $B\subset\C^m$ and a set of compact complex manifolds $\{M_t\}_{t\in B}$ such that $\bigcup_{t\in B}M_t=\mathcal{M}$ is a complex manifold equipped with a holomorphic map $\varpi:\mathcal{M}\to B$ such that (1)
$\varpi\inv(t)$ is a complex submanifold of $\mathcal{M}$ for each $t\in B$, (2) $\varpi\inv(t)=M_t$, and (3)
the rank of the Jacobian of $\varpi$ is equal to $m$ at each point of $\mathcal{M}$.
\end{definition}

\noindent
This definition can easily be translated to the real-analytic category or even smooth category, by declaring a family of smooth manifolds to be a smooth manifold $\Delta$ and a smooth manifold $\mathcal{X}$ equipped with a smooth proper submersion $\pi\colon \mathcal{M}\to\Delta$.
This describes the \emph{type} of object we want; as the continuity of the family $\{X_\delta\mid \delta\in\Delta\}$ is given precisely by the fact that all the members fit together to form a smooth manifold, with their location in the family determined by a proper submersion onto the parameter space.
However, there are two problems with this proposed definition. 
Firstly, the members of the family, $\fam{X}_\delta=\pi\inv(\delta)$ are necessarily compact, thus such families cannot hope to capture things like the $\Hyp^2$ to $\E^2$ transition.
Moreover, all manifolds occurring in such a family are homeomorphic, as an immediate corollary of Ehresmann's Fibration Theorem.

\begin{theorem}[Ehresmann's Fibration Theorem]
Let $M$, $N$ be smooth manifolds and $f\colon M\to N$ a proper surjective submersion.  Then $f$ is a locally trivial fibration of $M$ over $N$.	
\end{theorem}

Thus, even ignoring the compactness issue we could not hope to formalize examples such as the $\S^2$ to $\E^2$ transition.
Both of these issues are resolved by relaxing the demand that the map onto parameter space be proper.

\begin{definition}[Family of Smooth Manifolds]
\label{def:Smooth_Families}
A smooth family of manifolds parameterized by a smooth manifold $\Delta$ is a triple $(\fam{X},\Delta,\pi)$ of smooth manifolds $\fam{X},\Delta$ equipped with a smooth submersion $\pi:\mathcal{X}\to\Delta$.
The space $\fam{X}$ is the \emph{total space} and $\Delta$ is the \emph{base} of the family.
The fibers $\fam{X}_\delta:=\pi\inv\set{\delta}$ are the \emph{members} of the family, and are said to vary smoothly over $\Delta$.
\end{definition}

\begin{figure}
\centering\includegraphics[width=0.6\textwidth]{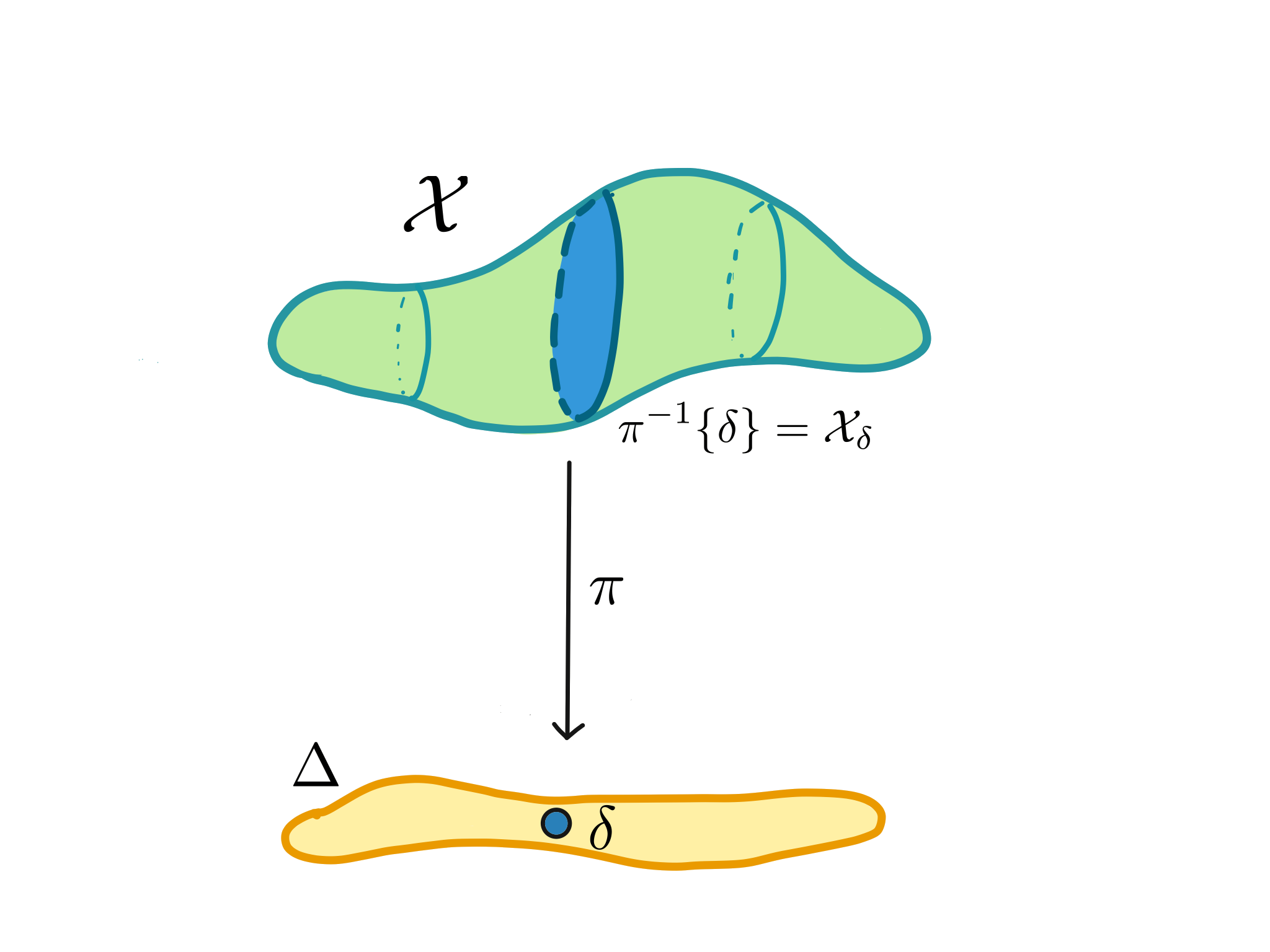}
\caption{A smooth family of manifolds, schematically.}	
\end{figure}

\noindent
A family contains a \emph{transition} if there are non-isomorphic members over a single connected component of the base. An object $X$ \emph{has transitions} if it is a member of a transitioning family.  Otherwise $X$ is \emph{rigid}.
Here we record some basic examples.

\begin{example}
Any manifold is a family of points over itself when equipped with the identity map. Any covering space is a smooth family with fibers dimension zero manifolds.
\end{example}

\begin{example}
\label{ex:C_to_C}
Branched covers are not families as the covering map is not a submersion.
For example, $\pi\colon\C\to\C$ given by $z\mapsto z^2$ does not determine a family of points over $\C$.
\end{example}

\begin{example}
Any product $X\times \Delta$ is a family over $\Delta$.	
Any fiber bundle $E\to B$ with fiber $F$ is a family of copies of $F$ over $B$.	
\end{example}

\noindent
Of course the interesting families are not fiber bundles or even fibrations, and have fibers that change homotopy type.

\begin{example}
\label{ex:First_Example}
Let $\fam{V}=\set{(x,y,t)\in\R^3 \mid x^2+t y^2=1}$ and $\pi\colon\fam{V}\to\R$ be the restriction of the projection map $(x,y,t)\to t$.  Then $\fam{V}\stackrel{\pi}{\to}\R$ is a smooth family, with ellipses as fibers for $t>0$ and hyperboloids for $t<0$.
\end{example}
\begin{proof}
The normal vector $\nabla(x^2+ty^2)=\langle 2x,2ty,y^2\rangle$ to $\fam{V}$ is never parallel to the $t$-axis, and so the coordinate vector field $\partial_t$ on $\R^3$ projects to a nowhere zero vector field on $\fam{V}$, defining a flow $\Phi_s:\fam{V}\to\fam{V}$ which gives sections $\sigma(s)=\Phi_s(x,y,t)$ of $\pi$ through each $(x,y,t)\in\fam{V}$.
Thus, $\pi$ is a submersion when restricted to $\fam{V}$ so $\fam{V}$ is a family.
\end{proof}

\begin{figure}
\centering\includegraphics[width=0.75\textwidth]{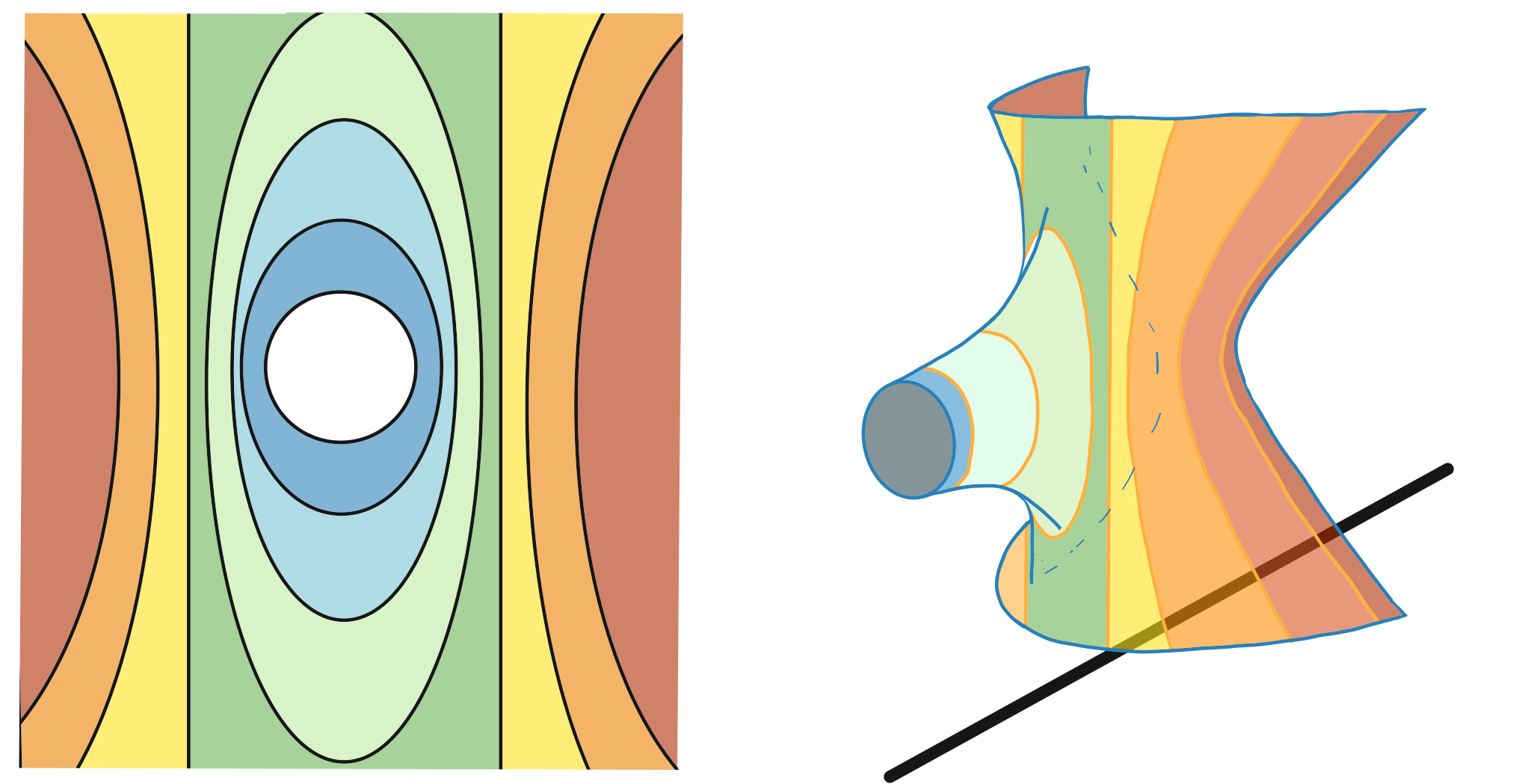}
\caption{The family of conics $\fam{V}=\{(x,y,t)\in\R^3 \mid x^2+t y^2=1\}$, with total space the punctured plane.  On the left, this is realized as $\R^2$ minus the open unit disk, with a projection onto $[0,\infty)$ given by color.  On the right, the same total space is constructed as a subvariety of $\R^3$ with projection onto one of the coordinate axes.}
\end{figure}

\noindent
Topological change in the fibers happens \emph{out at infinity}, and is allowed by the noncompact nature of the total space.

\begin{example}
Consider the smooth manifold $\mathcal{X}$ given by the union of the $x$ axis with the graph of $y=1/x$ in the plane.  
The projection map $\pi\colon (x,y)\mapsto x$ is restricts to a surjective smooth submersion $\mathcal{X}\to \R$, and the preimage of all points is a discrete set with two points except for the singleton above $x=0$.
\end{example}

\noindent
It is often useful to consider \emph{subfamilies} or \emph{restrictions} of a larger family.

\begin{definition}
A \emph{subfamily} $\fam{Y}\to\Delta$ of a family $\pi\colon\fam{X}\to\Delta$ is given by a closed subset $\fam{Y}\subset\fam{X}$ on which the restricted projection map remains a submersion.  
The \emph{restricted family} of $\fam{X}\to\Delta$ corresponding to a subset $D\subset \Delta$ has total space $\fam{X}|_D:=\pi\inv(D)$ equipped with the restricted projection map.
\end{definition}

\begin{figure}
\centering\includegraphics[width=0.6\textwidth]{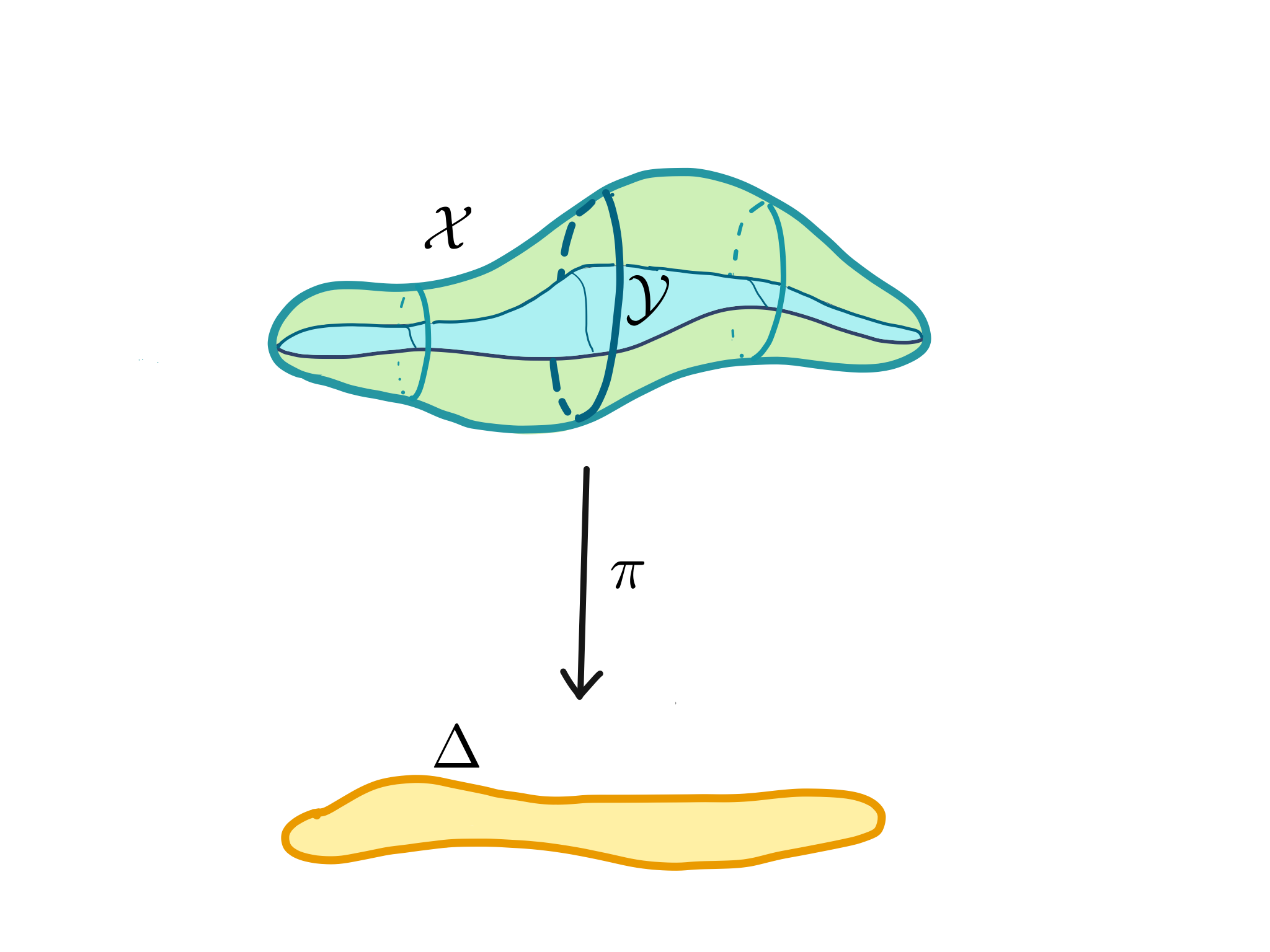}
\caption{A subfamily of a family, schematically.}	
\end{figure}

\noindent
Thus in Example \ref{ex:First_Example} above, $\fam{V}$ is a subfamily of the trivial family $\R^2\times \R\to\R$ with projection map $\pi(x,y,t)=t$.  Any open subset $\fam{U}\subset\fam{X}$ of a family inherits the structure of a family as $\pi|_{\fam{U}}\colon\fam{U}\to\Delta$ still admits local sections, but is not a \emph{subfamily} unless $\fam{U}$ is also closed.

The notion of family can be generalized beyond the smooth category, although in many categories of topological spaces there is not a unique obvious generalization of \emph{submersion}.
In fact, there are two inequivalent notions often called topological submersions, stemming from the fact that in $\mathsf{Diff}$ submersions can be described both as the class of maps admitting local sections, and those which are locally projections $\R^{n+k}\to \R^k$.

\begin{definition}
A map $f\colon X\to Y$ \emph{admits local sections}	if for each $x\in X$ there is an open neighborhood $U\ni f(x)$ and a map $\sigma\colon U\to X$ such that $f\circ\sigma=\id|_U$ and $x\in\sigma(U)$.
\end{definition}

\begin{definition}
A map $f\colon X\to Y$ is \emph{locally a projection} if for each $x\in X$ there is a neighborhood $U$ such and a map 
$\phi\colon U\to \pi(U)\times Z$ such that the following square commutes
\begin{center}
\begin{tikzcd}
U \arrow[r, "\phi"] \arrow[d,"\pi",swap]
& \pi(U)\times Z \arrow[d, "\mathsf{pr}"] \\
\pi(U) \arrow[r,  "\mathsf{id}"]
& \pi(U)
\end{tikzcd}	
\end{center}
\end{definition}

\begin{remark}
A smooth map $f\colon X\to Y$ is a submersion if and only if $f$ admits smooth local sections through each point of the codomain.
Similarly, $f$ is a submersion if and only if it is locally a projection.
\end{remark}

These two generalizations of submersion provide two means of extending Definition \ref{def:Smooth_Families} describing families to other categories.
Being locally a projection is strictly stronger than admitting local sections, and so we refer to these two potential generalizations as \emph{weak families} and \emph{strong families}.
In what follows, $\cat{C}$ denotes a category of topological spaces, for example the category of topological manifolds $\cat{C}=\cat{Man}$, or all locally compact Hausdorff spaces $\cat{C}=\cat{LCH}$.

\begin{definition}
A weak $\cat{C}$-family of spaces is a triple $(\fam{X},\Delta,\pi)$ such that $\pi\colon \fam{X}\to\Delta$ is a $\cat{C}$-morphism admitting $\cat{C}$-local sections.
\end{definition}

\begin{definition}
A strong $\cat{C}$-family of spaces is a triple $(\fam{X},\Delta,\pi)$ such that $\pi\colon\fam{X}\to\Delta$ is a $\cat{C}$-map which is locally a projection.
If additionally a single $Z$ suffices for all points of $\fam{X}$, the family $\fam{X}\to\Delta$ is called a \emph{family locally modeled on $Z$}.
\end{definition}

\begin{figure}
\centering
\includegraphics[width=0.95\textwidth]{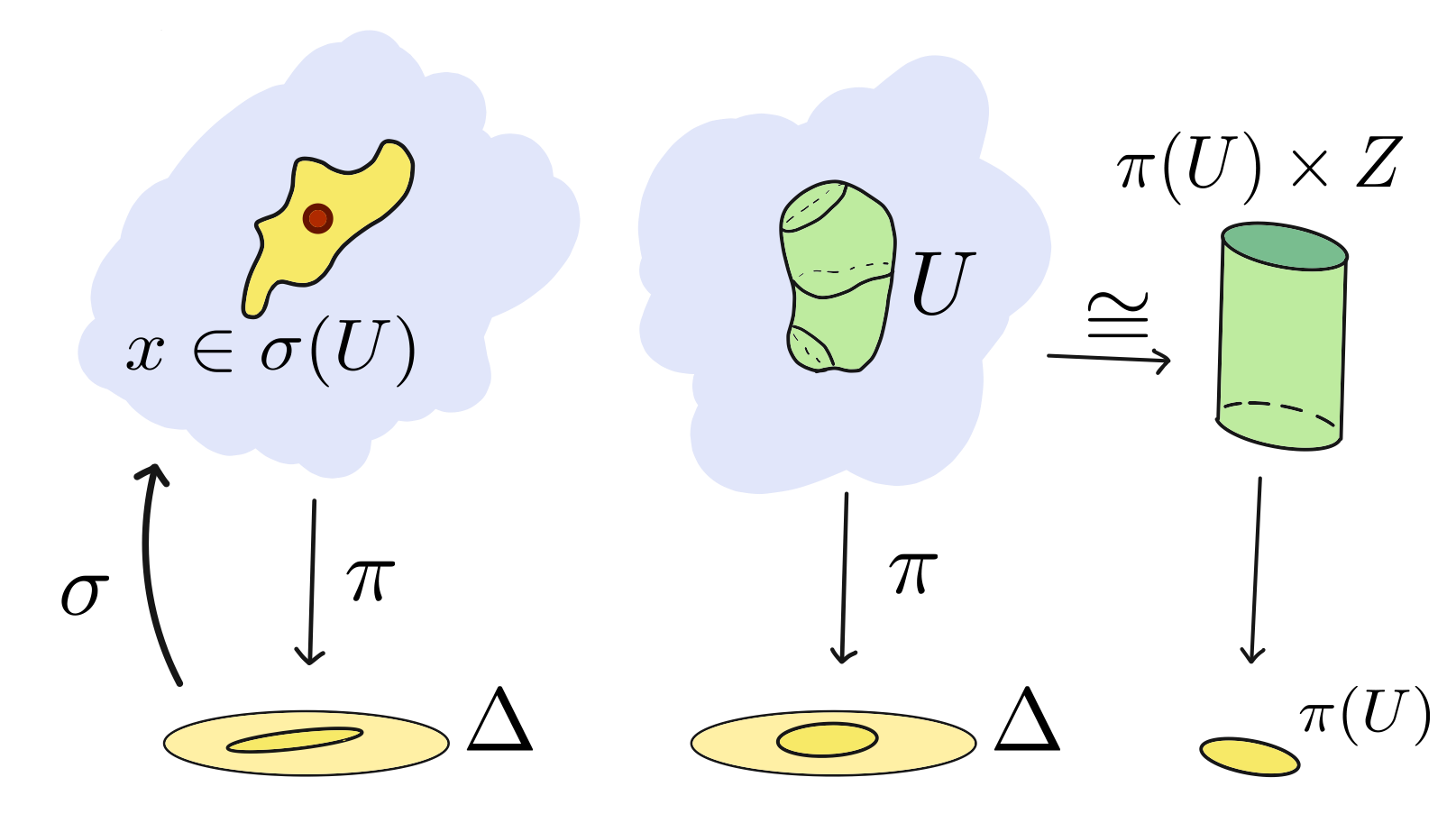}
\caption{Weak families (left), and strong families (right) schematically.}
\end{figure}

\noindent
It is an ongoing project to determine for which topological categories $\cat{C}$ and for which notion of family the various theorems characterizing the theory of smooth families generalize.
In this thesis, when a result is easily proven using the local section condition, we do so; and note that the result then holds for all weak families over all topological categories.
Conversely, when a result crucially uses techniques of smooth topology, we make note of that as well.
In most instances where smoothness is crucial, assuming only that the parameter map is a local projection appears to suffice.
However, the arguments become more technical and so this level of generality is not pursued here.

\subsection{Families and Chabauty Continuity}

We take a brief detour from further developing the theory of families to relate this perspective the the familiar notion of continuity in a Chabauty space.
Any continuous map $f\colon X\to Y$ induces a function, the \emph{fiber map} $f_\ast\colon Y\to\Cl_X$ by $f_\ast(y)=f\inv\set{y}$, and so the fiber map of any  family $\pi\colon\fam{X}\to\Delta$ is a function from the base into the Chabauty space of $\fam{X}$.

\begin{lemma}
\label{lem:Chab_Cts}
Let $\fam{X}\stackrel{\pi}{\to}\Delta$ be a continuous map of Hausdorff spaces.  Then the induced map $\pi_\ast:\Delta\to\Cl_\fam{X}$ is continuous if and only if $\pi$ is open.
\end{lemma}
\begin{proof}
First assume $\pi\colon\fam{X}\to\Delta$ is open.  Let $\fam{O}_{K,U}$ be a subbasic open set for the Chabauty topology on $\fam{X}$ and $\delta\in\pi_\ast\inv\{\fam{O}_{K,U}\}$ for $K\subset \fam{X}$ compact, $U\subset\fam{X}$ open.  
As $\pi(K)$ is a compact subset of $\Delta$ not containing $\delta$ there is some open $V\ni\delta$ disjoint from $\pi(K)$.  
As $\pi$ is open, $W=V\cap\pi(U)$ is an open neighborhood of $\delta$.  
Note $W\subset\pi_\ast\inv\set{\fam{O}_{K,U}}$ as if $\eta\in W$ then $\eta\not\in\pi(K)$ so $K\cap\pi_\ast(\eta)=\varnothing$ and $\eta\in\pi(U)$ so $\pi_\ast(\eta)\cap U\neq\varnothing$.  Thus $\pi_\ast\inv\set{\fam{O}_{K,U}}$ is open and $\pi_\ast$ is continuous.

Conversely, assume the continuity of $\pi_\ast$, and let $U\subset\fam{X}$ be open.  
The subbasic open set $\fam{O}_{\varnothing,U}$ contains all closed sets of $\fam{X}$ intersecting $U$, and so $\pi_\ast\inv\set{\fam{O}_{\varnothing,U}}=\pi(U)$.  Thus $\pi(U)$ is open.
\end{proof}

\begin{corollary}
Any family $\pi\colon\fam{X}\to\Delta$ has members $\fam{X}_\delta=\pi_\ast(\delta)$ varying continuously in the Chabauty space of $\fam{X}$.
\end{corollary}

The examples of continuously varying groups which were constructed in Chapters \ref{chp:Limits_of_Geos}, \ref{chp:Orthogonal_Groups} and \ref{chp:Heisenberg_Plane} can be recast as families: with the path of groups $t\mapsto H_t<G$ forming the subfamily $\fam{H}=\bigcup_{t\in\R}H_t\times\{t\}$ of $G\times\Delta\to\Delta$.
In switching to the formalism of families, we might wonder if we have inadvertently introduced anything \emph{new} in this context: are there subfamilies of $G\times\Delta$ whose members do not vary continuously in $\Cl(G)$? 
Below we see the answer is no.

\begin{proposition}
\label{prop:WeakCts_Families}
Fix a space $Y$.  Then the Chabauty continuous maps $\Delta\to\Cl_Y$ are in $1:1$ correspondence with the subsets $\fam{X}\subset Y\times\Delta$ onto which  $\mathsf{pr}_\Delta\colon Y\times\Delta\to\Delta$ restricts to an open map $\mathsf{pr}_\fam{X}$.
\end{proposition}
\begin{proof}
Let $\fam{X}$ be a closed subset $\fam{X}\subset Y\times\Delta$ such that $\mathsf{pr}_\Delta|_{\fam{X}}$ is open.  
By Lemma \ref{lem:Chab_Cts} the map $\iota\mathsf{pr}|_{\fam{X}\ast}\colon\Delta\to\Cl_\fam{X}\inject\Cl_{Y\times\Delta}$ is continuous.  
As each of the closed sets lives in a single fiber $Y\times\delta$, following with the projection onto $Y$ gives a continuous map $\Delta\to\Cl_Y$.

Given a Chabauty continuous map $\phi:\Delta\to\Cl_Y$ one may construct the subset $\fam{X}=\set{(x,\delta)\mid x\in\phi(\delta)}\subset Y\times\Delta$.  We show $\fam{X}$ is closed. If $\set{x_i}\subset \fam{X}$ converges to $x_\infty$ in $Y\times\Delta$ with $x_i\in\phi(\delta_i)$ then $\delta_i\to\delta_\infty$ by the continuity of $\mathsf{pr}$ and $\phi(\delta_i)\to\phi(\delta_\infty)$ by the continuity of $\phi$.  Thus $x_\infty\in\phi(\delta_\infty)$ so $x_\infty\in\fam{X}$.  

To see that the restriction $\pi$ of $\mathsf{pr}\colon X\times \Delta\to\Delta$ to $\fam{X}$ is open, note that any open $U\subset\fam{X}$ is of the form $\tilde{U}\cap\fam{X}$ for $\tilde{U}$ open in $\fam{X}\times\Delta$, and for our purposes we may without loss of generality assume $\tilde{U}=V\times W$ for $V\subset X$, $W\subset \Delta$ open.  
Then $\pi(U)=\set{\delta\mid \pi(u)=\delta, u\in U}=\set{\delta\mid \exists (v,\delta) \textrm{ such that } (v,\delta)\in U}$.  But $(v,\delta)\in U=(V\times W)\cap\fam{X}$ implies $\delta\in W$ and $v\in \phi(\delta)\cap V$ so we may re-express this set as $\pi(U)=\set{\delta\in W\mid V\cap\phi(\delta)\neq \varnothing}$.  
This is precisely the set $W\cap\phi\inv(\fam{O}_{V,\varnothing})$ however, for $\fam{O}_{V,\varnothing}=\set{Z\in\Cl_X\mid Z\cap V\neq\varnothing}$ a basic open set of $\Cl_X$.  
This is open as $\phi$ is continuous, so $\pi(U)$ is open as required.
\end{proof}

\begin{corollary}
The fibers of any subfamily $\fam{X}$ of $Y\times \Delta\to\Delta$ vary continuously in the Chabauty space of $Y$.
\end{corollary}

\section{The Category of Families}
\label{sec:Cat_of_Fams}
\index{Family!Category}

It is often useful not to study single families in isolation, but rather consider the entire \emph{category of families}.

\begin{definition}
The category $\Fam$ has as objects all smooth families $\pi\colon \fam{X}\to\Delta$ and morphisms $(\fam{X},\Delta,\pi)\to(\fam{X}',\Delta',\pi')$ 
are pairs $(\Phi,\phi)\in\Hom_{\mathsf{Diff}}(\fam{X},\fam{X}')\times\Hom_{\mathsf{Diff}}(\Delta,\Delta')$ making the relevant square commute.
\begin{center}
\begin{tikzcd}
\mathcal{X}\arrow[r,"\Phi"]\arrow[d,"\pi"]&\mathcal{X}'\arrow[d,"\pi'"]\\
\Delta\arrow[r,"\phi"] &\Delta'	
\end{tikzcd}
\end{center}
\end{definition}

\noindent
This category is at times the relevant object to consider (for instance, when constructing pullbacks of families) though more often we will be interested in the subcategories defined by fixing a base smooth manifold $\Delta$.
In analogy to bundles, we do not take the full subcategory of families with base $\Delta$, but rather only the morphism pairs of the form $(\Phi,\id_\Delta)$.

\begin{definition}
The category $\Fam_\Delta$ has as objects all families $\pi_\fam{X}\colon\fam{X}\to\Delta$, with morphisms $\phi\in\Hom(\fam{X}\labelarrow{\pi_\fam{X}}\Delta, \fam{Y}\labelarrow{\pi_\fam{Y}}\Delta)$ given by maps $\phi\in C^\infty(\fam{X},\fam{Y})$ such that $\pi_\fam{X}=\pi_\fam{Y}\phi$.
\begin{center}
\begin{tikzcd}
\mathcal{X}\arrow[rr,"\phi"]\arrow[dr,"\pi_\fam{X}"]&&\mathcal{Y}\arrow[dl,"\pi_\fam{Y}"]\\
&\Delta &	
\end{tikzcd}
\end{center}
\end{definition}

\noindent
When there is no ambiguity within $\Fam_\Delta$, families will often be referenced simply via their total space $\fam{X}$.
We begin our discussion of $\Fam_\Delta$ by considering some basic results about the category, which we will have use for when constructing new families, or defining families of algebraic gadgets.

\begin{observation}
The category $\cat{Fam}_\Delta$ has as initial object the empty family $\varnothing\to\Delta$ and final object the trivial family $\Delta\stackrel{\id}{\to}\Delta$. 	
\end{observation}

\begin{observation}
Monomorphisms in $\Fam_\Delta$ are injections $\phi:\fam{X}\to\fam{Y}$, and epimorphisms are surjections.	
\end{observation}

\noindent
Products in $\cat{Fam}_\Delta$ are given by the pullback of the projection maps, and coproduct by disjoint union of the total spaces.  
The category $\cat{Fam}_\Delta$ has all finite products and coproducts, verified below.
In both cases we show the result is a family by verifying that the projection admits local sections - thus this result remains true for all weak families in all topological categories.

\begin{lemma}
The product of $\fam{X}$ and $\fam{Y}$ in $\cat{Fam}_\Delta$ is has total space $\fam{X}\timesd\fam{Y}$ the pullback of the projections $\pi_\fam{X}$, $\pi_\fam{Y}$ and projection $\pi\colon\fam{X}\timesd\fam{Y}\to\Delta$ the diagonal of the pullback square.
\end{lemma}
\begin{proof}
It is immediate from the diagram describing the universal property of products that if $\fam{X},\fam{Y}\in\cat{Fam}_\Delta$ have a product, it is given by the pullback $\fam{X}\timesd\fam{Y}$.  
Thus we need only show the projection $\fam{X}\timesd\fam{Y}\to\Delta$ admits local sections.  
Let $(x,y)\in\fam{X}\timesd\fam{Y}$, that is $\pi_\fam{X}(x)=\pi_\fam{Y}(y)=\delta$.  
We may choose sections $\sigma_\fam{X}$ and $\sigma_\fam{Y}$ through $x,y$ respectively, simultaneously defined on a sufficiently small neighborhood $U\ni\delta$.  
Then $\sigma\colon U\to\fam{X}\timesd\fam{Y}$ given by $t\mapsto (\sigma_\fam{X}(t),\sigma_\fam{Y}(t))$ is a section through $(x,y)$.
\end{proof}

\begin{lemma}
The coproduct of $\fam{X}$ and $\fam{Y}$ in $\cat{Fam}_\Delta$ has total space the disjoint union of spaces $\fam{X}\sqcup\fam{Y}$ and projection map the union of maps $\pi=\pi_\fam{X}\cup\pi_\fam{Y}$.
\end{lemma}
\begin{proof}
Given two families $\fam{X},\fam{Y}$ over $\Delta$ we define the family $\pi_\fam{X}\sqcup\pi_\fam{Y}\colon\fam{X}\sqcup\fam{Y}\to\Delta$, and observe that this satisfies the universal property of the coproduct.  
Furthermore $\fam{X}\sqcup\fam{Y}$ is an object of $\cat{Fam}_\Delta$ as the disjoint union of maps admitting local sections also admits local sections.
\end{proof}

\section{Families of Groups}
\label{sec:Fam_Grp}
\index{Family!Groups}

The families described so far have been purely topological,  capturing only what it means for the topological structure of the fibers to vary continuously over the base.  In many applications it is important to keep track of how additional data, such as an algebraic structure varies continuously, and so in this section we lay out the necessary formalism.

It is easy to give a "good definition" of a family of groups: it should be a family of spaces, where each space has the additional structure of a group, and the group operation varies continuously over the base.
A particularly succinct way to construct this, which generalizes easily to the definition of other algebraic gadgets, is through the notion of group objects in a category, which we review below.
A group may be defined as a pointed set $e\in G$ together with morphisms $\mu\colon G\times G\to G$ and $\iota:G\to G$ satisfying commutative diagrams encoding the axioms of multiplication and inversion.  

\begin{definition}
A group is a set $G$ together with a chosen element $e\colon\{\star\}\to G$ and a  pair of maps $\mu\colon G\times G\to G$, $\iota\colon G\to G$ satisfying the following axioms:
\begin{enumerate}
\item $\mu$ is associative: $\mu\circ (\mu\times \id_G)=\mu\circ (\id_G\times\mu)$
\item $e$ is the multiplicative identity for $\mu$: as maps $\mu(e(-),-)=\mu(-,e(-))=\id_G$.
\item $\iota$ is an inverse for multiplication: if $\delta\colon G\to G\times G$ is the diagonal map $g\mapsto (g,g)$ then $\iota\times \id_G\circ\delta=\id_G\times \iota\circ\delta=\id_G$.
\end{enumerate}
\label{def:GrpAxioms}
\end{definition}

\noindent
This definition of group carefully avoids mentioning any particular elements of $G$ (the identity element is even encoded via a map $e\colon \{\star\}\to G$) and instead formalizes the group operations in terms of commutative diagrams satisfied by $e,\mu,\iota$.
Given a category $\cat{C}$ other than set, we may directly port this definition of a group to define a \emph{group object} of $\cat{C}$: that is, an object $G\in \cat{C}$ together with morphisms $e,\mu,\iota$ acting like the identity element, multiplication and inversion.

\begin{definition}
Let $\cat{C}$ be a category with finite products, and a terminal object $\star\in\cat{C}$.
Then a \emph{group object} in $\cat{C}$ is a quadruple $(G,e,\iota,\mu)$ for $G\in\cat{C}$ an object, and $e\in\Hom_\cat{C}(\star,G)$, $\iota\in\Hom_\cat{C}(G,G)$ and $\mu\in\Hom_\cat{C}(G\times G,G)$ satisfying the group axioms of Definition \ref{def:GrpAxioms}.

\end{definition}

\begin{example}
A group is a group object in $\mathsf{Set}$.
A topological group is a group object in $\mathsf{Top}$.
A Lie group is a group object in $\mathsf{Diff}$.
An abelian group is a group object in the category of groups.	
\end{example}

\subsection{Group Objects in $\Fam_\Delta$}

\begin{definition}[Family of Groups]
A family of groups over $\Delta$ is a group object in $\cat{Fam}_\Delta$.  
\end{definition}

\noindent
Recalling that a morphism of families $\fam{X}\to\fam{Y}$ has to commute with the projections $\fam{X}\to\Delta$, $\fam{Y}\to\Delta$, we see that a commutative diagram of families is satisfied by a collection of maps \emph{if and only if} that same commutative diagram is satisfied by the restriction of the maps to the fibers over each $\delta\in \Delta$ satisfy the same diagram.
This gives a convenient, constructive definition of the group objects of $\Fam_\Delta$.

\begin{definition}
A family of groups is a family $\fam{G}\to\Delta$ together with maps $\mu\colon\fam{G}\timesd\fam{G}\to\fam{G}$ and $\iota\colon\fam{G}\to\fam{G}$ and a global section $\fam{e}\colon\Delta\to\fam{G}$ such that each fiber $\pi\inv\set{\delta}=G_\delta$ has the structure of a group with multiplication $\mu|_{G_\delta\times G_\delta}$, inversion $\iota|_{G_\delta}$ and identity $\fam{e}(\delta)$	
\end{definition}

\noindent
Thus a group object in $\cat{Fam}_\Delta$ is a Lie groupoid with $\cat{Mor}=\fam{G}$, $\cat{Obj}=\Delta$ and the both source and target maps given by $\pi$.  In particular this is a Lie groupoid with $s=t$ over $\Delta$, referred to as a \emph{bundle of groups} in \cite{BettiolPS14} directly generalizing our earlier construction of \emph{one parameter families of groups} from Section \ref{sec:HC_HRR_Family}.

\begin{example}
\label{ex:SO(1,t)}
Any fiber bundle of groups $\fam{G}\to\Delta$ is a group object in $\cat{Fam}_\Delta$.  
But transitions may also occur.  For instance the groups $\SO\smat{1&0\\0&t}$ form a subfamily of $\GL(2;\R)\times\R \to \R$ as $t$ varies, transitioning from $\SO(1,1)$ to $\SO(2)$.  (The underlying topological family in this example is homeomorphic to that in \cref{ex:First_Example}).
\end{example}

\noindent
Much as the space of closed subgroups of a group $G$ is a closed subset of the space of closed subsets, the collection of families of subgroups of a family of groups $\fam{G}$ is closed in a suitable sense.  
More precisely, the following lemma proves useful.

\begin{lemma}
Let $\fam{G}\to\Delta$ be a family of groups and $\fam{H}\to\Delta$ a subfamily.  Then if $\Omega\subset\Delta$ is a dense open subset and $\fam{H}|_\Omega$ is a family of groups, all members of $\fam{H}$ are groups.
\label{lem:Fams_of_Groups}
\end{lemma}
\begin{proof}
Let $\delta\in\partial{\Omega}$ and $x,y\in \fam{H}_\delta$.  Choosing sections $\sigma_x$, $\sigma_y$ through them, we may their product $\sigma_x\cdot\sigma_y$ is well defined in $\fam{G}$ and lies in $\fam{H}$ on the open dense subset $\fam{H}|_\Omega$.  
But $\fam{H}$ is closed so in fact the image of $\sigma_x\cdot\sigma_y$ lies fully in $\fam{H}$.
In particular, $(\sigma_x\cdot\sigma_y)(\delta)=xy$ so $xy\in\fam{H}_\delta$.
Similarly, as inversion is continuous on $\fam{G}$ the section $(\sigma_x)\inv$ has image in $\fam{H}$ so $x\inv\in\fam{H}_\delta$.  
Thus $\fam{H}_\delta$ has the structure of a group.
\end{proof}

\noindent
Homomorphisms between families of groups are morphisms in $\Fam_\Delta$ which restrict fiberwise to homomorphisms of the member groups, defining the subcategory of \emph{families of groups} over $\Delta$.  
More abstractly, the sheaf of local sections of $\fam{G}\to\Delta$ is a groupoid where $\sigma\colon U\to\fam{G}$ can be multiplied by $\tau\colon V\to\fam{G}$ to produce $\sigma\cdot\tau\colon U\cap V\to\fam{G}$ when the domains overlap.  A homomorphism $\fam{G}\to\fam{H}$ in $\Fam_\Delta$ is then a morphism which induces a groupoid homomorphism on the sheaves of sections.  This defines the subcategory of \emph{families of groups} over $\Delta$.	

\begin{definition}
Fix a smooth manifold $\Delta$.  The category $\mathsf{Grp}_\Delta$ of families of groups over $\Delta$	 has as objects the families of groups $\pi\colon\fam{G}\to\Delta$ and morphisms $\Phi\colon\fam{G}\to\fam{H}$ the morphisms of families which restrict fiberwise to group homomorphisms $\fam{G}_\delta\to\fam{H}_\delta$.
\end{definition}

\section{Families of Algebraic Gadgets}

The example of groups provides a template for defining families of algebraic objects.  
Given an algebraic gadget $A$, a \emph{family of $A$s} is given by an \emph{$A$-object in $\Fam_\Delta$}.  
Two others that will be important to us are families of rings and modules, which lead to families of algebras and vector spaces.

\begin{definition}
\label{def:Ring_Object}
A family of rings over $\Delta$ is a ring object in $\Fam_\Delta$.
Unpacking this, a family of rings is given by the data of a family $\fam{R}\to\Delta$ together with morphisms $\fam{R}\timesd\fam{R}\labelarrow{\mu,\alpha}\fam{R}$ for multiplication, addition and sections $\fam{0},\fam{1}\colon\Delta\to\fam{R}$ that give each fiber the structure of a ring.
\end{definition}

\begin{definition}
\label{def:Module_Object}
Given a family of rings $\fam{R}\to\Delta$, a \emph{family of $\fam{R}$-modules} is a family of abelian groups $\fam{M}\to\Delta$ together with an \emph{action map} $\fam{R}\timesd\fam{M}\to\fam{M}$ fiberwise equipping $\fam{M}_\delta$ with the structure of an $\fam{R}_\delta$ module.
\end{definition}

\noindent
A \emph{family of fields} is simply a family of rings where each fiber is actually a field.
Fields (and consequently vector spaces) provide examples of rigid objects in the category of families, which follows directly from the classification of locally compact connected fields.

\noindent
{\bfseries\sffamily Fact:}
The only connected locally compact topological fields are $\R$ an $\C$.

\begin{corollary}
Connected fields are rigid.
\end{corollary}
\begin{proof}
Let $\fam{F}\to\Delta$ be a family of fields.
Then for each $\delta\in\Delta$, the field $\fam{F}_\delta$ has underlying space a smooth manifold, which is locally compact, and connected by assumption.
Thus $\fam{F}_\delta\cong\R$ or $\fam{F}_\delta\cong\C$.
As the dimension of fibers of a smooth family is invariant, the isomorphism type of the fibers of $\fam{F}\to\Delta$ is constant, and the family contains no transitions.
\end{proof}

\noindent
Thus the study of families containing $\R$ or $\C$ is the same as the study of \emph{families of $\R$ or $\C$}.  
Note this does not imply that all families are trivial, for instance the $\C$ bundle over $\S^1$ twisted by the Galois action $z\mapsto \bar{z}$ is nontrivial. 
A family of vector spaces $\fam{V}\to\Delta$ is a family of modules over a family of fields, and by the above any family that contains a real or complex vector space is actually an entire family of real or complex vector spaces.
And, as the parameter map is a submersion all fibers are of the same dimension, and thus isomorphic.

\begin{corollary}
Vector spaces are rigid.	
\end{corollary}

\noindent
However, even more is true: any family of vector spaces is locally trivial topologically, and thus families of vector spaces are precisely vector bundles.

\begin{proposition}
All families of real \& complex vector spaces are locally trivial.
\label{prop:VS_Trivial}
\end{proposition}
\begin{proof}
Let $\mathbb{F}\in\set{\R,\C}$ and let $\fam{V}\to\Delta$ be a family of finite dimensional $\F$-vector spaces.
Because the dimension of the member of a smooth family is an invariant of the family, $\dim \fam{V}_\delta=m$ for all $\delta\in\Delta$.
Choosing $\delta\in\Delta$ fix a basis $\set{b_i}$ for $ \fam{V}_\delta$ and sections $b_i(\cdot)\colon U_i\to\fam{V}$ through $b_i=b_i(\delta)$ all defined on the neighborhood $\delta\in U=\cap_i U_i$. 
From these we can construct the map of families $\Psi\colon\mathbb{F}^m\times U\to\fam{V}|_U$ given by $((a_i),t)\mapsto \sum_i a_i b_i(t)$, which is a $\cat{C}$-map as it uses only the sections and addition, scalar multiplication operations from the family.  
Define the kernel of $\Psi$ to be 
$\ker\Psi=\set{(x,t)\in\mathbb{F}^m\times U\mid \Psi(x,t)=0}$.

Note that $\ker\Psi$ is the inverse image of the zero section of $\fam{V}|_U$ under $\Psi$, and so is closed as $\Psi$ is continuous.  
Additionally, $\ker\Psi(\cdot,\delta)=\{\vec{0}\}$ as $\{b_i(\delta)\}$ is a basis for $\fam{V}_\delta$.  
We claim that there is a neighborhood $W\ni\delta$ such that for all $t\in W$ it also holds that $\ker\Psi(\cdot, t)=\{\vec{0}\}$.  
Given this, the map $\mathbb{F}^m\times W\to \fam{V}|_W$ is injective on each fiber, and thus also surjective as each fiber is dimension $m$.  
Thus $\Psi$ is an isomorphism of vector space families, so $\fam{V}$ is trivial over $W$.

Thus it remains only to prove the claim.  Assume for the sake of contradiction that this is not the case, so every neighborhood of $\delta$ contains points where $\Psi$ is not injective.  
Let $\{W_n\}$ be a collection of neighborhoods of $\delta$ such that $\cap_n W_n=\{\delta\}$.  
Let $\S$ be the unit sphere in $\mathbb{F}^m$.
In each $W_n$ there is some $t_n$ with $\ker\Psi(\cdot,t_n)\neq\{0\}$, and so $\ker\Psi(\cdot,t_n)\cap\S\neq\varnothing$ as it is a linear subspace.  
Pick an element $x_n\in\ker\Psi(\cdot,t_n)\cap \S$ for each $n$.  Now let $K\ni \delta$ be a compact neighborhood.  
Then $K$ contains infinitely many of the $W_n$, and hence $\S\times K$ contains infinitely many of the $x_n$.  
This sequence converges $x_n\to x_\infty$ by the compactness of $K$, and projecting onto $U$ has $t_n\to\delta$ so $x_\infty\in\R^m\times\{\delta\}$.  
By the continuity of $\Psi$ together with the fact that $\Psi(x_n,t_n)=0$ shows $\Psi(x_\infty,\delta)=0$.  
But this means $\ker\Psi(\cdot, \delta)\neq\{0\}$, a contradiction.
\end{proof}

\section{Families of Algebras}
\label{sec:Fam_Algs}
\index{Family!Algebras}

As the theory of vector spaces and fields yields no interesting transitions, we expand our scope and look to the theory of modules over families of algebras.

\begin{definition}
A family of algebras is an algebra object in $\Fam_\Delta$.  That is, a family of vector spaces $\fam{A}\to\Delta$ over a family of fields $\fam{F}\to\Delta$ equipped with a bilinear operation $\mu\colon\fam{A}\timesd\fam{A}\to\fam{A}$ giving each fiber $\fam{A}_\delta$ the structure of and $\fam{F}_\delta$ algebra.
\end{definition}

But by the rigidity results above, in the smooth category we have a much simpler description.

\begin{corollary}
A family of  $\mathbb{F}$-algebras over $\Delta$ is given by the data of a $\mathbb{F}$-vector bundle $\fam{A}\to\Delta$ together with a multiplication $\mu\colon\fam{A}\timesd\fam{A}\to\fam{A}$ giving each fiber the structure of a $\mathbb{F}$-algebra.
\end{corollary}

\noindent
Restricting the action to $\R\times\Delta\subset\fam{A}$, a family $\fam{M}$ of $\fam{A}$-modules has an underlying family of vector spaces, which by Proposition \ref{prop:VS_Trivial} is a vector bundle over $\Delta$.  Thus families of algebras and their modules remain locally trivial topologically.  However the algebraic structure is allowed to vary in much more interesting ways through the allowance of zero divisors, leading to an abundance of interesting transitions.

\begin{example}
Let $\pi\colon\R^3\to\R$ be the projection onto the last coordinate be the trivial $\R^2$ bundle over $\R$, and equip $\R^3$ with the multiplication map $\mu:\R^3\times_{\R}\R^3\to\R^3$ given by $(x,y,z)\times (x',y',z)=(xy+zx'y',xy'+yx',z)$.  This defines an algebra multiplication on each fiber of $\pi$ such that $\pi\inv\set{-1}\cong\C$ and $\pi\inv\set{1}\cong\R\oplus\R$.  
\end{example}

\subsection{Families of Lie Algebras}

As we the study of families of geometries will involve a significant number of families of Lie groups, we take a brief moment to introduce their infinitesimal counterparts: families of Lie algebras.

\begin{definition}
A family of Lie algebras $\fam{g}\to\Delta$ is a Lie algebra object in $\cat{Fam}_\Delta$.  That is, it is a family of vector spaces equipped with a bilinear map $[\cdot,\cdot]\colon\fam{g}\timesd\fam{g}\to\fam{g}$ giving each fiber the structure of a Lie algebra.
\end{definition}

\noindent
Note that the Lie algebra objects of $\cat{Fam}_\Delta$ are also known as \emph{weak Lie algebra bundles} in the literature \cite{Coppersmith77}.

\begin{proposition}
Every smooth family of groups $\fam{G}\to\Delta$ has a corresponding smooth family of Lie algebras $\fam{g}\to\Delta$.	
\end{proposition}
\begin{proof}
Let $\fam{G}$ be a family of groups in $\Fam_\Delta$. 
Then $\fam{G}$ is a smooth manifold with tangent bundle $T\fam{G}$.
The family projection $\pi\colon\fam{G}\to\Delta$ is a smooth submersion, defining the sub-bundle $T^\pi\fam{G}=\bigcup_{\delta\in\Delta}T\pi\inv(\delta)\subset T\fam{G}$ consisting of the tangent bundles to each $\fam{G}_\delta$.
The tangent spaces at the identity $\fam{e}_\delta$ of each fiber $\fam{G}_\delta$ form the pullback bundle $\fam{g}:=\fam{e}^\ast (T^\pi\fam{G})\to\Delta$, which each inherit a natural Lie algebra structure arising from $\fam{G}_\delta$.
Thus it only remains to show that these Lie algebra structures vary continuously over $\Delta$.  

Let $\delta\in\Delta$ and $v,w\in\mathfrak{g}_\delta$.  Then let $\sigma,\tau\colon U\to\fam{g}$ be sections of $\fam{g}\to\Delta$ through $v,w$ respectively.  Define the vector fields $V, W$ as the left-invariant vector fields generated by $\sigma, \tau$: for any $g\in \fam{G}_t\subset\fam{G}|_U$, $V(g)$ is equal to the pushforward of $\sigma(t)$ by the derivative of the homeomorphism induced by some section $\alpha$ of $\fam{g}\to\Delta$ through $g$ and similarly for $W$.  Then $[V,W]$ is the vector field defined by $[V,W](f)=V(W(f))-W(V(f))$ for $f\in C^\infty(\fam{G}|_U)$ and $[v,w]=[V,W]_p$, so the Lie bracket structure arises from a continuous construction on vector fields of $\fam{G}$.
\end{proof}

\noindent
Going the other direction, and integrating a family of Lie algebras into a family of Lie groups is much more delicate, related to difficult problems in the theory of Lie groupoids.
Partial results in this direction will be treated in Chapter \ref{chp:Constructing_Fams}.

\chapter{Constructing Families of Geometries}
\label{chp:Constructing_Fams}

This chapter builds up the theory of families, providing techniques for constructing families of spaces, groups.
The main tools will be description of \emph{actions of families}, the construction of \emph{pullbacks, exponentials} and \emph{quotients}.
This provides the necessary language to define \emph{families of geometries} and develop their basic theory.

\section{Pullbacks}
\label{sec:Pullback_Families}
\index{Family!Pullback}
\index{Pullback Families}

One of the most important constructions for future applications is the \emph{pullback} of families along morphisms, which we define and study below.

\begin{definition}
Let $\fam{X}\to\Delta$ be a family, and $f\colon D\to\Delta$	 be a morphism.  
Then the pullback family $f^\star\fam{X}\to D$ has total space $\fam{X}\timesd D=\set{(x,d)\mid f(d)=\pi(x)}$ and projection 
$f^\star\fam{X}=\fam{X}\timesd D\stackrel{\pi^\star}{\longrightarrow} D$
defined by $(x,d)\mapsto d$.
\end{definition}

\begin{lemma}
The projection map $\pi^\star\colon f^\star\fam{X}\to D$ in the definition above admits local sections.
\label{lem:Pullbacks_Admit_Sections}
\end{lemma}
\begin{proof}
Let $(x,d)\in\fam{X}\timesd D$.  
Then $\delta=f(d)=\pi(x)$ so $x\in\fam{X}_\delta$.  
As $\fam{X}\to\Delta$ is a family let $\sigma:V\to\fam{X}$ be a local section of $\pi$ through $x$.  
Pulling back gives a map $\sigma\circ f:f\inv\{V\}\to\fam{X}$ from which the map $f^\star\sigma=(\sigma f,\id_D):f\inv\{V\}\to\fam{X}\times D$ can be created.  
As $\pi(\sigma(f(d)))=f(d)$, the map $F$ has image in $\fam{X}\timesd D$, and $\pi^\star\circ (f^\star\sigma)(d)=\pi^\star(\sigma(f(d)),d)=d$ so $f^\star\sigma$ is a section.  
Finally noting $f^\star(d)=(\sigma(f(d)),d)=(\sigma(\delta),d)=(x,d)$ shows $f^\star\sigma$ is a section of $\pi^\star$ through $(x,d)$.
\end{proof}

\begin{figure}
\centering\includegraphics[width=0.7\textwidth]{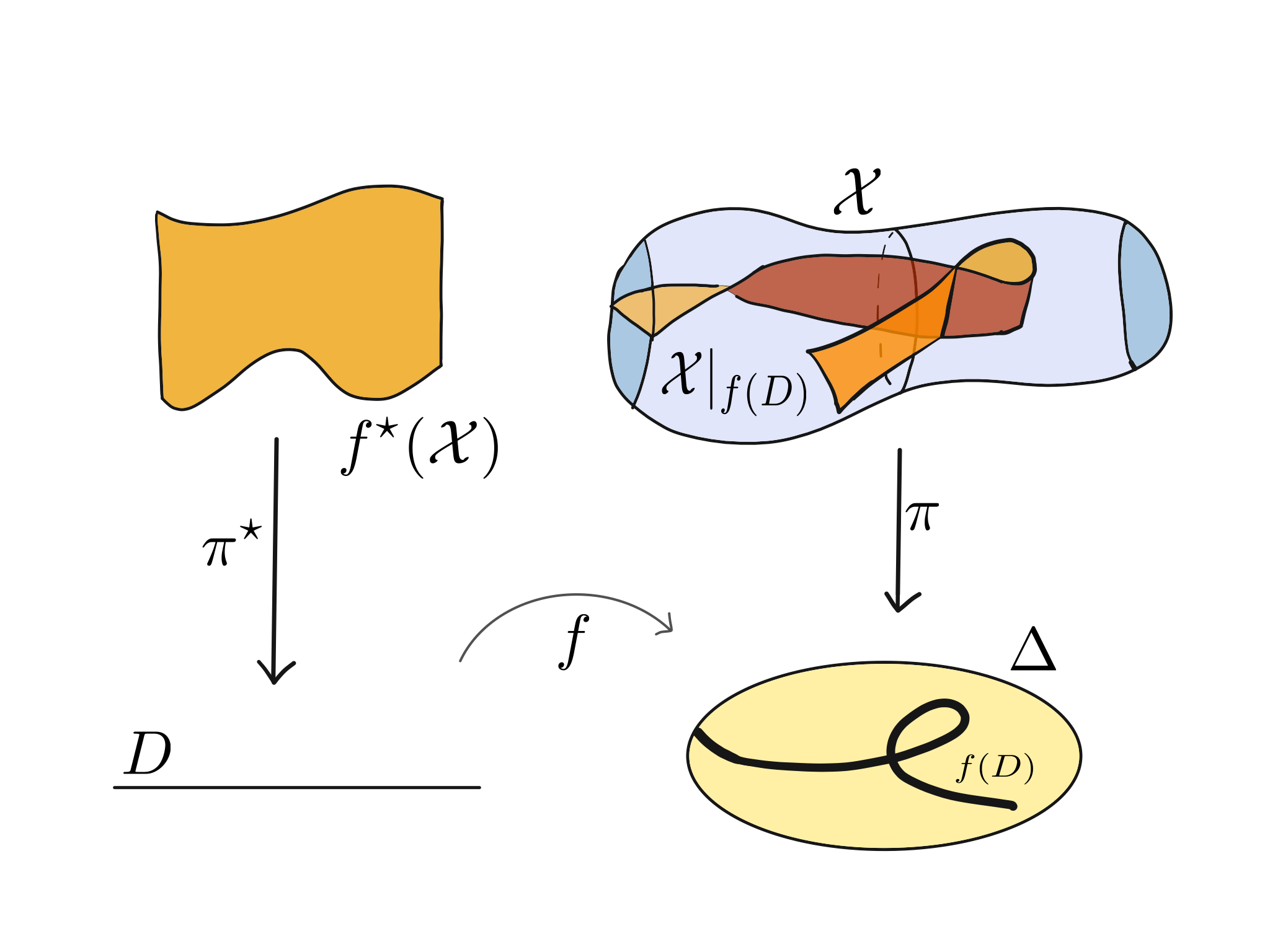}
\caption{Schematically illustrating pullback families.}	
\end{figure}

\noindent
Thus, $\fam{X}\timesd D\to D$ is an object of $\Fam_D$ whenever the fibered product $\fam{X}\timesd D$ exists.
This always holds in the smooth category, as the pullback is along a submersion $\pi\colon\fam{X}\to\Delta$;
but the result above applies to a wide number of topological categories for weak families as well.  
The ubiquity of the pullback construction in applications arises from the fact that it is not only defined on objects, but in fact determines a functor, coherently pulling all families defined over one base back to another.

\begin{observation}
A morphism $D\labelarrow{f}\Delta$ induces a functor $\Fam_\Delta\labelarrow{f^\star}\Fam_D$.
\label{obs:Pullback_Functor}
\end{observation}
\begin{proof}
If $\Phi\colon\fam{X}\to\fam{Y}$ is a morphism of families over $\Delta$ and $f\in\Hom(D,\Delta)$ then $f^\star\Phi\colon f^\star\fam{X}\to f^\star\fam{Y}$ defined by $(f^\star\Phi)(x,d)=(\Phi(x),d)$ is a morphism of families.  This assignment obviously respects composition, as $f^\star(\Phi\circ\Psi)=(f^\star\Phi)\circ(f^\star\Psi)$ and so the operation of pullback defines a functor $\Fam_\Delta\to\Fam_D$
\end{proof}

\noindent
The notions of subfamily and restricted family can be phrased categorically. 
A \emph{(open) subfamily} $\fam{Y}\subset\fam{X}$ in $\Fam_\Delta$ is a pair $(\fam{Y},\iota)$ of a family together with a monomorphism $\fam{Y}\labelarrow{\iota}\fam{X}$.  
When in addition $\iota$ is a closed map, $\fam{Y}$ is a \emph{subfamily} of $\fam{X}$.  
A \emph{restricted family} $\fam{X}|_U$ is the pullback of $\fam{X}\to\Delta$ along the inclusion $\iota\colon U\inject \Delta$.  
A \emph{restricted subfamily} $\fam{Y}|_U$ of $\fam{X}$ is then the combination of these, the pullback along the inclusion $U\inject \Delta$ of the monomorphic image of $\fam{Y}\inject\fam{X}$.

\begin{observation}
\label{obs:Pullback_Eqns}
Let $\fam{X},\fam{Y}$ be objects in $\Fam_\Delta$ and $\Phi\colon\fam{X}\to\fam{Y}$ a morphism of families which is a smooth submersion.
We may then think of $\Phi$ as equipping $\fam{X}$ with the structure of a family over $\fam{Y}$.
Then the pullback $\sigma^\star\fam{X}$ of $\fam{X}\to\fam{Y}$ along any section $\sigma\colon U\to\fam{Y}$ is naturally a restricted subfamily of $\fam{X}\to\Delta$.
\end{observation}
\begin{proof}
The pullback $\sigma^\star\fam{X}\to U$ is a family over $U$, and so it suffices to show that the projection map $\mathsf{pr}\colon\fam{X}\times U\to\fam{X}$, $(x,u)\mapsto x$ is a monomorphism when restricted to $\sigma^\star\fam{X}$.  But if $\mathsf{pr}(x,u)=\mathsf{pr}(y,v)$ then $x=y$ so $\sigma(u)=\Phi(x)=\Phi(y)=\sigma(v)$ and hence $u=v$ as $\sigma$ is injective.
\end{proof}

\subsection{Constructing Pullbacks}
\label{subsec:Construction_Pullback}

As seen in Observation \ref{obs:Pullback_Eqns}, when an equation $\fam{X}\labelarrow{\Phi}\fam{Y}$ between families over $\Delta$  gives $\fam{X}$ the structure of a family over $\fam{Y}$, the pullback along any section of $\fam{Y}\to\Delta$ captures the solutions to $\Phi=\sigma$ as a subfamily of $\fam{X}$.  
Many natural objects can be defined as the solution sets to such equations (point stabilizers are $\set{g\mid g.x=x}$, orthogonal groups are $\set{A\mid A^TJA=J}$ etc) and so understanding when a map $\Phi\in\Hom_{\Fam_\Delta}(\fam{X},\fam{Y})$ actually gives a family $\Phi\colon\fam{X}\to\fam{Y}$ will be of substantial use.

\noindent
Family projections in the smooth category are given by submersions, and the techniques of smooth topology provide us with some useful checks for when a map between families is actually submersive.


\begin{lemma}
Let $\sigma:\Delta\to\fam{X}$ be a section of $\pi\colon\fam{X}\to\Delta$.  
Then for each $x=\sigma(\delta)$ the tangent space $T_x\fam{X}$ decomposes as a direct sum $T_x\fam{X}=T_x\sigma(\Delta)\oplus T_x\pi\inv\set{\delta}$.
\end{lemma}
\begin{proof}
Note $T_x\sigma(\Delta)=\mathrm{img}(d\sigma_\delta)$ and $T_x\pi\inv\set{\delta}=\ker(d\pi_x)$.
As $\sigma$ is an embedding $\pi$ a submersion, $\dim \sigma(\Delta)=	\dim\Delta$ and $\dim\pi\inv\set{\delta}=\dim\fam{X}-\dim\Delta$ respectively.  
Thus the tangent spaces to these submanifolds direct sum to $T_x\fam{X}$ iff $\mathrm{img}(d\sigma_\delta)\cap\ker(d\pi_x)=\set{0}$.  
But $\pi\sigma=\id_\Delta$ so $d\pi_x\circ d\sigma_\delta=\id_{T_\delta\Delta}$, so if $v=d\sigma_\delta(w)\in\ker\pi_x$ then $d\pi_x d\sigma_\delta(w)=0$ so $w$, and hence $v=0$.  
Thus the tangent spaces intersect trivially.
\end{proof}

\noindent
Given a decomposition of the tangent space to a point in the codomain of a smooth map, one can check the map is a submersion by checking that its differential is onto each subspace in the decomposition.  
This gives a \emph{fiberwise} check for when a map between families is a submersion.

\begin{lemma}
Let $\Phi\colon\fam{X}\to\fam{Y}$ be a map of smooth families over $\Delta$.  
Then $\Phi$ gives $\fam{X}$ the structure of a family over $\fam{Y}$ if for each $\delta\in\Delta$ the restriction $\Phi_\delta\colon\fam{X}_\delta\to\fam{Y}_\delta$ is a family. 
\label{prop:Fiberwise_Family=Family}
\end{lemma}
\begin{proof}
Let $x\in\fam{X}$ with $\pi_\fam{X}(x)=\delta$ and choose a section $\sigma\colon U\to\fam{X}$ through $x$.  Then $\Phi\sigma\colon U\to \fam{Y}$ is a section through $y=\Phi(x)$, and so by the above lemma $\sigma$ and $\Phi\sigma$ provide the direct sum decompositions
$T_x\fam{X}=T_x\sigma(U)\oplus T_x\fam{X}_\delta$ and $T_y\fam{Y}=T_y\Phi\sigma(U)\oplus T_y \fam{Y}_\delta.$
Restricting $\Phi$ to $\sigma(U)$ gives a homeomorphism $\sigma(U)\to\Phi\sigma(U)$ so $d\Phi_x|_{T_x\sigma(U)}$ is an isomorphism onto $T_y\Phi\sigma(U)$.  
But by assumption the restriction $\Phi_\delta\colon\fam{X}_\delta\to\fam{Y}_\delta$ is a map of families, 
and so a submersion, thus $d\Phi_x|_{T_x\fam{X}_\delta}$ maps onto $T_y\fam{Y}_\delta$ so all together $d\Phi_x\colon T_x\fam{X}\to T_y\fam{Y}$ is surjective.  
Thus $\Phi$ is a submersion so $\Phi\colon\fam{X}\to\fam{Y}$ is a smooth family. 
\end{proof}

\begin{corollary}
If $\fam{X}\labelarrow{\Phi}\fam{Y}$ fiberwise gives families $\fam{X}_\delta\to \fam{Y}_\delta$ then given any $f:\Delta\to\fam{Y}$ the solution space of $\Phi(\cdot)=f(\delta)$, denoted $\Sigma(\Phi=f)$, is a family over $\Delta$.	
\end{corollary}

\begin{figure}
\centering\includegraphics[width=0.65\textwidth]{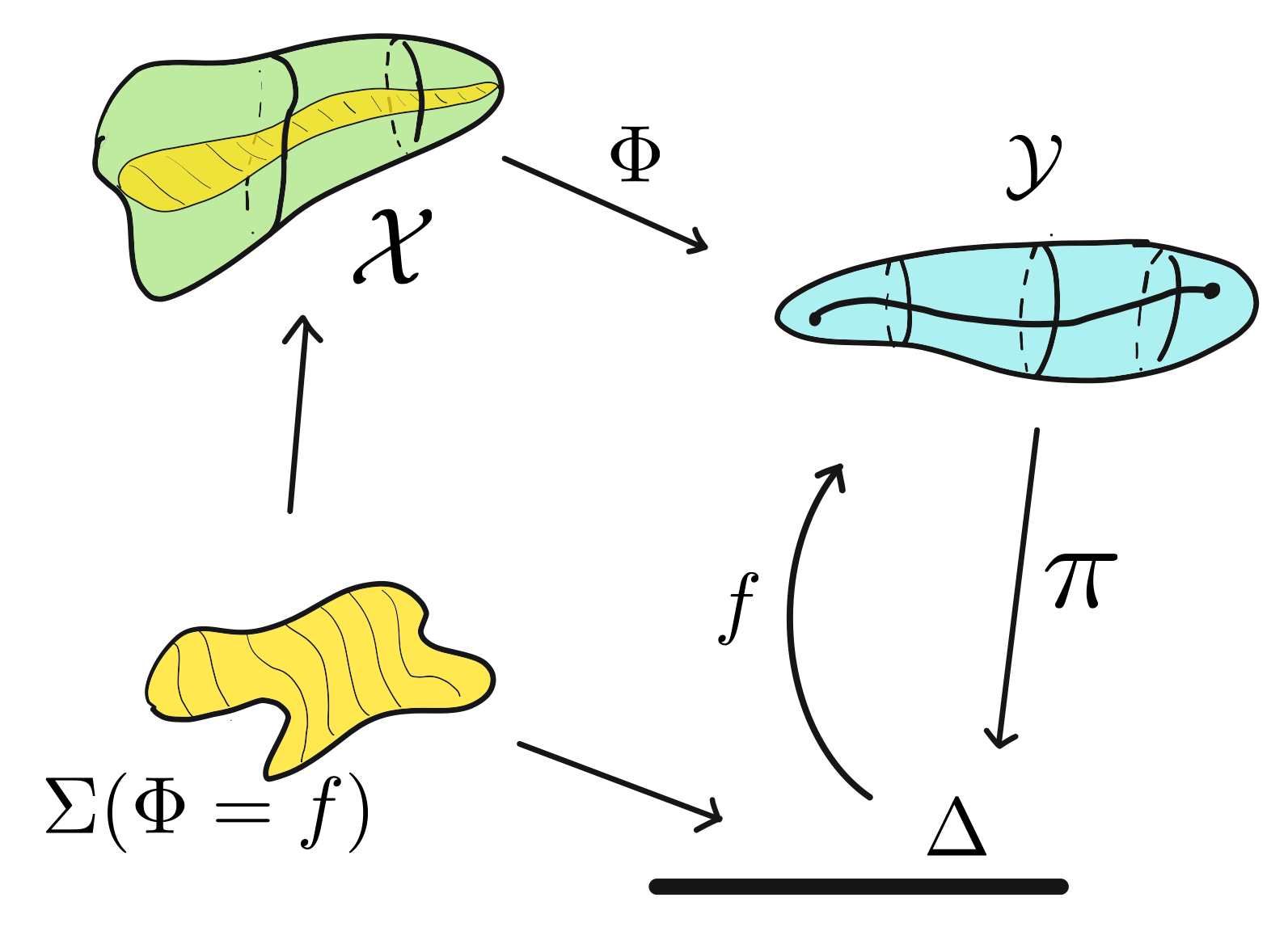}
\caption{A pullback family as a solution to an equation in $\Fam_\Delta$.}	
\end{figure}

\noindent
We can use this to more quickly prove certain subfamilies are families; for example we recall the elements of unit norm int he $\C$ to $\R\oplus\R$ transition.

\begin{example}
Let $\Lambda_\R$ be the family of algebras from Definition \ref{def:Lambda_R} and $\nu$ be the norm map sending a point $x\in \Lambda_\delta$ to $\nu(x)=(x\bar{x},\delta)$ in the trivial $\R$ family $\fam{R}=\R\times\R\to\R$.
Then the pullback of the constant section $\fam{1}\colon\R\to\fam{R}$ sending $\delta\mapsto(1,\delta)$ is the family of units $\U(\Lambda_\R)$.

\end{example}

\noindent
Assuming $\fam{X}\labelarrow{\Phi}\fam{Y}$ is a family is necessary to solve the equations posed by \emph{all} sections $\sigma\colon\Delta\to\fam{Y}$. 
Given a \emph{specific} section $\sigma$, the pullback $\sigma^\ast\fam{X}\to\Delta$ exists under much weaker conditions, however we will not require such refined analysis in this work.

\begin{observation}
If $\fam{X}\to\fam{Y}$ is a family and $\Delta\to\fam{Y}$ a section with $d\Phi(T\fam{X})$ containing $d\sigma(T\Delta)$ then the pullback family exists.	
\end{observation}

\section{Exponentials}
\label{sec:Exponential}

\noindent
We saw in Section \ref{sec:Fam_Algs} that every family of Lie groups has associated to it a family of Lie algebras, with underlying space the pullback of the vertical tangent bundle to $\fam{G}$ with respect to $\pi\colon\fam{G}\to\Delta$ under the identity section $\fam{e}\colon\Delta\to\fam{G}$.
The inverse problem of integrating families of Lie algebras into families of Lie groups has been studied under other names \cite{DouadyL66} (recall, a family of Lie algebras is a \emph{weak Lie algebra bundle} and a family of Lie groups is a \emph{Lie groupoid with equal source and target}) and is quite technically delicate: such an integrated family does not always exist if the Lie groupoid is required to be Hausdorff \cite{Coppersmith77}!  

\noindent
Here we concern ourselves with a more concrete question: given a family of groups $\fam{G}\to\Delta$ and a \emph{subfamily} $\fam{h}\to\Delta$ of its corresponding Lie algebra family, when does the exponential of $\fam{h}$ have the structure of a family of groups?

\begin{proposition}
Let $\fam{G}\to\Delta$ be a family of Lie groups with Lie algebra bundle $\fam{g}$, and exponential map $\fam{exp}\in\Hom_{\Fam_\Delta}(\fam{g},\fam{G})$.  
If $\fam{h}\to\Delta$ is a subfamily of $\fam{g}$, let $\fam{H}$ denote the collection of groups generated by the exponential $\langle\fam{exp}(\fam{h})\rangle\subset\fam{G}$.
Then the projection map $\pi\colon\fam{G}\to\Delta$, restricted to $\fam{H}$, admits local sections.
\label{prop:Exp_Family}
\end{proposition}
\begin{proof}
Let $A\in\langle \fam{exp(h)}\rangle$ with $\pi(A)=\delta$.  
Then $A=A_1\cdots A_n$ for $A_i\in\exp(\mathfrak{h}_\delta)$, and so $A_i=\exp(X_i)$ for some $X_i\in\mathfrak{h}_\delta$. 
As $\fam{h}\to\Delta$ is a family by assumption, there are local sections $\sigma_i\colon U_i\to \fam{h}$ with $\sigma_i(\delta)=X_i$, which exponentiate to sections $\tau_i=\fam{exp}\circ\sigma_i$ through $A_i$ as $\fam{exp}$ is smooth.  Using that multiplication is smooth on the entire family $\fam{G}$, the product of these is a smooth section $\tau=\prod_{i=1}^n\tau_i$ defined on the neighborhood $\delta\in\cap_i U_i$.  Evaluating at $\delta$ shows $\tau(\delta)=A$ and so $\pi\colon\langle \fam{exp(h)}\rangle\to\Delta$ admits local sections.
\end{proof}

\noindent
Unfortunately, whether or not the resulting collection actually forms a \emph{subfamily} of $\fam{G}$ is quite delicate; the Barber Pole of Example \ref{ex:Spiraling_Cyl} comes back yet again.

\begin{example}
\label{ex:Spiraling_Cyl}
Let $G=\S^1\times\R$ and consider the trivial family $G\times\R\to\R$ with corresponding trivial abelian Lie algebra family $\R^2\times\R\to\R$.  Let $\mathfrak{h}_t\subg\R^2$ be the one dimensional Lie algebra $\mathfrak{h}_t=\R(\cos t,\sin t)$ with exponential $H_t=\set{(e^{i s \cos t},s\sin t)\mid s\in\R}$.  Then the collection $\fam{H}=\bigcup_{t\in\R} H_t\times\set{t}$ is \emph{not} a subfamily of $G\times\R$, as the groups $H_t$ are not even Chabauty continuous in $G$.  Indeed as $t\to 0$ the geometric limit of the $H_t$ is the entire cylinder $\S^1\times\R$, but the group $H_{t}=\set{(e^{is},0)\mid s\in\R}$ is just the $\S^1$ factor.
\end{example}

\begin{example}
Consider the trivial family $\GL(2;\R)\times\R\to\R$ and for each $t\in\R$ let 	$\fam{H}_t=\SO(\diag(t,1))$.  Then $\fam{H}\to\R$ is a smooth family of Lie groups transitioning from the two component group $\SO(1,1)$ to the one-component group $\SO(2)$.  
But as the exponential of the Lie algebra family contains only the connected component of the identity in each slice, $\langle \fam{exp(h)}\rangle$ is an open subset of $\fam{H}$ and not a subfamily. 	
\end{example}

\noindent
Resolving this in general is a future goal
 of this research.  However even with this limited understanding the following gives an easily checkable condition for when a collection of subgroups actually forms a subfamily.
 
 \begin{figure}
 \centering
 \includegraphics[width=0.65\textwidth]{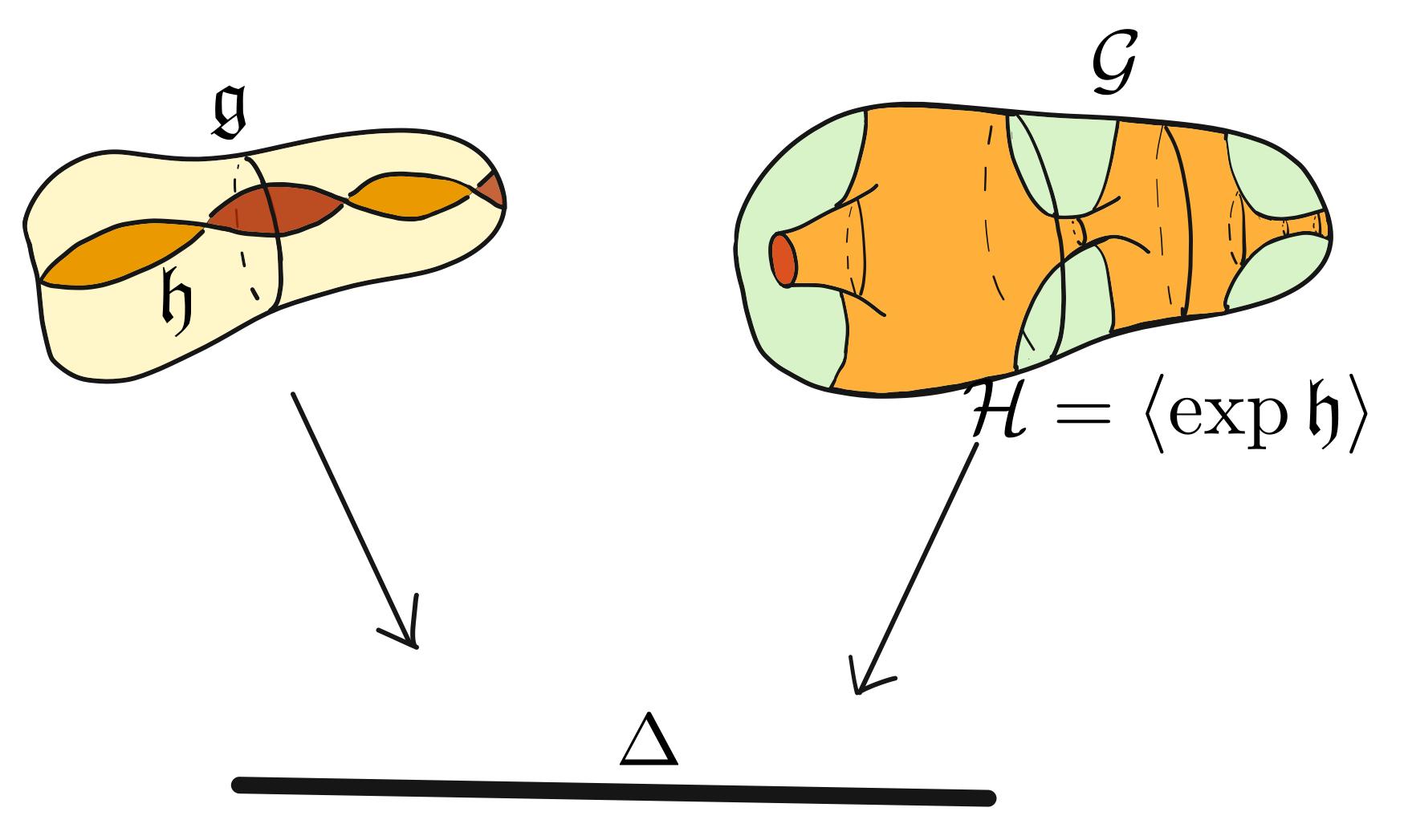}
 \caption{A subfamily $\fam{h}<\fam{g}$ and its exponential.}	
 \end{figure}

\begin{proposition}
Let $\fam{H}\subset\fam{G}$ be a closed submanifold such that each fiber $\fam{H}_\delta$ is a group, and the Lie algebras $\fam{h}\to\Delta$ form a subfamily of $\fam{g}\to\Delta$.
If each $\fam{H}_\delta$ is connected, then $\fam{H}$ is a subfamily of $\fam{G}$.  

If some $\fam{H}_\delta$ are disconnected, then under the additional assumption at least one point of each connected component is contained in the image of a section $\sigma\colon U\to\fam{H}$, the collection $\fam{H}$ is also a subfamily of $\fam{G}$.
\label{prop:Exponential_Check}
\end{proposition}
\begin{proof}
In the case that $\fam{H}_\delta$ is connected then $\langle \exp(\mathfrak{h}_\delta)\rangle=\fam{H}_\delta$ and so the result follows immediately from Proposition \ref{prop:Exp_Family} and the additional assumption that $\fam{H}$ is closed.
In the case that $\fam{H}$ has disconnected slices, we need to slightly modify the argument of Proposition \ref{prop:Exp_Family} to show that the restricted projection continues to admit local sections.  Let $A\in\fam{H}_\delta$, and let $B$ be a point in the same component lying in the image of a section $\sigma\colon U\to\fam{H}$. Then $B\inv A$ is in the connected component of the identity, and so by the previous proposition there is a section $\tau\colon V\to\fam{H}$ through $B\inv A$.  Multiplying by the section through $B$ gives a section $\sigma\cdot\tau\colon U\cap V\to\fam{H}$ through $A$.
\end{proof}

\noindent
The additional hypothesis that each component of each fiber group contains at least one point contained in the image of a local section may seem rather contrived, but it is quite common and easily checkable in practice.  In particular, when considering conjugacy limits there are \emph{global} sections through any points of the original group invariant under the conjugation action.

\section{Actions of Families}
\label{sec:Family_Actions}
\index{Family!Action}
\index{Group Action!Families}

Just as the definition of homogeneous spaces requires the notion of group actions, defining \emph{families} of homogeneous spaces requires a notion of \emph{families} of group actions.

\begin{definition}
An action of $\fam{G}$ on $\fam{X}$ in $\Fam_\Delta$ is given by a morphism $\alpha:\fam{G}\timesd\fam{X}\to\fam{X}$ denoted $\alpha(g,x)=g.x$ such that $\alpha(\fam{e},\cdot)=\id_{\fam{X}}$ and $g.(h.(-))=gh.(-)$ as maps $\fam{X}_\delta\to \fam{X}_\delta$, for all $g,h\in\fam{G}_\delta$.
We may think of this as saying \say{$\alpha$ fiberwise determines an action of $\fam{G}_\delta$ on $\fam{X}_\delta$.}  
\end{definition}

\begin{figure}
\centering\includegraphics[width=0.65\textwidth]{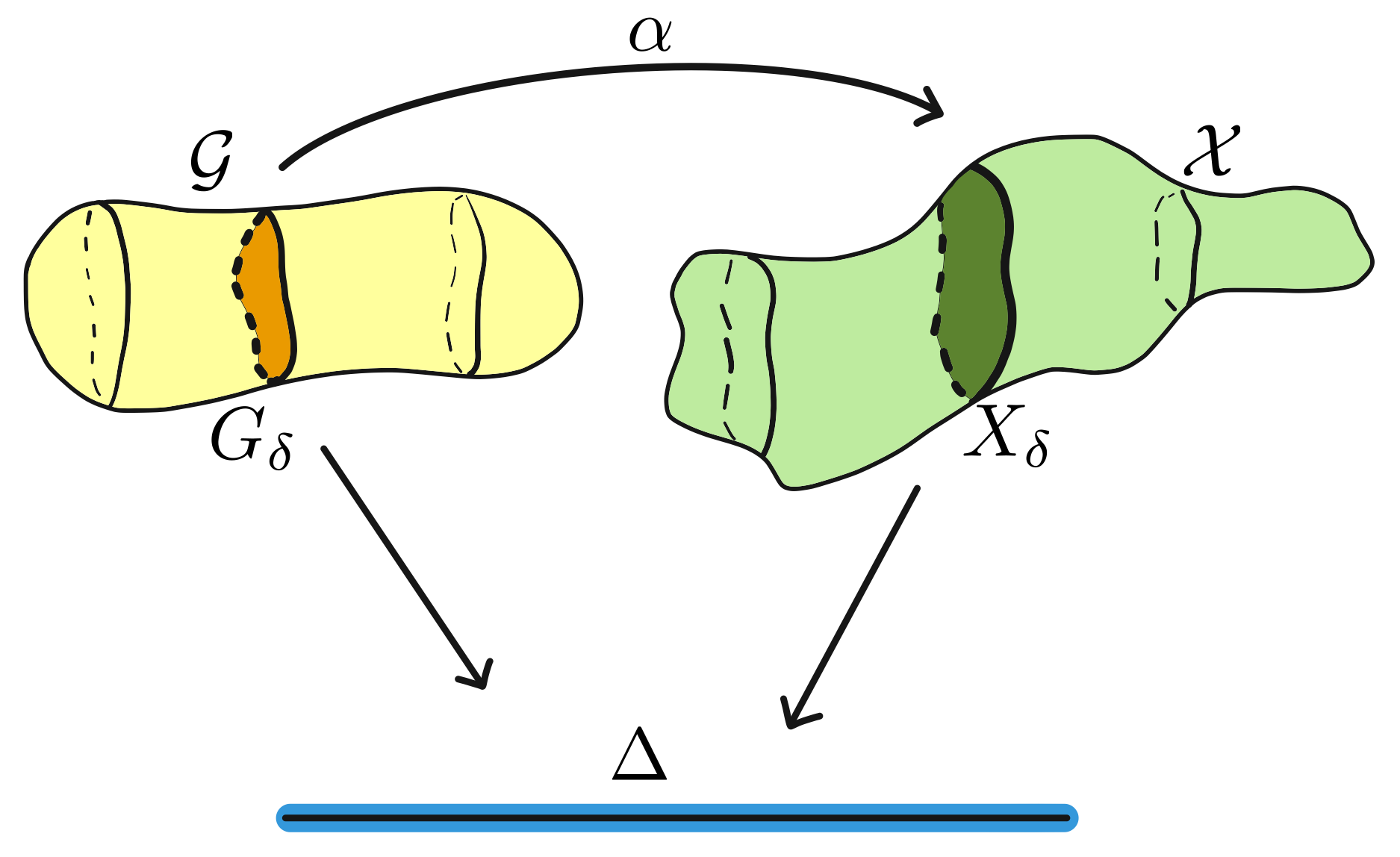}
\caption{A family of group actions.}	
\end{figure}

\begin{definition}
An action $\fam{G}\acts\fam{X}$ is \emph{proper} if the map $\fam{G}\timesd\fam{X}\to\fam{X}\timesd\fam{X}$ defined by $(g,x)\mapsto (x,g.x)$ is a proper map.  
An action $\fam{G}\acts\fam{X}$ is \emph{free} if $g.x=x\implies g\in e(\Delta)$; or equivalently $G_\delta\acts X_\delta$ freely for all $\delta$. 
\end{definition}

\noindent
As an example that will be important later, the action of right translation of a family of subgroups $\fam{H}<\fam{G}$ given by $(g,h)\mapsto gh$ is a free, and also proper as shown below.

\begin{proposition}
\label{prop:Subgroup_Translation}
Let $\fam{G}\to\Delta$ be a family and $\fam{H}\subg\fam{G}$ a subgroup family.  Then the action of $\fam{H}$ on $\fam{G}$ by translation is proper.
\end{proposition}
\begin{proof}
We need to show that the corresponding map $\alpha:\fam{G}\timesd\fam{H}\to\fam{G}\timesd\fam{G}$ given by $(g,h)\mapsto (g, gh)$ is a proper map.  
Let $K\subset\fam{G}\timesd\fam{G}$ be compact with
$\alpha\inv(K)=\set{(g,h)\in\fam{G}\timesd\fam{H}\;\mid\; (g,gh)\in K}$.
Choose a sequence $(g_i,h_i)\in\alpha\inv(K)$, 
then $(g_i, g_ih_i)\in K$ subconverges $(g_{i_k},g_{i_k}h_{i_k})\to p$.  
Projecting onto each factor shows $g_{i_k}\to g_\infty$ and $g_{i_k}h_{i_k}\to k$ and so $p=(g_\infty, k)\in K$.  

Inversion is a morphism $\fam{G}\to\fam{G}$, so $g_{i_k}\inv$ converges to $g_\infty\inv$, and  $(g_{i_k}\inv, g_{i_k}h_{i_k})$ converges in $\fam{G}\timesd\fam{G}$ to $(g_\infty\inv,k)$.  
But multiplication is a morphism so $\mu(g_{i_k}\inv, g_{i_k}h_{i_k})=g_{i_k}\inv g_{i_k}h_{i_k}=h_{i_k}$ converges to $h_\infty=g_\infty\inv k\in\fam{G}$.  
As $\fam{H}$ is a subfamily, it is closed and $h_\infty\in\fam{H}$.
Thus, $(g_{i_k},h_{i_k})\to (g_\infty,h_\infty)\in \fam{G}\timesd\fam{H}$.  
But in fact $\alpha(g_\infty, h_\infty)=(g_\infty,g_\infty h_\infty)=(g_\infty, g_\infty g_\infty\inv k)=(g_\infty, k)\in K$ so $(g_\infty, h_\infty)\in \alpha\inv(K)$.  
Thus this space is sequentially compact, and hence compact as the total space / base, being smooth manifolds, are metrizable.
\end{proof}
 
In the usual theory of group actions, a group element $g\in G$ induces a diffeomorphism $X\to X$.
For families of actions, it is not individual elements but rather the \emph{sections of $\fam{G}\to\Delta$} which fulfill this role.

\begin{lemma}
Given an action $\fam{G}\acts\fam{X}$ and a local section $\sigma:W\to\fam{G}$ of $\fam{G}\to\Delta$, the induced map $\hat{\sigma}:\fam{X}|_W\to\fam{X}|_W$ given by
$\hat{\sigma}(x)=\sigma(\pi_\fam{X}(x)).x$
is a diffeomorphism.	
\end{lemma}
\begin{proof}
Let $\sigma:W\to\fam{G}$ be a local section of $\fam{G}\to\Delta$ and $\fam{X}|_W$ the corresponding restriction of $\fam{X}$.  Then $\hat{\sigma}\in\mathsf{End}(\fam{X}|_W)$ as it is expressible as a composition of morphisms,
$\hat{\sigma}(x)=\alpha(\sigma\circ\pi_\fam{G}(\cdot )),\cdot)$, so it suffices to show $\hat{\sigma}$ is invertible.  
As inversion is given by a morphism $\iota\in\mathsf{End}(\fam{G})$, the composition $\iota\circ\sigma$ is a section inducing $\hat{\iota\circ\sigma}\in\mathsf{End}(\fam{X}|_W)$, and $(\hat{\iota\circ\sigma}\hat{\sigma})(x)=\iota(\sigma(\delta(x))).\sigma(\delta(x)).x=x$.
\end{proof}

\begin{figure}
\centering\includegraphics[width=0.65\textwidth]{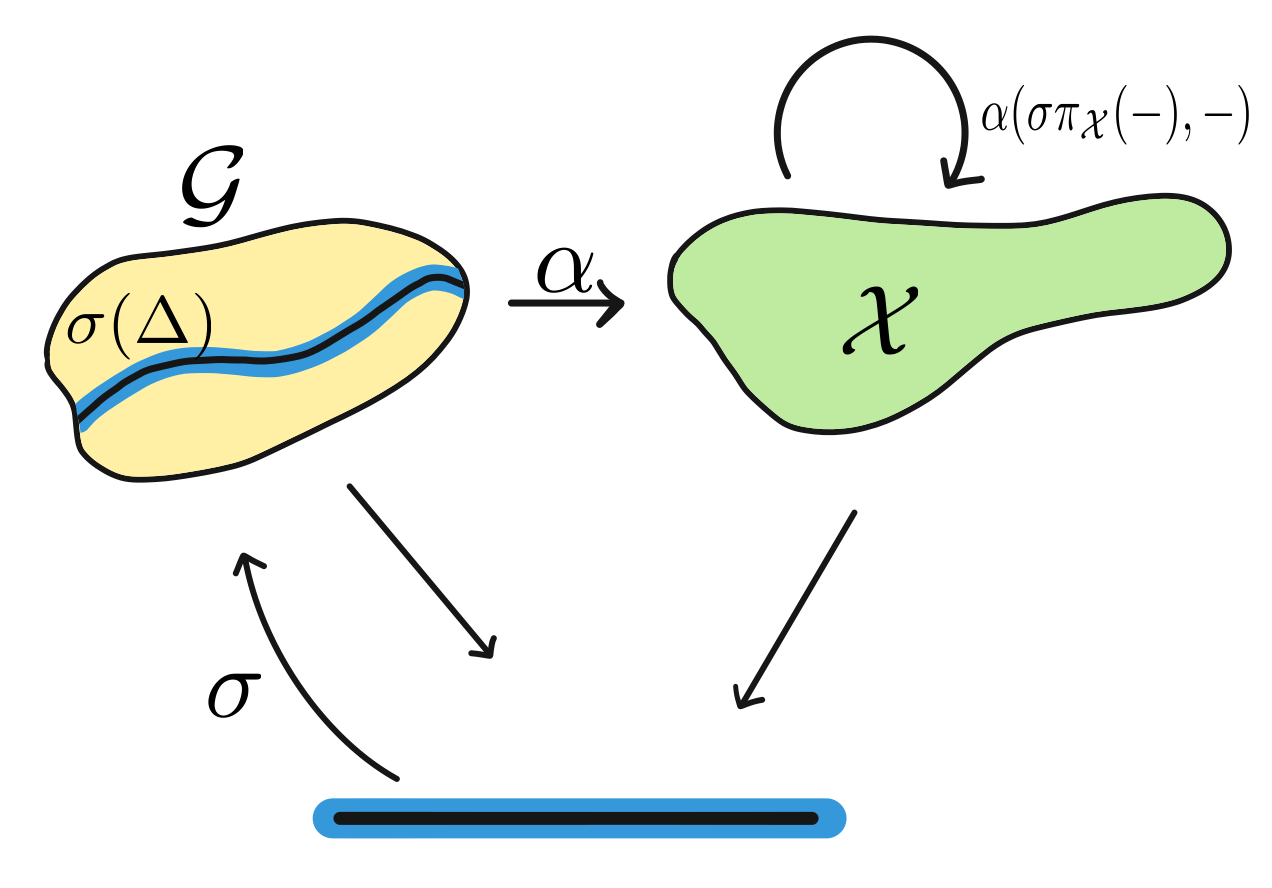}
\caption{The diffeomorphism induced by a section.}
\end{figure}

\noindent
Family actions are intimately related to the standard theory of group actions.
Indeed, actions of the trivial $G$-family over $\Delta$ on families $\fam{X}\to\Delta$ are precisely given by the data of a $G$ action on the total space $\fam{X}$. 
And for nice enough actions, this also works in reverse as seen in Lemma \ref{lem:Quotient_Fam_Action} below.

\begin{lemma}
Let $\fam{G}=G\times\Delta\to\Delta$ be a trivial family of groups and $\fam{X}\to\Delta$ a family.  Then the projection $\fam{G}\to G$ naturally pairs any family action $\fam{G}\acts\fam{X}$ with a standard group action $G\acts \fam{X}$.
\label{lem:Grp_Action_to_Family_Action}
\end{lemma}
\begin{proof}
Let $\tilde{\alpha}\colon\fam{G}\timesd\fam{X}\to\fam{X}$ be the family action and $\mathsf{pr}\colon\fam{G}\timesd\fam{X}\to G\times X$ the projection $((g,\delta),x)\mapsto (g,x)$ of $G\times\Delta$ to $G$ on the first coordinate.  
Then $\alpha\colon G\times\fam{X}\to\fam{X}$ defined by $\alpha(g,x)=\tilde{\alpha}((g,\pi(x)),x)$ defines an action of $G$ on $\fam{X}$. 
\end{proof}

\begin{lemma}
Let $G$ be a group acting on a space $X$ and assume that the orbit map $\piO\colon x\mapsto G.x$ is a submersion onto the orbit space $X/G=\fam{O}$.  
Then the $G$ action on $X$ induces an action of the trivial family $G\times\fam{O}\to\fam{O}$ on $X\to\fam{O}$ in $\Fam_\fam{O}$.
\label{lem:Quotient_Fam_Action}
\end{lemma}
\begin{proof}
Let $\alpha\colon G\times X\to X$ be the action map and $\fam{G}=G\times\fam{O}$.
Then the map $\tilde{\alpha}\colon\fam{G}\times_\fam{O}X\to X$ defined by $\tilde{\alpha}((g,O),x)=\alpha(g,x)$ is a morphism of families as $\pi((g,O),x)=O$ implies $x\in O$ and so $gx\in O$ lies in the same $G$ orbit, thus $\piO\tilde{\alpha}((g,O),x)=O$ so $\piO\tilde{\alpha}=\pi$.
But $\tilde{\alpha}$ obviously satisfies the axioms of a group action fiberwise, as it is just the original action of $G$ restricted to a single orbit.
\end{proof}

\noindent
Viewing this at a higher level of abstraction, we may think of the group action $G\acts X$ as a family of group actions over a point $\set{\star}$.  Then the family $G\times X/G$ is the pullback of $G\to\set{\star}$ over the constant map $X/G\to\set{\star}$.  This suggests a generalization, taking families of actions to \emph{families of families} of actions.

\begin{proposition}
If $\fam{G}\acts\fam{X}$ in $\Fam_\Delta$ such that the projection to the orbit space $\fam{X}\mapsto\fam{X}/\fam{G}$ is a family, then $\fam{G}\acts\fam{X}$ induces an action of a family of groups on the family of orbits $\fam{X}\to\fam{X}/\fam{G}$.
\end{proposition}
\begin{proof}
Let $\bar{\pi}\colon\fam{X}/\fam{G}\to\Delta$ be the family projection.  
Then $\fam{G}\to\Delta$ pulls back to $\fam{G}^\star:=\bar{\pi}^\star\fam{G}\to\fam{X}/\fam{G}$, where the fiber over an orbit $O$ is $\fam{G}_\delta$ for $\delta=\bar{\pi}(O)$ the basepoint over which the orbit lies.  
The action of $\fam{G}^\star$ on $\fam{X}$ is defined fiberwise by the action of $\fam{G}$ on $\fam{X}$; we have simply enlarged the base from $\Delta$ to $\fam{X}/\fam{G}$.  
\end{proof}

\noindent
In both these cases, the original space has been replaced with the \emph{family of orbits}, and has converted a (family) group action into a (family of) \emph{fiberwise transitive} family actions.  This will have important consequences in the coming theory of geometries, allowing us to construct families of geometries from group and family actions.

\subsection{Stabilizers of Actions}

\noindent
If $\fam{G}\acts\fam{X}$ is an action of families over $\Delta$, for each $\delta\in\Delta$ and $x_\delta\in\fam{X}_\delta$ the stabilizer subgroup $\mathsf{stab}_{\fam{G}_\delta}(x_\delta)\subg\fam{G}_\delta$ consists of all elements fixing $x$.  Stabilizers play an important role in the theory of families of geometries to come, so it is necessary to be familiar with some of the subtleties.

\begin{definition}
Let $\fam{G}\acts\fam{X}$ be an action of families over $\Delta$, and $\fam{x}\colon\Delta\to\fam{X}$ a section of the projection map.  Then the \emph{stabilizer collection} of the action is $\fam{stab_G(x)}=\bigcup_{\delta\in\Delta}\mathsf{stab}_\fam{G}(\fam{x}(\delta))$.
\end{definition}

\noindent
It is not always true that the stabilizer collection of a group action is a smooth family, as can be seen in even simple cases such as Example \ref{ex:RP_dS_LightCone} below.
However we will see in \ref{sec:Fams_of_Geos} that stabilizer families of \emph{fiberwise transitive} actions are smooth families for any choice of section, which will be crucial in relating two distinct notions of \emph{family of geometries} to come.

\begin{example}
Consider the standard projective action of $\SO(2,1)$ on $\RP^2$, and produce from this the constant family of groups $\SO(2,1)\times\R\to\R$ acting on the constant family of spaces $\R^3\times\R\to\R$.
Let $\gamma\colon\R\to\RP^2\times\R$ be a section of this projection map such that $\gamma(t)$ is inside t$\H^2\subset\RP^2\times\{t\}$ for $t<0$, $\gamma(0)$ lies on the projectivized lightcone and $\gamma(t)$ is outside the lightcone for $t>0$.
Then the stabilizers of $\gamma(t)$ are one dimensional for $t\neq 0$ but $\mathsf{stab}_{\SO(2,1)}(\gamma(0))$ is $2$ dimensional, so $\mathsf{stab}(\gamma)$ is not a smooth family of groups.
\label{ex:RP_dS_LightCone}
\end{example}

This jump in dimension of the stabilizing subgroup is because of projectivization: if instead of computing the projective stabilizers of $\gamma(t)$ we computed the point stabilizers of some lift $\tilde{\gamma}(t)$, we see in Chapter \ref{chp:Applications}
that these form a smooth family.

\section{Quotients}
\label{sec:Quotient_Families}
\index{Family!Quotient}
\index{Quotient Families}

An action of families $\fam{G}\acts\fam{X}$ gives rise to the \emph{orbit relation} on $\fam{X}$ where $x\sim x'$ if $g.x=x'$ for some $g\in\fam{G}$.  
The quotient is the \emph{orbit space} $\fam{X}/\fam{G}$.  
This orbit space can be badly behaved in general, and so it is of interest to determine which actions have reasonable quotients.
A result of great importance to us is the Quotient Family Theorem, which gives sufficient conditions for the quotient $\fam{X}/\fam{G}$ to be a family in $\Fam_\Delta$. 

\begin{theorem}[Quotient Family Theorem]
Let $\fam{G}\acts\fam{X}$ be a proper free action in $\Fam_\Delta$.  
Then $\fam{X}\labelarrow{\pi}\Delta$ factors as $\fam{X}\labelarrow{\piO}\fam{X}/\fam{G}\labelarrow{\bar{\pi}}\Delta$ with $\fam{X}/\fam{G}\to\Delta$ in $\Fam_\Delta$, as a family of families $\piO\colon\fam{X}\to\fam{X}/\fam{G}$ and $\bar{\pi}\colon\fam{X}/\fam{G}\to\Delta$.
\end{theorem}

\begin{figure}
\centering\includegraphics[width=0.95\textwidth]{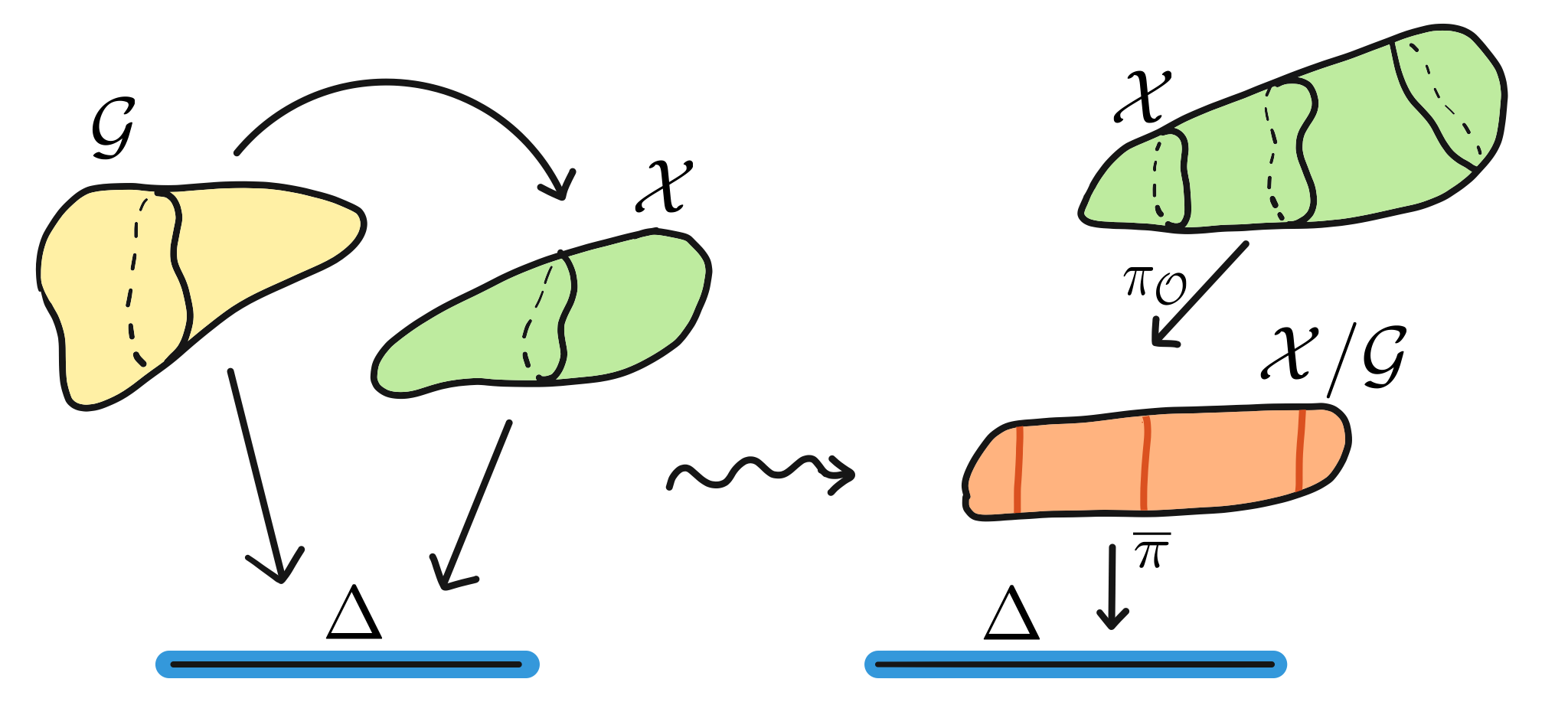}
\caption{The quotient family theorem provides sufficient conditions to take a quotient in the category of families.}	
\end{figure}

\noindent
This is easily the most technical result of Part III, but the proof is a rather straightforward generalization of the Quotient Manifold theorem of smooth topology, with no particularly enlightening new insights.

\begin{theorem}[Quotient Manifold Theorem]
Let $X$ be a smooth manifold and $G$ a Lie group.
Then the orbit space $X/G$ of any proper free action of $G$ on $X$ is a smooth manifold, and the projection $X\to X/G$ is a smooth submersion.
\end{theorem}

\noindent
The main use of this result is to prove that the family-theoretic analogs of Group-Space perspective and Automorphism-Stabilizer perspective on families of geometries are equivalent, which allows us to switch perspectives at will.
In general, the Quotient Family Theorem provides one of the main tools, along with products and pullbacks, of creating new families from old; however we have few independent uses of it in this thesis, and by restricting oneself to the Automorphism Stabilizer perspective; this section may be safely skimmed or skipped on a first read through.

\subsection{Topological Preliminaries}

\noindent
The propositions required to prove the Quotient Family theorem roughly divide into two parts: those of a (non-smooth) topological nature, and those dealing crucially with smooth topology.
The topological results here record basic facts about the quotient space $\fam{X}/\fam{G}$ and the associated projections $\fam{X}\to\fam{X}/\fam{G}$ and $\fam{X}/\fam{G}\to\Delta$.
The remainder of this section is devoted to the rather substantial work involved in the smooth category. 
Note that in the case $\Delta=\{\star\}$ the smooth case implies the quotient manifold theorem, which is already quite technical in the details.  
The proof of this theorem (particularly the proof in \cite{Lee}) provides a guidepost for the argument below.

 \begin{proposition}
Let $\fam{G}$ be a family of groups and $\fam{X}$ a family of spaces both over $\Delta$.  
Then for any action of $\fam{G}$ on $\fam{X}$, the projection $\pi:\fam{X}\to\Delta$ factors through the orbit map $\pi_\fam{O}:\fam{X}\to\fam{X}/\fam{G}$ to $\bar{\pi}\colon\fam{X}/\fam{G}\to\Delta$, which admits local sections through each point of the domain.
\end{proposition}
\begin{proof}
First note that the projection $\pi:\fam{X}\to\Delta$ factors through the orbit map $\pi_\fam{O}:\fam{X}\to\fam{X}/\fam{G}$ as $\pi$ is constant on orbits (the action of $\fam{G}$ on $\fam{X}$ preserves the fibers of $\pi$).
Let $\fam{G}.x\in\fam{X}/\fam{G}$, and let $\sigma:U\to\fam{X}$ be a local section of $\fam{X}\to\Delta$ through $x$.  
Then if $\piO:\fam{X}\to\fam{X}/\fam{G}$ is the orbit map, $\piO\sigma:U\to\fam{X}/\fam{G}$ is a continuous map with $\fam{G}.x\in\piO\sigma(U)$.  
Furthermore $\bar{\pi}(\piO\sigma)=(\bar{\pi}\piO)\sigma=\pi\sigma=\mathsf{id}_U$ so $\piO\sigma$ is a local section of $\bar{\pi}$.  	
\end{proof}

\begin{proposition}
Let $\fam{G}\acts\fam{X}$ be a proper action of family of groups on family of spaces.  Then the orbit space $\fam{X}/\fam{G}$ is locally compact Hausdorff.
\label{prop:Hausd_Orbit_Space}
\end{proposition}
\begin{proof}

Let $\mathcal{G}.x\in\fam{X}/\fam{G}$, and let $K$ be a compact neighborhood of $x\in\fam{X}$.  Then as $\piO$ is an open map, $\piO(K)$ is a compact neighborhood of $\piO(x)=\fam{G}.x\in\fam{X}/\fam{G}$, so the orbit space is locally compact.
Recall that a quotient $Z/\sim$ is Hausdorff if and only if the equivalence relation $\sim$ is a closed subset of $Z\times Z$. Thus to show $\fam{X}/\fam{G}$ is Hausdorff it suffices to show the collection $\set{(x,g.x)\;\mid\; \pi_\fam{X}(x)=\pi_\fam{G}(g)}$ is closed in $\fam{X}\times\fam{X}$.  Note that this is simply the image of the action map $\fam{G}\timesd\fam{X}\to\fam{X}\times\fam{X}$ given by $(g,x)\mapsto (x,g.x)$ which is a proper map by assumption.  But $\fam{X}$ is LCH so $\fam{X}\times\fam{X}$ is, and any continuous proper map into an LCH space is closed, so the orbit relation is closed and we are done.
\end{proof}

\begin{proposition}
Let $\fam{G}\acts\fam{X}$ be any action of a family of groups on a family of spaces.  Then the orbit map $\piO\colon\fam{X}\to\fam{X}/\fam{G}$ is open.
\end{proposition}
\begin{proof}
Let $U$ be open in $\fam{X}$, then we want to show that $\piO(U)$ is open in $\fam{X}/\fam{G}$.  But as $\fam{X}/\fam{G}$ is equipped with the quotient topology, this is open iff $\piO\inv\piO(U)$ is open in $\fam{X}$.
$$\piO\inv\piO(U)=\set{x\;\mid\; \piO(x)\in\piO(U)}=\set{g.u\;\mid\;\pi_\fam{G}(g)=\pi_\fam{X}(u),\; u\in U}=\fam{G}.U$$
Let $g.u\in\piO\inv\piO(U)$ be arbitrary.  As $\fam{G}\to\Delta$ is a family, choose a local section $\sigma:V\to \fam{G}$ of the projection through $g$.  Then $\pi_\fam{X}\inv(V)$ is an open subset of $\fam{X}$ on which  $\sigma(V)$ acts via a homeomorphism.  Let $W=U\cap\piO\inv(V)$. Then $W$ is an open set containing $u$ and  $\sigma(V).W$ is an open set containing $g.u$ contained in $\fam{G}.U=\piO\inv\piO(U)$, so $\piO$ is an open map.
\end{proof}

\subsection{$\fam{G}$-Adapted Charts}

\noindent
The first step is the construction of a particularly nice atlas of charts for $\fam{X}$, which not only clearly separate fibers of $\fam{X}\to\Delta$ but also $\fam{G}$ orbits.  We call such charts $\fam{G}$-adapted charts.

\begin{definition}
Given a smooth $\fam{G}$ action on $\fam{X}$, a chart $(U,\phi)$ on $\fam{X}$ is said to be $\fam{G}$-\emph{adapted} if 
$\phi:U\to I^k\times I^\ell\times I^m$ is a homeomorphism onto a cube such that
\begin{itemize}
\item The fibers of $\delta$ are precisely $\{x\}\times I^{\ell+m}$ in coordinates all fixed $x\in I^k$.
\item The fibers of $\pi_\mathcal{O}$ are precisely $\{x,y\}\times I^m$ for all fixed $x,y\in I^{k+\ell}$.
\end{itemize}

\noindent
That is, the coordinate chart is sliced into parallel copies of $I^{\ell+m}$ each representing the intersection of some $M_\delta$ with $U$, and these are further sliced into parallel copies of $I^m$ representing the intersection of $\mathcal{G}$ orbits with $U$.

\end{definition}

\begin{figure}
\centering\includegraphics[width=0.5\textwidth]{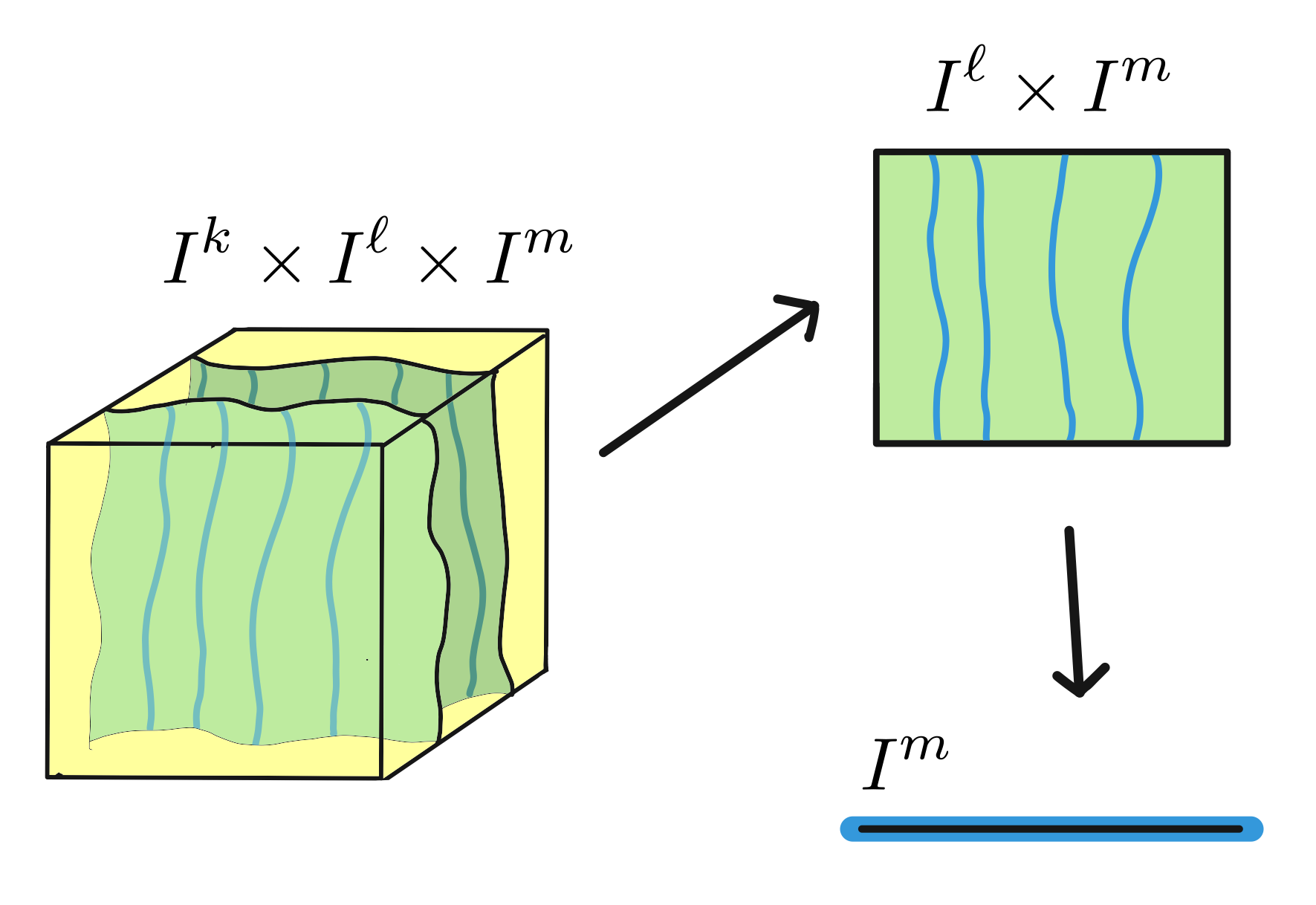}
\caption{A $\fam{G}$ adapted chart for $\fam{X}$ allows both the projection onto the orbit space and the further projection onto the base to be realized as smooth submersions.}
\end{figure}

\noindent
A $\fam{G}$-adapted chart $U$ is said to be \emph{centered at $p\in U$} if $\phi(p)=(0,0,0)$.	
It simplifies things to be able to restrict our attention to coordinate charts centered at specific points of interest; below we show that any chart can be easily modified to be centered at any point in its domain.

\begin{lemma}
\label{lem:Recentering_Charts}
If $(U,\phi)$ is a $\fam{G}$-adapted coordinate chart and $p\in U$ then there is another $\fam{G}$-adapted chart $(U,\psi)$ centered at $p$.
\end{lemma}
\begin{proof}
Let $h:I\to\R$ be the homeomorphism given by  $x\mapsto \frac{x}{1-x^2}$ and $H_N:I^N\to\R^N$ be the same map coordinate-wise.  Then if $N=k+\ell+m$ we may consider the point $H_N\circ\phi(p)\in\R^N$, and create a new chart $U\to\R^N$ via 
$H_N\circ\phi(\cdot)-H_N\circ\phi(p)$
This new chart is clearly centered at $p$, and as $H_N$ fixes the origin, so is the chart 
$\psi=H_N\inv(H_N\circ\phi(\cdot)-H_N\circ\phi(p))$.
This chart continues to be $\fam{G}$-adapted as if $\Pi$ is any hyperplane defined by fixing some coordinates then neither applying $H_N$ nor translation affects which coordinates are free and which are constants.
\end{proof}

\noindent
The technical hurdle to overcome is now to show that such $\fam{G}$-adapted charts actually exist when the given action is free and proper.

\begin{proposition}[$\fam{G}$-adapted charts for $\fam{X}$]
\label{prop:G_Adapted}
Let $\fam{G}$ act freely \& properly on $\fam{X}$.  Then for each $p\in \fam{X}$ there is a $\fam{G}$-adapted chart centered at $p$.	
\end{proposition}

\noindent
This follows from the lemmas below, which are modeled directly off the approach to the Quotient Manifold Theorem given in \cite{Lee}.

\begin{observation}
As $\fam{G}\to\Delta$ is a submersion, each member $\fam{G}_\delta$ is of the same dimension.  By the assumption that the $\fam{G}$ action is free, each orbit in $\fam{X}_\delta$ is diffeomorphic to $\fam{G}_\delta$, in particular $\fam{X}$ is a disjoint union of equidimensional submanifolds $\fam{G}.p$.
\end{observation}

\noindent
Say the dimension of each group, and thus each orbit is $d$.  Then the assignment of each $p\in \fam{X}$ to the tangent plane of the $\fam{G}$ orbit through $p$ determines a section of $\Gr(d,T\fam{X})\to\fam{X}$.  

\begin{lemma}
The map $\fam{X}\to\Gr(d,T\fam{X})$ given by $p\mapsto T_p\fam{G}.p$ defines a smooth $d$-dimensional distribution $\fam{D}$ on $\fam{X}$.	
\end{lemma}
\begin{proof}
We show that about each $p\in\fam{X}$ there is a smoothly varying local frame for $\fam{D}$.	Let $\delta=\pi(p)$ and $U\ni\delta$ a neighborhood about which the Lie algebra family $\fam{g}\to\Delta$ of $\fam{G}\to\Delta$ is locally trivial.  Via this trivialization we choose a local frame for $\fam{g}|U$ via maps $X_i\colon U\to \fam{g}$.

Any such $X_i$ determines a smooth flow on $\fam{X}|_U$, $\Phi_i\colon\R\times\fam{X}|_U\to\fam{X}|_U$ via $(t,p)\mapsto \exp(tX_{\pi(p)}).p$ and thus to a smooth vector field on $\fam{X}|_U$ by differentiation $Y_i\colon\fam{X}|_U\to T\fam{X}|_U$, $p\mapsto \dd{t}|_{t=0}\exp(tX_{\pi(p)}).p$.
Another description of the vector fields $Y_i$ is as follows: each $X_i$ determines a left-invariant vector field on $\fam{G}$, and the action on $\fam{G}$, and for each $p\in \fam{X}$ the vector $Y_i(p)$ is the pushforward of $X_i$ under the homeomorphism of $\fam{G}_\delta$ with $\fam{G}_\delta.p$.

Thus as $X_i\neq 0$, $Y_i(p)\neq 0$ for all $p\in\fam{X}|_U$ and the vectors $\set{Y_i(p)}$ are linearly independent in $T_p\fam{X}$.  Moreover as $\set{X_i}$ is a basis for $T_e\fam{G}_\delta$, the $\fam{Y}_i$ are a basis for $T_p\fam{G}_\delta.p$ and so this gives a smoothly varying local frame for $\fam{D}$ over $U$.

\end{proof}

\begin{lemma}
There are flat charts for the distribution $\fam{D}$.	
\end{lemma}
\begin{proof}
By the proposition above, $\fam{D}$ is a smooth distribution on $\fam{X}$.  As $\fam{D}$ was defined as the tangent planes to the smooth manifolds $\fam{G}.p$, these are integrable surfaces for the distribution, so $\fam{D}$ is \emph{totally integrable}.  An application of the Frobenius theorem from smooth topology then says that flat charts exist.
\end{proof}

\noindent
Flat charts for $\fam{D}$ are consist of neighborhoods $V_p\ni p$ for each $p\in \fam{X}$ and homeomorphisms $\psi_p\colon V_p\to I^{k+\ell}\times I^{n}$ such that the orbits of $\fam{G}$ appear after the homeomorphism as slices $\mathsf{const}\times I^k$.  

\begin{lemma}
The homeomorphisms for the flat chart may be taken so that $	\psi\colon V_p\mapsto I^k\times I^\ell \times I^n$ and the fibers of $\pi\colon\fam{X}\to\Delta$ are unions of slices $\set{x}\times I^{\ell+m}$ and the fibers of $\piO\colon \fam{X}\to\fam{X}/\fam{G}$ are unions of $\set{x,y}\times I^m$.
\end{lemma}
\begin{proof}
The fibers of $\pi$ are the members $\fam{X}_\delta$ of $\fam{X}$, and their tangent spaces form a totally integrable distribution on $\fam{X}$.  As the orbits of $\fam{G}$ are contained in the fibers $\fam{X}_\delta$, in the coordinates $I^{k+\ell}\times I^m$ from the proposition above these surfaces are constant in the $I^m$ direction, and projecting this off gives a totally integrable distribution on $I^{k+\ell}$ , which again by Frobenius admits flat charts $\phi\colon I^{k+\ell}\to I^k\times I^\ell$ such that single the leaves are of the form $\set{x}\times I^\ell$. 
Then the composition $V_p\mapsto I^k\times I^\ell \times I^m$ given by $w\mapsto (\phi(\pi_{I^{k+\ell}}(\psi(p)),\pi_{I^{m}}(\psi(p))$ is such a chart.
\end{proof}

\noindent
Now our work is almost done, we need only shrink the neighborhood $V_p$ such that each orbit passes through only once.

\begin{lemma}
We may choose a neighborhood $V_p$ of $p\in\fam{X}$ such that each $\fam{X}_\delta$ and each $\fam{G}$ orbit passes through $V_p$ at most once.
\end{lemma}
\begin{proof}
This is easy for the $\fam{X}_\delta$: every member $\fam{X}_\delta$ passes through $\fam{X}|_U$ only once for any open $U\subset\Delta$.  
Within each $\fam{X}_\delta$, this follows directly from the argument for the quotient manifold theorem in \cite{Lee}.
Thus it only remains to show that we can do both of these at once.
Choose a point $p$ and a neighborhood $U\ni\pi(p)$.  
For each $u\in U$ use the argument in the quotient manifold theorem to produce a sufficiently small neighborhood in $\fam{X}_u$, through which each orbit passes only once.
Because the groups in the family, and thus their orbits, vary continuously, over a compact neighborhood $V$ of $\pi(p)$ we may choose a uniformly small neighborhood of $p$ diffeomorphic to $\mathbb{B}^n\times V$ such that the intersection with each $\fam{X}_u$ is an $n$-ball through which each orbit passes once.
\end{proof}

\noindent
Thus we have finished the proof of Proposition \ref{prop:G_Adapted}; $\fam{G}$-Adapted charts exist.

\subsection{The Orbit Space $\fam{X}/\fam{G}$}

\noindent
Using $\fam{G}$-adapted charts makes it particularly easy to understand the local topology of $\fam{X}/\fam{G}$.

\begin{proposition}
The orbit space $\fam{X}/\fam{G}$ is a topological manifold of dimension $k+\ell$.	
\label{prop:Orbits_are_Top_Man}
\end{proposition}
\begin{proof}

The quotient space $\fam{X}/\fam{G}$ is Hausdorff by Proposition  \ref{prop:Hausd_Orbit_Space}, and is second countable as $\fam{X}$ was.  
To see that $\fam{X}/\fam{G}$ is locally Euclidean, let $q=\fam{G}.p$ be an arbitrary point of $\fam{X}/\fam{G}$.  
Choosing the representative $p$ of the orbit we let $(U,\phi)$ be a $\fam{G}$-adapted chart for $\fam{X}$ centered at $p$ with $\phi(U)=I^k\times I^\ell\times I^m$.  
Let $V=\piO(U)$ and note that $q\in V$, which is open as $\piO$ is an open map.  
The coordinates on $U$ are given by triplets $(x,y,z)$ for $x\in I^k$, $y\in I^\ell$ and $z\in I^m$.  
Let $Z\subset U$ be the points with third coordinate zero, $Z=\{(x,y,0)\in U\}$.  Then its easy to see that $\piO$ is a bijection $Z\to V$ using the properties of an adapted chart.  
If $\mathcal{G}.\xi\in V=\piO(U)$ then there is some $\eta\in \mathcal{G}.\xi\cap U$, in coordinates $\eta=(x,y,z)$ and so $(x,y,0)\in U$ as well.  
But the fact that $U$ is $\fam{G}$-adapted means that $(x,y,0)$ and $(x,y,z)$ lie in the same $\fam{G}$ orbit, so $\piO(x,y,0)=\piO(x,y,z)=\mathcal{G}.\xi$ so $\piO$ is surjective. 
Injectivity is clear as if $(x,y,0)$ and $(u,v,0)$ map to the same point under $\piO$ then they are in the same $\fam{G}$-orbit, but since $U$ is $\fam{G}$-adapted, $\fam{G}$-orbits intersect $U$ only in single slices of the form $(x,y)\times I^k$ and so in fact $x=u$, $y=v$.

And now we've almost finished!  
Clearly $Z$ is homeomorphic to $I^k\times I^\ell$ and so we just need that the continuous bijection above is a homeomorphism in a neighborhood of $p=(0,0,0)$.  
But restricting to any compact neighborhood of the origin gives us a continuous bijection from a compact space to a Hausdorff space and so we are done.  
\end{proof}

\noindent
Piecing the last three lemmas together shows $\fam{X}/\fam{G}$ is a topological manifold. 

\begin{proposition}
\label{M/G_Is_Manifold}
If $\fam{G}\acts\fam{X}$ is a proper free action the orbit space $\fam{X}/\fam{G}$ is a topological manifold.	
\end{proposition}

\noindent
To begin discussing the smooth properties of $\fam{X}/\fam{G}$, we need first to produce a candidate smooth atlas.  To do so we will look a little closer at the argument in \cref{prop:Orbits_are_Top_Man}, and produce actual charts. 

\begin{lemma}
Any $\fam{G}$-adapted chart $(U,\phi)$ gives rise to a chart $(V,\eta)$ on $\fam{X}/\fam{G}$ with $V=\piO(U)$ and $\eta:V\to I^{k+\ell}$ such that $\eta\circ\piO=\pi_{12}\circ\phi$.
\end{lemma}
\begin{proof}
Let $(U,\phi)$ be a $\fam{G}$-adapted chart and $Z\subset U$ the points which have third coordinate zero under $\phi$.  
We have already seen $\piO$ is a bijection on $Z$, but it is also an open map as if $W\subset Z$ is open, the projection of $W$ is the same as the projection of $\set{(x,y,z)\in U\;\mid\; (x,y,0)\in W}$ which is open in $U$ (its the preimage of $W$ under the continuous projection $U\to I^k\times I^\ell\times\{0\}$) and $\piO$ is an open map and thus a homeomorphism.  
Let $\sigma: V\to Z$ be the inverse, which is a section of $\piO$.

Define the map $\eta:V\to I^k\times I^\ell$ by taking $\piO(x,y,z)\mapsto (x,y)$.  This is well defined due to the fact that $U$ is $\fam{G}$-adapted: if $(x,y,z)$ and $(u,v,w)$ are in the same $\fam{G}$ orbit then $(x,y)=(u,v)$ as orbits are $I^m$-slices. We may actually express $\eta$ as a composition of known maps here $\eta=\pi_{\scriptscriptstyle XY}\phi\sigma$ for $\sigma=\piO|_V\inv $,$\phi$ the original chart map, and $\pi_{\scriptscriptstyle 12}\!$ the projection $I^{k+\ell+m}\to I^{k+\ell}$ removing the $z$-coordinate.  Because $\sigma:V\to Z$ is a homeomorphism and $\pi_{\scriptscriptstyle 12}\phi$ is a homeomorphism when restricted to $Z$, this gives us $\eta:V\to I^{k}\times I^\ell$ is a homeomorphism.  

That this chart $\eta$ satisfies the claimed property $\eta\circ\piO=\pi_{12}\circ\phi$ is easy to see.  Indeed if $p\in U$ let $\mathcal{O}=\piO(p)\in V$ then $\eta(\mathcal{O})=\pi_{12}\circ\phi\circ \sigma(\mathcal{O})$ and $\sigma(\mathcal{O})$ is the point in $\mathcal{O}$ with third coordinate zero with respect to $\phi$.  This is in the same orbit as $p$ and so it has the same two first coordinates as $p$ (this is part of the definition of a well-adapted chart) and so $\pi_{12}\circ\phi(\sigma(\mathcal{O}))=\pi_{12}\circ\phi(p)$ and thus $\eta\circ \piO(p)=\pi_{12}\circ\phi(p)$ as claimed.
	
\end{proof}

\noindent
We call the chart $(V,\eta)$ constructed from $(U,\phi)$ the \emph{induced chart} on $\fam{X}/\fam{G}$. 

\begin{observation}
The equation $\pi_{12}\circ\phi=\eta\circ\piO$ gives a convenient description of $\eta$.  We have $\eta(\mathcal{O})=(x,y)$ if and only if one point (and hence all points by $\fam{G}$-adaptivity) of $\mathcal{O}\cap U$ have first two coordinates $(x,y)$ with respect to $\phi$.
\end{observation}

\begin{lemma}
\label{Vertically_Translated_Charts}
Let $(U,\phi)$ be a $\fam{G}$-adapted coordinate chart for $\fam{X}$ and let $W=\delta(U)$, and $\sigma:W\to \fam{G}$ a section of $\fam{G}\to\Delta$.  
Then $\sigma$ induces a homeomorphism $\hat{\sigma}:\fam{X}|_W\to\fam{X}|_W$ and from this and $(U,\phi)$ we may produce a new coordinate chart
 $(\hat{U},\hat{\phi})=(\hat{\sigma}\inv(U),\hat{\sigma}\circ\phi)$.  Then the induced charts $(V,\eta)$ and $(\hat{V},\hat{\eta})$ are equal.	 
\end{lemma}
\begin{proof}
We first show that $V=\hat{V}$.  
Recall that $V=\piO(U)$ and $\hat{V}=\piO(\hat{U})$, and assume $\mathcal{O}\in V$.  
Then there is some $p\in U$ such that $\fam{G}.p=\mathcal{O}$, but $p\in U$ means $\hat{\sigma}\inv(p)\in\hat{U}$ and so $\fam{G}.\hat{\sigma}\inv(p)\in\hat{V}$.  
But $\hat{\sigma}\inv(p)=\sigma(\delta(p))\inv.p\in\fam{G}.p$ and so $\fam{G}.\hat{\sigma}\inv(p)=\mathcal{O}$ thus $V\subset \hat{V}$.  
Similarly we show $\hat{V}\subset V$.  

To see that $\eta=\hat{\eta}$, let $\mathcal{O}\in V=\hat{V}$ be a $\fam{G}-orbit$, and say $\hat{\eta}(\mathcal{O})=(x,y)$.  
Then by the above observation describing the induced coordinate maps we have that all points of $\mathcal{O}\cap \hat{U}$ have first two coordinates $(x,y)$ with respect to $\hat{\phi}$ and so in particular there is some $q\in \hat{U}\cap\mathcal{O}$ with $\hat{\phi}(q)=(x,y,0)$.  
But then $\hat{\phi}(q)=\phi\circ\hat{\sigma}(q)=\phi(\sigma(\delta(q)).q)=(x,y,0)$ meaning that $\sigma(\delta(q)).q$ has coordinates with first coordinates $(x,y)$ with respect to $\phi$.  
Thus all points of $\mathcal{O}\cap U$ do and so $\eta(\mathcal{O})=(x,y)$, showing $\eta=\hat{\eta}$.
\end{proof}

\begin{proposition}
The charts $(V,\eta)$ constructed from $\fam{G}$-adapted charts $(U,\phi)$ give $\fam{X}/\fam{G}$ the structure of a smooth manifold.
\end{proposition}
\begin{proof}
Let $(V,\eta)$ and $(\tilde{V},\tilde{\eta})$ be two adapted charts for $\fam{X}/\fam{G}$, arising from the charts $(U,\phi)$ and $(\widetilde{U},\widetilde{\phi})$.  We first consider the case that both $U$ and $\tilde{U}$ are centered at the same point $p$.  
Writing the $\fam{G}$-adapted coordinates on each respectively as $(x,y,z)$ and $(\tilde{x},\tilde{y},\tilde{z})$ we recall that by the definition of adapted chart, two points of $U$ lie in the same $\fam{G}$ orbit iff their first two coordinates are identical, and same for $\tilde{U}$.  
The transition map $U\to \tilde{U}$ 
is given by some smooth map 
$F:I^N\to I^N$,$$(x,y,z)\mapsto F(x,y,z)=(F_1(x,y,z),F_2(x,y,z),F_3(x,y,z))=(\tilde{x},\tilde{y},\tilde{z})$$
For $F_i$ smooth maps defined on a neighborhood of the origin.
As both coordinate charts are $\fam{G}$-adapted, fixing $x,y$ and letting $z$ vary traces out points in the same $\fam{G}$ orbit, and so their $\tilde{U}$ representations have constant $\tilde{x},\tilde{y}$ and varying $\tilde{z}$.  That is, the coordinate functions $F_1$ and $F_2$ are independent of $z$ and so
$F(x,y,z)=(\hat{F}(x,y), F_3(x,y,z))$ 
for $\hat{F}:I^{k+\ell}\to I^{k+\ell}$ a smooth map in a neighborhood of $(0,0)$.  

Now we turn our attention to the transition map $V\to\tilde{V}$ given by $\tilde{\eta}\eta\inv$.  This takes a point $(x,y)$ to $\fam{G}.\phi\inv(x,y,0)\in V$ and then returns via $\tilde{\eta}$.  Writing this out

$$\tilde{\eta}\circ\eta\inv(x,y)=\tilde{\eta}\left(\fam{G}.\phi\inv(x,y,0)\right)=\pi_{12}\circ\tilde{\phi}\left(\phi\inv(x,y,0)\right)$$
$$=\pi_{12}\circ\left(\tilde\phi\circ\phi\inv\right)(x,y,0)=\pi_{12}\circ F(x,y,0)=\pi_{12}(\hat{F}(x,y),F_3(x,y))=\hat{F}(x,y)$$

Where here we have used $\tilde{\eta}$ applied to a $\fam{G}$-orbit is equal to $\pi_{12}\tilde{\phi}$ applied to a representative in $\tilde{U}$.  But we already know $\hat{F}$ is smooth, and so these charts are smoothly compatible!

We now have to consider the general case, where we have two charts $(V,\eta)$ and $(\tilde{V},\tilde{\eta})$ and $q\in V\cap \tilde{V}$.  Let $(U,\phi)$ and $(\tilde{U},\tilde{\phi})$ be corresponding charts for $\fam{X}$, and $p,\tilde{p}$ points with $\piO(p)=\piO(\tilde{p})=q$.  We can easily modify the charts so that they are centered at $p$ and $\tilde{p}$ respectively (Proposition \ref{Recentering_Charts}), and so we assume this is the case.  Since $p$ and $\tilde{p}$ are in the same $\fam{G}$-orbit, there is some $g\in\fam{G}$ such that $g.p=\tilde{p}$.

We can use this $g$ to produce a modified chart centered at $p$ which still induces $(\tilde{V},\tilde{\eta})$.  Recall that from $g\in \fam{G}$ we can produce a local section of $\fam{G}\to\Delta$, $s:W\to\fam{G}$ such that $s(W)\ni g$ (Proposition \ref{Local_Sections}).  Then following Proposition \ref{Vertically_Translated_Charts} we produce the chart $(\hat{U},\hat{\phi})=(\hat{s}\inv \tilde{U},\tilde{\phi}\circ\hat{s})$ which induces the same chart as $(\tilde{U},\tilde{\phi})$ on $\fam{X}/\fam{G}$.  We note that this new chart is centered at $p$ as 
$$\hat{\phi}(p)=\tilde{\phi}\circ\hat{s}(p)=\tilde{\phi}(s(\delta(p)).p)=\tilde{\phi}(g.p)=\tilde{\phi}(\tilde{p})=0$$
Where $s(\delta(p))=g$ as $g\in s(W)$ by design and as $g.p$ is defined, $\delta(g)=\delta(p)$ so $s(\delta(p))\in G_{\delta(g)}$ but as this is a section there can only be one such point, namely $g$.  Thus, we now have two $\fam{G}$-adapted charts centered at $p$, and so by the work above we know the associated transition map for $V\to\hat{V}$, given by $\hat{\eta}\eta\inv$ is smooth.  But $\hat{V}=\tilde{V}$ and $\hat{\eta}=\tilde{\eta}$ and so we are done!
  \end{proof}

\noindent
We will call this the \emph{induced} smooth structure on $\fam{X}/\fam{G}$.

\subsection{The Families $\fam{X}/\fam{G}\to\Delta$ and $\fam{X}\to\fam{X}/\fam{G}$}

\noindent
We are now in a position to show the main result of the quotient family theorem, that $\fam{X}/\fam{G}$ is an object in $\Fam_\Delta$.

\begin{proposition}
The map $\delta:\fam{X}\to\Delta$ induces a surjective smooth submersion $\overline{\delta}:\fam{X}/\fam{G}\to \Delta$.
\end{proposition}
\begin{proof}
From $\fam{X}$ we have the projection $\delta$ onto the base and $\piO$ onto the orbit space.  
Since each $\fam{G}$-orbit is contained in a single slice $M_\delta\subset\fam{X}$ the coordinate $\delta$ is constant on the fibers of $\piO$, and so by Proposition \ref{Pass_Through_Quotient}, $\delta$ descends to a unique smooth map $\overline{\delta}:\fam{X}/\fam{G}\to\Delta$

\begin{center}
\begin{tikzcd}
\fam{X}\arrow[d,"\piO"']\arrow[dr,"\delta"] &\\
\fam{X}/\fam{G}\arrow[r,dashed,"\overline{\delta}"']&\Delta
\end{tikzcd}
\end{center}

This is clearly surjective as $\delta$ is, and so it only remains to see $\bar{\delta}$ is a submersion.  
Let $p\in\fam{X}$ be arbitrary, and let $(U,\phi)$ be a $\fam{G}$-adapted coordinate chart centered at $p$.  Let $(V,\eta)$ be the corresponding coordinate chart for $\fam{X}/\fam{G}$ centered at $\fam{G}.p$.  
On $U$ the projection $\delta$ looks like the map $(x,y,z)\mapsto x$ and $\piO$ looks like $(x,y,z)\to (x,y)$.  Thus on $V$ the map $\overline{\delta}$ looks like $(x,y)\mapsto x$, which is clearly a submersion.
\end{proof}

\noindent
The existence of $\fam{G}$-adapted charts for $\fam{X}$ gives even more: $\fam{X}\to\fam{X}/\fam{G}$ is a family.

\begin{proposition}
With respect to the original smooth structure on $\fam{X}$ and the induced smooth structure on $\fam{X}/\fam{G}$, the orbit projection $\piO$ is a smooth surjective submersion.	
\end{proposition}
\begin{proof}
Let $p\in\fam{M}$ and $U$ be a $\fam{G}$-adapted chart centered at $p$, with $(V,\eta)$ the induced chart on $\fam{X}/\fam{G}$.  Then with respect to these coordinates, the map $\piO$ is expressed as $(x,y,z)\mapsto(x,y)$ which is clearly a smooth submersion.	The map $\piO$ is surjective by definition, so we are done.
\end{proof}

\section{Families of Geometries}
\label{sec:Fams_of_Geos}

A \emph{family of Klein geometries over 	$\Delta$} is given by a pair $(\fam{G},\fam{X})$ of groups $\fam{G}\to\Delta$ acting fiberwise-transitively on the spaces $\fam{X}\to\Delta$.  
Much as in the classical case, we will see that in making things precise there is both a Group-Space and Automorphism-Stabilizer perspective, and that these two perspectives are equivalent.

\subsection{Group-Space \& Automorphism-Stabilizer}

\begin{definition}[Group-Space]
A \emph{family of Klein geometries over 	$\Delta$} is given by a triple $(\fam{G},(\fam{X},\fam{x}))$ of a family of groups $\fam{G}\to\Delta$ acting fiberwise-transitively on a family of spaces $\fam{X}\to\Delta$ over the same base, equipped with a global section $\fam{x}\colon \Delta\to \fam{X}$ choosing a basepoint in each fiber.  
A morphism of geometries $\Phi\colon (\fam{G},(\fam{X},\fam{x}))\to(\fam{G}',(\fam{X}',\fam{x}'))$ is given by a family homomorphism $\phi_{\fam{Grp}}\colon\fam{G}\to\fam{G}'$ together with an equivariant map $\phi_{\fam{Sp}}\colon\fam{X}\to\fam{X}'$ such that $\phi_{\fam{Sp}}\circ\fam{x}=\fam{x}'$.
The category of such geometries is denoted $\cat{GrpSp}$.
\end{definition}

\begin{figure}
\centering\includegraphics[width=0.75\textwidth]{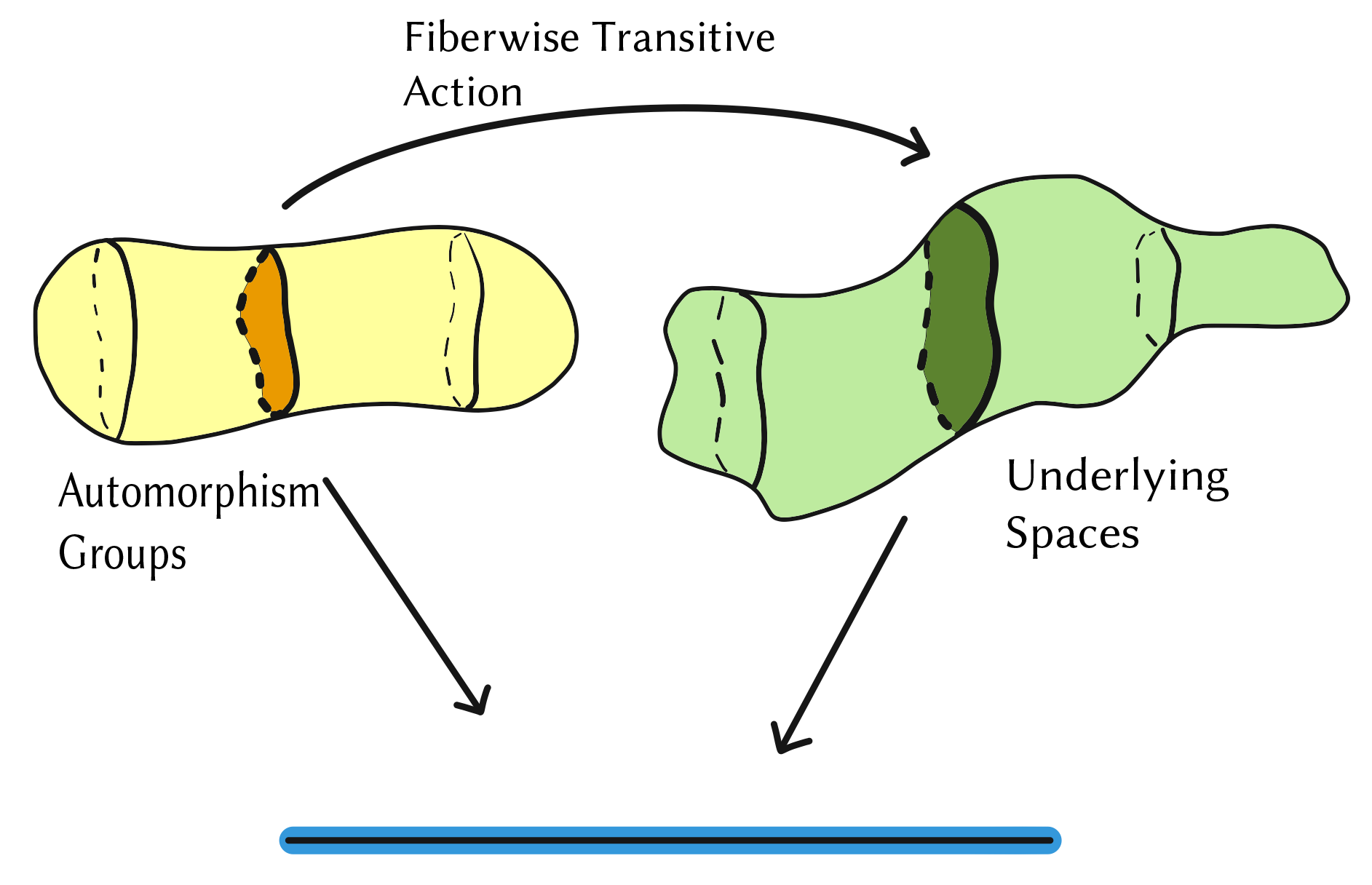}
\caption{A family of geometries from the Group-Space perspective.}	
\end{figure}

\noindent
 This generalizes the group-space viewpoint on Klein geometries.  Alternatively, we may wish to generalize the group-stabilizer perspective, which encodes homogeneous spaces purely group-theoretically.

\begin{definition}[Automorphism-Stabilizer]
A family of Klein geometries over $\Delta$ is given by a pair $(\fam{G,C})$ of a family of groups $\fam{G}\to\Delta$ and a closed subfamily $\fam{C}\subg\fam{G}$.	  
A morphism $\Phi:(\fam{H,K})\to (\fam{G,C})$ is a homomorphism of families $\Phi:\fam{H}\to\fam{G}$ with $\Phi(\fam{K})\subset\fam{C}$.
The category of these geometries is denoted $\cat{AutStb}$.
\end{definition}

\begin{figure}
\centering\includegraphics[width=0.5\textwidth]{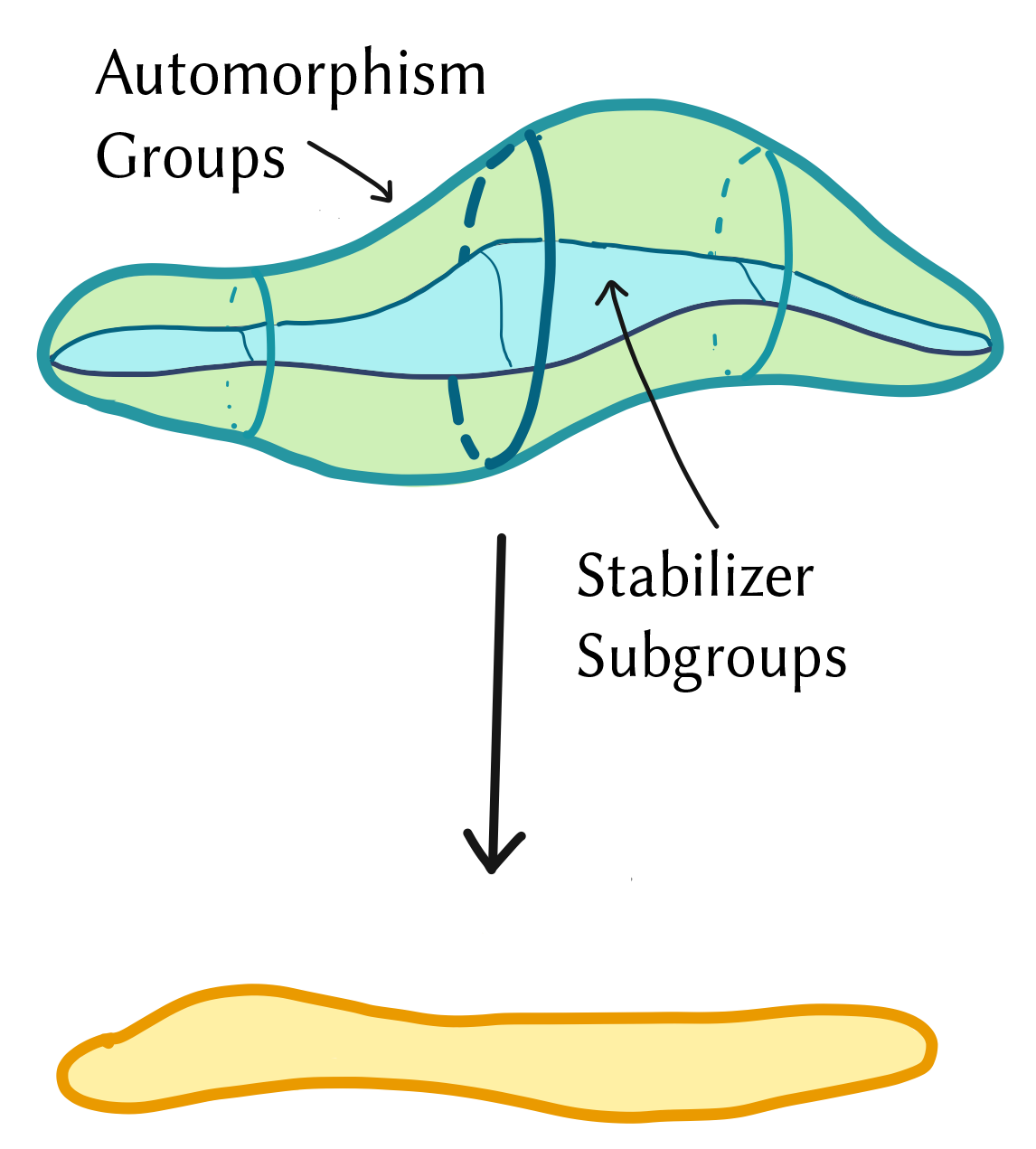}
\caption{A family of geometries from the Automorphism-Stabilizer perspective.}	
\end{figure}

\noindent
Many other definitions from the theory of Klein geometries have obvious analogs.  The \emph{kernel collection} of a family is the subset $\fam{ker}\subset\fam{G}$ of elements which act trivially on their respective members.  A family $(\fam{G},\fam{X})$ is \emph{effective} if its kernel is the trivial family $\fam{e}\subg \fam{G}$, and locally effective if $\fam{ker}$ is discrete in each fiber.  In the group-stabilizer framework, the kernel of $(\fam{G},\fam{K})$ is the \emph{core} of $\fam{K}$ in $\fam{G}$: fiberwise equal to $\mathsf{core}_{\fam{G}_\delta}(\fam{K}_\delta)$.

\begin{definition}
An \emph{embedding of a family of geometries} is given by a monomorphism $(\fam{H},\fam{Y})\stackrel{\iota}{\inject}(\fam{G},\fam{X})$.  If in addition $\iota(\fam{H})$ is a subfamily of $\fam{G}$ and $\iota(\fam{Y})$ is open in $\fam{X}$, it is said to be a family of \emph{subgeometries} of $(\fam{G},\fam{X})$. 	
\end{definition}

\begin{definition}
 A \emph{fibration} of $(\fam{G},\fam{X})$ over $(\fam{H},\fam{Y})$ is given by an epimorphism $(\fam{G,X})\surject(\fam{H,Y})$.
 \end{definition}

\subsection{Equivalence of Perspectives}

The theory of Klein geometries begins begins from these definitions with the identification of the important natural transformations relating them.  
The equivalence of categories between the group-stabilizer and (pointed) group-space viewpoints, together with the forgetful functor from pointed geometries to their non-pointed counterparts allows us to freely pass between these notions at will.
Techniques for producing pullbacks in the smooth category and the quotient family theorem give tools to build up this theory of families of geometries.

\begin{observation}
Deleting the basepoint section $(\fam{G,(X,x)})\mapsto (\fam{G,X})$ is a forgetful functor from the category of pointed to non-pointed families of geometries.	
\end{observation}
\begin{proof}
The action on objects is simply to forget the global section of points.  This has no effect on the morphisms, and automatically determines a functor.
\end{proof}

We turn now to showing the equivalence of the $\cat{GrpSp}$ and $\cat{AutStb}$ perspectives. One direction, constructing the family of spaces for a group-space geometry out of a family of automorphism groups and corresponding family of stabilizer subgroups, is a direct application of the Quotient Family Theorem.

\begin{lemma}
The map $\cat{F}:\cat{AutStb}\to\cat{GrpSp}$ sending a group-stabilizer geometry $(\fam{G,K})$ to the group-space geometry $\fam{(G,(G/K,K))}$ defines a functor.
\end{lemma}
\begin{proof}

As $\fam{K}$ is a closed subfamily of $\fam{G}$, \cref{prop:Subgroup_Translation} shows that left translation by $\fam{K}$ on $\fam{G}$ is a free and proper action.  
Thus by the quotient family theorem, $\fam{G}/\fam{K}$ is a smooth family over $\Delta$.  
The action of $\fam{G}$ on $\fam{G}/\fam{K}$ is just the usual action of $\fam{G}$ on itself followed by the quotient map, which is fiberwise transitive as $\fam{G}_\delta$ acts transitively on itself.  
The natural inclusion $\fam{K}\inject \fam{G}/\fam{K}$ (equivalently, the projection of the identity section $\fam{e}$) provides the section of points.
Given a morphism $\Phi:(\fam{K,H})\to (\fam{C,G})$ we define $\cat{F}(\Phi)=(\Phi, \bar{\Phi})$ where $\bar{\Phi}(g\fam{C}_\delta)=\Phi(g)\fam{K}_\delta$.
This is $\Phi$-equivariant and well-defined as $\Phi(\fam{C})\subset\fam{K}$.  

\end{proof}
The connection between the group-space and group-stabilizer viewpoints is more subtle in the theory of families, as it was noted in Section \ref{sec:Family_Actions} that the collection of stabilizers of an arbitrary family action need not always form a subfamily.
Thus, creating a geometry $(\fam{G},\fam{stab}_\fam{G}(\fam{x}))$ from a geometry $(\fam{G},(\fam{X},\fam{x}))$ is delicate, and potentially problematic\footnote{In fact, working with weak families in the continuous category, one cannot always do this.}
The proposition below shows that these concerns only materialize for non-fiberwise transitive actions, so families of geometries always have families of stabilizing subgroups.

\begin{proposition}
\label{prop:Stabilizer_Subfamily}
Let $(\fam{G,X})$ be a family of geometries in the smooth category.  
Then the point stabilizers $\mathsf{stab}_{\fam{G}_{\pi(x)}}(x)$ form a family over $\fam{X}$.	
\label{prop:Stabilizer_Families_of_Family}
\end{proposition}
\begin{proof}
The action of $\fam{G}$ on $\fam{X}$ is given by the map $\fam{G}\timesd\fam{X}\to\fam{X}$, $(g,x)\mapsto g.x$.  We will consider the associated map $\alpha\colon\fam{G}\timesd\fam{X}\to\fam{X}\timesd\fam{X}$ given by $(g,x)\mapsto(x,gx)$.

Assume temporarily that $\alpha$ gives $\fam{G}\timesd\fam{X}$ the structure of a family over $\fam{X}\timesd \fam{X}$.  
Pulling this family back via the diagonal map $\delta\colon\fam{X}\to\fam{X}\timesd\fam{X}$ gives a family $\fam{S}\to \fam{X}$ consisting of the elements $\fam{S}=\set{((g,x),y)\mid (x,gx)=(y,y)}$ that is, the fiber above $x\in\fam{X}$ is the stabilizer subgroup $\mathsf{stab}_{G_{\pi(x)}}(x)$.  
Thus it suffices to show that $\alpha$ is a family projection.

Since $\alpha$ is a smooth map of families by \cref{prop:Fiberwise_Family=Family}, this follows if $\alpha$ is a map of families fiberwise, or equivalently for any fixed smooth geometry $(G,X)$ the map $G\times X\to X\times X$ given by $(g,x)\mapsto (x,gx)$ is a submersion.  
Fix a particular $(g,x)\in G\times X$.
As the tangent space to the image decomposes as a product $T_{(x,gx)}X\times X=T_xX\times T_{gx}X$, it is enough to show that $d\alpha_{(g,x)}$ is onto each factor.  

Fixing $g$, we consider the restricted map $\alpha(g,\cdot)\colon \set{g}\times X\to X\times X$ sends $x$ to $(gx,x)$, and so the derivative is the graph of $L_g$ (left multiplication by $g$) in $T_{x}X\times T_{gx}X$. 
Fixing $x$, we consider the map $\alpha(\cdot,x)\colon G\times\set{x}\to X\times X$, which is constant on the first factor and is the orbit map $G\to X$, $g\mapsto g.x$ on the second.  
This map factors through the projection onto the coset space $G\to G/\mathsf{stab}(x)$ to a diffeomorphism $G/\mathsf{stab}(x)\to X$ as the action is transitive.  
But the projection onto the coset space is a submersion by the quotient manifold theorem, so $\alpha(\cdot, x)$ is onto $\set{0}\times T_{gx}X$.
Noting that  $\set{(v,L_g(v))\mid v\in T_xX}$ and $\set{(0,w)\mid w\in T_xX}$ sum to all of $T_{(x,gx)}X\times X$ finishes the argument.

\end{proof}

\begin{corollary}
The stabilizer family $\fam{stab}_\fam{G}\to\fam{X}$ with fiber $\fam{stab}_{\fam{G}}(x)$ above each $x\in X$ pulls back along any section $\fam{x}\colon\Delta\to\fam{X}$ to give a smooth family of point stabilizers $\fam{stab}_\fam{G}(\fam{x})\to\Delta$.
Thus, every pair $(\fam{G},(\fam{X,x}))$	 is canonically associated to a pair $(\fam{G},\fam{stab}_\fam{G}(\fam{x}))$.
\end{corollary}

\noindent
This suggests the definition of a functor from group-space to group-stabilizer geometries in the smooth category.

\begin{lemma}
In the smooth category, the map $\Psi\colon \cat{GrpSp}\to\cat{AutStb}$ sending a geometry $(\fam{G,(X,x)})$ to $(\fam{G},\fam{stab}_\fam{G}(\fam{x}))$ defines a functor.	
\end{lemma}
\begin{proof}
By the previous proposition, the entire collection of point stabilizers forms a family over $\fam{X}$.  Pulling this back along the section $\fam{x}\colon\Delta\to\fam{X}$ gives a family $\fam{x}^\star\fam{S}\to\Delta$ for which the projection into $\fam{G}\to\Delta$ is an embedding by \cref{obs:Pullback_Eqns}.  Thus $(\fam{G},\fam{stab_G(x)})$ is a geometry of the group-stabilizer variety.
Recalling that a morphism $\Phi:\fam{(G,(X,x))}\to\fam{(H,(Y,y))}$ consists of a group homomorphism $\Phi_{\cat{Grp}}$ and an equivariant map $\Phi_{\cat{Sp}}$ 
between the spaces, the image $\Psi(\Phi)=\Phi_{\cat{Grp}}$ is simply the group homomorphism, which is well-defined as $\Phi_\cat{Sp}\circ\fam{x}=\fam{y}$ together 
with equivariance implies that $\Phi_\cat{Grp}(\fam{stab(x)})\subset\fam{stab}(\fam{y})$.   
\end{proof}

\noindent
To finish our understanding of the family-theoretic analog of (2), we show that in the smooth category these pair up to form an equivalence of categories.
This proof is identical in structure to Proposition CITE, we have merely replaced the relevant categories of geometries with the categories of families of geometries.

\begin{proposition}
In smooth categories $\cat{Diff}$-$\Fam_\Delta$, the functors $\cat{F},\Psi$ above define an equivalence of categories $\cat{GrpSp}\cong\cat{AutStb}$.	
\end{proposition}
\begin{proof}
The composition $\Psi \cat{F}$ is the identity on $\cat{AutStb}$, and the composition $F\Psi$ takes the geometry 
$(\fam{G},(\fam{X},\fam{x}))$ to 
$(\fam{G},(\fam{G}/\fam{stab}_\fam{G}(\fam{x}),\fam{stab}_\fam{G}(\fam{x})))$.  

The collection of maps $\eta|_{(\fam{G},\fam{X})}\colon \left(\fam{G},(\fam{X,x})\right) \to \left( \fam{G}, (\fam{G/Stab_G(x),Stab_G(x)})\right) $ given by 
$\eta=(\mathsf{id}_\fam{G},\xi_{(\fam{G},\fam{X})})$ where $\xi_{(\fam{G},\fam{X})}(p)=g\mathsf{Stab}_\fam{G}(x)$ if $\mathsf{Stab}_\fam{G}(p)=g\mathsf{Stab}_\fam{G}(x)g\inv$ forms a natural transformation from $\mathsf{id}_{\cat{GrpSp}}$ to $\cat{F}\Psi$.   

 To see this it suffices to check that $\bar{\Phi_\cat{Grp}}\circ\xi_{(\fam{G},\fam{X})}=\xi_{(\fam{H},\fam{Y})}\circ\Phi_\cat{Sp}$.  
Let $p\in\fam{X}_\delta$ and $g\in\fam{G}_\delta$ be such that $g.x(\delta)=p$.  Then $\xi_{(\fam{G},\fam{X})}(p)=g \mathsf{Stab}_\fam{G}(x(\delta))$ and $\bar{\Phi_\cat{Grp}}(g\mathsf{Stab}_\fam{G}(x(\delta)))=\Phi_\cat{Grp}(g)\mathsf{Stab}_\fam{H}(\fam{y}(\delta)))$.  Computing the other way around we find $\Phi_\cat{Sp}(p)=\Phi_\cat{Sp}(g.\fam{x}_\delta)=\Phi_\cat{Grp}(g)\Phi_\cat{Sp}(\fam{x}_\delta)=\Phi_\cat{Grp}(g)\fam{y}_\delta$ and $\xi_{(\fam{H},\fam{Y})}(\Phi_\cat{Grp}(g)\fam{y}_\delta))=\Phi_\cat{Grp}(g)\mathsf{Stab}_\fam{H}(\fam{y}_\delta)$.

\begin{center}
\begin{tikzcd}
\left(\fam{G},(\fam{X,x})\right) 
\arrow{d}[swap]{(\Phi_{\cat{Grp}},\Phi_{\cat{Sp}})}
\arrow{rr}{(\mathsf{id}_\fam{G},\xi_{(\fam{G},\fam{X})})}
&& 
\left( \fam{G}, (\fam{G/Stab_G(x),Stab_G(x)})\right) 
\arrow{d}{(\Phi_\cat{Grp},\bar{\Phi_\cat{Grp}})} 
\\
\left(\fam{H},(\fam{Y,y})\right)
\arrow{rr}[swap]{(\mathsf{id}_\fam{H},\xi_{(\fam{H},\fam{Y})})}
&& \left( \fam{H}, (\fam{H/Stab_H(y),Stab_H(y)})\right) 
\end{tikzcd}
\end{center}
	
\end{proof}

\subsection{Hyperbolic To Euclidean Transition}

As a first example of these definitions, we formalize the familiar transition from hyperbolic to spherical geometry through Euclidean, not as a conjugacy limit but as a family.
We begin by constructing the family of spaces.

\begin{proposition}
The variety $\fam{V}=V(tx^2+ty^2+z^2-1)\subset\R^4$ equipped with the restricted projection onto the $t$-coordinate is a family of spaces over $\R$.
\end{proposition}
\begin{proof}
This is just the higher dimensional analog of Example \ref{ex:First_Example}.
$\fam{V}$ is a smooth subvariety, and hence a smooth submanifold of $\R^4$.
The normal vector to $\fam{V}$ in $\R^4$ is given by $\nabla(tx^2+ty^2+z^2-1)=(2tx, 2ty,2z, 1)$ is nowhere parallel to the $t$ axis, so the tangent spaces to $\fam{V}$	are transverse to the foliation $\R^3\times\{t\}$, and the restricted projection is a submersion.
\end{proof}

\noindent
The family $\fam{V}$ has members transitioning from hyperboloids of 2 sheets for $t<0$ to ellipsoids for $t>0$ through a pair of parallel planes at $t=0$.
Each of these slices admits a free $\Z_2$ action sending a point to its antipode, and so $\fam{V}$ admits a free and proper action of $\fam{Z}=\Z_2\times\R\to\R$.
By the quotient family theorem, the quotient $\fam{X}=\fam{V}/\fam{Z}$ is a smooth family of subsets of $\RP^2$ over $\R$.

\begin{figure}
\centering\includegraphics[width=0.75\textwidth]{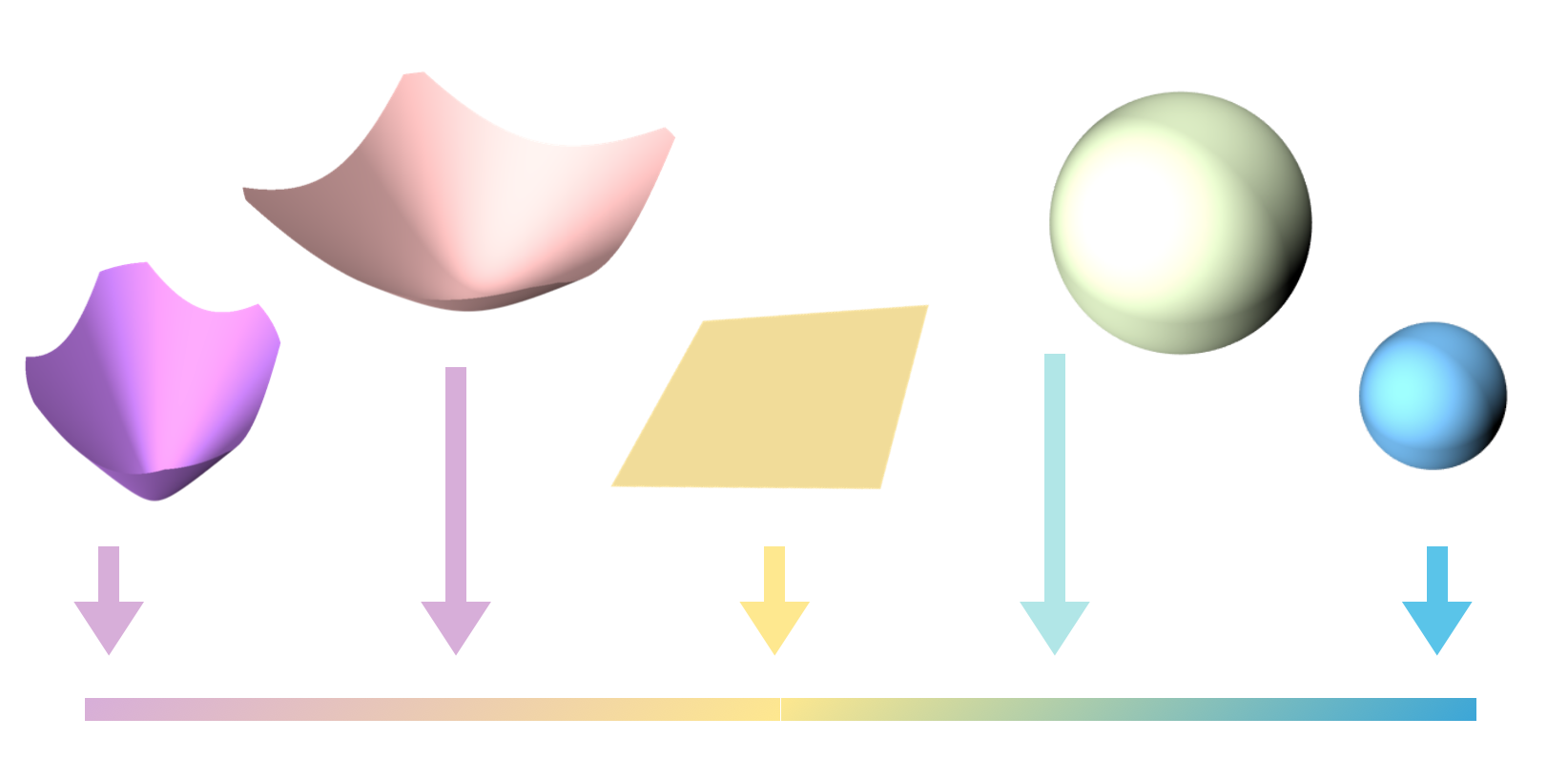}
\caption{Family of spaces for the $\Hyp^n\to\S^n$ transition.}	
\end{figure}

\noindent
Now we turn to the family of groups.
For each $t\neq 0$, the surface $\fam{V}_t$ is a quadratic hypersurface in $\R^3\times\{t\}$, and the group of linear transformations preserving it forms the orthogonal group $\O(\diag(t,t,1))$.

\begin{proposition}
Let $\fam{G}\subset\GL(3,\R)\times\R$ be	 the collection of groups 
$$\fam{G}=\bigcup_{t\in\R_-}\SO(\diag(t,t,1))\times\{t\}\cup \Euc(2)\times\{0\}\cup\bigcup_{t\in\R_+}\SO(\diag(t,t,1))\times\{t\}$$
Then $\fam{G}$ is a family of groups equipped with the restricted projection from $\GL(3;\R)\times\R\to\R$.
\end{proposition}

\begin{proof}
Applying the contragredient automorphism $A\mapsto A^{-T}$ to each member of $\GL(3;\R)\times\R\to\R$ gives a smooth automorphism of the family, taking $\fam{G}$ to $\fam{G}^{-T}=\bigcup_{t\in\R}G_t^{-T}\times\{t\}$.
We show that this collection forms a family directly; and then applying once more then contragredient automorphism gives the same result for the original $\fam{G}$.
The reason for this seemingly strange approach is just notational:
$G_t^{-T}=\SO(t,t,1)^{-T}=\SO(1,1,t)$ for $t\neq 0$, and when $t=0$ the group $\SO(1,1,0)$ is precisely the contragredient Euclidean group $\Euc(2)^{-T}$, and so the entire family $\fam{G}^{-T}$ can be succinctly described as $\fam{G}_t^{-T}=\SO(\diag(1,1,t))$ regardless of the value of $t$ (this does not hold in the original case, as $\SO(\diag(0,0,1))$ strictly contains the Euclidean group).
Now $\fam{G}^{-T}$, is a nonsingular subvariety of $\M(3;\R)\times\R$\footnote{$\fam{G}^{-T}=\bigcup_{t\in\R}\SO(\diag(1,1,t))\times\{t\}$ is cut out by the equations $(X^T\diag(1,1,t)X,t)=(\diag(1,1,t),t),\det X=1$ for $X=(x_{ij})_{1\leq i,j,\leq 3}$.} and thus a closed smooth submanifold.

The Lie algebras $\so(\diag(1,1,1,t))$ form a continuous family of Lie algebras in $\M(3;\R)\times\R$ as follows immediately from computation.
And while the number of components of $\SO(\diag(1,1,t))$ changes along the transition (from 2 when $t<0$ to 1 when $t>0$), each component always contains one of the matrices of the form $\diag(\pm 1, \pm 1, \pm 1)$, so by Proposition \ref{prop:Exponential_Check}, $\fam{SO}(\diag(1,1,t))$ is a family of groups.
Thus so is its contragredient image, $\fam{G}$.
\end{proof}

\noindent
The action of $G_t$ on $\fam{X}_t$ is transitive, and so $(\fam{G},\fam{X})$ is a family of Klein geometries.  
To get a family of pointed geometries, it suffices to choose any $\gamma\colon\R\to\RP^2$ such that $\gamma(t)\in\fam{X}_t$; for instance $\gamma(t)=[0:0:1]$.


\chapter{Geometries Over Algebras}
\label{chp:Geos_over_Algs}
\index{Real Algebras!Geometries}
\index{Geometries!Real Algebras}

Hyperbolic geometry arises as subgeometry of $\RP^n$ through the familiar Klein model.  Generalizing this picture to $\CP^n$ gives produces the geometry of complex hyperbolic space, and the further generalization of Chapter \ref{chap:HC_and_HRR}, extended this to yet two more geometries, analogs of hyperbolic space over $\R_\ep$ and $\R\oplus\R$.
This chapter is continue in this direction, and generalize this to both other choices of algebras, and other familiar geometries.
In particular, over an arbitrary finite dimensional real (associative) algebra $A$, we will define an associated projective geometry $A\mathsf{P}^n$, as well as analogs of the classical unitary and orthogonal together with their corresponding geometries.

\section{Real Algebras}
\label{sec:Real_Algs}
\index{Real Algebras}

A commutative algebra over $\R$ is a real vector space $A$ equipped with a bilinear multiplication $\mu\colon A\times A\to A$. 
An algebra $A$ is commutative if $\mu(a,b)=\mu(b,a)$ for all $a,b$; 
$A$ topological if $\mu$ is continuous, and of category $\cat{C}$ if $\mu\in\Hom_\cat{C}(A^2,A)$.  
An element $a\in A$ is a left zero divisor if $\mu(a,\cdot)$ has a nontrivial kernel, and a left unit if $\mu(a,\cdot)$ is an isomorphism, analogously for right zero divisors and units.
As convention, when not specified we will always mean \emph{left} zero divisor and \emph{left} units.
Let $A^\times$ denote the set of units, and $A_Z$ the set of zero divisors. 
If $A$ is finite dimensional then $A=A^\times\sqcup A_Z$.   

\begin{lemma}
The zero divisors $A_Z\subset A$ of a topological algebra form a closed subset.	
\end{lemma}
\begin{proof}
Let $A$ be an $N$-dimensional algebra over $\R$ and $\{z_i\}\subset A_Z$ be a sequence of zero divisors, converging to $z\in A$.  For each $z_i$ there is some $w_i\in A\setminus\set{0}$ with $z_iw_i=0$, and in fact all scalar multiples $\R^\times w_i=[w_i]$ satisfy this as well.  The sequence $\{[w_i]\}$ subconverges in $(A\smallsetminus \set{0})/\R^\times\cong\RP^{N-1}$ to $[w]$ by compactness, so choose representatives $w_i\to w$ (say, on the unit sphere).  Then as $w_iz_i=0$, continuity of multiplication forces $wz=0$ so $z\in A_Z$.  
\end{proof}

\begin{corollary}
As every element of a finite dimensional algebra is either a zero divisor or unit, the units $A^\times\subset A$ are an open subset.
\end{corollary}

\noindent
If $A$ is a smooth algebra the group of units $A^\times$ is an open subset, thus a submanifold, and so $A^\times$ is a Lie group.  Furthermore any closed subgroup of $A^\times$ is a Lie subgroup.
An \emph{involution} on an algebra $A$ is an element $\sigma\in\mathsf{End}(A)$ of order two 
The action of $\sigma$ on the underlying vector space satisfies $\sigma^2-1=0$ and decomposes $A$ as a direct sum of the $+1$ and $-1$ eigenspaces, $A=\mathsf{Fix}(\sigma)\oplus\mathsf{Neg}(\sigma)$.  
For any choice of $j\in\mathsf{Fix}(\sigma)$ the involution provides a map $\phi_j\colon A\to \mathsf{Fix}(\sigma)$ given by $\phi_j(x)=\sigma(x)jx$.  
When $j=1$ this map is multiplicative, thus a group homomorphism called the \emph{norm}, $x\mapsto\sigma(x)x$.  
The preimage of $\set{1}$ is the 1-dimensional \emph{unitary group}, $\U(A):=\set{\alpha\in A^\times\mid \sigma(\alpha)\alpha=1}$.  

Given a real algebra $A$ the matrix algebras $\M(n;A)$ are given by imposing matrix multiplication on the spaces $A^{n^2}$.  As this multiplication is built directly out of that of $A$, the matrix algebras are $\cat{C}$-algebras iff $A$ is, respectively.  
An involution $\sigma\colon A\to A$ extends via component-wise application to $\M(n;A)$ and induces an involution analogous to the conjugate transpose, $X^\dagger=\sigma(X)^T$.  The decomposition of $\M(n,A)$ corresponding to $\dagger$ determines the \emph{Hermitian}  $\mathsf{Fix}(\dagger)=\Herm(n;A,\sigma)$ and \emph{skew-Hermitian} $\mathsf{Neg}(\dagger)=\SkHerm(n;A,\sigma)$ matrices.
For commutative algebras $A$, the usual formula for the determinant provides a map $\det\colon \M(n,A)\to A$.  Cramer's shows $B\in\M(n;A)$ is invertible iff $\det(B)$ is.  As $\det$ is polynomial in the matrix entries, $\det\in\Hom_\cat{C}(\M(n,A),A)$ and inversion (given by the matrix of cofactors) is a $\cat{C}$-morphism on the complement of $\det\inv\set{A_Z}$.

Thus $\GL(n,A)=\det\inv\set{A^\times}$, which is an
 open subset (thus submanifold) of the $\M(n,A)$.
The group operations of multiplication and inversion are $\cat{C}$-morphisms on $\GL(n,A)$, providing the structure of a $\cat{C}$-group.  
The determinant provides a group homomorphism $\det\colon\GL(n,A)\to A^\times$ and preimages of subgroups give important subgroups of $\GL(n;A)$.  As our interest is particularly in the smooth category, the following provides a method of producing Lie subgroups.

\begin{proposition}
Let $A$ be smooth and commutative, then $\det\colon\GL(n;A)\to A^\times$ is a submersion.
\label{prop:det_is_submersion}
\end{proposition}
\begin{proof}
Let $B\in\GL(n;A)$, then for each $X\in\M(n;A)$ the path $B_t=(I+tX)B$ passes through $B$ and $\dd{t}|_{t=0}\det(B_t)=\tr(X)\det(B)$ so for any $\alpha\in A$ the choice $X_\alpha=\frac{\alpha}{n\det(B)}I$ shows the derivative surjects onto $A=T_{\det B}A^\star$.
\end{proof}

\begin{corollary}
If $A$ is a smooth commutative algebra and $G\subg A^\times$ a closed subgroup, then $\det\inv\set{G}$ is a Lie subgroup of $\GL(n;A)$.	In particular the closed subgroup $\set{1}\subg A^\times$ corresponds to the special linear group $\SL(n;A)=\det\inv\set{1}$.
\end{corollary}

\section{Projective Geometries}
\label{sec:Proj_Geos}
\index{Real Algebras!Projective Geometries}
\index{Geometries!Projective Geometry over Algebras}

Classically, projective geometry is given by the projectivization of the linear action of $\GL(n,\F)$ on $\F^n$.
Taking the group-space viewpoint, this is the action of $\GL(n;\F)$ on the projective space $\mathbb{F}\mathsf{P}^{n-1}=(\F^n\smallsetminus 0)/\F^\times$.
Taking the automorphism-stabilizer viewpoint, the geometry of projective space corresponds to the pair $(\GL(n,\F),\mathsf{stab})$ with $\mathsf{stab}$
the stabilizer of a projective point in $\P(\F^n)$ which realizes projective space as the quotient $\P(\F^n)=\GL(n,\F)/\mathsf{stab}$.

The geometry corresponding to $(\GL(n,F),\mathsf{Stab}[p])$ is independent of the choice of point $p\neq 0$ for projective geometry over a field $\F$, but this does not remain true for a general algebra $A$.
We will say that a point $p\in A^n$ is \emph{good} if the point stabilizer is of minimal dimension, and \emph{bad} otherwise.
One way to choose good points is as follows.
For a point $p\in A^n$ let $I_p\subg A$ be the ideal generated by its coordinates, $I_p=\langle p_1,\ldots p_n\rangle$.  
Note that for any $X\in\GL(n;A)$ the ideals $I_p$ and $I_{Xp}$ are identical and so this is an invariant of $\GL(n;A)$ orbits. 
Conversely if $I_p=I_q$ then each $q_i$ is a $A$-linear combination of the $p_i$ so $q=Xp$, so in fact the ideal $I_p$ determines the orbit.
 Generically, $I_p=A$ and strictly smaller ideals appear only when no coordinate (and no linear combination of the coordinates) is a unit.
 Such points are \emph{bad}, the generic case are the \emph{good points}.
 
 We may also take the group-space perspective, and try to define an analog of projective space over an algebra directly.
 Here, the bad points are the analog of $\vec{0}\in\F^n$, points on which the action of the units $A^\times$ is not free.
As the analogs of zero, we denote this collection by $Z(A^n)$.  The points of $A^n\setminus Z(A^n)$ constitute a single $\GL(n,\A)$ orbit, and so have isomorphic point stabilizers. 

\begin{definition}
The projective space $\mathsf{AP}^n$ is the quotient of $A^n\smallsetminus Z(A^n)$ by the left action $a.(v_i)=(av_i)$ of $A^\times$.
\end{definition}

 \begin{definition}
 Let $A$ be a finite dimensional commutative algebra over $\R$, and $n\in\N$.  Then $\mathsf{St}(n;A)$ is the stabilizer of $(0,\cdots, 0,1)$ under the linear action of $\GL(n,A)$ on $A^n$.
 $$\mathsf{St}(n;A)=\left\{ \pmat{X &\vec{0}\\\vec{v}&\alpha}
\mid \alpha\in A^\star, v\in A^{n-1}, X\in\GL(n-1;A)\right\}.$$
 \end{definition}

\noindent
We denote the intersection $\mathsf{St}(n;A)\cap\SL(n;\A)=\mathsf{SSt}(n;A)$.
Note that $(0,\cdots, 0,1)\in A^n\setminus Z(A^n)$ for any algebra $A$, and so we may use $\St(n;A)$ to define projective geometry generally.

\begin{definition}
The $(n-1)$ dimensional projective geometry over $A$ is given by the pair $(G,K)=(\GL(n;A),\St(n;A))$,
The effective version of this geometry is given by projectivization, $(\P\St(n;A), \PGL(n;A))$ and 
another convenient incarnation is $(\SL(n;A),\mathsf{SSt}(n;A))$ when $A$ is commutative.
The projective space $\mathsf{AP}^{n-1}=\P(A^n)$ is defined as the coset space $\GL(n;A)/\St(n;A)$.
\end{definition}

\noindent
Alternatively, from the group-space perspective, we have the following equivalent definition.

\begin{definition}
The $n-1$ dimensional projective geometry over $A$ has domain $\mathsf{AP}^{n-1}=(A^n\smallsetminus Z(A^n))/\sim$ for 
$\vec{v}\sim\vec{w}$ if there is an $a\in A^\times$ such that $a\vec{v}=\vec{w}$.  The (non-effective) automorphism group is $\GL(n;A)$.
\end{definition} 

\noindent
To see that smooth algebras define \emph{smooth} projective geometries, we need to show that $\mathsf{AP}^{n-1}$ is a smooth manifold, or equivalently that $\St(n;A)$ is a Lie subgroup of $\GL(n,A)$.  
This second fact is immediate from the closed subgroup theorem as $\St(n;A)$ is the intersection of a linear subspace of $\M(n;A)$ with $\GL(n;A)$; however we give an explicit argument which will be used in the generalization to families.

\begin{proposition}
The map $\GL(n;A)\to A^{n-1}$ projecting onto the first $n-1$ entries of the last column is a submersion.	
\end{proposition}
\begin{proof}
Let $\pi\colon\GL(n;A)\to A^{n-1}$ be the projection map $(X_{ij})\mapsto (X_{1,n},\ldots, X_{n-1,n})$.  Then for any $B\in\GL(n;A)$ and $v\in A^{n-1}\cong T_{\pi(B)}A^{n-1}$ the path $B_t=B+t\smat{O&\vec{v}\\\vec{0}&0}$ has $\dd{t}|_{t=0}\pi(B_t)=v$ so $(D\pi)_B$ is surjective.
\end{proof}

\section{Unitary Geometries}
\label{sec:Unitary_Geos}
\index{Unitary!Groups over Algebras}
\index{Geometries!Unitary Geometries over Algebras}
\index{Real Algebras!Unitary Groups}

Fix an algebra with involution $(A,\sigma)$ and a nondegenerate $J\in\Herm(n;A,\sigma)$.  
A matrix $X$ is said to \emph{preserve} $J$ if $X^\dagger JX=J$. 
 The map $\Phi_J\colon\M(n;A)\to\Herm(n;A,\sigma)$ given by $X\mapsto X^\dagger JX$ defines the \emph{generalized unitary group} for $J$.

\begin{definition}
The generalized unitary group $\U(J,A,\sigma)=\Phi_J\inv\set{J}$ consists of the matrices preserving $J$: $\U(J;A,\sigma)=\set{X\mid X^\dagger JX=J}$.
\end{definition}

\noindent
The map $\Phi_J$ is a $\cat{C}$-morphism as it is built out of algebra operations and the involution.  
Thus in particular $\U(J;A,\sigma)$ is a closed subgroup of $\GL(n;A)$.  
In the case that $A$ is a smooth algebra, this is enough to conclude the unitary groups are Lie groups.  
However the following direct argument will prove useful later on.

\begin{lemma}
The map $\Phi_J\colon \GL(n;A)\to\Herm(n;A,\sigma)$ is a submersion when $A$ is a smooth algebra.	
\label{prop:Unitary_Submersion}
\end{lemma}
\begin{proof}
Let $B\in\U(J;A,\sigma)$, then for any $X\in\M(n,A)$ we may construct the path $B_t=B+tX$ which remains in $\GL(n,A)$ for small $t$.  
Computing the derivative we see $\dd{t}|_{t=0}\Phi_J(B_t)=X^\dagger JB+B^\dagger JX$, and so $\Phi_J$ is a submersion if $X\mapsto X^\dagger JB+B^\dagger JX$ surjects onto $T_{\Phi_J(B)}\Herm(n;A,\sigma)=\Herm(n;A,\sigma)$.  
This map is $\R$-linear and so we proceed by dimension count, noting $\dim\mathrm{image}\;\Phi_J=\dim\M(n,A)-\dim\ker\Phi_J$.  
The kernel of $\Phi_J$ is given by $\ker\Phi_J=\set{X\mid X^\dagger JB=-B^\dagger JX}$, which as $B,J$ are invertible can be expressed $\ker\Phi_J=(B^\dagger J)\inv\SkHerm(n;A,\sigma)$.  
Thus $\dim\ker\Phi_J$ is the dimension of the space of skew-Hermitian matrices, so the dimension count above shows $\dim\mathrm{image}\;\Phi_J$ to be the same as the dimension of the space of Hermitian matrices (the complementary subspace to $\SkHerm$ in $\M(n,A)$).  
But $\Herm(n;A,\sigma)$ is the codomain so $(D\Phi_J)_B$ is surjective, and $\Phi_J$ is a submersion.
\end{proof}

\noindent
Taking the determinant of the equation $\Phi_J(X)=J$ gives $\det(X^\dagger)\det(X)=1$ as $J$ is nondegenerate, and $\det(X^\dagger)=\sigma(\det(X))$ so $\det X\in\U(A,\sigma)$.  Thus the determinant restricts to a homomorphism $\det\colon\U(J;A,\sigma)\to\U(A,\sigma)$.

\begin{lemma}
The determinant $\det\colon\U(J;A,\sigma)\to\U(A,\sigma)$ is a submersion when $A$ is a smooth commutative algebra.	
\label{prop:Unitary_Determinant}
\end{lemma}
\begin{proof}
The determinant is a group homomorphism $\U(J;A)\to\U(A)$ defining the closed subgroup (hence Lie subgroup, and manifold $\SU(J;A)$).  Together these three form a short exact sequence
$$1\to \SU(J,A)\to \U(J;A)\to\U(A)\to 1$$
so topologically $\U(J;A)$ is a product $\SU(J;A)\times\U(A)$ and in these coordinates the determinant is the projection map, which is a smooth submersion.
\end{proof}

\begin{corollary}
Preimages of closed subgroups of $\U(A,\sigma)$ give Lie subgroups of $\U(J;A,\sigma)$.  In particular, $\det|_{\U(J;A,\sigma)}\inv\set{1}=\SU(J;A,\sigma)$ is a Lie subgroup.
\end{corollary}

\noindent
This generalized notion of unitary group encompasses both the classical orthogonal and unitary groups, together with many new examples.

\begin{example}
Let $A=\C$ and choose the trivial involution $\sigma=\id_\C$.  Then the unitary groups corresponding to $J=\diag(I_p,-I_q)$ are the classical orthogonal groups, $\U(J,\C,\id)=\O(p,q;\C)$.  If instead $\sigma(x+iy)=x-iy$ is complex conjugation, the generalized unitary group for $J$ is the classical indefinite unitary group $\U(J;\C,\sigma)=\U(p,q;\C)$.
\end{example}

\noindent
The unitary geometries are determined by the action of the groups $\U(J;A,\sigma)$ on $\mathsf{AP}^n$, or equivalently by $\U(J;A,\sigma)$ together with its intersection with a point stabilizer of the $\GL(n+1,\A)$ on $\mathsf{AP}^n$. 

\begin{definition}
A unitary geometry over $(A,\sigma)$ is given by the pair $(G,C)=(\U(J;A),\mathsf{Stab}([p])\cap \U(J;A))$ for $J\in\Herm(n;A)$ and $[p]\in\mathsf{AP}^n$ and is called the \emph{unitary geometry corresponding to} $(J,p)$
\end{definition}

\noindent
When $p\in\mathsf{AP}^n$ is not on the lightcone of the Hermitian form $J$ (that is, $p^\dagger Jp\neq 0$) this embeds as a subgeometry of projective geometry.  A priori a unitary geometry depends on both a choice of Hermitian form $J$ and projective point $[p]$, and at times it is useful to be able to vary these two parameters independently.  However the choice of point can be absorbed into the choice of Hermitian form as the proposition below shows, which we will often do out of convenience.

\begin{lemma}
\label{prop:Unitary_Basepoint_Change}
Let $(A,\sigma)$ be an algebra with involution, and 	$J\in\Herm(n;A)$.  Then if $p,q\in A^n$ have the unitary geometry corresponding to $(J,p)$ is isomorphic to that of $(C^\dagger JC, q)$ for some $C\in\GL(n;A)$.
\end{lemma}
\begin{proof}
Let $J\in\Herm(n;A)$ and $p,q\in A^n$.  Taking $C\in\GL(n;A)$ with $Cp=q$ note that $\mathsf{stab}(q)=C\mathsf{stab}(p)C\inv$ and conjugation by $A$ gives an isomorphism between the group-stabilizer geometries $(\U(J;A),\mathsf{stab}(p))$ and $(C\U(J;A)C\inv, C\mathsf{stab}(p)C\inv)$.  
But $C\U(J;A)C\inv=\U(C^\dagger JC;A)$ and so we have an isomorphism of geometries $(\U(J;A),\mathsf{stab}(p))$ and $(\U(C^\dagger JC;A),\mathsf{stab}(q))$ as claimed.
	
\end{proof}

\noindent
Thus we will fix the point $p=(0,\ldots,0,1)$ and talk of \emph{the} unitary geometry corresponding to $\U(J;A)$ as the geometry corresponding to the pair $(J,[p])$.

\begin{definition}
The \emph{unitary geometry} for $\U(J;A)\subg\GL(n+1;A)$ is given by the pair $(\U(J;A,\sigma),\mathsf{USt}(J;A))$ for $\mathsf{USt}(J;A)=\U(J;A)\cap\St(n+1, A)$.
\end{definition}

\noindent
The fact that $\USt(J;A)$ is a Lie group is obvious as its closed in $\St(n;A)$, but again we give a more detailed argument for future use.

\begin{lemma}
The restriction of $\Phi_J\colon X\mapsto X^\dagger JX$ to $\St(n;A)$ is a submersion onto $\Herm(n;A)$, for $J$ diagonal (surely this restraint can be removed)
\label{prop:UnitaryStabilizer_Submersion}
\end{lemma}
\begin{proof}
For clarity write $D\Phi_J=\phi$ and $\St(n;A)=\St$.  
As $\St$ is the intersection of a linear subspace $\bar{\St}\subset\M(n;A)$ with $\GL(n;A)$ for each $B\in\St$ the tangent space $T_B\St=\bar{\St}$.  
The kernel of the restricted map $\phi_{\bar{\St}}$ is the intersection of $\ker\phi$ with $\bar{\St}$, allowing us to calculate the dimension of the image of using
$$\dim\mathrm{img}(\phi|_{\bar{\St}})_B=\dim\bar{\St}-\dim(\ker (\phi)_B\cap \bar{\St}).$$ 

Thus calculating the dimension of $\mathrm{img}(\phi|_{\bar{\St}})_B$ amounts to understanding the relationship between $\ker (\phi_B)$ and $\bar{\St}$ in $\M(n;A)$.  In particular if these subspaces sum to all of $\M(n;A)$ we are done, as 
$$\dim M(n;A)=\dim(\ker\phi+\bar{\St})=\dim\ker\phi+\dim\bar{\St}-\dim(\ker\phi\cap\bar{\St})$$
$$=\dim\ker\phi+\dim\mathrm{img}(\phi|_{\bar{\St}})$$

By previous work \pref{prop:Unitary_Submersion} $\ker\phi$ is the same dimension as the space of skew-Hermitian matrices, which would imply that the image of $\phi|_{\bar{\St}}$ has the same dimension as the Hermitian matrices, which are its codomain so $\phi|_{\bar{\St}}$ is surjective.  Thus it only remains to show $\M(n;A)=\ker\phi+\bar{\St}$.

The only restriction on the matrices of $\bar{\St}$ is that the first $n-1$ entries of their last column are zero.  Thus it suffices to show that any $v\in A^{n-1}$ can appear as the first $n-1$ entries of the final column of a matrix in $\ker\phi$.  Recall from \cref{prop:Unitary_Determinant} that $\ker\phi=(B^\dagger J)\inv\SkHerm(n;A)$, and observe that all but the last entry of the final column of matrices in $\SkHerm(n;A)$ can be arbitrary (the last element must be zero).
Then $C=(B^\dagger J)\inv$ acts via a homeomorphism $A^n\to A^n$ on vectors, in particular on the last column of matrices in $\SkHerm$.

Specializing  now to the case $J\in\Diag$, the matrix $(B^\dagger J)\inv$ is of the form $\smat{X&v\\0&\alpha}$ for $X\in\GL(n-1;\R)$, which sends $(\vec{v},0)$ to $(Xv+w,0)$ and restricts to a homeomorphism $A^{n-1}\to A^{n-1}$.  Thus any vector can arise as the last column in $\ker\phi$ and we are done.
\end{proof}

\section{Isomorphism Type}
\label{sec:Iso_Type_GeosAPn}

In the sections above, we have defined unitary/orthogonal and projective geometries over arbitrary (finite dimensional commutative) real algebras.
To begin to tame the maddness we need to develop an understanding of the different flavors of geometry which appear.
An algebra $A$ is \emph{decomposable} if is is isomorphic to a nontrivial direct sum of algebras.
An algebra with involution $(A,\sigma)$ is decomposable if $A=A_1\oplus A_2$ and $\sigma=\sigma_1\oplus\sigma_2$ decomposes as a direct sum of involutions.
The main result of this section is that to understand projective and unitary geometries over algebras, it suffices to understand the indecomposable ones.

\subsection{Projective Geometries}

\begin{proposition}
Let $A=A_1\oplus A_2$ be a direct sum of commutative algebras.  Then  projective geometry over $A$	 decomposes as a direct product of the projective geometries over $A_1$ and $A_1$.
\label{prop:Proj_Geos_Over_Reducible_Algs}
\end{proposition}
\begin{proof}
Let $e_1,e_2$ be orthogonal primitive idempotents so $A=A_1e_1+A_2e_2$ as a direct sum.  
Then $\GL(n,A)=\GL(n,A_1)\oplus\GL(n,A_2)$ and $\St(n;A)=\St(n;A_1)\oplus\St(n;A_2)$ are easily checked, and as the linear action of $\St(n;A)$ on $\GL(n;A)$ by translation preserves this decomposition,
$(\St(n;A),\GL(n;A))
\cong (\St(n;A_1),\GL(n;A_1))\times(\St(n;A_2),\GL(n;A_2))$.  

\end{proof}

\noindent
To understand this decomposition better in terms of spaces it helps to think about the set $Z((A_1\oplus A_2)^n)$: a point $(p_1,p_2)$ is a 'generalized zero' if $\langle p,q\rangle\neq A_1\oplus A_2$.  
This occurs precisely when one of the $p_i$ is in $Z(A_i^n)$, so the complement consists of points $(p_1,p_2)$ with $[p_i]\in \mathsf{A_iP}^{n-1}$ . 
Quotienting by the action of $A^\star=A_1^\star\times A_2^\star$ on this sends $(p_1,p_2)$ to $([p_1],[p_2])\in \mathsf{A_1P}^{n-1}\times\mathsf{A_2P}^{n-1}$.

Two obvious examples of indecomposable real algebras are $\R$ itself and $\C$, with corresponding projective spaces $\RP^n$ and $\CP^n$.  
The algebra $A=\R\oplus\R$ provides decomposable examples, for instance $(\R\oplus\R)P^1$ is a geometry on the torus.  
 A new example is provided by the algebra of dual numbers, $\R_\ep=\R[\ep]/(\ep^2)$ which is an indecomposable two dimensional algebra with nilpotents.  
 Both $\R_\ep\mathsf{P}^n$ and $(\R\oplus\R)\mathsf{P}^n$ will be discussed in detail in the final section on applications.

\subsection{Unitary Geometries}

\begin{proposition}
If $A=A_1\oplus A_2$ and $\sigma$ preserves the factors $\sigma_1\oplus\sigma_2: A_1\oplus A_2\to A_1\oplus A_2$, then $\U(J;A,\sigma)\cong \U(J_1;A_1,\sigma_1)\times \U(J_2;A_2,\sigma_2)$ decomposes as a product for $J=J_1e_1+ J_2e_2\in\M(n,A)$.
\label{prop:Unitary_Geo_Decompsable}
\end{proposition}
\begin{proof}
First note that $\Herm(n;A,\sigma)=\Herm(n;A_1,\sigma_1)\oplus\Herm(n;A_2,\sigma_2)$ as $J^\dagger=(J_1e_1+ J_2e_2)^\dagger=(\sigma_1(J_1)^Te_1+ \sigma_2(J_2)^Te_2)$.  
Fix a nondegenerate $J=J_1e_1+J_2e_2\in\Herm(n;A,\sigma)$ and let $X=X_1e_1+X_2e_2\in\U(J;A,\sigma)$.  
The condition $X^\dagger JX=J$ decouples as two independent equations along the direct sum decomposition as $\sigma$ preserves the factors, $X_i^\dagger J_iX_i=J_i$ for $i\in\set{1,2}$.  
Thus $X_i\in\U(J_i;A_o,\sigma_i)$ and so the map $X\mapsto(X_1,X_2)$ provides a group homomorphism $\U(J,A,\sigma)\to\U(J_1;A_1,\sigma_1)\times \U(J_2;A_2,\sigma_2)$.  
By the same reasoning any pair $(X_1,X_2)$ with $X_i\in\U(J_i;A_i,\sigma_i)$ corresponds to an element $X_1e_1+X_2e_2\in\U(J;A,\sigma)$ so this is an isomorphism.
\end{proof}

\noindent
As with projective geometries, it suffices to understand the indecompsables.  
The simplest such case is provided by pairs $(A,\sigma)$ where $A$ is decomposable but $\sigma$ does not preserve the decomposition - in particular we are interested in algebras $\Lambda=A\oplus A$ with $\sigma$ the \emph{swap map} $\sigma(x,y)=(y,x)$.
Here rather surprisingly the isomorphism type of the generalized unitary groups $\U(J;A,\sigma)$ is independent of the choice of $J$.

\begin{proposition}
Let $\Lambda=A\oplus A$ and $\sigma:\Lambda\to\Lambda$ be the coordinate swap map.  
Then $\U(J;\Lambda,\sigma)\cong\GL(n,A)$ for any nondegenerate $\sigma$-hermitian matrix $J$.	
\label{prop:Unitary_Geo_Coordinate_Swap}
\end{proposition}
\begin{proof}
Let $J=J_1e_1+J_2e_2$ be $\sigma$-Hermitian, then $(J_1e_1+J_2e_2)^\dagger=J_2^Te_1+J_1^Te_2$ so $J_1^T=J_2$ and $\Herm(n;\Lambda,\sigma)\cong \M(n,A)$.  
As $\det(Xe_1+Ye_2)=\det(X)e_1+\det(Y)e_2$ in $A\oplus A$, the nondegenerate Hermitian matrices arise from $\GL(n,A)$.
Given a nondegenerate $\mathsf{J}=Je_1+J^Te_2\in\Herm(n,\Lambda,\sigma)$ the corresponding unitary group 
$$\U(\mathsf{J},\Lambda,\sigma)=\set{ Xe_1+Ye_2\mid (Xe_1+Ye_2)^\dagger (Je_1+J^Te_2)(Xe_1+Ye_2)=(Je_1+J^Te_2)}$$
expanding this component-wise gives the redundant equations
$Y^TJX=J$ and $X^TJ^TY=J^T$.  
Taking the determinant of the first gives $\det(Y)\det(X)\det(J)=\det(J)$ and by the assumption that $J$ is nondegenerate, $\det(Y)\det(X)=1$ so both $X,Y$ are invertible.  
Rearranging gives $Y=J^{-T}X^{-T}J^T$ and so all elements of $\U(\mathsf{J},\Lambda,\sigma)$ are of the form $Xe_1+(JX\inv J\inv)^Te_2$ for some $X\in\GL(n,A)$.  
Running this argument backwards shows that any $X\in\GL(n;A)$ gives an element $Xe_1+(JX\inv J\inv)^Te_2$ of $\U(\mathsf{J};\Lambda,\sigma)$ and so $X\mapsto Xe_1+(JX\inv J\inv)^Te_2$ is a bijection $\Phi\colon\GL(n;A)\to\U(\mathsf{J};\Lambda,\sigma)$.  
Its an easy check that this is a group homomorphism, and so we're done.
\end{proof}

\begin{corollary}
With $\Lambda,A,\sigma$ as above, $\SU(J,\Lambda,\sigma)\cong\SL(n,A)$.	
\end{corollary}
\begin{proof}
Taking the determinant and simplifying gives $\det(Xe_1+(JX\inv J\inv)e_2)=\det(X)e_1+\det(X)\inv e_2$.  
This is only real if $\det(X)=\det(X)\inv$, and is only $1$ if furthermore $\det(X)=1$, so the image of $\SL(n;A)$ under $\Phi$ is precisely $\SU(J;\Lambda,\sigma)$.
\end{proof}

\noindent
This result has a natural generalization to involutions of the form $\sigma(x,y)=(\phi(y),\tau(x))$ for $\phi,\tau$ involutions of $A$.  
Recall the equalizer of two maps $f,g:X\to X$ is $\mathsf{Eq}(f,g)=\set{x\mid f(x)=g(x)}$.

\begin{proposition}
Let $\Lambda=A\oplus A$ and $\sigma\colon\Lambda\to\Lambda$ 	be of the form $\sigma(x,y)=(\phi(y),\psi(x))$ for $\phi,\psi$ involutions of $A$.  
Then $\U(J;\Lambda,\sigma)\cong \mathsf{Eq}(\Phi,\Psi)\cap\GL(n,A)$ for $\Phi,\Psi$ the extensions of $\phi,\psi$ to $\M(n,\A)$ respectively.
\end{proposition}
\begin{proof}
	Proceeding similarly to above, note that $(J_1,J_2)\in\Herm(n,\Lambda,\sigma)$ if $(J_1,J_2)^\dagger=(\phi(J_2)^T,\psi(J_1)^T)=(J_1,J_2)$, so $\phi(J_2)^T=J_1$, $\psi(J_1)^T=J_2$.  
	Applying $\psi$ to the second equation gives $\psi^2(J_1)^T=J_1^T=\psi(J_2)$ and comparing with the transpose of the first gives $J_1^T=\phi(J_2)=\psi(J_2)$ thus $J_2\in\mathsf{Eq}(\phi,\psi)$ and $\Herm(n;\Lambda,\sigma)=\set{(\phi(J)^T,J)\mid J\in\mathsf{Eq}(\phi,\psi)}$.
	
Fix a nondegenerate $\mathsf{J}=(\phi(J)^T,J)\in\Herm(n,\Lambda,\sigma)$ and let $(X,Y)\in\U(\mathsf{J};\Lambda)$.  
Then $(X,Y)^\dagger (\phi(J)^T,J)(X,Y)=(\phi(J)^T,J)$ which expands component-wise to the two equations
$\Phi(Y)^T\Phi(J)^TX=\Phi(J)^T$ and  $\Psi(X)^TJY=J$.
Taking the determinant of both equations and using that $J$ is nondegenerate gives that $X$ and $Y$ are invertible, playing around with the equations gives two ways to solve for $Y$,  $J\inv \Psi(X)^{-T}J=Y=J\inv\Phi(X)^{-T}J$.  
Thus $\Psi(X)=\Phi(X)$ so $X\in\mathsf{Eq}(\Phi,\Psi)$.

In fact, given any $X\in\GL(n,A)\cap\mathsf{Eq}(\Phi,\Psi)$ the matrix $(X,J\inv \Phi(X)^{-T}J)$ is an element of $\U(\mathsf{J},\Lambda)$ as is easily checked, so the map $f\colon \GL(n;A)\cap\mathsf{Eq}(\Phi,\Psi)\to \U(\mathsf{J};\Lambda,\sigma)$ is a bijection.  
That $f$ is a group homomorphism follows immediately from writing down $f(X)f(Y)$ and $f(XY)$.

\end{proof}

\noindent
There's a potentially useful perspective to take on this result.  The collection $\mathsf{Eq}(\phi,\psi)$ is a subalgebra of $A$ on which $\phi=\psi$ restricts to an involution.  
We can think of both $\phi$ and $\psi$ as extensions of this involution to $A$.
 In this light, $\mathsf{Eq}(\Phi,\Psi)=\M(n,\mathsf{Eq}(\phi,\psi))$ and $\mathsf{Eq}(\Phi,\Psi)\cap\GL(n;A)=\GL(n,\mathsf{Eq}(\phi,\psi))$. 
  Thus we may more succinctly write the result above as 
 $$\U(J;\Lambda,\sigma)=\GL(n;\mathsf{Eq}(\phi,\psi))$$

\subsection{Specific Examples}

We briefly mention some elementary examples that have shown up throughout this dissertation (or will show up in the following chapter!).
When $A=\R$ we recover the usual geometries $\RP^n$ and the pseudo-Riemannian geometries $X(p,q)$ associated to the orthogonal groups $\O(p,q;\R)$ of Chapter \ref{chp:Orthogonal_Groups}.
When $A=\C$, we recover complex projective geometry $\CP^n$, the geometry of the complex orthogonal group $\O(n;\C)$ (remember, all orthogonal groups are conjuate over $\C$) and the complex unitary geometries of $\U(p,q;\C)$, including complex hyperbolic space.

When $A=\R\oplus\R$, Proposition \ref{prop:Proj_Geos_Over_Reducible_Algs} implies that the associated projective geometries $(\R\oplus\R)\mathsf{P}^n\cong \RP^n\times\RP^n$ are products of real projective space with itself.
Likewise, Proposition \ref{prop:Unitary_Geo_Decompsable} to analyze the orthogonal groups, and associated orthogonal geometries over $\R\oplus\R$: they similarly turn out to be products $\O(p,q;\R\oplus\R)\cong\O(p,q;\R)\times\O(p,q;\R)$.
The unitary geometries over $\R\oplus\R$ with respect to the coordinate swap map are all isomorphic to point-hyperplane projective space, as first noticed in Chapter \ref{chap:HC_and_HRR}.

As a non-commutative example, we quickly mention the quaternions: as a divsion ring there are no surprises in defining quaternionic projective geometries, and identically to $\C$ all quaternionic orthogonal groups are conjugate.  
The generalized unitary groups over the quaternions with respect to quaternionic conjugation are the compact symplectic groups, and in particular $\U(n,1;\H)$ is the automorphisms of quaternionic hyperbolic space.

\chapter{Applications}
\label{chp:Applications}

In this chapter we give some basic applications of the theory of families of geometries, producing many new examples of geometric transitions.
In particular, we focus on generalizations of the transition from $\Hyp_\C^n$ to $\Hyp_{\R\oplus\R}^n$, showing that any given family of algebras produces corresponding families of projective, unitary and orthogonal geometries.
We then turn briefly to another application, and study transitions that occur from a group action on a space, when we may interpret the collection of orbits as a smoothly transitioning family of spaces.
This will, among other things, provide a means of transitioning between Hyperbolic and de Sitter geometry, which does not arise within an ambient projective geometry.

\section{Families of Real Algebras}
\label{sec:Families_Real_Alg}

Recall that a family of algebras may be thought of as a vector bundle together with a map $\mu:\fam{A}\timesd\fam{A}\to\fam{A}$ restricting slicewise to the multiplication of an algebra structure on $\fam{A}_\delta$.

\begin{proposition}
The units $\fam{A}^\times\to\Delta$ of a family of algebras form a family.
\label{prop:Fam_Units_Open}
\end{proposition}
\begin{proof}
We will show $\fam{A}^\times\subset\fam{A}$ is open, which if $\fam{A}_Z$ is the collection of zero divisors of $\fam{A}$, is equivalent to showing $\fam{A}_Z$ is closed.
Let $\set{z_i}$ be a sequence of zero divisors in $\fam{A}_Z$ converging to $z\in\fam{A}$.  Write $\pi(z)=\delta$, and $\pi(z_i)=\delta_i$ for convenience. Forgetting the multiplicative structure $\fam{A}\to\Delta$ is a family of real vector spaces, and so by \cref{prop:VS_Trivial} we may choose a compact trivializing neighborhood $\delta\in U$ and $h\colon U\times\R^n\to\fam{A}|_U$ a trivialization.  The set $\fam{A}_Z$ is invariant under real scaling, so we may choose a $w_i\in h(\S^{n-1}\times\set{\delta_i})$ for each $z_i$ such that $z_iw_i=0$.  Thus $\set{w_i}\subset h(\S^{n-1}\times U)$ is a subset of a compact space, subconverging $w_i\to w$.  As $w_iz_i=0$ for all $i$, $zw=0$ by continuity of multiplication so $z$ is a zero divisor.
\end{proof}

\noindent 
An involution is a map of families $\fam{A}\stackrel{\sigma}{\to}\fam{A}$ squaring to the identity and restricting slicewise to an algebra involution.  
On each algebra $\fam{A}_\delta$, the restricted involution $\sigma_\delta$ gives a direct sum decomposition $\fam{A}_\delta=\mathsf{Fix}(\sigma_\delta)\oplus\mathsf{Neg}(\sigma_\delta)$.  
The maps $\Phi_\pm\colon \alpha\mapsto \alpha\pm\sigma(\alpha)$ are the projections onto the factors of this direct sum decomposition.  

\begin{proposition}
Let $\fam{A}\to\Delta$ be a family of algebras with involution $\fam{A}\stackrel{\sigma}{\to}\fam{A}$.  
Then the collections $\fam{Fix}(\sigma)=\set{\alpha\in\fam{A}\mid \alpha=\sigma(\alpha)}$ and $\fam{Neg}(\sigma)=\set{\alpha\in\fam{A}\mid \sigma(\alpha)=-\alpha}$ are subfamilies of $\fam{A}\to\Delta$.
\label{prop:Fix_Family}
\end{proposition}
\begin{proof}
We detail the argument for $\fam{Fix}(\sigma)$, the remaining case is argued analogously.
  We define $\Phi_-(\alpha)=\alpha-\sigma(\alpha)$ on and note that $\fam{Fix}(\sigma)=\Phi_-\inv\set{\fam{0}(\Delta)}$ is the preimage of the zero section.  
  Restricted to any fiber, $\Phi_-$ is the projection $A\to\mathsf{Neg}(\sigma)$ described previously.  
  Thus when $\fam{A}\to\Delta$ is a smooth family of algebras, the restriction of $\Phi_-$ to each fiber is a smooth submersion.  
  Applying \cref{prop:Fiberwise_Family=Family}, if a smooth map of families is a submersion fiber-wise, it is itself a submersion, and thus gives $\fam{A}$ the structure of a smooth family over $\fam{Neg}(\sigma)$.  
We may then apply \cref{obs:Pullback_Eqns} to pull this family back along the zero section $\fam{0}\colon\Delta\to\fam{Neg}(\sigma)\subset\fam{A}$ to get a family $\fam{0}^\star\fam{A}\to \Delta$.
The elements of $\fam{0}^\star\fam{A}$ satisfy $\Phi_-(\alpha)=\fam{0}_{\pi(\alpha)}$ or $\alpha-\sigma(\alpha)=0$.  Thus $\fam{0}^\star\fam{A}=\fam{Fix}(\sigma)$.
\end{proof}

\noindent
A family $\fam{A}\to\Delta$ gives rise to a family of matrix algebras $\fam{M}(n,\fam{A})\to\Delta$, constructed on the underlying space $\fam{A}^{n^2}\to\Delta$ by imposing matrix multiplication.  
An involution $\sigma$ on $\fam{A}$ can be promoted to an involution $\dagger\colon\fam{M}(n;\fam{A})\to\fam{M}(n;\fam{A})$ given by $X^\dagger=\sigma(X)^T$.  
Applying \cref{prop:Fix_Family} to $\dagger$ gives the families $\fam{Fix}(\dagger)=\fam{Herm}(n;\fam{A},\sigma)$ and $\fam{Neg}(\dagger)=\fam{SkHerm}(n;A,\sigma)$ of Hermitian and skew-Hermitian matrices, respectively. 
The usual formula for the determinant provides a $\cat{C}$-map of families $\fam{det}\colon\fam{M}(n,\fam{A})\to\fam{A}$.

Two families which we use to illustrate the theory are as follows.

\begin{definition}
The family $\Lambda_\R$ of 2-dimensional algebras over $\R$ from Chapter \ref{chp:HC_To_HRR_Transition}, $\Lambda_\delta=\R[\lambda]/(\lambda^2=\delta)$ transitioning from $\C$ when $\delta<0$ to $\R\oplus\R$ when $\delta>0$.
\end{definition}

\begin{definition}
A \emph{quaternion algebra} over $\R$ is a four dimensional noncommutative real algebra defined by two real parameters $a,b\in\R$.
The multiplication on $\R^4=\R\{1,i,j,k\}$ is defined so that $i^2=a$ and $j^2=b$ together with $ij=-ji=k$.
When $a=b=-1$ this recovers the usual quaternions.
\end{definition}

\begin{definition}
The family $\fam{H}\to\R^2$ of quaternion algebras has total space $\fam{H}=\R\{1,i,j,k\}\times\R^2$ and multiplication on each $\fam{H}_(a,b)=\R\{1,i,j,k\}$ is defined such that $i^2=a$ and $j^2=b$.
This is a continuous family of algebras transitioning from the usual quaternions when $a,b<0$ to the algebra of $2\times 2$ matrices when either $a$ or $b$ is $>0$.
\end{definition}

\section{Families of Projective Geometries}
\label{sec:Families_Proj_Geo}

Given a smooth family of algebras, constructing a smooth family of geometries it amounts to showing that the given automorphism and stabilizer groups vary smoothly along with the algebra.

\begin{proposition}
Let $\fam{A}\to\Delta$ be a smooth family of algebras.  Then $\fam{GL}(n,\fam{A})\to\Delta$ is a family of Lie groups.	
\end{proposition}
\begin{proof}
The general linear family is the units of the matrix algebra $\fam{GL}(n;\fam{A})=\fam{M}(n,\fam{A})^\star$ and so is an open subset by \cref{prop:Fam_Units_Open}. 
Thus the restricted projection map gives $\fam{GL}(n;\fam{A})$ the structure of a smooth family.
\end{proof}

\begin{proposition}
\label{prop:det_GLN_family}
Let $\fam{A}\to\Delta$ be a smooth family of commutative algebras and $\fam{det}\colon\fam{M}(n,\fam{A})\to\Delta$ the determinant map.  Then $\fam{GL}(n,\fam{A})\stackrel{\fam{det}}{\to}\fam{A}^\times$ is a family.
\end{proposition}
\begin{proof}
The determinant is a map of families $\fam{M}(n,\fam{A})\to\fam{A} $ over $\Delta$, so by \cref{prop:Fiberwise_Family=Family} it is a submersion if its restriction to the vertical slices are.  But this is the content of \cref{prop:det_is_submersion}, $\GL(n,A)\to A^\times$ is a submersion for any smooth algebra $A$.  
\end{proof}

\begin{corollary}
The groups $\fam{SL}(n;\fam{A})$ are a subfamily of $\fam{GL}(n;\fam{A})\to\Delta$ when $\fam{A}\to\Delta$ is commutative.
\end{corollary}
\begin{proof}
By the previous proposition, $\fam{GL}(n;\fam{A})\stackrel{\fam{det}}{\to}\fam{A}^\times$ is a family, and let $\fam{1}\colon\Delta\to\fam{A}$ be the identity section.  Then the pullback $\fam{1}^\star\fam{GL}(n;\fam{A})\to\Delta$ embeds as the subfamily $\fam{SL}(n;\fam{A})\to\Delta$.
\end{proof}

\noindent
Thus, it remains only to show that the stabilizer subgroups vary smoothly.

\begin{proposition}
	The stabilizer groups 
	$$\fam{St}(n+1;\fam{A})=\set{\pmat{X& \vec{0}\\\vec{v}&\alpha}\mid X\in\fam{GL}(n;\fam{A}),\vec{v}\in\fam{A}^n;\alpha\in\fam{A}^\times}$$
	are a subfamily of $\fam{GL}(n;\fam{A})$.
\end{proposition}
\begin{proof}
The choices of elements $X,\vec{v}$ and $\alpha$ are independent, so topologically $\fam{St}(n+1;\fam{A})=\fam{GL}(n;\fam{A})\timesd\fam{A}^n\timesd\fam{A}^\times$ is a product of families  and so is abstractly a family.  In line with previous arguments however the map $\fam{GL}(n+1;\fam{A})\to\fam{A}^n$ sending each matrix to the $n$ first elements of the last column is a submersion as it is one fiberwise \pref{prop:Fiberwise_Family=Family} and so the pullback of the zero section $\fam{0}\colon\Delta\to\fam{A}^n$ is a subfamily of $\fam{GL}(n+1;\fam{A})$ easily seen to be $\fam{St}(n+1;\fam{A})$. 
\end{proof}

\noindent 
By similar reasoning, when $\fam{A}\to\Delta$ is commutative we can see that the collection $
\fam{SSt}(n;\fam{A})=\fam{St}(n;\fam{A})\cap\fam{SL}(n;\fam{A})$ is a subfamily of the special linear family.

\begin{theorem}
A smooth family of algebras $\fam{A}\to\Delta$ determines a smooth family of projective geometries $\fam{A}\mathsf{P}^n\to\Delta$ for each $n\in\N$. 	
\label{thm:Proj_Geo_Families}
\end{theorem}

\noindent
This has a lot of instances, one for each family of algebras.
In particular, it applies to the $\C\to\R\oplus\R$ transition utilized extensively in Chapters \ref{chap:HC_and_HRR} and \ref{chp:HC_To_HRR_Transition}.

\begin{corollary}
The projective spaces $\Lambda_\delta\mathsf{P}^n$ form a continuous family of geometries, transitioning from $\CP^n$ to $(\R\oplus\R)\mathsf{P}^n\cong\RP^n\times\RP^n$.	
\end{corollary}

\noindent
In dimension 1, this provides a transition from the geometry of $\CP^1$ to the torus with an action of $\SL(2;\R)\times\SL(2;\R)$.
Interpreting these as the boundary of $\Hyp^3$ and $\mathsf{AdS}^3$ respectively, this gives an alternative means of constructing the transition of Danciger \cite{Danciger11} in dimension 3.

\begin{figure}
\centering
\includegraphics[width=0.85\textwidth]{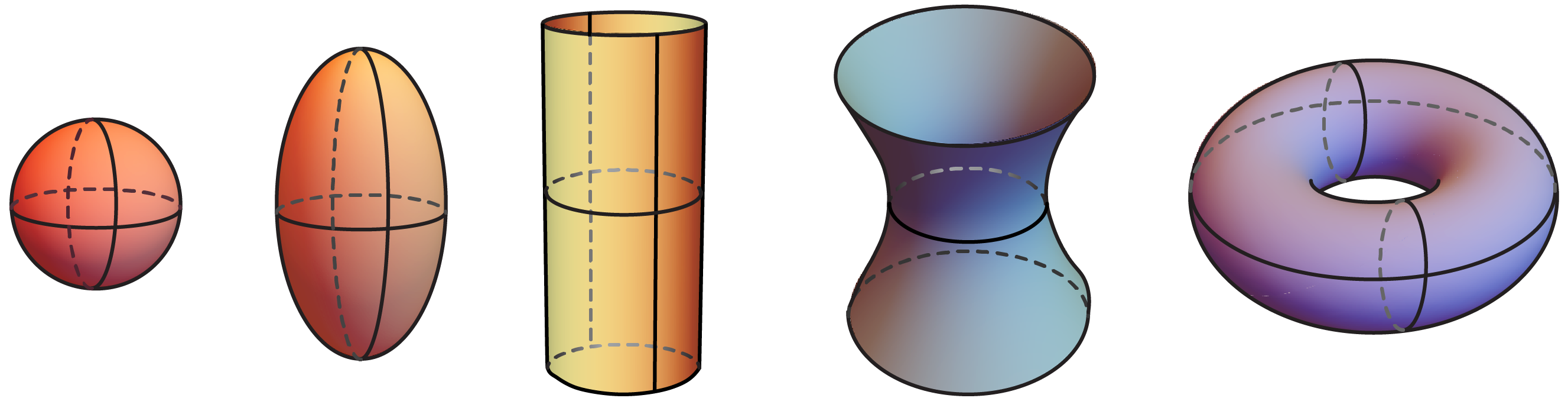}
\caption{The transition $\CP^1$ to $(\R\oplus\R)\mathsf{P}^1$.}	
\end{figure}

\begin{corollary}
Applying Theorem \ref{thm:Proj_Geo_Families} to the family $\fam{H}\to\R^2$ of real quaternion algebras gives a transition of quaternionic projective space to projective space defined over $\M(2;\R)$.
It is an interesting future direction to consider what these transitions look like, and in particular analyze $\M(2;\R)\mathsf{P}^n$.
\end{corollary}

\section{Families of Unitary Geometries}

Given a nondegenerate section $\fam{J}\colon\Delta\to\fam{Herm}(n;\fam{A},\sigma)$, one can define for each $\delta$ the unitary group $\U(\fam{J}_\delta;\fam{A}_\delta,\sigma_\delta)\subg\GL(n;\fam{A}_\delta)$.  The union of these is the \emph{generalized unitary family} corresponding to $\fam{J}$ over $\Delta$.  We check here immediately that this is indeed a family.

\begin{proposition}
Let $(\fam{A},\sigma)\to\Delta$ be a family of algebras and $\fam{J}\colon\Delta\to\fam{Herm}(n;\fam{A},\sigma)$ a smooth nondegenerate section.  
Then $\fam{U(J;A)}$ is a smooth subfamily of $\fam{GL}(n;\fam{A})$.
\end{proposition}
\begin{proof}
The map of families $\Phi_\fam{J}\colon\fam{GL}(n;\fam{A})\to\fam{Herm}(n;\fam{A})$ given by $X\mapsto X^\dagger \fam{J}_{\pi(X)}X$ is a smooth map, and by \cref{prop:Unitary_Submersion} is fiber-wise a submersion.  
Thus by \cref{prop:Fiberwise_Family=Family} actually gives $\fam{GL}(n;\fam{A})$ the structure of a family over $\fam{Herm}(n;\fam{A})$.  
The section $\fam{J}$ then gives a pullback family $\fam{J}^\star \fam{GL}(n;\fam{A})$ over $\Delta$, which selects out those matrices in $\fam{GL}(n;\fam{A})$ such that $\Phi(X)=J_{\pi(X)}$. That is, $X^\dagger \fam{J}_{\pi(X)}X=\fam{J}_{\pi(X)}$, which is the definition of $\fam{U(J,A)}$.

\begin{center}
\begin{tikzcd}
\mathcal{J}^\star\mathcal{GL}(n;\mathcal{A}) \arrow[rr, hook] \arrow[d, two heads] &  & \mathcal{GL}(n;\mathcal{A}) \arrow[d, "\Phi_\mathcal{J}", two heads] \\
\Delta \arrow[rr, "\mathcal{J}"', hook]                                            &  & \mathcal{Herm}(n;\mathcal{A})                                       
\end{tikzcd}
\end{center}

\end{proof}

\noindent
Recalling \cref{prop:Unitary_Determinant} that for a fixed smooth algebra $\det\colon\U(J;A)\to \U(A)$ is a submersion, applying \cref{prop:Fiberwise_Family=Family} as above shows the determinant gives $\fam{U(J,A)}$ the structure of a family over $\fam{U(A)}$.  Pulling back along the identity section gives the family of special unitary groups.

\begin{corollary}
The special unitary groups $\fam{SU(J;A)}$ are a subfamily of $\fam{U(J;A)}$.
\end{corollary}

\noindent
Unitary geometries are defined via a pair $(\U(J;A),\USt(J;A))$, and so given a family of algebras $(\fam{A},\sigma)\to\Delta$ and a smooth section $\fam{J}\colon\Delta\to\fam{Herm}^\times(n;\fam{A})$ the corresponding collection of geometries is given by $(\fam{U(J;A)},\fam{USt(J;A)})$ for $\fam{USt(J;A)}=\fam{St}(n;\fam{A})\cap\fam{U(J;A)}$. 
 As we have already studied the unitary families, to see this is a smooth family of geometries it suffices to show that the stabilizers form a subfamily of $\fam{U(J;A)}$.
 
 \begin{proposition}
 The unitary stabilizers $\fam{USt(J;A)}$ form a subfamily of $\fam{U(J;A)}\to\Delta$.	
 \end{proposition}
 \begin{proof}
 Let $\Psi\colon\fam{U(J;A)}\to\fam{A}^{n-1}$ be the map sending each matrix to the first $n-1$ entries of its last column.  
 This is a map of families over $\Delta$ and \cref{prop:UnitaryStabilizer_Submersion} shows that it is fiberwise a submersion, thus in fact gives $\fam{U(J;A)}$ the structure of a family over $\fam{A}^{n-1}$.  
 Pulling this family back over the zero section $\fam{0}\colon\Delta\to\fam{A}^{n-1}$ gives the family $\fam{0}^\star\fam{U(J;A)}\to\Delta$ with total space the intersection $\fam{U(J;A)}\cap\fam{St}(n;\fam{A})$.  
 \end{proof}
 
 \begin{theorem}
 \label{thm:Unitary_Families}
 Given a smooth family of algebras $\fam{A}\to\Delta$ and a "constant" section $\fam{J}\colon\Delta\to\fam{Herm}(n;\fam{A})$, $\delta\mapsto (J,\delta)$, there is a corresponding smooth family of unitary geometries $(\fam{U(J,A)},\fam{UST(J;A)})$.	
 \end{theorem}

\noindent 
This theorem immediately implies the transition of Chapter \ref{chp:HC_To_HRR_Transition}, among other things.
 
 \begin{corollary}
 There is a transition $\Hyp_{\C}^n$ to $\Hyp_{\R\oplus\R}^n$ through $\Hyp_{\R_\ep}^n$ by considering the signature $(n,1)$ unitary geometries over $\Lambda_\R$.	
 \end{corollary}
 
 \noindent 
 But recalling that over $\R\oplus\R$ the signature of a unitary group is not well-defined and all unitary geometries are isomorphic (in fact, they are all isomorphic to point-hyperplane projective space); we also have the following corollary.
 
 \begin{corollary}
 Given any $(p,q)$; there is a transition from the pseudo-Riemannian unitary geometry of signature $(p,q)$ over $\C$ to Point-Hyperplane projective space.	
 \end{corollary}
 
 \noindent
 Letting the involution in the definition of generalized unitary groups be trivial, we may consider the families of orthogonal geometries along the transition as well.
 In this case, signature is meaningless over $\C$, and all orthogonal geometries are isomorphic.
 
 \begin{definition}
 The $n$ dimensional orthogonal geometry over $\C$ is given by the pair $(\SU(n+1;\C),\mathsf{USt}(n+1;\C))$.	
 \end{definition}
 
 \noindent 
 Over $\R\oplus\R$, the trivial involution defining the orthogonal groups implies that they all split as a product: $\O(p,q;\R\oplus\R)\cong\O(p,q;\R)\times\O(p,q;\R)$, and the corresponding geometry is the product of the pseudo-Riemannian homogeneous geometry of signature $(p,q)$ with itself.
 Together with the above this gives another class of transitions between homogeneous spaces.
 
 \begin{corollary}
 For every $(p,q)$ there is a transition between the product geometry of $(\O(p,q),X_{p,q})$ with itself, and the $(p+q-1)$-dimensional complex orthogonal geometry.
 \end{corollary}

\noindent 
As a specific example, even just thinking on the level of automorphism groups the transition $\SO(2;\Lambda_\delta)$ is interesting.

\begin{example}
The transition from $\O(2;\C)$ to $\O(2;\R\oplus\R)$ is topologically a transition from two cylinders to four tori, two of them 'coming in from infinity':	
\end{example}

\begin{figure}
\centering
\includegraphics[width=0.6\textwidth]{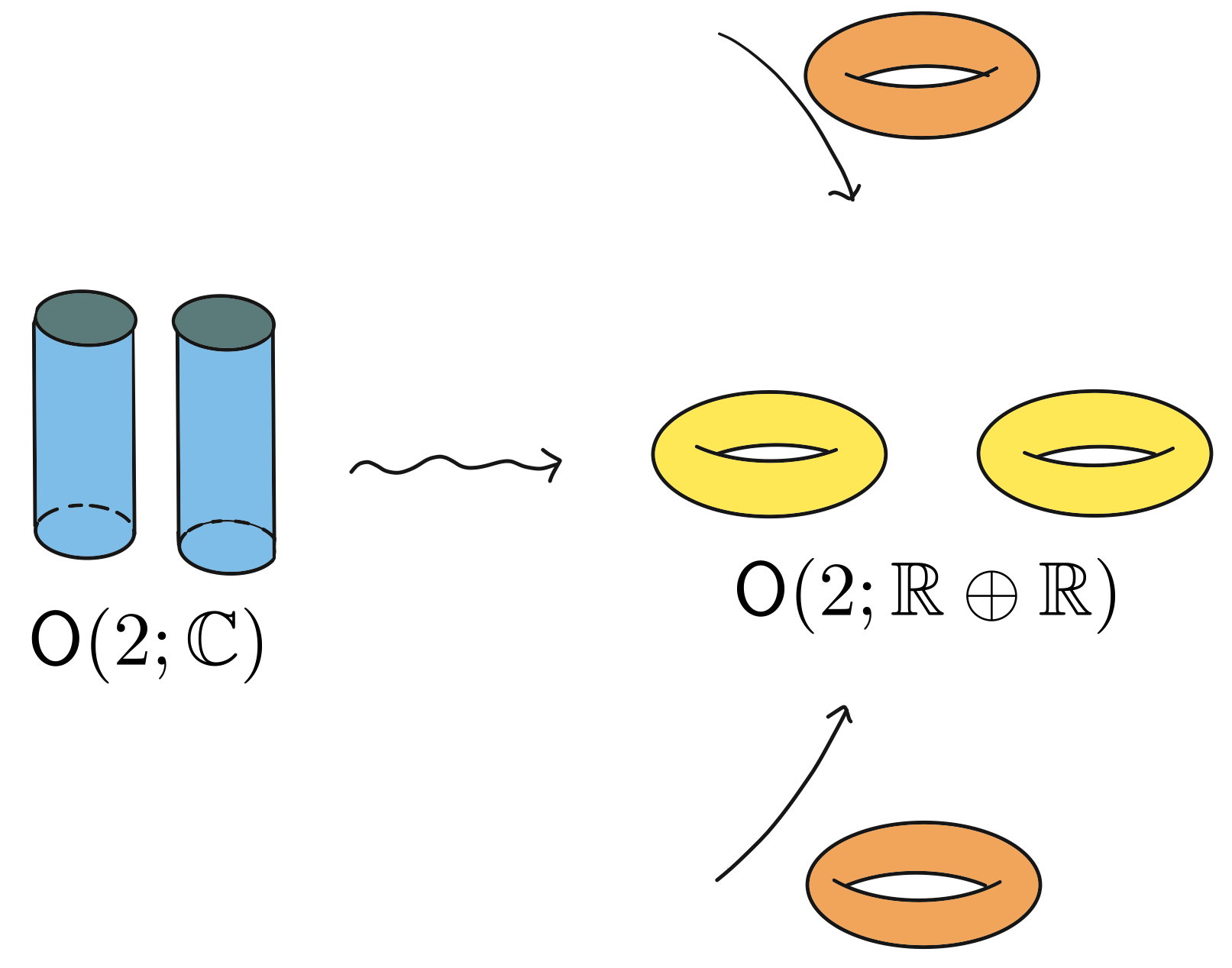}
\caption{The transition of orthogonal groups $\O(2;\C)$ to $\O(2;\R\oplus\R)$.}
\end{figure}

\noindent 
We may perform a similar analysis over the family $\fam{H}$ of quaternion algebras.
Understanding the unitary and orthogonal geometries defined over $\M(2;\R)$ is a topic of current research.

\begin{corollary}
There is a transition of quaternionic hyperbolic geometry to the signature $(n,1)$ unitary geometry over $\M(2;\R)$.	
\end{corollary}

\section{Varying the Basepoint}

Intuitively, the pointed geometry $(G,(X,x))$ is the homogeneous space $(G,X)$ viewed from $x$, and the question \emph{what does $(G,X)$ look like from infinity} can be interpreted as \emph{what pointed limit geometries arise as the basepoint is moved into an end of $X$}?

From the group-stabilizer viewpoint, its clear for general reasons that a limiting geometry exists.  
Indeed, the pointed geometries with automorphism group $G$ depend only on the stabilizer $K\subg G$ and so can be thought of as points in the Chabauty space $\Cl_G$.  
If $(G,X)$ is such a geometry, and $x_t\in X$ is a path of points leaving every compact set,
the corresponding stabilizer groups $K_t=\mathsf{stab}_G(x_t)$ subconverge in $\Cl_G$ by compactness to a closed subgroup $C$, and thus a limiting geometry $(G,C)$.

Restricting our attention to the orthogonal and unitary groups we can concretely understand such limiting geometries and realize them as transitions between pairs of well known classical geometries.
A motivating example to keep in mind is the hyperbolic plane $\Hyp^2$ thought of as a subgeometry of $\RP^2$.  The quadratic form defining $\Hyp^2$ has signature $(2,1)$ dividing $\RP^2$ into the hyperbolic plane and an open Mobius band, separated a circle (the projectivization of the null cone).  
Much as the action of $\SO(2,1)$ on the disk gives hyperbolic space, its action on Mobius band gives the other projective geometry with automorphism group $\SO(2,1)$, a Lorentzian geometry called \emph{de Sitter space}.  
Any path of points $x_t$ remaining in the disk give models of hyperbolic space $(\SO(2,1), \mathbb{D}^2, x_t)$ and any points in the Mobius band give models of de Sitter space $(\SO(2,1),\mathsf{Mob},x_t)$.  
Throughout the rest of this section we focus on families of points crossing between the two.

More generally, if $G$ is any orthogonal or unitary subgroup of $\GL(n;\R)$ or $\GL(n;\C)$  the associated quadratic / hermitian form defines a positive and negative cone, whose projectivizations $X_+$ and $X_-$ are the domains for the two projective geometries $(G,X_+)$, $(G,X_-)$ with automorphism group $G$.  The isomorphism type of the geometries depend on the signature $(p,q)$ of the form: $X_+$ is not isomorphic to $X_-$ unless $p=q$.  The main theorem of this section provides a transition between these geometries.

\begin{theorem}
There is a transition from $(G,X_+)$ to $(G,X_-)$ for any orthogonal or unitary group $G$.	
\end{theorem}

\begin{proof}
Fix an orthogonal or unitary group $G\subg\GL(n+1;\F)$ and consider its linear action on $\F^{n+1}$.  
The group preserves a quadratic / Hermitian form $J$, and the level sets of $J$ are precisely the orbits of $G$ on $\F^{n+1}\smallsetminus\vec{0}$ (the origin is fixed by the linear action).  
In fact, the map $q_J\colon \F^{n+1}\smallsetminus\set{0}\to\R$ is a submersion, and gives $\F^{n+1}\smallsetminus\set{0}$ the structure of a family over $\R$ (if $J$ has signature $(n,0)$ this only maps onto $\R^+$).
As these level sets are the $G$ orbits $\fam{O}=(\F^{n+1}\smallsetminus \set{0})/G$, we are exactly in the situation of \cref{lem:Grp_Action_to_Family_Action}, and the action of $G$ on $\mathbb{F}^{n+1}\smallsetminus \set{0}$ 
induces an action of families $G\times \fam{O}$ on $\F^{n+1}\smallsetminus\set{0}\to\fam{O}$.   
This provides a transition from $(G,X_-)$ to $(G,X_+)$ as non-pointed geometries, because the negative level sets of $q_J$  projectivize to $X_-$ and similarly $\P q_J\inv(\R_+)=X_+$.\end{proof}

The level sets of $q_J$ foliate the complement of $\vec{0}$, each determining a geometry when equipped with the action of $G$.  
The transition occurs passing through the zero level set, which is the null cone of the form (of course, there is no nontrivial transition for signature $(n,0)$).
Thus the geometry $(G,X_+)$ transitions to $(G,X_-)$ through the geometry associated to the $G$ action on the \emph{non-projectivized lightcone} $X_0=\set{v\neq 0\mid q_J(v)=0}$.

\begin{figure}
\centering
\includegraphics[width=0.75\textwidth]{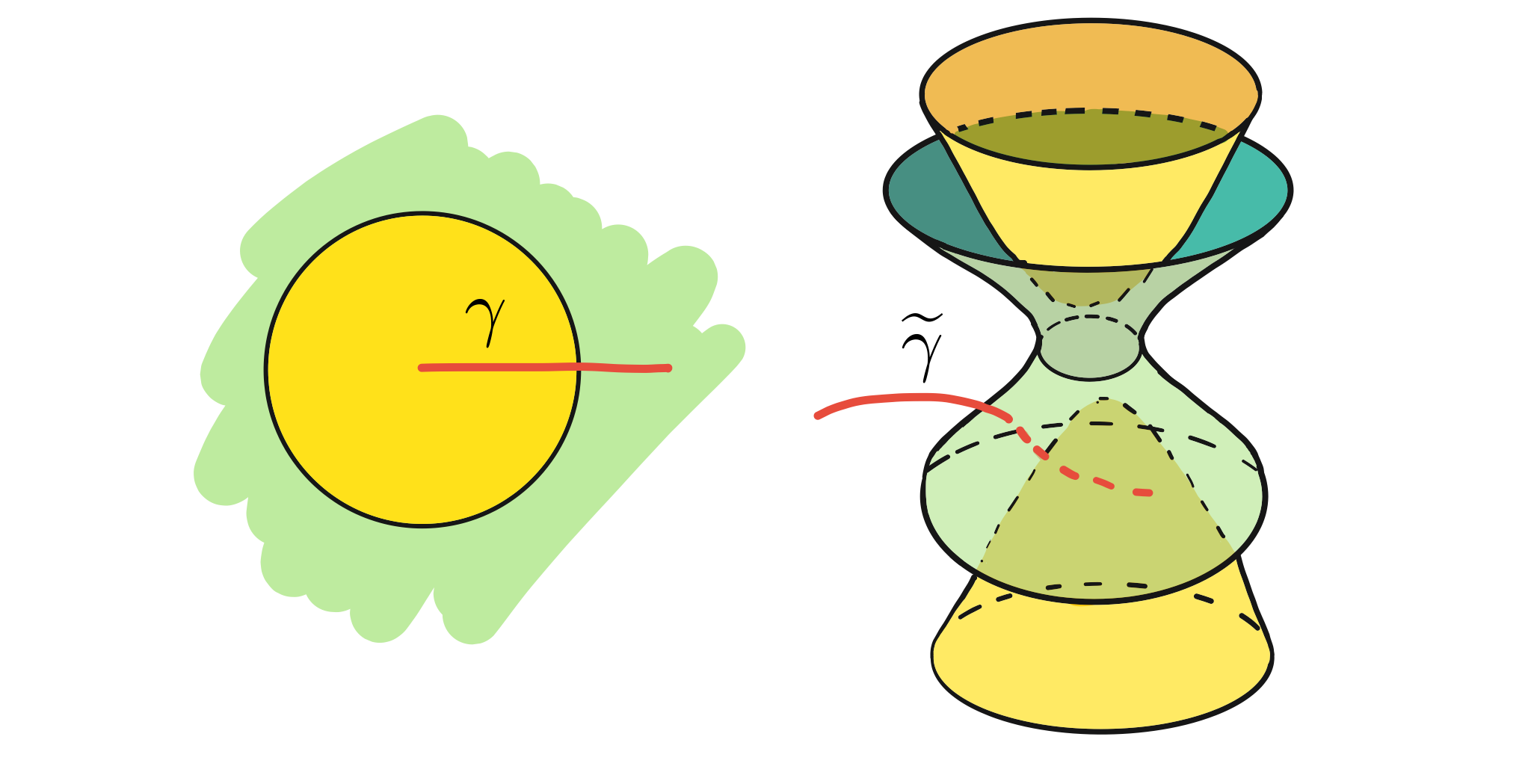}
\caption{Points in $\RP^n$, lifts to $\R^{n+1}$ and the associated stabilizers.}
\end{figure}

\begin{corollary}
To each classical orthogonal / unitary group	 there corresponds a family of pointed geometries with base $\mathbb{F}\mathsf{P}^n$.
\end{corollary}
\begin{proof}
Given the smooth family of geometries $(G\times\R,\mathbb{F}^{n+1}\smallsetminus\set{0})$ above, \cref{prop:Stabilizer_Families_of_Family} shows that the collection of point stabilizers form a family with base the total space of the geometry, here $\mathbb{F}^{n+1}\smallsetminus\set{0}$. 
Thus we have a family of pointed geometries (given in the point stabilizer formalism) with base $\F^{n+1}\smallsetminus\set{0}$.  
As the action on $\F^{n+1}$ is linear however, the point stabilizer assigned to $x$ and $\alpha x$ are equal for all $\alpha\in\F^\times$, to this descends to a family $\fam{stab}\to\F\mathsf{P}^n$.
This induces the claimed family of geometries $(G\times\F\mathsf{P}^n,\fam{stab})$ in $\Fam_{\F\mathsf{P}^n}$.
\end{proof}

\noindent
The resulting family is \emph{almost} a family of projective geometries, in the sense that for $[x]\in\F\mathsf{P}^n$ with $q_J(x)\neq 0$ the member above $[x]$ is isomorphic to $(G,X_+)$ or $(G,X_-)$.  
However for points $[x]$ lying on the null cone, the geometry is \emph{not} a projective geometry as the domain is the unprojectivized cone.
Thus these are examples of transitions between two projective geometries, which do not occur \emph{through} projective geometries.

\begin{figure}
\centering
\includegraphics[width=0.55\textwidth]{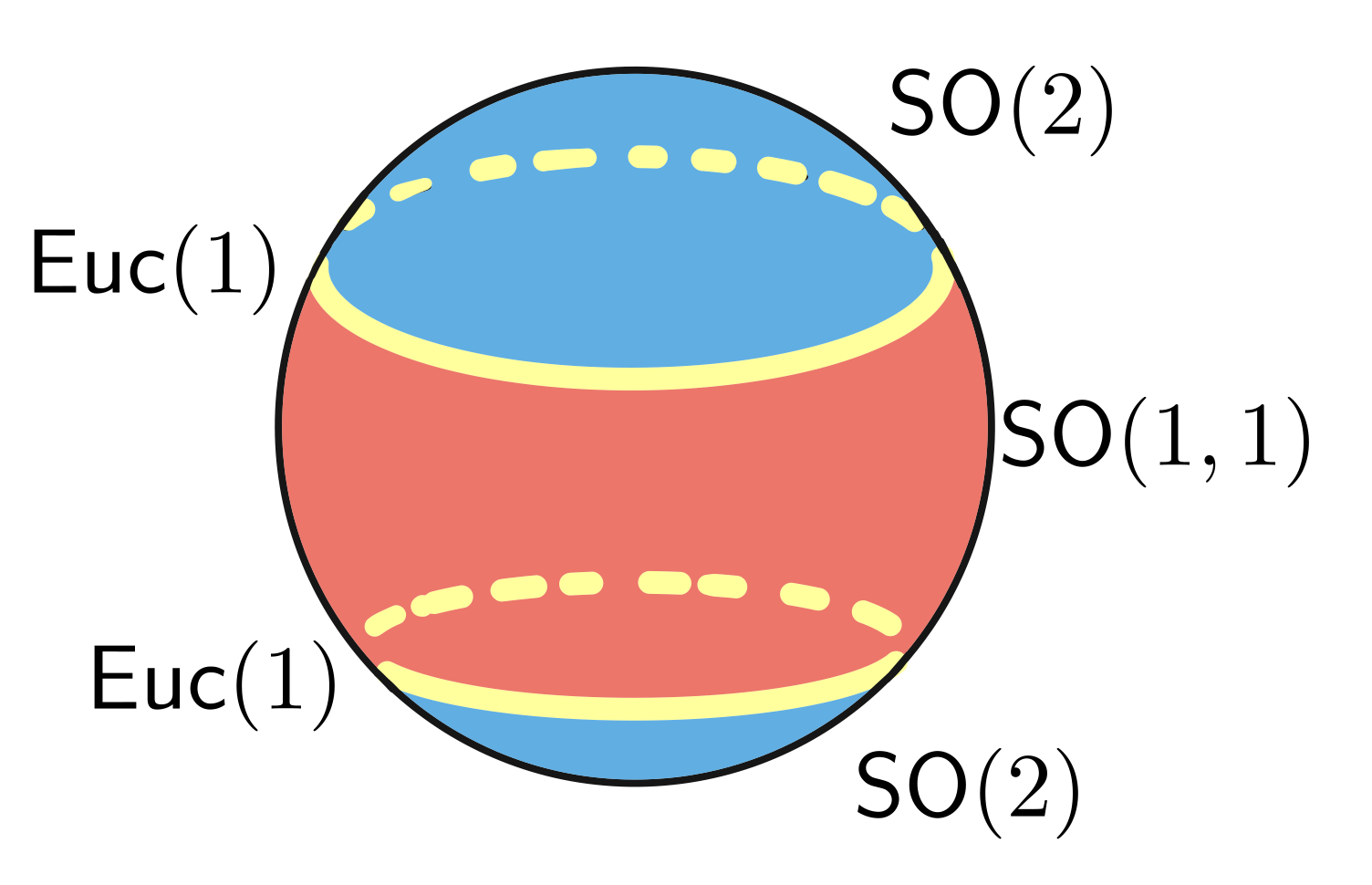}
\caption{The point stabilizers for the action of $\SO(2,1)$ on $\R^3\smallsetminus \{0\}$, as a family over $\RP^2$.}	
\end{figure}

\subsection{The Hyperbolic - de Sitter Transition}

In all dimensions, the null cone for the $(n,1)$ form divides $\RP^n$ into an $n$ ball and its complement; the action of $\SO(n,1)$ on $\mathbb{D}^n$ defines the Klein model of $\Hyp^n$ and on the complement a projective model of de Sitter space $\mathsf{d}\mathbb{S}^n$.  
Here we briefly discuss the transitional geometry in this case.  The lightcone of the $(n,1)$ form projectivizes to $\S^{n-1}\subset\RP^n$ forming the common boundary to $\Hyp^n$ and $\mathsf{d}\mathbb{S}^n$.  
The action of $\SO(n,1)$ on $\S^{n-1}$ determines a model of conformal geometry (the isometries of hyperbolic space determine conformal transformations of the ideal boundary), and so realizing the null cone as the canonical line bundle to the projective $\S^{n-1}\subset\RP^n$, the light cone geometry is just the geometry of the canonical line bundle to the conformal sphere.

\begin{corollary}
There is a transition from $\Hyp^n$ to $\mathsf{d}\S^n$ through the geometry of the canonical line bundle to the conformal $n-1$ sphere.	
\end{corollary}

\begin{figure}
\centering
\includegraphics[width=0.5\textwidth]{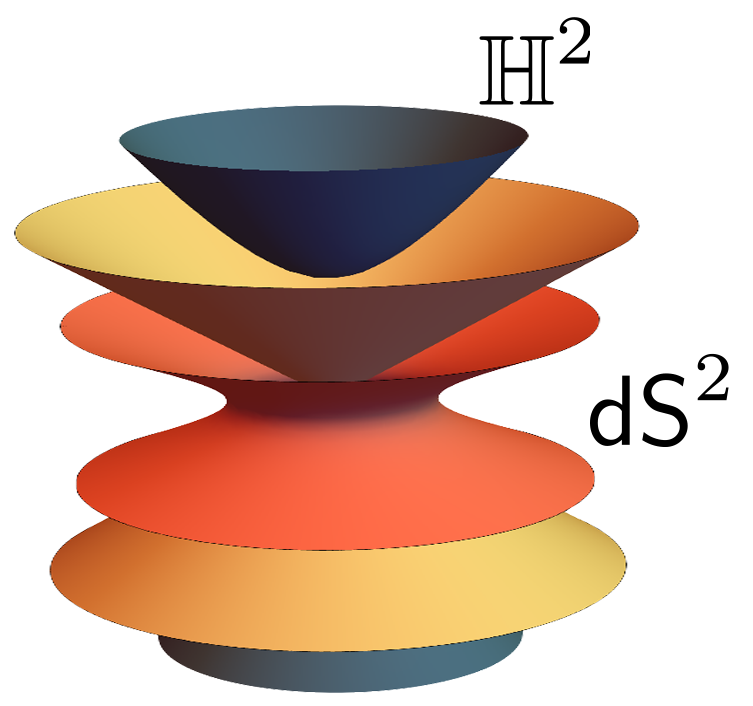}
\caption{The natural embedding of this family as a subset of $\R^3$.}
\end{figure}

\backmatter


 \clearpage

\printindex

\bibliography{Refs}
\bibliographystyle{plain}

\end{document}